\documentclass[12pt,reqno]{amsbook}
\pdfoutput=1
\usepackage{amsmath}
\usepackage{amssymb}
\usepackage{euscript,eufrak}
\usepackage{mathrsfs}
\usepackage[all]{xy}
\usepackage{graphicx}
\usepackage{color}
\usepackage{amsfonts}
\usepackage{latexsym}
\usepackage[colorlinks=true,linktocpage=true,pagebackref=false, urlcolor=black, citecolor=black,linkcolor=black]{hyperref}
\usepackage{ upgreek } 
\usepackage[T1]{fontenc}
\newcommand{\nsq}[1]{\,\framebox{\rule[0cm]{0cm}{0cm}$#1$}\,}
\usepackage{amsthm}
\usepackage{quoting}
\quotingsetup{font=small}

\usepackage{scalerel}

\catcode`\@=11
\newbox\rist@
\def\ristr#1#2{\setbox\rist@=\hbox{$\hbox{$#1$}_{#2}$}\setbox0\hbox{$#1$}
{{#1}\,\vrule width.39998pt height\ht\rist@ depth\dp\rist@\,
 \hbox{\vrule depth\dp0 height \ht0 width0pt}_{#2}}}
\def\ristrset#1#2{\setbox\rist@=\hbox{$\ssize\hbox{$\ssize#1$}_{#2}$}\setbox0\hb
ox{$\ssize#1$}
 {{#1}\,\vrule width.39998pt height\ht\rist@ depth\dp\rist@\,
 \hbox{\vrule depth\dp0 height \ht0 width0pt}_{#2}}}
\catcode`\@=12

\usepackage{multirow}
\usepackage{pifont}
\usepackage[table]{xcolor}

\def\interior{\mathaccent"0017 }

\normalbaselines
\newtheorem{thm}{Theorem}[section]

\newtheorem{lem}[thm]{Lemma}
\newtheorem{prop}[thm]{Proposition}

\newtheorem{cor}[thm]{Corollary}

\theoremstyle{definition}
\newtheorem{defn}[thm]{Definition}
\newtheorem{defns}[thm]{Definitions}

\newtheorem*{problem}{Problem}

\theoremstyle{remark}
\newtheorem{remark}[thm]{Remark}
\newtheorem{remarks}[thm]{Remarks}
\newtheorem{example}[thm]{Example}
\newtheorem{examples}[thm]{Examples}

\newtheorem{quest}[thm]{Question}

\numberwithin{equation}{section}

\newcommand{\K}{{\mathbb K}} \newcommand{\N}{{\mathbb N}}
\newcommand{\Z}{{\mathbb Z}} \newcommand{\R}{{\mathbb R}}
\newcommand{\Q}{{\mathbb Q}} \newcommand{\C}{{\mathbb C}}

\newcommand{\sph}{{\mathbb S}} \newcommand{\D}{{\mathbb D}}
 
\newcommand{\PP}{{\mathbb P}}

\newcommand{\psd}{{\mathcal P}} 
\newcommand{\an}{{\EuScript O}} 

\newcommand{\hol}{{\mathcal H}} 
 \newcommand{\J}{{\mathcal J}}
 
 \newcommand{\I}{{\mathcal I}}
\newcommand{\ideal}{{\mathcal I}}
\newcommand{\ceros}{{\mathcal Z}}

\newcommand{\Mm}{{\mathfrak M}}
\newcommand{\mm}{{\mathfrak m}}

\newcommand{\gtp}{{\mathfrak p}} \newcommand{\gtq}{{\mathfrak q}}
\newcommand{\gtm}{{\mathfrak m}} \newcommand{\gtn}{{\mathfrak n}}
\newcommand{\gta}{{\mathfrak a}} \newcommand{\gtb}{{\mathfrak b}}
\newcommand{\gtP}{{\mathfrak P}} \newcommand{\gtQ}{{\mathfrak Q}}
 
 \newcommand{\gtA}{{\mathfrak A}}
 \newcommand{\gtM}{{\mathfrak M}}
\newcommand{\gtg}{{\mathfrak g}}  \newcommand{\gtV}{{\mathfrak V}}
\newcommand{\gtN}{{\mathfrak N}}

\newcommand{\Oo}{{\EuScript O}}
\newcommand{\Ii}{{\EuScript  I}}
\newcommand{\Jj}{{\EuScript J}}
\newcommand{\Ff}{{\EuScript F }}
\newcommand{\Aa}{{\EuScript A }}
\newcommand{\Gg}{{\EuScript G}}

\newcommand{\Rr}{{\EuScript R}}
\newcommand{\Cc}{{\EuScript C}}

\newcommand{\Env}{{\EuScript E}}

\newcommand{\shprod}{{\text{\Large$\displaystyle\uppi$}}}

\newcommand{\Reg}{\operatorname{Reg}}
\newcommand{\Sing}{\operatorname{Sing}}

\newcommand{\Int}{\operatorname{Int}}
\newcommand{\cl}{\operatorname{Cl}}

\newcommand{\supp}{\operatorname{supp}}

\newcommand{\id}{\operatorname{id}}

\newcommand{\rg}{\operatorname{rk}}

\newcommand{\zar}{\operatorname{zar}}

\newcommand{\GL}{\operatorname{GL}}
\newcommand{\re}{\operatorname{Re}}
\newcommand{\ima}{\operatorname{Im}}

\newcommand{\x}{{\tt x}} \newcommand{\y}{{\tt y}} 
\newcommand{\z}{{\tt z}} 
 \renewcommand{\u}{{\tt u}}
 \newcommand{\h}{{\tt h}}

\newcommand{\veps}{\varepsilon}

\newcommand{\half}{\frac{1}{2}}
\newcommand{\ol}{\overline}
\newcommand{\qq}[1]{\langle{#1}\rangle}

\numberwithin{equation}{section}

\newcommand{\bounded}{{\mathcal B}}

\newcommand{\stb}{{\EuScript E}}
\newcommand{\e}{{\tt e}}

\usepackage{mathabx}
\usepackage{tikz-cd}

\usepackage{imakeidx}
\makeindex[intoc]

\begin{document}

\title{Topics in global real analytic geometry}
\author{F. Acquistapace, F. Broglia, J. F. Fernando}
\maketitle

\null

secondary: 58A07, 32Q35.

\bigskip

\sl Francesca Acquistapace; Fabrizio Broglia.

\rm Dipartimento di Matematica, Universit\`a degli Studi di Pisa, Largo Bruno Pontecorvo, 5, 56127 PISA (ITALY).

francesca.acquistapace@unipi.it, fabrizio.broglia@unipi.it

\medskip

\sl Jose Francisco Fernando.

\rm Departamento de \'Algebra,Geometr\`\i a y Topolog\'\i a, Facultad de Ciencias Matem\'aticas, Universidad Complutense de Madrid, 28040 MADRID (SPAIN). 

josefer@mat.ucm.es

\pagenumbering{Roman}
\setcounter{tocdepth}{1}
\pagenumbering{arabic}
\setcounter{page}{0}
\tableofcontents

\newpage

\chapter*{ Introduction.}

In the first half of twentieth century the theory of complex analytic functions and of their zerosets was fully developed.  The definition of holomorphic function has a local nature. Germs of holomorphic functions form a distinguished subring of the ring of germs of continuous functions. 
This way came out the notion of {\em analytic space}. The definition of a complex analytic set is by {\em local models} as in the case of complex manifolds. But while local models for manifolds are open sets of $\C^n$, a local model of an  analytic space is the  zeroset of finitely many analytic functions on an open set of $\C^n$ together with a sheaf of continuous function to be called {\em holomorphic}.

Towards the  years $50$ of  the last century, Cartan,  Whitney, Bruhat
and  others tried  to  formulate  the notion  of  analytic space  over
$\R$.  Immediately  they  realize  that  the  real  sets  verifying  a
definition similar to the complex  one form a cathegory whose elements
can get unpleasant behaviour. In particular this cathegory does not get
the  good  properties of  complex  analytic  spaces, as  for  instance
coherence of their structural sheaves and Theorems A and B do not hold
in  general.  In  contrast  with  what happens  for  instance  in  the
algebraic case it is not always  possible to see a real analytic space
as the  fixed point set of  a suitable {\em conjugation}  on a complex
analytic space. In this situation some doubts arose on the interest of
such  investigations. For  instance  Grothendieck  wrote in  \cite{c5}
exp. 9 pg 12

\begin{quoting}
Lorsque $k$ est alg\'ebriquement clos, il est probablement vrai que
tout espace analytique r\'eduit \`a un point est de la forme qu'on vient
d'indiquer, ce qui serait une des variantes du ``Nullstellensatz''
analytique. Signalons par contre tout de suite que rien de tel n'est vrai si
$k$ n'est pas alg\'ebriquement clos, par exemple si $k$ est le corps des
r\'eels $\R$. Ainsi, le sous-espace analytique de $\R^2$ d\'efini par
l'id\'eal engendr\'e par $x^2+y^2$ est r\'eduit au point origine, mais son
anneau local en ce point n'est pas artinien, mais de dimension de Krull
\'egale \`a $1$. L'int\'er\^et des espaces analytiques, lorsque $k$ n'est
pas alg\'ebriquement clos, est d'ailleurs douteux.
\end{quoting}

On the contrary Cartan worked to find the obstructions to get a good real cathegory. 
He proved that Theorems A and B pass through direct limits. So, since $\R^n$ has in $\C^n$ a fundamental system of open Stein neighbourhoods, he proved that analytic subsets of $\R^n$ defined as the zeroset of {\em global} analytic functions are support of a coherent sheaf of ideals. This sheaf of ideals defines a complex analytic subset of a Stein open neighbourhood of $\R^n$ in $\C^n$ hence Theorems A and B hold true. So he found a good class of real analytic spaces globally defined in $\R^n$ and hence getting a good {\em complexification}. He wrote \cite[p.49]{c1} 

\begin{quoting}
  \ldots  la seule  notion de  sous-ensemble analytique  r\'eel (d'une
  vari\'et\'e analy\-ti\-que-r\'eelle $V$) qui ne conduise pas \`a des
  propri\'et\'es  pathologiques  doit   se  r\'ef\'erer  \`a  l'espace
  complexe ambiant:  il faut consid\'erer les  sous-ensembles ferm\'es
  $E$ de $V$  tels qu'il existe une complexification $W$  de $V$ et un
  sous-ensemble analytique-complexe $E'$ de $W$, de mani\`ere que $E=W
  \cap E'$. On d\'emontre que ce  sont aussi les sous-ensembles de $V$
  qui  peuvent  \^etre  d\'efinis   globalement  par  un  nombre  fini
  d'\'equations    analytiques.    La    notion    de    sous-ensemble
  analytique-r\'eel  a ainsi  un  caract\`ere essentiellement  global,
  contrairement  \`a  ce  qui   avait  lieu  pour  les  sous-ensembles
  analytiques-complexes.
\end{quoting}

Cartan  uses  complex notions to describe real properties: for instance he
 defines the {\em complexification} of 
a germ of real analytic space $V_x$ at a point $x\in \R^n$ and proves that $V_x$ is coherent if and only if the complexification of $V_x$ induces the complexification of $V_y$  on  points $y$  close to $x$.

All these considerations  led Cartan  to the following characterization of what a good cathegory of real analytic sets should be, proving that 
for a closed real analytic subset $X\subset  \R^n$ the following statements are equivalent.
\begin{enumerate}
\item The set $X$ is the zeroset of finitely many real analytic functions.
\item There is a coherent ideal subsheaf of $\Oo_{\R^n}$ whose zeroset is $X$.
\item There is an open neighbourhood $\Omega$ of $\R^n$ in $\C^n$ and a closed subspace $Y \subset \Omega$ such that $Y\cap \R^n = X$.
\end{enumerate}

So, the notion of {\em complexification} plays a central role.
   
Bruhat  and Witney  extended the  notion of  complexification to  real
analitic manifolds  and introduced the  name {\em C-analytic}  for the
analytic  subsets  of   a  real  analytic  manifold   $M$  induced  by
intersection with $M$ of an analytic subset of the complexification of
$M$. Finally Tognoli extended the notion of complexification of a real
analytic  space  admitting  a  coherent  structure,  in  analogy  with
condition 2 of Cartan.

Tognoli distinguished  3 types of  real analytic spaces:  the coherent
ones, whose reduced  structure is coherent, the ones  carrying one (or
several)  coherent   structure  (for   instance  Whitney   and  Cartan
umbrellas) and  the ones  not admitting  any coherent  structure (wild
examples of Cartan, Bruhat-Cartan etc.)

\smallskip

Since a C-analytic space $X$ is globally defined, the ring $\Oo(X)$ of
(real) analytic functions on $X$ becomes interesting. One could follow
the  development of  Real Algebraic  Geometry and  try to  imitate its
theory. This is easy in the local case, not in the global one. This is
because an important  step in the algebraic theory is  the fact that a
non-negative  polynomial is  a sum  of squares  of rational  functions
(Artin's solution of Hilbert $17^{th}$  Problem). In the analytic case
this is  true for the ring  of germs, while  in the global case  it is
proved, as far  as we know, only in some  special particular cases, so
one  cannot expect  to  get  results in  analogy  with Real  Algebraic
Geometry.

In this book we follow another path, closer to Cartan's point of view, that is, deduce results for C-analytic spaces from the properties of their complexification.

In the  first two  chapters of  this book  we mainly  recall classical
results.  More precisely  in Chapter 1, after exposing  the main facts
on   complex  analytic   spaces,  we   give  the   costruction  of   a
complexification, showing why it is  necessary to pass to C-spaces. We
give also some bad examples of  real analytic sets in $\R^n$ following
Cartan  and  Bruhat.  In  Chapter   2  there  is  the  costruction  of
irreducible components  of complex and  real analytic sets.  After, we
see how this notion works when dealing with {\em normalization} from a
local  and a  global point  of view.  Concerning {\em  divisors} of  a
C-space, we  try to  ask to  the question of  what are  the conditions
under which a divisor is the divisor of a global analytic function.

As  we  said, in  the  local  case results  are  very  similar to  the
algebraic  ones. For  instance  the Nullstellensatz  for  the ring  of
complex or real analytic germs is exactly  the same as for the ring of
complex or real polynomials. The situation is no more the same for the
ring of global  analytic functions. In Chapter 3 we  give the proof of
Nullstellensatz for  closed ideals  in $\Oo(X)$ where  $X$ is  a Stein
space, following O.Forster. The  primary (infinite) decomposition of a
closed  ideal allows  to consider  irreducible components  of a  Stein
space with {\em multiplicity} as in the algebraic case.

We remark that in this case  there is a numerical function (namely the
primality  index)  which  controls whether  Hilbert's  Nullstellensatz
holds true  or not for a  closed ideal.  It  holds if and only  if the
numerical function is uniformely bounded. A somewhat similar result we
get for Hilbert's $17^{th}$ Problem. If it has a positive solution for
$\Oo(\R^n)$, then, the  Pythagoras number of the  field of meromorphic
functions is bounded.\footnote{The Pythagoras number  of a ring $R$ is
  a positive number $p$  if all sums of squares in  $R$ can be written
  as  sums  of $p$  squares,  otherwise  it  is $\infty$.}  Note  that
Hilbert's Problem and the calculus of Pythagoras number are completely
separate problems in  the algebraic case, while we  find an unespected
relation in the analytic case.

Also in Chapter 3 we give a real Nullstellensatz for the ring $\Oo(X)$
where $X$  is a C-analytic space.  The radical we use  is not Risler's
real radical,  it is {\em \L  ojasiewicz radical} which is  in general
larger than the  real radical of a  given ideal and is  the radical of
the convex  hull of the  given ideal. This is  because we do  not know
whether a positive semidefinite analytic  function is a sum of squares
of meromorphic functions. Indeed, in case the zeroset $Y$ of the given
ideal  is such  that positive  semidefinite functions,  having $Y$  as
zeroset, {\em are } sums of  squares of meromorphic functions, the two
radicals  coincide   and  we   get  a   result  similar   to  Risler's
Nullstellensatz.

So,  Hilbert  $17^{th}$  Problem  is crucial  also  in  Real  Analytic
Geometry. Hence  in Chapter  4 we give  the state of  the art  on this
problem, whose  solution is far  from being complete. Also  we discuss
some weaker questions which involve infinite sums of squares.

When dealing with real objects, inequalities appear immediately, equalities are not enough. Look for instance to the orthogonal projection of a circle from $\R^2$ to $\R$.

\L  ojasiewicz and  Hironaka  defined a  class of  subsets  of a  real
analytic manifold, that is {\em  semianalytic sets}, which are locally
defined by  analytic equalities  and inequalities. This  class behaves
well  with  respect  to  topological  properties  as  closure,  taking
connected  components  and so  on,  but  is  not stable  under  proper
projections. For this  reason the class of {\em  subanalytic sets} was
also introduced.  Nevertheless one  can ask whether  semianalytic sets
defined  by  finitely  many  global analytic  functions  ({\em  global
  semianalytic  sets}) could  get better  properties. Several  authors
investigated these sets applying the  algebraic thory of orders to get
topological  properties,  mainly  in  dimension less  than  three.  An
important result in general dimension is  that the closure of a global
semianalytic set  is locally  global. We  give another  more geometric
proof of this result in Chapter 5.

We define  an intermediate  class of semianalytic  sets, that  we call
{\em C-semianalytic  sets} in  analogy with  the notion  of C-analytic
spaces, that is locally finite  unions of global semianalytic sets. We
prove that this class is stable under topological operations. Moreover
it is stable  under proper invariant holomorphic maps. We  do not need
to  define  {\em  C-subanalytic  sets}   because  we  can  prove  that
subanalytic  sets  can  be  defined  replacing  semianalytic  sets  by
C-semianalytic sets.

Several remarkable  subsets of a  C-analytic set, as for  instance the
set  of points  of a  given  dimension, or  the  set where  it is  not
coherent, or  the set  of local  extrema of  an analytic  function are
proved to be C-semianalytic sets.

A theory of irreducibility and of irreducible components, analogous to
the  one  developed   for  semialgebraic  sets,  does   not  hold  for
C-semianalytic sets. Neverteless  there is a smaller  class where this
theory applies.  It is the  class of {\em amenable}  semianalytic sets
which are locally  finite unions of sets of the  type $U\cap X$, where
$U$ is an open set and $X$ is a C-analytic set.



In the text  references are mainly concentred in a  section at the end
of  each chapter  together  with some  historical  notes: for  general
notions  in complex  analysis we  refer to  \cite{gr,gu,ssst} and  for
notions in commutative algebra to  some classical texts as \cite{mat1}
and \cite{amd}.

\chapter{The class of C-analytic spaces.}

In this chapter we  introduce, following the ideas collected in \cite{c,t,wb}, the class of \em real analytic spaces. \rm \index{real analytic space} This type of spaces appear in the literature also as C-analytic spaces\index{analytic space! C-analytic --}. To introduce them we need to recall before the concept of {\em complex analytic space}\index{analytic space! complex --}. In what follows all involved topological spaces are supposed to be Hausdorff, paracompact and second-countable. 
\section{Complex analytic spaces.}

 A {\em ringed space}\index{ringed space} is a couple $(X,\Oo_X)$, where $X$ is a topological space and $\Oo_X$ is a subsheaf of the sheaf of germs of continuous functions on $X$.
We recall shortly the notion of {\em coherent sheaf}\index{coherent sheaf}.

\begin{defn}  A sheaf $\Ff$ on a ringed space $(X,\Oo_X)$ is called {\em $\Oo_X$-coherent}\index{coherent sheaf} if it is a sheaf of $\Oo_X$-modules {\em of finite presentation}, that is it verifies the following two conditions.
\begin{itemize}
\item[(i)] It is a \em finite type sheaf\em, that is, for each $x\in X$ there exists an open neighborhood $U^x$ and finitely many sections $\{G_1,\ldots,G_k\}$ on $U^x$ generating the fiber $\Ff_y$ for each $y\in U^x$. 
\item[(ii)] For each open set $U\subset X$ and for each finite number of sections $H_1,\ldots,H_p\in\Ff(U)$ the \em sheaf of relations \em among them, that is, the kernel of the homomorphism of sheaves $\sigma:\Oo_X^p|_U\to\Ff|_U$ given on each open set $V\subset U$ by
$$
\sigma:\Oo_X(U)^p\to\Ff(U),\ (A_1,\ldots,A_p)\mapsto A_1H_1+\cdots+A_pH_p
$$ 
 is a finite type sheaf.
\end{itemize}
\end{defn}

First of all we define what we mean by  local model of complex analytic space.

A  {\em local model of complex analytic space}\index{analytic space! local model of -} is a pair $(Y,\Oo_Y)$ constituted by a closed subset $Y$ of an open set $\Omega\subset\C^n$ that is the common zeroset of finitely many holomorphic functions $F_1,\ldots,F_k$ on $\Omega$ and a structure provided by a suitable sheaf $\Oo_Y$. We can define this structure sheaf mainly in two ways. Remember that  by Oka's Theorem $\Oo_\Omega$ is a coherent sheaf of $\Oo_\Omega$-modules, hence a sheaf $\Ff$ of $\Oo_\Omega$-modules is coherent if and only if $\Ff$ is of finite type. In particular, finitely generated sheaves of $\Oo_\Omega$-modules are coherent.

\noindent{\bf Way 1.} Let $\Oo_\Omega$ be the sheaf of germs of holomorphic functions on $\Omega$ and let $\Ii_Y$ be the sheaf of ideals
 consisting of all holomorphic function germs vanishing on $Y$. It is a coherent sheaf of ideals (Oka's Theorem). Consider now the coherent sheaf $\Oo_Y={\Oo_\Omega}/{\Ii_Y}$. Clearly, 
$$
Y={\rm supp}({\Oo_\Omega}/{\Ii_Y})=\{y\in \Omega:\ \Ii_{Y,y}\neq\Oo_{\Omega,y}\}
$$ 
and $(Y,\Oo_Y)=({\rm supp}(\Oo_\Omega/\Ii_Y),\Oo_\Omega/\Ii_Y)\hookrightarrow(\Omega,\Oo_\Omega)$ is a ringed space. The homomorphism of sheaves $\Oo_\Omega\to\Oo_\Omega/\Ii_Y$ is surjective. 

\noindent{\bf Way 2.} Consider the subsheaf $\Jj$ of $\Oo_\Omega$ generated by the holomorphic functions $F_1,\ldots,F_k$ on $\Omega$ and define $\Oo_Y={\Oo_\Omega}/{\Jj}$. As $\Jj=(F_1,\ldots,F_k)\Oo_\Omega$ is a finitely generated subsheaf of ideals of the coherent sheaf $\Oo_\Omega$, it is coherent itself, so also $\Oo_Y$ is a coherent sheaf. Again we have 
$$
Y={\rm supp}\left({\Oo_\Omega}/{\Jj}\right)=\{y\in \Omega:\ \Jj_y\neq\Oo_{\Omega,y}\}
$$ 
and $(Y,\Oo_Y)=({\rm supp}\left(\Oo_\Omega/\Jj\right),\Oo_\Omega/\Jj)\hookrightarrow(\Omega,\Oo_\Omega)$ is a ringed space. Again the homomorphism of sheaves $\Oo_\Omega\to\Oo_\Omega/\Jj$ is surjective.

\begin{defns}\label{defspaziocomplesso}
A {\em complex analytic space} \index{analytic space! complex --} is a (Hausdorff, paracompact) topological space $X$ endowed with a sheaf of rings $\Oo_X$ such that the pair $(X,\Oo_X)$ is locally isomorphic as a ringed space to a local model endowed with the structure provided in Way 2. In case the local models are chosen using the structure provided in Way 1, we say that $(X,\Oo_X)$ is a {\em reduced complex analytic space}\index{analytic space! reduced --}.
\end{defns}

Forgetting the structure sheaf we get the notion of
  {\em complex analytic set}\index{Complex analytic set} $X$; it is a closed subset of an open set $\Omega\subset\C^n$ that admits a local description as the zero set of finitely many holomorphic functions, that is, for each point $x\in\Omega$ there exists an open neighborhood $U^x$ and finitely many holomorphic functions $F_1,\ldots,F_k$ on $U^x$ such that 
$$
X\cap U^x=\{y\in U^x: F_1(y)=0,\cdots,F_k(y)=0\}.
$$
One can provide $X$ a structure by considering the sheaf of holomorphic function germs on it. Namely $\Oo_X=\Oo_\Omega/\Ii_X$ where $\Ii_X$ is the sheaf of germs vanishing on $X$. This structure is often called the {\em natural structure on $X$} and it is obtained following Way 1 above. Instead of $\Ii_X$ we can consider any coherent sheaf of ideals $\Jj\subset\Oo_\Omega$ having $X$ as zeroset  and take $\Oo_X=\Oo_\Omega/\Jj$. As $\Jj$ is locally finitely generated, local models for this structure are those provided by Way 2.

In case $(X,\Oo_X)$ is a non reduced complex analytic space there exists a {\em reduction morphism}\index{reduction morphism} $\rho:(X,\Oo_X^r)\to(X,\Oo_X)$, where $\Oo_X^r$ is the reduced structure on $X$, that behaves as follows. For each local model $(X\cap U,\Oo_{X\cap U})$ the sheaf $\Oo_{X\cap U}=\Oo_U/\Jj$ is the quotient of $\Oo_\Omega$ by a sheaf of ideals $\Jj$, which is in general properly contained in the sheaf of ideals $\Ii_{X\cap U}$. Then $\rho_x:\Oo_{X,x}\to{\Oo_{X,x}^r}$ maps each germ $g\in\Oo_{X,x}$ to its class modulo $\Ii_{X,x}$ for each $x\in U$.  

When we do not mention the structure sheaf $\Oo_X$ of a complex analytic space we are implicitly considering its reduced structure. 
 
\subsection{Local properties.}

We recall now the main properties of a reduced complex analytic space.  We denote $\Oo_n$ the local ring of holomorphic function germs at the origin $0\in\C^n$. As a consequence of Weierstrass Preparation and Division Theorems one proves that $\Oo_n$ is an integrally closed, noetherian, factorial domain. In particular, each ideal $\gta$ of $\Oo_n$ is a finite intersection $\gta=\bigcap_{i=1}^r\gtq_i$ of primary ideals $\gtq_i$ of $\Oo_n$ and the ideal of germs $I(\ceros(\gta))$ vanishing identically on its zeroset $\ceros(\gta)$ is exactly its radical $\sqrt{\gta}=\bigcap_{i=1}^r\sqrt{\gtq_i}$ (R\"uckert's Nullstellensatz). The zeroset $X_0=\ceros(\gta)=\ceros(\sqrt{\gta})$ is a finite union of {\em irreducible components}\index{irreducible component}, which are precisely the zerosets of the prime ideals $\gtp_i=\sqrt{\gtq_i}$, if the decomposition $\sqrt{\gta}=\bigcap_{i=1}^r\sqrt{\gtq_i}$ is irredundant, that is, $\gtp_i\not\subset\gtp_j$ if $i\neq j$. 

Recall that the polydisc $\Delta(x,\veps)$ in $\C^n$ of center $x=(x_1,\ldots,x_n)$ and polyradius $\veps=(\veps_1,\ldots,\veps_n)$, where each $\veps_i>0$, is $\Delta(x,\veps)=\prod_{i=1}^nD(x_i,\veps_i)$, where $D(x_i,\veps_i)=\{z\in\C:\ |z-x_i|<\veps_i\}$ is the disc in $\C$ of center $x_i$ and radius $\veps_i$. 

 The local properties of the zeroset $X_0$ of a prime ideal $\gtp\subset\Oo_n$ are described by the following. 

\begin{thm}
There exists a linear change of coordinates and a polydisc $\Delta(0,\veps)=\Delta_1 \times\Delta_2$, where $\Delta_1\subset\C^d$ and $\Delta_2\subset\C^{n-d}$ are polydiscs centered at the origin, such that:
\begin{itemize}
\item Each function germ of a fixed finite subfamily of $\gtp$ has a representative on $\Delta(0,\veps)$.
\item $A={\Oo_n}/{\gtp}$ is an integral extension of $\Oo_d$. 
\item There exists a representative $X$ of $X_0$, which is a complex analytic subset of the polydisc $\Delta(0,\veps)$, such that outside the (thin) zeroset of a non-zero $D\in\Oo_d$, the projection $\pi:X\setminus \ceros(D)\to\Delta_1\setminus \ceros(D)$ is a covering map. In particular, $\dim(X)=d$.
\item The difference $M=X\setminus \ceros(D)$ is a complex analytic manifold defined as the zeroset in $\Delta(0,\veps)$ of the representatives of $n-d$ elements of $\gtp$ whose Jacobian matrix has rank $n-d$ at each point of $M$. 
\item $M=X\setminus \ceros(D)$ is connected and dense in $X$. 
\end{itemize}
\end{thm}
The projection $\pi:X\to \Delta_1$ is a branched covering, also called {\em analytic cover}. 

The description above applies to the irreducible components of the zeroset $X_0$ of any radical ideal $\gta$ of $\Oo_n$. In this case the ideal $\gta$ is the intersection of the prime ideals $\gtp_i$ associated to the minimal prime ideal $\gtP_i=\gtp_i/\gta$ of $A=\Oo_n/\gta$. Then for each minimal prime ideal $\gtp_i$ we apply the argument to a suitable representative $X_i$ of $X_{i,0}=\ceros(\gtp_i)$, which is a complex analytic subset of a polydisc $\Delta(0,\veps)$. This polydisc is the same for all the germs $X_{i,0}$ (after applying a linear change of coordinates that works simultaneously for all the irreducible components of $X_0$). For each $i$ we obtain a complex analytic manifold $M_i\subset X$, which is dense in $X_i$. Then $M=\bigcup_iM_i\setminus\bigcup_{i\neq j}(M_i\cap M_j)$ is an open and dense subset of $X=\bigcup_iX_i$. Each connected component of $M$ is dense in an irreducible component $X_i$ of $X$ and each $X_i$ is the closure of a connected component of $M$. As $\gtP_i=\gtp_i/\gta$ is a minimal prime ideal of $A$ we get 
$$
\dim(A)=\max_i\{\dim(A/\gtP_i)\}=\max_i\{\dim(\Oo_n/\gtp_i)\}=\max_i\{\dim(X_{i,x})\}=\dim(X_x).
$$

\begin{remark}\label{algebraanalitica} Call \em analytic algebra \em any ring $A$ isomorphic to $\Oo_n/\gta$ for some $n$ and some ideal $\gta \subset \Oo_n$. The description above shows that any analytic set germ is equipped with an analytic algebra, but conversely an analytic algebra determines an analytic set germ, namely the zeroset of the ideal $\gta$ in a neighbourhood of $0\in \C^n$. Moreover if $f:X_x\to Y_y$ is a holomorphic map between analytic set germs, it induces an algebra homomorphism $f^*:\Oo_{Y,y}\to \Oo_{X,x}$ and viceversa. It is easy to prove that $f$ injective implies $f^*$ surjective  and $f$ surjective implies $f^*$ injective and viceversa. Hence $f$ is an isomorphism if and only if $f^*$ is an isomorphism.
\end{remark}

\subsubsection{Regular points of a reduced complex analytic space.}

Let $(X,\Oo_X)$ be a reduced complex analytic space. As the notion of \em regular point \em \index{regular point}has a local nature, we assume $(X,\Oo_X)$ is a reduced local model. Thus, $X$ is a closed subset of an open set $\Omega\subset\C^n$ and $\Oo_X=\Oo_\Omega/ \Ii_X$, where $\Ii_X$ is the sheaf of ideals of all holomorphic germs vanishing on $X$. For each $x\in X$ let $F_1,\ldots,F_\ell$ be generators of $\Ii_{X,x}$. We write 
$$ 
r_x={\rm rk}\left(\frac{\partial(F_1,\ldots,F_\ell)}{\partial(\x_1,\ldots,\x_n)}(x)\right)={\rm rk}\left(\frac{\partial F_i}{\partial\x_j}(x)\right)_{\substack{1\leq i\leq\ell\\1\leq j\leq n}}\leq \min\{n,l\}.
$$

It is straightforward to show that changing the set of generators, the value $r_x$ does not change, that is, it depends only on $\Ii_{X,x}$. Indeed, if $H_1,\ldots,H_s$ is another system of generators, it is enough to prove the rank of the jacobian of $F_1,\ldots F_\ell, H_i$ is the same as the rank of the jacobian of  $F_1,\ldots F_\ell$. Put  $F_{\ell+1}=H_i=G_1F_1+\cdots+G_\ell F_\ell$ where $G_i\in\Oo(\Omega)$. Thus, if $x\in X$, we have $F_1(x)=\cdots=F_\ell(x)=0$ and we conclude
{\small\begin{multline*}
{\rm rk}\left(\frac{\partial F_i}{\partial\x_j}(x)\right)_{\substack{1\leq i\leq\ell+1\\1\leq j\leq n}}={\rm rk}\left(\Big(\frac{\partial F_i}{\partial\x_j}(x)\Big)_{\substack{1\leq i\leq\ell\\1\leq j\leq n}}\,\Big(\sum_{i=1}^\ell G_i\frac{\partial F_i}{\partial\x_j}(x)+\sum_{i=1}^\ell\frac{\partial G_i}{\partial\x_j}(x)F_i(x)\Big)_{1\leq j\leq n}\right)\\
={\rm rk}\left(\Big(\frac{\partial F_i}{\partial\x_j}(x)\Big)_{\substack{1\leq i\leq\ell\\1\leq j\leq n}},\Big(\sum_{i=1}^\ell G_i\frac{\partial F_i}{\partial\x_j}(x)\Big)_{1\leq j\leq n}\right)={\rm rk}\left(\frac{\partial F_i}{\partial\x_j}(x)\right)_{\substack{1\leq i\leq\ell\\1\leq j\leq n}}=r_x.
\end{multline*}}

\begin{defn}
The {\em rank} of the sheaf of ideals $\Ii_X$ at a point $x\in X$ is 
$$
r(x)=\min_{W^x}\{\max_{y\in W^x}\{r_y\}\}\leq n.
$$
where $W^x$ runs over all the open neighborhood of $x$ in $X$.
\end{defn}

As the sheaf of ideals $\Ii_X$ is coherent, $r_x\leq r_y$ for each $y$ in a small neighborhood of $x$. In addition, \em $\{x\in X:\ r_x=r(x)\}$ is an open subset of $X$\em.
Indeed,
pick a point $x\in X$ such that $r_x=r(x)$. Let $W^x\subset X$ be an open neighborhood of $x$ such that $\displaystyle r(x)=\max_{y\in W^x}\{r_y\}$. As $r_x=r(x)$ and the sheaf of ideals $\Ii_X$ is coherent, we may assume that $r_y=r(x)$ for all $y\in W_x$, so $r_y=r(y)$ for all $y\in W_x$, so $W_x\subset\{x\in X:\ r_x=r(x)\}$ and this set is open.

\begin{defn}
The point $x\in X$ is {\em regular} if $r_x=r(x)$.
\end{defn}
We denote $\Reg(X)$ the set of regular points of $X$ and $\Sing(X)=X\setminus\Reg(X)$
the set of {\em singular points} of $X$. As $\Reg(X)$ is an open subset of $X$, the set $\Sing(X)$ is closed in $X$.

Let us prove next that each connected component $N$ of $\Reg(X)$ is a complex aalytic manifold. More precisely.

\begin{prop}\label{localmanifold}
Let $x_0\in\Reg(X)$ be a regular point of $X$. Then there exists an open neighborhood $U$ of $x_0$ in $\Omega$ such that $X\cap U$ is a complex manifold of dimension $n-r(x_0)$.
\end{prop}
\begin{proof}
Fix $r=r(x_0)$ and sections $F_1,\ldots,F_r$ of $\Ii_X$ in a neighborhood $U$ of $x_0$ such that their Jacobian matrix has rank $r$ at each point of $X\cap U$. The set $M=\{F_1=0,\ldots,F_r=0\}\subset U$ is by the Implicit Function Theorem a complex manifold of dimension $n-r$, which in addition contains $X\cap U$. 
As for the converse
we show that there exists a perhaps smaller neighborhood $W\subset U$ of $x_0$ such that $X\cap W=M\cap W$, or equivalently, $X_{x_0}=M_{x_0}$. To that end, it is enough to show that each function germ $H\in\Ii_{X,x_0}$ vanishes on $M_{x_0}$. 

As the holomorphic functions $F_1,\ldots,F_r$ have Jacobian matrix of rank $r$, we can complete this collection with linear functions $L_{r+1},\ldots,L_n$ depending on the variables $(\z_1,\ldots,\z_n)$ in such a way that 
\begin{align*}
\y_1=&\,F_1(\z_1,\ldots,\z_n)\\ 
&\vdots\\ 
\y_r=&\,F_r(\z_1,\ldots,\z_n)\\
\y_{r+1}=&\,L_{r+1}(\z_1,\ldots,\z_n)\\ 
&\vdots\\ 
\y_n=&\,L_n(\z_1,\ldots,\z_n)
\end{align*}
provide a holomorphic system of coordinate on an open neighborhood $V\subset U$ of $x_0$, that maps $M\cap V$ onto an open subset of the linear subspace $\{\y_1=0,\ldots,\y_r=0\}$. We may assume that $M\cap V$ is connected. In this way, we can define
$$
\varphi:V\to V'=\varphi(V),\ z=(z_1,\ldots,z_n)\mapsto(F_1(z),\ldots,F_r(z),L_{r+1}(z),\ldots,L_n(z))
$$
and consider its inverse $\psi=(\psi_1,\ldots,\psi_n):V'\to V$. Observe that 
$$
\psi(\{\y_1=0,\ldots,\y_r=0\}\cap V')=M\cap V.
$$
In addition, $F_i'=F_i\circ\psi=\y_i$ for $i=1,\ldots,r$. To prove that $H_{|M\cap V}$ is identically zero, we show that the restriction of $H'=H\circ\psi$ to $\{\y_1=0,\ldots,\y_r=0\}\cap V'$, which is a holomorphic function of the last $n-r$ coordinates, is identically zero. Write $y_0=\varphi(x_0)$. As $\{\y_1=0,\ldots,\y_r=0\}\cap V'$ is connected, we have to show 
$$
\displaystyle\frac{\partial^\alpha H'}{\partial\y_{r+1}^{\alpha_1}\ldots \partial \y_n^{\alpha_n}}(y_0)=0\quad \mbox{\rm for each multi-index}\quad
\alpha=(\alpha_1,\ldots \alpha_{n-r})\in(\N)^{n-r}.$$ 

For $i=r+1,\ldots,n$ one has
\begin{equation}\label{uno}
\det\begin{vmatrix} 
\frac{\partial F'_1}{\partial\y_1}&\cdots&\frac{\partial F'_1}{\partial\y_r}&\frac{\partial F'_1}{\partial\y_i}\\
\vdots&\ddots&\vdots \\ 
\frac{\partial F'_r}{\partial\y_1}&\cdots&\frac{\partial F'_r}{\partial\y_r}&\frac{\partial F'_r}{\partial\y_i}\\
\frac{\partial H'}{\partial\y_1}&\cdots&\frac{\partial H'}{\partial\y_r}&\frac{\partial H'}{\partial\y_i}\\
\end{vmatrix}
=\det \begin{vmatrix} 
1&0&\cdots&0&0 \\ 
0&1&\cdots&0&0 \\
\vdots&&\ddots&\vdots&\vdots \\
0&0&\cdots&1&0\\
\frac{\partial H'}{\partial\y_1}&\frac{\partial H'}{\partial\y_2}&\cdots&\frac{\partial H'}{\partial\y_r}&\frac{\partial H'}{\partial\y_i}\\
\end{vmatrix}
=\frac{\partial H'}{\partial\y_i}.
\end{equation}
Next, define 
$$
D_s(H)=\det\Big|\frac{\partial(F_1,\ldots,F_r,H)}{\partial(\z_{s_1},\ldots,\z_{s_{r+1}})}\Big|
$$
for each $s=(s_1,\ldots,s_{r+1})\in\{1,\ldots,n\}^{r+1}$ such that $1\leq s_1<\cdots<s_{r+1}\leq n$. As the rank of $\Ii_X$ is $r$, we deduce that $D_s(H)$ vanishes on $\Reg(X)\cap U$. Consequently, as $\Reg(X)$ is dense in $X_{x_0}$, we have $D_s(H)\in\Ii_{X,x_0}$. Thus, $D_s(H)\circ\psi\in\Ii_{X',y_0}$ where $X'=\varphi(X\cap V)$ and $\Ii_X'$ is the sheaf of ideals on $V$ consisting of all holomorphic function germs vanishing on $X'\cap V$.

We have
$$ 
\frac{\partial F'_j}{\partial\y_\ell}=\sum_{k=1}^n\Big(\frac{\partial F_j}{\partial\z_k}\circ\psi\Big)\cdot\frac{\partial\psi_k}{\partial\y_\ell}\quad\text{and}\quad\frac{\partial H'}{\partial\y_\ell}=\sum_{k=1}^n\Big(\frac{\partial H}{\partial\z_k}\circ\psi\Big)\cdot\frac{\partial\psi_k}{\partial\y_\ell}.
$$
for $j=1,\ldots,r$ and $\ell=1,\ldots,n$. Thus, for $i=r+1,\ldots,n$ we have
\begin{equation}\label{due}
\begin{pmatrix} 
\frac{\partial F'_1}{\partial\y_1}&\cdots&\frac{\partial F'_1}{\partial\y_r}&\frac{\partial F'_1}{\partial\y_i}\\
\vdots&\ddots&\vdots \\ 
\frac{\partial F'_r}{\partial\y_1}&\cdots&\frac{\partial F'_r}{\partial\y_r}&\frac{\partial F'_r}{\partial\y_i}\\
\frac{\partial H'}{\partial\y_1}&\cdots&\frac{\partial H'}{\partial\y_r}&\frac{\partial H'}{\partial\y_i}\\
\end{pmatrix}
=\begin{pmatrix} 
\frac{\partial F_1}{\partial\z_1}&\cdots&\frac{\partial F_1}{\partial\z_n}\\
\vdots&\ddots&\vdots \\ 
\frac{\partial F_r}{\partial\z_1}&\cdots&\frac{\partial F_r}{\partial\z_n}\\
\frac{\partial H}{\partial\z_1}&\cdots&\frac{\partial H}{\partial\z_n}\\
\end{pmatrix}
\begin{pmatrix} 
\frac{\partial\psi_1}{\partial\y_1}&\cdots&\frac{\partial\psi_1}{\partial\y_r}&\frac{\partial\psi_1}{\partial\y_i}\\
\vdots&\ddots&\vdots&\vdots\\ 
\frac{\partial\psi_n}{\partial\y_1}&\cdots&\frac{\partial\psi_n}{\partial\y_r}&\frac{\partial\psi_n}{\partial\y_i}.
\end{pmatrix}
\end{equation}
Hence, using \eqref{uno}, \eqref{due} and Binet-Cauchy formula for the determinant of the product of two rectangular matrices of transposed shapes, we deduce
$$
\frac{\partial H'}{\partial\y_i}=\sum_{\substack{s=(s_1,\ldots,s_{r+1})\\
s_1<s_2<\ldots <s_{r+1}}}(D_s(H)\circ\psi)\det\Big|\frac{
\partial(\psi_{s_1},\ldots,\psi_{s_{r+1}}) }{\partial(\y_1,\ldots,\y_r,\y_i)}\Big|\in \Ii_{X',y_0} 
$$
This implies that if $H\in\Ii_{X,x_0}$, then $\displaystyle \frac{\partial H'}{\partial\y_i}\in\Ii_{X',y_0}$ for $i=r+1,\ldots,n$. Consequently, $\displaystyle \frac{\partial H'}{\partial\y_i}\circ\varphi\in\Ii_{X,x_0}$. Then we can apply recursively the same trick to $\displaystyle \frac{\partial H'}{\partial\y_i}\circ\varphi$ for $i=r+1,\ldots,n$, to deduce
$$
\frac{\partial^2H'}{\partial\y_i\partial\y_s}\in\Ii_{X',y_0}
$$
for $r+1\leq i,s\leq n$, and so on. Thus, all derivatives of $H'$ of all orders vanish at $y_0$, so $H'$ is the zero function on $M\cap V$ and we conclude $X_{x_0}=M_{x_0}$, as required.
\end{proof} 
\begin{remarks}
Let $N$ be a connected component of $\Reg(X)$. 

(i)  If $x,y\in N$, then $r(x)=r(y)$. Indeed,
by Proposition \ref{localmanifold} the number $r(x)$ is locally constant in $\Reg(X)$. Hence it is constant on the connected components of $\Reg(X)$. 



(ii)  $N$ is a connected complex analytic manifold of dimension $n-r(x_0)$, where $x_0$ is any of the points of $N$.

(iii) The closure in $X$ of a connected component of  $\Reg(X)$ is an irreducible subset of $X$. This can be proved by the same argument used in R\u ckert's Nullstellensatz.
\end{remarks}

\subsubsection{Zariski's tangent space.}
We approach now regular points from another point of view. This requires to introduce the Zariski's tangent space. Let $(X,\an_X)$  denote a reduced complex analytic space.

\begin{defn}
Let $x\in X$ be a point and $F_1,\ldots,F_k$ be generators of the ideal $\Ii_{X,x}$.  The {\em Zariski's tangent space}\index{Zariski's tangent space}  $T_xX$ of $X$ at $x$ is defined by 
$$
T_xX=\ker(J(F_1,\ldots,F_k)(x)) 
$$
where $J(F_1,\ldots,F_k)(x)$ is the Jacobian matrix of $F_1,\ldots, F_k$ at the point $x$.
\end{defn}

By definition $\dim(T_xX)=n-r_x$ and it has minimal dimension when $x\in\Reg(X)$, that is, when $J(F_1,\ldots,F_k)(x)$ has rank $r(x)$.
 
\begin{lem} 
Let $\mm_x$ be the maximal ideal of $\Oo_{X,x}$ and recall that $\Oo_{X,x}/\mm_x\cong\C$. Then the $\Oo_{X,x}/\mm_x$-linear space $\mm_x/\mm_x^2$ is isomorphic to $T_xX$ as $\C$-linear spaces.
\end{lem}
\begin{proof} 
Denote $\Mm_x$ the maximal ideal of $\Oo_{\C^n,x}$. Observe that
$$
\mm_x/\mm_x^2=(\Mm_x/\Ii_{X,x})\Big/((\Mm_x^2+\Ii_{X,x})/\Ii_{X,x})\cong\Mm_x/(\Mm_x^2+\Ii_{X,x}).
$$

Let ${\C^n}^*$ be the dual linear space of $\C^n$ and consider the linear map 
$$
L:\Mm_x\to{\C^n}^*,\ f\mapsto L(f)=\sum_{i=1}^n\frac{\partial f}{\partial\z_i}(x)u_i
.$$ 
Observe that $L$ is surjective and $\ker(L)=\Mm_x^2$. Consequently, $L$ induces an isomorphism $[L]:\Mm_x/\Mm_x^2\to{\C^n}^*,\ F+\Mm_x^2\mapsto L(F)$. 

Let $F_1,\ldots,F_k$ be a system of generators of $\Ii_{X,x}$. Then we get
$$
T_xX=\ker(L(F_1))\cap\cdots\cap\ker(L(F_k))
$$ 
and we can identify the dual space $T_xX^*$ with the quotient ${\C^n}^*/\qq{L(F_1),\ldots,L(F_k)}$, where $\qq{L(F_1),\ldots,L(F_k)}$ denotes the subspace spanned by $L(F_1),\ldots,L(F_k)$.

Indeed, consider the linear map
$$
\Gamma:{\C^n}^*\to T_xX^*,\ H\mapsto H_{|{T_xX}}.
$$
As each linear form $G:T_xX\to\C$ is the restriction of a linear form $H:\C^n\to\C$, the previous linear map is surjective. Consequently, $T_xX^*\cong{\C^n}^*/\ker(\Gamma)$. As $T_xX=\ker(L(F_1))\cap\cdots\cap\ker(L(F_k))$, we conclude that $\ker(\Gamma)=\qq{L(F_1),\ldots,L(F_k)}$.

We have the following commutative diagram
$$
\xymatrix{
\Mm_x\ar@{->>}[r]^L\ar@{->>}[dr]&{\C^n}^*\ar@{->>}[r]^{\Gamma}\ar@{->>}[dr]&T_xX^*\\
&\Mm_x/\Mm_x^2\ar@{<->}[u]_{[L]}&\ar@{<->}[u]_{[\Gamma]}{\C^n}^*/\qq{L(F_1),\ldots,L(F_k)}
}
.$$
So,
$$
L^{-1}(\qq{L(F_1),\ldots,L(F_k)})=\qq{F_1,\ldots,F_k}+\ker(L)=\Ii_{X,x}+\Mm_x^2,
$$ 
and we conclude
$$
\mm_x/\mm_x^2\cong\Mm_x/(\Mm_x^2+\Ii_{X,x})\cong\Mm_x/\ker(L)\cong{\C^n}^*/\qq{L(F_1),\ldots,L(F_k)}\cong T_xX^*,
$$
as required.
\end{proof}

As a consequence we get a characterization of a regular point $x\in X$ in terms of the algebraic properties of the ring $\Oo_{X,x}$. Recall that a local noetherian ring $(A,\mm)$ is called {\em regular}\index{regular ring} if $\dim(A) =\dim_\kappa(\mm/\mm^2)$,\footnote{The \em dimension $\dim(A)$ of a ring $A$ \em is the supremum over the lengths $\ell$ of each chain $\gtp_0\subsetneq\cdots\subsetneq\gtp_\ell$ of prime ideals of $A$.} where $\kappa=A/\mm$ is the \em residue field of $A$\em. 

Consider the local noetherian ring $A=\Oo_{\C^n,x}/\Ii_{X,x}$ and assume that $\Ii_{X,x}$ is a prime ideal. We know that, up to a linear change of coordinates, $A$ is an integral extension of the local ring $\Oo_d$ of holomorphic germs in $d$ variables. Thus, $\dim(A)=d$. 
\begin{lem}
Under the hypothesis above one has $d=n-r(x)$ and there exists an open neighborhood $X'\subset X$ of $x$ such that $r(y)=r(x)$ for each $y\in X'$. 

In particular, $d=n-r(x)\leq n-r_x=\dim(\mm_x/\mm_x^2)$.
\end{lem}
\begin{proof}
The germ $X_x$ has, after a linear change of coordinates, a representative $X'=X\cap\Delta(x,\veps)$ in a polydisc $\Delta(x,\veps)$ centered at $x$ and an open subset $M\subset X'$, which is connected and dense in $X$ such that the projection $\pi:\C^n\to\C^d$ onto the first $d$ coordinates induces a covering from $M$ to an open subset of $\C^d$. Thus, $\dim(M)=d$ and  $M$ is defined by representatives of $n-d$ elements in $\Ii_{X,x}$ whose Jacobian matrix has rank $n-d$ at each point of $M$. Observe that $M\subset\Reg(X)$ because at each point $z\in M$ one has $r_z=n-d$. Pick a point $y\in X'\setminus M$. As $M$ is dense in $X'$, there exists $z\in M$ close to $y$, so $r_y\leq r_z=n-d$. Thus, $r(x)=n-d$ because for each point $y\in X$ close to $x$, we have $r_y\leq n-d$ and as $M$ is dense in $X'$ there exists points $z\in M$ close to $x$ and at these points $r_z=n-d$. We have already proved in addition that if $y\in X'$, then $r(y)=r(x)$, as required.
\end{proof}

\begin{thm}
Let $(X,\Oo_X)$ be a reduced complex analytic space. Then a point $x\in X$ is regular if and only if the ring $\Oo_{X,x}$ is a local regular ring.
\end{thm}
\begin{proof}
We may assume that $(X,\Oo_X)$ is a reduced local model. If $x$ is regular, then $r_x=r(x)$, so $\dim(T_xX)=n-r_x=n-r(x)=\dim(\Oo_{X,x})$, that is, the ring $\Oo_{X,x}$ is regular. Conversely, if $\Oo_{X,x}$ is regular, then $n-r(x)=\dim(\Oo_{X,x})=\dim_\C(\mm_x/\mm_x^2)=\dim(T_xX)=n-r_x$, so $r_x=r(x)$ and the point $x$ is regular. 
\end{proof}

\begin{cor} 
The subset $\Sing(X)$ of a reduced complex ana\-lytic spa\-ce 
$(X,\Oo_X)$ is a complex ana\-lytic sub\-set of $X$.
\end{cor}
\begin{proof}
Pick a point $x\in X$. We distinguish two cases:
\begin{enumerate}
\item Suppose first that the germ $X_x$ is irreducible and denote $r=r(x)$. There exist an open neighborhood $U$ and sections $F_1,\ldots,F_r$ of the sheaf $\Ii_X$ defined on the open set $U$ such that they generate $\Ii_{X,y}$ for each $y\in U$. After shrinking $U$, we may assume that $r(y)=r(x)$ for each point $y\in U$. Then, $y\in U$ is a singular point if and only if $r_y<r(x)=r$. Thus, $\Sing(X)\cap U$ is the set of the points $y\in X$ at which all minors of order $r$ of the Jacobian matrix
$$
\Big(\frac{\partial F_i}{\partial\x_j}(x)\Big)_{\substack{1\leq i\leq r\\1\leq j\leq n}}
$$
are zero. Thus, $\Sing(X)\cap U$ is a complex analytic subset of $U$.

\item Suppose next that the germ $X_x$ is reducible. Then, there are an open neighborhood $U$ and complex analytic subsets $X_1,\ldots,X_r$ of $U$ such that $X\cap U=X_1\cup\cdots \cup X_r$, each germ $X_{i,x}$ is irreducible and $X_i$ contains a connected and dense complex analytic manifold of the same dimension as $X_i$. Using the fact the ring $\Oo_{X,y}$ is not an integral domain if the germ $X_y$ is reducible, we conclude that 
$$
\Sing(X\cap U)=\bigcup_{i=1}^r\Sing(X_i)\cup\bigcup_{i\neq j}(X_i\cap X_j).
$$
Using  (1) and the fact that each $X_i$ has a finite system of holomorphic equations in $U$ (shrinking $U$ if necessary), we conclude that $\Sing(X\cap U)$ is a  complex analytic subset of $U$, as required.
\end{enumerate} 
\end{proof}

 For 
an arbitrary complex analytic space $(X,\an_X)$ not necessarily reduced we say that a point $x\in X$ is \em regular \em if the ring $\an_{X,x}$ is a local regular ring. Otherwise, we will say that the point $x$ is a \em singular point of $X$\em. Again $\Reg(X)$ denotes the set of regular points of $X$, while $\Sing(X)$ denotes the set of singular points of $X$. 

\subsection{Stein spaces.}

Let $(X,\Oo_X)$ be a complex analytic space. Denote $\Oo(X)$ the algebra of its holomorphic functions. This algebra can be very small. For instance, if $X$ is compact as the projective space $\PP^n(\C)$, Maximum Principle shows that $\Oo(X)$ reduces to the set $\C$ of constant functions. Conversely, if $X$ is a closed analytic subset of $\C^n$, it has a lot of holomorphic functions. We give now a list of desirable properties that analytic subsets of $\C^n$ get. The first one is {\em to provide local coordinates}.

 
\begin{defn}
Let $(X,\Oo_X)$ be a reduced complex analytic space and let $x\in X$. We say that finitely many holomorphic functions $F_1,\ldots,F_n$ on an open neighborhood $U\subset X$ of $x$ {\em provide local coordinates} if they define a closed embedding $F=(F_1,\ldots,F_n):U\to\Omega\subset\C^n$, where $\Omega$ is an open subset. Observe that $F$ induces an isomorphism between $(U,\Oo_{|U})$ and a local model $(Y=F(U),\Oo_\Omega/\Ii_Y)$ in $\Omega$. 
\end{defn}

Here is the announced list of properties of a closed analytic subset  $X\subset \C^n$.

\begin{itemize}

\item[(i)] For any unbounded sequence of points $\{x_m\}_m$ in $X$ there exists a holomorphic function $f$ on $X$ such that $\displaystyle \lim_{n\to \infty}|f(x_m)|=\infty$. \footnote{For unbounded sequence we mean an infinite sequence such that it intersects all compact sets in a finite number of points}
\item[(ii)] Holomorphic functions on $X$ separate points and provide local coordinates at each point $x\in X$.
\item[(iii)] $X$ is not compact unless $X$ is finite. 
\end{itemize} 

These three properties characterize a larger class of analytic spaces. 

\begin{defn}\label{Stein}
A complex analytic space $(X,\Oo_X)$ is a {\em Stein space}\index{Stein space} if it satisfies conditions (i), (ii), (iii) above. 
\end{defn}

Among the most important results concerning Stein spaces we point out Cartan's Theorems A and B.

\begin{thm}\index{Theorems A and B}
Let $(X,\Oo_X)$ be a Stein space. Then each coherent sheaf $\Ff$ of $\Oo_X$-modules on $X$ satisfies the following properties.
\begin{itemize}
\item[(A)] Each fiber $\Ff_x$ is generated by global sections of $\Ff$. 
\item[(B)] $H^q(X,\Ff)=0$ for each $q>0$.
\end{itemize}
\end{thm}

In particular, when $(X,\Oo_X)$ is a complex analytic subspace of $\C^n$, one gets the exact
sequence 
$$ 
0\to\Ii_X\to\Oo_{\C^n}\to\Oo_X\to0,
$$
where each involved sheaf is by Oka's theorem a coherent $\Oo_{\C^n}$-module. This implies using Cartan's Theorem B that each holomorphic function on $X$ is the restriction to $X$ of a holomorphic function on $\C^n$.

Next we give a characterization of open Stein subsets of $\C^n$.

\begin{thm}\label{steinopen}{\rm(Characterization of open Stein sets)}
Let $\Omega$ be a connected open subset of $C^n$. The following are equivalent.
\begin{itemize}
\item[(i)] $\Omega$ is a Stein manifold.
\item[(ii)] $\Omega$ is holomorphically convex.
\item[(iii)] $\Omega$ is a holomorphy domain.\footnote{A connected open set $U\subset \C^n$ is {\em holomorphically convex} if the holomorphic envelop of a compact subset of $U$ is compact. It is a {\em holomorphy domain} if there do not exist non-empty open sets $\Omega \subset U$ and $V \subset \C^n$ connected and not included in $U$  such that $\Omega \subset U\cap V$ and there  are holomorphic function $f\in \Oo(U), g\in \Oo(V)$ such that the restriction of $g$ to $\Omega$ coincides with the restriction of $f$ to $\Omega$. Roughly speaking a holomorhy domain is a set which is maximal in the sense that there exists a holomorphic function on this set which cannot be extended to a bigger set.} 
\end{itemize}
\end{thm}

A very relevant example of Stein open subset of $\C^n$ is a polydisc. Note that one can choose local models as subspaces of a polydisc. As a consequence of Theorem B, a  closed subspace of a Stein space is also Stein, so 
we get that any complex analytic space is locally Stein. In particular closed subspaces of $\C^n$ are Stein spaces. The converse is almost true: any Stein space can be embedded in $\C^n$ as a closed analytic subspace for some $n$ large enough under the additional hypothesis that $\sup_{x\in X}\dim(T_{x}X)<\infty$. More precisely

\begin{thm}{\rm (Narasimhan)}\label{immersione}
Any Stein space $(X,\Oo_X)$ of dimension $n$ admits a one to one proper holomorphic map into $\C^{2n+1}$, that is, a holomorphic embedding at each regular point of $X$. Assume in addition that for each point $x\in X$ there exists an open neighborhood in $X$ that can be holomorphically embedded as a closed analytic subset of an open subset of $\C^N$ (with the analytic structure induced by $\C^N$), where $N>n$ is fixed. Then, there exists a one to one proper map $\varphi:X\to\C^{N+n}$ whose image (with the induced analytic structure provided by $\C^{N+n}$) is isomorphic to $X$ by means of $\varphi$.

In both cases above, the set of embeddings is dense in the space of holomorphic maps into $\C^m$ (where $m=2n+1$ in the first case and $m=N+n$ in the second case) if we endow such space with the compact-open topology.
\end{thm}

As an application of Theorem B we recall an argument due to Grauert. 

\begin{prop}\label{globalequations} A Stein subspace of a Stein space has finitely many global holomorphic equations. More precisely, if $n$ is the dimension of the ambient Stein space, it is the zeroset of at most $n+1$ global holomorphic equations.
\end{prop}
\begin{proof}
If $Y\subset X$ is a closed subspace of a Stein space $X$, we take first a non-identically zero holomorphic function $F_1$ on $X$ vanishing identically on $Y$, so $\dim(\{F_1=0\})=\dim(X)-1$. It exists by Theorem B, indeed as $Y$ is a Stein subspace of the Stein space $(X,\Oo_X)$, pick a point $p\in X\setminus Y$, as $Y \cup\{p\}$  is a Stein subspace of the Stein space $(X,\Oo_X)$, we can  consider a holomorphic function $F_1$ on $X$ that takes values $0$ on $Y$ and $1$ on $p$. If the zeroset of $F_1$ is $Y$, we are done. Otherwise, we decompose the zeroset $\{F_1=0\}$ as the union of its irreducible components. We pick a point $p_Z$ in each irreducible component $Z$ of $\{F_1=0\}$ that does not lie inside $Y$. Then, there exists a holomorphic function $F_2$ on $X$ that vanishes identically on $Y$ and takes the value $1$ at each point $p_Z$. Now, the common zeroset of $F_1,F_2$ outside $Y$ has strictly smaller dimension than the dimension of $\{F_1=0\}\setminus Y$. We repeat the same trick until we find holomorphic functions $F_3,\ldots,F_k$ such that $Y\subset\{F_1=0,\ldots,F_k=0\}$ and 
$$
\dim(\{F_1=0,\ldots,F_k=0\}\setminus Y)\leq 0. 
$$
Observe that $\{F_1=0,\ldots,F_k=0\}=Y\cup D$, where $D$ is a (may be empty) discrete set. If $F_{k+1}$ is holomorphic on $X$, vanishes identically on $Y$ and takes value $1$ at each isolated point of the discrete set $D$, we describe $Y$ as the common zeroset of $F_1,\ldots,F_{k+1}$. By construction $k+1 \leq n+1$.
\end{proof}

\subsubsection{Theorems A and B and direct limits.}

Next result plays a fundamental role in this framework, because it implies Theorems A and B for a large class of real analytic spaces, that will be introduced later. 

\begin{thm}\label{AandBreal} 
Let $Z$ be a closed subset of a complex analytic space $(X,\Oo_X)$ and define $\Oo_Z=\Oo_{|Z}$. Suppose that $Z$ has a fundamental system of open Stein neighborhoods in $X$. Then  Theorems A and B hold for $Z$, that is, for each $\Oo_Z$-coherent sheaf of modules $\Ff$ on $Z$ we have:
\begin{itemize}
\item[(A)] $\Ff_x$ is generated (as $\Oo_{Z,x}$-module) by the image of natural map $H^0(Z,\Ff)\to\Ff_x$ for each $x\in Z$, that is, $\Ff_x$ is generated by global sections of $\Ff$.
\item[(B)] $H^q(Z,\Ff)=0$ for each $q>0$. 
\end{itemize}
\end{thm}

In what follows, given a complex analytic space $(X,\Oo_X)$ and a closed subset $C\subset X$, the sheaves of $\Oo_{|C}$-modules will be called \em analytic sheaves on $C$\em\index{analytic sheaves}. First of all we need the following proposition.

\begin{prop}\label{tripla} 
Let $C$ be a closed subset of a complex analytic space $(X,\Oo_X)$. Let $\Gg$ be a coherent analytic sheaf on $C$. Then there exists a triple $(U,\Ff,\varphi)$, where $U$ is an open neighborhood of $C$, $\Ff$ is a coherent analytic sheaf on $U$ and $\varphi$ is an analytic sheaf isomorphism between $\Gg$ and the induced sheaf $\Ff_{|C}$.

Moreover this triple is unique up to isomorphism in the following sense: if $(U',\Ff',\varphi')$ is another triple satisfying the previous properties then there exists on open neighborhood $U''\subset U\cap U'$ of $C$ and an sheaf isomorphism of $\Oo_C$-modules between $\Ff_{|U''}$ and $\Ff'_{|U''}$.
\end{prop}
\begin{proof}

\noindent{\sc Step 1.} {\em Let $\Ff$, $\Ff'$ be coherent analytic sheaves on $C$. If $\phi,\psi:\Ff\to\Ff'$ are analytic sheaf homomorphisms, the set $\{x\in C:\ \phi_x=\psi_x\}$ is an open subset of $C$.}

Pick a point $x\in C$ and assume $\phi_x=\psi_x$. Let $F_1,\ldots,F_s$ be sections of $\Ff$ that generate $\Ff_y$ for all $y$ in a neighborhood $V$ of $x$. Then $G_i =\phi(F_i)$ is a section of $\Ff'$ over a neighborhood of $x$. Denote $G'_i=\psi(F_i)$ and observe that $G_{i,x}=G'_{i,x}$, so $G_i=G'_i$ on a neighborhood of $x$ in $C$. Consequently, $\phi=\psi$ on an open neighborhood of $x$ in $C$. 

\noindent{\sc Step 2.} {\em Let $\Ff$, $\Ff'$ be coherent analytic sheaves on $X$ and $\phi:\Ff_{|C}\to\Ff'_{|C}$ be an analytic sheaf homomorphism. Then there exists a neighborhood $U$ of $C$ and an analytic sheaf homomorphism $\Phi:\Ff_{|U}\to\Ff'_{|U}$ extending $\phi$.}

Take $x\in C$ and let $F_1,\ldots,F_s$ be sections of $\Ff$ that generate $\Ff_y$ for all $y$ in a neighborhood of $x$ in $X$. As $\Ff$ is coherent, the sheaf of relations among $F_1,\ldots,F_s$ is also finitely generated by some holomorphic section $B_{ik}$ for $k=1,\ldots,\ell$, which are defined and satisfy $\sum_{i=1}^sB_{ik}F_i=0$ (for each $k=1,\ldots,\ell$) on a neighborhood of $x$. Observe that if $\sum_{i=1}^sB_iF_i=0$ where $B_i$ is holomorphic in a neighborhood of $y$ close to $x$, then $B_i=\sum_{k=1}^\ell A_kB_{ik}$, where $A_k$ is holomorphic on a neighborhood of $y$, because the relations among $F_1,\ldots,F_s$ at $x$ generate the relations among $F_1,\ldots,F_s$ at $y$ for $y$ close to $x$. Now, $\phi_x(F_i)=\sum_{j=1}^rE_{ij}G_j$ where $G_1,\ldots,G_r$ are sections that generate $\Ff'$ on a neighborhood of $x$ in $X$ and each $E_{ji}$ is holomorphic on a neighborhood of $x$. In addition, for each $k=1,\ldots,\ell$, we have
\begin{equation}\label{three}
0=\sum_{i=1}^sB_{ik}\phi_x(F_i)=\sum_{i=1}^sB_{ik}\sum_{j=1}^rE_{ij}G_j=\sum_{i=1}^s\sum_{j=1}^rB_{ik}E_{ij}G_j
\end{equation} 
and the previous equality holds on a neighborhood of $x$. 

We claim: \em the formula $\Phi_y(F_i)=\sum_{j=1}^rE_{ij}G_j$ provides a well-defined extension $\Phi$ of $\phi_x$ to a neighborhood $U^x$ of $x$ in $X$\em. 

Indeed, assume that $\sum_{i=1}^sB_iF_i=0$ where $B_i$ is holomorphic in a neighborhood of $y$ close to $x$. We have seen that $B_i=\sum_{k=1}^\ell A_kB_{ik}$, where $A_k$ is holomorphic on a neighborhood of $y$. Thus, 
$$
\sum_{i=1}^sB_iF_i=\sum_{i=1}^s\sum_{k=1}^\ell A_kB_{ik}F_i=\sum_{k=1}^\ell A_k\Big(\sum_{i=1}^sB_{ik}F_i\Big)
$$
and we deduce using also \eqref{three} that
$$
\phi_y\Big(\sum_{i=1}^sB_iF_i\Big)=\sum_{k=1}^\ell A_k\Big(\sum_{i=1}^sB_{ik}\Phi_y(F_i)\Big)=\sum_{k=1}^\ell A_k\Big(\sum_{i=1}^sB_{ik}\sum_{j=1}^rE_{ij}G_j\Big)=0.
$$
Consequently, $\Phi$ is well-defined.

Consider next the open covering $\Rr=\{\{U^x\}_{x\in C}, \ X\setminus C\}$ of the paracompact space $X$. Then, $\Rr$ has locally finite open refinements $\{U_i\}_{i\in I}$ and $\{V_i\}_{i\in I}$ such that $\overline{V_i}\subset U_i$. Let $T$ be the union of all $\overline{V_i}$ intersecting $C$, which is a closed neighborhood of $C$ in $X$. In addition, if $U_i\cap C\neq\varnothing$, there exists an analytic homomorphism $\Phi_i:\Ff_{|U_i}\to\Ff'_{|U_i}$ extending $\phi_{|U_i\cap C}$. If $x\in C\cap U_i\cap U_j$, we have $\Phi_{i,x}=\phi_x=\Phi_{j,x}$. For each $x\in T$ there exists a finite set of indices $I(x)$ such that $x\in\overline{V_i}$ if and only if $i\in I(x)$, and if $y$ close to $x$, then $I(y)\subset I(x)$.  

By {\sc Step 1} the set $V=\{x\in T: \Phi_{i,x}=\Phi_{j,x},\ \forall i,j \in I(x)\}$ is an open subset of $T$ that contains $C$. For $x\in V$ we can define $\Phi_x$ as $\Phi_{i,x}$ if $i\in I(x)$. Note that the local sheaf homomorphims $\Phi_i$ glue together to provide a sheaf homomorphism $\Phi:\Ff_{|V}\to \Ff'_{|V}$that extends $\phi$. 

\noindent{\sc Step 3.} {\em Existence of the extension.} 
Let $x\in C$. By definition of a coherent sheaf, there exist an open neighborhood $U^x\subset X$ of $x$, a coherent sheaf $\Ff$ on $U^x$ and an analytic isomorphism between $\Gg_{|C\cap U^x}$ and the sheaf $\Ff_{|C\cap U^x}$. Again the family $\Rr=\{X\setminus C,\{U^x\}_{x\in C}\}$ defines an open covering of $X$ and there exist two locally finite open coverings $\{U_i\}_{i\in I}$ and $\{V_i\}_{\in I}$ refining $\Rr$ such that for each $\overline{V_i}\subset U_i$ for each $i\in I$. If $U_i$ meets $C$, we have a coherent sheaf $\Ff_i$ on $U_i$ and an analytic isomorphism $\varphi_i$ between the sheaf $\Gg_{|C\cap U_i}$ and ${\Ff_i}_{|C\cap U_i}$. For each couple $(i,j)$ such that $C\cap U_i\cap U_j\neq\emptyset$ the compositions $\varphi_i\circ\varphi_j^{-1}$ is an isomorphism $\varphi_{ij}$ between ${\Ff_j}_{|C\cap U_i\cap U_j}$ and ${\Ff_i}_{|C\cap U_i\cap U_j}$. In addition, if $C\cap U_i\cap U_j\cap U_k\neq\varnothing$, we have $\varphi_{ij}\circ\varphi_{jk}=\varphi_{ik}$.

By {\sc Step 2} applied to the space $U_i\cap U_j$ and to the closed subset $C\cap U_i\cap U_j$, there exist an open set $U_{ij}$, containing $C\cap U_i\cap U_j$ and contained in $ U_i\cap U_j$, and an analytic homomorphism $\Phi_{ij}:{\Ff_j}_{|U_{ij}}\to{\Ff_i}_{|U_{ij}}$ that extends $\varphi_{ij}$. If $i=j$ we write $U_{ii}=U_i$ and we take $\Phi_{ii}$ as the identity. 

We claim that there exists an open set $W$ containing $C$ such that for each couple $(i,j)$ we get:
$$ 
W\cap\overline{V_i}\cap\overline{V_j}\subset U_{ij}.
$$  

Indeed, each point $x\in C$ has an open neighbourhood $W^x$ which meets only a
finite number of $\overline{V_i}$. For each couple $(i,j)$ such that $\overline{V_i}$ and $\overline{V_j}$ meet $W^x$, the intersection $U_{ij}$ is an open set that contains $x$. The intersection 
$$
W'^x=W^x\cap\bigcap_{(i,j)}U_{ij}
$$ 
is an open neighbourhood of $x$. The open set $W$ we seek is $W=\bigcup_{x\in C}W'^x$.

Denote $V$ the set of those points $y\in W$ satisfying
$$ 
\Phi_{ij,y}\circ\Phi_{jk,y}=\Phi_{ik,y}
$$
for each triple $i,j,k$ such that $y\in\overline{V_i}\cap\overline{V_j}\cap\overline{V_k}$. The set $V$ contains $C$ and is open by {\sc Step 1}. 

Finally, let $\displaystyle U=\bigcup_{V_i\cap C\neq\varnothing}V\cap V_i$, which is open and contains $C$. We define the sheaf $\Ff$ on $U$ by ``glueing'' the sheaves already constructed. Observe that: 
\begin{itemize}
\item On $U\cap V_i$ we take the sheaf ${\Ff_i}_{|U\cap V_i}$.
\item On $U\cap V_i\cap V_j\subset U_{ij}$ (recall that $U\subset W$) we have the analytic homomorphism ${\Phi_{ij}}_{|U\cap V_i\cap V_j}:{\Ff_j}_{|U\cap V_i\cap V_j}\to{\Ff_i}_{|U\cap V_i\cap V_j}$.
\item On $U\cap V_i\cap V_j\cap V_k$ we have the relation $\Phi_{ij,y}\circ\Phi_{jk,y}=\Phi_{ik,y}$ because $U\subset V$. Making $i=k$ we deduce that $\Phi_{ij}$ is an isomorphism. 
\end{itemize}

Let $\Ff$ be the sheaf defined by glueing the ${\Ff_i}_{|U\cap V_i}$ via the transition isomorphisms $\Phi_{ij}$. As the sheaves ${\Ff_i}_{|U\cap V_i}$ are coherent, $\Ff$ is coherent too. The sheaf $\Ff_{|C}$ induced by $\Ff$ on $C$ is obtained by glueing the sheaves ${\Ff_i}_{|C\cap V_i}$ via the isomorphisms $\varphi_{ij}$ restricted to $C\cap V_i\cap V_j$, which are precisely the ones induced by the $\Phi_{ij}$ on $C\cap V_i\cap V_j\subset C\cap U_i\cap U_j$. Now, $\varphi_{ij}=\varphi_i\circ\varphi_j^{-1}$, where $\varphi_i:\Gg_{|C\cap V_i}\to{\Ff_i}_{|C\cap V_i}$, is an isomorphism. This implies that the collection $\{\varphi_i\}_{i\in I}$ defines an isomorphism $\varphi$ between the sheaf $\Gg$ and the sheaf $\Ff_{|C}$ induced by $\Ff$ on $C$.

\noindent{\sc Step 4.} {\em Uniqueness of the extension up to isomorphism.}
The proof of the uniqueness of the extension up to isomorphism follows as an easy consequence of {\sc Steps 1} and {\sc 2}.
\end{proof}

{\sc Proof of Theorem \ref{AandBreal}.}
Given a sheaf $\Ff$ of $\Oo_X$-modules on $(X,\Oo_x)$ and a closed subset $Z\subset X$, we denote $\Gamma(Z,\Ff_{|Z})$ the module of global sections of $\Ff_{|Z}$ on $Z$.
Let $\Ff$ be a coherent sheaf on $Z$ and let $\{\Omega_i\}_{i\in I}$ be a fundamental system of Stein neighborhoods of $Z$ in $X$. The coherent sheaf $\Ff$ on $Z$ extends by Proposition \ref{tripla} to a coherent sheaf $\Ff_i$ on $\Omega_i$ for some $i\in I$ and the restriction of $\Ff_i$ to $\Omega_j$ is $\Ff_j$ whenever $\Omega_j\subset \Omega_i$. To prove (A), we fix the Stein neighborhood $\Omega_i$ with the property that ${\Ff_i}_{|Z}=\Ff$. By Cartan's Thereom A, the $\Oo_{X,x}$-module $\Ff_x$ is generated by some sections of $\Ff_i$ on $\Omega_i$. Hence, `a fortiori' by some sections on $Z$.   

Let us prove (B). Each global section belonging to $\Gamma(Z,\Ff)$ can be extended to a section on an open neighborhood of $Z$ in $X$, so it can be extended to a section on some $\Omega_i$. In addition, if two sections of $\Ff_{|\Omega_i}$ coincide on $Z$, they coincide on a Stein neighborhood $\Omega_j\subset\Omega_i$ of $Z$.

If we prove that

$$ H^q(Z,\Ff)=\underrightarrow \lim {}  H^q(\Omega_i,\Ff_i)$$ 
where the direct limit is taken over the open sets $\Omega_i$ containing $Z$, 
we deduce (B) for $Z$ because by Cartan's Theorem B applied to $\Omega_i$, all the groups $H^q(\Omega_i,\Ff_i)$ are $0$ for $q>0$. 

The claim  has been already proved for $q=0$. For $q>0$, we use a {\em fine resolution} of the sheaf $\Ff$, that is, an exact sequence of shaves 
$$ 
0\to\Ff\to\Ff_0\to\Ff_1\to\ldots\to\Ff_q\ldots 
$$
where all the sheaves $\Ff_i$ are {\em fine} sheaves. We refer the reader to \cite{dl} for further details concerning fine sheaves. It is proved in \cite[Thm. A, p. 89]{dl} that there exists a canonical isomorphism between $H^q(X,\Ff)$ and the $q^{th}$ homology group of the complex given by the sequence 
$$ 
0\to\Gamma(X,\Ff_0)\overset{\delta^1}{\to}\Gamma(X,\Ff_1)\overset{\delta^2}{\to}\cdots\overset{\delta^{q-2}}{\to}\Gamma(X,\Ff_{q-2})\overset{\delta^{q-1}}{\to}\Gamma(X,\Ff_{q-1})\overset{\delta^q}{\to}\Gamma(X,\Ff_{q})\overset{\delta^{q+1}}{\to}\cdots
$$
that is, there exists a canonical isomorphism between $H^q(X,\Ff)$ and the quotient of the kernel of the homomorphism $\delta^{q+1}:\Gamma(X,\Ff_{q})\to\Gamma(X,\Ff_{q+1})$ by the image of the homomorphism $\delta^{q}:\Gamma(X,\Ff_{q-1})\to\Gamma(X,\Ff_{q})$. This result is also true for $Z$ and the Stein neighborhoods $\Omega_i$ that contains $Z$.

The complex $\sum_{q\geq0}\Gamma(Z,\Ff_q)$ is the direct limit of the complexes $\sum_{q\geq0}\Gamma(\Omega_i,\Ff_q)$ where $\Omega_i$ belongs to the fundamental system of Stein open neighborhoods of $Z$ (using the case $0$ for each sheaf $\Ff_q$). In addition, to consider the homology of a complex commutes with taking direct limits. Thus, we conclude that $H^q(Z,\Ff)$ is the direct limit of the $H^q (\Omega_i,\Ff)$. 
\qed

\subsubsection{Characterizations of Stein spaces.}
Recall that a $\Cc^2$ function $f:\Omega\subset\C^n\to\R$ is {\em convex} \index{convex}or \em plurisubharmonic \em if it is subharmonic on the intersection of each complex line $L$ with $\Omega$, that is, if the form $i
\partial\overline{\partial}f_{|L\cap \Omega}$ is positive definite, where $\partial,\overline{\partial}$ denote the Dolbeault operators. It holds that $f$ is convex if and only if $i\partial\overline{\partial}f=\omega$ is a K\"ahler form (or K\"ahler potential) on $X$. If $\Omega =\C^n$, then $f(z_1,\ldots,z_n)=\sum_{i=1}^n|z_i|^2$ is an example of convex function. 

We say that $f$ is  {\em strongly convex} \index{strongly convex}or \em strongly plurisubharmonic \em if for each open set $U\subset\Omega$ relatively compact in $\Omega$ there exists $\varepsilon>0$ such that if $h:\Omega\to\R$ is a smooth function and the absolute values of $h$ and its derivatives of order $\leq2$ are smaller than $\varepsilon$, then $f+h$ is convex in $U$. It holds that $f$ is strongly convex if and only if the form $i (\partial\overline{\partial}f-\omega)$ is positive, for some K\"ahler form $\omega$ on $\Omega$. The \index{Levi form}{\em Levi form} of a $\Cc^2$ function $f$ is 
$$
\displaystyle L(f) =\sum_{i,j}\frac{\partial^2 f}{\partial z_i \partial {\overline z}_j} dz_i d\overline z_j.
$$
It is known that $f$ is convex if and only if the quadratic form associated to $L(f)$, which is given by the square symmetric matrix $(a_{ij})$ whose entries are $\displaystyle a_{ij}=\frac{\partial^2 f}{\partial z_i\partial{\overline z}_j}$, is positive definite. Convexity is a local property, that is, a function is convex if and only if it is convex in a neighborhood of each point.

A function $f$ on $\Omega\subset\C^n$ is strongly convex if and only if for each open set $U\subset\Omega$ relatively compact in $\Omega$, the function $f$ can be written in the form $f=f_1+f_2$, where $f_1$ is convex on $U$ and $f_2$ is a smooth function such that $L(f_2)$ is positive semidefinite.

\begin{defn}
A real valued continuous function $f:X\to\R$ on a Stein space $(X,\Oo_X)$ is \em convex \em if for each $x\in X$, there exists a neighborhood $U$ and an analytic diffeomorphism of $\varphi:U\to Z\subset\Delta$ onto a complex analytic set $Z$ in a polydisc $\Delta\subset\C^n$ in which there exist a convex function $\widetilde{f}$ such that $f=\widetilde{f}\circ\varphi$ on $U$. If $\varphi$, $f$ and $\widetilde{f}$ can be chosen so that $\widetilde{f}$ is strongly convex in $\Delta$, then we say that $f$ is \em strongly convex \em on $X$.
\end{defn}

It can be shown that the previous definition is independent of the particular choice of the local embedding. 

There are several characterizations of Stein spaces. In the next theorem we recall two of them that will be useful in the sequel. 

\begin{thm}{\rm (Narasimhan)}\label{Narcrit} 
Let $(X,\Oo_X)$ be a complex analytic space whose topology is second countable. The following assetions are equivalent:
\begin{itemize}
\item[(i)] $X$ is a Stein space.
\item[(ii)] There exists a strongly convex function $f:X\to\R$ such that the set $\{x\in X:\ f(x)<\alpha\}$ is relatively compact in $X$ for each $\alpha\in\R$.
\item[(iii)] There exist a continuous convex function $f:X\to\R$ such that the set $\{x\in X:\ f(x)<\alpha\}$ is relatively compact for each $\alpha\geq 0$, and in addition there exists a continuous strongly convex function $g:X\to\R$.
\end{itemize}
\end{thm}

\bigskip

\section{Real analytic spaces.}\label{ras}

Replacing in Definition \ref{defspaziocomplesso} {\em holomorphic} with {\em
 real analytic}, one gets for the real case the definitions resp. of
{\em reduced real analytic space} and \index{analytic space!reduced real --}
 {\em C-space}\index{analityc space! C-space}. More precisely.
\vfill\eject

\noindent
\begin{defns}\label{defspazioreale}\null \hfill
\begin{itemize} 
\item A {\em reduced real analytic space} \index{analytic space!real --} is a ringed space $(X,\Aa_X)$ locally
isomorphic, as ringed space, to a local model $(Y,\Aa_Y)$ where $Y\subset U\subset \R^n$ is the zeroset of finitely many analytic functions in $\Aa(U)$ and $\Aa_Y$ is the quotient sheaf of the sheaf $\Aa_{U}$ of germs of
analytic functions on $U$ by the ideal sheaf 
$\Ii_Y$ of analytic function germs vanishing on $Y$.
\item A {\em real C-analytic space} is a ringed space $(X,\Aa_X)$ locally isomorphic to a local model $(Y,\Aa_Y)$ where $Y\subset U\subset \R^n$ is the zeroset of finitely many analytic functions $f_1,\ldots,f_s\in \Aa(U)$ and $\Aa_Y$ is the quotient sheaf of the sheaf $\Aa_{U}$ by the ideal sheaf generated by $f_1,\ldots,f_s$ \quad i.e.
 $\Aa_Y=\Aa_U/(f_1,\ldots,f_s)\Aa_U$.
\item When we consider only the set $X$ without any structure, we refer to it as an {\em analytic set}. So a {\em coherent analytic set } is an analytic set whose reduced structure is coherent. A {\em C-analytic set } is an analytic set admitting 
one (possibly several) coherent structure.
\end{itemize}
\end{defns}

\begin{remark} The ideal sheaf $\Ii_Y$ in (1) of Definition \ref{defspazioreale} needs not to be coherent since in the real case Oka's theorem does not hold. (see Examples \ref{Whitneyumbrella}).
There are also analytic sets that are not C-analytic (see Example \ref{a(z)}).

For a real C-analytic space the sheaf $\Aa_X$ is coherent since in all its models the ideal sheaf is a finitely generated subsheaf of $\Aa_U$. 
\end{remark}

We will see in the next section a criterion for a real analytic space to be coherent.

\subsection{Complexification.}

We will suppose always $\R^n$ canonically embedded in $\C^n$ as the fixed point set of the complex conjugation.
In this section we wonder whether a given real analytic space can be viewed as a subspace of a suitable complex space in a similar way.

\subsubsection{The local case.}

\begin{prop} Let $V_x$ be a real analytic set germ at a point $x\in \R^n\subset \C^n$. There exists in $\C^n$ one and only one germ of complex analytic set $\tilde {V}_x$ such that
\begin{enumerate}
\item $V_x\subset \tilde {V}_x$, $\tilde {V}_x\cap \R^n=V_x$. 
\item If a germ at $x$ of holomorphic function vanishes on $V_x$ then it vanishes on $\tilde {V}_x$.
\item $\tilde {V}_x$ is minimal with respect to properties (1) and (2).
\item If $\Ii_x$ is the ideal of germs of functions vanishing on $V_x$ and $\tilde \Ii_x$ the one of germs in $\Oo(\C^n)_x=\Aa(\R^n)_x\otimes \C$ vanishing on $\tilde {V}_x$ we have $\tilde \Ii_x=\Ii_x\otimes \C$.
\end{enumerate}
\end{prop} 
\begin{proof} The ideal $\Ii_x$ is finitely generated. Let $(f_1,\ldots,f_k)$ be a finite set of real analytic functions, defined in a neighborhood $U\subset \R^n$ of $x$, such that their germs at $x$ generate $\Ii_x$.

They extend to holomorphic functions on a complex neighborhood $\Omega$ of $x$ in $\C^n$ and define in $\Omega$ a closed complex analytic set $\tilde V_x$. Then $V_x\subset \tilde V_x$ and $\tilde V_x\cap \R^n=V_x$. Moreover any holomorphic function germ vanishing on $V_x$ vanishes on $\tilde V_x$, since its real and imaginary part belong to $\Ii_{X,x}$ hence they belong to the complex ideal generated by $f_1,\ldots,f_k$ at $x\in \C^n$.

Assume now there exists another germ $W_x$ with the same properties as $\tilde V_x$. Then any holomorphic function germ vanishing on $W_x$ vanishes also on $V_x$ hence on $\tilde V_x$. This implies $W_x\supset V_x$.

Furthermore condition (2) implies $\tilde \Ii_x=\Ii_x \otimes \C$ because if $g$ belongs to $ \tilde \Ii_x$ then $g\in (f_1,\ldots,f_k)\Oo_{\C^n,x}=(f_1,\ldots,f_k)\Aa_{\R^n,x}\otimes \C=\Ii_x\otimes \C$.
\end{proof} 

\begin{defn} Let $V_x$ be a real analytic germ. We define  {\em complexification}\index{Complexification of a real germ} of $V_x$ the germ at $x$ of a complexification of a representant of $V_x$ in $\R^n$.
\end{defn}

\begin{remark}\label{zariski}
If $(X,\Oo_X)$ is a real analytic space and $x\in X$, the germ $X_x$ is always
realized as the germ of a local model of $X$. One could think that $\tilde
X_x$
depends on this model. As a matter of fact it does not for the following reasons.
\begin{itemize}
\item $X_x$ can be realized in its Zariski tangent space  which only depends on
 the germ $X_x$ and is the minimal dimension affine space in which $X_x$ embeds. \item If $\tilde X_x^{(1)},\tilde X_x^{(2)}$ are two complexifications relative
 to different local models, then the identity map of $X_x$ extend to an
 holomorphic isomorphism between Stein neighborhoods of $X_x$ resp. in
 $\tilde X_x^{(1)}$ and $\tilde X_x^{(2)}$ as we will see in next Section. 
\end{itemize}
\end{remark}

\begin{prop}\label{Noncoerenza}
Let $A_x$ be a real analytic set germ at a point $x$ in $\R^n$ and $B$ a complex analytic subset af a complex neighborhood of $x$ in $\C^n$ such that $B_x$ is the complexification of $A_x$. Then $A_x$ is coherent if and only if $\forall y$ close to $x$, $B_y$ is the complexification of $A_y$.
\end{prop} 
\begin{proof} {\it The condition is necessary.} If $A$ is a coherent analytic set inducing $A_x$ and $\Ii_A$ is its ideal sheaf then there exist finitely many sections $
f_1,\ldots,f_k$ such that they generate $\Ii_{A,y}$ for any $y$ in an open neighborhood of $x$. Hence their holomorphic extension define a complex analytic set $B$ such that $B_x\cap \R^n=A_x$ and moreover $B_y\cap \R^n=A_y$ and $B_y$ is the complexification of $A_y$ $\forall y$ close to $x$.

{\it The condition is sufficient.} Let $\Ii_y$ the ideal attached to
$B_y$ for $y$ in a small neighborhood of $x$. $\Ii$ is a coherent
ideal sheaf and $\Ii_y=\Ii_{A,y}\otimes \C$ since by hypothesis $B_y$
is the complexification of $A_y$. Take finitely many holomorphic
sections $f_1,\ldots,f_k$ of $\Ii$ such that they generate $\Ii_z$
for $z$ in a complex neighborhood of $x$. Then the real and imaginary
parts of these sections generate $\Ii_{A,y}$ in a neighborhood since
$\Ii_{A,y}\otimes\C=\Ii_y$. So $\Ii_{A}$ is coherent in that
neighborhood and $A$ is a coherent real analytic set.
\end{proof}

\begin{prop} If the germ $A_a$ is the union of a finite
family of germs of real analytic sets $A'_a$, its complexification
$\tilde A_a$ is the union of the complexifications of the $A'_a$. If
moreover the $A'_a$ are the irreducible components of $A_a$, then the
complexifications $\tilde A'_a$ are the irreducible
components
\footnote{A germ $A_a$ is reducible if there exist two germs $B_a$,
$C_a$ such that $A_a=B_a\cup C_a$, $B_a\neq A_a$, $C_a\neq A_a$;
$A_a$ is \mbox{\em irreducible} if this is not the case. Since the ring $\Oo_{n,a}$ of germs of analytic functions at the point $a$ is noetherian,
one sees easily that any real analytic set germs the union of a
finite family of irreducible germs $A^i_a$ and that they are uniquely
determined if one suppose $A^i_a \not\subset A^j_a$ for $i\neq j$ (in this
case the $A^i_a$ are called ``irreducible components'' of $A_a$). One has the same notions and the same terminology for the complex analytic set germs.}
of the complexification $\tilde A_a$.
\end{prop}

\begin{proof} The union of the $\tilde A'_a$ contains $A_a$ and any germ of
holomorphic function vanishing on $A_a$ vanishes on the union of the $\tilde
A'_a$, hence from the axiomatic definition of the complexification one gets that the
latter is equal to the union of the $\tilde A'_a$. Moreover if $A_a$ is
irreducible, $\tilde A_a$ is irreducible too: indeed, if it was $\tilde A_a=
B'_a \cup C'_a$ with $B'_a, C'_a$ complex germs, distinct from $\tilde A_a$
we would get 

$$ A_a=B_a \cup C_a \ \mbox{\rm with}\ B_a=B'_a \cup \R^n, \quad C_a=C'_a
\cap \R^n,$$ 

so, we would get for instance, $B_a=A_a$, hence $B'_a \supset A_a$, hence
$B'_a \supset \tilde A_a$ and finally $B'_a=\tilde A_a$ against the
 hypothesis. 
Assume now that $A_a$ decomposes in the irreducible components $A^i_a$: the
argument above shows that the complexifications $\tilde A^i_a$ are irreducible
and $\tilde A_a$ is the union of the $\tilde A^i_a$. To prove that they are
actually the irreducible components of $\tilde A_a$ it is enough to show that
$\tilde A^i_a$ is not contained in $ \tilde A^j_a$ for $i \neq j$: now,
if it was $\tilde A^i_a\subset \tilde A^j_a$ we would get $A^i_a\subset
A^j_a$, against the hypothesis.
\end{proof}

\begin{remark} Assume the complexification $B_a$ of a real analytic set germ $A_a$ to be irreducible. Then $B_a$ is equidimensional of some dimension $p$. Then arbitrarily close to $a$ there exist points where $A$ is an analytic manifold of dimension $p$ and $A_a$ is also irreducible. Nevertheless there can be points arbitrarily close to $a$ at which $A$ has dimension strictly less than $p$. 
\end{remark}

Next we give some examples of real analytic spaces that are not coherent.

\begin{examples}\label{Whitneyumbrella}\hfill

\begin{enumerate}  
\item $X=\{x^2 -zy^2=0\}\subset \R^3$. This is  known as {\em Whitney umbrella}.
\item $X=\{x^3 -z(y^2+ x^2)=0\}\subset \R^3$. This is known as   {\em Cartan umbrella}.

{\parindent = 0pt
Both surfaces are irreducible \footnote{Note that the surface in (2) is the cone on an irreducible cubic whose real part has two connected component, namely a point and a branch, and recall that for an irreducible homogeneous polynomial analytic and algebraic reducibility are equivalent.} and not equidimensional: so by Proposition \ref{Noncoerenza} their ideal sheaf cannot be coherent at the point $(0,0,0)$. }
\item $X=\{z(x+y)(x^2+y^2)-x^4=0\}\subset \R^3$

{\parindent =0pt This time $X$ has dimension 2 everywhere since the line corresponding to the
$z$-axis is embedded into its 2-dimensional part, again Proposition \ref{Noncoerenza} implies that $\Ii_x$ can not be
coherent at $(0,0,0)$. }

\end{enumerate}
\end{examples}

\subsubsection{A glueing lemma.}{\hfill}

Analytic spaces are defined as spaces locally isomorphic  to local models embedded in euclidean spaces. The reverse procedure is more delicated: given the local models, how can we glue them in order to get a an analytic space, without loosing for instance the Hausdorff property or adding undesired singularities?

Next proposition describes a general construction,  which avoids theses difficulties. The main tools are paracompactness, so that one can shrink a given covering, and a concrete topological space working as a guideline to follow during the glueing.

\begin{prop}\label{incollamentiT2}

  Let $X$ be  a paracompact topological space  with countable topology
  and let $\{T_i\}_{i\in I}$ a locally finite open covering. Assume to
  have  locally compact  spaces  $\{T_i^*\}_{i\in  I}$ and  comtinuous
  embeddings $\varphi_i:T_i\to T_i^*$ where $\varphi_i(T_i)$ is closed
  in $T_i^*$.

  Assume further that $\varphi_j\circ \varphi_i^{-1}:\varphi_i(T_i\cap
  T_j)\to\varphi_j(T_i\cap   T_j)$   extends    to   a   homeomorphism
  $\psi_{i,j}: T_  {i,j}^*\to T_  {j,i}^* $, where  $T_ {i,j}^*$  is a
  suitable relatively compact  open neighborhood of $\varphi_i(T_i\cap
  T_j)$  in   $T_i^*$,  and  that  $\psi_{j,k}\circ   \psi_{i,j}(x)  =
  \psi_{i,k}(x)$  for all  $x\in  T_k\cap T_i\cap  T_j$. Then,  there
exist open sets $V_i^*$ relatively compact in $T_i^*$ such that.

\begin{itemize}
\item The space $X^*=\bigcup V_i^*/\mathcal R$ is Hausdorff, where $\mathcal R$ is the equivalence relation on $\bigcup V_i^*$ given by $x\mathcal R y $ if $x\in V_i^*$, $y \in V_j^*$ and $\psi_ {i,j}(x)=y$.
\item $\pi: \bigcup V_i^*\to X^*$ is an open map.
\item $X$ embeds into $X^*$ as closed subspace. 
\end{itemize}
\end{prop}
\begin{proof}
 Take two shrinking $V_i\subset U_i\subset T_i$, in such a way
 that $\overline {V_i}\subset U_i$ and $\overline {U_i}\subset T_i$.
 We can find in $T_{i,j}^*$ a relatively compact open set $U
 _{i,j}^*$ such that $\varphi^{-1}_i(U_{i,j}^*)=U_i\cap U_j=
 \varphi_j^{-1}(U_{j,i}^*)$ and $\varphi_i^{-1}(\overline{U_{i,j}^*})=
\overline{U_i\cap U_j}$. Also we can assume
 $\psi_{i,j}(U_{i,j}^*)=U_{j,i}^*$.

From now on we will identify $\varphi_i(A)$ with $A$, for any $A\subset T_i$.
We have to find relatively compact open sets $V_i^*$ such that $V_i^*\cap T_i=V_i,\overline{V_i^*}\cap T_i=\overline{V_i}$.

Since $\overline{V_i}\cap \psi_{j,i}(\overline{V_j}\cap
\overline{U_{i,j}})$ is compact in $U_{i,j}^*$ we can find an open set
$W_{i,j}^*\subset U_{i,j}^* $ containing it. So $\overline{V_i}
\setminus W_{i,j}^*$ and $\psi_{j,i}(\overline{V_j}\cap
\overline{U_{i,j}})$ are disjoint compact sets in $T_i^*$. Hence there
are disjoint open sets $A_{i,j}^*, B_{i,j}^*$ such that
$\overline{V_i}\subset A_{i,j}^*\cup W_{i,j}^*,
\psi_{j,i}(\overline{V_j}\cap\overline{U_{i,j}})\subset
B_{i,j}^*\cup W_{i,j}^*$.

Next take $A_i^*\subset T_i^*$ such that $A_i^*\cap T_i=V_i,\overline{A_i^*} 
\cap T_i=\overline{ V_i}$. Also we ask for $A_i^*\subset A_{i,j}^*\cup W_{i,j}^*$ for each $j$ such that $T_i\cap T_j \neq \varnothing$. This can be done because $\overline{V_i}\subset A_{i,j}^*\cup W_{i,j}^*$. These $A_i^*$'s verify
$$\overline{\psi_{i,j}(A_i^*\cap U_{i,j}^*)}\subset \psi_{i,j}(\overline{V_i}\cap\overline{U_{i,j}}).$$
Indeed 
$\overline{\psi_{i,j}(A_i^*\cap U_{i,j}^*)}\subset \psi_{i,j}(\overline {A^*_i}\cap\overline { U_{i,j}^*})$ and, since $\overline {A^*_i}\cap T_i=\overline {V_i}$,
intersecting with $T_i$ we get 
$\psi_{i,j}(\overline {A^*_i}\cap\overline {U_{i,j}^*})\cap T_i=\psi_{i,j}(\overline{V_i}\cap\overline{U_{i,j}}).$

Next we proceed with the construction of $V_i^*$. For any $x\in U_i$ we choose an open neighborhood $U_{i,x}^*\subset T_i^*$ verifying the following properties.
\begin{itemize}
\item $U_{i,x}^*\subset U_{i,j}^*$ for all $j$ such that $x\in U_i\cap U_j$.
\item For any $j$ such that $x\in \psi_{i,j}(\overline {V_i}\cap\overline{U_{i,j}})$ one gets $U_{i,x}^*\subset W_{i,x}^*\cup B_{i,x}^*$.
\item For any $j$ such that $x\notin\overline{V_j}$ one has $U_{i,x}^*\cap \psi_{j,i}(A_j^*\cap U_j^*)=\varnothing$. This fact is trivial when $U_i\cap U_j=\varnothing$; otherwise for the finitely many remaining $j$ this can be done because $x\notin\overline{V_j}$ implies $x\notin \psi_{j,i}(A_j^*\cap U_{i,j}^*)$ since $\psi_{j,i}(A_j^*)\cap T_i\subset V_j$. 
\item For any $j,k$ such that $x\in U_i\cap U_j\cap U_k=(U_i\cap U_j)\cap (U_i\cap U_k)$ (there exist finitely many such couples) we want $ U_{i,x}^*\subset \psi_{j,i}(U_{j,i}^*\cap U_{j,k}^*)\cap \psi_{k,i}(U_{k,i}^*\cap U_{k,j}^*)$. 
\end{itemize} 
 
Define $U_i^*=\bigcup_{x\in U_i}U_{i,x}^*$ and take $V_i^*$ as a relatively compact open neighborhood of $V_i$ contained in $A_i^*$.

By construction
\begin{align*} 
V_i\subset V_i^*\cap T_i\subset A_i^*\cap T_i=V_i.\\
\overline {V_i}\subset\overline{ V_i^*}\cap T_i\subset\overline {A_i^*}\cap T_i=\overline{V_i}.
\end{align*}

Define $V_{i,j}^*=V_i^*\cap \psi_{j,i}(V_j^*\cap U_{j,i}^*)$ and $V_{i,j,k}^*=V_{i,j}^*\cap V_{i,k}^*$. Also  by construction $V_{i,j}^*\subset U_{i,j}^*$ and $\psi_{i,j}$ is a homeomorphism between $V_{i,j}^*$ and $V_{j,i}^*$ inducing a homeomorphism between $V_{i,j,k}^*$ and $V_{j,i,k}^*$. Also $\psi_{i,j}=\psi_{k,j}\circ \psi_{i,k}$.

Next consider the disjoint union $\bigcup V_i^*$ and put on it the relation equivalence $\mathcal R$ defined as follows:

\begin{center}{
$ x\mathcal R y $ iff $x=y$ or $x \in V_{i,j }^*, y\in V_{j,i}^*$ and $y=\psi_{i,j}(x)$.}  
\end{center}

Put $X^*=(\bigcup V_i^*)/\mathcal R$. So $X^*$ is a countable locally finite union of paracompact spaces with countable topology. Hence it is paracompact as soon as we prove that it is Hausdorff.

Note that there exists an embedding $X\hookrightarrow X^*$ as a closed subspace because $\mathcal R$ induces on $(\bigcup V_i)$ simply the glueing $x\mathcal R y$ if and only if $x=y\in V_i\cap V_j$. 

So let us see that $X^*$ is Hausdorff. 

Note firstly that $\overline {V_{i,j}^*}\subset U_{i,j}^*$. It is enough to prove this inclusion when $T_i\cap T_j \neq \varnothing $. Take $y\in V_{i,j}^*\subset V_i^*\subset U_i^*$. So there exists one $x\in U_i$ such that $y\in U_{i,x}^*$. If $x\notin \psi_{j,i}(\overline{V_j}\cap\overline U_{i,j})$ then $x\notin V_j$ hence $y\notin \psi_{j,i}(A_j^*\cap U_{j,i}^*)$ so $y \notin \psi_{j,i}(V_j)^*\cap (U_{j,i}^*)$, contradiction, since $y\in V_{i,j}^*$. So $x\in \psi_{j,i}(\overline {V_j}\cap\overline {U_{j,i}})$, hence $y\in W_{i,j}^*\cup B_{i,j}^*$. But $y \in V_i^*\subset A_i^*\subset A_{i,j}^*\cup W_{i,j}^*$ and since $B_{i,j}^*\cap A_{i,j}^*=\varnothing$ we get $y\in W_{i,j}^*\subset U_{i,j}^*$. Now from $ {V_{i,j}^*}\subset W_{i,j}^*$ we get $\overline {V_{i,j}^*}\subset\overline{ W_{i,j}^*}\subset U_{i,j}^*$.

Next take $x,y\in\bigcup V_i^*$ and assume $\pi(x)\neq
 \pi (y)$. We have to find open neighbourhhods $A, B$ such that $x\in A, y\in B$ and no point in $A$ is equivalent to a point in $B$. Assume not. Then there exist sequences $\{x_k\}\subset V_{i,j}^*,\{y_k\}\subset V_{j,i}^*, $ tending to $x$, resp. to $y$, such that $y_k=\psi_{i,j}(x_k)$ which implies $x\in\overline{ V_{i,j}^*}\subset U_{i,j}^*, y_k \in\overline{ V_{j,i}^*}\subset U_{j,i}^*$ and by continuity $y=\psi_{i,j}(x)$. But then $y \in V_j^*\cap U_{j,i}^*$ and $x\in V_i^*\cap \psi_{j,i}(V_j^*\cap U_{j,i}^*)=V_{i,j}$. So $x,y$ are equivalent and we have got a contradiction. 
\end{proof}

\subsubsection{The global case.}{\hfill}

A complex analytic space $X$ can be viewed as a real space identifying $\C$
with $\R^2$: we will call this structure \index{ underlying real structure}
{\em the underlying real structure} on $X$ and we will denote this space with $X^{\R}$. In this sense a real analytic space $Y$ can be viewed as a subspace of a complex space $X$ if it is a subspace of $X^{\R}$.

\begin{defn}\label{Complessificazione} Let $X$ be a real analytic space. We will call {\em complexification} of $X$ any complex space $\tilde X$ such that
\begin{itemize}
\item $X$ is a closed real analytic subset in $\tilde X$ considered as a real analytic space.
\item For any $x\in X$ the germ $\tilde X_x$ is the complexification of $X_x$.
\end{itemize}
\end{defn}

Note that if a real analytic space $(X,\Oo_X)$ has a complexification then Proposition \ref{Noncoerenza} implies that it is a coherent analytic space. 
Next theorem implies a converse of this fact.

\begin{thm}\label{TognoliComplessificazione} Let $(X,\Oo_X)$ be a $C$-analytic space, then there exists a complex analytic space $(Z,\Oo_Z)$ such that 
\begin{itemize}
\item $X$ embeds in $Z$ as a closed subset.
\item For all $x\in X$ one has $\Oo_{Z,x}=\Oo_{X,x}\otimes \C$.
\end{itemize}

Moreover if $ Z_1$ and $Z_2$ are two such spaces for $X$, the identity map $i: (X,\Oo_X)\to (X,\Oo_X)$ extends to an isomorphism $\tilde i:(U_1,\Oo_{{Z_1}|_{_{U_1}}}x)\to (U_2,\Oo_{{Z_2}|_{_{U_2}}})$ between two open sets in $Z_1$ and $Z_2$, which are neighbourhoods of the image of $X$.
\end{thm}

\begin{proof} Take a locally finite covering $\{T_i\}_{i\in I}$ of $X$ by relatively compact open sets. Assume to have for any $i$ a realization as local model 
$$
(\varphi_i,\eta_i): (T_i,\Oo_{X|_{_{T_i}}})\to \left(\varphi_i(T_i),
\left(\frac{\Oo_{A_i}}{\Ii_{i}}\right)_{|{_{\varphi_i (T_i)}}}\right)
$$
where $\varphi_i$ is an embedding of $T_i$ in an open set $A_i\subset \R^{n_i}$
 and $\eta_i: \left(\frac{\Oo_{A_i}}{\Ii_{i}}\right)_{|_{\varphi_i (T_i)}}\to 
\Oo_{X|_{_{T_i}}}$ 
is a sheaf isomorphism.

We can assume the ideal sheaf $\Ii_i$ to be generated by finitely many analytic functions on $A_i$, say $f_1,\ldots, f_{k_i}$. These functions extend holomorphically to a complex neighborhood $A_i^*$ of $A_i$ in $\C^{n_i}$ and generate an ideal sheaf $\Ii^*$ such that 
 $\displaystyle \left(\frac{\Oo_{A_i^*}}{\Ii_i^*}\right)_x=\left(\frac{\Oo_{A_i}}{\Ii_i}\right)_x
 \otimes \C
 $ for any $x\in \varphi_i(T_i)$.
 
Denote by $T_i^*$ the zeroset of $\Ii_i^*$ in $A_i^*$. 
When $T_i\cap T_j \neq \varnothing$ we have an isomorphism of coherent sheaf 
$\eta_{i,j}:\left(\frac{\Oo_{A_j}}{\Ii_{j}}\right)_{|{\varphi_j (T_i\cap T_j)}} 
\to \left(\frac{\Oo_{A_i}}{\Ii_{i}}\right)_{|{\varphi_i (T_i\cap T_j)}} $, 
hence between the two sheaves tensorized by $\C$. So $\eta_{i,j}$ extends to an isomorphism of the complex sheaves on a complex neighborhood of $\varphi_{j}(T_i\cap T_j)$.
 
 As we saw in Remark \ref{algebraanalitica} an  isomorphism of complex analytic algebras induces a biholomorpsm of the corresponding germs. Since $\Oo_X$ is coherent we get a biholomorphism $\psi_{i,j}: T_{i,j}^*\to T_{j,i}^*$ 
where $T_{i,j}^*$ is a suitable neighborhood of $\varphi_i(T_i\cap T_j)$ 
in $T_i^*$ such that $T_{i,j}^*\cap \R^{n_i}=\varphi_i(T_i\cap T_j)$ and 
${ \psi_{i,j}}_{|\varphi_i(T_i\cap T_j)}=\varphi_j\circ \varphi_i^{-1} $.
 
 Note that $\psi_{i,j}\circ \psi_{j,k}=\psi_{i,k}$ and $\psi_{i,j}^{-1}=\psi_{j,i}$.
By Proposition \ref{incollamentiT2} there exist relatively compact open sets $V_i^*\subset T_i^*$ with the following properties.
\begin{itemize} 
\item if $V_i=V_i^*\bigcap \R^{n_i}$ then $\{\varphi_i^{-1}(V_i)\}_{i\in I}$ is a 
shrinking of the covering $\{T_i\}$, hence still a covering of $X$ by local 
models.
\item In the disjoint union $\bigcup_i V_i^* $, the relation $x \mathcal R y$ if 
and only if either $x=y$ or $y=\psi_{i,j}(x)$ is an equivalence relation.
\item $\displaystyle X^*=\bigcup_i V_i^*/\mathcal R$ is Hausdorff and the projection $\pi$ is an open map.
\item $\pi(\bigcup_i V_i) $ is analytically isomorphic to $X$ and we identify it with $X$. 
\end{itemize} 

We are left with the proof that $X^*$ has a complex structure $\Oo_X^*$ and that $\Oo_{X^*,x}=\Oo_{X,x} \otimes \C$ for all $x\in X$.

Take a point $x\in X^*$. Its inverse image by $\pi$ is
finite. There exists an open set $U^x$ containing $x$ such that its inverse image is the disjoint union of biholomorphic open sets equipped with isomorphic sheaves. They give $U^x$ a unique complex structure $\Oo_{U^x}$. By construction, for $x\in X$ all these holomorphic sheaves have stalks in the corresponding points isomorphic to $\Oo_ {X,x}\otimes \C$.

Now if $(Z_1,\Oo_{Z_1})$ and $(Z_2,\Oo_{Z_2})$ are two such spaces both containing $X$ as a closed subset they verify $ \Oo_{Z_1,x}=\Oo_{Z_2,x}$ for all $x \in X$. 
Since the two sheaves $\Oo_{Z_1}, \Oo_{Z_2}$ are coherent this implies they are isomorphic on neighbourhoouds $U_1, U_2$ of $X$ respectively in $Z_1, Z_2$. In turn this isomorphism of sheaves induces a biolomorphism between $U_1$ and $U_2$. This ends the proof. 
\end{proof}

\begin{remark}The complex space $(Z,\Oo_Z)$ of the previous theorem depends on the structure sheaf $\Oo_X$ on the $C$-analytic space $X$. We shall call $(Z,\Oo_Z)$ the $\Oo_X$-complexification of $(X,\Oo_X)$. If $(X,\Oo_X)$ is reduced as C-analytic space then, it is coherent, hence $(Z,\Oo_Z)$ is the complexification of $(X,\Oo_X)$ as defined before in Definition \ref{Complessificazione}.
\end{remark}

\begin{thm}\label{intornist}
 Let $(X,\Oo_X)$ be a $C$-analytic space and $(Z,\Oo_Z)$ be a $\Oo_X$-comple\-xifi\-ca\-tion. Then $X$ has in $Z$ a fundamental system of open Stein neighborhoods.
 \end{thm}
\begin{proof}
We prove this statement applying the second critherion in Theorem \ref{Narcrit}. 
So first we have to construct suitable convex functions in a neighborhood of $X$ in $Z$. This is done in the next two steps.

{\sc Step 1.} {\em For any $C^2$-function $r:Z\to \R$ there exist a neighborhood $V$ of $X$ in $Z$ and a nonnegative strongly convex function $f:V\to\R$ vanishing on $X$ together with its first derivatives.}

  Take a locally finite covering $\{U_i\}_{i\in I}$ of $X$ and a shrinking 
$\{V_i\}_{i\in I}$ by relatively compact open sets in $Z$. We assume furtherly that for each $i$ there exists an isomorphism $\rho_i:U_i\to U_i'\subset \Omega_i\subset \C^{n_i}$ in such a way that $\rho_i(X\cap U_i)=U_i'\cap \R^{n_i}$.

Define $\displaystyle f'_i=4 \sum_{j=1}^{n_i} y_j^2=-\sum_{j=1}^{n_i}(z_j-\overline z_j)^2$ where $(z_1,\ldots,z_{n_i})$ are cohordinates in $\C^{n_i}$. 

Then the matrix $(L(f'_i))$ is equal to $2 I_{n_i}$. So for any $i$ $f'_i$ is strongly convex, vanishes on $\rho_i(X\cap U_i)$ together with its derivatives of order 1. The same is true for $f_i=f'_i\circ \rho_i$.

Let be $U=\bigcup U_i$. Take a smooth partition of unity $\alpha_i$ on $U$ subordinate
 to the covering $\{U_i\}$ verifying $\alpha_i>0$ on $V_i$. Since $\overline{V_i}$ is
 compact the number $k_i=\inf_{x\in\overline{V_i}} \alpha_i(x)$ is strictly positive.
 Now on $\rho_i(V_i\cap X)$ one has $(L(\alpha_if'_i))\geq 2k_iI_{n_i}$ since
 $(L(\alpha_if'_i))=\alpha_i(L(f'_i))=2\alpha_i I_{n_i}$ because the first
 derivative of $f'_i$ vanishes on $\rho_i(U_i\cap X)$. Being $(L(\alpha_if_i))$
 bounded from below by a positive definite matrix there exists a constant $h_i$ such
 that $h_i\alpha_i f_i + r$ is strongly convex on $V_i\cap X$, hence in a
 neighborhood of $V_i\cap X$ in $Z$.

Define $f=\sum_{i \in I}h_i\alpha_if_i$. It is well defined since the covering is locally finite, vanishes on $X$ together with its first derivatives and is nonnegative on $U$, is strongly convex on a neighborhood $V$ of $X$ and, by construction, $f+r$ is strongly convex on $V$.

 {\sc Step 2.} {\em There are neighborhoods $V,U$ of $X$ in $Z$, $\overline V\subset U$ and a smooth function $\eta:V\to \R$ with the following properties.
\begin{itemize}
\item $\eta \geq 0, \eta _{|{X}}=0, \eta $ is convex on $U$.
\item $\eta$ is strongly convex where it is positive.
\item $\eta(x)\geq 1 $ for any $x\in \partial V\subset U $.
\end{itemize}}

 \noindent Take a locally finite $\{U_i\}$ and a shrinking $\{V_i\}$, $V_i$ relatively compact in $U_i$ as in Step 1 with $\rho_i: U_i\to U'_i\subset \Omega_i\subset \C^{n_i}$ and $ \rho_i(U_i\cap X)=U'_i\cap \R^{n_i}$. Assume further the following.
\begin{itemize}
\item For any $i$ the distance $\delta_i=d(\partial (\rho_i(V_i)\cap \R^{n_i}), \partial (U'_i\cap \R^{n_i}))>0$.
\item For any $x\in U'_i\cap \rho_i (U_i\cap (\bigcup_{j\neq i}V_j))$ one has $d(x,\R^{n_i})<\frac{1}{4}\delta_i$.
\end{itemize}
This is always possible by a suitable choice of the covering $U_i$.

Let $z_j=x_j+iy_j$, $j=1,\ldots,n_i$ be coordinates in $\C^{n_i}$. For $a=(a_1,\ldots, a_{n_i})\in \R^{n_i}$ define 
$$
\psi^i_a(z)=2\sum_{j=1}^{n_i} y_j^2-\sum_{j=1}^{n_i} (x_j-a_j)^2=(d(z,\R^{n_i}))^2 - (d(\Re z,a))^2.$$

An easy calculation shows that the matrix associated to the Levi form of $\psi^i_a$
is $\frac{1}{2} I$. We want to prove that the function $\displaystyle\eta '_{i,a}=e^{-\frac{1}{\psi^i_a}}$ is strongly convex in the set where $\displaystyle 0<\psi^i_a< \frac{1}{2}$ (which is a neighborhood of $a$). Indeed if $\displaystyle \eta (t)=e^{-\frac{1}{t}}$
one has that $\displaystyle \frac{\mbox{\rm d}\eta}{\mbox{\rm d}t}=t^{-2}e^{-\frac{1}{t}}$ is positive for $t\neq 0$ while $\displaystyle \frac{\mbox{\rm d}^2\eta}{\mbox{\rm d}t^2}=t^{-3}e^{-\frac{1}{t}}(\frac{1}{t}-2)$ is positive for $\displaystyle 0<t<\frac{1}{2}$.

Let us evaluate the Levi form of $\eta '_{i,a}$ on a complex line $L$ with coordinate $z$.
$$ \frac{\partial^2\eta '_{i,a}}{\partial z \partial\overline z}=
\frac {\partial }{\partial\overline z}\left(
\frac{\partial\eta '_{i,a}}{\partial z} \right)
=\frac {\partial }{\partial\overline z} \left(
\frac{\partial\eta '_{i,a}} {\partial \psi^i_a} 
\cdot 
\frac{\partial \psi^i_a }{\partial z} \right)=
$$

$$=\frac{\partial^2\eta '_{i,a}}{\partial {\psi^i_a}^2 }\cdot 
\frac{\partial \psi^i_a}{\partial\overline z}\cdot 
\frac{\partial \psi^i_a}{\partial z}+ \frac{\partial \eta '_{i,a}}{\partial \psi^i_a} \cdot \frac{\partial^2 \psi^i_a}{\partial z\partial\overline z}.$$

Since $\psi^i_a$ is real valued and strictly plurisubarmonic, $\displaystyle \frac{\partial \psi^i_a}{\partial\overline z}\cdot 
\frac{\partial \psi^i_a}{\partial z} $ is nonnegative, hence for $0 <\psi^i_a<
\frac{1}{2}$ the first summand is nonnegative while the second one is positive 
 that is $\eta'_{i,a}$ is strictly subarmonic on every complex line hence strongly convex.
 Next define 
 $$
 \widehat{\eta} '_{i,a}=
 \begin{cases} 
 \eta '_{i,a} &\text{ if } \eta '_{i,a} >0 \\
 0 &\text{ otherwise } 
 
 \end{cases}
 .$$

This function is a $C^\infty$-function because so is for the function 
$$ h=
\begin{cases} 
 e^{-\frac{1}{t}} &\text{ if } t >0 \\
 0 &\text{ if } t\leq 0
 \end{cases} 
 .$$
Consider the function $ \widehat{\eta} '_{i,a}\circ \rho_i$. It is also a $C^\infty$-function on $U_i$.
 
Let $W$ be a neighborhood of $X$ contained in $\bigcup V_i$. Clearly $\partial W\cap X=\varnothing$. Also for any $i$ , $\rho_i(U_i\cap \partial W)\cap \R^{n_i}=\varnothing.$
Indeed if $y\in \rho_i(U_i\cap \partial W)\cap \R^{n_i}$ then $y=\rho_i(x)$ and $x\in X\cap U_i\cap \partial W$. But $\partial W\cap X=\varnothing$ so this is a contradiction.

Now we globalize. If $a\in \R^{n_i}$ put 
$$
 A^i_a=\{z\in \C^{n_i}|\eta^i_a(z)>0 \} \quad \quad B^i_a=\rho_i^{-1}(A^i_a\cap U_i')
. $$
 
We can assume that each point $x\in \partial W$ is internal to some $B^i_a$ where $a$ verifies $d(a,V'_i\cap \R^{n_i})<\displaystyle \frac {\delta_i}{4}.$ 

So $\displaystyle \partial W\subset \bigcup_a B^i_a$ and we can refine this covering into a locally finite open covering $\{C^{i_j}_{a_j}\}_{j\in J}$ where $C^{i_j}_{a_j}\subset B^{i}_{a_j}$ and $d(a_j,V'_i\cap \R^{n_i})<\displaystyle \frac {\delta_i}{4}$.
 
 We can extend $\eta^i_{a_j}$ to $W$ with $0$ outside $U_i\cap W$ and all these
 extended fuctions are of class $C^\infty$. Finally we can take positive constants
 $\beta_j$ with $j\in J$ such that the function $\eta=\sum_j \beta_j
 \eta^i_{a_j}$ verifies the following properties.
 \begin{itemize}
 \item It is well defined in any point of $W$ because only a finite number of terms are different from 0.
 
 \item It is smooth and convex in a neighborhood of $X$, vanishes on $X$ and it is nonnegative on $W$.
 
 \item It is strongly convex where it is positive (it is enough to take the open sets $U_i$ so small that on $U'_i$ arrives $\psi^i_{a_j}<\displaystyle \frac{1}{2}$). 
\end{itemize}
 So Step 2 is proved.

Now to prove the theorem we are left with finding for any neighborhood $U$ of $X\subset Z$ a smaller Stein neighborhood $V\subset U$. 

By Step 1 there exists an open set $U_1\subset U$ and a smooth function $g$ strongly convex on $U_1$. Since $U_1$ is locally compact and has countable topology there exists a continous nonnegative function $r:U_1\to \R$ such that the sets $U_1^{\lambda}=\{x\in U_1| r(x)\leq 
\lambda\}$ are compact for $\lambda >0$. By Whitney approximation theorem we can assume $r$ to be smooth. 

Again by Step 1 there exists a strongly convex smooth function $\varphi$ on a neighborhood $U_2\subset U_1$ of $X$ such that $\varphi$ vanishes on $X$ with its first derivatives and $\varphi + r$ is strongly convex on a smaller neighborhood $U_3\subset U_2$. 

By Step 2 there exists a convex smooth function $\eta$ on a neighborhood $U_4\subset U_3$ such that $\eta$ vanishes on $X$, $\eta \geq 0$ on $U_4$ and $\eta$ is strongly convex where it is positive. 

Define for positive $\lambda$ $U_4^\lambda=\{ x\in U_4| \eta(x)<\lambda \}$. Then $X\subset U_4^\lambda\subset U_4$ because $\eta$ vanishes on $X$. Hence on $V=U_4^1$ we can define, being $\eta <1$ on $V$, $\displaystyle h=\frac{1}{1-\eta}=1 +\eta + \eta^2 + \cdots$. 

Finally consider on $V$ the function $f=h+\varphi +r$ which is strongly convex by construction.

We claim that the set
$D^\lambda=\{ x\in V| f(x)<\lambda \}$ has compact closure in $V$.

Indeed   if   $V^\lambda_h=\{   x\in    V|   h(x)<\lambda   \}$   then
$D^\lambda\subset V^\lambda_h$.  Also $D^\lambda\subset  U_1^\lambda $
which is compact, hence $\overline{D^\lambda}$ is compact.
 Since $h$ diverges on $V$  we get $\overline{U_h^\lambda}\subset V$ so
$\overline{D^\lambda}\subset\overline{U_h^\lambda}\subset  V$.  So  by
Theorem \ref{Narcrit} the set $V$ is Stein.

\end{proof}

As a consequence of Theorem \ref{AandBreal} we get Theorems A and B for a $C$-analytic space.
\begin{thm}\label{ABreali} Let $(X,\Oo_X)$ be a $C$-analytic space. Then for any $\Oo_X$-coherent sheaf $\Ff$ on $X$
\begin{itemize}
 \item {\sc  A)} For any $x\in X$, the stalk  $\Ff_x$ is
 generated by global sections of $\Ff$. 
\item {\sc B)} $H^q(X, \Ff)=0$ for any $q>0$.
\end{itemize}
\end{thm}

\subsection{Anti-involutions and real parts.}

There are several ways to see a
real analytic space as a closed subset of a complex one. One way is to
consider local models $(Y,\Oo_Y)\subset (\Omega,\Oo_\Omega)\subset (\C^n,
\Oo)$ such that $\Omega $ is stable by complex conjugation and $X=Y\cap
\R^n$.

Another way is to endow the complex space with an antiholomorphic involution
$\sigma: Y\to Y$ in such a way that $X$ is the fixed point set $Y^\sigma$ of $\sigma$.

We need some definitions.

\begin{defn} {\hfill}
\begin{itemize}
\item Let $\tilde X, \tilde Y$ be two complex analytic spaces. A continous map
 $\varphi:\tilde X\to \tilde Y$ is \index{Antiholomorphic map} {\em
 antiholomorphic} if for any $x\in \tilde X$ there exist neighborhoods
 $U$ of $x$ and $V$ of $\varphi(x)$ and realizations of $U$, resp. $V$ as
 local models $U'\subset \C^n, V'\subset \C^m$ in such a way that
 $\varphi':U'\to V'$ is the restriction of an antiholomorphic map between
 open sets in $\C^n, \C^m$.
\item
A map $\sigma:\tilde X\to \tilde X$ will be said \index{Anti-involution} {\em anti-involution} if it is antiholomorphic and $\sigma\circ \sigma={\rm id}$.
\item Let $X$ be a complex analytic space and $Y$ be a real analytic subspace of $X^{\R}$. We call $Y$ a \index{Real part} {\em real part} of $X$ if there exists an open overing $U_i$ of $X$ and realizations $\varphi_i:U_i\to U'_i\subset \C^n$ in such a way that $ \varphi_i(U_i\cap Y)=\varphi_i(U_i)\cap\R^n$.
\end{itemize}
\end{defn}

\begin{remark}\label{realpart} Given a C-analytic space $(X,\Oo_X)$ we constructed its $\Oo_X$-com\-plexi\-fi\-ca\-tion starting from its real local models. Hence $(X,\Oo_X)$ is always the real part of its $\Oo_X$-complexification as soon as we take the local models $T^*_i$ invariant under complex conjugation.
\end{remark} 

We want to show the following characterization of C-analytic spaces.

\begin{thm}\label{c-analytic} Let $X$ be a (connected) real analytic set. The following conditions are equivalent.
\begin{itemize}
\item $X$ is a C-analitic set.
\item $X$ is a real part of a complex analytic space.
\item $X$ is a fixed part of a complex analytic space under an antinvolution.
\end{itemize}
\end{thm}

First of all we prove the following lemma.

\begin{lem}
Let $Z$ be the reduction of an $\Oo_X$-complexification of a C-analytic space $(X,\Oo_X)$. Assume to have neighborhoods $U_1,U_2$ of $X$ in $Z$ together with anti-involutions $\sigma_i:U_i\to U_i, i=1,2$ such that $U^{\sigma_i}_i=X, i=1,2$.
Then there exists a smaller neighborhood $U_3\subset U_1\cap U_2$ such the two anti-involutions coincide on $U_3$.
\end{lem}
\begin{proof} Take $U_1'\subset U_1$ such that $\sigma_1(U_1')\subset U_2$ and consider 
$\sigma_2\circ\sigma_1:U'_1\to U_2$. It is an holomorphic map which is the identity on $X$, hence by Proposition \ref{tripla} it is the identity on some neighborhood $U_3$ of $X$ . There $\sigma_2=\sigma_1^{-1}=\sigma_1$.
\end{proof}
\begin{prop}\label{antinvolution} Let $X$ be a C-analytic set and $\Oo_X$ be a coherent structure on $X$. Then there exists an $\Oo_X$-complexification $(Z, \Oo_Z)$ and an anti-involution $\sigma$ on $Z$ such that $X=Z^\sigma$.
\end{prop}
\begin{proof} In the proof of Theorem \ref{TognoliComplessificazione} the sets $T_i^*$ with their sheaves are complexifications of the sets $T_i$. We can assume the sets $V_i^*, T_i^*$ to be invariant under the complex conjugation so that $X$ is a real part of $Z$. Call $\sigma_i$ the anti-involution $V_i^*$ induced by the complex conjugation. Since $Z=\pi(\bigcup_i V^*_i)$ and $\pi$ restricted to $V_i^*$ is a biholomorphism we get anti-involutions $\sigma_i'$ on $\pi(V_i^*)$ for all $i\in I$.

Now $ \sigma_i'$ and $\sigma_j'$ have the same fixed point set in $\pi(V_i^*)\cap \pi(V_j^*)$ hence they coincide in a smaller neighborhood of $\pi(V_i)\cap \pi(V_j)$.

This allows to find an open neighborhood $Z'$ of $X$ in $Z$ where these local anti-involutions glue together in an anti-involution $\sigma: Z'\to Z'$ such that $X=\Z'^\sigma$.
\end{proof}
\begin{prop} \label{embedding} Let $Z$ be a reduced Stein space and $\sigma:Z\to Z$ be an antiholomorphic involution. 

If $\alpha:Z\to \C^N$ is a proper embedding, there exists a proper embedding $\beta:Z\to \C^{2N}$ such that $\beta\circ \sigma $
is the restriction of the complex conjugation to $\beta(Z)$.
\end{prop}
\begin{proof} First assume $\alpha=\mbox{id}_Z$, that is $Z$ is a closed subspace of $\C^N$. Define for $z\in Z$
$$
\beta (z)=\left(z+\overline{\sigma(z)}, i(z-\overline{\sigma(z)})\right).
$$
\begin{enumerate}
\parindent=0pt
\item $\beta$ is holomorphic, since $\sigma$ and the complex conjugation are both antiholomorphic.
\item $\beta $ is injective. 

Indeed $$\beta(z_1)=\beta(z_2) \Longrightarrow \begin{cases}
z_1+\overline{\sigma(z_1)}=z_2+\overline{\sigma(z_2)}\\
z_1-\overline{\sigma(z_1)}=z_2-\overline{\sigma(z_2)}
\end{cases}
 \Longrightarrow z_1=z_2.
$$
\item $ x\in Z$ is fixed by $\sigma$ if and only if 
$\beta(x)=\overline {\beta(x)}$. 

Indeed if $x=\sigma(x)$ then $\beta(x)=(2\Re x, 2 \Im x)\in \R^{2N}$. Conversely $\beta (x)=\overline {\beta (x)}$ means $\beta (x)=
\beta (\sigma(x))$. Since $\beta$ is injective this implies $x=\sigma(x)$.
\item $\beta (\sigma (x))=\overline {\beta (x)}$ by the same argument as in (3).
\item $\beta$ is proper. 

Since $\beta$ is injective we only need to prove that $\beta$ is closed. Take a convergent sequence $\{z_n\}$ in $Z$ with limit $L$. Then $\{\sigma(z_n)\}$ converges to $\sigma(L)$ hence $ \{\beta(z_n)\}$ converges to $\beta(L)$.
\item $\beta^{-1}:\beta(Z)\to Z$ is holomorphic.

Consider the map $\gamma: \C^{2N}\to \C^{N}$ given by $\gamma(t,w)=\frac{1}{2}(t-iw)$. 

It is holomorphic and 
since $\beta (Z)=\{(t,w)| t=z+\overline{\sigma(z)}, w=i(z-\overline{\sigma(z)})\}$ 
its restriction to $\beta(Z)$
gives $\gamma(t,w)=z$. So the restriction of $\gamma$ to $\beta(Z)$ is the inverse of $\beta$.
\end{enumerate}
Hence $\beta:Z\to \beta(Z)$ is a biholomorphism hence an isomorphism between the reduced Stein spaces $Z$ and $\beta(Z)$.

For the general case take $Z'=\alpha(Z)$ and $\sigma'=\alpha \sigma \alpha^{-1}$. Then $\sigma'$ is an antiholomorphic involution on $Z'$. Define $\beta:Z'\to \C^{2N}$ as before. Then $\beta\circ \alpha:Z\to \C^{2N}$ is the proper embedding we looked for. 
\end{proof}
\begin{prop} \label{localembedding} Let $X $ be a real analytic set that is the fixed part of an antiholomorphic involution $\sigma$ on a reduced Stein space $Z$ of dimension $n$. Then $X$ is the real part of a Stein space.
\end{prop}
\begin{proof} We can find an open covering $\{U_i\}$ of $Z$ with the following properties. 
\begin{itemize}
\item For all $i$, $U_i$ is relatively compact.
\item For all $i$, $U_i$ is a local model, i.e. there exists $\rho_i:U_i\to U'_i\subset \Omega_i\subset \C^{n_i}$.
\item For all $i$ such that $U_i\cap X \neq \varnothing$, $U_i$ 
is $\sigma$-invariant and $\Omega_i$ is Stein.
\end{itemize}
Hence $U'_i$ is a Stein subspace of $\Omega_i$ and inherits an antiholomorphic involution $\sigma'$ induced by $\sigma$. Property (2) implies that the dimension of the Zariski tangent space $T_z U'_i$ is bounded on $U_i$ by some integer  $N_1\leq n_i$. Hence by Theorem \ref{embedding} there exists a proper embedding $\alpha_i:U'_i\to \C^{n+N_i}$. Using the proposition above we find a proper embedding $\beta_i:U'_i\to \C^{2(n+N_i)}$ which trasforms the local model in a new local model where $\sigma$ becomes the complex conjugation.
So $\beta_i\circ \rho_i (X\cap U_i)$ is the real part of $\beta_i\circ \rho_i (U_i)$.
\end{proof}
\begin{remark} In the proof of the theorem above we got more. Indeed not only we got 
$\beta_i\circ \rho_i (X\cap U_i)$ to be the real part of $\beta_i\circ \rho_i (U_i)$ but in any local model the complex space $\beta_i\circ \rho_i (U_i)$ is invariant by complex conjugation. This is not always true: for instance take a complex line $L$ in $\C^2$ intersecting $\overline L$ only at the origin $(0,0)$. Then the origin is the real part of $L$ but $L$ is not stable by conjugation.
\end{remark}
\begin{cor}\label{Canalcaratt}Let $X$ be a real coherent space. Then
\begin{itemize}
\item $X$ admits a complexification.
\item $X$ is the real part of a complex analytic space.
\item $X$ is  the fixed part of a complex analytic space under an antinvolution.
\end{itemize}
\end{cor}
Note that the first condition is equivalent to $X$ to be coherent, unlike
conditions (2) and (3).

Summing up we get
 
{\sc proof of Theorem \ref{c-analytic}. }
(1)$\Rightarrow$ (3). This is Remark \ref{realpart}.

(3)$\Rightarrow$ (2). This is Proposition \ref {antinvolution}

(2)$\Rightarrow$ (3). This is Proposition \ref {localembedding}.

(3)$\Rightarrow$ (1). This is because (3) implies that the local models of $X$  get coherent structures. In fact in a local model $T_i^*\subset \Omega_i\subset \C^{n_i}, T_i^*$ is the zeroset of a finitely generated ideal sheaf $\Ii_i$. Take in this ideal the germs which are invariant under conjugation and restrict this smaller coherent ideal sheaf to $T_i^*\cap \R^{n_i}$. So $X$ is the support of a coherent sheaf, hence it is C-analytic. 
\qed

We end this section with a construction that generalises what done in the proof of Proposition \ref{embedding}. 

\subsection{Real structure of a complex analytic set.}

We assume for simplicity $X$ to be a closed subspace of $\R^n$.

 Let $\Omega$ be an open invariant neighbourhood of $\R^n$ in $\C^n$, $F:\Omega \to \C$ be a holomorphic function and denote by $\sigma$ the complex conjugation
.

\begin{defn}\label{re e im}\hfill
\begin{itemize}
\item $F$ is {\em invariant} if $F(z) = \overline{F\circ \sigma(z)}$. In this case $F$  restricts  to a real analytic function on $\R^n$. Conversely analytic functions on $\R^n$ extend to invariant holomorphic functions on a suitable neighbourhood  of $\R^n$ in $\C^n$.
\item For a general holomorphic function $F$ on $\Omega$ define:
$$\Re (F): \Omega \to \C, \, z\mapsto \frac{F(z) + \overline{F\circ \sigma(z)}}{2}$$
 and
$$ \Im (F):\Omega \to \C, \, z \mapsto \frac{F(z) - \overline{F\circ \sigma(z)}}{2i}.$$
\end{itemize}
\end{defn} 

Both $\Re(F)$ and $\Im(F)$  are invariant holomorphic functions and it holds 
$$F= \Re(F) +i \Im(F).$$
They are not the real and imaginary part of $F$. Indeed $F(x+iy) =  \Re^*(F)(x,y)+i\Im^*(F)(x,y)$ where
$$
\Re^*(F)(x,y)=\frac{F(z)+\ol{F(z)}}{2}\quad\text{and}\quad\Im^*(F)(x,y)=\frac{F(z)-\ol{F(z)}}{2i}
$$
are real valued analytic functions on $\Omega\equiv\Omega^\R$ understood as an open subset of $\R^{2n}$.  

Next consider a local model
$(Z,\an_Z)$ of a  complex analytic space, defined by a coherent sheaf of ideals ${\mathcal I}\subset\an_{\C^n}|_{\Omega}$, that is, $Z=\supp({\an_{\C^n}}_{|{\Omega}}/{\mathcal I})$ and $\an_Z={({\an_{\C^n}}_{|{\Omega}}/{\mathcal I})}_{|Z}$. 
Suppose that ${\mathcal I}$ is generated by finitely many holomorphic functions $F_1,\ldots,F_r$ on $\Omega$. 
Let ${\mathcal I}^\R$ be the coherent sheaf of ideals of ${\an_{\R^{2n}}}_{|{\Omega^\R}}$ generated by $\Re^*(F_i),\Im^*(F_i)$ for $i=1,\ldots,r$ and consider $(Z^\R,\an_Z^\R)$ be the local model for a real analytic space defined by the coherent sheaf of ideals ${\mathcal I}^\R$. This construction leads to the following fact. 

\begin{quotation}
For every complex analytic space $(X,\an_X)$ there exists a structure of real analytic space on $X$ that we denote $(X^\R,\an_X^\R)$ and it is called the \em real underlying structure of $(X,\an_X)$\em. If $(X,\an_X)$ is a reduced complex analytic space, it may fail that $(X^\R,\an^\R_X)$ is coherent or reduced, in any case it is C-analytic.
\end{quotation}

Also we get.

\begin{quotation}
Let $\varphi:(X,\an_X)\to(Y,\an_Y)$ be a morphism of $C$-analytic spaces. Let $(\widetilde{X},\an_{\widetilde{X}})$ and $(\widetilde{Y},\an_{\widetilde{Y}})$ be respective complexifications of $(X,\an_X)$ and $(Y,\an_Y)$. There exist:
\begin{itemize}
\item[(i)] a Stein open neighborhood $U\subset\widetilde{X}$ of $X$ and an anti-involution 
$$\sigma:(U,{\an_{\widetilde{X}}}_{|U})\to(U,{\an_{\widetilde{X}}}_{|U})$$ whose fixed part space is $(X,\an_X)$,
\item[(ii)] a Stein open neighborhood $V\subset\widetilde{Y}$ of $Y$ and an anti-involution 
$$\tau:(V,{\an_{\widetilde{Y}}}_{|V})\to(V,{\an_{\widetilde{Y}}}_{|V})$$ whose fixed part space is $(Y,\an_Y)$,
\item[(iii)] a morphism of Stein spaces 
$$\widetilde{\varphi}:(U,{\an_{\widetilde{X}}}_{|U})\to(V,{\an_{\widetilde{Y}}}_{|V})$$
 such that $\widetilde{\varphi}_{|X}=\varphi$ and $\widetilde{\varphi}^\R\circ\sigma=\tau\circ\widetilde{\varphi}^\R$.
\end{itemize}
In addition, if $\varphi$ is an isomorphism (resp. embedding), shrinking $U$ and $V$, also $\widetilde{\varphi}$ is an isomorphism (resp. embedding). 
\end{quotation}

\subsection{Real analytic subspaces of $\R^n$.}

First of all we present a closed analytic subset of $\R^3$ which does not admit any coherent structure. Hence not all real analytic spaces are C-analytic.

\begin{examples}\label{a(z)} \hfill

Let $a(z)$ be the function defined as follows.   
$$
a(z)=\begin{cases} \exp{\frac{1}{z^2-1}} &\text{ if } -1<z<1 \\ 0 &\text{ if } z \leq -1 \text{ or } z \geq 1 \end{cases}
$$

(1)
Let  $X \subset \R^3$ be the zeroset  of the smooth funtion $a(z)x^3 -z(x^2+ y^2)$. To verify that it is an analytic subset of $\R^3$ it is enough to do it in a neighborhood of the points $(0,0,\pm 1)$, where $X$ can be defined by the equations $x=0, y=0$.

We claim:  if $f \in \Oo(\R^3)$ vanishes on $X$ then $f\equiv 0$.

\begin{proof}\hfill

\begin{itemize}
\item First of all note that the germ at the origin $X_{(0,0,0)}$ is irreducible. Indeed assume $ z(x^2+ y^2)- a(z)x^3=p\cdot q$ with $p,q$ convergent power series. The order of the first member at the origin is $3$ and the first homogeneous polynomial in its developpement is $ z(x^2+ y^2)- a(0)x^3=z(x^2+ y^2)- \frac{1}{e}x^3$ which is the equation of the cone over the irreducible algebraic curve $x^2+ y^2- \frac{1}{e}x^3 $. Hence it cannot be factorised by the initial terms of $p$ and $q$.
\item $X$ reduces to a line for $|z|\neq 0$ in a neighborhood of $(0,0,z)$.
\item Let $X'$ be the complex analytic subset of $\C^3\setminus \{z=\pm1\}$ defined by the equation : $ z(x^2+ y^2)- x^3\exp{\frac{1}{z^2-1}}=0$.
Clearly $X'_{(0,0,0)}$ is the complexification of $X_{(0,0,0)}$. 
\item If we  still call $f$ a holomorphic extension of $f$ to a neighborhood $U$ of $\R^3$, $f$ vanishes on $X'$ in a neighborhood of the origin, in particular it vanishes on an open set of the (connected) set of regular points of the germs $X'_{(0,0,0)}$. Call $M$ the connected component of ${\mathcal Reg}(X')$ that contains the regular points near $(0,0,0)$. Since $M$ is a connected analytic manifold, $f$ must vanish on $\overline M$ (analytic identity principle).
\item If the open set $U$ is sufficiently small then ${\ Reg} (X')\cap U \supset \{y\neq 0\}\cap X'\cap U$. It is enough to look at the derivative with respect to $y$ that is $2yz$. On $X'\cap \{y\neq 0\}\cap U$ this derivative does not vanish unless $z=0$. But then also $x$ is $0$ and the derivative with respect to $z$ does not vanish. 
\item In a neighborhood of the real segment $ x=y=0, 0\leq z <1$ 
$${\Reg} (X')\subset M.$$ 
Indeed this is true for $0\leq z<\varepsilon$ for a suitable $\varepsilon$ since 
${\mathcal Reg} (X'_{(0,0,0)})\subset M$. Then the set of those $ t\in [0,1]$ such that ${\mathcal Reg} (X')\subset M\}$ in a neighborhood of $\{x=0,y=0,0\leq z<t\}$ is a not empty open and closed subset of the segment $[0,1]$. 
\item In a sufficently small neighborhood $V$ of $(0,0,1)\in \C^3$ the only connected component of ${\Reg} (X')\cap U$ in this neighborhood is $M$. 

It is enough to prove that $\forall \varepsilon$ if two points $(x_0,y_0,z_0), (x_1,y_1,z_1) \in X'$ verify
$$|x_i|<\varepsilon, 0<|y_i|<\varepsilon, |z_i-1|<\varepsilon \quad \quad i=1,2$$
then there exists a path joining them in 
$$X'\cap\{ |x|<\varepsilon \sqrt{3}, 0<|y|<\varepsilon, |z-1|<\varepsilon\}$$.

In order to  construct  the path define 
$$\lambda(z)=\frac{1}{z} \exp{\frac{1}{z^2-1}}.$$
We can join $z_0$ and $z_1$ in the anulus $0<|z-1|<\varepsilon$ by an arc $\Gamma$.

Let $B$ the upper bound of $|\lambda(z)|$ on $\Gamma$ and take $\eta>0$ such that $B\eta^3+\eta^2<1$.

Now we can join $(x_0,y_0,z_0)$ to $(x'_0,\eta y_0,z_0)$ where $(x_0')^2=\lambda(z_0)\eta^3x_0^3 -\eta^2x_0^2$. This is done by putting

\begin{align*}
x(t) &=t x_0 \\ 
y(t) &=\sqrt{\lambda(z_0)x(t)^3-x(t)^2}\\
z(t) &=z_0
\end{align*}

Note that $|\lambda(z_0)x(t)^3-x(t)^2|=|(\eta x_0)^3 - (\eta x_0)^2|<2\varepsilon^2$ by our choice of $\eta$. Hence $|y(t)|<3\varepsilon^2$.

Similarly we can join $(x_1,y_1,z_1)$ to the point $(\eta x_1,y'_1,z_1)$ where $(y'_1)^2=\lambda (z_1)\eta^3 x_1^3 - \eta^2x_1^2$.

We are left with the points $(\eta x_0,y'_0,z_0),(\eta x_1,y'_1,z_1)\in {\mathcal Reg}(X') $ that we can join easily joining $x_0$ to $x_1$ in $|x|<\varepsilon$ and putting for $y$   $y=\sqrt{\lambda (z)\eta^3 x^3- \eta^2x^2}$ while $x$ varies between $x_0$ and $x_1$ and $z$ varies between $z_0$ and $z_1$.

So far we got: $f$ vanishes in any regular point of $X'$ sufficientely close 
to $(0,0,1)$.

\item Fix $x_0$, $y_0$ 
close to $0$ with $x_0$ and $y_0$ different from $0$ and look at the function $\lambda(z)$ on the complex line $x=x_0,y=y_0$. It has a singularity at $z=1$ which is not a pole. By Picard theorem there exists a sequence of points converging to $(0,0,1)$ in $ {\mathcal Reg}(X')\cap \{x=x_0,y=y_0\}$ on which $\lambda(z)=\frac{x_0^2+y_0^2}{x_0^3}$ and $f$ vanishes on this sequence. Hence $f(x_0,y_0,z)=0$ for all $z$.
So the real function $f$ vanishes on the open set $V\subset \R^3$ $V=\{0<|x|<\varepsilon, 0<|y|<\varepsilon\}$ which implies by the identity principle that $f$ is the zero function.
\end{itemize}

\end{proof}

(2) If in the previous example we take the function 
$$z^2(1-2z^2)(x^2+y^2) -(y^4+x^4)a(z)$$
 instead of $a(z)x^3 -z(x^2+ y^2)$,
 we get that  its zeroset $X$ is an analytic set whose 2-dimensional part  is {\em compact}
 and,then, arguing as before, again any analytic function vanishing on $X$ does
 vanish on $\R^3$. 

(3) Finally using the function 
 $(1-4(x^2+y^2+z^2))((x^2+z^2-1)^2 +y^2) - ((x^2+z^2-1)^4 +y^4)a(z)$ to define  $X$, we get a compact example with the same properties as the other ones.

\end{examples}

\begin{remark}\label{anydim} The last example can be made in any dimension taking the same function with $y^2$ replaced by $\sum_{i=1}^p y_i^2$.  Then the same properties as before hold true for its zeroset $X_p \subset \R^{p+2}$ . Also, since all $X_p$ are compact, we can embed each $X_p$ in an open neighbourhood of a regular point of $X_{p+1}$, starting from $X= X_1$. This way after proving $X_p$ to be irreducible we get that any $f\in \Oo(\R^{p+2})$ that vanishes on the image of $X_1$ has to be identically zero. 
\end{remark}

Special properties characterize $C$-analytic subspaces of $\R^n$. They were
first described by Cartan.

The following proposition characterizes $C$-analytic subsets of $\R^n$.

\begin{prop}
Let $X$ be a closed analytic subset of $\R^n$. Then the following conditions are
equivalent. 
\begin{itemize}
\item There exist a coherent ideal sheaf $\Ii$ of $\Oo$ such that $X=V(\Ii)$.
\item There exist an open neighborhood $U$ of $\R^n$ in $\C^n$ and a complex
 analytic set $Z\subset U$ such that $Z\cap \R^n=X$.
\item There exists a finite number of analytic functions $f_1,\ldots, f_r$ on
 $\R^n$ such that $X$ is their common zeroset. 
\end{itemize}
\end{prop}

\begin{proof}
The proof of the equivalence of these three statements is not difficult. Since
a coherent ideal sheaf on $\R^n$ extends to a coherent ideal sheaf in an open
Stein neighborhood of $\R^n$ in $\C^n$ (Proposition \ref{tripla}), $\Ii$  defines a complex analytic set $Z$
intersecting $\R^n$ in $X$. Since $Z$ can be defined by the vanishing of a
finite number of holomorphic functions, see Proposition \ref{globalequations}, the real and the imaginary parts of
these functions define the set $X$. In turn if $X=\{f_1=0,\ldots, f_r=0\}$,
these functions define a coherent ideal sheaf $(f_1,\ldots, f_r)\Oo$ whose
zeroset is $X$. By Corollary \ref{Canalcaratt} a real analytic subset of $\R^n$ satifying these properties is $C$-analytic.
\end{proof}

\subsection{Well reduced structure.}
The structural sheaf $\Oo_X$ of a C-analytic space is a coherent sheaf. 
As seen in the examples above
we cannot use the reduced structure unless it is already coherent. Neverthless
when $X$ is a closed subset of $\R^n$ or of an analytic manifold $M$ there exists a
largest coherent sheaf of ideals in $\Oo_M$ having $X$ as zeroset.

\begin{lem}\label{wellreduced} 
Let $X\subset M$ be a real closed C-analytic subset of $M$. Denote by $I(X)$ the
set of global analytic functions on $M$ that vanish on $X$. This ideal is not
the $0$ ideal because $X$ is C-analytic and its zeroset is precisely $X$.
Then $I(X)\Oo_m$ is the largest coherent sheaf of ideals having $X$ as zeroset. 
\end{lem}

\begin{proof} For any coherent sheaf of ideals $\Ii$ having $X$ as zeroset
 Theorem A applies, hence
$\Ii_x$ is generated by $H^0(M, \Ii)$ whose elements are global analytic
functions on $M$ vanishing on $X$. Hence $\Ii\subset I(X)\Oo_M$.
\end{proof}

The structure $\Oo_M/I(X)\Oo_M$ will be called {\em well reduced structure}\index{well reduced structure} of
$X$.

We stress the fact that the complexification of a C-analytic space depends on
the structural sheaf $\Oo_X$ and can be of course not reduced. Nevertheless
the complexification of a well reduced structure is reduced as we prove in the
next proposition.

\begin{prop}
Let $(X,\Oo_X)\subset (R^n,\Oo_{\R^n}) $ be a C-analytic space. Then 
$(X,\Oo_X)$
is the well reduced structure if and only if the $\Oo_X $-complexification of 
$(X,\Oo_X) $ is reduced. 
\end{prop}

\begin{proof} By hypothesis $X$ is a closed subset of $\R^n$ and $\Oo_X=
\Oo/\Ii$. Then $ \Ii\otimes\C$ defines a germ at $\R^n$ of complex
analytic set which can be realised in an open neighborhood $\Omega$ of $\R^n$ 
which can be taken Stein and stable for the complex conjugation. Assume that
$Z\subset\Omega\subset\C^n$ has got its reduced structure $\Oo_Z=\Oo_\Omega/ \Ii_Z$. By
Theorem A all fibres $\Ii_{Z,z}$ are generated by $H^0(\Omega, \Ii_Z)$ which is
the set of holomorphic functions on $\Omega$ that vanish on $Z$. Since
$\Omega$ and $Z$ are invariant by conjugation the same holds for $H^0(\Omega,
\Ii_Z)$ hence it is generated by invariant holomorphic functions vanishing on
$Z$ which extend real analytic functions vanishing on $X$. So, as a germ at
$\R^n$ 
$H^0(Z, \Ii_Z)$ contains all real analytic functions vanishing on $X$, hence
$\Ii$ gives 
$X$ its well reduced structure. 

Conversely if $X$ has its well reduced structure, then for any point $x\in X$ 
the ideal $I(X)\Oo_x$ is radical. It remains radical when tensoring with $\C$.
But then, by R\"uckert Nullstellensatz (see Chapter 3),  as a germ at $\R^n$ 
$(Z, \Oo_Z)$ has the reduced structure.
\end{proof}
\bigskip
\bigskip

The well reduced structure on $X$ allows to define the set $\Reg (X)$ in a similar way as in the complex case. But in the real case there are some differences.

\begin{defn}
Let $X\subset\R^n$ be a C-analytic set and $\Oo_X = \Oo_{\R^n}/I(X)\Oo_{\R^n}$ be its well reduced structure so that the fiber $\Ii_x =(I(X)\Oo_{\R^n})_x$ is generated by a finite number of germs of functions in $I(X)$.
Then $x\in X$ is {\em regular} if the ring $\Oo_{X,x}$ is a regular ring. It is {\em smooth } of dimension $k$ if there is an open neighbourhood $V$ of $x$ in $X$ such that $V$ is a manifold of dimension $k$. 
\end{defn} 

Note that a regular point is smooth, while the converse is not true in general.
In fact when $x$ is regular and $\Ii_x$ is generated by $f_1, \ldots , f_s \in I(X)$, the jacobian matrix of  $f_1, \ldots , f_s $ has constant rank $r$ in a neighbourhood $U$ of $x$ so that $X\cap U$ has a natural structure of $n-r$ dimensional manifold. When $x$ is smooth there are $n-k$ analytic functions on $U$, whose jacobian has rank $n-k$, such that their zeroset is a manifold, but these functions are not in general in  $I(X)\Oo_{\R^n}(U)$. For instance in Whitney umbrella $X$ $\Reg(X)$ is the set $\{x^2-zy^2 =0, z\geq 0\}\setminus\{x=0, y=0\}$ while the points $\{(0,0,t), \quad t<0\}$ are smooth points of dimension $1$,       
Note also that a regular point of $X$, endowed with its well reduced structure, is regular for its complexification $\widetilde X$. It is an easy exercise to prove that in this case $\Reg (X) = \Reg(\widetilde X)\cap \R^n$.

For a general C-analytic space $(X, \Oo_X)$ we get the  same definition of $\Reg(X)$ and of smooth points, but it may happen that $\Reg(X)= \varnothing$. 
\bigskip

Let us summarize what we got in this chapter.
\smallskip

Take a real analytic space $(X,\Oo_X)$ with its local models. If in all local models the structural sheaf is {\em coherent}, then these real structures, tensorized by $\C$ glue together in a complex structure of a complex analytic space (the {\em $\Oo_X$-complexification} of $(X,\Oo_X)$) whose real part is $X$. In this case we call $(X,\Oo_X)$ a $C$-analytic space. If the coherent structure is the reduced one we call $(X,\Oo_X)$ a coherent analytic space and its $\Oo_X$-complexification is a true complexification in the sense that the germ of the complex space at any real point is the complexification of the real germ. In both cases Theorems A and B hold true for $(X,\Oo_X)$, since it admits in its $\Oo_X$-complexification a fundamental system of open Stein neighborhoods.

We have also seen the equivalence between the property of being "real part" of a complex space and being "the fixed point set" of an antiholomorphic involution on a complex space which in turn is equivalent to be C-analytic.
 
 Since, as a consequence of embedding theorems for Stein spaces, a C-analytic space is a subspace of some $\R^n$ under the condition that the dimension of its Zariski tangent spaces is bounded, we have that C-analytic spaces are in general 
the zeroset of finitely many global analytic functions in some euclidean space because to be C-analytic in $\R^n$ is equivalent to have global equations.

Finally we have seen that 
not all real analytic spaces admit coherent structures, that is there exist real analytic spaces that are not $C$-analytic.

\subsection*{Bibliographic and Historical Notes.} The content of this chapter is completely classical.  Literature about complex analytic spaces is very rich and one can refer for exemple to \cite{gr} or \cite{gare}: we refer mainly to S\'eminaires Cartan where in particular a proof of Oka's coherence theorems can be found \cite{c1}.  Behnke and Stein took analytic covers as local models for complex analytic spaces. Grauert and Remmert proved the equivalence between these local models and the reduced ones   due to Cartan-Serre that we used in Section  1. 
Criteria for an analytic space to be Stein of Theorem \ref{Narcrit} can be found in \cite{n1,n4,n5}. Whitney's approximation theorem cited in the proof of  Theorem 3.14 is in \cite{Wh} and 
can be found also for instance in \cite{n}.
 Concerning embedding theorem for Stein spaces it can be found in \cite{n2}. The fact that an analytic space locally embeds in its Zariski tangent space can be found in \cite{gare1}.

Passing to the real framework,   a very rich collections of examples of "bad" spaces  are in the papers  \cite{bc1},\cite {bc2}. The characterisation of closed C-analytic subsets of $\R^n$ and examples 1 and 2 in Section 3C are due to  Cartan \cite{c}, which is the first work where real analytic spaces appear and are  treated in detail, in particular in this paper Theorems A and B are extended to  direct limits of Stein spaces and so they are proved for real analytic sets globally defined in $R^n$. Also here clearly appears the central role of complexification, or better the importance to be the fixed part of a complex space under an anti-involution. Whitney and Bruhat extended this idea to real analytic manifolds in \cite{wb}. There appears the word \em C-analytic \rm for real analytic sets admitting, as in  Cartan's paper,  a complexification, i.e. a coherent structure. In the same paper they defined irreducible components of a complex analytic set and hence of a C-analytic set.  Our proof of the \em glueing lemma \em is inspired by the complexification of a real analytic manifold of \cite{wb}.  Finally Tognoli in  \cite {t} extended these ideas in an abstract context  for general real analytic spaces admitting a coherent structure. A large part of  Section \ref{ras} is based on \cite{t}.

In Section 3D  the notion of  {\em well reduced structure } of a C-analytic set is introduced follwing  \cite {abt}. It appears also in \cite{gal} and  can be found also in  Chapter VIII of  \cite{abr}.

\newpage
\parindent=0pt

\chapter{More on analytic sets.}

In this chapter we develop more properties of real and complex analytic sets.
We begin with the notion of irreducible component of an analytic set. We will see in Chapter 3 the same concept from a more algebraic point of view. Then we deal with the normalization of a complex analytic set in Sections 3 and 4. 
This notion does not behave so well in the real case, but is still a useful tool.
Finally we will look at divisors, that is codimension one analytic subsets, in a C-analytic set.

\section{Irreducible components.}

We saw irreducible components of complex or real analytic set germs in
Sections 1 and 3 of Chapter 1.
In this section  we describe global irreducible components of a complex
analytic set, mainly generalising the ideas of the local structure. We will
see by several old examples the difficulties arising in the real case, which
force once again to pass to the smaller class of real C-analytic sets.

\subsection{Irreducible components of a complex analytic set.}

Let $(X, \Oo_X)$ be a reduced complex analytic set. As usual we always assume  that $X$,
as a topological space, is paracompact with countable topology. 

We know from the local
theory that the set $\Reg(X)$ of regular points of $X$ is an open dense subset of
$X$ and its complement $\Sing(X)$ is an analytic subset of $X$ of smaller
dimension.

\begin{defn} We say that $X$ is {\em reducible} if there are $X_1,X_2 \subset
  X$ analytic subsets such that $X=X_1\cup X_2$ and both are different from
  $X$. It is {\em irreducible}\index{irreducible complex analytic set} if $X= X_1 \cup X_2$ implies $X_1=X$ or
  $X_2=X$.
\end{defn}    

Next proposition characterizes irreducible analytic subsets of a Stein open set.

\begin{prop}\label{sottosp.irriducibili} Let $\Omega$ be a Stein open set in $\C^n$ and  $X\subset \Omega$ be a closed analytic subset of $\Omega$.
$X$ is irreducible if and only if the ideal $I(X) =
\{f\in \Oo(\Omega): X\subset \ceros(f)\}$ is a prime ideal.
\end{prop}

\begin{proof}
Assume $X= X_1 \cup X_2$ to be reducible. Since $\Omega$ is Stein, by Theorem B,  there are
holomorphic functions $f_1, f_2 \in \Oo(\Omega)$ such that $f_i$ vanishes on
$X_i$ but not on the whole $X$, for instance it takes the value $1$ on a point in $X \setminus X_i$. Hence $f_1f_2 \in I(X)$ but $f_i$ does not belong to $I(X)$, hence $I(X)$
is not prime. Conversely if $I(X)$ is not prime there are $f,g$ not in $I(X)$
such that $fg\in I(X)$, so $X= (X\cap\{f=0\}) \cup (X\cap \{g=0\})$ is not
irreducible.    
\end{proof}

Next we characterize irreducible components of an analytic set $X$ in terms of the set $\Reg(X)$.

\begin{prop}\label{chiusure}Let $A$ be a connected component of $\Reg(X)$ and
  denote by $V$ its closure. Then $V$ is an irreducible analytic subset of
  $X$. 
\end{prop}
\begin{proof}
We know fom the local theory that $V$ is analytic in a neighbourhood of each
point of $X$. Hence $V\subset X$ is a closed analytic subset. We only have to
prove that $V$ is irreducible. Assume $V = V_1 \cup V_2, V_i \neq V$. 
Then $V_1 \cap V_2
\subset \Sing(V)$ and at least one of them contains an open set $B$ of $V$, say
$V_1$.

 Hence $V_1 \cap A\subset  V_1\setminus {\rm Sing}(V)$ is closed in $A$, being a closed analytic set, but also
open. Hence $V_1\cap A =A$ which in
turn implies $V_1= V$. 
\end{proof}

\begin{thm}\label{decomposition} Any reduced analytic space $(X,\Oo_X)$ has a unique locally finite
  decomposition $X= \bigcup_i X_i$ where $\{X_i\}$ is a locally finite family
  of irreducible analytic subsets of  $X$ and $X_i$ is not included in $ X_{i'}$ for
  $i\neq i'$.
\end{thm}

\begin{proof} The open set $\Reg(X)$ has countably many connected components. Their
  closures are locally finite, since
 there are only finitely many components entering in a neighbourhood of a
 point $x\in X$, and are irreducible analytic sets whose union is $X$. 

To prove
 the uniqueness assume $X=\bigcup_i X_i = \bigcup_j Y_j$ are different
 decompositions. Then for each $i, X_i = \bigcup_j (Y_j\cap X_i)$. Since $X_i$ is
 irreducible there is an index $j(i)$ such that $X_i\subset Y_{j(i)}$. By the
 same argument there is an $i'= i'(j(i))$ such that  $X_i\subset Y_{j(i)}
 \subset X_{i'}$, hence $i=i'$ and $X_i=Y_{j(i)}$. 
\end{proof}

We can caracterize irreducible analytic spaces from an algebraic point of view.

\begin{prop}Let  $(X,\Oo_X)$ be a reduced Stein space. Then $X$ is
  irreducible if and only if $(0)$ is a prime ideal in $\Oo_X(X)$. 
\end{prop}
\begin{proof}It works the same as the proof of Proposition
  \ref{sottosp.irriducibili} replacing $I(X)$ by the ideal $(0)$. 
\end{proof}

\subsection{Irreducible components of a C-analytic set.}

Coming to the real setting once again we have to reduce to C-sets in order to
avoid the bad examples that can be found in \cite{wb, bc1,bc2}.
Let us see one such  example of a real analytic set which is not C-analytic.

\begin{example} Consider the irreducible cone  $X= \{(x,y,z): z(x^2+y^2) -x^3=0\}$ over the plane cubic $\{x^2+y^2-x^3=0\}$ in the plane $\{z=1\}$. The singular set of $X$ is the line $D= \{x=0, y=0\}$. But, being $X$ a cone, the union of the  regular part $S= X\setminus D $ and the vertex $(0,0,0)$ contains a sequence of distinct lines $\{D_k, k\geq 1\}$. Next consider all finite sequence $I$ of positive integers and define inductively on the number of terms in $I$ a translation $T_I$ with the following properties.
\begin{itemize}
 \item[(i)] $T_\varnothing$ is the identity.
 \item[(ii)] If $I = \{J,k\}$ one has:
\begin{enumerate} 
\item $T_I(\overline S) \cap \{x^2+y^2+z^2 \leq s(J)+k\} = \varnothing$, where $s(J)$ is the sum of the elements  in $J$.
\item $T_I(D) = T_{J}(D_k)$
      \end{enumerate}
\item[(iii)] The images $X_I = T_I(X)$ are all distinct.
\end{itemize}

Let $Y = \bigcup_I T_I(X)$. Then $Y$ is irreducible. Indeed $\overline{Y\setminus X_I}$ is not analytic because $X_I$ contains the tails of several $X_{I'}$. Nevertheless $Y$ contains countably many C- analytic subsets of dimension $2=$dim $Y$.  
In particular $Y$ is not C-analytic.  
\end{example}
\bigskip

Let $(Y,\Oo_Y)$ be a reduced Stein space endowed with an anti-involution $\sigma$
and let $X$ be its fixed part. Starting from the irreducible components of $Y$
we want to find a decomposition of $X$ into irreducible C-analytic subsets.


\begin{thm}\label{realcomponents} Let $X$ be a C-analytic set which is the  fixed point set of an anti-involution on a
  reduced Stein space. Then $X$ admits a unique decomposition into irreducible
  C-analytic sets.
\end{thm}
\begin{proof}
Let $(Y,\Oo_Y)$ be the Stein space and $\sigma$ its anti-involution. For any
irreducible component $Y_i$ of $Y$ do the following:
\begin{itemize}
\item If $\sigma(Y_i) = Y_j, j\neq i$, replace $Y_i\cup Y_j$ by $Y_i\cap Y_j$
  which is now invariant by $\sigma$.
\item If $Y_i$ is $\sigma$ invariant, consider $\sigma$ restricted to
  Reg$(Y_i)$. $X\cap \Reg(Y_i)$ can be empty. In this case replace
  $Y_i$ by its singular part $\Sing(Y_i)$. If $X\cap \Reg(Y_i)$ is not empty it is made by regular points of $X$: so, leave $Y_i$ unchanged.
  \end{itemize}  

This way we get a closed subspace $Y_1 \subset Y$ still endowed with an
antiinvolution having $X$ as fixed point set. We can apply again the same
procedure to the irreducible components of $Y_1$. After a finite number of
steps we are left with a Stein space $Y^*$ such that all its components are
$\sigma$ invariant and their fixed point set has the same real dimension as
the complex dimension of that component.
Then we can define the C-irreducible components as the fixed point sets of the
irreducible components of $Y^*$. They are uniquely determined by $X$.  
\end{proof}

\begin{remark}\hfill\label{doppioombrello}
\begin{itemize}
\item An irreducible C-analytic set can be reducible as analytic 
  set: consider for instance the zeroset $X$ of the polynomial $x^2 -
  (z^2-1)y^2$ in $\R^3$.
 The polynomial  $p = x^2-(z^2-1)y^2$  is irreducible as analytic  function and
generates $I(X)$.   Nevertheless, $X$ is  the union of two  analytic subspaces
$X_1$  and $  X_2$  that  are not  global,  each one  isomorphic  to a  Whitney
umbrella. More precisely we get

$$\displaystyle X_1\cap \{z > -\frac{1}{2}\} = \{x^2-(z^2-1)y^2 = 0 \} \cap \{z > -\frac{1}{2}\}$$
while 
$$\displaystyle X_1\cap \{z<0\} = \{x=0, y=0\} \cap  \{z<0\};$$
$$\displaystyle X_2\cap \{z <\frac{1}{2}\} =  \{x^2-(z^2-1) y^2 = 0\} \cap \{z <\frac{1}{2}\}$$
while 
$$\displaystyle X_2\cap \{z >0 \}  = \{x=0, y=0\}  \cap \{z> 0\}.$$
Both $X_1$ and $X_2$  are analytic sets but they are not C-analytic.
{\parindent=0pt
So, we get that  a C-analytic set  $X\subset \R^n$ is  {\it irreducible} if
$I(X)$ is a real prime ideal in $\Oo(\R^n)$.}
\item To define C-irreducible components of a C-analytic set $X$ we cannot avoid
  to consider a bigger space in which $X$ is embedded. Other approaches as the
  one of Whitney and Bruhat  also consider a real subspace $X$ of a real analytic
  manifold $M$. Then $M$ has a complexification $N$, well defined as a germ at
  $M$, and inside $N$ lies a smaller complex subspace $Y$ whose real part is $X$.\end{itemize}
\end{remark}

\section{Normalization.}\label{norm}

In this section we recall the notion of   normalization of a reduced complex analytic space and then we look at what can be generalised in the real case.
As we did in Chapter 1, we recall the main facts on normalization of complex spaces while we give  complete proofs for real C-analytic spaces.

\subsection{The normalization sheaf $\widecheck{\Oo}_X$.}

Let $(X,\Oo_X)$ be a reduced complex analytic space. Denote by $\mathcal{ M}$ its sheaf of germs of meromorphic functions. For every point $x\in X$ consider the set of  elements $h_x\in \mathcal{M}_x$ which are integral over $\Oo_{X,x}$, that is, they satisfy an equation

$$h_x^p+ a_{1,x}h_x^{p-1} +\cdots + a_{p,x} =0 .$$ 
Here $p\geq 1$ and $a_{1,x}, \ldots ,a_{p,x}$ are germs in $\Oo_{X,x}$.

The set of integral elements is a ring $\check{\Oo}_{X,x}$ and one has $\Oo_{X,x}\subset \check{\Oo}_{X,x} \subset \mathcal{M}_x$. Also one knows from commutative algebra that $\check{\Oo}_{X,x}$ is the largest extension of $\Oo_{X,x}$ which is a finitely generated $\Oo_{X,x}$-submodule of $\mathcal{M}_x$.

Note that if $h\in \mathcal{M}(U)$ is a meromorphic function on an open set $U\subset X$ which verifies at some point $x\in U$

$$h_x^p+ a_{1,x}h_x^{p-1} +\cdots + a_{p,x} =0 $$ 
 
then, there is a neighbourhood $V\subset U$ of $x$ and holomorphic functions $a_1, \ldots , a_p \in \Oo_X(V)$ such that on $V$ one has  
$$h^p+ a_{1}h^{p-1} +\cdots + a_{p} =0 $$
Indeed if $a_1, \ldots , a_p \in \Oo_X(V)$ are representative of $a_{1,x}, \ldots ,a_{p,x}$ on a suitable open set $V$, the meromorphic function $h^p+ a_{1}h^{p-1} +\cdots + a_{p} $ has germ $0$ at $x$, hence it is identically $0$ on an open connected neighbourhood of $x$.

This means that the relation of integral dependence for meromorphic germs induces a relation of integral dependence for meromorphic functions on a small neighbourhood.

Define $\check{\Oo}_X = \bigcup_{x\in X} \check{\Oo}_{X,x}$. By the remark above it is an open subset of the sheaf $\mathcal M$, hence it is an analytic subsheaf containing $\Oo_X$. It is called the {\em normalization of the sheaf $\Oo_X$} or the {\em normalization sheaf}. 

\begin{defn} We say that a complex analytic space $(X,\Oo_X)$  is {\em normal} at a point $x\in X$ if the germ $X_x$ is irreducible and the ring $\Oo_{X,x}$ is integrally closed in its field of quotient, i.e. $\Oo_{X,x} = \check{\Oo}_{X,x}$.
\end{defn} 

\begin{remark}\label{normalgerms} 
Assume the germ at $x$ of an analytic space $(X,\Oo_X)$ to be irreducible. Then we get $\Oo_{X,x} \subset \check{\Oo}_{X,x}\subset \mathcal M_x$. This proves that 
$\Oo_{X,x} $ and $\check{\Oo}_{X,x}$ are both integral domains and share the same quotient field. Moreover 
the local algebra $\check{\Oo}_{X,x}$  is the local algebra of a well defined normal  analytic germ $\check{ X}_x$ 
(see Remark \ref{algebraanalitica} in Chapter 1). The inclusion homomorphism is injective, hence we get a surjective finite analytic map $p:\check X_x\to X_x$. The germ $\check X_x$ is called the {\em normalization} of the germ $X_x$.
When $X_x = X_1 \cup \cdots \cup X_l$ is reducible, each irreducible component $X_i$ gets a normalization $\check {X_i}$ and the normalization of $X_x$ is  defined as the disjoint union of $ \check{ X_1},  \ldots,  \widecheck{X_l}$.
\end{remark} 

As a matter of fact  the normalization sheaf is coherent as proved by Oka. The proof  can be found in \cite{c1} and \cite{gare1}. 
This implies the following remarkable fact.

\begin{cor}\label{okacor}
The set $Y= \{x\in X: X \,\mbox{\rm is not normal at}\, x\}$ is an analytic subset of $X$ and $Y\subset \Sing(X)$
\end{cor}
\begin{proof}$Y$ is the set of points where $\check{\Oo}_{X,x}$ is different from 
$\Oo_{X,x}$, that is, it is the support of the coherent sheaf \  $\displaystyle \frac{\check{\Oo}_X}{\Oo_X}$. Hence it is an analytic set. Moreover $X$ is normal at any regular point, since $\Oo_{\C^k}$ is integrally closed in its field of fractions.
\end{proof}

Next we would like to find a global normalization of a given complex analytic space. Let us look more deeply at the local situation.

Let $(X,\Oo_X)$ be an analytic space and assume  its germ at the point $a\in X$ to be irreducible. The germ $\check{X_a}$ is a germ at a point $a'$ of a complex analytic space which is normal at $a'$. We call it $\check X$. Moreover there is a  map $p:\check X \to X$. Since $p$ is surjective from  $\check{ X_{a'}}$ to the germ $X_a$, up to shrink the neighbourhoods of $a$ and $a'$, we can assume $p$ to be surjective. Also by Corollary \ref{okacor} we can assume $\check X$ to be normal, since it is normal at $a'$.    

One has the following properties of the map $p$.

\begin{itemize}
\item[(i)] $p$ is proper, 
\item[(ii)] the fibers of $p$  are finite,
\item[(iii)] there is an open dense set $X_0\subset X\cap U$ (containing all regular points of $X\cap U$) such that the fiber is a singleton over $X_0$.  
\end{itemize}  

\begin{prop}\label{prop2} 
Let $(Y,\Oo_Y), (X,\Oo_X)$ be analytic spaces and  $p:Y\to X$ be a surjective analytic map. Assume $Y$ is normal and of constant dimension and $p$  satisfies properties (i), (ii), (iii) above.
  Then, if $p^{-1}(x) = \{x_1,\ldots, x_l\}$, the irreducible components $X^i_x$ of $X_x$ are precisely the images $p(Y_{x_i})$. Moreover, the restriction 
$p:Y_{x_i} \to X^i_x$ identifies $Y_{x_i}$ with the normalization of the germ $X^i_x$. 
\end{prop}

\begin{proof} Since $p$ is proper, the image $p(Y_{x_i})$ is an irreducible analytic subset of $X_x$ of the same dimension, hence an irreducible component. It  never happens $p(Y_{x_i}) \subset p(Y_{x_j})$ when $j\neq i$, because this would contradict property (iii) of $p$. 
As $p$ is surjective, all irreducible components of $X_x$ are obtained like this.
The last assertion can be proved in this way. The surjection $p:Y_{x_i}\to X^i_x$ induces an inclusion of the corresponding analytic algebras, hence an inclusion of the corresponding quotient fields. Since $p$ is finite, the degree of the bigger field over the smaller one has to be $1$. Hence both fields coincide. 
Since $ \Oo_{Y,x_i}$ is integrally closed, because it is normal,  it contains the integral closure of $\Oo_{X^i,x}$. Now if $R$ is a domain with quotient field $F$, its integral closure $\check R$ is the biggest extension that is a finite $R$-module.
Since $ \Oo_{Y,x_i}$ is a finite extension of  $\Oo_{X^i,x}$, we conclude $ \Oo_{Y,x_i}= \check{\Oo}_{X^i,x}$.
\end{proof} 

The following proposition gives uniqueness of normalization up to isomorphism.

\begin{prop}\label{prop3}
Let $p_1:Y_1 \to X, p_2:Y_2 \to X$ satisfy Proposition \ref{prop2}. 
Then there exists a unique isomorphism $f:Y_1\to Y_2$ such that $p_2\circ f =p_1$.
\end{prop} 
\begin{proof}
Both $p_1,p_2$ are isomorphisms on the inverse image of an open dense subset of $X$. Only   the points where $X$ is reducible are left, because if $x\in X$ is a singular point and $X_x$ is irreducible, both $Y_1$ and $Y_2$ induce on $x$ the germ of the normalization of $X_x$. Take $x_1 \in Y_1$. Then $p_1((Y_1)_{x_1})$ is an irreducible component of the germ $X_{p_1(x_1)}$. There is a unique $x_2\in Y_2$ such that $p_2(x_2) = p_1(x_1)$ and $p_2 ((Y_2)_{x_2})$ is the same irreducible component. Define $f(x_1) =x_2$. It is straightforward to check that $f$ is well defined  and satisfies  $p_2\circ f =p_1$. It is an isomorphism because extends an isomorphism on an open dense subset of the normal space $Y_1$  onto an open dense subset of the normal space $Y_2$.
\end{proof}

\begin{defn}
Let $(X,\Oo_x)$ be an analytic space with costant dimension at all its points. A {\em normalization} of the space  $(X,\Oo_X)$ is a normal analytic space $(Y, \Oo_Y)$ together with an analytic map $p:Y \to X$ satisfying the hypothesis of Proposition \ref{prop2}. 
\end{defn}

By Proposition \ref{prop3} if $X$ has a normalization, it is unique up to isomorphisms.   

\begin{prop}\label{prop4}
Any analytic space  $(X,\Oo_X)$ of pure dimension admits a normalization.
\end{prop}

\begin{proof}
Let $a\in X$ be a point. Both when $X_a$ is irreducible or when it is reducible we can construct the normalization  of $X_a$ and we know that it works in a neighbourhood of $a$. We can find a locally finite open covering of $X$ by open sets $U_\alpha$ such that for each $\alpha$ there is a normal analytic space $(Y_\alpha, \Oo_\alpha)$ together with a map $p_\alpha: Y_\alpha\to U_\alpha$ such that it induces the normalization of $X_a$ for all $a\in U_\alpha$, that is, $p_\alpha$ verifies Proposition \ref{prop2}. By Proposition \ref{prop3} all these normal spaces glue together in a unique way to form a normal analytic space $(Y,\Oo_Y)$. Also the maps $p_\alpha$ glue together to a map $p:Y\to X$ satisfying Proposition \ref{prop2}.    
\end{proof}

\subsection{Properties of the normalization.}

We list some properties of the normalization of a complex analytic space of pure
 dimension.
Let $p: Y\to X$ be a normalization of $X$ and $\Sing(X)$ its singular part.

\begin{itemize}
\item $\Sing(Y)$ has codimension at least $2$ in $Y$.
\item  $p: Y\setminus p^{-1}(\Sing(X)) \to X\setminus \Sing(X)$ is a biholomorphism between open dense subsets of $Y$ and $X$.
\item for each open  set $V\subset X$ there is a natural isomorphism
between $\Gamma(V, \check{\Oo}_X)$ and $\Gamma(p^{-1}(V),\Oo_Y)$.  The space  
$\Gamma(V, \check{\Oo}_X)$ is the ring of {\em weakly holomorphic functions}, that is, holomorphic functions   on $V \setminus \Sing(X)$ that are bounded in a neighbourhood of any singular point.
\item (Riemann extension theorem.) Each weakly holomorphic function on $V$ becomes holomorphic in $p^{-1}(V)$ when composed with $p$.
\item If $X_x$ is irreducible, $p^{-1}(x)$ is only one point in $Y$, while if $X_x$ has several irreducible components, the number of such components coincides with the cardinality of  $p^{-1}(x)$.
\end{itemize}

\subsection{The real case.}

The definition of normal point in a C-analitic space is analogous to the one above. A point is normal when  its local algebra is integrally closed in its total ring of fractions.

The first remark is that the properties of the normalization of a complex analytic space do not hold even for analytic manifolds. For instance $\R^2$ is normal but there are rational functions that are bounded but not analytic, for instance, $\displaystyle \frac{x^4+y^4}{x^2 +y^2}$, so Riemann extension theorem fails in the real case. Hence, we cannot hope to get a notion of real normalization enjoying the good properties of the complex one. Nevertheless since we deal with C-analytic spaces we can consider the normalization of their complexifications and look to what can be said.

Let $(X,\Oo_X)$ be a C-analytic space. Assume $X\subset \R^n$ and $\Oo_X$ is its well reduced structure, as defined in Chapter 1. Hence its complexification $(Y, \Oo_Y)$ is a reduced Stein space. Let $(\check{ Y}, \Oo_{\check{ Y}})$ be the normalization of  $(Y, \Oo_Y)$. We prove first that the anti-involution $\sigma$ on $Y$ (whose fixed part is $X$) lifts to an anti-involution on $\check {Y}$.
   
\begin{prop}\label{lifting} 
Let $(Y,\Oo_Y)$ be a reduced Stein space endowed with an anti-involution $\sigma$
with not empty fixed part. Then there is an anti-involution $\check {\sigma}$ on $\check {Y}$ which makes commutative the following diagram 

$$
\begin{xy}
\xymatrix{
\check{ Y}\ar[r]^{\check{\sigma}} \ar[d]^p&\check{Y}\ar[d]^p\\
Y \ar[r]^{\sigma}&Y
}
\end{xy}
$$
where $p$ is the projection.
\end{prop}

\begin{proof}
Recall that points in $\check{Y}$ are in one to one correspondence with couples $(Y_i,y)$ where $Y_i$ is an irreducible component of the germ $Y_y$. Since $\sigma(Y_i)$ is an irreducible component of $Y_{\sigma(y)}$, we can define $\check{ \sigma}(Y_i,y) = (\sigma(Y_i), \sigma(y))$. Now it holds that $\check{ \sigma}$ is an involutive homeomorphism of $\check {Y}$. To see that it is antiholomorphic consider an antiholomorphic function $h$ on $Y$ and an  open set $V\subset Y\setminus \Sing(Y)$. Let $p:\check {Y}\to Y$ be the projection. Since $p$ is a biholomorphism on $\check{ Y} \setminus p^{-1}(\Sing(Y))$ and $h \circ \sigma$ is holomorphic on $V$, we get $h\circ \sigma\circ p$ is holomorphic on $p^{-1}(V)$. But $\sigma\circ p= \check{ \sigma}$, so $ h\circ  \check{ \sigma}$ is holomorphic on $ p^{-1}(V)$ and $\check{ \sigma}$ is antiholomorphic. This argument works on $\check{ Y}\setminus p^{-1}(\Sing Y)$. Since $\check{ Y}$ is normal, the holomorphic function $ h\circ  \check{ \sigma}$ extends holomorphically to $\check {Y}$ because it is bounded on a neighbourhood of any point, so $\check{\sigma}$ is antiholomorphic as required.      
\end{proof}

\begin{remark}\hfill
\begin{itemize}
\item In general the fixed  point set $\check{Y}^{\check {\sigma}}$ of
  $\check {\sigma}$ is strictly  contained in $p^{-1}(Y^\sigma)$, that
  is,   $p:  \check{   Y}^{\check{  \sigma}}   \to  Y^\sigma$   is  not
  surjective.  However we   call  $\check{  Y}^{\check{  \sigma}}$   {\em  the
    normalization}  of $Y^\sigma$.  It  depends on  the complex  space
  $Y$. Of course if $X$ is embedded as a
closed subset of $\R^n$ it has a well defined well reduced structure, which 
makes $Y$ uniquely defined as germ at $X$, hence also $\check{ Y}$ is uniquely 
determined. 
\item One can define a  well reduced structure as soon as $X$ is given as a closed subspace of a real analytic manifold $M$ (or of a coherent C-analytic space). It is given by the  'biggest'  coherent sheaf of ideals in $\Oo_M$ having $X$ as zero set. One can show that this biggest sheaf of ideals is given by ${\rm I}(X) \Oo_M$ because global sections of such sheaves of ideals are analytic functions on $M$ that vanish at $X$.
\end{itemize} 
\end{remark}

When $(X,\Oo_X)$ is a coherent real analytic space, then it gets a true normalization as we show in the next proposition. We use the same notations as in Proposition \ref{lifting}. 

\begin{prop}\label{cohe} Let $(X,\Oo_X)$ be a coherent C-analytic space. Then $p: \check{ Y}^{\check{ \sigma}} \to X$ is surjective and $p^{-1}(X)= \check{ Y}^{\check{ \sigma}}$.
\end{prop}
\begin{proof}We can assume $X$ to be globally irreducible of pure dimension $s$.
Denote by $Y$ its complexification and by $\check {Y}$ the normalization of $Y$.
Recall that for all $x\in X$ the germ $Y_x$ is the complexification of $X_x$.
Let $x\in X$ be a point where the germ $X_x$ is irreducible. Then $Y_x$ is irreducible and $\sigma$-invariant, hence $p^{-1}(x)$ is a single point $y\in \check{ Y}^{\check {\sigma}}$ and $p(y)= x$. 

Next assume $X_x = X_1 \cup X_2\cup\ldots \cup X_l$ is reducible. Then $Y_x$ has irreducible components $Y_1, \ldots , Y_l$ where  $Y_i$ is the complexification of $X_i$ and is $\sigma$-invariant for each $i$. So $p^{-1}(x)$  is a $l$-uple of points $\{y_1, \ldots ,y_l\}$ where $\check{ Y}_{y_i}$ is $\check{\sigma}$-invariant. So $\{y_1, \ldots ,y_l\}\subset \check{ Y}^{\check{\sigma}}$ and $p(\{y_1, \ldots ,y_l\}) =x$.    
\end{proof}

We will see in Chapter 5 the  role of the set $p^{-1}(X) \setminus \check {Y}^{\check{\sigma}}$ for a general C-analytic space.
Next we show an example of a normal C-analytic space which is not coherent. This
disproves  an old conjecture saying that normal spaces are always   coherent spaces.

\begin{example}\label{alberto}  Define 

$$X= \{x\in \R^4: x_3(x_1^2 + x_2^2) -x_1^2 x_4^2 +x_4^2 =0\}.$$

The polynomial that defines $X$ is irreducible in $\C[z_1,z_2,z_3,z_4]$ and the germ of its complex zeroset $Y$  at $(0,0,0,0)$ is the complexification of the real set germ of $X$ at $(0,0,0,0)$. Now $Y$ is a hypersurface in $\C^4$ and $\Sing Y$ has codimension $2$ in $Y$. Indeed   a simple calculation shows that $\Sing Y = \{z_1^2 +z_2^2 =0, z_3=0, z_4=0)\}\bigcup \{z_1=z_2=0, z_4=0\}$. 

By a criterion due to Oka \footnote{Oka's Criterion: An irreducible hypersurface $X\subset \C^n$ is normal if and only if Sing$(X)$ has codimension at least 2, see \cite{ab}.}and generalised by Abhyankar, $Y$ is a normal space. Also $X_0$ is a normal germ. Indeed if a meromorphic  germ $m_0$ satisfies a relation of integral dependence in $\Oo_{X,0}$, it satisfies a relation of integral dependence in $\Oo_{Y,0}$. But $Y_0$ is normal, hence (the complexification of) $m_0\in\Oo_{Y,0}$. So $m_0\in \Oo_{X,0}$, because it is invariant.
Now $Y_0$ is irreducible and of constant dimension $3$ in a neighbourhood of $0$. Also $X_0$ is irreducible but not equidimensional because at points $(0,0,t,0)$ with $0<t<1$ the space $X$ has dimension $1$. Then it cannot be coherent because $Y_x$ is not the complexification of $X_x$ at points arbitrarily close to $0$.
 \end{example}

\section{Divisors  in  C-analytic  Sets.}

In  this section  we consider  the analogous of second Cousin's problem   for  a  C-analytic  set $X\subset \R^n$, that is,  deciding whether a given divisor is the divisor of a meromorphic function on $X$.

\subsection{Multiplicities.}

Let $X \subset \R^n$  be an irreducible  C-analytic set,  endowed with its well reduced  structure as defined in Lemma \ref{wellreduced} of Chapter 1.

\medskip
 Let  $Y  \subset  X$  be  an irreducible C-analytic  subset  of
codimension  1.  We define  as before the  coherent  sheaf  of ideals  $\Ii_Y$  as
$\Ii_Y = I(Y) \Oo_X$, where $I(Y) = \{ f \in \Oo(X) : \, f = 0
\mbox{ \rm on } Y   \}$.

Suppose that at  some point $x \in Y$ the ideal  $\Ii_{Y,x}$ is principal, say,
$\Ii_{Y,x} = g_x \Oo_{X,x}$ for some $g_x  \in \Oo_{X,x}$. Then, the germ  of any $f
\in  \Oo(X)$ at  $x$ can  be written  as $f_x  = g_x^r  v_x$ for  some non-negative
integer $r$ and  some $v_x \notin \Ii_{Y,x}$.
\smallskip

Note that $r$ does not depend on the generator $g_x$. Indeed, suppose $f_x  =
h_x^s  w_x$ for another generator $h_x$ and, say, $s\leq r$. Then there is a unit
$u_x$ such that $h_x= u_x g_x$, hence $w_x u^s_x g_x^s = v_x g^r_x$, so $v_xg^{r-s}_x =w_x u^s_x \notin
\Ii_{Y,x}$. This implies $s=r$.
\smallskip

Also, since $\Ii_{Y,x}$ is the stalk of a  coherent sheaf, the relation $f  = u g^r$ holds in a neighbourhood of $x$ and $g$ generates
$\Ii_{Y,y}$ for $y$ close to $x$ outside the zero set of $u$. In particular the integer $r$ is the same for $x$ and $y$.
\smallskip

The integer $r$  will be called
{\it multiplicity of $f$ along $Y$ at the point $x$}\index{multiplicity of a function!along a divisor at a point} and will be denoted as
$m_{Y,x}(f)$.  The multiplicity of a meromorphic function $\displaystyle f =
\frac{f_1}{f_2} \in {\mathcal M}(X)$ where $f_1, f_2
\in \Oo(X)$ (and $f_2$ is not  a zero divisor of $\Oo(X)$) is defined as
$m_{Y,x}(f) =  m_{Y,x}(f_1) - m_{Y,x}(f_2)$. It  is straightforward to
check that
$$V_{Y,x} = \{ f \in {\mathcal M}(X) : \, m_{Y,x}(f) \geq 0 \} \supset
\Oo(X)$$ is a discrete valuation ring.\footnote{Let $K$ be a field. A {\em valuation ring} $V\subset K$ is a ring such that for each $ a \in K$ either $a$ or $a^{-1}$ belongs to $V$. Associated to $V$ there is a {\em valuation group }$\Gamma = K^*/V^*$ together with a {\em valuation map} $v: K^* \to \Gamma$. The valuation ring is called {\em discrete} if $\Gamma =\Z$.} 
\medskip

First of all  we prove that given $Y$ as above we can  find a {\em uniformizer}
of the multiplicity along $Y$, that is an $h \in
\Oo(X)$  generating  $\Ii_{Y,x}$ for  almost all points  $x \in  Y$. We
recall  that a  C-analytic  subset  $W \subset  X$ always  admits a  {\it
positive equation}, that  is, a nonnegative function $g \in \Oo(X)$ whose zero
set  is $\ceros(g) =  W$.  One  can take,  for instance,  $g  = f_1^2+\dots
+f_q^2$, where $ f_1,\ldots, f_q  \in \Oo(X)$ are such that $W=\{f_1=0, \ldots,
f_q=0\}$. Note that any such equation has multiplicity greater than 1 over $Y$.
In particular we will get
that $m_{Y,x}$ and consequently also $V_{Y,x}$ do not
depend on the point  $x \in Y$, provided $\Ii_{Y,x}$ is  principal.

\begin{lem}\label{uniformizer}
Let $Y  \subset X$ be an  irreducible C-analytic subset  of codimension $1$
such  that $\Ii_{Y,p}$  is  principal  for some  $p  \in Y$.  Then  there is  a
uniformizer $h \in \Oo(X)$  such that $h_x \Oo_{X,x} = \Ii_{Y,x}$ for all
$x \in Y$ up to a real analytic set of codimension $1$ in $Y$.
Moreover,  given  any global  analytic  subset $Y'  \subset  X$  such that  $Y
\not \subset Y'$ the uniformizer $h$ can be chosen so that $\ceros(h) \cap Y'$
has no components of codimension $1$ in $X$.
\end{lem}

\begin{proof}
Assume $\Ii_{Y,p} = g \Oo_{X,p}$.
By Theorem A there is a finite number of global analytic functions on
$X$ which generate the ideal $\Ii_{Y,p}$. At least one of these functions, call
it $f$, has multiplicity $1$ at $p$ along $Y$.

Let $\widetilde  X, \widetilde  Y $ be complexifications of  $X$ and
$Y$ respectively, in an open Stein neighbourhood  $\Omega\subset \C^n$ of $\R^n$.    Up  to shrinking  $\Omega$ the function $f  \in \Oo(X)$ can be
extended to an  analytic function on $\widetilde X$, which will still be called
$f$. The ideal $\Ii_{\widetilde {Y},p} = I(\widetilde Y) \Oo_{\widetilde{ X},p}$ is also principal,
generated by the same $g$ and
$f_p = v_p g_p$, where $v_p \in \Oo_{\widetilde {X},p}
\setminus\Ii_{\widetilde{ Y},p}$.
Then, in a small complex neighbourhood $U$ of $p$ where $v$ is defined,
  $f_x$
generates $\Ii_{\widetilde{Y},x}$ for all  $x \in \widetilde Y \cap U \setminus \{v=0\}$.
This last set is not empty, because
$\widetilde Y$ is  pure dimensional. Thus, the set of points at  which $f_x$ is a
generator of $\I_{Y,x}$ is not empty.

Consider the coherent sheaf of ideals  $\Jj$ defined by ${\Jj}_x = (f_x \Oo_{\widetilde{
X},x}:\Ii_{\widetilde {Y},x})$, where  $x \in \Omega$, that is, $h_x  \in {\Jj}_x$ if and
only if $h_x\Ii_{\widetilde {Y},x} \subset f_x \Oo_{\widetilde{ X},x}$.  Thus ${\Jj}_x = \Oo_{\widetilde{X},x}$ if and only if $f_x$ generates $\Ii_{\widetilde {Y},x}$.  Therefore, the support

$$\mbox{supp} \,(\Oo_{\widetilde X}/  {\Jj}) = \{x  \in \widetilde X \, :  \, f_x \mbox{
does not generates }\I_{\widetilde{Y},x} \}$$
is a closed analytic set $\widetilde W$ that does not contain $\widetilde Y$. As $\widetilde Y$
is  irreducible, $\widetilde Y  \cap \widetilde  W$ has  at least  codimension $1$  in $\widetilde Y$. Hence, $f_x$ generates $\Ii_{\widetilde{Y},x}$ for all $x \in \widetilde Y \setminus \widetilde W$.  Then, also $f_x$ generates $\Ii_{Y,x}$ for all $x \in Y \setminus W$, where
$W = \widetilde W \cap \R^n$.

Note that $W \cap Y$ is a subset of codimension at least $1$ in $Y$. Suppose
that $W  \supset Y_{\rm max}$, where  $Y_{\rm max}$ denotes the  part of maximal  dimension of
$Y$. Then it would be $W \supset Y$ and so $\widetilde W \supset Y$. But as $\widetilde Y$
is the complexification  of $Y$, this  would imply  $\widetilde W \supset 
\widetilde Y$, which is a contradiction.

Now, let $Y'$ be any analytic  set not containing $Y$. Take positive equations
$f_Y, f_{Y'}  \in \Oo(X)$ of  $Y$ and $Y'$,  respectively. Then $f' =
f_{Y'}f + f_Y$ has the required properties.
\end{proof}

Thus, we write  $m_Y$ and  $V_Y$ for the multiplicity along $Y$\index{multiplicity of a function!along a divisor} and
its valuation ring.
The following proposition  gives another characterization of $V_Y$.

\begin{prop}\label{valuation}
Let  $Y  \subset  X$  be  an irreducible global  analytic  subset  of
codimension $1$ such  that $Y \cap \mbox{\rm Reg}( X) \neq \varnothing$. Then
$V_Y = \Oo(X)_{I(Y)} = S^{-1}\Oo(X)$ where $S$ is the complement of $I(Y)$  in $\Oo(X)$.
In particular $m_Y$ is a real valuation.
\end{prop}

\begin{proof}
It   is  straightforward  to   check   that   $V_Y \supset\Oo(X)_{I(Y)}$.

To prove the converse inclusion take  some  point  $x \in  \mbox{\rm Reg}( Y)  \cap
\mbox{\rm Reg}( X)$. Such a point exists because otherwise $\mbox{Reg} (Y) \subset
\mbox{\rm Sing}( X)$ and then $Y  \subset \mbox{\rm Sing}( X)$. Let ${\gtM}_x \subset  \Oo(X)$ be the  ideal of analytic  functions on $X$  vanishing at
$x$.  As  ${\gtM}_x \supset I(Y)$,  we have that  $\Oo(X)_{I(Y)} = (\Oo(X)_{{\gtM}_x})_{I(Y)}$.  The ring  $\Oo(X)_{{\gtM}_x}$ is  regular,
 so its  localization at $I(Y) \Oo(X)_{{\gtM}_x}$ ( which
is  a  prime  ideal  of  height  one) is  a  discrete  valuation  ring.  Hence,
$\Oo(X)_{I(Y)} \supset V_Y$.

Finally,  note  that the  residue  field of  $\Oo(X)_{I(Y)}$  is  the field  of
meromorphic functions on $Y$, which is a real field.
\end{proof}

\subsection{Divisors.}

Let $X$ be a C-analytic set in $\R^n$. Set $q =$ dim $X$. 

\begin{defn}
Let $\{Y_i\}_{ i\in J}$ be a locally finite family of irreducible C-analytic subsets of $X$, with for every $i$
${\rm dim} Y_i = q-1$ and $Y_i \cap \Reg(X) \neq \varnothing$ for every $i$. We  call a {\em divisor} in $X$ the 
formal sum $  \sum_{i\in J} n_i Y_i$ where $n_i\in \Z$.
As usual,
the divisor is called {\em reduced} if for all $i$ $n_i =1$ and {\em positive}
when $n_i >0$.
The {\em support} of a divisor is the C-analytic set $Y= \bigcup_i Y_i$. 
It is a C-analytic subset of $X$,
because the family $\{Y_i\}_{ i\in J}$ is locally finite.
The $Y_i$ in the family are called {\em components } of the divisor.

Finally, we say  that two divisors $Y, Y'$ are {\em  coprime} if their supports
do not share any irreducible component.
\end{defn}

\smallskip
The set $\mathcal D$ of divisors has a natural structure of abelian group.

\smallskip

The  multiplicities $m_{Y_i}$  along  the  components of  a  divisor are  well
defined. We  shall say that $Y  = \sum_{i\in J} n_i  Y_i$ is the  divisor of an
analytic function $g$ and we shall write  $Y = {\rm div} (g)$ if $m_{Y_i}(g) =
n_i$ and the zero set of $g$ is the support of $Y$. In this case we  call
$Y$ {\em principal}.  \smallskip

Let $Y$ be (the support of) a divisor. 
By classical results  on triangulations,  we may
find a locally finite triangulation of  the couple $(X,Y)$. This means that we
have  a  simplicial complex  $K$  together with  a  subcomplex  $K_Y$ and  a
homeomorphism $f:K \to  X$ such that $f(K_Y)=Y$ and  for each simplex $\sigma$
of  $K$,   the  restriction  $f_{|\interior{\sigma}}$   is  an  analytic
isomorphism.    So,   for   any    $j$   we   have   isomorphisms   
$$f_*:{\rm
H}^{\infty}_{j}(K,\Z/2\Z)\to  {\rm H}^{\infty}_{j}(X,\Z/2\Z)$$

where H$^{\infty}_{j}(X,\Z/2\Z)$ is the  homology group based on
infinite chains. 

 Also, by  the construction  above,  each component  $Y_i$ of  the
 divisor  defines in  a natural  way  an element  $[Y_i]$ in  the group  ${\rm
 H}^{\infty}_{q-1}(X,   \Z/2\Z)$. Note that a real analytic set carries a fundamental class.   Since
 any   two  such   triangulations  are
 PL-equivalent by Hauptvermutung, this fact allows to define
 a group homomorphism

$$ {\mathcal D}  \to {\rm H}^{\infty}_{q-1}(X, \Z/\Z2)$$

sending the divisor  $Y$ to the class $\sum_i n_i [Y_i]$.
\smallskip

From now on we  use the same symbol for both a divisor and its support
when there is no risk of confusion.
\smallskip

The following theorem  finds for any  reduced divisor $Y$  of $X$ what  we will call  a {\it
generic equation}, that  is, an  analytic function $h$ vanishing
on $Y$ with multiplicity 1 along each component $Y_i$ of $Y$.

\begin{thm}\label{generic equation} 
Let $Y =  \sum_i Y_i $ be a reduced divisor of $X$. Then there  exists $h \in
\Oo(X)$ such that $m_{Y_i}(h) = 1$ for all $i \in I$.
In particular, $h$ changes sign at every point of maximal dimension of $Y$ and
it is a local generator of $\Ii_{Y,x}$ for  all $x \in Y$ up to an analytic set
of codimension $1$ in $Y$.

Moreover, given  any C-analytic subset  $W \subset X$  not containing any
component of $Y$, the function $h$ can be  chosen  such that ${\ceros}(h) \cap W$
has codimension at least $2$ in $X$.
\end{thm}

\begin{proof}
For each $x\in X$ we set $J_x$ for the finite set of indices $i\in J$ such
that $x\in Y_i$. Then
we define the coherent sheaf of ideals ${\Jj}$ as
$$
{\Jj}_x = \left( \prod_{i\in J_x }(h_i)_x\ ,\ \prod_{i\in J_x}
(g_{Y_i})_x \right),
$$
where each $h_i$ is a uniformizer of $m_{Y_i}$ not vanishing on $Y_j$ for $j \ne i$, given by  Lemma \ref{uniformizer}, and
$g_{Y_i}$ is a positive equation of $Y_i$.  Since the stalk ${\Jj}_x$ is generated by at most two functions for every $x$, the sheaf ${\Jj}$
is globally generated by finitely many global sections $f_1, \ldots, f_r$ (see \cite {Co}).  Note that for each $Y_i$ at
least one $f_j$ has multiplicity one along $Y_i$.

Set $I_0 = \varnothing$ and define $ \displaystyle I_j = \{i \in  J \,:\,
m_{Y_i}(f_j)  = 1  \} \setminus  \bigcup_{t=0}^{j-1} I_t$.

We define the functions $f'_j = f_j + e_j$, for $j= 1, \ldots, r$ where $e_j$ is a
positive equation  of   $\bigcup_{i \in  I_j} Y_i$.

Note that $\displaystyle {\ceros} (f_j') =\bigcup_{i \in I_j} Y_i$  and $m_{Y_i}(f'_j) = 1$
 for  $i  \in I_j$.  Moreover  for  each $Y_i$  there  is  exactly one  $f'_j$
 vanishing on $Y_i$. So $f = f'_1 \cdot f'_2 \cdots f'_r$ has multiplicity one along each $Y_i$.

Finally,  if  $g_Y, g_{W} \in
\Oo(X)$ are  positive equations of $Y$  and $W$, respectively, the  zero set of
$ f' = g_{W}f + g_Y$ cuts $W$ along a set of codimension at least $2$, as in the  proof of Lemma \ref{uniformizer}.
\end{proof}

This theorem  says that the zeroset of $h$ can be written as
$Y \cup Y'$  for some analytic set
  $Y'$. In general  we can say few about the set
$Y'$ of ``extra'' zeroes of the function $h$, except that, if it is a divisor,
it is coprime with $Y$ and can be chosen coprime with any divisor $W$ fixed in
advance.

Thus, two  questions arise. 

\begin{quotation}
Is it possible  to find a generic  equation of $Y$
being a local  generator of $\I_{Y,x}$ at  every $x \in Y$?   And,  what can be
said about  $Y'$? 
\end{quotation}

In  the next  section we will  answer these  questions under
some additional hypotheses on the sheaf $\I_Y$ and on the space $X$.

If $Y$  is a  divisor, then a  positive equation  of $Y$ has  even multiplicity
along each component  $Y_i$ of $Y$. What  we show in the next  theorem is that
given any sequence of even positive integers $\{m_i = 2n_i\}_{i \in J}$ we can
find a positive  equation of $Y$ with precisely  multiplicity $m_i$ along each
$Y_i$.

\begin{thm}\label{poseq} 
Let $Y = \sum_i 2 n_i Y_i$ be a positive  even divisor.   Then, there  is a
positive analytic function  $f$  such that  $Y = {\rm div}(f)$.
\end{thm}

\begin{proof}
With the same notations of the previous theorem we define the coherent sheaf
$$
{\Jj}_x = \left( \prod_{i\in J_x}{(h_i^{n_i}})_x\ ,\ \prod_{i\in J_x}
(g_{Y_i}^{n_i})_x \right).
$$
Again    this sheaf  is  generated by  a finite  number of  global
sections  $f_1,  \ldots,f_r$.  Let  $f=   f^2_1  +  \ldots  +  f^2_r$.  It  is
straightforward to see that $f$ is a positive equation of $Y = \bigcup_i Y_i$.

Now, for a given $Y_i$ take some  point $x \in Y_i \setminus \bigcup_{j \neq i}
Y_j$  such  that  ${h_i}_x$  generates  $\Ii_{Y_i,x}$.   Then,  ${h_i^{n_i}}_x$
generates  ${\Jj}_x$   so,  the ideal $ ({h_i^{n_i}}_x)=  {\Jj}_x   =  ({f_1}_x,  \ldots,
{f_r}_x)$.  Thus, $m_{Y_i}(f_\ell) \geq n_i$ for all $\ell = 1, \ldots, r$ and
$m_{Y_i}(f_k) = n_i$ for some $k$.   As $m_{Y_i}$ is a real valuation, we have
 $\displaystyle m_{Y_i}(f) = 2 \min_\ell \{m_{Y_i}(f_\ell)\} = 2n_i$.
\end{proof}

As a  corollary of the  last two theorems, we prove that for any divisor $Y
=\sum_i n_i Y_i$ there  is a meromorphic  function $f$ such that $m_{Y_i} (f) =
n_i$ for each $i \in  I$.  But  note again that, unless all multiplicities
are  even, the set of points where $f$ is zero or not analytic can be 
larger than $\supp Y$.

\begin{cor}
Let $Y =  \sum_i m_i Y_i $ be  a divisor in $X$ where $\{m_i\}$ is
any  sequence of integers.  Then, there  is $f  \in \mathcal M(X)$  such that
$m_{Y_i}(f)=m_i$ for all $i \in I$.
\end{cor}

\begin{proof}
Write $m_i=2n_i$  or $m_i=2n_i +1$ according  to the parity of  $m_i$.  By 
Theorem  \ref{poseq}  there  is  a  sum  of  squares  $h_-  \in  \Oo(X)$  such  that
$m_{Y_i}(h_-) = 2\vert  n_i\vert$ for all $i  \in I$ such that $n_i  < 0$ with
$\ceros(h_-) =
\bigcup_{n_i< 0}Y_i$. Similarly there is a sum of squares $h_+ \in \Oo(X)$ such
that $m_{Y_i}(h_+) = 2n_i$ for all $i \in I$ such that $n_i > 0$.  Take $g \in
\Oo(X)$ such  that $m_{Y_i}(g)=1$ when  $m_i$ is odd  and not vanishing  on any
$Y_i$ such that $m_i$ is even.

Then,  $f=g \dot h_+/ h_-$ has the  required multiplicities.  To check  this just  note that
$m_{Y_i}(f) = m_{Y_i}(g)+ m_{Y_i}(h_+) - m_{Y_i}(h_- )$.
\end{proof}

\subsection{Locally principal divisors.}

Let $X$ be again an irreducible  C-analytic set in  $\R^n$ of dimension  $q$ endowed with its well reduced structure.

Let $Y \subset  X $ be a  reduced divisor. We are interested  in the following
 question.
\begin{quotation}   
\rm Under  what hypotheses   the ideal  $I(Y)$ is a principal  ideal in
 $\Oo(X)$, that is,  there is $g\in \Oo(X)$ such that $Y  = {\rm div}(g)$?  
\end{quotation}

Note
 that if a  global function $g$ generates $I(Y)$, then $X  \setminus Y$ has at
 least two  connected components and the set  $\{y\in Y | \ {\rm dim}_y Y=
 {\rm dim}_y X -1\}$ bounds one of the  regions where $g$ has a given sign.  So, in
 order to have  $Y = {\rm div}(g)$   $Y$  disconnects $X$ and
 the  class  $[Y]$ must vanish  in  the  group ${\rm  H}^{\infty}_{q-1}(X,\Z_2)$.
 However these conditions are not sufficient, as the following example shows.

\bigskip

\begin{example}\label{bicoppa}
\rm 
Consider the set

$$X = \{(x,y,z) \in \R^3 | \ z^2 = x^4 +y^4 \}$$

It is an irreducible algebraic variety and also an irreducible C-analytic set. This can be seen as follows. Since the only singularity of $X$ is the point $(0,0,0)$ if $X$ were reducible, its components would be $X\cap \{z \geq 0\}$ and $X\cap \{z \leq 0\}$. Since both have codimension $1$ the only possibility is to split the equation as $z^2 = x^4 +y^4 = (z- \sqrt{x^4+y^4})(z+\sqrt{x^4+y^4})$ but the factors are not analytic at $(0,0,0)$.

If $Y =\{p \in X |\  x=0, z \geq 0\}$, then
$Y$ is  a parabola  and $[Y] =  0$ in  H$^{\infty}_{q-1}(X, \Z/2\Z)$.
This  is clear
because $Y$ is the boundary of the open set $\{ x>0, z>0\} \cap X$. Nevertheless
the ideal $I(Y) \Oo_{X,0}$ is not principal; this is  because $X$ is irreducible,  so the function $x\in \Oo(X)$ defines a double parabola $Y\cup Y'$. 
\end{example}

\bigskip

The previous  example shows  that  another  necessary  condition for  $Y$ to be
 principal is that the ideal sheaf $I(Y)\Oo_X$ is {\em locally principal}, that
 is, for  any point $x\in  X$ there is  a function germ $f\in  \I_{Y,x}$ that
 generates the  stalk $\Ii_{Y,y}$ for any  $y$ in a neighbourhood  of $x$.  So,
 from now  on we  will assume  this condition.

In order to generalise this result
 in the singular case we shall  use some classical exact sequences of coherent
 sheaves and the vanishing of the cohomology  with coefficients in a coherent
 sheaf.

As a first step we improve Theorem \ref{generic equation}.  We will find a
function $h$ vanishing with multiplicity $1$  not only along $Y$, but along its
whole zero set.

\begin{prop}\label{fratello}
Let $X \subset \R^n$ be  a C-analytic set and let $Y \subset
X$ be  a reduced divisor such  that $\I_Y$ is locally  principal.  Then, there
are  an open  neighbourhood  $U$ of  $X$ in  $\R^n$  and global  analytic
hypersurfaces $W$,  $W'$ of $U$  such that $I(W\cup  W') $ is generated  by an
analytic  function $h\in  \Oo (U)$,  $W\cap  X= Y$  and $W'$  is an  analytic
manifold transversal  to $X$ and  $Y$. Hence, setting  $Y'= W'\cap X$  we have
that  $I(Y\cup  Y')\subset  \Oo(X)$  is  generated by  $h_{|X}$,  and  $h$ generates the stalk $\Ii_{Y,x}$  at any point $x \in Y \setminus
Y'$.
\end{prop}

\begin{proof}
Since  $\Ii_Y$  is  locally  principal, for  any  $x$  there  is  a  germ  $f_x\in\Oo_{X,x}$ that generates $\Ii_{Y,y}$ for any $y$ in an open neighbourhood
$U_x$ of $x$.  So, refining the  open covering $\{U_x\}_{x\in X}$,
we  can find a countable locally finite open covering $\{U_i\}$
of $X$, analytic  functions $f_i$ on $U_i$ such  that $f_i \Oo_{X,x} =\Ii_{Y,x}$
for any $x\in  U_i$. Hence, $f_i/f_j$ is invertible on  $U_i \cap U_j$.  These
functions define an  analytic cocycle in $H^1(X, \Oo^*)$,  that is, an analytic
line bundle  $\mathcal{F}$ on  $X$. The collection  $\{f_i\}_i$ defines a  section of
$\mathcal{F}$ vanishing exactly on $Y$.

We  may find   open sets
$V_i \subset \R^n$ such that $V_i \cap X = U_i$.
Since $\Oo_X$  is  a coherent  sheaf  each $f_i$  extends to  an
analytic function $F_i$ on the open set $V_i$ of $\R \sp n$.

AS $f_i/f_j$ is invertible,  after shrinking the $V_i$'s, we
may assume that $F_i / F_j$ is  invertible on $V_i \cap V_j$. So, $V =
\bigcup_i  V_i$ is  an open  neigbourhood  of $X$  in $\R^n$  and the
functions
$F_i/F_j  : V_i  \cap  V_j  \to \Oo^*$ define  an  analytic line  bundle
$\mathcal{G}$ on  $V$ that extends $\mathcal{F}$.  The  collection $\{F_i\}$ defines  an analytic  section of
$\mathcal{G}$ whose zero set $W$ cuts $X$ along $Y$.

We may find  an analytic section $G  = \{ G_i \}_i$ of  $\mathcal{G}$ transversal to the  zero  section,  hence, its  zero  set  is an  analytic
manifold $W'$ in  $V$.  We want to  prove that $I(W \cup W')$  is principal in
$\Oo(V)$. Consider the  line bundle  defined by
$F_i G_i$  whose cocycle is $(F_i/ F_j)^{2}$. We have  to prove that
this cocycle  (and hence the  line bundle) is  trivial, that is, $(F_i  G_i) /(F_j
G_j) =
\lambda_j^{-1} \lambda_i$,  where $\lambda_i  \in \Oo^*_n (V_i)$.   If this happens,
$\{F_i G_i \lambda_i\}$  glue together and give a generator  $h$ for $I(W \cup
W')$. Define $Y' = W' \cap X$: then, ${h}_{|X}$ generates $I(Y \cup Y')$.

Consider  the exponential  map  and  the associated  usual  exact sequence  of
coherent sheaves.
$$
0 \rightarrow \Oo _X \rightarrow \Oo _X ^{\ast}
\rightarrow \Oo _X ^{\ast} / \Oo_X ^+ =
\Z/2\Z \rightarrow 0.
$$
where $\Oo _X ^{\ast}$ is the sheaf of units and $\Oo_X ^+$ is the sheaf of positive units.

Since $H^i(X, \Oo_X) =  0$ for $i >0$, it induces an  isomorphism between $ H^1
(X, \Z/2\Z) $ and $ H^1 (X, \Oo _X ^{\ast} )$.  Under this isomorphism the image
of a line bundle is the cocycle  of the signs of its transition functions.

In our case $(F_i  G_i) /(F_j G_j) = F^2_i/F^2_j$, so the line bundle is trivial and the proof is complete.
\end{proof}

\begin{remark}
 It is  easy to check that if $\{ Y_i \}_{i  \in I}$ is a locally
finite  family of  locally principal  divisors then  $\bigcup_i Y_i$  is  also a
locally principal  divisor. On the contrary, there are locally principal
divisors with some components that are not locally principal.

For example, take $X \subset \R^3$ as the cone of equation $z^2 = x^2 + y^2$
and consider the divisor $Y = \{ x = 0 \} \cap X$ which is locally principal
with generator $g=x$. The divisor $Y$ splits into two straight lines $Y_1$
and $Y_2$ neither of which are locally principal.
\end{remark}
\bigskip

The following is the main result of this section.

\begin{thm}\label{mainth}
Let $X$ be a $q$-dimensional irreducible  C-analytic set in $\R^n$, so that    
the ideal $I(X) \subset \Oo(\R^n)$ is prime.  Let $Y\subset X$ be a reduced divisor such that its  ideal sheaf \,  $\Ii_Y =I(Y)\Oo_X$ is  locally principal. 

Assume that
$[Y] = 0$ in ${\rm H}^{\infty}_{q-1}(X,\Z/2\Z)$  and that $X \setminus Y$ is not
connected.  Then, there is an open neighbourhood $U$ of the set $X_{\rm max} =
\{x\in X :\,  {\rm dim}_x X=q\}$  and $g \in \Oo_X(U)$ such that
$\Ii_{Y,x} = g \Oo_{X,x}$,
for any $x\in U$. In particular, if $X$ has pure dimension, then $Y = {\rm
div} (g)$.
\end{thm}

\begin{proof}
Arguing as in  Proposition \ref{fratello} we have to prove that there exists $U$ as stated
such that the line bundle $\mathcal{F}$  defined by $Y$ is trivial when restricted
to $U$. To do  this it is enough to find local  generators $\{f_i\}$ of $\Ii_Y$
on a countable open covering $\{U_i\}$ of a neighbourhood of $X_{\rm max}$ such
that ${f_i}_{|U_i \cap U_j}$ and ${f_j}_{|U_i \cap U_j}$ have the same
sign.

Take a  locally finite triangulation $f:K\to X$ of the  couple $(X,Y)$, where
$K$  is a  simplicial complex  and  there is  a subcomplex  $K_Y$ such  that
$f(K_Y)=Y$.   In  particular, for  any  $j$  we  have isomorphisms  $f_*:{\rm
H}^{\infty}_{j}(K,\Z/2\Z)\to {\rm H}^{\infty}_{j}(X,\Z/2\Z)$.

The fact  that $[Y] =0$  means that the union of all $q-1$ simplexes in $K_Y$
bounds some subcomplex $H$ of $K$. The boundary of the region $f(H) \subset X$
is the set $Y_{\rm max}$ of points in $Y$ where $Y$ has dimension $q-1$.

  Thus for each  $j$ such that $U_j  \cap X_{\rm max} \neq \varnothing  $, we may
choose local generators $g_j \in \Oo_X(U_j)$ in such a way that:

\begin{itemize}
\item if $f(H)\cap U_j \neq \varnothing$, then  $g_j> 0$    on $(f(H) \cap U_j) \setminus  Y$;
\item if $f(H) \cap U_j =\varnothing$, then $g_j \geq 0$ when $U_j$ and $f(H)$ lie in the same connected component of $X\setminus Y$ and  $g_j\leq 0$  on $U_j$ otherwise.
\end{itemize}

Hence,  $g_i/g_j >  0$  on  $U_i \cap  U_j$  and $\mathcal  {F}$  is trivial,  when
restricted  to  $U$.    This  means  that  we  can   find  analytic  functions
$\lambda_i  \in \Oo^*(U_i)$  such that  $g_i/g_j =  \lambda_j/ \lambda_i$.
So,    the    sections    $g_i     \lambda_i:    U_i    \to    \Oo_X$    satisfy
${g_i\lambda_i}_{|U_i  \cap U_j} =  {g_j \lambda_j}_{|U_i  \cap U_j}$,
that  is, they  define  an analytic  function $g$  on  $U =  \bigcup U_i$.  By
construction $ g_x$ generates $\Ii_{Y,x}$ for any $x\in U$.
\end{proof}

\begin{cor}
Under the hypothesis  of Theorem \ref{mainth}, if $X$ is pure dimensionsional, then
$Y$  is principal. In  particular, when  $X$ is  coherent, $Y  = {\rm
div}(g)$ for some $g \in \Oo(X)$.
\end{cor}

\begin{remark}\label{moreabout}
If  $X$ is a  manifold, the condition  $[Y] = 0$ implies  that $Y$
 divides $X$  in two or  more connected components  and it is the  boundary of
 some of them. This is not  true in general, not even in the case
 of a coherent singular space $X$. As  an example, one can consider $X$ as  a
 real $2$-dimensional  torus with  one meridian collapsed  to a point:  one can
  write an analytic function on $\R^3$ with such a zeroset. Take as $Y$
 any other meridian. Then $[Y] = 0$  since it is homotopic to one point, and
 of course  $Y$ is locally principal,  but it cannot be  principal because its
 complement is connected.
\end{remark}
\smallskip

As a consequence  of Theorem \ref{mainth} we have the  following 
result, similar to the one on polynomial ideals proved  by Shiota. Recall that an  ideal $\mathfrak {p}$ is \index{real ideal} {\em real} when $\sum_i a_i^2 \in \mathfrak{p}$ implies $a_i \in  \mathfrak {p}$ for all $i$.

\begin{cor}
Let $X$ be a  coherent C-analytic set in $\R^n$ and assume that the ideal $I(X) \subset\Oo(\R^n)$ is prime. Let $\mathfrak {p}_i$ for $i\geq 1$ be  prime real ideals in
$\Oo(X)$ of height $1$ for $i\geq 1$. Denote by $Y_i$ the associated divisor, that is the zero set
of $\mathfrak {p}_i$,  and assume that the family $\{  Y_i\}_i$ is locally finite  and that for any $i$ the
ideal sheaf\,  $\mathfrak{p}_i
\Oo_X$ is locally principal. 
Then, the ideal
$ \prod_i \mathfrak{p}^{m_i}_i$
is principal if and only if the cycle
$ \sum_i m_i [Y_i]$ vanishes in $ {\rm H}^{\infty}_{q-1}(X, \Z/2\Z)$
and
$X \setminus \bigcup_{m_i \, odd} Y_i$
is not connected.
\end{cor}

\begin{proof}
Since the family of irreducible divisors $\{ Y_i\}_i$ is locally finite,
the ideal sheaf $
(\prod_i \mathfrak{p}^{m_i}_i)\Oo_X$ is also locally  principal. Write $m_i=2k_i$
or $m_i=2k_i+1$ according to the parity of $a_i$.
 Split the  class $ \sum_i
m_i [Y_i]$  as 
$$\sum_i 2k_i [Y_i] +  \sum_{i}(2k_i+1) [Y_i].$$

  The ideal  sheaf $ \Jj=\ (\prod_i  \mathfrak{p}^{2k_i}_i)\Oo_X$ is  principal. In
fact, it is  locally generated by a  square, hence, arguing as in  the proof of
Theorem \ref{mainth},  its associated line bundle is  trivial, so  we can find  a global section  $g$ of $\Jj$ such  that $g_x$
generates its stalk at any  point $x\in X$. In addition $Y=\sum_i(2k_i+1) Y_i$
satisfies   the  hypothesis  of   Theorem  \ref{mainth},   so  its   ideal  is
principal, say  generated by $f$. Thus $fg$ generates $  \prod_i \mathfrak{p}^{m_i}_i$ as required.

The converse is clear.
\end{proof}

\bigskip

\subsection*{ Bibliographic and Historical Notes.}\rm

Our definition of irreducible components of a complex analytic set comes in a natural way from the local description of a complex analytic set germ $X_x$ given in Chapter 1. There we saw that the closure of a connected component of an open set of regular points, or better of its germ at $x$, provides an irreducible component of $X_x$.

A more sophisticated notion of irreducible component can be found in \cite{wb}, where irreducible components of a C-analytic set were defined.

When $X$ is a closed subspace of $\C^n$ or of an analytic manifold $M$ one can look at the ideal $I(X)$ of holomorphic functions vanishing on $X$ and it isstraightforward to prove that $X$ is irreducible if and only if  $I(X)$ is a prime ideal.

\smallskip

Concerning normalisation the main reference is \cite{gare}, but this notion was already considered in the works of Oka, who proved the coherence of the normalization sheaf, and in S\'eminaire Cartan, 1953-1954, expos\'e 10 \cite{c1}. The criterion for an analytic space to be normal can be found in \cite{ab} pag 435.

\smallskip

Multiplicity along a divisor was introduced for curves in a normal analytic surface by Andradas, Diaz-Cano and Ruiz in \cite{AnDCRz} and their argument extends to divisor in a C-analytic set, see \cite{acdc}.  
For more details on valuation rings see \cite[Chapter 10]{bcr}.
The argument in the proof of Remark  \ref{moreabout}  comes from \cite{moreabout}. Proposition \ref{fratello} was already proved in \cite{BrPie} when the ambient space is a real analytic manifold. For analytic vector bundles on open sets of $\R^n$ one can see \cite{tog}.
For the definition  and generalities on the groups
${\rm  H}^{\infty}_{j}(X,\Z/2\Z)$  we  refer  to  \cite{massey}.
The fact that a sheaf of ideals that is locally finitely generated with a bounded number of generators is in fact globally finitely generated is a classical result of Coen \cite{Co}.

Corollary 3.13 is the analogous of an algebraic result of  Shiota \cite{sh} that may be found in \cite[12.4.1]{bcr}.

Hauptvermutung, that is, uniqueness of the P-L type of triangulations, was proved in \cite{shio}.

The section on divisors is strongly inspired by \cite{acdc}.

\newpage

\parindent=0pt

\chapter{Nullstellens\"atze.}


The main result of this Chapter is the Nullstellensatz for the algebra $A=
H^0(X,\Oo_X)$ where $X$ is either a Stein space or a C-analytic space.

If $A$ is an algebra of real or complex functions on a space $X$,
for any  subset $E\subset A$ we can look at its   zeroset, that is,  $\ceros (E)= \{x\in X |\ f(x)= 0\  \forall  f \in E\}$.  Conversely given a subset $C\subset X$ we consider $ I (C)=\{f\in A |\ f(x)=0\  \forall x\in C\}$. It is clear that $ I(C)$ is  an ideal  in $A$ and the ideal generated by $E$ has the same zeroset  as $E$.

The word  {\it  Nullstellensatz} \index{Nullstellensatz} refers to a  characterization of
the ideal $ I(\ceros (\gta))$, where $\gta$ is an ideal in $A$. An ideal 
$\gta$ is said to have {\it the zero property} if $I(\ceros(\gta)) = \gta$.

 The  first Nullstellensatz is  Hilbert's one 
for complex polynomials. For the real case the result came later and is quite different. More precisely, 

\smallskip

{\rm(Hilbert, 1893).}   If $\gta  \subset \C[x]=\C[x_1,\ldots ,x_n]$ then
 $$ I(\ceros(\gta))= \{f\in \C[x] |\  \exists n \ \mbox{\rm such that} \ f^n\in \gta\}= \sqrt{\gta}.$$

\smallskip

{\rm(Risler, 1970).}   If $\gta \subset \R[x]=\R[x_1,\ldots ,x_n]$  then 
$$  I(\ceros (\gta))= \{f\in \R[x] |  \exists n,  f_1,\ldots , f_k \in \R[x] \ \mbox{ such that }   f^{2n}+ f_1^2+\cdots +f_k^2\in \gta\}= \sqrt[r]{\gta}.$$ 

\smallskip

The results for convergent power series  are very similar. 

\smallskip

 {\rm (R\"uckert, 1933).}   If $\gta  \subset \C\{x\}=\C\{x_1,\ldots ,x_n\}$
then
 $$I(\ceros(\gta))= \{f\in \C\{x\} | \exists n \ \mbox{\rm   such that } f^n\in \gta\}= \sqrt{\gta}.$$

\smallskip

 {\rm (Risler, 1976).} If  $ \gta \subset \R\{x\}= \R\{x_1,\ldots ,x_n\}$   then
 $$I(\ceros (\gta))= \{f\in  \R\{x\}  |  \exists n, f_1,\ldots , f_k \in \R\{x\}  \ \mbox{\rm   such that }    f^{2n}+ f_1^2+\cdots +f_k^2\in \gta\}= \sqrt[r]{\gta}.$$ 

\medskip

The  results for the local analytic and polynomial settings are of the same nature  because   the ring of germs of holomorphic  (resp.  real analytic) functions is noetherian and UFD.

When passing from rings of germs to rings of holomorphic  (resp. real analytic) functions on an open set of $\C^n$ (resp.  $\R^n$) the situation is very different. These algebras are neither noetherian nor UFD, but deeper difficulties arise.
In fact, in these algebras there are  proper prime ideals
 with empty zeroset  and 
 the multiplicity of an holomorphic (or real analytic) function on a sequence of points  can be unbounded, as we see in the following examples.

\vfill\eject

{\sc Examples.} \null \hfill
 
\begin{enumerate} \parindent=0pt
\item Let  $\mathfrak A$  be an  ultrafilter\footnote{A collection $\mathfrak F$ of subset of a given set $X$ is called a {\it filter} if
\begin{enumerate}
\item $\mathfrak F$  is nonempty.
\item  If  $A, B \in\mathfrak F$ then $A\cap B \in \mathfrak F $  
 \item    If  $A \in \mathfrak F$ and $A \subset B$ then $B \in \mathfrak F$.
 \end{enumerate}

Remark that  by (b) $\mathfrak F$ coincides with the entire $\mathcal P(X)$ if and only if $\varnothing \in \mathfrak F$  because  $\forall A\subset X $ $ \varnothing \subset A$. So a proper filter cannot contain simultaneously $A$ and $X\setminus A$. If  $\mathfrak F$  and  $\mathfrak G$ are two filters we say that $\mathfrak F   \prec \mathfrak G  $ if for any $A\in \mathfrak F$  there exists a $B\in \mathfrak G$ such that $B\subset A$.

An {\it ultrafilter} is a filter which is maximal with respect to this order: a filter is an {\it ultrafilter} if and only if for any $A\in \mathcal P(X)$  either $A$ or $X\setminus A$ belongs in $\mathfrak F$.

If $X$ is not finite, the filter generated by the collection of all complements of finite subsets of $X$ is called {\it cofinite filter}. In general this is not an ultrafilter: for instance if $X=\N$ neither the set $A=\{\mbox{odd numbers}\}$  nor its complement, the set of even ones,  belong to the filter.}\index{ultrafilter}  
of  subsets  of $\N$  containing  all cofinite subsets.
 Let $\K$ denote either $\C$ or $\R$. For an analytic
function  $f\in\an(\K)$,  we  denote  by ${\rm  mult}_z(f)$  the  {\em
  multiplicity}    of    $f$    at    the    point    $z\in\K$ \index{multiplicity of a function! at a point} and write
$M(f,m)=\{n\in\N:\  {\rm  mult}_n(f)  \geq  m\}$. Consider  the
non-empty set
$$
\gta=\{f\in\an(K):\ M(f,m)\in\gtA\ \forall m> 0 \}.
$$

Let us  check that $\gta$ is  a prime ideal.  

Indeed, if $f,g\in\gta$,
then    $M(f,m)\cap    M(g,m)\subset    M(f+g,m)$,    because    ${\rm
  mult}_n(f+g)\geq\min\{{\rm  mult}_n(f),{\rm  mult}_n(g)\}$, so
$M(f+g,m)\in\mathfrak A$ for  all $m\geq0$. On the other  hand, if $f\in\gta$
and   $g\in\an(\K)$  then   ${\rm   mult}_n(fg)={\rm  mult}_n(f)+{\rm
  mult}_n(g)$ and so $M(fg,m)\supset M(f,m)\in\gtA$ for all $m\geq0$. Hence, $\gta$ is an ideal.

Next,  suppose that  $f_1f_2\in\gta$,  but $f_1,f_2\not\in\gta$.  Thus,
there         exists          $m_1,m_2\geq0$         such         that
$M(f,m_1),M(g,m_2)\notin\gtA$.  Take  $m_0=\min\{m_1,m_2\}$ and  note
that  $M(f,m_0),M(g,m_0)\notin\gtA$.  Hence, $\N\setminus(M(f,m_0)\cup
M(g,m_0))\in\gtA$, but $M(f,m_0)\cup M(g,m_0)\supset M(fg,m_0)\in\gtA$, which is
a contradiction. Thus, $\gta$ is a prime ideal.

A similar argument shows for $\K=\R$ that $\gta$ is  \sl real,  \rm 
 that is,  if  $\sum_{i=1}^p  a_i^2  \in \gta$  then  $a_i  \in  \gta$, for
$i=1,\ldots,p$.       Indeed,      assume     that      $a=a_1^2+\cdots
+a_p^2\in\gta$. Since
$ \displaystyle 
{\rm mult}_n(a_1^2+\cdots+a_p^2)=2\min\{{\rm mult}_n(a_1),\ldots,{\rm mult}_n(a_
p)\},
$  then $M(a,2m)\subset   M(a_i,m)$  for  all  $m\geq0$  and
$i=1,\ldots,p$.   Thus, since  each $M(a,2m)\in\gtA$,  we  deduce that
$M(a_i,m)\in\gtA$ for all $m> 0$, that is, each $a_i\in\gta$.

Finally, observe  that $\ceros(\gta)=\varnothing$. For  each $k\geq1$,
let    $g_k\in\an(\K)$   be   an    analytic   function    such   that
$\ceros(g_k)=\{n\in\N:\ n\geq k\}$  and ${\rm mult}_n(g_k)=2n$ for all
$n\geq k$. Since $\gtA$ contains  all cofinite subsets, we deduce that
each                 $g_k\in\gta$                and                so
$\ceros(\gta)\subset\bigcap_{k\geq1}\ceros(g_k)=\varnothing $.\qed

\item  
Let $f,g \in \an(\K)$ be the analytic functions given by 
$$f(x) = \prod _n \big( 1 - \frac{x}{n^2} \big) \quad 
\mbox{ and } \quad g(x) = \prod _n \big( 1 - \frac{x}{n^2} \big) ^n.$$ 

We have  that $\ceros  (f) = \ceros(g)  = \{ n^2  \mid n  \in \N ^+  \}$.  Let
$\gta$ be  the ideal generated by  $g$.  Now, if the  Nullstellensatz held for
$\an(\K )$, there would exist an integer $m \geq 0$, a global analytic function
$t  \in \an(\K)  $ and,  in the  real case,  a sum  of squares  $s$  such that
respectively $f^m=gt$ or  $f^{2m} + s= gt$. In both cases  if we compare orders
at the  point $(m+1) ^2$ we  achieve a contradiction.   This computation shows
that $I (\ceros (\gta) ) \neq \sqrt{ (\gta)}$ (resp. $I (\ceros (\gta) ) \neq
\sqrt[r]{\gta}$).
\end{enumerate}

These exemples are  in a certain sense paradigmatic:  in fact, one gets 
results  for an algebra  $A=H^0(X,\Oo)$ of  global sections  of a  Stein space
when one avoids these two difficulties.  One may use  that such an
algebra has  a natural structure of  Fr\'echet space and  that for
this  topology {\em  closed} prime  ideals cannot  have  empty  zeroset and
 they have the zero property, that is, $ I(\ceros(\gtp))=\gtp$
(Corollaries \ref{primary1}  and \ref{primary2}  below). To avoid  the second
difficulty one proves  that any {\em closed} ideal $\gta$  admits, in a similar
way  as an  ideal  in  a  noetherian   ring,  a  locally  finite   irredundant  primary
decomposition $\gta=\bigcap_{i\in  I}\gtq_i$, where each $\gtq_i$  is a closed
primary  ideal and  the prime  ideals $\gtp_i=\sqrt{\gtq_i}$  are  closed and
pairwise different.

In  general infinitely many $\gtq_i$ may appear and ``locally finite'' refers to their zerosets.

\section{Nullstellensatz for Stein spaces.}

By \index{Stein algebra} a  \sl Stein algebra \rm one  means  in general the algebra of global sections of a reduced Stein space. We keep  the same name   in the real case for global sections  of a coherent or C-analytic space.

As we  said in the introduction,  these algebras are neither  UFD nor   noetherian rings as one can prove  in the case of an open set in $\C$. Nevertheless using  their  natural  Fr\'echet structure one can recover some algebraic properties as a consequence of Cartan's theorem about closure of modules.

\subsection{Complex Stein algebras as Fr\'echet spaces.}

Let $X$ be a Stein space and $A=H^0(X,\Oo)$.

\begin{prop} The compact-open topology makes $A$ a Fr\'echet space\footnote{A Fr\'echet space is a topological vector space which is  Hausdorff, locally convex and metric complete. One can see that this structure can be induced by a countable family of semi-norms $\|.\|_k, k = 0,1,2,...$.\index{Fr\'echet space}}. 
\end{prop}

\begin{proof}
We give just a sketch of proof.
Recall that in this topology a  system of neighbourhoods for a function $f$
is given by the sets $U(K,f,\varepsilon)=\{g: \, \sup_K |g-f|< \varepsilon\} $
where $K$ is compact.
For any compact set $K\subset X$ there is  a seminorm on $A$, namely  $\|f\|_K = \sup_{K}|f|$. Since $X$ has an exhaustion $\{K_m\}$ of compact sets with $K_m \subset \interior K_{m+1}$ there is a countable family of seminorms  
$$ \|f\|_m = \sup_{K_m}|f|.$$

The algebra $A$ is complete with respect to the metric induced by this family, hence it is a
 Fr\'echet space.   

\end{proof}

Next, we state  two classic theorems.

\begin{thm}\label{modulichiusi}{\rm (Closure of Modules)}
Let $\mathcal M_0$ be a submodule of $(\Oo_{n,0})^q$. Then there are generators $\{f_1, \ldots, f_p \}$ of $\mathcal M_0$ over $\Oo_{n,0}$ and a fundamental system of open neighbourhoods $\{U_m\}_{m\in \N}$ of $0\in \C^n$ such that $f_1, \ldots, f_p $ are defined on $U_m$ for any $m$. Moreover
\begin{enumerate}
\item If $h$ is a $q$-uple of holomorphic functions on $U_m$ such that its germ $h_0$ belongs to $ \mathcal M_0$ and  $\|h\|_{U_m} <\infty$, then there are $\alpha_1, \ldots, \alpha_p \in \Oo(U_m)$ such that $h= \sum_{i=1}^p \alpha_i f_i$. 
\item If $\{h_k\}_{k\in \N}$ is a sequence of $q$-uples of holomorphic functions on an open set $U\ni 0$ that converges uniformely on the compact sets of $U$  to $h$ such that $(h_k)_0 \in \mathcal M_0$, then $h_0 \in \mathcal M_0$. 
\end{enumerate}  
\end{thm}

\begin{thm}\label{Runge}{\rm (Runge's approximation Theorem)}
 Let $(X, \Oo)$ be a reduced Stein space and $U\subset X$ be an open set. The following statements are equivalent.
\begin{enumerate}
\item For any $f\in H^0(U, \Oo)$, any compact $K\subset U$ and any positive $\varepsilon$ there is $g \in H^0(X, \Oo)$ such that $\|f-g\|_K <\varepsilon$.
\item $U$ is Stein.
\end{enumerate}  
\end{thm}
A first important result is the following. 

\begin{thm}\label{Moduli Chiusi Stein}  Let $(X,\Oo)$ be a Stein space. 
\begin{enumerate}
\item Let $\Ii$ be  a sheaf of ideals of $\Oo$. Then $H^0(X,\Ii)$ is closed in $H^0(X,\Oo)$.
\item Let $E$ be a subset of $H^0(X,\Oo)$. Then the sheaf\, $\Ff$ generated by $E$ is coherent.
\end{enumerate} 
\end{thm}
\begin{proof} \null\hfill

\begin{itemize}\parindent= 0pt
\item [Step 1.] \it It is enough to prove {\rm (1)} and {\rm (2)} for a Stein open set $\Omega$ in $\C^n$.\rm

In fact  (1) and (2) for  $ \Omega$ imply (1) and (2) for any closed analytic subspace $Y\subset \Omega$, since any analytic function on $Y$ extends to $\Omega$ by Theorem B. So we get  (1) and (2) for any local model of $(X, \Oo_X)$. This is enough to get (2), because the statement concerns a local property. 
For (1), we get $H^0(U,\Ii)$ is closed in $H^0(U,\Oo_U)$ for any local model $U$ of $X$. But this is enough to get   $H^0(X,\Ii)$ closed, because if $f\in A$ is the limit of a sequence in $H^0(X,\Ii)$, then for each local model $U$ the restriction of $f$ is the limit of the restricted sequence, hence it belongs to $H^0(U,\Ii)$ for each $U$, which means $f\in H^0(X,\Ii)$.   

\item [Step 2.]\it Case   $X=\Omega\subset\C^n .$\rm

To prove (1) let $\{f_k\} \subset H^0(\Omega,\Ii)$ be a sequence converging to $f$.   Then the  germ $(f_k)_x\in \Ff_x$ for any $k$.  By Theorem \ref{modulichiusi} also the germ $f_x$ belongs to the fiber $\Ff_x$ and since this is true  for any $x\in \Omega$, we get  $f\in H^0(\Omega,\Ff)$.

To prove (2) since $\Ff$ is a subsheaf of a coherent sheaf,   it is enough to prove that it is finitely generated. At any point $x$ the fiber $\Ff_x$ is generated by finitely many function germs, because  $\Oo_x$ is noetherian and these functions are global because they can be taken in $E$.

If $y$ is another point, the generators of $\Ff_y$, that belong to $E$, have their germs at $x$ in $\Ff_x$, hence  by the second part of Theorem \ref{modulichiusi} they   are combinations   of the generators of $\Ff_x$ with  holomorphic coefficients in a suitable neighbourhood $V$ of $x$. So the  generators of $\Ff_x$ generate the fiber $\Ff_y$ for any $y\in V$.
\end{itemize}
\end{proof}

As a consequence we have Cartan's characterization of the closure of ideals in a Stein algebra.

\begin{thm}\label{Closure}Let $X$ be a Stein space and $\gta$ an ideal in
 $H^0(X,\Oo)$. Then $\overline \gta = H^0(X,\Oo \gta)$.
\end{thm}  
\begin{proof} By 1) of Theorem \ref{Moduli Chiusi Stein} it is enough to prove that $\gta$ is dense in $H^0(X,\Oo \gta)$.

Let $U_m$ an exhaustive sequence of relatively compact open sets in $X$ such that $(X,U_m)$ is a Runge pair for each $m$, that is,   Theorem\ref{Runge} applies. 
Note that such an exhaustion always exists. It is  enough to take polydisks as open sets where local models are defined and to argue as in (1) of Theorem \ref{Moduli Chiusi Stein}. 

Let $f\in H^0(X,\Oo \gta)$. For any $x$ there exists a neighbourhood $V_x$ and finitely many germs of elements in $\gta$ such that $f_{|V_x}=\sum_i m_ia_i$. Now $\overline {U}_m$ is compact, hence there exist finitely many $a_1,\ldots , a_{i_m}$ that generate the fiber at any point of $U_m$, that is, one gets  $\displaystyle f=\sum_{i=1}^{i_m} m_ia_i$ for any $U_m$. In other words, we have a surjective morphism 

$$\eta :  \Oo^{i_m} \to \Oo \gta$$

given by 
$\eta(g_1, \ldots g_{i_m}) = \sum_j g_ja_j$. Since  $U_m$ is a Stein domain by Theorem B, we get a surjective morphism $H^0(U_m,\Oo^{i_m}) \to H^0(U_m,\Oo \gta)$.

For any $f\in H^0(X,\Oo \gta)$ we obtain that $\displaystyle f= \sum_{i=1}^{i_m} h_i a_i$ where each $h_i$ is holomorphic in $U_m$ and each $a_i$ belongs to $\gta$.

By Theorem\ref{Runge} there exist global $\tilde {m}_i $ such that $|\tilde {m}_i -m_i|<\varepsilon $ on $\overline U_m$ so that $\displaystyle \sum\tilde {m}_i a_i$ is an element of $\gta$ that approximates $f$  as required.
\end{proof}

\begin{remarks}\label{idealifiniti}\null \hfill

\begin{enumerate}\parindent=0pt 
\item The ideal $\gta$ in the first example of this chapter  is not closed because the constant function $1$ belongs to $H^0(\C,\Oo\gta)$.

\item The last argument proves, as a consequence of Theorem B, that  a finitely generated ideal $\gta \subset \Oo(X)$ is closed. Indeed let $f_1,\ldots ,f_k$ be generators of $\gta$, then we get a surjective map $ \Oo^{k} \to \Oo \gta$ that maps the $k$-uple $  (a_1,\ldots ,a_k)$ to $\displaystyle \sum a_if_i$. 
Since $X$ is Stein this map induces a surjective map 
$$\varphi:H^0(X,\Oo^k) \to   H^0(X,\Oo\gta), \quad \displaystyle \varphi(a_1,\ldots,a_k) =\sum a_if_i.$$ Hence $H^0(X,\Oo\gta)=\gta $.

\end{enumerate}
\end{remarks}

We have another characterization of the closure of an ideal in $A$. 

\begin{defn} Let $(X,\Oo)$ be a Stein space  and 
$\gta\subset \Oo(X)$ be an ideal.
Define:
\begin{itemize}
\item  $\mathfrak C_1(\gta) = \{g\in \Oo(X): \forall K\subset X \mbox{\rm\ 
    compact} \ \exists f\in \Oo(X) \  \mbox{ \rm  such  that } \ \ceros(f)\cap K =\varnothing \  \mbox{\rm 
    and} \ fg\in \gta\}$.
\item $\mathfrak C_2(\gta) = \{g\in \Oo(X): \forall x\in X \ \exists f\in \Oo(X) 
\ \mbox{\rm such that} \  f(x)\neq 0 \  \mbox{\rm and} \  fg\in \gta\}$. 
\end{itemize} 
\end{defn}

\begin{lem}
Let $(X,\Oo)$ be  a Stein space and 
$\gta\subset \Oo(X)$ be an ideal. 
Then,$$\overline \gta = \mathfrak
C_1(\gta)=\mathfrak C_2(\gta).$$
\end{lem}
\begin{proof} As the chain of inclusions $\mathfrak C_1(\gta) \subseteq \mathfrak
  C_2(\gta)\subseteq \overline \gta$ is clear, it only remains to check $\overline
  \gta \subseteq \mathfrak C_1(\gta)$. Take a compact set $K\subset X$. Since the convex hull of $K$ is  compact, we can assume    $K$ to be 
  holomorphically convex. By Cartan's Theorem A there is an open
  neighbourhood $\Omega$ and elements $A_1,\ldots, A_m \in \gta$ such that for
  all $x\in \Omega$ the ideal $\overline\gta\Oo_x = \gta\Oo_x$ is generated by their 
germs at  $x$. Consider the closed ideal $\gtg = (A_1,\ldots,
A_m)\Oo(X)$. The ideal $(\gtg :\overline\gta) = \{h \in\Oo(X): h\overline\gta\subset
\gtg \}$ is closed too, see \cite[2, Satz3]{of}. Since $\overline\gta\Oo_x
=\gta\Oo_x$, we deduce $(\gtg :\overline\gta)\Oo_x = \Oo_x$ for all $x\in
\Omega$. In particular $1$ is a section of $(\gtg :\overline\gta)\Oo$ on
$\Omega$. Since $K$ is holomorphically convex there is a global section $h$ of
$(\gtg :\overline\gta)\Oo$ close to $1$ on $K$. Thus $h$ does not vanish on $K$ and  for any $f\in \overline\gta$ we have $hf\in \gtg\subseteq \overline\gta$, that is, $f\in  
\mathfrak C_1(\gtg)\subseteq \mathfrak C_1(\gta)$.
\end{proof}

\subsection{Nullstellensatz for closed primary ideals.}
The following proposition is a key step.

\begin{prop} \label{principioidentita} Let $X$ be a Stein space and  $\gtq$ a primary closed ideal in $H^0(X,\Oo)$. For $f\in H^0(X,\Oo)$ the following  facts are equivalent
\begin{enumerate}
\item $f\in \gtq.$
\item There exists a point $\overline x$ in the zeroset  $V(\gtq)$ of $\gtq$ such that $f_{\overline x}\in \Oo_{\overline x}\gtq$.
\end{enumerate}
\end{prop}   
\begin{proof} (1) $\Rightarrow$ (2) is straightforward. 

To prove the converse, let $\Gg$ be  the sheaf $(\Oo \gtq:f)$. 
This sheaf is   coherent because  it can be understood
as  the kernel of a morphism between coherent sheafs. In fact, consider the applications $\psi:\Oo \to \Oo$ induced by multiplication by $f$, that is, for each $x$ we have
 $\psi(h_x)= h_x f_x$ and the projection $\pi: \Oo\to \displaystyle \frac{\Oo}{\Oo \gtq}$.
The fiber $\Gg_x$ is $\{h_x\in \Oo_x | h_x f_x\in \Oo \gtq\}$; so $\Gg$ is exactly the kernel of the composition $\pi \circ \psi$.

Statement (2) implies that $\Gg_{\overline x}= \Oo_{\overline x}$, because $1_{\overline x}\in \Gg_{\overline x}$. By  Theorem A
 there exist global setions  $g_1,\ldots ,g_k$ of $\Gg$
such that  $1_{\overline x}= \sum m_j(g_j)_{\overline x}$ so it cannot be $ g_j(\overline x)= 0$ for each $j$. Hence there exists  $g_i$ such that $g_if\in H^0(X,\Oo \gtq)= \gtq$  but no powers of $g_i$ can belong to $\gtq$ because $g_i(\overline x) \neq 0$. So  $f\in \gtq$ because $\gtq$ is primary.
\end{proof}

The first consequence is that closed proper prime ideals cannot have empty zeroset  and this fact solves the first obstruction. The converse is also true.

\begin{cor}\label{primary1}
Let $X$ be a Stein space, $\gtq \subset \Oo(X)$  a primary ideal and $F\in \Oo(X)$. Then
\begin{itemize} 
\item[(i)] If $x\in \ceros(\gtq)$, then $F\in \gtq$ if and only if $F_x\in \gtq\Oo_{X,x}$
\item[(ii)] $\gtq$ is closed if and only if $\ceros(\gtq) \neq \varnothing$.
\item[(iii)]$ \ceros(\gtq)$ is connected.
\end{itemize}
\end{cor}
\begin{proof} (i) followss from  Proposition \ref{principioidentita}, whose proof does not use the closure of $\gtq$.
The only if  implication of (ii) is clear. The converse implication  holds, because if $x\in \ceros(\gtq)$, then $\gtq = \{F\in \Oo(x): F_x\in \gtq\Oo_{X,x}\}$, that is, $\gtq$ is the saturation of a local ideal, so it is closed.

(iii) is true because disjoint analytic closed sets in a Stein space are separated by holomorphic functions, so the zero set of a primary ideal cannot be disconnected. 
\end{proof}

The second important consequence  is  that for the class of closed primary ideals   Nullstellensatz holds true.

\begin{cor}\label{primary2}If $\gtq$ is a closed primary ideal in $A=H^0(X,\Oo)$, then 
\begin{enumerate}
\item $ I(\ceros(\gtq))=\sqrt{\gtq}$
\item There exists $n$ such that $(\sqrt{\gtq})^n\subset \gtq$
\end{enumerate}
\end{cor}
\begin{proof}It is  a consequence of Proposition \ref{principioidentita}.  Both statements are true in a point of the zeroset of $\gtq$.
\end{proof}

\subsection{Primary decomposition.} 

In this section we deal with lattice properties of families of ideals in a Stein algebra $A$. 

\begin{defn} A family $\{\gta_i\}_{i\in I}$  of ideals in a Stein algebra $A$ is {\em locally finite} if the family $\ceros(\gta _i)$ is locally finite in the Stein space $X$. 
\end{defn}

Denote by $\mathcal M$ the family of closed maximal ideals in $A$.  If $\gta$ is an ideal in $A$ we can define
\index{radical} the {\em radical }\index{Radical of an ideal } of $\gta$ by
$\displaystyle Rad (\gta) = \bigcap_{\gta\subset M\in \mathcal M}M$. 

 Even if $\gta$ is not, $Rad (\gta)$ is closed because it is  intersection of closed ideals. Since a maximal ideal is prime, if it is closed it must be the ideal of one point $x_0\in X$. So there is a bijection between $X$ and the set  of closed maximal ideals given by $x\mapsto M_x= \{f\in A: f(x)=0\}$. Note also that $x\in \ceros(\gta)$ if and only if $M_x \supset \gta$. Thus  
$$ I( \ceros(\gta)) = Rad (\gta).$$    

We get $\gta \subset \sqrt{\gta} \subset Rad (\gta) = I(\ceros(\gta))$. 

The following theorem deals with the radical of the intersection of a locally finite family of ideals.

\begin{thm}\label{radical} Let $\{\gta_i\}_{i\in I}$ be a locally finite family of ideals in a Stein algebra $A$. Then
\begin{enumerate}
\item $Rad (\bigcap_i \gta_i) = \bigcap_i Rad(\gta_i)$.
\item $\ceros(\bigcap_i \gta_i) = \bigcup_i \ceros(\gta_i)$.
\end{enumerate}
\end{thm}

Before proving Theorem \ref{radical} we need some preliminary lemmas.

\begin{lem}\label{intorni di 1} Let $A$ be a Stein algebra. There is a countable basis of  neighbourhoods  $\{U_n\}_{n\in \N}$ of the constant function $1$ such that for any $n\in \N$ and any $m>n$ one has $U_n \supset U_{n+1} \cdot U_{n+2}\cdot \ldots \cdot U_m$.
In addition if for any $n$ we choose $f_n\in U_n$, the infinite product $\prod_n f_n$ converges in $A$.     
\end{lem}
\begin{proof} There is an open covering of $X$ by local models $X_i$. We can assume $(X_i ,\Oo_{X_i})$ to be realised as a closed subspace of a polydisc $\Delta_i \subset \C^{m_i}$ of polyradius $1$ so that $ A_i=H^0(X_i, \Oo_{X_i})$ is the quotient of $\tilde A_i = H^0(\Delta_i, \Oo_{m_i})$ by a closed ideal $\gtb_i$.
Let us define in $\tilde A_i$ a countable basis of neighbourhoods  of $1$. For each $n\in \N$ define 
$$\tilde U_n^{(i)} =\left\{ f\in \tilde A_i : |f(z)-1|< 3^{-n} \forall z\quad  \mbox{\rm such that}\quad |z_j|<1-\frac{1}{n}, j=1,\ldots,m_i\right\}\quad (*)$$

Let $U^{(i)}_n$ be the image of $\tilde U^{(i)}_n$ in $A_i$. Next consider the restriction map $r_i: A\to A_i$ and define $U_n =\{f\in A: r_i(f)\in U^{(i)}_n \forall i<n\}$. One can check that the family $\{U_n\}_{n\in \N}$ has the required property.

Take now $f_n \in U_n$ and let $P_n$ be the finite product $f_1\cdots f_n$. Using equation $(*)$  one proves that $\{P_n\}_{n\in \N}$ is a Cauchy sequence, hence it converges in the Fr\'echet space $A$.   
\end{proof}

\begin{lem}\label{indici finiti}
Let $\{\gta_i\}_{i\in I}$ be a locally finite family of ideals in a Stein algebra $A$.
Then for any open set $U\subset A$ there is a finite set $K\subset I$ such that for any finite subset $L\subset I\setminus K$ one has $U\cap \bigcap_{i\in L}\gta_i \neq \varnothing$.
\end{lem}
\begin{proof} Take an exhaustion $X_n$ of $X$ by relatively compact open Stein sets. The topology of $A$ is the less fine making all restrictions $r_n$ continuous.
This means that for any open set $U\subset A$ there is an $n$ and an open set $W\subset A_n = H^0(X_n, \Oo)$ such that $r_n^{-1}(W) \subset U$. Since the family is locally finite, we can consider the finite set $K$ of those indices $k\in I$ for which $\ceros(\gta_k)\cap X_n \neq\varnothing$. Take $L\subset I\setminus K$ finite and denote $\gtb = \bigcap_{i\in L}\gta_i$. So $\ceros(\gtb) \cap X_n=\varnothing$, which implies $r_n(\gtb)\Oo_{X_n} = \Oo_{X_n}$. Thus  $r_n(\gtb)$ is dense in $A_n$. So $r_n(\gtb) \cap W \neq \varnothing$, and since  $r_n^{-1}(W) \subset U$, we get $U\cap \gtb \neq \varnothing$. 
\end{proof}

{\sc Proof of Theorem \ref{radical}.}
The first statement is a consequence of the second. Also in the second statement the inclusion $\ceros\bigcap_{i\in I}\gta_i \supset \bigcup_{i\in I}(\ceros(\gta_i))$ is clear. 

Take $\displaystyle x\notin  \bigcup_{i\in I}\ceros(\gta_i)$. We
have to prove $x\notin\ceros (\bigcap_{i\in I}\gta_i)$. Let $\{U_n\}$
be  a basis ofneighbourhoods  of  $1\in A$  as in  Lemma \ref{intorni  di
  1}.  We can  assume for  any  $n$ that no function in  $U_n$ vanishes  at
$x$. By  Lemma \ref{indici  finiti} for all $n$ there is a finite set $L_n \subset I\setminus K$ such that $\displaystyle U_n\cap \bigcap_{i\in L_n}\gta_i\neq \varnothing$. This means that we can decompose $I$ into finite sets $I=  K\cup \bigcup_n L_n, L_n\subset  I\setminus K$.  Take elements  $\displaystyle f_n \in  \bigcap_{i\in L_n}\gta_i\cap U_n$ and  $\displaystyle g \in \bigcap_{i\in K} \gta_i$ such that  $g(x)\neq 0$. 
 By Lemma \ref{intorni  di 1} the  infinite product $g\cdot \prod  f_n$
converges to an element $ f\in  \bigcap_{i\in I}\gta_i$ such that $f(x)\neq 0$,  
hence $x\notin \ceros(\bigcap_{i\in I}\gta_i)$.
\qed

\bigskip

Next step is to erase redundant ideals.

\begin{prop}\label{primochiuso}
Let $\{\gta_i\}_{i\in I}$ be a not empty locally finite family of ideals in a Stein algebra $A$  and $\gtp$ be a closed prime ideal. If for all $i$ the ideal $\gta_i$ is not included in $ \gtp$, then $\bigcap_i \gta_i$ is not included in $\gtp$.
\end{prop}

\begin{proof}
The difference $U = A\setminus \gtp$ is an open set in $A$. By Lemma \ref{indici finiti} there is a finite subset $K\subset I$ such that $U\cap\bigcap_{i\in I\setminus K} \gta_i \neq \varnothing$. Hence there is a function  $f\in \bigcap_{i\in I\setminus K} \gta_i \setminus \gtp$ and functions $f_i \in \gta_i \setminus \gtp$,for each $i\in K$. Thus $f\cdot \prod_{i\in K}f_i \in \bigcap_{i\in I}\gta_i \setminus \gtp$. 
\end{proof}

\begin{prop}\label{locale0}
Let $\{\gta_i\}_{i\in I}$ be a locally finite family of closed ideals in a Stein algebra $A = H^0(X,\Oo)$. Then for all $x\in X$
$$\left( \bigcap_i \gta_i\right) \Oo_x = \bigcap_i (\gta_i \Oo_x).$$
\end{prop}

\begin{proof} For $I$ finite it is a consequence of Theorem A. For the general case consider the coherent sheaf $ \big( \bigcap_i \gta_i\big) \Oo$. For any $x\in X$ one has $\gta_i \Oo_x = \Oo_x$ except for a finite set of indices. Hence $\bigcap_{i\in I}\gta_i \Oo_x$ is a finite intersection. In other words for all $x$ there is a finite set $K_x\subset I$ such that $x\notin \ceros(\gta_i)$ for $i\notin K_x$. Hence 

\begin{eqnarray*}
 \left( \bigcap_{i\in I} \gta_i\right)\Oo_x= \left( \bigcap_{i\in K_x} \gta_i\right)\Oo_x \cap\left(\bigcap_{i\in I\setminus K_x} \gta_i\right)\Oo_x =\\
=\left (\bigcap_{i\in K_x} \gta_i \Oo_x\right)\cap \Oo_x= \left(\bigcap_{i\in K_x} \gta_i \Oo_x\right)\cap\left(\bigcap_{i\in I\setminus  K_x} \gta_i \Oo_x\right)=\bigcap_{i\in I} (\gta_i \Oo_x).
\end{eqnarray*}
\end{proof}

\begin{defn} \index{Irreducible decomposition} We say that a decomposition $\gta=\bigcap _{i\in I}\gtb_i$ of an ideal  $\gta$  in a Stein algebra $A$  is {\em irreducible} if for any  proper subset $J\subset I$ one has $\gta \neq \bigcap_{i\in J} \gtb_i$.
\end{defn}
Not all decompositions can be made irreducible making $I$ smaller. For instance, 
this is not possible when $I$ is not finite and the family $\gtb_i$ is not locally finite as the following example shows.

\begin{example}
Let  $A$ be the Stein algebra of holomorphic functions on $\C$. Consider $\gtq_i = \{f\in A : m_0(f) \geq i\}$. The family $\{\gtq_i\}_{i\in \N}$ is a family of closed primary ideals but the corresponding family of their zerosets is not locally finite. One has $\bigcap_i \gtq_i = (0)$ but the decomposition cannot be made irreducible.
\end{example}

\begin{thm}\label{irreducibl} Let $\{\gtb_i\}_{i\in I}$ be a locally finite family of closed ideals in a Stein algebra $A$. Then there is a minimal subset $I_0\subset I$ such that $\gta= \bigcap _{i\in I_0} \gtb_i$ and this decomposition is irreducible.
\end{thm}
\begin{proof} Denote by $\Gamma $ the family of subsets $J\subset I$ such that $\gta = \bigcap _{j\in J}\gtb_j$ which is partially ordered by inclusion. We have to find a minimal element in $\Gamma$. By Zorn lemma it is enough to consider a linearly ordered chain $C\subset \Gamma$. Let $J_0 = \bigcap_{J\in C} J$. If $J_0\in \Gamma$, we are done.

Take $x\in X$. Since the family $\{\gtb_i\}$ is locally finite, $x$ belongs to finitely many zerosets, say $x\in \ceros(\gtb_{i_1}) \cap \dots \cap \ceros(\gtb_{i_k})$. Now for each $l= 1,\dots, k$, either $i_l \in J_0$ or not. In the last case there is an element in $C$ not containing $i_l$. So there is an element $J_1 \in C$ such that $x\notin \bigcup_{j\in J_1\setminus J_0} \ceros(\gtb_j)$.

Apply Proposition \ref{locale0}. We get:  

$$\left(\bigcap_{j\in J_o} \gtb_j\right) \Oo_x =\bigcap_{j\in J_o} \left(\gtb_j \Oo_x\right) = \left(\bigcap_{j\in J_1} \gtb_j\right) \Oo_x = \gta \Oo_x.$$ 
 
This is true for any $x\in X$ (for a suitable $J_1$) hence $\bigcap_{j\in J_0}\gtb_j = \gta$ and $J_0 \in \Gamma$.
\end{proof}

\begin{defn}\index{normal primary decomposition} Let $\gta = \bigcap_i \gtq_i$ be a primary locally finite decomposition of a closed ideal $\gta$ in a Stein algebra $A$. We say that the decomposition is {\em normal} if it is irreducible and the prime ideals $\gtp_i = \sqrt{\gtq_i}$ are pairwise different.  
\end{defn}

\begin{lem}\label{normal} If a closed ideal has a primary decomposition, then it has a normal decomposition. 
\end{lem}
\begin{proof} By Theorem \ref{irreducibl} we can assume the decomposition $\gta = \bigcap \gtq_i$ is irreducible. Since the family $\{\gtq_i\}$ is locally finite, only finitely many primary ideals $\gtq_i$ can have the same radical ideal $\gtp$. Their intersection is still primary and has the same radical ideal. 
\end{proof}

Let $\gta$ be a closed  ideal and  $\gta = \bigcap_{i\in I}\gtq_i$ be a normal decomposition in closed primary ideals.  We introduce the following definitions.
\begin{enumerate}
\item $\gtq_i$ is a {\em primary component} of $\gta$.
\item $\gtp_i = \sqrt{\gtq_i}$ is {\em a prime ideal associated} to $\gta$.
\item A primary component $\gtq_l$ is called {\em immersed} if there is $m\neq l$ such that $\ceros(\gtq_l)\subset \ceros(\gtq_m)$, otherwise it is called {\em isolated}.
\item A set of primary components $\{\gtq_m\}_{m\in M\subset I}$ is {\em isolated} 
if 
$$   \ceros(\gtq_m)\subset \ceros(\gtq_i) \Rightarrow  i\in  M, \forall m\in M. $$ 
The intersection  $\bigcap_{m\in M} \gtq_m $  is also an isolated component of $\gta$.
\end{enumerate}

To prove the existence of a primary decomposition for closed ideals in a Stein algebra we use the following fact.

\begin{lem}\label{chiuso}Let $\gta$ be an ideal in $\Oo_x$ and $\gtN_x=\{f\in A |\, f_x\in \gta\}$. Then $\gtN_x$ is a closed ideal in $A$.
\end{lem}

\begin{proof} The fact that $\gtN_x$ is an ideal is evident. To prove that this ideal is closed, observe that  $(\gtN_x\Oo)_y$ is equal to $\Oo_y$ if $y\neq x$
 and it is equal to  $\gtN_x$ if $y=x$, because in a Stein space, given two different points $x$ and $y$, there is  always a functions taking  arbirtrary values in $x$ and $y$. So $H^0(X, \gtN_x\Oo)=\{ f \in A |\  \forall y\  f_y \in ( \gtN\Oo_x)_y \}= \gtN_x$. 
\end{proof}

Now we can prove the main result of this section.
\begin{thm}\label{decprimaria}
Any closed ideal $\gta$ in a Stein algebra $A$ has a locally finite normal primary decomposition into closed primary ideals.
\end{thm}
\begin{proof}
Consider the coherent sheaf $\Oo\gta$ generated by the closed ideal $\gta$. For any $x\in X$ define $\gtN_x = \{f\in A| f_x\in  \gta\Oo_x\}$ and $\mathcal M_x$ the sheaf generated by $\gtN_x$. By Lemma \ref{chiuso} $\gtN_x$ is closed and by definition
$\gta = \bigcap_{x\in X}\gtN_x$. Since $\gta\Oo$ is coherent, for any $x\in X$ there is a neighbourhood $U_x$ such that $\gtN_x \subset \gtN_y$ for all $y\in U_x$. Indeed $f_x\in \gta\Oo_x$ implies $f_y\in \gta\Oo_y$ for all $y\in U_x$.   The topology of $X$ has a countable basis, so we can find a sequence $\{x_i\}_{i\in \N}\subset X$ such that $X= \bigcup_i U_{x_i}, \gtN_{x_i} = \gtN_x \forall x \in U_{x_i}$ and $\gta = \bigcap_i \gtN_{x_i}$. Hence $\gta\Oo_x  =  \gtN_{x_i}\Oo_x \forall x\in U_{x_i}$.

Now  $\Oo_{x_i}$   is  a  noetherian  ring,   hence  $\gta\Oo_{x_i}  =
\gtQ_{i_1}\cap  \dots \cap  \gtQ_{i_{r_i}}$,  where $\gtQ_{i_j}$  is a  primary
ideal  in  $\Oo_{x_i}$.  Define  $\gtq_{i_j} =  \{f\in  A  |f_{x_i}\in
\gtQ_{i_j}\}$. Then $\gtq_{i_j}$  is closed by the same  argument as before
and  it is  primary. Indeed  if $f  g \in  \gtq_{i_j}$ and  $f^k\notin
\gtq_{i_j}$ for each  $k$,  then  $f_{x_i}  g_{x_i}  \in  \gtQ_{i_j}$  but
$f_{x_i}^k\notin \gtQ_{i_j}$ for each $k$. So  $g_{x_i} \in \gtQ_{i_j} $ which
implies  $g\in  \gtq_{i_j}$.   Hence  we  get  $\displaystyle \gtN_{x_i}  =\bigcap_1^{r_i}\gtq_{i_j}$ and so
$$\gta = \bigcap_i \left(\bigcap_{j=1}^{r_i}\gtq_{i_j}\right)$$
is a primary decomposition of $\gta$.
But it could be not locally finite.

Put $U_i = U_{x_i}$ and $R_i = \{p: 1\leq p\leq r_i \mbox{ \rm and } \ceros(\gtq_{i_p})\cap U_j = \varnothing \, \forall j <i\}$. Then the family $\{\gtq_{i_p}, i\in \N, p\in R_i\}$ is locally finite and its intersection $\gtb$ contains $\gta$. We want to show $\gta = \gtb$. Since both are closed it is enough to prove $\gta\Oo_x  = \gtb\Oo_x$ for all $x\in X$. We proceed by induction on the smallest $i$ such that $x\in U_i$.

{\em Initial step}. $i=1$ and $R_1 = \{1,\ldots, r_1\}$. Hence 
$$ \gtb\Oo_x =  \left(\bigcap_{i=1}^\infty \left(\bigcap_{p\in R_i} \gtq_{i_p}\right)\right)\subset\bigcap_{p\in R_1} \gtq_{1_p}\Oo_x  = \gtN_{x_1}\Oo_x = \gta\Oo_x.$$

{\em Inductive step}. We assume $$\gtb\Oo_y \subset  \bigcap_{j=1}^{i-1} \left(\bigcap_{p\in R_j} \gtq_{j_p}\Oo_y\right)  \subset \gta\Oo_y $$

for all $y\in U_1 \cup U_2\cup \dots \cup U_{i-1}$.       

Take any index $\sigma \in S_i = \{1, \ldots, r_i\} \setminus R_i$. 
Then there is a point  $\displaystyle a\in \ceros(\gtq_{i_\sigma})\cap \left(\bigcup_{j=1}^{i-1} U_j\right)$. For any $\displaystyle f\in \bigcap_{j=1}^{i-1} \left(\bigcap_{p\in R_j} \gtq_{j_p}\right)$ we have $f_a \in \gta\Oo_a$ by the induction hypothesis. Thus

$\gta\Oo_a  \subset  \gtq_{i_\sigma}\Oo_a$ So $f_a\in  \gtq_{i_\sigma} $ and $f\in \gtq_{i_\sigma}\Oo_a$. This implies 

$$\bigcap_{j=1}^{i-1} \left(\bigcap_{p\in R_j} \gtq_{j_p}\right) \subset \bigcap_{\sigma\in S_i}\gtq_{i_\sigma}.$$
Hence
$$\bigcap_{j=1}^i \bigcap_{p\in R_j} \gtq_{j_p} \subset  \left(\bigcap_{\sigma\in S_i}\gtq_{i_\sigma}\right) \bigcap \left( \bigcap_{p\in R_i}\gtq_{i_p}\right) = \bigcap_{p=1}^{r_i} \gtq_{i_p} \Oo_x.$$
  For $x\in U_i$ one gets 

$$\gtb\Oo_x =  \left(\bigcap_{j=1}^i \left(\bigcap_{p\in R_j} \gtq_{j_p}\right)\right) \Oo_x \subset  \left( \bigcap_{p=1}^{r_i} \gtq_{i_p}\Oo_x\right) =  \gtN_{x_i}\Oo_x =\gta.$$

This concludes the proof that $\gta = \gtb$. Hence there is a subfamily $J= \bigcup_j R_j\subset I$ such that $\gta = \bigcap_{j\in J}\gtq_j$ is a locally finite primary decomposition. 
\end{proof}

\subsection{Nullstellensatz for closed ideals.}

Using the previous fact and an application of Baire's theorem to the Fr\'echet space $\Oo(X)$ one gets the following result.

\begin{thm}{\rm  (Closed general case)}\label{fo}
  Let $\gta\subset\Oo(X)$  a closed ideal and  let $\gta=\bigcap_{i\in
    I}\gtq_i$ be  a normal primary  decomposition of $\gta$.  For each
  $i\in I$ define
\begin{align*}
\h(\gtq_i,\gta)=&\inf\left\{k\in\N:\ f^k\in\gtq_i,\ \forall f\in Rad(\gta)\right\},\\
\h(\gtq_i)=&\inf\{k\in\N:\ f^k\in\gtq_i,\ \forall f\in\sqrt{\gtq_i}\},\\ 
\h(\gta)=&\inf\left\{k\in\N:\ f^k\in\gta,\ \forall f\in Rad(\gta)\right\}. 
\end{align*}
Then we have
\begin{itemize}
\item[(i)]     $\h(\gta)=\sup_{i\in     I}\{\h(\gtq_i,\gta)\}$     and
  $\sqrt{\gta}$ is closed if and only if $\h(\gta)<+\infty$.
\item[(ii)]  If $\gta$  does  not have  immersed primary  components,
  $\h(\gta)=\sup_{i\in I}\{\h(\gtq_i)\}$.
\item[(iii)] $I(\ceros(\gta))=\sqrt{\gta}$ if and only if $\h(\gta)<+\infty$
  and in such case $(\sqrt{\gta})^{\h(\gta)}\subset\gta$.
\end{itemize}
\end{thm}
\begin{proof}We can assume without loss of generality that $\gtq_i \neq A$ for
  all $i$. Put $\rho = \sup \h(\gtq_i, \gta)$ and assume $\rho <\infty$. Hence
  if $f\in Rad(\gta)$ then $f^\rho \in \gta$ so $Rad(\gta) \subset
  \sqrt{\gta}\subset Rad(\gta)$ and we got the {\em if} part of (i). In order to get the {\em only if} part we
  have to prove that $\rho$  is minimal in the following sense.
\begin{quotation}
For all $\sigma < \rho$ there is $f_0\in Rad(\gta) = \bigcap_i Rad(\gtq_i) =
  \bigcap_i \sqrt{\gtq_i}$ such that $f_0^\sigma \notin \gta$.
\end{quotation}       
Define $\rho_{\gtq_i} =$ the smallest integer $\rho_i$ such that for all $f\in
Rad(\gta)$ one has $f^{\rho_i} \in \gtq_i$.

We want to prove $\rho_{\gtq_i}= \h(\gtq_i,\gta)$.

Trivially $\rho_{\gtq_i}\leq \h(\gtq_i,\gta)$, because $Rad(\gta) \subset
Rad(\gtq_i) = \sqrt{\gtq_i}$.

Define $M_{\gtq_i} =\{j\in I $ such that $\ceros( \gtq_j )\subset \ceros(\gtq_i)\}$.   
For all $\tau < \h(\gtq_i,\gta)$ by definition there is $\displaystyle f\in \bigcap_{j\in M_{\gtq_i}} \sqrt{\gtq_j}$ such that $f^\tau \notin \gtq_i$. We can find
  $\displaystyle g\in \bigcap_{j\notin M_{\gtq_i}}\sqrt{\gtq_j}$ but $g\notin
  \sqrt{\gtq_i}$. Indeed $A \setminus \sqrt{\gtq_i}$ is an open dense subset
  that does intersect the closed set $\displaystyle \bigcap_{j\notin
    M_{\gtq_i}}\sqrt{\gtq_j}$.
 The function $\displaystyle gf\in \bigcap_{i\in I}Rad(\gtq_i) = Rad(\gta)$, while $(gf)^\tau
\notin \gtq_i$ because $f^\tau\notin \gtq_i$ and no power of $g$ belongs to
$\gtq_i$. This implies $\tau < \rho_{\gtq_i}$. Since $\tau$ was any positive
number smaller than $\h(\gtq_i,\gta)$,  we get $\h(\gtq_i,\gta)\leq \rho_{\gtq_i}
\leq \h(\gtq_i,\gta)$.
\smallskip

Next define $R_{\gtq_i}= \{f\in Rad(\gta): f^{\rho_{\gtq_i-1}}\in \gtq_i\}$.
We claim that $R_{\gtq_i}$ is closed with empty interior in the closed
set $Rad(\gta)\subset A$. 
\smallskip

Take $f\in R_{\gtq_i}$ and $g\in Rad(\gta) \setminus R_{\gtq_i}$. Consider the
complex line $f+\lambda g$, for $\lambda\in \C$. This line cuts $R_{\gtq_i}$ in at
most $n=\rho_{\gtq_i} -1$ points. Indeed, if there were $n+1$ points we would
get $n+1$ distinct complex numbers $\lambda_0, \dots \lambda_n$ such that
$(f+\lambda_h g)^n \in \gtq_i$. Write this as a linear system.

$$\sum_{k=0}^n \binom{n}{k} \lambda_h^k f^{n-k}g^k =F_h\in \gtq_i$$

for each $h=0,\dots,n$

Write $y_k =  f^{n-k}g^k$ . Then the matrix of the system above has
coefficients $\displaystyle \binom{n}{k}\lambda_h^k$ and determinant $\displaystyle \prod_{k=0}^n
\binom{n}{k}\prod_{h<l} (\lambda_h -\lambda_l)$. So the system has a
unique solution. Note that $y_n = g^n \in \gtq_i$, because it is a linear
combination with complex coefficients of $F_0, \dots, F_n$ and all belong to
$\gtq_i$. This is a contradiction, because $g\notin \sqrt{\gtq_i}$.

Thus in any open neighbourhood of $f\in R_{\gtq_i}$
there are elements of the complement and this proves that $R_{\gtq_i}$ has
empty interior, so $Rad(\gta) \setminus R_{\gtq_i}$ is open and dense in $Rad(\gta)$.

\smallskip
We can apply Baire's theorem, so $\displaystyle Rad(\gta)\setminus \bigcap_{i\in I} R_{\gtq_i}= \bigcap_{i\in I}
(Rad(\gta)\setminus R_{\gtq_i})$ is dense in $Rad(\gta)$. Take an element
$f_0$ in $\displaystyle Rad(\gta)\setminus \bigcap_{i\in I} R_{\gtq_i}$. If $\sigma < \sup_i \h(\gtq_i,\gta)$, there is $j\in I$ such that
$\sigma < \rho_{\gtq_j}$. Since $f_0\notin  R_{\gtq_j}$, one gets
$f_0^\sigma\notin \gtq_j$, so $f_0^\sigma\notin \cap_i \gtq_i = \gta$ as required.  \end{proof}

\begin{remark}Let $\gta=\bigcap \gtq_i$ be a primary normal decomposition of a closed ideal $\gta$. If the sequence $\{\h(\gtq_i,\gta)\}$ is not bounded, the ideal $\sqrt \gta$
is not closed. Indeed, its closure is precisely $\ideal(\ceros(\gta))=\bigcap
\sqrt{\gtq_i} $. Recall that  $\sqrt{\gtq}= Rad(\gtq)$ is 
closed when $\gtq$ is primary and closed. Since the sequence $\{\h(\gtq_i,\gta)\}$ is not bounded,   by  Theorem \ref{fo}  $\overline{\sqrt \gta}$ is   strictly greather than $\sqrt \gta$.  We  conclude 
$$\ideal(\ceros(\gta)= \overline{\sqrt{\overline{\gta}}}.$$

In particular  $\ideal(\ceros(\gta)) =\sqrt \gta$ if and only if $\sqrt{\gta}$ is closed and $\ideal(\ceros(\gta)) =\gta$, that is, ${\gta}$ is radical and
  closed.

\end{remark}

\section{The  real case.}

Let  $(X,\Oo_X)$  be  a  $C$-analytic set in $\R^n$ endowed  with  its 
well reduced structure  and  let $\Oo(X)$ be its  algebra of
global analytic functions. 
The ring  $\Oo(X)=H^0(X,\Oo_X)=\Oo(\R^n)/I(X)$ can
be understood as a subset of the Stein algebra $\Oo(\widetilde{X})$ of
its \em complexification \em  $\widetilde{X}$ (understood as a complex
analytic set germ at $X$). We stress that $X$ needs not to be coherent
as an  analytic set,  but it  is by definition the  support of  a coherent  sheaf of
$\Oo_{\R^n}$-modules. 
We  define the saturation of an ideal $\gta\subset \Oo(X)$ 
as
$\tilde{\gta}=H^0(X,\gta\Oo_X)=\{f\in\Oo(X):\    f_x\in\gta\Oo_{X,x}\
\forall\,x\in X\}$. 
Of course,  $\gta\subset\tilde{\gta}\subset I(\ceros(\gta))$. Note that finitely generated ideals are saturated by the same argument as in Remarks \ref{idealifiniti}.
Note that, unlike in the complex case, the saturation is not the closure of $\gta$ with respect to the topology induced by the Fr\'echet topology on $\Oo(\widetilde {X})$. Nevertheless, \index{Saturated ideal} saturated ideals admit a {\em normal} primary decomposition similar to the one  in the complex case. Indeed, 
by Theorem \ref{Closure}, an ideal in a complex Stein algebra is closed if and only if it is saturated and it is this second fact that allows primary decomposition.

The previous definition of saturation is similar to the one proposed by Whitney for ideals in the ring of smooth functions over a real smooth manifold.

\begin{prop}\label{deb}Let  $(X,\Oo_X)$  be  a  $C$-analytic set in $\R^n$.
\begin{enumerate}
\item Proposition \ref{principioidentita} and Corollary \ref{primary1} hold true  replacing {\em closed} by {\em saturated}.
\item Let $\gta \subset \Oo(X)$ be a saturated ideal. Then, $\gta$ admits a normal primary decomposition $\displaystyle \gta= \bigcap_{i\in I} \gtq_i$ and all primary ideals $\gtq_i$ are saturated. Moreover, the prime ideals $\gtp_i = \sqrt{\gtq_i}$ and the primary isolated components are uniquely determined by $\gta$ and do not depend on the normal primary decomposition of $\gta$.
\item If $\displaystyle \gta= \bigcap_{i\in I} \gtq_i$ is a normal primary decomposition of a saturated ideal $\gta \subset \Oo(X)$, then $\widetilde{\sqrt{\gta}} = \bigcap_i \sqrt{\gtq_i}$. 
\end{enumerate}
\end{prop}

\begin{proof}
(1) and (2) can be proved in the same way as was done for Stein algebras.
To prove (3) we apply Corollary \ref{primary1} (ii) to prove that each $\sqrt{\gtq_i}$ is a prime saturated ideal. On the other hand, for each $x\in X$ there is a finite set $J_x \subset I$ such that $ \gtq_i\Oo_{X,x}= \Oo_{X,x}$ for each $\forall i\notin J_x$ ( so $\sqrt{\gtq_i}\Oo_{X,x}= \Oo_{X,x}$).

Let us first show 
$$\left(\bigcap_{i\in I}\sqrt{\gtq_i}\right)\Oo_{X,x} = \sqrt{\gta}\Oo_{X,x}.$$

Since $J_x$ is finite, we get

$$\sqrt{\gta} \subset  \left(\bigcap_{i\in I}\sqrt{\gtq_i}\right)\Oo_{X,x} \subset \bigcap_{i\in I}( \sqrt{\gtq_i}\Oo_{X,x}) = \bigcap_{i\in J_x}( \sqrt{\gtq_i}\Oo_{X,x}).$$

The last term is equal to 

$$\bigcap_{i\in J_x}( \sqrt{\gtq_i}\Oo_{X,x})= \sqrt{\bigcap_{i\in J_x}\gtq_i}\Oo_{X,x}= \gtb \Oo_{X,x}.$$

Thus it is enough to prove $\gtb \Oo_{X,x} \subset \sqrt{\gta}\Oo_{X,x}$.
One has $\bigcup_{i\in I\setminus J_x} \ceros(\gtq_i) = \ceros (\bigcap_{i\in I\setminus J_x}\gtq_i)$, hence there is $h\in \bigcap_{i\in I\setminus J_x}\gtq_i$ such that $h(x)\neq 0$. If $f_x \in \gtb \Oo_{X,x}$, we get 

$$f_xh_x \in \left(\bigcap_{i\in I}\gtq_i\right)\Oo_{X,x} = \sqrt{\gta}\Oo_{X,x}.$$

As $h_x$ is a unit  $f_x \in \sqrt{\gta}\Oo_{X,x}$ as required.    
\qed

One can  check the good properties of normal primary decompositions as stated in the following corollary.

\begin{cor}\label{goodproperties} 
Let $\gta \subset \Oo(X)$ be a saturated ideal and $\gta = \bigcap_i \gtq_i$  a normal primary decomposition of $\gta$. We have.
\begin{itemize}
\item[(i)] If $\gta$ is radical, then each $\gtq_i$ is prime and the decomposition is unique.
 \item[(ii)] If $\gta$ is a real  ideal, then each $\gtq_i$ is a real  prime ideal and the decomposition is unique.    
\end{itemize}
\end{cor}

\bigskip

It seems  difficult to prove a real
Nullstellensatz for  $\Oo(X)$  considering  the real radical as in the results  of Risler quoted at the beginning of this chapter. One can try an
alternative approach that  involves  `convexity' for  ideals.  

\begin{defn}\index{Convex ideal}An ideal $\gta$ of $\Oo(X)$ is \em convex \em if each $g\in\Oo(X)$ satisfying $|g|\leq f$ for some $f\in\gta$ belongs to $\gta$. We define the \em convex hull $\hat{\gta}$ of an ideal $\gta$ of $\Oo(X)$ \em by
$$
\hat{\gta}=\{g\in\Oo(X):\ \exists f\in\gta\text{ such that }|g|\leq f\}.
$$
\end{defn}

Notice that $\hat{\gta}$ is the smallest convex ideal of $\Oo(X)$ that
contains $\gta$ and $\hat{\gta}\subset I(\ceros(\gta))$. 

\begin{defn}\index{ \L ojasiewicz radical ideal} We define the {\em
  \L ojasiewicz radical ideal} of an ideal $\gta\subset\Oo(X)$ as
$\sqrt[\text{\L}]{\gta}=\sqrt{\hat{\gta}}$. 
\end{defn}
In particular, \L ojasiewicz's 
radical is a  convex radical ideal. 

This radical is appropriate to describe a Nullstellensatz in the real global analytic context. In Section \ref{NSSgenerale} we will prove the following result. 

\begin{thm}{\rm (Real Nullstellensatz)}\label{perfect2}
Let $X\subset\R^n$ be a $C$-analytic set and 
$\gta$ an ideal of the ring $\Oo(X)$. Take  $\gtb = \widetilde{\gta}$  Then, 
$I (\ceros (\gta)) =\widetilde{\sqrt[\text{\L}]{\ {\gtb}\, }}$.
\end{thm}

\subsection{A global \L ojasiewicz's inequality.}
The classic \L ojasiewicz's inequality for continuous semialgebraic functions  states the following.

\begin{thm}{\rm(Classic \L ojasiewicz's inequality)}
Let $K$ be a compact semialgebraic set and let $f,g$ be continuous semialgebraic functions on $K$ such that $\{f=0\}\subset\{g=0\}$. Then there exist an integer $m\geq 1$ and a constant $c\in\R$ such that $|g|^m\leq c|f|$ on $K$.
\end{thm}

We need a more global result. Let $Z$ be a $C$-analytic subset of $\R^n$ and let $f,g\in\Oo(Z)=\Oo(\R^n)/I(Z)$ be such that $\{f=0\}\subset\{g=0\}$.  Fixed a compact set $K\subset Z$, there exist a proper $C$-analytic subset $X_1\subset\{f=0\}\setminus K$ and an open neighborhood $U$ of $\{f=0\}\setminus X_1$ on which $|g|^m\leq|f|$ for some integer $m\geq 1$. The precise statement of the weak \L ojasiewicz's inequality is the following.

\begin{thm}{\rm (Weak \L ojasiewicz's inequality)}\label{weakLoj}
Let $Z$ be a $C$-analytic subset of $\R^n$ and let $A\subset Z$ be a global
semianalytic subset.\footnote{A global semianalytic set is a semianalytic set
  of the form $A= \bigcup_{i=1}^t \{f_i=0, g_{i_1} >0, \dots, g_{i_{s_i}}>0\}$
  where all the involved functions are analytic on $Z$. We will provide more details in Chapter 5.} Let $f,g\in\Oo(Z)$ be such that $\{f=0\}\cap\overline A\subset\{g=0\}\cap \overline A$. Fix a compact set $K\subset Z$ and denote $X=\{f=0\}$. Then there exist a proper $C$-analytic subset $X_1\subset X$ such that $X_1 \cap K=\varnothing$, an open neighborhood $U\subset Z$ of $X\setminus X_1$ and a positive integer $m\geq 1$ such that $|g|^m<|f|$ on $(U\cap\overline A)\setminus X$. 
\end{thm}

The first step in the proof is the following elementary lemma.

\begin{lem}\label{elem}
Let $p:M \to N$ be a surjective proper map between affine C-analytic
sets $M\subset \R^m$, $N\subset \R^n$.
Let  $X\subset  M$  be  a   saturated  set, that is,  $p^{-1}(p(X))=X$,  and  let $U\supset  X$  be  a  neighbourhood  of $X$.  Then  $p(U)$  is  a
neighbourhood of $p(X)$.
\end{lem}

\begin{proof} Assume $U$ is open and  let $C$ be the complement of $U$.  Then,
$p(C)$ is  closed, since $p$  is proper. Hence  its complement is an  open set
$V$.  Since $X$ is saturated one has $p(X)\subset V\subset p(U).$
\end{proof}

{\sc Proof of Theorem \ref{weakLoj}.} As both $Z$ and $A$ are global semianalytic sets they are defined respectively
  by a finite set of analytic functions on an open neighbourhood $\Omega$ of $Z$ in $\R^n$, say
  $f_1,\dots, f_t \in \Oo(\Omega)$. We may assume  $f,g$ are also analytic on $\Omega$.

Apply   Hironaka's  desingularization  theorem    to the function
$$F = f_1\cdot \ldots \cdot f_t \cdot  h\cdot f.$$
We find  an analytic manifold $M$ of  dimension  $n$ and a  proper surjective analytic
map $p:M\to \Omega$  such that for any $a\in M$  there exist an open
neighbourhood $U_a$ of $a$ and  a coordinate system $x_1,\ldots , x_n$
on $U_a$ centred at $a$ such that:
$$\ristr {F\circ p}{U_a} = x_1^{\gamma _1}\ldots x_n^{\gamma _n}.$$

This means that:

\begin{enumerate}

\item $p^{-1}(A)  \cap U_a$  is a  union of  open quadrants
 each  one lying in  some coordinate $l$-plane, for some $l                      \leq  n$.
\item $f\circ p =\varphi  x_1^{\alpha _1}\ldots x_n^{\alpha  _n},             
\, h\circ  p =\psi x_1^{\beta _1}\ldots x_n^{\beta _n} $ on $U_a$,
with $\varphi(0) \neq 0$, $\psi(0)\neq  0$;
\item  for  any  $i$ such  that  the hyperplane  $\{x_i=0\}$
intersects $p^{-1}(A)$ or lies in its  boundary, one has
$\beta_i \neq 0 \, \mbox{ \rm if } \, \alpha_i \neq 0.$
This is because $\{f=0\} \cap \overline A \subset \{ h =0\}\cap \overline A.$
\end{enumerate}

 For $a\in M$, define  $\delta(a) = \alpha _1 + \ldots + \alpha _n$
the degree of $f \circ p$ at $a$ as a monomial.

It is well defined, in the sense that it  does not depend on
the choice of the  local coordinate system  at the point $a$.  In fact,
the  admissible changes of coordinates  keeping $F\circ p$ as a monomial
 are only  permutations of the variables appearing  in $F\circ p$ (that is, for
 those variables such that $\gamma_i \neq 0$)  up  to  a scaling factor.
Moreover, $\delta(b)\leq \delta(a) $ for $b\in U_a$, that is,  $\delta$ is
semicontinuous.

A computation shows that for each $j<\delta(a)$ and for each
multiindex $\gamma =(\gamma_1,\ldots ,\gamma_n)$ with $\sum_i \gamma_i= j$ one has
$\displaystyle \frac {\partial ^\gamma (f\circ p)}{\partial                     
x_1^{\gamma_1}\ldots                                                            
\partial x_n^{\gamma_n}}(0) =0$,
while  $\displaystyle \frac {\partial ^{\delta(a)} (f\circ p)}{\partial         
x_1^{\alpha_1}\ldots                                                            
\partial x_n^{\alpha_n}}(0) =k\varphi(0) \neq 0$ where $k$ is a
suitable positive integer.

Now we do the following.
Consider  a  complexification    $(\widetilde  {\Omega},\widetilde              
{X})$ of the couple $(\Omega,  X)$ and let $\widetilde {F}= \widetilde {f}_1\cdot  
\ldots   \cdot\widetilde  f_k   \cdot\widetilde  h\cdot\widetilde   f$   be  a
holomorphic  extension of  $F$  to  a suitable  neighbourhood  of $\Omega$  in
$\widetilde  \Omega$. Up  to shrink  $\widetilde  \Omega$ we  can assume  that
$\widetilde F$  is defined on  $\widetilde\Omega $. We  can perform the
same blowings-up that reduce $F$ to monomial form on the manifold $\widetilde  
{\Omega}$.  So, we  get a proper holomorphic map $\widetilde  p : \widetilde M  
\to \widetilde {\Omega}$ such that $M$ is the real part of $\widetilde M $ and
$ {\widetilde p}_{|M}= p$.
We define
$$Y_\delta = \{ a\in \widetilde p^{\,-1}(\widetilde X) | J^\delta               
(\widetilde f\circ \widetilde p) (a)\equiv 0\}$$
where $J^\delta (\widetilde f\circ \widetilde p) (a)$ denotes the
$\delta$-jet of the
function $\widetilde f\circ \widetilde p$ at the point $a$.
The set $Y_\delta$  is a complex analytic subset of $\widetilde p^{\,-1}(\widetilde              
X)$. Moreover, the family $\{Y_\delta\}_{\delta \in \N}$ is locally finite,
since not all the derivatives of $\widetilde f\circ \widetilde p$ can vanish
simultaneously at a point. Thus, if $K\subset X$ is compact, so is
$\widetilde p^{\,-1}(K)$, so it meets only finitely many
$Y_\delta$'s. In other words, $Y_\delta \cap \widetilde p^{\,-1}(K)=\varnothing$
for $\delta$ large enough. In particular,
$$X_1 = \widetilde p(Y_\delta)\cap \Omega \cap X$$

is a proper C-analytic subset of $X$ and it is the one we were looking for.

We look at
the expression of $f\circ p$ and $h\circ p$ on $U_a$,for $a\in p^{-1} (X)\setminus
Y_{\delta}$. Since they are both monomials  and the
degree $\delta(a)$ of $f\circ p$ at $a$ is less or equal to  $\delta$, one can
find an integer $m$ that does not depend on $a$ such that
$$|(h\circ p)|^m\ < |f\circ p| \quad \mbox{\rm on} \quad (V_a\cap                
p^{-1}(\overline{A}))\setminus p^{-1}(X) $$
where $V_a$ is a perhaps smaller neighbourhood contained in $U_a$.
This is because  $f$  does not vanish on $\overline{A} \setminus X$ .

The union $\displaystyle  \bigcup _{a\in p^{-1}(X)\setminus                     
Y_{\delta}}V_a$ is a neighbourhood $V_1$ of
$p^{-1}(X)\setminus Y_{\delta}$ and  one has
$$|(h\circ p)|^m < |f\circ p|$$
on $(V_1 \cap p^{-1}(\overline{A}))\setminus \ p^{-1}(X).$

Finally $ U_1  =p(V_1)$ is a neighbourhood of  $X\setminus X_1$. Indeed, for
any  $ x  \in  X\setminus X_1$,  $p^{-1}(x)$  is a  saturated  compact set  contained in
$p^{-1}(X)$, which does not meet $Y_\delta$.  We can apply  Lemma \ref{elem}.
By construction, one has $|h|^m < |f|$ on $(p(V_1)\cap \overline{A})\setminus X\
$.     \qed

Given analytic functions $f,g\in\Oo(Z)$ such that $\{f=0\}\subset\{g=0\}$, the previous result provides a recursive procedure to construct an analytic function $t$ on $Z$ whose zeroset does not meet a compact set $K$ and satisfies $|t||g|^{m}\leq|f|$ for some integer $m\geq 1$. More precisely we have the following result.

\begin{prop}\label{Loj} 
Let $Z$ be a $C$-analytic set in $\R^n$ and $f,g\in\Oo(Z)$ such that $\{f=0\}\subset\{g=0\}$. Let $K\subset Z$ be a compact set. Then there exist an integer $m\geq1$
and an analytic function $t\in\Oo(Z)$ such that $\{t=0\}\cap K=\varnothing$ and $|t||g|^{m}\leq |f|$.
\end{prop}
\begin{proof}Let $\{K_i,\}_{i\in \N}$ be an exhaustion of $\R^n$ by compact sets such that
$K_j \subset \Int{K}_{j+1}$ and $K_1 =K$.
By Theorem \ref{weakLoj} we find a proper analytic subset $X_1\subset X$,
an integer $m_0\geq 1$ and an open  neighbourhood
$U_0$ of $X\setminus X_1$ such that $X_1 \cap K=\varnothing$ and $$|g|^{m_0} <
|f| \quad \mbox{\rm on}\quad U_0\setminus X.$$ 
We can 
assume that $U_0 \setminus X =\{ x\in A\, : \                        
|g|^{m_0}(x)< |f(x)|  \}$ and  $U_0\cap X_1 =\varnothing$. 

Let $h_1$ be a positive equation of $X_1$. We may assume $h_1  
<1$ by replacing it with $\displaystyle \frac{h_1}{1+h_1^2}$. 
Define $A_1 = \{x\in Z : |g|^{m_0}(x)< |f(x)|  \}\setminus U_0$.
So $\{f=0\} \cap A_1 \subset X_1 = \{h_1 =0\}$. We  apply again Theorem
\ref{weakLoj} and find an analytic subset $X_2\subset X_1$ such that $X_2\cap K_2 =\varnothing$, an integer $m_1$ and a neighbourhood $U_1$ of $X_1\setminus X_2$
such that $h_1^{m_1} <|f|$ on $(U_1\cap A_1) \setminus X_1$.

We can start again with a positive equation $h_2<1 $ of $X_2$.
Arguing as before and writing $\displaystyle \frac{g}{g^2+1}= h_0$ and $m=m_0$ we find inductively countably many proper
analytic subsets $X_i \subset X_{i-1}$ with positive equations $h_i <1$, open
semianalytic neighbourhods $U_i$ of $X_{i}\setminus X_{i+1}$ and integers
$m_i>0$ such that
$X_i \cap K_i = \varnothing$ and

$$ h_{i}^{m_i}< |f| \quad \mbox{\rm on }\quad (U_i\cap \overline{A_{i}})\setminus X_{i}.$$

Since  $X_i \cap K_i =\varnothing$, the family
$\{X_i\}$ is locally finite.

 We  define the sheaf  $\Ii$ generated at each
$a\in \R^n$ by the (finite) product of those   $ h_{j}^{m_{j}}$ that
vanish at $a$.

This is possible since the family of their zero sets   is
locally finite. By construction $\Ii$ is coherent and locally principal,
because all but a finite number of $h_i$ are invertible at $a$.
Since $H^1(\R^n, \Oo^*) = H^1(\R^n,{\Z}/{2\Z}) =0$, any locally principal
sheaf of ideals
is actually principal. Let $H$ be a global generator for $\Ii$, that is, $\Ii=     
(H)\Oo$ and take $h=H^2$.
Then,
$\{h=0\}= \mbox{supp }  ({\Oo}/{\Ii}) = X_1$ and $h\geq 0$. By
construction $h_x$ is divisible by $(h_j)_x^{2 m_{j}}$ for any $j$.

Outside  $X_i$,  $h$ is given  by the product  $\prod _{j=0}^{i-1}        
h_j^{n_j}$  times a positive unit $v_i$, where $n_j = 2 m_{j}> m_{j}$.

We want to find a positive analytic function $u(x)$ on $\R^n$ such that
$u(x) v_i(x) < 1$  on $U_i$ for any $i$ so that the function $uh$ satisfies 
$uh<1$ and $ uh <|f|$ on $U_i \setminus X_i$ for all $i$.

Let $M_j$ be the maximum of $v_j$ on the compact set $K_j$, which does not
intersect $X_j$. Note that for any $i < j$ one has outside $X_j$
$$ v_i(x) = h_{i+1}(x)\cdots h_{j-1}(x) v_j(x) <v_j(x),$$
because $h_l(x) <1$ for any $l$.

Take a positive continuous function $\varepsilon (x)$ on $\R^n$ such that
$0 < \varepsilon (x) < 1/M_j$ for $x\in K_j$. By Whitney's approximation theorem
we may find an analytic function $u(x)$ on $\R^n$ satisfying

$$ |u(x) -\varepsilon (x)| < \displaystyle \frac{1}{2M_j} $$

for $ x\in K_j$.
Then, on $U_i $ we get $u h = u v_i \prod _{j=1}^{i-1}h_j^{n_j}$.
So, since all the $h_j$ are smaller than 1,  one has

$${uh}_{|(U_i\setminus X_{i})}                          
< h_{i-1}^{n_{i-1}} <  h_{i-1}^{m_{i-1}} < {|f|}_{|(U_i\setminus X_{i-1})} \qquad \mbox{and}\qquad                                   
{h}_{|U_i}  \leq {|f|}_{|U_i}.$$

This  implies $h<|f|$ on
$ \displaystyle \bigcup _{i\in \N}U_i\setminus  X$. Now  $\displaystyle U =                  
\bigcup _{i\in  \N}U_i$ is a  neighbourhood of  $X$. Indeed, $\bigcap_i X_i =     
\varnothing$, because it is a locally finite family. 
Note also that $g^m = g^{m_0}=
h_0^{m_0}(g^2+1)^{m_0}$ divides $h$, because $g^2 +1$ is a unit, that is, $h = g^mt$, where $t$ is the
analytic function we were looking for.

We have only to extend the previous inequality  to the whole  $Z$
 and this is done in the next lemma.
\end{proof}

\begin{lem}\label{unita}Let $f, h$  and $X$  be as before.
Assume $h$ is a positive semidefinite  analytic  function on $Z$  and that
$h < |f|$ on $V\setminus X$, where $V$ is a neighbourhood of
$X$. Then, there is $h_1 \in \Oo(Z)$, $h_1>0$ such that
$$h_1h < |f| \quad \mbox{\rm on} \quad \ Z\setminus X$$
\end{lem}
\begin{proof}Consider the open covering
$ U_1 = V$, $ U_2 = Z \setminus \overline{U}$, where $                         
\overline{U}\subset V$ and $U$ is a smaller neighbourhood of
$X$. Define $f_i:U_i\to \R$  as $ \displaystyle f_1 = \frac {1}{2}$,
$\displaystyle f_2 = \frac {1}{2}\frac{|f|}{h}$,
which  is defined and continuous on $U_2$. Let $\{\sigma_1,                    
\sigma_2\}$ be a
smooth partition of  unity subordinated to the given covering and
consider the function $\varphi = \sigma_1 f_1 + \sigma_2 f_2$.

Note that $\varphi \geq 0$ on $Z$ and $\varphi $ does not vanish on
$Z$. In fact, it can vanish only outside the support of
$\sigma_1$ and where $f_2$ vanishes. Now, if $\varphi(x) =0$,
then $x \in \{f=0\}= X \subset U$, hence $x\notin U_2$
and $\sigma_2(x)=0$, which implies $\sigma_1(x)=1$, which is a contradiction.
So we can extend ${\varphi}_{|\overline U}$ to a
continuous strictly
positive function $\Phi$.

Now, take an analytic function $h_1$ on $Z$ such that
$\displaystyle |h_1 - \Phi|< \frac {1}{2} \Phi$.
Then $hh_1 \leq  h|h_1 - \Phi| + h \Phi < \frac {3}{2} h \Phi$,  so
one has:
$$ h h_1 \leq \frac {3}{2} h \Phi = \frac {3}{2} h \varphi = \frac {3}{4}      
(\sigma_1 h + \sigma_2 |f|).$$

Hence, $\displaystyle  h h_1 < \frac {3}{4} |f|$ on $ V \setminus X$ and
$\displaystyle hh_1 \leq                                                      
\frac {3}{4} |f|$ on $ Z\setminus V$, that is,
 $\displaystyle  hh_1 \leq \frac {3}{4} |f| < |f|$ on $Z\setminus X\
$.
\end{proof}

Proposition \ref{Loj} is a key result to prove Theorem \ref{perfect2}, but it has the Nullstellensatz for principal ideals as a direct consequence. 

Indeed, we have the following corollary.

\begin{cor}\label{principale}  Theorem \ref{perfect2} holds true for principal ideals in a real Stein algebra.
\end{cor}
\begin{proof} Let $\gta =(f)$ be a principal ideal in $\Oo(X)$. In particular it is saturated. Let $g\in I(\ceros (f))$. By Proposition \ref{Loj} for any compact  set $K\subset X$ there is an analytic function $t$ on $X$ that does not vanish on $K$ such that $|f|\geq|t||g|^{m}$. This means that $g\in \widetilde{\sqrt[\text{\L}]{(f)}}$, because its germ belongs to $\sqrt[\text{\L}]{(f)}\Oo_{X,x}$ for all $x\in X$.  
\end{proof}

\subsection {Nullstellensatz for general ideals.}\label{NSSgenerale}
 
We want to reduce the general problem to the case of a principal ideal. To that end  given an ideal $\gta\subset\Oo(X)$, we are looking for  an analytic function $f\in \gta$ that has the same zeroset as $\gta$. 
The following proposition provides the analytic function we need,
but there is  a price to pay. The analytic  function $f$ does not
belong in general to the ideal $\gta$ but  to $\tilde{\gta}$. 
Observe that $I(\ceros(\gta))=I(\ceros(f))$, so each $g\in I(\ceros(\gta))$ satisfies $\ceros(f)\subset \ceros(g)$ and $g\in \widetilde{\sqrt[\text{\L}]{(f)}}$. which is a subset of the saturation of the \L ojasiewicz radical of $\widetilde{\gta}$
and the proof will be done. 

\begin{prop}\label{crespina}
Let $\gta\subset \Oo(X)$ be an ideal. There is $f\in \widetilde{\gta}$ such that $\ceros(f) = \ceros(\gta)$.
\end{prop}
\begin{proof} Consider the sheaf of ideals   $\gta\Oo_X$. We extend it as a coherent invariant sheaf of ideals $\Jj$ on an open invariant Stein neighbourhood $\Omega$ of $\R^n$ in $\C^n$. Take an exhaustion $\{L_l\}_l$ of $\Omega$ by compact sets. By  Theorem A there is a countable family $\{G_k\}\subset H^0(\Omega,\Jj)$  of invariant holomorphic functions such that for each $l$ there is $k(l)$ such that for all $z\in L_l$, the germs of $G_1,\ldots ,G_{k(l)}$ generate $\Jj_z$.

For $k\geq 1$ define $\mu_k = \mbox{\rm max}_{L_k}|G_k^2| +1$ and $\displaystyle \gamma_k = \frac{1}{\sqrt{2^k\mu_k}}$. Then, the series $F= \sum_{k\geq 1}\gamma_k^2 G_k^2$ converges uniformely on compact sets of $\Omega$. Indeed, if $L$ is such a compact set,  there is $k_0\geq 1$ such that $L\subset L_k$ for all $k\geq k_0$. On $L_k$ one has $\displaystyle \gamma_k^2|G_k^2| < \frac{1}{2^k}$. Thus for $z\in L$

$$\left |\sum_{k\geq k_o}\gamma_k^2 G_k(z)^2\right | \leq \sum_{k\geq k_o}\gamma_k^2
\left |G_k(z)^2\right | \leq \sum_{k\geq k_o} \frac{1}{2^k} \leq 1.$$

Consider the partial sums $S_m = \sum_1^m\gamma_k^2 G_k^2 \in H^0(\Omega,\Jj)$ of the series $F$. Then, $F= \displaystyle \lim_{m\to \infty}S_m$ in the Fr\'echet topology of $\Oo(\Omega)$. Since $H^0(\Omega,\Jj)$ is closed, we get $F\in H^0(\Omega,\Jj)$ and $F$ is invariant. So its restriction $f$ to $\R^n$ belongs to $\tilde{\gta}$. Denote $f_k =\gamma_k{G_k}_{|\R^n}$. Then $f=\sum_k f_k^2$ and $\ceros (f) = \ceros (\gta)$. Indeed, 

$$\ceros (f) = \bigcap_k \ceros (f_k) =\bigcap_k \left(\ceros (G_k)\cap \R^n\right) = \left(\bigcap_k \ceros (G_k)\right)\cap \R^n = \mbox{\rm supp}(\Jj) \cap \R^n = \ceros(\gta).$$
\end{proof}

Theorem \ref{perfect2} is proved as a consequence of  Proposition \ref{crespina}and Proposition \ref{Loj}.

\subsection{Real radical and \L ojasiewicz radical.}\label{rilo}

In general, if $\gta\subset\Oo(X)$ is an ideal, $\widetilde{\sqrt[{r}]{\gta}}\subset\widetilde{\sqrt[\text{\L}]{\gta}}$ and it is a natural question to determine under which conditions the two ideals coincide. This question has a close relation with Hilbert $17^{th}$ Problem for the ring of global analytic functions. Indeed, if we compare the radical ideals $\sqrt[{r}]{\gta}$ and $\sqrt[\text{\L}]{\gta}$, we obtain the following:
\begin{itemize}
\item $g\in\sqrt[\text{\L}]{\gta}$ if and only if there exist $f\in\gta$ and $m\geq1$ such that $f-g^{2m}\geq0$.
\item $g\in\sqrt[{r}]{\gta}$ if and only if there exists $m\geq1$ and $a_1,\ldots,a_s\in\Oo(X)$ such that $g^{2m}+a_1^2+\cdots+a_k^2=f\in\gta$ or equivalently if there exist $f\in\gta$ and $m\geq1$ such that $f-g^{2m}$ is a sum of squares in $\Oo(X)$.
\end{itemize}
Thus, we would have $\sqrt[\text{\L}]{\gta}=\sqrt[{r}]{\gta}$ if any non-negative analytic
function were a sum of squares. Unfortunately, this is not true even for polynomials and the best result one can afford is the following: {\em a polynomial $f\in\R[x_1\ldots,x_n]$ such that $f(x)\geq 0$ for each $x\in\R^n$ is a sum of squares in the field $\R(x_1,\ldots, x_n)$ of rational functions} (Artin 1927). In other words, in general denominators are needed to obtain representations as sum of squares. We formulate Hilbert $17^{th}$ Problem for analytic functions as follows.

\begin{quest}
Let $f\in\Oo(\R^n)$ be such that $f(x)\geq 0$ for each $x\in\R^n$. Are there  analytic functions $g,a_1,\ldots, a_k$ such that $Z(g) \subset Z(f)$ and $g^2f=a_1^2+\ldots+a_k^2$?
\end{quest}

The answer is not known in general, but there are partial results related to the topological properties of the zeroset of the given non-negative function $f\in\Oo(\R^n)$. See next chapter for further details. 

 This lack of global information suggests the following definition.

\begin{defn} 
A C-analytic set $Z\subset\R^n$ is an {\em H-set} if each positive semidefinite analytic function $f\in\Oo(\R^n)$ whose zeroset is $Z$ can be represented as a  sum of squares of meromorphic functions on $\R^n$.
\end{defn} 

\begin{remark}\label{Hset}\hfill

\begin{itemize}
\item[(i)] Let $Y\subset Z\subset \R^n$ be C-analytic sets. If $Z$ is an H-set, then also $Y$ is an H-set. Indeed, let $f$ be a positive semidefinite analytic function such that $\ceros(f)=Y$ and $g$ an analytic function with zero set $Z$. Then $h=g^2f$ is positive semidefinite and its zero set is $Z$, hence $h$ is a sum of squares of meromorphic functions and this implies the same for $f$. In particular irreducible components of an H-set are H-sets.
\item[(ii)] An equivalent  condition for $Z$ to be an H-set is the following. 
\begin{quotation}{\em 
  There exists a positive semidefinite analytic function $f$ with $\ceros(f) = Z$ and each $h\in \Oo(\R^n)$ with $\ceros(h)=Z$ and $0\leq h\leq f$ is a sum of squares of meromorphic functions on $\R^n$.}\end{quotation}
{\parindent=0pt

One direction is clear. Conversely, assume there exists $f$ with the properties of the statement and let $g$ be another positive equation for $Z$. Note that
$\displaystyle f - \left(\frac{f}{\sqrt{1+fg}} \right)^2g 
= \displaystyle f\left(1-\frac{fg}{1+fg}\right)\geq 0 $  and $\displaystyle  \ceros\left (\left(\frac{f}{\sqrt{1+fg}}\right)^2 g\right) =Z$.

Thus  $\displaystyle 0\leq h= \left(\frac{f}{\sqrt{1+fg}}\right)^2g\leq f$, so $h$ is a sum of squares of meromorphic functions and the same holds true for $g$. Hence, $Z$ is an H-set.  }
\end{itemize}
\end{remark}
\bigskip

The result we want to prove is the following.

\begin{thm}\label{h17nss}
Let $X\subset\R^n$ be a C-analytic set and $\gta$ an ideal of $\Oo(X)$ such that $\ceros(\gta)$ is a H-{set}. Then $I(\ceros(\gta))=\widetilde{\sqrt[{r}]{\gta}}$.
\end{thm}

As we will see in the next chapter,
the previous result applies if $X$ is either an analytic curve, a coherent analytic surface or a C-analytic set with finitely many compact connected components, so the real Nullstellensatz holds for such an $X$ in terms of the real radical.

Since in general denominators cannot be avoided, before proving Theorem \ref{h17nss}, we give a lemma that allows to move denominators outside a given ideal.

\begin{lem}\label{denominators}
Let $b,f \in \Oo(\R^n)$ be not constant analytic functions. Let $\Omega$ be  an open invariant neighbourhood  of $\R^n$ in $\C^n$ to which both extend to holomorphic invariant functions $B,F$. Let $Z\subset \Omega$ be a closed analytic subset such that its germ $Z_{x_0}$ at $x_0$ is not included in $\ceros(F)_{x_0}$ for some point $x_0 \in \ceros(f)$. Then, there exists an analytic diffeomorphism $\varphi:\R^n \to \R^n$ such that:

\begin{itemize}
\item[(i)] $f\circ \varphi =fu$ for some unit $u\in \Oo(\R^n)$.
\item[(ii)] $Z_{x_0}$ is not included in $\ceros(B_0)_{x_0}$, where $B_0$ is a holomorphic extension of $b \circ \varphi$ to a perhaps smaller neighbourhood $\Omega'$. 
\end{itemize}
\end{lem} 
\begin{proof}
We can assume $Z_{x_0}\subset \ceros(B)_{x_0}$, because otherwise we could take $\varphi$ to be the identity.
We divide the proof into two steps.

Step 1. {\em We construct a family of analytic diffeomorphisms $\varphi_\lambda:\R^n\to \R^n$ that satisfy condition (i).}

Take a strictly positive analytic function $\varepsilon \in \Oo(\R^n)$. 
For each n-uple $(\lambda_1, \ldots ,\lambda_n)=\lambda \in [-1,1]^n$ consider the analytic map

$$\varphi_\lambda:\R^n\to \R^n, x\mapsto x+f^2(x)\varepsilon(x) \lambda.$$
 
If $\varepsilon$ is sufficientely small, $\varphi_\lambda $ is a diffeomorphism close to the identity for all $\lambda \in [-1,1]^n$. Now the function 
$$f_0:\R^n\times \R^n\times \R \to \R, (x,y,t) \mapsto f(x+ty) -f(x)$$
vanishes identically on $\R^n\times \R^n\times \{0\}$, hence there exists an analytic function $h\in \Oo(\R^n\times \R^n\times \R)$ such that $f_0 = ht$. Thus,

$$f\circ \varphi_\lambda(x) = f(x)+ f(x)^2 \varepsilon(x)h(x,\lambda,f(x)^2\varepsilon (x))= f(x) u_\lambda(x)$$

where $u_\lambda(x)= 1+ f(x)\varepsilon (x) h(x,\lambda, f(x)^2 \varepsilon(x))$.

Note that $\ceros(f\circ \varphi_\lambda) \subset \ceros(f)$, hence  $\ceros(f\circ \varphi_\lambda) = \ceros(f)$. Indeed, if $f\circ \varphi_\lambda(x) =0$, then $\varphi_\lambda(x)=y \in \ceros(f)$. But $\varphi_\lambda$ is bijective and $\varphi_\lambda (y) =y$, because $f(y)=0$. So $x=y\in \ceros(f)$.

Also $u_\lambda$ is a unit in a neighbourhood of $\ceros(f)$ by definition and does not vanish outside $\ceros(f)= \ceros(f\circ \varphi_\lambda)$, hence it is a unit in $\Oo(\R^n)$ for all $\lambda \in [-1,1]^n$. So, $\varphi_\lambda$ satisfies property (i) for all $\lambda \in [-1,1]^n$.

Step 2. {\em We find $\lambda_0\in [-1,1]^n$ such that $\varphi_{\lambda_0}$ satisfies also property (ii).} 

We can consider the family of analytic diffeomorphisms $\varphi_\lambda$ as a map

$$\varphi:\R^n\times [-1,1]^n \to \R^n, (x,\lambda) \mapsto \varphi_\lambda(x).$$ 
 
After shrinking $\Omega$, we may assume that $\varepsilon , b$ extend holomorphically to $E,B\in \Oo(\Omega)$. Assume also $\Omega$ to be connected. Thus $\varphi$ extends to a holomorphic map 

$$\Phi:\Omega\times \C^n \to \C^n,  (z,\mu) \mapsto z+ F(z)^2 E(z)\mu. $$

Define $U=\Phi^{-1}(\Omega)$ and consider the holomorphic function 

$$B \circ \Phi:U \to \C, (w,\mu) \mapsto B\circ\Phi(w,\mu)= B\circ \Phi_\mu(w).$$

Fix a polydisc $\Delta_0 \times \Delta_1 \subset \Omega \times \C^n$ with center $(x_0,0)$ and radius $0< \rho <1$ contained in $U$.

We want to prove that the map $(B\circ \Phi)_w:\Delta_1\to \C, \mu \mapsto B\circ \Phi(w,\mu)$ is not identically zero for each $w\in \Delta_0$.

Otherwise, there would be a point $w\in \Delta_0$ such that

$$B\circ \Phi_w(\mu)= B\circ \Phi(w,\mu) = B(w+ F(w)^2E(w) \mu)$$

were identically zero on the polydisc $\Delta_1$. By the identity principle we would get that $B$ is identically zero, contradicting the hypothesis that $b$ is not constant.

Next we use that $Z_{x_0}$ is not included in $\ceros(F)_{x_0}$. By the complex curve selection lemma there is a complex analytic curve $\gamma$ from a small disc $\D_\delta$ to $Z$  such that $\gamma(\D_\delta) \subset \Delta_0, \gamma(0) =x_0$ and $\gamma(s) \notin \ceros(F)$ for all $s\neq 0$.

Consider the holomorphic function 

$$G:\D_\delta \times \Delta_1 \to \C, (s,\mu) \mapsto  B\circ \Phi(\gamma(s), \mu).$$

We  know fom the remark above that $G_s:\Delta_1 \to \C, \mu\mapsto G(s,\mu)$ is not identically zero for all $s\in \D_\delta$. Choose a sequence $\{s_k\}_k \subset \D_\delta$ converging to $0$ and remember that $F(\gamma(s_k)) \neq 0$ for all $k$. Define $W_k = (\Delta_1\cap \R^n)\setminus \ceros(G_{s_k})= [-\rho, \rho]\setminus \ceros(G_{s_k})$ and note that this set is open and dense in $\Delta_1\cap \R^n$ because $\ceros(G_{s_k})$ is a proper analytic subset of $\Delta_1$. So, $W=\bigcap_k W_k$ is dense in $\Delta_1\cap \R^n$ and we choose $\lambda_0 \in W$.

Put $b_0 = b\circ \varphi_{\lambda_0}$ and let $B_0 = B\circ \Phi_{\lambda_0}$ its holomorphic extension. By the choice of $\lambda_0$ we have $B_0(\gamma(s_k)) \neq 0$ for all $k$. Hence $B_0\circ \gamma$ is not identically zero on $\D_\delta$ which implies that the germ $(B_0\circ \gamma)_0\neq 0$. We conclude $Z_{x_0}$ is not included in $\ceros(B_0)_{x_0}$ as required.      
\end{proof}

{\sc Proof of Theorem \ref{h17nss}.}
 Assume first that $\gta =\gtp \subset \Oo(\R^n)$ is a saturated real  prime ideal whose zero set is an H-set. One implication is clear because 
$I(\ceros(\gtp))$ is  real and saturated. For the converse consider the ideal sheaf $\gtp \Oo_{\R^n}$ which can be extended to a coherent ideal sheaf $\J$ on an invariant connected open Stein neighbourhood $\Omega$ of $\R^n$ in $\C^n$.
Of course, for all $x\in \R^n$ one has $\J_x = \gtp\Oo_{\R^n,x} \otimes \C = \gtp\Oo_{\C^n,x}$. Also $\gtp = H^0(\R^n, \gtp\Oo_{\R^n})$, because $\gtp$ is saturated. Define $Z$ to be the support of  $\J$ in $\Omega$, that is, $Z=\{z\in \Omega: \J_z \neq \Oo_{\C^n,z}\}$. 

Let us check that $\gtb = H^0(\Omega, \J)$ is a prime closed ideal of $\Oo(\Omega)$ and that its zero set is $Z$. 

Pick $F_1, F_2 \in \Oo(\Omega)$ such that $F_1F_2 \in \gtb$. Write $F_j =\re(F_j) +i \ima(F_j), j=1,2$,
 and observe

$$(\re(F_1)^2 + \ima(F_1)^2)(\re(F_2)^2 + \ima(F_2)^2) = F_1F_2 \overline{F_1\circ\sigma}\overline{F_2\circ\sigma} \in \gtb.$$ 

Restricting to $\R^n$ we deduce that 
$$(\re(F_1)^2 + \ima(F_1)^2)_{|\R^n}(\re(F_2)^2 + \ima(F_2)^2)_{|\R^n} \in H^0(\R^n, \gtp\Oo_{\R^n}) = \gtp.$$

As $\gtp$ is prime and real, we may assume $\re(F_1)_{|\R^n}, \ima(F_1)_{|\R^n} \in \gtp$, so  $\re(F_1), \ima(F_1) \in \gtb$ and $F_1 = \re(F_1) +i \ima(F_1) \in \gtb$. Thus $\gtb$ is prime. It is also closed because the zeroset of $\gtp$  is not empty, so the zeroset of $\gtb$ is not empty.  By Corollary \ref{primary1} it is closed. 
The equality $Z= \ceros(\gtb)$ holds because the zero set of $\gtb$ is precisely the support of $\J= \gtb\Oo_{\C^n}$.

Assume now  that there exists a function $g\in I(\ceros(\gtp))\setminus \gtp$. After shrinking $\Omega$ we may assume that $g$ extends to an holomorphic function $G$ on $\Omega$. 

Now we claim there is $x\in \R^n$ such that $Z_x$ is not contained in $\ceros(G)_x$. Otherwise, $Z\subset \ceros(G)$ in a smaller neighbourhood. Hence $G\in \I(Z)= \I(\ceros (\gtb))=\gtb$ because $\gtb$ is prime and closed. So $g\in \gtp$,which is  a contradiction. Hence we have a point $x_0\in \R^n$ such that $Z_{x_0}$ is not contained in $\ceros(G)_{x_0}$. 

By Proposition \ref{crespina} there exists $f\in \tilde \gtp =\gtp$ such that $\ceros(f) = \ceros(\gtp)$. By Proposition \ref{Loj} there are $h\in \Oo(\R^n)$ and  a positive integer $m$  such that $h(x_0) \neq 0$ and $f_0= f-h^2g^{2m}\geq 0$. As $h(x_0)\neq 0$ we have $h\notin \gtp$.  Substitute $f_0$ by $f_1= f -h_1^2g^{2m}$ where $\displaystyle h_1= \frac{h}{\sqrt{1+h^2g^{2m}}}$. Thus we get $\ceros (f_1) = \ceros(f) = \ceros(\gtp)$ which is an H-set. Indeed, since $h_1\leq h$ we have $f_1 \geq 0$. In addition  $f\geq 0$, because the difference between $f$ and a square is $\geq 0$. So we get

$$\ceros(f_1) =\ceros ((f-h^2g^{2m}) +fh^2g^{2m}) =\ceros(f-h^2g^{2m}) \cap \ceros(fh^2g^{2m})$$
$$= \ceros(f)\cap \ceros(hg) = \ceros(f).$$   

Since $\ceros(\gtp) = \ceros(f_1)$ is an H-set there is $b$ not identically zero such that $b^2f_1= \sum_{i\geq 1}a_i^2$ for some $a_i\in \Oo(\R^n)$.

After shrinking $\Omega$, we can assume that $f_1, h$ have holomorphic extensions $F_1, H$ to $\Omega$. 
We want to apply Lemma \ref{denominators} to $b, f_1, Z, \Omega$. 
We have to show first that $Z_{x_0}$ is not included in $\ceros(F_1)_{x_0}$. Otherwise, as $F\in \gtb$ and $H(x_0) \neq 0$, we get 

$$
Z_{x_0} \subset \ceros(F)_{x_0}\cap \ceros(F_1)_{x_0} \subset \ceros (F-F_1)_{x_0}= \ceros (H_1^2G^{2m})_{x_0}=$$
$$= \ceros (H_1)_{x_0} \cup \ceros(G)_{x_0}= \ceros(G)_{x_0},
$$
which is a contradiction.

By Lemma \ref{denominators} there is an analytic diffeomorphism $\varphi:\R^n\to \R^n$ such that:
\begin{enumerate}
\item $f_1\circ \varphi = f_1u$ for some unit $u\in \Oo(\R^n)$
\item $Z_{x_0}$ is not included in $\ceros(B_1)_{x_0}$ where $B_1:\Omega_1\to \C$ is the holomorphic extension of $b_1= b\circ \varphi$ to an enough smaller neighbourhood $\Omega_1\subset \Omega$ of $\R^n$ in $\C^n$.
\end{enumerate}

Take a positive unit $v\in \Oo(\R^n)$ such that $v^2 =u^{-1}$. Then 
 
$$b_1^2f = b_1^2h_1^2g^{2m}+b_1^2f_1 = b_1^2h_1^2g^{2m} + \sum _{i\geq 1}((a_i\circ \varphi)v)^2. $$

Now $b_1\notin \gtp$ because the germ at $x_0$ of the complex zero set of $\gtb$ is not included in $\ceros(B_1)_{x_0}$. As $f\in \gtp$ and $\gtp$ is real, we deduce $b_1h_1g^m \in \gtp$, which is a contradiction, because $\gtp$ is prime and $b_1, h_1, g \notin \gtp$. We conclude $I(\ceros(\gtp)) =\gtp$.

Now assume $\gta$ is a saturated real ideal in $\Oo(\R^n)$ whose zero set is an H-set. As we know,  $\gta$ has a normal primary decomposition
$\gta = \bigcap_i \gtq_i$ where each $\gtq_i$ is a saturated real prime ideal, because the real ideal $\gta$ is a radical ideal. 

So $\ceros(\gta) = \bigcup_i \ceros (\gtq_i)$. Since this set is an H-set, also $\ceros(\gtq_i)$ is an H-set. Hence, we get by the previous step $\I(\ceros(\gtq_i))= \gtq_i$. Thus,

$$I(\ceros(\gta)) = I\left(\bigcup_i \ceros(\gtq_i)\right)= \bigcap_i I(\ceros(\gtq_i)) = \bigcap_i \gtq_i = \gta.$$    

We are left with the general case. We have an ideal $\gta \subset \Oo(\R^n)$ whose zero set is an H-set.

Since $I(\ceros(\gta)) = I(\ceros(\widetilde{\sqrt[{r}]{\gta} }))$, it is enough to check that $\widetilde{\sqrt[{r}]{\gta} }$ is a real  ideal, so that we are in  the previous case.

Assume  $\sum^r_{k=1}a_k^2 \in \widetilde{\sqrt[{r}]{\gta} }$ and let $K$ be a compact set. Then there is $h\in \Oo(\R^n)$ such that $\ceros(h) \cap K =\varnothing$ and $h (\sum_{k=1}^ra_k^2) \in \sqrt[{r}]{\gta}$. So  $\sum_{k=1}^r(ha_k)^2\in \sqrt[{r}]{\gta}$.
As $\sqrt[{r}]{\gta}$ is real  we get $ha_k \in \sqrt[{r}]{\gta}$. This happens for all compact sets so, $a_k \in \widetilde{\sqrt[{r}]{\gta} }$ for all $k$ and $\widetilde{\sqrt[{r}]{\gta} }$ is real, as required.
\end{proof}

\subsection*{ Bibliographic and Historical Notes.} \rm

Hilbert's  Nullstellensatz can be found in any text of commutative algebra, for instance \cite[pg.85]{amd}. The same result for rings of germs of holomorphic functions was proved by R\" uckert in \cite{ru}, see also \cite{gr}.

Real Nullstellensatz for $\R [x_1,\ldots ,x_n]$ and for  $\R\{x_1, \ldots ,x_n \}$ were both proved by Risler in \cite{r2} and \cite{r3}.

For further details on the Fr\'echet structure of $H^0(X,\Oo_X)$ and for classical results without proof we refer the reader to \cite{gr}.

A Nullstellensatz for Stein algebras was proved in \cite{of}, where Forster developed  
  classical results of Commutative Algebra of noetherian rings in the setting of a Stein algebra.
His  main results can also be found  in \cite{siu}.
Primary decomposition for closed ideals in a Stein algebra was extended to the real case by P. de Bartolomeis in \cite{db} and \cite{db1},
where closed ideals are replaced by saturated ideals. \index{saturated ideals}

Whitney's closure of ideals of smooth functions can be found for instance in \cite[II.1.3]{m1}.
 
The notion of \L ojasiewicz's radical 
has been used by many authors to approach different
problems mainly related to rings of germs, see for instance
\cite[p. 104]{d}, \cite[1.21]{k} or \cite[\S6]{dm}, but also in the
global smooth case \cite{abn}.

 Classical \L ojasiewicz's inequality for continous semialgebraic functions can be found in \cite[2.6.7]{bcr}.

The weak \L ojasiewicz's inequality is  the main result of \cite{abs}, up to a slight modification. Hironaka's
 desingularization theorem for a single analytic function appears in \cite{hi1}. It can be found also in \cite{bm}. 

The proof of the real Nullstellensatz \ref{perfect2} is a simplification of the proof of \cite{abf1} where there is also  the comparison between real  and \L ojasiewicz radicals.   

\newpage

\parindent=0pt

\chapter{ The $17^{th}$ Hilbert Problem for real analytic functions.}

Classical  $17^{th}$  Hilbert's Problem concerns polynomials and asks whether a positive semidefinite polynomial on $\R^n$ is a sum of squares in the field $\R(x_1, \ldots ,x_n) = \R(x)$  of 
rational functions. Namely

$$ \{\forall x\in \R^n \ p(x)\geq 0\} \ \Longrightarrow \exists \   g_1,\ldots ,g_k \in 
\R(x) \ |\  p=\sum_i g_i^2$$

This problem was solved by E. Artin in 1927. One main tool in the proof is the characterization of sums of squares in a real field $F$, that is, a field  admitting a total ordering,  as the elements of $F$ that are positive  with respect to  all orderings on $F$. This caracterization  comes from Artin-Schreier theory of real fields.
   
In this chapter we consider the same problem for analytic functions on $\R^n$ or more generally for  the algebra $\an(M)$ of analytic functions on a real connected analytic manifold $M$. Clearing denominators we can reformulate the problem as

\begin{problem}[$17^\text{th}$ Hilbert's Problem for analytic functions.]

 $$\{\forall x\in \R^n \  f(x)\geq 0\}  \ \Longrightarrow \exists \  g_0,\ldots, g_k \in \an(\R^n) \ | \ g_0^2f =\sum_{i=1}^k g_i^2$$

where (in order to avoid trivial cases) we assume $\ceros(g_0) \subset \ceros(f)$. 
\end{problem}

In the analytic setting also infinite convergent sums of squares have a meaning, but we need to make precise what means convergent. We will approach this point in Section \ref{infinito}.

In the analytic case 
Hilbert's $17^\text{th}$  problem (H17 for short) was completely solved in the local case, that is, for germs. In the global case only partial results are known and they depend on the zeroset of the given semidefinite analytic function. In the following sections we will rewiew some important facts from Artin-Schreier theory and the solution of H17 for germs. Then we will see the particular global cases when H17 has a positive solution. We will discuss some extensions of the problem. In particular we will see in Section \ref{p(M)} that the two questions, ``whether a positive semidefinite function is a sum of squares'' and ``the number of squares needed to represent a sum of squares'', which are completely different questions in the algebraic case, are   unexpectedly closely related in the real analytic case. 

\section{Artin-Schreier theory and the local case.}

The theory of  ordered fields was developed by  E. Artin and O. Schreier
in the 1920s  with the aim of solving 
Hilbert's $17^{\rm th}$ problem. We rewiew some facts of this theory.

\begin{defn} A {\em real} field $F$  is a field that admits at least one
  total ordering $\leq$, compatible with the field structure, i.e.
\begin{itemize}
\item If $x\leq y$, then for each $z\in F \ z+x \leq z+y $.
\item If $x\geq 0$ and $y\geq 0$, then $xy \geq 0$.
\end{itemize}
\end{defn}    

In a  real field squares are  positive with  respect to  any ordering,
because either $x\geq 0$ hence $x^2\geq 0$ or $-x \geq 0$ hence $x^2 =
(-x)(-x) \geq  0$.  As a  consequence, a real field  has characteristic
$0$. Indeed  $1+1 + \ldots +1 >1>0$.  So a real field  contains rational
numbers.

If we look at  the set $P$ of nonnegative elements  with respect to an
ordering in $F$, we find the following properties.

\begin{enumerate}
\item $P + P \subset P, \ P\cdot P \subset P$.
\item $F^2 \subset P$. 
\item $-1 \notin P$.
\item $P\cup -P = F$. 
\end{enumerate}

We call {\em cone} a subset of $F$ that satisfies (1) and (2). The cone is {\em proper} if it satisfies  also (3). A proper cone is the set  nonnegative elements of an ordering if it verifies also (4). 

Property (3) implies that $P$ is not trivial. Indeed  if $-1\in P$ then for all $x\in F$ we get $\displaystyle x = \left(\frac{x+1}{2}\right)^2  +(-1) \left(\frac{x-1}{2}\right)^2$, so $x\in P$. Observe also that property (3) implies $P\cap -P = \{0\}$.

The first example of proper cone is the set $\sum F^2$ of sums of squares.  

\begin{prop} Let $P$ be a proper cone of $F$. If $-a \notin P$, then $P[a] = \{x+ay: x,y \in P\}$ is also a proper cone.
\end{prop}
\begin{proof} Properties (1) and (2) are clear. Assume $-1 \in P[a]$. Then $-1 =x+ay$ for some $x,y \in P$. One has $y\neq 0$ because $-1 \notin P$, hence $\displaystyle -a = \frac{x+1}{y} = \frac{y(x+1)}{y^2} \in P$, which is a  contradiction. 
\end{proof}

\begin{prop} If a field $F$ has a proper cone $P$, then $F$ is real.
\end{prop}

\begin{proof} The set of proper cones is partially ordered by inclusion. Any chain of proper cones has a maximal element, which is the union of the cones in the chain.  By Zorn's lemma there is a maximal proper cone $Q$. We have to prove that $Q\cup -Q =F$. Take $a\in F$. If $a\notin Q$ then $Q[-a]$ is a proper cone that contains $Q$. As  $Q$ is maximal, we get $Q= Q[-a]$, hence $-a \in Q$ as required.   
\end{proof}

Next theorem  is a characterization of real fields.

\begin{thm} Let $F$ be a field. The following assertions are equivalent:
\begin{enumerate}
\item $F$ is real.
\item $F$ has a proper cone.
\item $ -1 \notin \sum F^2$
\item For all $x_1, \ldots, x_n \in F$
$$\sum_i x_i^2 =0 \Longrightarrow x_1= \dots = x_n=0$$.
\end{enumerate}
\end{thm}
\begin{proof}  (1) implies (2) and (3).  Also (3) implies (4)
and  (4) implies (3), which  implies (2). Finally (2) implies (1) by the previous proposition. 
\end{proof}

Sums of squares are caracterized by the following theorem.

\begin{thm}\label{AS} Let $F$ be a real field. Then $\sum F^2$ is precisely the set of elements that are not negative with respect to all orderings of $F$.
\end{thm}
\begin{proof} We know that $\sum F^2 \subset \bigcap P$. We have to prove the 
converse inclusion. Assume $a \in \bigcap P \setminus \sum F^2$. Since the last is a proper cone,  $ (\sum F^2)[-a]$ is also a proper cone which is contained in an ordering where $a$ is negative, which is a contradiction.
\end{proof}

Another important notion from Artin-Schreier theory is the one of {\em real closed field}.

\begin{defn}
A real field $F$ is {\em real closed} if no algebraic extension of $F$ is real.
\end{defn}

A caracterization of real closed fields is the following.

\begin{thm}\label{realclosed} For a field $F$ the following assertions are equivalent:
\begin{enumerate}
\item $F$ is real closed.
\item $F$ admits a unique ordering whose positive elements are squares and all polinomials of odd degree have a root in $F$.
\item $F$ is real and the algebraic extension $\displaystyle F[i]  = \frac{F[x]}{(x^2 +1)}$ is algebraically closed.\footnote{As usual we denote $i$ the imaginary unit, that is, a square root of $-1$}
\end{enumerate}
\end{thm}

Of course $\R$ is  a real closed field, but it is  not the smallest. The
smallest  is $\R_{alg}$, the field of real algebraic numbers, which is  real  closed because  $ \R_{alg}[i]  =
\C_{alg}$ that is the algebraic closure of $\Q$.

Real closed  fields and  $\R$ share several  results as  zeroes theorem,
Rolle's theorem and all their  consequences, in particular Sturm's theorem
holds true for real closed  fields. Theorem \ref{realclosed} implies that
 polynomial in one variable with coefficients in a real  closed field 
factorize into irreducible factors in the usual way  as in $\R[x]$.

All  real  fields  equipped  with  an ordering  admit  a  real  closed
algebraic  extension. More  precisely  Artin and  Schreier proved  the
following result.

\begin{thm}  Let $F$  be  a  real field  and  $\beta$  an ordering  on
  $F$. There exists an algebraic extension  $R$ of $F$ which is a real
  closed  field and  whose ordering  extends $\beta$.  Moreover, it  is
  unique up to isomorphisms that preserve the ordering.
\end{thm}

We give just an idea of the proof.

\begin{proof}  Take an  algebraic  closure $\overline  F$  of $F$  and
  consider the  family of  real fields $K$,  equipped with  an ordering
  that extends  $\beta$, such   that  $F\subset  K\subset  \overline
  F$. Apply Zorn's lemma to this family and take a maximal element $R$.
  One verifies directly that $R$  is real closed.  Uniqueness of the real closure can be proved in the  same way as for
  algebraic closures.
\end{proof}

Next we study  rings of germs of analytic functions. We  consider only the ring $\an_n$ of germs of analytic functions at $0\in \R^n$, that is, the ring of convergent power series with coefficients in $\R$, instead of considering a point in a C-analytic space $(X,\an_X)$. The main reason is that, as we saw in Chapter 1, C-analytic spaces can be embedded into an euclidean space $\R^N$ in such a way that the imbedding is a real analytic isomorphism in a neighbourhood of a given pont $x_0$. It is not so difficult to deduce H17 for $\Oo_{X,x_0}$ once we got H17 for $\Oo_N$.

Another important remark is that what we do can be repeated verbatim for germs at $0\in R^n$ where $R$ is an arbitrary real closed field and its topology is defined by the  ordering, that is, topology on $R$ is generated by open intervals.     

As we know from Theorem \ref{AS}, to prove that a germ $f\in \an_n$ is a sum of squares in the field $\mathcal M_0$ of germs of meromorphic functions, we have to show that $f$ is positive in all orderings of $\mathcal M_0$. Next theorem, which can be considered as {\em Artin-Lang homomorphism theorem} for the local analytic case, is the crucial tool. 

\begin{thm}\label{al} Let $\beta$ be an ordering in $\mathcal M_0$ and $f\in \an_n, \ f\neq0$. Assume $f$ is the germ of an analytic  function in an open neighbourhood  $U$ of $0$. Then $f$ is positive with respect to $\beta$ if and only if in all neighbourhoods of $0$ sufficientely small there is a point $y$ such that $f$ is defined in $y$ and $f(y)>0$.
\end{thm}

Artin-Lang homomorphism theorem has H17 as a consequence. Indeed if $f$  is negative with respect to a given ordering, there is a point $y$ close to $0$ such that $f(y) < 0$ while $f\geq 0$. 

\begin{remark}\label{al2} \hfill

\begin{enumerate}
\item The classical formulation of Artin-Lang theorem says that if we take an ordering $\beta$ on the quotient field $F$  of $\R[x_1,\ldots ,x_n]$ and an algebra  homomorphism from $\R[x_1,\ldots ,x_n]$ to the real closure of $(F, \beta)$  which respect the ordering, then there is an algebra homomorphism from $\varphi:\R[x_1,\ldots ,x_n] \to \R$ respecting the ordering. This means that for a polynomial $f$ positive with respect to $\beta$ there is a point $y= (\varphi(x_1),\ldots ,\varphi(x_n))\in \R^n$ such that $f(y) >0$. This is analogous to the statement of Theorem \ref{al}.  
\item The conclusion of Artin-Lang theorem is equivalent to the existence of  a germ of analytic arc $\gamma(t)$ at $0$ such that $f(\gamma(t)) >0$ for $t>0$. The arc  $\gamma(t)$ has $n$  components $ (\gamma_1(t),\ldots, \gamma_n(t)) \in \R\{t\}^n$.
 The quotient field of $\R\{t\}$ is a real field; an ordering on it is induced by the position of $t$ with respect to $\R$, for instance  by putting  $0<t< a$ where $a$ is any positive real number. This ordering is frequently called {\em standard ordering}. It is not real closed. 
Its real closure is the field of convergent Puiseux series $\widehat{\R\{t\}} $.
\end{enumerate}
\end{remark}

We will prove Theorem \ref{al} in a slightly different formulation. 
Consider the following statements.

\begin{itemize}
\item[$ (A_n)$] {\em Let $A$ be a finitely generated $\R$-algebra of dimension $n$ and assume that $A$ can be ordered. Let $f\in A$ be not a zero divisor. Then there is a homomorphism of $\R$-algebre $\varphi: A\to {\R\{t\}} $ such that  $\varphi(f)\neq 0$.}

\item[$(B_n)$] {\em Let $A$ be an analytic $\R$-algebra, of dimension $n$. Assume A is an integral domain and totally ordered  and let $f_1, \ldots f_k \in A\setminus\{ 0\}$. Then, there is a homomorphism of $\R$-algebre $\varphi: A\to {\R\{t\}} $ such that $\varphi (f_j) \neq 0$ for $j=1,\ldots, k$.}

\item[$ (C_n)$] {\em Let $\beta$ be an ordering in $\mathcal M_0$ and  $f_1, \ldots f_k \in \an_n$ be different from $0$. Then there is a $\R$-homomorphism $\varphi:\an_n \to {\R\{t\}} $ such that the sign of $\varphi(f_i)$ with respect to the standard ordering   of $\R\{t\}$ is the same as the sign of $f_i$ with respect to $\beta$ for $i= 1, \ldots, k$.}

\end{itemize}
An {\em analytic algebra} is
a quotient $A = \displaystyle  \frac{\an_m}{\gta}$. If $A$ is an integral
domain, the ideal $\gta$ is a prime  ideal. If $A$ is totally ordered
then  $\gta$  is  a  real  ideal.   Up  to  a  linear  change  of
coordinates, if $A$ has dimension $n$   we can assume $\an_n  \subset A$ and
$A$  is integral  over  $\an_n$.   The  quotient  field  of $A$  is
generated over the quotient field of $\an_n$ by an element $\theta \in
A$ (primitive element theorem). By Noether's normalization theorem, if $\delta \in \an_n$ is the 
discriminant of the minimal polynomial $p(T) \in \an_n[T]$ of $\theta$ then 
$\delta A \subset \an_n[\theta]$.  If $A$ is not an integral domain, up to a linear change of coordinates we can assume $\an_n  \subset A \subset K$, where is the total ring of fractions of $A$ and one has $K=\prod K_j$, where $K_j =$ frak $A/\gtp_j$ and $\gtp_1, \ldots ,\gtp_s$ are the minimal prime ideals of $A$.

Note that $(C_n)$ is Artin-Lang homomorphism theorem.

\begin{thm}\label{anbn}
For all $n>0$ one has
$$ (A_{n-1}) \Longrightarrow (C_n) \Longrightarrow (A_n), (A_n) \Longrightarrow (B_n).$$
Since $(A_0)$ and $(C_0)$  are trivial,   $(A_n),(B_n)$ and $(C_n)$ hold for all $n$.  
\end{thm}

\begin{proof}\hfill

$(A_{n-1}) \Longrightarrow (C_n)$. 

 Let $F_n, F_{n-1}$ be the real closures respectively  of $\mathcal M_0$ endowed with the ordering $\beta$ and of the quotient field of $\an_{n-1}$ endowed with the restriction of $\beta$. We have $f_1, \ldots f_k \in \an_n$. Up to  a linear change of coordinates we can assume they are regular with respect to the last coordinate. Hence there are Weierstrass polynomials $p_1,\ldots,p_k\in \an_{n-1}[x_n]$ and units $u_1,\ldots, u_k$ such that $f_j = u_j p_j, j=1,\cdots ,k$. The sign of $u_j$ is constant, we can assume $f_j=p_j, j =1,\ldots ,k$.
We can write $p_j(x_n) = \prod_i(x_n -\alpha_{j,i})^{n_{j,i}} \prod_l((x_n +b_{j,l})^2 +a_{j,l}^2)$. Hence the sign of $p_j$ only depends on the position of $x_n$ with respect to the  roots of $p_j$. 
 
Consider all the roots $\alpha_1< \cdots< \alpha_s$  in $F_{n-1}$ of the polynomials $p_1,\ldots, p_k$ and put $\alpha_0= -\infty, \alpha_{s+1} = +\infty$.  Assume
$\alpha_r<x_n<\alpha_{r+1}$. We can fix $\theta \in F_{n-1}$ in the same interval as $x_n$, that is, $\alpha_r< \theta <\alpha_{r+1}$. Hence $\theta - \alpha_l = \pm e_l^2$ for each $l$.

Next consider the $\R$-algebra $A$ generated by $\an_{n-1}, \theta, e_1, \ldots ,e_s, a_{j,l}, b_{j,l}, j=1,\ldots, k$. Let $f$ be the product of all these elements.  Since $A\subset F_{n-1}$, it is an integral domain, it has dimension $n-1$ because it is finitely generated over $\Oo_n$ and it is totally ordered by $\beta$. We apply $A_{n-1}$ and  we get $\varphi: A\to {\R\{t\}} $ such that $\varphi(f) \neq 0$.

Thus we  define $\psi: \an_n \to {\R\{t\}} $ such that
\begin{itemize}
  \item $\psi(x_i) = \varphi(x_i), i= 1,\ldots,n-1$.
\item  $\psi(x_n) = \varphi(\theta)$.
\end{itemize}
By construction the sign of $p_j$ is preserved as well as the one of $f_j, j=1,\ldots,k$.

\smallskip

$ \,(C_n) \Longrightarrow (A_n)$

 If $f\in A$ is not divisor of $0$, there is a $j$ such that $f\neq 0 $ mod $\gtp_j$. Thus, we can replace $A$ by $A/\gtp_j$, that is we can assume $A$ is an integral domain and $f\in A \setminus \{0\}$. After a linear change of coordinates we assume $\an_n  \subset A \subset K$, where $K$ is the quotient field of $A$ and a finite algebraic extension of the quotient field $F$ of $\Oo_n$. Let $\theta$ be a generator of $K$ over $F$ and denote $p(T)$ the minimal polynomial of $\theta$.
If $\delta$ is the discriminant of $p(T)$,   
then $\delta f= q(\theta)$, where $q\in \an_n[T]$ is a polynomial of degree smaller than the degree of $p$.  Since $p$ is irreducible $q$ and $p$ are coprime. By Bezout's theorem in $F[T]$ one has  $a(T)q(T) + b(T)p(T) =h$ where $h\in \an_n$ is not $0$. Substituting $T=\theta$ one gets $a(\theta) \delta f= h\neq 0$. So to find $\varphi$ such that  $\varphi(f) \neq 0$ it is enough to have $\varphi (\delta h)\neq 0$ and that $\varphi(p)$ gets a root in ${\R\{t\}}$. To do this we use Sturm's theorem, but the price is that the root will be in the real closure of $\R\{t\}$, that is, the field of Puiseux series. 
Let $f_1, \ldots, f_r$ be a Sturm sequence for $p(T)$ and take $M\in \R$ such that $-M<\theta<M$. Since $p(T) \in \an_n[T]$, also the coefficients of $f_1,\ldots, f_r$  are in $\an_n$.  So we have a list of elements in $\an_n$, namely: $\delta, h, f_1(-M), \ldots, f_r(-M), f_1(M),\ldots, f_r(M)$. By $(A_n)$ there is 
$\psi:\an_n \to {\R\{t\}} $ such that all elements in the list have the same sign as their images. Since $\psi(f_1),\ldots,\psi(f_r)$ is a Sturm sequence for $\psi(p)$, it has a root $\alpha \in [-M,M] \subset \widehat{\R\{t\}} $ as $p(T)$ had the root $\theta$ in the interval $ [-M,M]$.

Thus we  define:
\begin{itemize}
\item $\varphi(x_j) = \psi(x_j), j=1\dots, n$
\item $\varphi(\delta^{-1}) = \psi(\delta)^{-1}$
\item $\varphi(\theta)= \alpha$.
\end{itemize}     
 By costruction $\varphi(f) \neq 0$. But since $\alpha \in \widehat{\R\{t\}}$, it may happen that $\varphi$ does not take values in  $\R\{t\}$. Nevertheless the series $\varphi(\theta)$ is a series in $\displaystyle
t^{\frac{1}{k}}$ for some positive integer $k$, so it is enough to compose  $\varphi$ with the map $$ \displaystyle q:\R\left\{t^{\frac{1}{k}}\right\}\longrightarrow \R\{s\}, \ q\left(t^{\frac{1}{k}}\right) = s$$
to prove $A_n$.

\smallskip

$(A_n) \Longrightarrow (B_n)$

Let $f_1, \ldots, f_k \in A$ as in $B_n$ and let $K$ be the quotient field of $A$, ordered by the total ordering of $A$. Since for $j=1,\ldots, k, f_j\neq 0$ one has $f_j =\pm e_j^2, e_j\in K$.

Let $B$ the algebra generated by $A$ and $e_1, \ldots, e_k$. By $A_n$ there is $\varphi: B\to \R\{t\}$ such that $\psi(e_j)\neq 0$ for $j=1,\ldots, k$. In particular for all $j$ one has $\varphi(f_j) = \pm (\varphi(e_j))^2$ and it has the same sign as $f_j$. 
\end{proof} 

\section{Low dimensional global case.}\label{lowdim}

In this section we consider analytic manifolds of dimension $1$ or $2$. As we will see in these cases not only the $17^{\rm th}$ Hilbert's Problem has a positive solution, but also denominators are not needed. 

\subsection{Dimension 1.}

A connected real analytic manifold of dimension $1$ is isomorphic either to the real line or to a circle $S^1$.
In the first case it is an easy exercise to prove that a real analytic function $f:\R\to\R$ such that $f(t)\geq 0 \  \forall t\in \R$ is a square in $\an (\R)$.

In the second case we can argue as follows. Let  $f: S^1\to \R$ be a positive semidefinite analytic function.  Assume $f$ is not the zero function. At any point $a\in S^1$ there is an analytic germ $g_a$ such that $f_a= g_a^2$. 
Consider the sheaf of ideals $\Ff_x = g_x\Oo_{S^1,x}$. It is coherent because, up to the sign, the square root is unique.  $\Ff$ is locally principal, hence, by a result of Coen, it is globally generated by $g_1, g_2 \in \Oo(S^1)$. Thus $g(_1^2+g_2^2)\Ff_x = f_x\Oo_{S^1,x}$. The quotient $\displaystyle \frac{f}{g_1^2+g_2^2}$ is a positive unit $u\in \Oo(S^1)$. Thus there is $v>0$ such that $u=v^2$ and $f= (vg_1)^2 + (vg_2)^2$ is a sum of two squares in $\Oo(S^1)$.

\subsection{Dimension 2.} 
Let $M$ be a real analytic manifold of dimension $2$ and $f\in \an(M)$ be a positive semidefinite  analytic function. The zero set $\ceros (f)$ will be in general an analytic curve with countably many global irreducible components and a discrete set $D$. This implies that the germ $f_a$ at a point of $M\setminus D$ is  a positive unit times a square, while at the points of $D$ it may be that two squares are needed. 

 We costruct a vector bundle of rank $2$ associated to $f$. If $D=\varnothing$ 
there is a countable  open covering $\{U_j\}_j$ of $M$ with the following properties.
\begin{itemize}
\item For all $j$ $U_j$ is connected.
\item For $j\neq l$ $U_j\cap U_l$ is connected
\item There is  $f_j\in \Oo(U_j)$ such that $f=f_j^2$ on $U_j$.
\end{itemize}
We take the $\R$-vector space $W$ generated by ${\bf e_1}$ and ${\bf e_2}$ and the quadratic form $f(t_1^2 +t_2^2)$. Since on $U_j$ the function $f$ has a well defined square root $f_j$ there is an ortonormal basis ${\bf g_{j,1}}$,  ${\bf g_{j,2}}$ such that ${\bf e_1}= f_j{\bf g_{j,1}},{\bf e_2}= f_j{\bf g_{j,2}}$ 

Consider now the vector bundle $E_f$ of rank 2 which is trivial on $U_j$ for all $j$ and whose transiction matrices  are diagonal with $\displaystyle \frac{f_j}{f_l}$ on the diagonal. As $f_j^2= f_l^2 =f$, these matrices are ortogonal. So the scalar product induced by the quadratic form induces a scalar product on the space of sections of the fiber bundle. Assume now there is a section $t:M \to E_f$ that does never vanish and that verifies $\langle t, t\rangle =1$. Consider the functions $T_1 = \langle t, {\bf e_1}\rangle, T_2 = \langle t, {\bf e_2}\rangle$. On $U_j$ we have $T_1 = t_{j,1}f_j, T_2 = t_{j,2}f_j$, so $T_1^2+T_2^2 = f(t_1^2 +t_2^2) =f$. Hence $f$ is a sum of two squares, provided $E_f$ has a never vanishing section. 

If $D\neq \varnothing$ we costruct an open covering $\{U_i\}$ of $M\setminus D$ as before. 
For a point $a\in D$ there are germs $h_a, \varepsilon_a$ such that $\varepsilon_a$ vanishes only at $a$ and $f_a = h_a^2 \varepsilon_a$. Moreover $\varepsilon_a =\alpha_a^2 + \beta_a^2$. These equalities hold true in an open neighbourhood $V_a$. We can assume $\alpha_a^2 + \beta_a^2 >0$ on $V_a \setminus \{a\}$ and that all these open sets $V_a$ are disjoint. 

We have  an open  covering $\{U_i\}_i \cup\{V_a\}_{a\in D}$ of $M$. The fiber bundle $E_f$ will be trivial on this open covering.
We have to construct the transition matrices. We consider the vector space $W$ generated by $\mbox{\bf e}_1, \mbox{\bf e}_2$ with respect to the field ${\mathcal M}(M)$ of meromorphic functions on $M$ endowed with  the quadratic form $f(t_1^2 + t_2^2)$. Hence there is a scalar product on $W$ and $\langle\mbox{\bf e}_1,\mbox{\bf e}_1\rangle = \langle\mbox{\bf e}_2,\mbox{\bf e}_2\rangle =f$

On each open set $U_j$ the function  $f$ has a square root $f_j$.        The transition matrices are costructed as follows. 

\begin{itemize}
\item On $U_k\cap U_j\quad \quad
g_{k,j} = \left( \begin{array}{cc}
\frac{f_j}{f_k}& 0\cr
0&\frac{f_j}{f_k}
\end{array} \right) $

\item On $U_j\cap V_a \quad \quad
g_{j,a}= \displaystyle \frac{h_a}{f_k} 
\left( \begin{array}{cc}
\alpha_a & \beta_a \cr
-\beta_a & \alpha_a
\end{array} \right) $

\end{itemize}

  All transition matrices are orthogonal and their determinant is 1. Hence the standard scalar product on $\R^2$ induces a well defined  Riemannian structure on $E_f$. 

The vectors $\mbox{\bf e}_1,\mbox{\bf e}_2$ define analytic sections $e_1,e_2$ of $E_f$, more precisely 
on $U_ j \ e_{1,j}(x) = (f_j(x), 0), e_{2,j}(x) = (0, f_j(x))$, while on $V_a \ e_{1,a}(x) = (h_a(x)\alpha_a(x), -h_a(x)\beta_a(x))$  and $ e_{2,a}(x) =   (h_a(x) \beta_a(x), h_a(x) \alpha_a(x))$. Moreover $\langle{e}_1,{ e}_1\rangle = \langle{ e}_2,{e}_2\rangle = f$.
\newpage
\begin{prop} \hfill
\begin{enumerate}
\item If $E_f$ has a never  vanishing analytic section, then $f $ is a sum of $2$ squares in $ \an(M)$.   
\item If $E_f$ has an analytic section which does not vanish on the zero set of $f$, then $f $ is a sum of $3$ squares in $ \an(M)$. 
\end{enumerate}
\end{prop}

\begin{proof} (1) Let $s = (s_1,s_2)$ be a never vanishing section of $E_f$. We can assume $<s,s> =1 = s_1^2+ s_2^2.$  Define $S_1 = \langle s, { e}_1\rangle, S_2= \langle s, e_2\rangle$. Then
there is a positive unit $u$ such that $f=u (S_1^2 +S_2^2)$. Indeed this is true where $f\neq 0$. So, fix a point $x$ such that $f(x) = 0$. Then $\displaystyle \frac{f}{S_1^2 +S_2^2} = \frac{f}{f(s_1^2+ s_2^2)} = \frac{1}{ s_1^2+ s_2^2} >0$.

(2) Let $s$ be an analytic section that does not vanish on the zero set of $f$. Define $S_1,S_2$ as before. Then, there is a positive unit $u$ such that $f= u(S_1^2 +S_2^2+f^2)$. Again this is true where $f$ does not vanish, but in any case $ \displaystyle \frac{f}{S_1^2 +S_2^2+f^2} = \frac{1}{s_1^2+ s_2^2 +f} >0$.   
\end{proof}

As a corollary we prove.

\begin{cor}\hfill

\begin{itemize}
\item Let $M$ be a not compact analytic surface and $f\in \an(M)$ be positive semidefinite. Then $f$ is a sum of $2$ squares in $\an(M)$. 
\item Let $M$ be a compact analytic surface and $f\in \an(M)$ be positive
semidefinite. Then $f$ is a sum of $3$ squares in $\an(M)$.
\end{itemize}
\end{cor}
\begin{proof}
If $M$ is not compact and dim $M =2$, rank $2$ bundles on $M$ are trivial and admit analytic never vanishing sections.
If $M$ is compact
 consider the vector bundle $E_f$. Since the zeroset of $f$ has at most dimension $1$, one can  construct a smooth section $\sigma$ that does not vanish on $\ceros(f)$. Then it is enough to approximate $\sigma$ by an analytic section $s$ on the compact open topology of $C^\infty(M, E_f)$. 
\end{proof}

\section{Pytagoras number for curves and surfaces.}\label{pitagora} 

\begin{defn} For a ring $R$ the {\em Pytagoras number} $p(R)$ is the minimum integer $p$ such that all sums of squares in $R$ can be written as sums of at most $p$ squares or $p(R) =\infty$ if such minimum does not exist. 
\end{defn}

For instance, what we got in Section 2 is that for a non-singular connected real analytic curve $X$ either $p(\Oo(X)) =1$ if $X$ is not compact or $p(\Oo(X)) =2$ if $X$ is compact. Also we get either $p(\Oo(X)) =2$ or $p(\Oo(X))=3$ if $X$ is a connected non-singular analytic surface which is  either not compact or compact.

In this section we racall what happens for singular curves or surfaces.

We begin by singular curves. Remember that the structural sheaf $\Oo_X$ of an analytic curve is coherent and that its normalization is not singular. Since the normalization $\check X$ and $X$ share the same ring of meromorphic functions, applying what we got for non-singular curves, we get $p(\mathcal M(X))=1$ if all connected components of $\check X$ are not compact, while $p(\mathcal M(X))=2$
if some of them are compact. 

So we look at $p(\Oo(X))$.
First of all there is a characterization of singular curves for which H17 holds true without denominators, namely.

\begin{prop}
Let $X$ be a real analytic curve. Then the following conditions are equivalent.
\begin{enumerate}
\item H17 holds in $\Oo(X)$ without denominators. 
\item H17 holds in $\Oo_{X,x}$ for all $x\in X$ without denominators.
\item Each germ $X_x$ is a finite union of nonsingular independent branches.
\end{enumerate} 
\end{prop} 

Hence for singular curves whose singularities are more involved, denominators are needed. Concerning Pytagoras number, one has the following result.

\begin{thm} Let $X$ be a real analytic curve. Then $p(\Oo(X)) =\displaystyle \sup_{x\in X}p(\Oo_{X,x}) +\varepsilon$, where $\varepsilon$  can be $0$ or $1$.
\end{thm}

Concerning singular analytic surfaces we consider only a coherent analytic surface $X$, because in this case  the normalization map $\pi:\check X\to X$ is surjective, so that both spaces share the ring of meromorphic functions $\mathcal M(X)$. Then we get.

\begin{thm} \hfill
\begin{enumerate}
\item  Let $X_x$ be a germ of  real analytic surface and $f\in \Oo(X_x)$ be a positive semidefinite function germ. Then there are analytic germs $g, h_1,h_2, h_3, h_4$ such that 
$$ g^2f = h_1^2 +h_2^2+h_3^2+ h_4^2$$
and $\{g=0\}\subset \{f=0\}$ 
 \item Let $f\in \Oo(X)$ be a positive semidefinite function on a normal analytic surface $X$.Then, there are analytic function  $g, h_1,h_2, h_3, h_4, h_5 \in \Oo(X)$ such that 
$$ g^2f = h_1^2 +h_2^2+h_3^2+ h_4^2 +h_5^2$$
and $\{g=0\}\subset \{f=0\}$  
\end{enumerate}
\end{thm}

\begin{cor} The Pythagoras number $p(\mathcal M(X)) =5$ for all coherent real analytic surfaces.
\end{cor}

\begin{proof} Indeed this is true for $\check X$ and $\mathcal M(\check X)= \mathcal M(X)$.
\end{proof}

\section{Excellent rings: some recalls.}\label{excellent}
The notion of excellent ring is rather relevant in the theory of real analytic functions. Here we recall definition and main properties.
 The general reference for this section is \cite{mat1} where proofs can be found.
\newpage
\begin{defn}\label{exc} A noetherian ring $A$ is \em excellent \rm if it has the following properties.
\begin{enumerate}
\item $A$ is {\em universally catenary}, that is, for any finitely generated $A$-algebra $B$ and any prime ideals $\gtq \subset \gtp$ in $B$ one has $\mbox{\rm ht}(\gtp) =\mbox{\rm ht}(\gtp/\gtq) +\mbox{\rm ht}(\gtq)$.   
\item For any localization $B=A_{\gtp}$ at a prime ideal $\gtp\subset A$ the homomorphism $B\to \widehat B$ of $B$ in its completion is regular. 
\item For any finitely generated $A$-algebra $B$, there exist $h_1,\ldots ,h_s\in B$ such that $B_{\gtp}$ is regular if and only if $h_i\notin \gtp$ for some $i$, that is, the {\em singular locus} is the zeroset of $h_1,\ldots ,h_s$.  
\end{enumerate}
\end{defn}
Several properties of excellent rings are collected in the following proposition.
 
\begin{prop}\label{excpro}\hfill

\begin{enumerate}
\item Let $S$ be a multiplicative subset of the ring $A$. If $A$ is excellent, then $S^{-1}A$ is excellent.
\item Let $B$ a finitely generated $A$-algebra. If $A$ is excellent, also $B$
is excellent.
\item Let $A$ be an excellent domain and $B$ a domain and a finitely generated $A$-algebra. Let $K,L$ be their fields of fractions. Let $\gtq \subset B$ be a prime ideal with residue field $k(\gtq)$, consider the prime ideal  $\gtp= \gtq\cap A$ with residue field $k(\gtp)$. Then,
$$\mbox{\rm ht}(\gtq)=\mbox{\rm ht}(\gtp) + \mbox{\rm deg}(L:K) -\mbox{\rm deg}(k(\gtq):k(\gtp)$$
\item Let $A$ be a local excellent ring. Then, its adic completion $\widehat A$ is reduced (resp. normal, regular) if and only if $A$ is reduced (resp. normal, regular).
\item Let $A$ be a local excellent ring and $\widehat A$ be its adic completion.
Let $\gtq$ be a prime ideal in $\widehat A$ and set $\gtp= \gtq\cap A$. Then, if $A_{\gtp}$ is regular $\widehat A_{\gtq}$ is regular too.
\item Let $A$ be an excellent domain. Then, there is $h\in A, h\neq 0$ such that $A_ {\gtp}$ is regular whenever $h\notin \gtp$. 
\item Let $A$ be an excellent domain with quotient field $K$ and let $L$ be a finite extension of $K$. Let $\check A$ be the integral closure of $A$ in $L$. Then $\check A$ is a finite $A$-module. 
\item Let $A$ be a local noetherian ring that verifies Condition 1 in Definition \ref{exc} and such that the homomorphism $A\to \widehat A$ is regular. Then $A$ is excellent.
\item The quotient of an excellent ring $A$ by an ideal is excellent.
\end{enumerate}
\end{prop}

One can prove that rings in the following list are excellent rings.
\begin{itemize} 
\item Finitely generated algebras over a field.
\item Analytic algebras over $\R$ or $\C$.
\item Formal or algebraic algebras over a field $k$ of characteristic $0$,  i.e homomorphic images of the ring of formal power series or the ring of algebraic power series.  
\end{itemize}

Next theorem is a characterization of excellent rings.

\begin{thm}\label{derivazioni}
Let $k$  be a field of characteristic 0 and let $A$ be a regular ring containing $k$ such that 
\begin{enumerate}
\item  For any maximal ideal $\gtm \subset A$ the extension $k \to A/\gtm $ is algebraic. 
\item All maximal ideals of $A$ have the same height, say $n$.
\item There exist derivations $D_1, \ldots ,D_n$ of $A$ over $k$ and elements $x_1,\ldots ,x_n \in A$ such that $ D_ix_j= \delta_{ij}$
\end{enumerate} 
then $A$ is an excellent ring.
\end{thm}

\section{Compact  zeroset.} 

As we saw in the previous sections, Artin-Lang theorem, as stated in Remark \ref{al2},   caracterizes positive elements with respect to an ordering on a real  field $F$. This theorem holds true for the ring of analytic function germs, ( Theorem \ref{anbn}), but it is not proved for analytic functions on a fixed open set. We have to use a different tool to decide whether a  function $f$ is  positive with respect  to a  given  ordering. 

Let $\sigma$ be an ordering on a real field $F$ that contains $\R$. Associated to $\sigma$ there is
a subring $W_\sigma$ of $F$ consisting of  bounded elements: $W_\sigma= \{f\in F: \exists n \in \N \ \mbox{ \rm such that } 
|f| < n\}$. It is a valuation ring with maximal ideal $\mathcal M_\sigma = \{f\in F:  |f|< \frac{1}{n}\, \forall n\in \N\}$. 

Note that the residue field of $W_\sigma$ is precisely $\R$. The residue homomorphism maps positive units onto positive real numbers and negative units onto negative real numbers. It is called {\em the place associated to} $\sigma$.

Let $M$ be a connected real analytic manifold   and let $\an(M)$  be  its ring of analytic function with quotient field  $F$. Consider  an ordering $\sigma$ on $F$.
\begin{defn}The ordering $\sigma$ is  {\em centered at the point $p\in M$} if all analytic functions that  are positive  at $p$ are positive with respect to $\sigma$. Otherwise, $\sigma$ is called {\em free}.
\end{defn}

If $\sigma$ is centered at $p$ we can consider the ring $L= \an(M)_{\gtM_p}$, where $\gtM_p$ is the maximal ideal of functions vanishing at $p$. It is a regular local ring. A system of parameters for $L$ is given by a minimal set of generators of $\gtM_p$. Also $L$ is an excellent ring.

\begin{prop} The local ring $L$ is dominated by $W_\sigma$, that is $L\subset W_\sigma$ and $\mathcal M_\sigma \cap L$ is the maximal ideal of $L$ .
\end{prop}  

\begin{proof} The residue homomorphism in $L$ is the evaluation at $p$. But also the place associated to $W_\sigma$  when restricted to $W_\sigma$ is the evaluation at $p$. So  $f\in L \mapsto f(p)< \infty$, hence $L\subset W_\sigma$.
Also $\mathcal M_\sigma \cap L = \gtM_\sigma$ because $g\in \mathcal M_\sigma \cap L$ if and only if $g(p)=0$.
\end{proof}

We have a chain of inclusions: $L\subset \an_p \subset \widehat{\an_p}$, where $\widehat{\an_p}$ is the completion of $\Oo_p$. The first inclusion follows  taking  germs at $p$ of functions in $L$. Note that $\widehat{\an_p}$ is also the completion of $L$.
 
\begin{lem}\label{estensione} The ordering $\sigma$ extends to an ordering on $\an_p$.
\end{lem}
\begin{proof}In this proof we use some properties of the real spectrum of  $L =\an(M)_{\gtM_p}$ and $\widehat{\an_p}$. References can be found at the end of this chapter. Both $L$ and ${\an_p}$ are excellent local rings and share the same completion. Since the residue field of $L$ is $\R$, the ordering 
$\sigma$ induces on $\R$ its unique ordering.  the unique ordering of $\R$, which is also the residue field of the completion $\widehat{ \Oo}_p$, induces on $\widehat{ \Oo}_p$ several orderings, but the map induced by the inclusion $L\subset \widehat{\Oo_p}$ between their real spectra is surjective. It associates to every ordering of  $\widehat{ \Oo}_p$ its restriction to $\an(M)_{\gtM_p}$, thus there is an ordering   $\widehat \sigma$ that extends $\sigma$. 
  The restriction of this ordering to $\an_p$ extends $\sigma$.
\end{proof}

Next we associate to $\sigma$ an ultrafilter of closed semianalytic sets. 
Consider  for all $f\in \mathcal M_\sigma$ the closed set $f^{-1}([-\delta,\delta])$, where $\delta $ is a positive real number, and let $G_\sigma$ be the collection of these sets.

\begin{prop}\label{ultrafiltro} 
The family $G_\sigma$ has the following properties.
\begin{enumerate}
\item $\varnothing \notin G_\sigma$.
\item  If $U,V \in G_\sigma$, then $ \exists W \in G_\sigma, W\subset U\cap V$.
\item If $U,V$ are closed sets and both intersect all elements in $G_\sigma$ then $U\cap V \neq \varnothing$. 
\item Let $V_0$ be compact. If it intersects all elements in $G_\sigma$, then 
$\bigcap_{V\in G_\sigma} V= \{p_0\}$ and $p_0 \in V_0$.
\end{enumerate}   
\end{prop}
 Note that the first two properties imply that $G_\sigma$ is a filter basis, hence it is contained in an ultrafilter of closed semianalytic sets, unique by property (3).
\begin{proof}
(1) If for some $f\in \mathcal M_\sigma$ and some $\delta$ one gets $f^{-1}([-\delta,\delta]) =\varnothing$, then 
$f$ has constant sign, assume $f(x) >\delta$ for each $x$. Then $f-\delta$ is a square, so $f>\delta$ with respect to $\sigma$. This is a contradiction because $f\in \mathcal M_\sigma$, that is, $f$ is infinitesimal with respect to $\sigma$.   

(2) If $U= f^{-1}([-\delta,\delta])$ and $V= g^{-1}([-\eta,\eta])$, take $W= (f^2 + g^2)^{-1}([-\varepsilon,\varepsilon])$, where $\varepsilon$ is the minimum between $\delta, \eta$. 

(3) Assume $U\cap V =\varnothing$. Then one can find an analytic function $f\in W_\sigma$ such that $f>2\delta$ on $U$ and $f<-2\delta$ on $V$. Let $r$ be the image of $f$ mod $\mathcal M_\sigma$ and take $W = (f-r)^{-1}([-\delta,\delta])$. Then $W$ cannot meet both $U$ and $V$.

(4) Consider the family $V_0\cap V$, where  $V\in G_\sigma$. The intersection of all elements  in this family cannot be empty because $V_0$ is compact. Hence the ultrafilter generated by $G_\sigma$ is not free and there is a unique point $p_0$ in the intersection of all the elements in the ultrafilter. But then $p_0 \in V_0$.
\end{proof}

Next result caracterizes when an analytic function $f$ is positive 
with respect to $\sigma$ in terms of $G_\sigma$.

\begin{thm}\label{units}
Let $f\in W_\sigma$ be a unit. Then $f$ is positive with respect to $\sigma$ if and only if  there is an element $V\in G_\sigma$ such that $f$ is strictly positive on $V$.
\end{thm}

\begin{proof}
First of all note that a unit in $W_\sigma$ is positive if and only if its image  $p_\sigma(f) \in \R = W_\sigma/\mathcal M_\sigma$ is positive.
Consider now the image by $f$ of $G_\sigma$. We have

$$\bigcap_{V\in G_\sigma} f(V) = \{p_\sigma(f)\}$$

Indeed $f= p_\sigma(f) +g$ where $g\in \mathcal M_\sigma$, so it is enough to prove the claim for $g$. Let $R\in g(G_\sigma)$. For all $n\in \N$, we have $g^{-1}([-1/n, 1/n]) \in G_\sigma$, hence $R\cap [-1/n, 1/n] \neq \varnothing$. But $R$ is closed, so $0 = p_\sigma(g)\in R$.

Next take a positive unit $f\in W_\sigma$. The set $Y= (f- p_\sigma(f))^{-1}([-\delta,\delta]) \in G_\sigma$. Taking $\displaystyle \delta = \frac{p_\sigma(f)}{2}$ one can check that $f$ is positive on $Y$. 

Conversely, assume $f$ is positive on an element $Y\in G_\sigma$. Then $p_\sigma(f) \in f(Y)$ and $f$ is positive with respect to $\sigma$.  
\end{proof}

\begin{cor}[Ultrafilter Theorem]
Let $f\in F$. If there is $Y\in G_\sigma$ such that $f$ is strictly positive on $Y$, then $f$ is positive with respect to $\sigma$. 
 
\end{cor}
\begin{proof}
Consider the function $g= \displaystyle \frac{f}{1+f^2}$. Then $g$ and $f$ have the same sign with respect to $\sigma$. Now $g\in W$ and it is positive on $Y$. Hence its sign is positive.  
\end{proof}

\begin{cor} Let $f\in \an(M)$ be such that  the set $\{x\in M |\  f(x) \leq
  0\}$ is compact. Then $f$ is  positive in all orderings of $F$ whose
  associated ultrafilter is free.
\end{cor}

\begin{proof} Let $\sigma$ be an ordering whose associated ultrafilter is free.
Then there is an element of $G_\sigma$ that does not meet the compact set  $C= \{f \leq 0\}$. Otherwisw $C$ would intersect all elements of  $G_\sigma$ which implies that all such elements would meet at a point of $C$. This is a contradiction because the ordering is free. So $f$ is strictly positive on some $Y\in G_\sigma$, hence positive with respect to $\sigma$ by the Ultrafilter Theorem.
\end{proof} 

Thus, the condition {\em  $\{f\leq 0\}$ compact} is enough to assure  positiveness with respect to free orderings. What about centered orderings?

\begin{prop}Let $f\in \an(M)$ be positive semidefinite. Then $f$ is positive with respect to any ordering of $F$ centered at some point of $M$. 
\end{prop}

\begin{proof}
Let $\sigma$ be an ordering on $F$ such that the intersection of all elements in its associated ultrafilter is a point $p_\sigma \in M$. Hence every $g\in \Oo(M)$ such that $g(p_\sigma) >0$ is positive with respect to $\sigma$. So, if $f(p_\sigma) > 0$, we are done.

If $f(p_\sigma) = 0$ we argue as follows.
Consider the local ring $\an(M)_{\mathcal N}$, where $\mathcal N$ is the maximal ideal of functions vanishing at $p_\sigma$. It embeds in the ring of germs $\an_{p_\sigma}$. We know by Lemma \ref{estensione} that the ordering $\sigma$ extends to a total ordering $\tau$  on $\an_{p_\sigma}$. Consider the germ of $f$ at $p_\sigma$.. Since $f\geq 0$ there is an arc germ $\gamma(t)$ such that $\gamma(0) = p_\sigma$ and $f(\gamma(t)) >0$ for
$t>0$. By Artin-Lang theorem the germ of $f$ is positive with respect to $\tau$. Thus  $f$ is positive with respect to $\sigma$. 
\end{proof}

We can prove the following result.

\begin{thm}\label{compatto} Let $f\in \an(M)$ be a positive semidefinite analytic function with compact zeroset. Then there are $g_0,\ldots ,g_k \in \an(M)\setminus\{0\}$ such that $g_0^2f= g_1^2+\dots +g_k^2$.   
\end{thm}

\begin{proof} As $f$ is positive semidefinite, it is positive with respect to all orderings of $F$ that have a center. Since $\{f\leq 0\}= \{f=0\}$ is compact it is positive with respect to all free orderings of $F$. By  Theorem \ref{AS}, we conclude $f$ is a sum of squares in $F$.    
\end{proof}

\begin{remark} If $M$ is compact all orderings on the quotient field $F$ of the ring $\Oo(M)$ are centered by property (4) of their associated ultrafilters (Proposition \ref{ultrafiltro}). Hence any positive semidefinite  function is a sum of squares in $F$. The same result as Theorem \ref{compatto} is true if we consider a compact set $K\subset M$ and the ring $A$ of germs of analytic functions at $K$. 
\end{remark}

\section{Infinite sums of squares.}\label{infinito}

As we said, in the analytic case also infinite sums of squares make sense, but we have to precise what means convergence.  Let us prove that in the compact open topology of $\Oo(\R^n)$ all positive semidefinite functions are sums of a convergent series of squares without denominators.

\begin{prop}\label{fabrizio} Let $f\in \Oo(\R^n)$ be positive semidefinite. Then there is a series of squares which converges uniformely to $f$ on all compact subsets of $\R^n$.
\end{prop}
\begin{proof} 

From the identity $$f=f\frac{f+1}{f+1}=\frac {f^2}{f+1}+\frac {f}{f+1},$$
as  $f$ is  positive semidefinite on all $\R^n$, we have that the first addend $\frac {f^2}{f+1}$ is a square, say $q_1^2$,  and  the second one is an analytic function 
$g_1=\frac {f}{f+1}$ satsfying $0\leq g_1<1$.
Thus, $f=q_1^2+g_1$. 

Applying the previous procedure to $g_1$ we get $g_1=q_2^2+g_2$ and iterating the procedure, we get after $k$ steps  
$$f= q_1^2+q_2^2+q_3^2 +\cdots +q_k^2+g_k$$
where $\displaystyle q_l^2= \frac{f^2}{((l-1)f+1)(lf+1)}$ and $g\displaystyle _k=\frac {f}{kf+1}$.

As $f\geq 0$ on  $\R^n$, one has $kf<kf+1$, so $\frac{f}{kf+1}<\frac {1}{k}$ on $\R^n$. Consequently $\displaystyle f- \sum_{h=1}^k q_h^2 =\frac {f}{kf+1}< \frac{1}{k}$, that is, the series of squares converges uniformely in $\R^n$ to $f$.

Observe that one can write this series as

$$\sum _{l\geq 1}\frac{f^2}{((l-1)f+1)(lf+1)}= \sum _{l\geq 1} \frac{f}{(l-1)f+1} -\frac{f}{lf+1}$$

As $\displaystyle \lim_{l\to \infty} \frac {f^2}{lf+1}=0$,   it holds that the infinite  sum converges to $f$.
\end{proof}

Convergence in the compact open topology endows $\Oo(\C^n)$ with a structure of Fr\'echet space as we saw in Chapter 3. This is not true in the real case, because $\Oo(\R^n)$ is a dense subspace of the Fr\'echet space $\mathcal C(\R^n)$ of continuous functions. Even if a series of analytic functions converges to an analytic function as in Proposition \ref{fabrizio} there are several unpleasant behaviours. For instance.

\begin{itemize}
\item The multiplicity of the sum is not the infimum  of the multiplicities of the addends. 
\item The Taylor series at a point of the sum is not the sum of the Taylor series of the addends.
\item Theorem \ref{modulichiusi} of Chapter 3 does not hold true. All partial sums of the series in Proposition \ref{fabrizio} belong to the ideal generated by $f^2$, but their uniform limit $f$ does not. 
\end{itemize}

Thus, we prefer to use a stronger concept of convergence.

\begin{defn}\label{strongconv}
A sequence of analytic functions $\{g_i\}_{i\geq 1} \subset \an(\R^n)$  \em strongly converges \rm to $f$ if 
there is an invariant open neighbourhood $\Omega$ of $\R^n$ in $\C^n$ to which  each $g_i$ admits   a holomorphic extension $G_i$ and the sequence 
$ \{G_i\}_{i\geq 1}$ converges uniformely on the compact subsets of $\Omega$ to a (holomorphic) function $F$ and $f=F_{|\R^n}$. 
\end{defn}

The topology associated to this convergence avoids the difficulties above because any convergent series comes from a complex one on a suitable neighbourhood $\Omega$ of $\R^n$ in $\C^n$ and $\Oo(\Omega)$ satisfies Thorem \ref{modulichiusi}.

Observe that if the  series $\sum g_i^2$ strongly  converges, it converges to  a positive semidefinite analytic function on $\R^n$.

From now on, by  convergence we mean the strong convergence introduced in Definition \ref{strongconv}.

Considering  infinite sums suggests to revisit some definitions. 
First we consider  the formulation of the Hilbert's $17^{th}$ Problem.  

\smallskip

\begin{problem}(Hilbert's $17^{th}$ Problem for analytic functions)

Let $f\in \Oo(\R^n)$ be a positive semidefinite function. Then there exists  
 a sequence  $\{g_k\}_{k\geq 0} \subset 
\an(\R^n)$ such that $\ceros(g_0) \subset \ceros(f)$,  $g_0^2f =\sum_{k\geq 1} g_k^2$
and the series  converges.
\end{problem}
\smallskip

In the statement above only finitely many  different  denominators are allowed in the representation as sum of squares.
Even if the number of addends is finite, denominators are needed, as we will see in Section \ref{esempi}.

We will find several situations where this problem has    a positive answer. As H17  is closely related to Nullstellensatz, it seems natural to consider a  new definition of real radical, where the finite sum of squares is replaced by an infinite sum of squares. 

\begin{defn} The {\em real analytic radical} of an ideal $\gta \subset \Oo(\R^n)$ is defined as

$$\sqrt[{ra}]{\gta} = \left\{g\in \Oo(\R^n) | \exists m\in \N: g^{2m} + \sum_{i=1}^\infty a_i^2 \in \gta\right\}.$$ 

We can define  an $H^{ ra}$-{set} as a C-analytic set $Z$ such that each $g\in \Oo(\R^n)$ with $\ceros(g) =Z$ is  an infinite sum of squares of meromorphic functions.
\end{defn} 

Coming back to Nullstellensatz, if the zeroset $Z$ of an ideal $\gta$ is an  $H^{ ra}$-{set} one wonders whether $\I(Z) =\widetilde{ \sqrt[{ra}]{\gta}}$.

If we use the compact open topology, then $\sqrt[{ra}]{\gta}= \sqrt[\text{\L}]{\gta}$ because of Proposition \ref{fabrizio} for each ideal $\gta$. 
If we  consider strong convergence we obtain the same. The proofs of Remark \ref{Hset} and Lemma \ref{denominators} in Chapter 3 can be modified to hold for $H^{ ra}$-{sets}. Thus, Theorem \ref{h17nss}   holds true if we replace $\sqrt[{r}]{\gta}$ with $\sqrt[{ra}]{\gta}$, after suitable modifications.  

We get  the following corollary.

\begin{cor}\label{ranss}Let $X\subset\R^n$ be a C-analytic set and $\gta$ a saturated ideal of $\Oo(X)$ such that $\ceros(\gta)$ is a $H^{ra}$-{set}. Then $I(\ceros(\gta))=\widetilde{\sqrt[{ra}]{\gta}}$.
\end{cor}

\section{Countably many compact sets.}

In this section we develop some strategies  to provide a positive answer to H17 in some  general cases. We consider representations of positive semidefinite analytic  functions as $g^2f = \sum_{i=1}^q f_i^2$, where $g, f_i \in \an(\R^n)$ and $q\in \N \cup \{\infty\}$. We will achieve the following two facts.

\begin{itemize} 
\item If $g^2f = \sum_{i=1}^q f_i^2$ we can get $\ceros(g)\subset \ceros(f)$ increasing the number of squares until  $2^{n+1}q$.
\item If $f$ has a representation as sum of $q$ squares in a neighbourhood of its zeroset (by global analytic functions), it is a sum of $q+1$ squares of analytic functions  on $\R^n$.
 \end{itemize}
 
These two facts have several consequences. For instance if all the connected components of the zeroset of $f$ are compact, then $f$ is a (possibly infinite) sum of  squares of meromorphic functions.

. 
\subsection{Controlling denominators.}

In this section we show how to control the zeroset of the denominator in a 
representation of an analytic function as sum of squares of meromorphic functions.

\begin{prop}\label{bad} Let $f:\Omega \to \R$ be an analytic function on the open set $\Omega \subset \R^n$ and $h$ be a nonzero analytic function such that $h^2f$ is a sum of $p\leq\infty$ squares of analytic functions on $\Omega$. Set $d=\mbox{\rm dim}\{h=0, f\neq 0\}$. Then there exists an analytic function $g: \Omega \to \R$ such that $g^2f$ is a sum of $q\leq 2^{d+1}p$ squares of analytic functions and $\{g=0\}\subset \{f=0\}$. Moreover on a smaller neighbourhood of $ \{f=0\}$ we can get $ q\leq 2^{d}p$.
\end{prop}
 \begin{proof} Consider the C-analytic set $Y= \{h=0\}$ and let $\{Y_i\}_i$ be its (global) irreducible components. For each irreducible component $Y_i$  not contained in $\{f=0\}$ pick a point $y_i\in Y_i\setminus \{f=0\}$. We can assume that these points form a discrete countable subset of $\Omega$, because the family $\{Y_i\}_i$ is locally finite.

For each $y_i$ take a small open disc $B_i$ in such a way that  $B_i \cap B_j = \varnothing$ for $i\neq j$ and  $(\bigcup_i B_i) \cap U  = \varnothing$ where $U$ is an open neighbourhood of $\{f=0\}$. For each $i$ choose a point $y_i' \in B_i \setminus  Y$ and let $\phi_i: \Omega \to \Omega$ be a difffeomorphism which is the identity outside $B_i$ and moves $y_i$ to $y'_i$. In this way we construct a diffeomorphism  $\varphi:\Omega \to \Omega$ which is the identity outside $\cup B_i$ and moves $y_i$ to $y_i'$ for each $i$. In particular $\varphi$ is the identity  on $U$. Thus, $\varphi - {\rm Id}$ is divisible by $f^2$, because it is  identically $0$ on a neighbourhood of  $\{f=0\}$. So we can write $\varphi  = {\rm Id} +f^2\mu$, for a suitable smooth map $\mu$.     
By Whitney's approximation theorem we can find an analytic map $\eta:\Omega \to \Omega$ close to $\mu$. Note that the analytic map $G = {\rm Id} +f^2\eta$ is 
close to $\varphi$ and hence it is an analytic diffeomorphism, becaise to be a diffeomorphism is an open condition in the topological vector space ${\mathcal E}(\Omega,\R^n)$ of smooth maps from $\Omega$ to $\R^n$. 

Consider now the analytic function $f\circ G :\Omega\to \R$ and observe that it has the same zeroset as $f$, because $G$ is the identity on the zeroset of $f$. Also, for each $x\in \{f=0\}$ we get $(f\circ G)_x =f_xu_x$ where $u_x$ is a unit in $\an_x$. Indeed, assume $x=0$ and look at the Taylor expansion of $f$ on a neighbourhood of $0$. We get
$$f(x+ty) = f(x) +t g(x,y,t), g\in \R\{x,y,t\}.$$ 

Replace $t$ by $f^2$ and $y$ by $\eta$and  we get
$$ f\circ G = f(x+f^2\eta) = f(x)g(x,\eta,f^2) =f(x)(1+fg(x,\eta,f^2)).$$ 

Note that, since $f(0)=0$, $1+fg(x,\eta,f^2)$ is a unit on a neighbourhood of $0$. 
Thus, $f\circ G -f$ is divisible by $f^2$ on a neighbourhood of $\{f=0\}$ and also outside where $f$ is a unit. Consequently we find an analytic function $\alpha$ such that $f\circ G = f +f^2\alpha = f(1+f \alpha)$. Now $ v=1+f \alpha$ is a unit on $\Omega$. Indeed, either  $f(x)= 0$ hence $v(x) =1$, or $f(x)\neq 0$ hence $f\circ G(x) \neq 0$ and so $v(x) = 1 +f(x)\alpha(x) \neq 0$.

Since $f, f\circ G$ are positive semidefinite we have  $v>0$, so we can write $v=u^2$ and $f\circ G = u^2 f$ on $\Omega$.
 By hypothesis $h^2 f =\sum_{j=1}^p h_j^2$, which provides
$$ u^2 (h\circ G)^2 f = (h\circ G)^2 (f\circ G) = \sum_{j=1}^p(h_j\circ G)^2.$$ 
Hence,
$$(h^2 +u^2(h\circ G)^2)f = \sum_j(h_j^2 +(h_j\circ G)^2).$$

Thus, multiplying both sides by $h^2 +u^2(h\circ G)^2$, we get 

$$b^2 f = \sum_{j=1}^p (h_j^2 +(h_j\circ G)^2)(h^2 +u^2(h\circ G)^2)$$
where $b= h^2 +u^2(h\circ G)^2$. When the sum is infinite, we get another infinite sum. When it is finite, using that the product of two sums of two squares is again the sum of two squares, we get the sum of $2p$ squares. 

 The zeroset of $b$ is $\{h=0\}\cap \{h\circ G =0\}=  \{h=0\} \cap G^{-1}( \{h=0\})$. Thus 
$$\displaystyle \{b=0\}\setminus\{f=0\}\subset \bigcup_{Y_i\not\subset \{f=0\}} Y_i\cap G^{-1}(Y).$$ 
But no irreducible component $Y_i$ is contained in $G^{-1}(Y)$, because $G$ moves $y_i\in Y_i$ off $Y$.  Hence dim $(Y_i\cap G^{-1}(Y))<$ dim $Y_i\leq d$. Thus, dim$(\{b=0\}\setminus \{f=0\})  <d$. So, repeating this procedure at most $d+1$ times, we get the first assertion. To get the second one, note that after $d$ steps the zero set of the denominator outside $\{f=0\}$ has at most dimension $0$, hence it is a discrete set $D$. So on $\Omega \setminus D$ we get $2^dp$ squares as needed. 
\end{proof}

\subsection{Globalisation of sums of squares.}
In this section we show that if an analytic  function $f$ is a sum of squares of meromorphic functions   as a germ at its zeroset, then it is globally a sum of squares  and the number of involved squares increases at most by $1$. Also the denominator remains controlled, that is its zeroset is contained in the zeroset of $f$.

We begin with two technical lemmata.

\begin{lem}\label{divide}
Let $\Omega$ be a an invariant open Stein neighbourhood of $\R^n$ in $\C^n$,  $\Phi: \Omega \to \C$ be an invariant holomorphic function and $K\subset \Omega$ be an invariant compact set. Let $\Omega_0 \subset \Omega$ be an invariant  open neighbourhood  of $\{\Phi =0\}$ in such a way that $\{\Phi =0\}\cap \Omega_0$ is closed in $\Omega$.  Then, there exist a real constant $\mu >0$ and an invariant compact set $L\subset \Omega_0$ such that for each invariant holomorphic function $C:\Omega_0 \to \C$ there exists an in\-va\-riant holomorphic function $A:\Omega \to \C$ such that $\Phi_{|{\Omega_0}}$ divides $A_{|{\Omega_0}} -C$ and {\rm sup}$_K|A|\leq \mu$ {\rm sup}$_L|C|$.   
\end{lem}

\begin{proof} Consider the sheaf of ideals $\Ii\subset \an$ generated by $\Phi$ and the exact sequence of coherent sheaves:
$$0\to \Ii \to \an_{\Omega} \to \frac{\an_\Omega}{\Ii} \to 0$$
Since $\Omega$ is Stein, we have the following commutative diagram:

$$\xymatrix{
H^0(\Omega,\Ii) \ar@{->}[r]\ar@{->}[d]& H^0(\Omega, \an) \ar@{->}[r]\ar@{->}[d]& H^0(\Omega, \frac{\an_\Omega}{\Ii})\ar@{->}[r]\ar@{->}[d]& 0\\
H^0(\Omega_0,\Ii)\ar@{->}[r]&  H^0(\Omega_0, \an) \ar@{->}[r] &H^0(\Omega_0, \frac{\an_\Omega}{\Ii})\ar@{->}[r]&0
}
$$ 

where the vertical maps are the restriction maps from $\Omega$ to $\Omega_0$.

Since $\{\Phi = 0\}\cap \Omega_0$ is closed in $\Omega_0$, each section of $\displaystyle \frac{\an_\Omega}{\Ii}$ on $\Omega_0$ can be extended by $0+\Ii$ outside $\Omega_0$. Hence, the last restriction is surjective and  we get a linear surjective homomorphism $\displaystyle \varphi: H^0(\Omega, \an) \to H^0(\Omega_0, \frac{\an_\Omega}{\Ii}) = \frac{H^0(\Omega_0,\an)}{H^0(\Omega_0, \Ii)}$. We equip these vector spaces with their natural Fr\'echet topology, namely the compact open topology on the first one and the quotient topology on the second one. Here we use that $H^0(\Omega_0, \Ii)$ is a closed subspace of $H^0(\Omega_0,\an)$ so the quotient is also Fr\'echet. Finally $\varphi$ is a continuous surjective homomorphism between Fr\'echet spaces, so it is also open.

Consider an exhaustion of $\Omega$ by compact sets $\{K_i\}$, where $K_i$ is contained in the interior part of $K_{i+1}$ (risp. an exhaustion of $\Omega_0$ by compact sets $\{L_j\}$, where  $L_j$ is contained in the interior part of $L_{j+1}$).
Next fix a compact set $K\subset \Omega$ and consider the open set $W= \{H\in H^0(\Omega,\an):$ sup$_K|H| <1\}$. Since $\varphi$ is open, $\varphi(W)$ is an open neighbourhood of $0$ in $\displaystyle \frac{H^0(\Omega_0,\an)}{H^0(\Omega_0,\Ii)}$, hence there is $\varepsilon >0$ such that $V= \{\xi : ||\xi||<\varepsilon \} \subset \varphi(W)$.
Choose $\displaystyle \mu= \frac{2}{\varepsilon}$ and $L=L_i$ with $i$ such that $\displaystyle \sum_{j\geq i}\frac{1}{2^j} < \frac{\varepsilon}{2}$. Note that $\mu, L\subset \Omega_0$ only depend on $\Phi, \Omega_0, K$. Let us check that they satisfy the desired condition.    
Let $C\in H^0(\Omega_0, \an)\setminus\{0\}$. Since the interior part of $L$ is not empty, $a= \sup_L|C| >0$. Denote $G= \frac{1}{a\mu}C\in H^0(\Omega_0,\an)$. Thus,   $\sup_L|G|= \frac{1}{\mu}<\frac{\varepsilon}{2}$ and we have

$$\|G\| = \sum_{j\geq 1}\frac{1}{2^j}\frac{\sup_{L_j}|G|}{1+\sup_{L_j}|G|} = \sum_{j=1}^i\frac{1}{2^j}\frac{\sup_{L_j}|G|}{1+\sup_{L_j}|G|} + \sum_{j>i}\frac{1}{2^j}\frac{\sup_{L_j}|G|}{1+\sup_{L_j}|G|} $$
$$< \sup_L|G|\cdot \sum_{j=1}^i \frac{1}{2^j} +\sum_{j>i}\frac{1}{2^j} <\frac{\varepsilon}{2}\cdot 1 +\frac{\varepsilon}{2} = \varepsilon$$ 
 
Thus, setting $\xi=G+H^0(\Omega_0,\Ii)$, we get
$\|\xi\| \leq \|G\| <\varepsilon$  that is, $\xi\in V\subset \varphi(W)$.
Hence there is $H\in W$ with $\varphi(H) = \xi$, consequently $H_{|\Omega_0} - G\in H^0(\Omega_0,\Ii)$. So the holomorphic function $F=a\mu H$ satisfies:
\begin{itemize}
 \item $F_{|\Omega_0} - C= a\mu(H_{|\Omega_0} -G)\in H^0(\Omega_0,\Ii)$, hence $\Phi$ divides $F_{|\Omega_0} - C$ in $H^0(\Omega_0, \an)$. 
\item $\sup_K|F| =a\mu \sup_K(H)<a\mu = \mu \sup_L|C|$, because $H\in W$.  
\end{itemize}

Finally, if $C$ is $\sigma$-invariant, we take $\displaystyle
A=\frac{1}{2}(F(z)+ \overline{F(\overline z)})=  \Re F$.
Indeed, $A_{|\Omega_0} - C =\Re (F_{|\Omega_0} -C)= \Re(B \Phi_{|\Omega_0}) =\Re B \Phi_{|\Omega_0}$ for some $B\in H^0(\Omega_0,\an)$, that is, $\Phi$ divides $A-C$. Moreover $\displaystyle \sup_K |A| = \sup_K \Re F = \sup_K \frac{1}{2}|F+ \overline {F\circ \sigma}| \leq \sup_K \frac{1}{2}(|F| +|\overline {F\circ \sigma}|) = \sup_K |F|\leq \mu\sup_L C$, as required.   
\end{proof}

The second lemma concerns infinite sums of squares of holomorphic functions.

\begin{lem}\label{serie}  Let $\Omega,\Phi,\Omega_0$ be as in Lemma \ref{divide}. Let $C_k:\Omega_0\to\C$, for $k=1,2,\dots$ be invariant holomorphic functions such that $\sum_k \sup_L |C_k|^2 <\infty$ for each compact set $L\subset\Omega_0 $. Then, there exist invariant holomorphic functions $A_k:\Omega\to \C$ such that $\sum_k \sup_K |A_k|^2 <\infty$ for each compact set $K\subset\Omega $  and $\Phi_{|{\Omega_0}}$ divides $A_{k|{\Omega_0}} -C_k$. 
\end{lem}
\begin{proof} Let $\{K_i\}_i$ be an exhaustion of $\Omega$ by invariant compact sets  such that $K_i$ is contained in the interior part of $K_{i+1}$ and the union of the interior parts of these compact sets is $\Omega$. By Lemma \ref{divide} for each $i$ there exist $\mu_i >0$ and a compact set $L_i\subset \Omega_0$ such that if $C\in \an(\Omega_0)$ is invariant, there exists an invariant holomorphic function $A\in \an(\Omega)$ such that $\Phi_{|\Omega_0}$ divides $A_{|\Omega_0} -C$ and $\sup_{K_i} |A|\leq \mu_i\sup_{L_i} |C|$. We can assume $L_i \subset L_{i+1}$ for all $i$. 

Since $\displaystyle \sum_{k\geq 1} \sup_{L_i}|C_k|^2 <\infty$ for all $i$, there exists a strictly increasing sequence $\{k_i\}_i$ of positive integers such that $\displaystyle\sum_{k\geq k_i}\sup_{L_i}|C_k|^2 \leq \frac{1}{2^i\mu_i^2}$. Now for each $k$ with $k_i\leq k<k_{i+1}$ we choose an invariant holomorphic function $A_k :\Omega \to \C$ such that for each $i \, \sup_{K_i}|A_k|\leq \mu_i \sup_{L_i}|C_k|^2$ and $\Phi_{| \Omega_0}$ divides $A_{k|\Omega_0} -C_k$.
Let us check that for each compact sets $K\subset \Omega$ the series $\displaystyle \sum_{k\geq 1} \sup_K |A_k|^2$ converges. Since $\bigcup_i \Int_{\C^n}(K_i) =\Omega$, it is enough to check that for each $i$ we have $\displaystyle \sum_{k\geq 1}\sup_{K_i}|A_k|^2 <\infty$. Indeed,
$$\sum_{k\geq1} { \sup}_{K_i}|A_k|^2 = \sum_{1\leq k<k_i}{ \sup}_{K_i}|A_k|^2 + \sum_{j\geq i} \sum _{k_j\leq k<k_{j+1}}{ \sup}_{K_i}|A_k|^2$$
$$ \leq T + \sum _{j\geq i} \sum_{k\geq k_i}{ \sup}_{K_i}|A_k|^2\leq T+  \sum _{j\geq i} \sum_{k\geq k_i}\mu_i^2{ \sup}_{L_i}|C_k|^2$$
$$\leq T+ \sum_{j\geq i} \frac{\mu_i^2}{2^j\mu_i^2}\leq T+1 <\infty$$  
as required.
\end{proof}

Next theorem is the announced globalization result.

\begin{thm}\label{globalization}
Let $f:\R^n\to \R$ be a positive semidefinite analytic function and let $Z$ be its zeroset. Suppose that $f$ is a sum of $p\leq \infty$ squares of meromorphic functions with controlled denominator on a neighbourhood $W$ of $Z$. Then, $f$ is a sum of $p+1$ squares of meromorphic functions with controlled denominator, that is $g^2 f= \sum_{k=1}^{p+1} f_k^2$ and $\{g=0\} \subset Z$.
\end{thm}
\begin{proof} The proof is conducted in several steps.

{ \sc Step 1: global denominator.}
We start with a representation of $f$ as a sum of squares $g^2 f= \sum_{k=1}^q f_k^2$ on an open neighbourhood $W$ of $Z$
  and $\{g=0\} \subset Z$.
The set $\{g=0\}$ is a closed subset of $\R^n$. Then, we can consider the locally principal coherent sheaf of ideals defined by
$\Ii_x= g_x \an_{\R^n,x}$ if $x\in \{g=0\}$ and by $\Ii_x= \an_{\R^n,x}$ otherwise.

This sheaf of ideals is  globally principal, that is $\Ii = g_0 \an_{\R^n}$ for some global analytic function $g_0$ on $\R^n$. The zeroset of $g_0$ coincides with  the one of $g$. Moreover, $\displaystyle v=\frac{g_0}{g}$ is a unit on $W$. Hence, on $W$ we can write:
$$g_0^2 f = (vg)^2f= \sum_{k=1}^q (vf_k)^2 = \sum_{k=1}^q b_k^2,\,  b_k\in \an(W).$$

Thus we can assume that the denominator $g$ is a global analytic function on $\R^n$.

{\sc  Step 2: global sum of squares.} 
Let $U$ be a Stein open neighbourhood of $\R^n$ in $\C^n$ to which the function $f'=g^2f$ extends to a holomorphic function $F'$.  There are $q$ invariant holomorphic function $C_k:\Omega_0' \to \C$ defined on an invariant open neighbourhood $\Omega_o'$ of $W$ in $\C^n$ such that $\displaystyle F'_{|\Omega_0'} = \sum_{k=1}^q C_k^2$, where the series converges in the compact open topology of $\an(\Omega_0')$. 
 
Now $\Omega_0' \cup \{F'\neq 0\}$ is an open neighbourhood of $\R^n$ in $\C^n$, hence $\R^n$ has an invariant Stein open neighbourhood $\Omega \subset \Omega_0' \cup \{F'\neq 0\}$. Denote $\Omega_0 = \Omega_0' \cap \Omega$. Note that $\Omega_0 \cap \{F'=0\}$ is closed in $\Omega_0$. We apply Lemma \ref{serie} to $\Phi = F'^2$ and find  invariant holomorphic functions $A_k: \Omega \to \C$ such that $\displaystyle \sum_{k=1}^q { \sup}_K|A_k|^2 <\infty$ for each  compact subset $K$ of  $\Omega$ and $F'^2$ divides $A_{k|\Omega_0} -C_k$ on $\Omega_0$. We have:
$$ \displaystyle \sum_{k=1}^q A_k^2 -F' =  \sum_{k=1}^q A_k^2 - \sum_{k=1}^q C_k^2 = \sum_{k=1}^q(A_k -C_k)(A_k +C_k)$$ 
on $\Omega_0$ and the series converges uniformely on the compact subsets of $\Omega_0$. Since $F'^2$ divides each term $A_k -C_k$  on $\Omega_0$, it divides the sum $\displaystyle \sum_{k=1}^q A_k^2 -F'$. Thus, there exists a holomorphic function $H:\Omega_0 \to \C$ such that
$$\displaystyle \sum_{k=1}^q A_k^2 = F' +H F'^2 = F'(1+HF')$$
on $\Omega_0$.
Clearly $1+HF'$ does not vanish on $\{F'=0\}$. Thus, it is a unit in a perhaps smaller neighbourhood that we denote again  $\Omega_0$.
Write $a_k = A_{k|\R^n}$ and note that in a perhaps smaller neighbourhood $W'$ of $Z= \{f=0\}$ in $\R^n$ we have 
$$f' = g^2f = \left(\sum_{k=1}^q a_k^2\right) w,$$
where $w\in \an(W')$ is a positive unit, but all the other functions appearing in the previous formula belong to $\an(\R^n)$.

{\sc Step 3: conclusion.}
Consider the global analytic function $ \displaystyle \sum_{k=1}^q a_k^2 +(g^2f)^2 $ and note that its zeroset is $Z$ and on $W'$ the quotient
$$\displaystyle \frac{\sum_{k=1}^q a_k^2 +(g^2f)^2}{g^2f}= \frac{\frac{g^2f}{w}+(g^2f)^2} {g^2f} = \frac{1}{w} +g^2f,$$
  which is a positive unit on $W'$.

Hence there exist a positive unit $u\in \an(\R^n)$ such that 
$$(ug)^2f = u^2g^2f = \sum_{k=1}^q a_k^2 +(g^2f)^2$$
as required.
\end{proof}

\subsection{Consequences.}
 
As we said before, Theorem \ref{globalization} has several consequences.

\begin{cor}\label{compatti} Let $f:\R^n \to \R$ be a positive semidefinite analytic function whose zeroset is a countable union of disjoint compact sets. Then, $f$ is a (possibly infinite) sum of squares of meromorphic functions.
\end{cor}
\begin{proof} The germ of $f$ at each compact connected component of its zeroset is a finite sum of squares (with denominator), but the number of such squares may be  not uniformely bounded on the connected components of the zeroset of $f$.
\end{proof}

\begin{remark} Corollary \ref{compatti} includes the case of a discrete zeroset. But in this case, as we will see in the following sections, $f$ is always a finite sum of squares (with denominator), because the number of squares needed at each point to represent a positive semidefinite function is uniformely bounded.
\end{remark}

We recall the definition of {\em Nash function}.

\begin{defn}\hfill 
\begin{enumerate}
\item An analytic function $g\in\an(\R^n)$ is called a \sl Nash function \rm if there is a monic  polynomial $P= y^d+a_{d-1}(x)y^{d-1}+\ldots +a_0(x)$ with $a_i\in \R[x_1,\ldots ,x_n]$, such that $g^d+ a_{d-1}(x)g^{d-1}+\ldots +a_0(x)$ is identically $0$. 
\item An analytic function $f:\R^n \to \R$ is of {\em Nash type} if the ideal $f\an(\R^n)$ can be generated by a Nash function $g$, that is, there is an analytic unit $u$ such that $f=ug$.
\end{enumerate}
\end{defn}

Denote by  ${\mathcal N}(\R^n)$ the ring of Nash functions on $\R^n$.

\begin{cor}Let $f:\R^n \to \R$ be a positive semidefinite analytic function. Consider for each connected component $X_k$ of the zeroset of $f$ the ideal sheaf $\Ii_k$ defined by $(\Ii_k)_x = f\an_x$ if $x\in X_k$ and $(\Ii_k)_x = \an_x$ otherwise. Assume for all $k$\,  $\Ii_k(\R^n)$ to be generated by a Nash function $g_k$. Then,
\begin{itemize}
\item  Each $g_k$ is a sum of $2^n$ squares with controlled denominator in the quotient field of  the ring ${\mathcal N}(\R^n)$ of Nash functions on $\R^n$.
\item $f$ is a sum of $2^n +1$ squares of meromorphic functions with controlled denominator.
\end{itemize}  
\end{cor}
\begin{proof}
Since H17 has a positive answer for Nash functions, 
 $g_k$ is a sum of $2^n$ squares with controlled denominator in the quotient field of  ${\mathcal N}(\R^n)$. 
By Theorem \ref{globalization} $f$ is a sum of $2^n +1$ squares of meromorphic functions with controlled denominator, as required.  
\end{proof}

\subsection{Consequences on Pythagoras numbers.}\label{p(M)}

We defined the Pytagoras number of a ring in Section \ref{pitagora}.

In the algebraic setting both problems, H17 and to estimate the Pythagoras number of a ring of polynomials, are quite independent  problems with independent answers. In the analytic setting there is a strong relation between them as the following result shows. We will denote by $p(\R^n)$ the Pythagoras number of the field $\mathcal M(\R^n)$ of meromorphic functions on $\R^n$.

\begin{thm}\label{pytagora}  
If H17 has a positive solution for $\an(\R^n)$, then the Pythagoras number $p(\R^n) <\infty$
\end{thm}

\begin{proof}
Assume  $p(\R^n)= \infty$. We will construct a positive semidefinite analytic function that is not a finite sum of squares of meromorphic functions.
By definition for all $p>0$ there exists an analytic function $f_p:\R^n \to \R$ which is a finite sum of squares of meromorphic functions but cannot be represented with less than $p$ squares.

{\sc First reduction.} 
We can assume dim $\{f_p =0\} \leq n-2$ 

Denote $g=f_p$. For each $x\in \{g =0\}$ we may write $g_x = a_x^2b_x$ where the germ $b_x$ has no multiple factors. Its germ of zeroset has codimension at least $2$, because otherwise some irreducible factor $c_x$ of $b_x$ would change sign at $x$, so $g_x$ would change sign too, against the fact that $g$ is positive semidefinite, because it is  a sum of squares (with denominator). Now the germs $\displaystyle\{a_x\}_{ x\in \{g=0\}}$ generate the locally principal coherent analytic sheaf of ideals  $\Ff$ given by
$\Ff_x = a_x\an_x$ for $x\in \{g=0\}$ and $\Ff_x = \an_x$ otherwise.

So the sheaf $\Ff$ is globally principal, generated by an analytic function $a:\R^n\to \R$. Hence $g=a^2b'$ where $b'$ is positive semidefinite and for $ x\in \{g=0\}$ verifies $b'_x = b_xu_x$ with $u_x$ a positive unit in $\an_x$. Thus, dim$\{b'=0\} \leq n-2$.

We take $b'$ instead of $g$ and we are done.       

{\sc Second reduction.} 
We may assume that the zeroset $Z_p = \{f_p=0\}$ is contained in the tube $V_p = \{x = (x',x_n) \in \R^n: \|x'-a_p\| <\frac{1}{4}\}$, where $a_p = (p,0,\ldots,0) \in \R^{n-1}, p\geq 1$

Since $Z_p$ has codimension at least $2$, by Sard theorem there is a line that does not intersect $Z_p$. After an affine change of coordinates we can assume that this line is the $x_n$-axis $\{x_1= \cdots =x_{n-1}=0\}$. Note that the function dist$(-, Z_p)$, that is, the  distance function to $Z_p$,  is a continuous positive function on the  $x_n$-axis. Thus, by Whitney's approximation theorem we can find an analytic function $\delta:\{x_1= \cdots =x_{n-1}=0\}\to \R$ close to dist$(-, Z_p)$, hence also strictly positive.

Consider the analytic diffeomorphism $\Psi:\R^n\to\R^n$ given by 
$$\Psi(x',x_n) = \left(\frac{\sqrt{1+x_n^2}}{\delta(x_n)}x', x_n\right).$$ 
Note that $\Psi$ moves $Z_p$ outside the set $\{\|x'\|<1+x_n^2\}$.
Next, consider the analytic diffeomorphism $\Phi:\R^n\to \R^n, \Phi(x',x_n) = (x', 2\|x'\|^2 -x_n)$ which maps the set $\{2\|x'\|^2 -x_n >0\}$ onto the set $\{2\|x'\|^2 -x_n <0\}$ and viceversa. Thus $\Phi\circ \Psi(Z_p) \subset \{2 \|x'\|^2 -x_n <0\}$. Finally consider $ \displaystyle \Theta:\R^n \to \R^n, \Theta(x',x_n) = \left(\frac{x'}{4(1+x_n^2)}, x_n\right)$. It maps the set $\{2\|x'\|^2 -x_n <0\}$ to the tube $V_p$ because for each point of the first one has 
$$\left\|\frac{x'}{4(1+x_n^2)}\right\| = \frac{1}{4}\cdot \frac{\|x'\|}{1+x_n^2} < \frac{1}{4}\cdot\frac{\|x'\|}{1+4\|x'\|^4} <\frac{1}{4}.$$

So the composition of these $3$ analytic diffeomorphisms moves $Z_p$ inside the tube $V_p$.

{\sc Conclusion.}
After our reductions the $(Z_p)'^s$ form a locally finite family, so  $Z= \bigcup_p Z_p$ is a (closed) analytic subset of $\R^n$. Also it is the zero set of the sheaf of ideals $\Ii$ defined as $\Ii_x = f_p\an_x$ if $x\in Z_p$ and $\Ii_x = \an_x$ if $x\notin Z$. Again $\Ii$ has a global generator $f$. Since its zeroset $Z$ has codimension at least $2$, the function $f$ has constant sign that we can assume positive.  Now on $V_p$ one has $f=f_pv_p^2$, hence $f$ is a sum of squares of meromorphic functions
in a neighbourhood $\bigcup_p V_p$ of its zeroset, but the number of squares is not bounded. Hence,  by Theorem \ref{globalization}, $f$ is an infinite sum of squares of meromorphic functions but it cannot be a finite sum.  
\end{proof}

\section{The discrete case.}\label{Milnor}

Let $f$ be a positive semidefinite analytic function whose zeroset is a singleton  $x_0$. Then we know that $f$ is a finite sum of squares of meromorphic functions (a point is a compact set). Alternatvly considering  germs at $x_0$ we find a neighbourhood $U$ of $x_0$ where $f$ admits a representation as finite sum of squares of meromorphic functions and then use Theorem \ref{globalization}  to get a similar representation on $\R^n$. 

The same argument shows that when the zeroset of $f$ is a countable discrete set  $f$ is a (possibly infinite) sum of squares.

We can improve the previous result. We can prove that in fact the number of squares needed to represent $f_x$ at one of its zeroes is smaller than a bound that does not depend on $f$ or $x$. This is in a sense the converse of Theorem \ref{pitagora}. We have a uniform bound on the number of squares needed to represent $f$ at its zeroset and we prove H17 for $f$. 

What we will prove is that an analytic  germ $f_a$ vanishing only at $a$ admits, after a small modification,   {\em finite determinacy} that roughly speaking means that $f_a$ is locally equivalent to a polynomial. Since the ring of real polynomial has finite Pythagoras number, this will provide the uniform bound. One way to prove finite determinacy is to check that the holomorphic extension $F$ of $f$ has an isolated singularity at $a$, because the Milnor number $\mu(F)$ provides the degree of a polynomial equivalent to $f$.

Let us begin by some definitions. We consider the ring of germs $\R\{x_1,\ldots, x_n\}$ canonically embedded in $\C\{x_1,\ldots, x_n\}$. Let $\gtm$ denote the maximal ideal of both rings.

\begin{defn}\hfill

\begin{itemize}
\item Two elements $f, g \in \gtm$ are {\em right equivalent} it there is an automorphism $\varphi$ of $\C\{x_1,\ldots, x_n\}$ (resp. of  $\R\{x_1,\ldots, x_n\}$) such that $\varphi(f) =g$
\item $f\in \gtm$ is {\em  $k$-determined} if for each $ g\in\C\{x_1,\ldots, x_n\}$ such that $f-g \in \gtm^{k+1}, g$ is right equivalent to $f$. 
\item $f$ is {\em finitely determined} if it is $k$-determined for some $k$.
\end{itemize}
\end{defn}

Thus,  finitely determined analytic functions are right equivalent to polynomials. 

The main result we will use  can be found in the book of de Jong and Pfister \cite{jp}. Denote $J(f)$ the ideal generated by the partial derivatives of $f$.

\begin{thm}\label{finitedeterminacy}  Let $f\in \gtm^2$. Assume $\gtm^{k+1} \subset \gtm J(f)^2$. Then $f$ is $k$-determined. In particular functions with an isolated singularity are finitely determined.  
\end{thm}

For our purposes, in order to obtain finite determinacy for $f$ it is enough to prove that, up to a finite sum of squares of analytic function germs, its complex holomorphic extension $F$ has an isolated singularity. Then, if its Milnor number is $\mu$, the function germ $f -\sum h_i^2$ will be  $(\mu+1)-$determined.

We need a preliminary lemma.   

\begin{lem}\label{pseudosard}
Let $X\subset\C^m$ be a complex analytic manifold and let $V\subset\C^n$ be an open set of $\C^n$
that intersect $\R^n$. Let $f:X\rightarrow V$ be an analytic function and let $Y\subset X$ be an
analytic subset of $X$ such that for all $y\in Y$ the map $d_yf:T_yX\rightarrow \C^n$ is not
surjective. Then the interior in $\R^n$ of $f(Y)\cap\R^n$ is empty.
\end{lem}
\begin{proof}
 Suppose, for a contradiction, that the interior in
$\R^n$ of $f(Y)\cap\R^n$ is not empty.

{\sc Claim}: \em There exists a countable collection of subsets $Y_i\subset Y$, where $i\in\N$,
such that:
\begin{itemize}
\item $Y_i$ is an analytic manifold of constant dimension $d_i$ or a point.
\item $Y=\bigcup_{i\in\N}Y_i$. 
\item If \mbox{\rm dim} $Y_i=d_i$ then the rank
$\rg(f_{|Y_i})$ is a constant $r_i<n$ in $Y_i$. Also there exist open sets
$W_i\subset V$, $U_i\subset \C^{d_i}$ and $W_i'\subset\C^n$ and two biholomorphisms
$\varphi_i:Y_i\rightarrow U_i$ and $\psi_i:W_i\rightarrow W_i'$ such that
$\psi_i\circ f_{|Y_i}\circ\varphi_i^{-1}$ is the map 
$$
(z_1,\ldots,z_{d_i})\mapsto(z_1,\ldots,z_{r_i},0,\ldots,0).
$$
\item $\ol{f({Y_i})}\subset W_i$.
\end{itemize}

\rm
 We begin by proving by induction on the dimension of $Y$ that  
$Y$ can be written as the union of a countable set $D\subset X$
and countably many analytic manifolds $Z_i$ of constant dimension $d_i$ such that $f_{|Z_i}$ has constant rank $r_i<n$.

Indeed, if $\dim Y=0$, then the analytic set $Y$  is a discrete set in $X$, hence countable. 
Suppose now
that the result is true for dimension less or equal to $d-1$ and suppose that the dimension of $Y$ is $d$.

Let $r=\max \{\rg_p(f_{|Y}):p\in Y\setminus\Sing(Y)\}$ and let $C=\{z\in Y:\ \rg_p(f_{|Y})<r\}$.
Then $C$ is an analytic subset of the analytic manifold $Y \setminus \Sing(Y)$ of dimension
$\leq d-1$ and  $\Sing(Y)$ is an analytic subset of the analytic manifold $X$. By induction
hypotheses, $C$ and $\Sing(Y)$ can be written as the union of 
countably many analytic manifolds of constant dimension such that the restriction of $f$ to each
manifold has constant rank $<n$. Thus, $Y$ can be written as the union of a countable set $D\subset X$
and countably many analytic manifolds $Z_j$ of constant dimension $d_j$ such that $f_{|Z_j}$ has
constant rank $r_j<n$.

Take $j\in\N$ and let $p\in Z_j$. By the Rank Theorem  applied to
$f_{|Z_j}:Z_j\rightarrow V$, there exists open sets
$U_p\subset \C^{d_j}$, $A_p\subset Z_j$, $W_p\subset V$, and $W_p'\subset\C^n$, and two biholomorphisms
$\varphi_p:Y\rightarrow U$ and $\psi_p:W\rightarrow W'$ such that $p\in A_p$ and 
$\psi_p\circ f|_{A_p}\circ\varphi_p^{-1}$ is the map: 
$$
(z_1,\ldots,z_{d_j})\mapsto(z_1,\ldots,z_{r_j},0,\ldots,0),
$$
Note that, shrinking $U_p$ and $A_p$ if necessary, we can suppose that $\ol{f(A_p)}\subset W_p$.
Note also that each $Z_j$ can be covered by countably many open sets in $Z_j$ which are contained
in the open set $A_p$ for some $p\in Z_j$ (to prove this use the fact that $Z_j$ is second countable). Combining this with the fact that $Y=D\bigcup\left(\bigcup_{j\in\N}Z_j\right)$, the claim is proved.

Next, we prove that there exists $i_0\in\N$ such that the interior of
$\ol{f(Y_{i_0})}\cap\R^n$ in $\R^n$ is not empty.

Suppose that the interior $\Int(D_i)$ of $D_i=\ol{f(Y_{i})}\cap\R^n$ in
$\R^n$ is empty for all $i\in\N$. Thus, the set $\R^n\setminus D_i$ is open and dense in $\R^n$
for all $i\in\N$. Since $\R^n$ is a Baire space, the intersection
$\bigcap_{i\in\N}(\R^n\setminus D_i)=\R^n\setminus\bigcup_{i\in\N}D_i$ is dense in $\R^n$. Since
$$
f(Y)\cap\R^n=\bigcup_{i\in\N}(f(Y_i)\cap\R^n)\subset\bigcup_{i\in\N}D_i,
$$
we conclude that $\R^n\setminus (f(Y)\cap\R^n)$ is dense in $\R^n$, which is a contradiction (because we
are supposing that the interior in $\R^n$ of $f(Y)\cap\R^n$ is not empty). Hence, there exists
$i_0\in\N$ with the desired property.

Finally, let $\pi_n:\C^n\rightarrow\C; z=(z_1,\ldots,z_n)\mapsto z_n$ be the projection over
the last coordinate. Note that $W_{i_0}$ contains $\ol{f(Y_{i_0})}$ and that the map
$\pi_n\circ\psi_{i_0}:W_{i_0}\rightarrow\C$ vanishes on $f(Y_{i_0})$ (see our initial claim).
Thus, it vanishes on $\ol{f(Y_{i_0})}$. Since $\ol{f(Y_{i_0})}\cap\R^n$ has not empty interior
in $\R^n$, we conclude that the real part and the imaginary part of $\pi_n\circ\psi_{i_0}$ vanish
in an open set of $\R^n$. By the Identity Principle, they are identically $0$. Thus,
$\pi_n\circ\psi_{i_0}$ is identically $0$, a contradiction because it is a coordinate of an
analytic diffeomorphism between open sets of $\C^n$.

Hence, the interior in $\R^n$ of $f(Y)\cap\R^n$ is empty, as required.
\end{proof}

\begin{thm}\label{isolated}
Let $f,h_1,\ldots,h_n\in\R\{x\}=\R\{x_1,\ldots,x_n\}$ be analytic germs such that
$h_1,\ldots,h_n$ generate an ideal of $\R\{x\}$ that contains a power of the maximal ideal
$\gtm_n$ of $\R\{x\}$. Then for every open set $B\subset\R^n$ there exists $d=(d_1,\ldots,d_n)\in
B$ such that the germ
$f-\sum_{i=1}^nd_ih_i^2$ has finite Milnor number.
\end{thm}
\begin{proof}
Indeed, let $U\subset\C^n$ be an open set such that: 
\begin{itemize}
\item $0\in U$,
\item $f,h_1,\ldots,h_n$ have holomorphic extensions $F,H_1,\ldots,H_n:U\rightarrow\C$, 
\item $\bigcap_{j=1}^nH_j^{-1}(0)=\{0\}$.\footnote{This last condition is possible because
$H_1,\ldots,H_n$ generate an ideal of $\C\{x\}$ that contains a power of the maximal ideal
$\gtm_n'=\gtm_n\C\{x\}$ of $\C\{x\}$.}
\end{itemize}
It is enough to prove that there exists $d\in B$ such that the analytic function
$F-\sum_{j=1}^nd_jH_j^2$ has an isolated singularity at the origin. To check this we proceed as
follows. Let $W$ be an open set of $\C^n$ such that $W\cap\R^n=B$ and let
$$
\Phi:U\times W\rightarrow\C,\ (z,d)\mapsto F(z)-\sum_{j=1}^nd_jH_j^2(z).
$$
Note that 
$$
D_{(z,d)}\Phi=\left(\ldots,\frac{\partial F}{\partial z_j}-2\sum_{k=1}^nd_kH_k\frac{\partial
H_k}{\partial z_j},\ldots,H_1^2,\ldots,H_n^2\right)
$$
vanishes in a point $(z,d)\in U\times W$ if and only if $z=0$ and $D_0F=\left(\frac{\partial
F}{\partial z_1},\ldots,\frac{\partial F}{\partial z_n}\right)=0$. If $D_0F\neq 0$, then for any
$d\in B$, the function $\Phi_d=F-\sum_{j=1}^nd_jH_j^2$ has an isolated singularity at $0$ and we
are done. Hence, we can assume that $D_0F=0$.

Consider the analytic manifold $X=\Phi^{-1}(0)\setminus(\{0\}\times W)$ which has (complex)
dimension $2n-1$. Consider the projection $\pi:X\rightarrow W,\ (z,d)\mapsto d$. Notice that
$\pi(X)=W$. Let 
$$
Y=\{y\in X: d_y\pi:T_y X\rightarrow\C^n \text{ is not surjective}\}.
$$ 
One can check that $Y$ is analytic subset of $X$. By Lemma \ref{pseudosard}, there exists $d^0\in
B\setminus\pi(Y)$. Let us see that $0$ is an isolated singularity of $\Phi_{d^0}$. Let
$(w,d^0)\in X\cap\pi^{-1}(d^0)$. Since $w\neq 0$, there exists
$1\leq j
\leq n$ such that
$H_j(w)\neq 0$. We can assume, to simplify notations, that $j=n$. 

By the Implicit Function Theorem applied to the analytic function $\Phi(z,d',T)$ (where
$d'=(d_1,\ldots,d_{n-1})$) and the point $(w,d^0)$, there exist open sets $V_1\subset X$,
$V_2\subset\C^{n-1}$ and $V_3\subset\C$ such that $w\in V_1$ and $(w,d^0)\in V_2\times V_3$ and
a unique analytic function $g:V_1\times V_2\rightarrow V_3$ such that $\Phi(z,d',g(z,d'))\equiv
0$. In particular, $g(w,d_1^0,\ldots,d_{n-1}^0)=d_n^0$. Note that the function
$$
\varGamma:V_1\times V_2\rightarrow X, (z,d')\mapsto(z,d',g(z,d'))
$$
is an analytic diffeomorphism over its image. Since $\pi$ has rank $n$ at $(w,d^0)$, we have that
$\pi\circ\varGamma$ has also rank $n$ at $(w,d_1^0,\ldots,d_{n-1}^0)$. Note that
$\pi\circ\varGamma(z,d')=(d',g(z,d'))$. Hence, the jacobian matrix of $\pi\circ\varGamma$ is of
the form
$$
\left(\begin{array}{c|c}
0&I_{m-1}\\
\\
\hline
\\
\cdots \frac{\partial g}{\partial z_j}\cdots&\cdots \frac{\partial g}{\partial d_k}\cdots
\end{array}\right).
$$
Since it has rank $n$ at $(w,d_1^0,\ldots,d_{n-1}^0)$, there exists $1\leq j\leq n$ such that
$$
\frac{\partial g}{\partial z_j}(w,d_1^0,\ldots,d_{n-1}^0)\neq 0.
$$
On the other hand, since $\Phi(z,d',g(z,d'))\equiv 0$, we have, taking partial derivatives with
respect to $z_j$,
$$
0=\frac{\partial F}{\partial z_j}-2\sum_{k=1}^{n-1}d_kH_k\frac{\partial
H_k}{\partial z_j}-2gH_n\frac{\partial
H_n}{\partial z_j}+H_n^2\frac{\partial g}{\partial z_j}.
$$
In particular, for $(w,d^0)$ we get:
$$
\frac{\partial F}{\partial z_j}(w)-2\sum_{k=1}^{n}d_k^0H_k(w)\frac{\partial
H_k}{\partial z_j}(w)=-H_n^2(w)\frac{\partial g}{\partial z_j}(w,d_1^0,\ldots,d_{n-1}^0)\neq 0.
$$
Thus, $w$ is a regular point of $\Phi_{d^0}^{-1}(0)$ and we conclude that $0$ is an isolated
singularity of $\Phi_{d^0}=F-\sum_{j=1}^nd_j^0H_j^2$, as required.
\end{proof}

\begin{remark} If $\ceros (f) =\{x_0\}$ then $f_{x_0}$ is a sum of $n+2^n$ squares. Indeed,  we need to modify $f$ using a sum of $n$ squares to make it right equivalent to a polynomial. Then this polynomial will be a sum of at most $2^n$ squares of rational functions. 
If the zeroset of $f$ is discrete, we apply Theorem \ref{globalization} to represent $f$  as a sum $n+2^n +1$ squares of meromorphic functions on $\R^n$.
\end{remark}

\section{Pfister trick.}

It is well known that the ring of analytic functions on a real analytic manifold $M$ is far from being a unique factorization domain because factorization involves, in the non compact case, infinitely many factors. In addition if $H^1(M,\Z_2)\neq0$, there are analytic functions on $M$ that admit different factorizations. However, if we focus on $\R^n$, we are able to associate to a real analytic function on $\R^n$ a (possibly  infinite) collection of irreducible factors (see \S\ref{irredfactorsc}), unique up to analytic units, and we provide some procedure involving sheaf theory (via a \em sheaf-product\em, see \S\ref{sheafprod}) to recover the initial analytic function from its irreducible factors. Thus, we manage an additional tool: \em factorization\em.  One of the main purposes of this  section is to reduce the representation of an analytic function $f:\R^n\to\R$ as a sum of squares (of meromorphic functions on $\R^n$) to the representation of the irreducible factors of $f$ that divide $f$ with odd multiplicity. This involves the development of  {\it countable  Pfister's multiplicative formulae} for the field of meromorphic functions on $\R^n$.

We begin with some well known useful constructions.

\subsection{ Globally principal ideal sheafs.}

Consider the standard  exponential sheaf sequences  for Stein spaces  and for C-analytic spaces

$$0\to \Z\to \Oo_X \to \Oo^*_X\to  0$$

$$0\to \Oo_X \to \Oo^*_X\to \frac{\Oo^*_X}{\Oo_X^+}= \frac{\Z}{2\Z} \to 0$$

where $\Oo_X^+\subset \Oo^*_X$ denote the sheaf of positive unit germs.

The corresponding long cohomology exact sequences together with Cartan's Theorem B give isomorphisms:
\begin{itemize}
\item $H^1(X, \Oo_X^*) \cong H^2(X,\Z)$ for a Stein space $(X,\Oo_X)$; 
\item $H^1(X, \Oo_X^*) \cong H^1(X, \Z/2\Z)$ for a C-analytic space.
\end{itemize}

We are mainly concerned with $X=\R^n$ or a contractible open Stein subsets $\Omega$ of $\C^n$. In both cases $H^1(X, \Oo_X^*)=0$. 

The vanishing of $H^1(X, \Oo_X^*)$ implies that a locally principal sheaf  $\Ff$ on $X$  is in fact globally principal. Indeed, consider a locally finite open covering $\{U_j\}_j$ of $X$ and assume that $\Ff$ is generated by $f_j$ on $U_j$. Then the family $\frac{f_j}{f_k}_{j,k}$ is an element of $H^1(X, \Oo_X^*)=0$. This implies that we can modify $f_j$ by a suitable unit $\lambda_j$ in such a way that for all $(j,k)$ one has $\lambda_jf_j= \lambda_kf_k$ on $U_j\cap U_k$. Thus all these generators glue together and give a global generator of $\Ff$.

\subsection{ Sheaf-products.}\label{sheafprod}

For our purposes, given a (countable) collection of either analytic functions on $\R^n$ or holomorphic functions on a contractible open Stein subsets $\Omega$ of $\C^n$ whose zerosets constitute a locally finite family, we need to construct an either analytic or holomorphic function that behaves (up to multiplication by a unit) as the (possibly infinite) product of all the functions in the collection. To that end, we recall a well-known procedure that will be called \em sheaf-product \em for short. More precisely,

\begin{defn}\label{realprod}

Let $\{f_k\}_{k\geq1}$ be a collection of analytic functions on $\R^n$ such that the family $\{X_k=\{f_k=0\}\}_{k\geq1}$ of their zerosets is locally finite in $\R^n$. Consider the locally principal analytic sheaf of ideals

$$
\Ff_x=
\left\{\begin{array}{ll}
\prod_{k\geq1,x\in X_k}f_{k,x}\an_{\R^n_x}&\text{ if $x\in X=\bigcup_{k\geq1}X_k$,}\\[4pt]
\an_{\R^n,x}&\text{ if $x\in\R^n\setminus X$.}
\end{array}\right.
$$ 

So, $\Ff$ is globally generated by $f$ which is unique up to units. 
We call this generator $f$ the \em sheaf-product $\shprod f_k$ of the analytic functions $f_k$\em. Note that if each $f_k$ is positive semidefinite, we may always assume that so is $\shprod f_k$.

In a similar way, given a collection $\{F_k\}_{k\geq1}$ of holomorphic functions on a contractible open Stein subset $\Omega$ of $\C^n$ such that the family $\{Z_k=\{F_k=0\}\}_{k\geq1}$ of their zerosets is locally finite in $\Omega$, we define the \em sheaf-product $\shprod F_k$ of the holomorphic functions $F_k$\em.
\end{defn}
 Note that, by means of the following lemma, if $\Omega\cap\R^n\neq\varnothing$,  $\Omega$ and each $F_k$ are invariant, we may always choose $\shprod F_k$ also invariant.

\begin{lem}\label{suit}
Let $\Omega$ be an invariant contractible open subset of $\C^n$ such that $\Omega\cap\R^n\neq\varnothing$ and let $F:\Omega\to\C$ be a holomorphic function. Suppose that there exists a unit $u\in\hol(\Omega)$ such that $\ol{F\circ\sigma}=uF$. Then, there exists a unit $v\in\hol(\Omega)$ such that $vF:\Omega\to\C$ is invariant.
\end{lem}
\begin{proof}
We may assume that $F$ is not identically zero. 
Since $\ol{F\circ\sigma}=uF$ and $|\ol{F\circ\sigma}| = |F|$ we get $|u|^2 = u \u\circ \sigma =1$, hence $F=\ol{u\circ\sigma}\ol{F\circ\sigma}$.
 Since $\Omega$ is contractible, there exists a unit $v\in\hol(\Omega)$ such that $v^2=u$. Notice that $v\ol{v\circ\sigma}=1$, because $(v\ol{v\circ\sigma})^2=u\ol{u\circ\sigma}=1$ and its restriction to $\Omega\cap\R^n\neq\varnothing$ is positive. Then,
$$
\ol{(vF)\circ\sigma}=\ol{v\circ\sigma}\ol{F\circ\sigma}=v(\ol{v\circ\sigma})^2\ol{F\circ\sigma}=v\ol{u\circ\sigma}\ol{F\circ\sigma}=vF,
$$
so, $vF$ is invariant, as required.
\end{proof}

\subsection{Sheaf-extension.}\label{rcext}

Let $X\subset\R^n$ be a C-analytic set and let $U$ be a neighborhood of $X$ in $\R^n$. Given an analytic function $g:U\to\R$ such that $\{g=0\}\subset X$, we would like to find a global analytic function $f:\R^n\to\R$ that behaves as an extension of $g$ (up to multiplication by a unit). We also approach the analogous situation in the complex case.

\begin{defn}\label{realext}
Let $X\subset\R^n$ be a C-analytic set and let $f:W\to\R$ be an analytic function on a neighborhood $W$ of $X$ in $\R^n$ such that $\{f=0\}\subset X$. Then, the unique global generator $\widehat{f}\in\an(\R^n)$ (up to multiplication by a unit) of the locally principal analytic sheaf of ideals 
$$
\Ff_x=\left\{
\begin{array}{ll}
f_x\an_{\R^n,x}&\text{ if $x\in X$,}\\[4pt]
\an_{\R^n,x}&\text{ if $x\in\R^n\setminus X$.}
\end{array}\right.
$$
will be  called  the \em sheaf-extension $\widehat{f}$ of $f$ to $\R^n$\em. Of course, if $f$ is positive semidefinite, we may assume that so is $\widehat{f}$.

Analogously, given a contractible open Stein subset $\Omega$ of $\C^n$, an analytic subset $Z$ of $\Omega$, an open neighborhood $U$ of $Z$ in $\Omega$ and a holomorphic function $F:U\to\C$ such that $\{F=0\}\subset Z$, we define the \em sheaf-extension $\widehat{F}$ of $F$ to $\Omega$\em. As before if $\Omega\cap\R^n\neq\varnothing$ and $\Omega$, $U$ and $F$ are invariant, we may always choose $\widehat{F}$ also invariant.
\end{defn}
For our purposes, it will be crucial to ``sheaf-extend'' holomorphically to a suitable common domain countable families of either analytic or holomorphic functions whose zerosets constitute a locally finite family. For this we need some lemmas. 

\begin{lem}\label{neighs}
Let $X$ be a paracompact second countable topological space and let $\{T_k\}_{k\geq1}$ be a locally finite family of closed sets in $X$. For each $k\geq1$, let $V_k$ be an open neighborhood of $T_k$ in $X$. Then, there exists open neighborhoods $U_k\subset V_k$ of $T_k$ in $X$ such that the family $\{U_k\}_{k\geq1}$ is locally finite in $X$.
\end{lem}
\begin{proof}
For each $x\in X$ there is an open neighborhood $W^x$ of $x$ that meets only finitely many $T_k$. The family $\{W^x\}_{x\in X}$ is an open covering of $X$. Thus, it has an open refinement $\{W_\ell\}_{\ell\geq1}$ that is countable and locally finite in $X$. We define $U_k'=\bigcup_{W_\ell\cap T_k\neq\varnothing}W_\ell$ and observe that $T_k\subset U_k'$. We claim that the family $\{U_k'\}_{k\geq1}$ is locally finite in $X$.

Indeed, fix a point $x\in X$ and let $V^x$ be a neighborhood of $x$ that intersects finitely many $W_\ell$, say $W_{\ell_1},\ldots,W_{\ell_r}$. Now, the union $\bigcup_{j=1}^rW_{\ell_j}$ meets only finitely many $T_k$, say $T_{k_1},\ldots,T_{k_s}$. Thus, if $k\neq k_1,\ldots,k_s$, the intersection $U_k'\cap V^x=\varnothing$. To finish, just take $U_k=U_k'\cap V_k$.
\end{proof}

\begin{lem}\label{bigneigh}
Let $\{\Omega_k\}_{k\geq1}$ be a locally finite family of open subsets of $\C^n$ and for each $k\geq1$ let $T_k\subset\Omega_k$ be a closed subset of $\Omega_k$ such that $\cl_{\C^n}(T_k)\cap\R^n=T_k\cap\R^n$. Then, there exists a ($\sigma$-invariant contractible Stein) open neighborhood $\Omega$ of $\R^n$ in $\C^n$ such that $\Omega\cap T_k$ is a closed subset of $\Omega$ for all $k\geq1$.
\end{lem}
\begin{proof}
The obstruction to take $\Omega_0=\bigcup_{k\geq1}\Omega_k$ as the open neighborhood of $\R^n$ in $\C^n$ in the statement appears essentially at the sets $\cl_{\C^n}(T_i)\cap\Omega_j$ for $j\neq i$. To get rid of such obstruction, we proceed as follows. 

First, since the family $\{\Omega_k\}_{k\geq1}$ is locally finite in $\C^n$, so are the families $\{T_k\}_{k\geq1}$, $\{\cl_{\C^n}(T_k)\}_{k\geq1}$ and $\{\cl_{\C^n}(T_k)\setminus\Omega_k\}_{k\geq1}$. Thus,
$$
S=\bigcup_{k\geq1}\cl_{\C^n}(T_k)\qquad\text{and}\qquad C_k=\bigcup_{j\neq k}\cl_{\C^n}(T_j)\setminus\Omega_j
$$
are closed subset of $\C^n$. Consider the open set of $\C^n$
$$
\Omega=(\C^n\setminus S)\cup\bigcup_{k\geq1}\Omega_k\setminus C_k.
$$
We check first that $\R^n\subset\Omega$. Since $\cl_{\C^n}(T_k)\cap\R^n=T_k\cap\R^n$, we have
$$
C_k\cap\R^n=\bigcup_{j\neq k}(\cl_{\C^n}(T_j)\cap\R^n)\setminus\Omega_j=\bigcup_{j\neq k}(T_j\cap\R^n)\setminus\Omega_j=\varnothing,
$$
for all $k\geq1$ and $S\cap\R^n=\bigcup_{k\geq1}T_k\cap\R^n$. Thus, since $T_k\subset\Omega_k$,
$$
\R^n\cap\Omega=\Big(\R^n\setminus\bigcup_{k\geq1}T_k\cap\R^n\Big)\cup\bigcup_{k\geq1}\Omega_k\cap\R^n=\R^n.
$$
Next, we show that each intersection $T_\ell\cap\Omega$ is closed in $\Omega$ for $\ell\geq1$. Fix $\ell\geq1$ and observe that since $\Omega$ is open in $\C^n$, we have $\cl_{\Omega}(T_\ell\cap\Omega)=\cl_{\C^n}(T_\ell)\cap\Omega$. Hence, it is enough to show that $\cl_{\C^n}(T_\ell)\cap\Omega\subset T_\ell$ for each $\ell\geq1$. Indeed, since $T_\ell$ is closed in $\Omega_\ell$, we have $\cl_{\C^n}(T_\ell)\cap\Omega_\ell=T_\ell$. Since $\cl_{\C^n}(T_\ell)\setminus\Omega_\ell\subset C_k$ if $k\neq\ell$, we deduce
\begin{multline*}
\cl_{\C^n}(T_\ell)\cap\Omega=\bigcup_{k\geq1}\cl_{\C^n}(T_\ell)\cap(\Omega_k\setminus C_k)\\[-10pt]
\subset\cl_{\C^n}(T_\ell)\cap\Omega_\ell\cup\bigcup_{k\neq\ell}\cl_{\C^n}(T_\ell)\setminus(\cl_{\C^n}(T_\ell)\setminus\Omega_\ell)=T_\ell,
\end{multline*}
as required. As always, we may assume that $\Omega$ is invariant, contractible and Stein.
\end{proof}
\begin{lem}\label{piccolo}
Let $V$ be an open neighborhood of $\R^n$ in $\C^n$ and let $Z\subset V$ be a closed subset of $V$. Let $\Omega$ be an open subset of $\C^n$ such that $Z\cap\R^n\subset\Omega\subset V$. Then, $\cl_{\C^n}(Z\cap\Omega)\cap\R^n=(Z\cap\Omega)\cap\R^n$.
\end{lem}
\begin{proof}
Since $Z$ is closed in $V$, we have $\cl_{\C^n}(Z)\cap V=Z$. Since $\R^n\subset V$, we have $\cl_{\C^n}(Z)\cap\R^n=\cl_{\C^n}(Z)\cap V\cap\R^n=Z\cap\R^n$. Thus,
$$
Z\cap\R^n\subset\cl_{\C^n}(Z\cap\Omega)\cap\R^n\subset\cl_{\C^n}(Z)\cap\R^n=Z\cap\Omega\cap\R^n,
$$
as required.
\end{proof}

Next, we point out here a sufficient and necessary condition to assure that a countable family of closed subsets of an open set $\Omega\subset\C^n$ is locally finite.
\begin{lem}\label{lfc}
Let $\Omega\subset\C^n$ be an open set and let $\{T_\ell\}_{\ell\geq1}$ be a family of closed subsets of $\Omega$. Then, the family $\{T_\ell\}_{\ell\geq1}$ is locally finite if and only if there exists an exhaustion $\{K_\ell\}_{\ell\geq1}$ of $\Omega$ by compact sets such that $T_\ell\cap K_\ell=\varnothing$ for all $\ell\geq1$.
\end{lem}
\begin{proof}
The if part is straightforward and we do not include the concrete details. For the converse, we begin by fixing an exhaustion $\{L_k\}_{k\geq1}$ of $\Omega$ by compact sets, that is, $\Int_{\Omega}(L_1)\neq\varnothing$, $L_k\subset\Int_{\Omega}(L_{k+1})$ and $\Omega=\bigcup_{k\geq1}L_k$. For each $k\geq1$, we define
$$
\ell(k)=\sup\{\ell\geq1:\ L_k\cap T_\ell\neq\varnothing\}.
$$
and notice that $\ell(k)<+\infty$ because the family $\{T_\ell\}_{\ell\geq1}$ is locally finite. Morever, the sequence $\{\ell(k)\}_{k\geq1}$ is increasing and it tends to $+\infty$ because $\Omega=\bigcup_{k\geq1}L_k$. Next, we choose a closed ball $L_0$ contained in $\Int_{\Omega}(L_1)\setminus\bigcup_{j=1}^{\ell(1)}T_j$ and define
$$
K_\ell'=\left\{\begin{array}{ll}
L_0&\text{ if $\ell(0)=1\leq\ell\leq\ell(1)$}\\[4pt]
L_k&\text{ if $\ell(k)+1\leq\ell\leq\ell(k+1)$}
\end{array}\right.
$$
Observe that $K_\ell'\cap T_\ell=\varnothing$, $K_{\ell(k)}'\subset\Int_{\Omega}(K_{\ell(k)+1}')$ for all $k\geq0$ and $\bigcup_{\ell\geq1}K_\ell'=\Omega$. Finally, modifying slightly the compact sets $K_\ell'$, we obtain the desired exhaustion $\{K_\ell\}_{\ell\geq1}$.
\end{proof}

After   the  previous   preparation,  we   approach  the   holomorphic
``sheaf-extension''  to  a  suitable  common  (complex)  domain  of  a
collection  of  either analytic  functions  on  $\R^n$ whose  zerosets
constitute a locally finite family  of $\R^n$ or holomorphic functions
whose  (complex) domains  of  definition constitute  a locally  finite
family of open subsets of $\C^n$. Namely,

\begin{prop}\label{commonext}
Let $\{f_k\}_{k\geq1}$ be a collection of analytic functions on $\R^n$ such that the family $\{X_k=\{f_k=0\}\}_{k\geq1}$ of their zerosets is locally finite in $\R^n$. Then, there exist an invariant open neighborhood $\Omega$ of $\R^n$ in $\C^n$ and invariant holomorphic functions $\widehat{F}_k\in\hol(\Omega)$ such that the family $\{\widehat{F}_k=0\}_{k\geq1}$ of their zerosets is locally finite in $\C^n$ and for each $k\geq 1$ there is a positive unit $u_k\in\an(\R^n)$ satisfying $\widehat{F}_{k|_{\R^n}}=u_kf_k$.
\end{prop}
\begin{proof}
We may assume that all the $f_k$ are non identically zero. For each $k\geq1$, let $V_k\subset\C^n$ be an invariant open neighborhood of $\R^n$ in $\C^n$ such that $f_k$ extends to an invariant holomorphic function $F_k:V_k\to\C$. By Lemma \ref{neighs}, there is a locally finite family in $\C^n$ of invariant open sets $\{\Omega_k\}_{k\geq1}$ such that $X_k\subset\Omega_k\subset V_k$. Let $T_k=\{F_k=0\}\cap\Omega_k$ which is invariant and, by Lemma \ref{piccolo}, satisfy $\cl_{\C^n}(T_k)\cap\R^n=X_k=T_k\cap\R^n$. 

Next, by Lemma \ref{bigneigh}, there is an invariant open neighborhood $\Omega$ of $\R^n$ in $\C^n$ such that each intersection $\Omega\cap T_k$ is a closed subset of $\Omega$. Next, for each $k\geq1$ let $\widehat{F}_k\in\hol(\Omega)$ be a $\sigma$-invariant sheaf-extension of $F_k$ to $\Omega$. A straightforward computation shows that the holomorphic functions $\{\widehat{F}_k\}_{k\geq1}$ satisfy the desired conditions.
\end{proof}
\begin{prop}\label{commonext2}
Let $\{\Omega_k\}_{k\geq1}$ be a locally finite family in $\C^n$ of $\sigma$-invariant open sets and for each $k\geq1$ let $G_k:\Omega_k\to\C$ be a $\sigma$-invariant holomorphic function. Denote $S_k=\{G_k=0\}$ and assume that $S_k\cap\R^n$ is a (closed) global analytic set of $\R^n$ and $\cl_{\C^n}(S_k)\cap\R^n=S_k\cap\R^n$ for each $k\geq1$. Then, there are a $\sigma$-invariant open neighborhood $\Omega$ of $\R^n$ in $\C^n$ such that $S_k\cap\Omega$ is closed in $\Omega$ and $\sigma$-invariant sheaf-extensions $\widehat{G}_k\in\hol(\Omega)$ of the functions $G_k$.
\end{prop}
\begin{proof}
We may assume that $G_k$ is not identically zero in any of the connected components of $\Omega_k$. Next, by Lemma \ref{bigneigh}, there is a $\sigma$-invariant contractible open Stein neighborhood $\Omega$ of $\R^n$ in $\C^n$ such that $Z_k=S_k\cap\Omega$ is a closed subset of $\Omega$ for all $k\geq 1$. Next, for each $k\geq1$ let $\widehat{G}_k\in\hol(\Omega)$ be a $\sigma$-invariant sheaf-extension of $G_k$ to $\Omega$. A straightforward computation shows that the holomorphic functions $\{\widehat{G}_k\}_{k\geq1}$ satisfy the required conditions.
\end{proof}

\subsection{Infinite sums of squares of analytic functions.}\label{strong}

A relevant consequence of the notion of sheaf extension and Theorem \ref{globalization} is the following result which concerns holomorphic extension of infinite sums of squares of analytic functions to the same complex neighbourhood.

\begin{prop}\label{extension}
Let $\{f_k\}_{k\geq1}$ be a collection of analytic functions on $\R^n$ such that: 
\begin{itemize}
\item the family $\{X_k=\{f_k=0\}\}_{k\geq1}$  is locally finite in $\R^n$, 
\item each $f_k$ is a sum of $p_k$ squares at $X_k$, where $1\leq p_k\leq +\infty$. 
\end{itemize}
Then, there are an invariant open neighborhood $\Omega$ of $\R^n$ in $\C^n$, positive real constants $c_k$ for $k\geq1$ and $\sigma$-invariant homolomorphic functions $G_k,\widehat{F}_k,A_{kj}\in\hol(\Omega)$ for $k\geq1$ and $1\leq j\leq q_k=2^{n+1}p_k$ such that: 
\begin{itemize}
\item[(i)] $\widehat{F}_k|_{\R^n}=u_kf_k$ for a positive unit $u_k\in\an(\R^n)$.
\item[(ii)] The family $\{G_k^2\widehat{F}_k=0\}_{k\geq1}$ is locally finite in $\Omega$ and $\{G_k^2\widehat{F}_k=0\}\cap\R^n=X_k$.
\item[(iii)] The sum $A_k=(d_kG_k^2\widehat{F}_k)^2+\sum_{j=1}^{q_k}A_{kj}^2$ converges in $\Omega$ (in the sense of Definition  \ref{strongconv} if $q_k$ is not finite) for each family of real values $\{d_k\}_{k\geq1}$ satisfying $d_k\geq c_k$. 
\end{itemize}

Moreover,
\begin{itemize}
\item $\{A_{k|\R^n}=0\}=X_k$ and the family $\{A_k=0\}_{k\geq1}$ is locally finite in $\Omega$; 
\item $A_k=G_k^2\widehat{F}_kQ_k$ for some $\sigma$-invariant $Q_k\in\hol(\Omega)$, which does not vanish on a neighborhood $V_k$ of $\{G_k^2\widehat{F}_k=0\}\cup\R^n$.
\end{itemize}
\end{prop}

\begin{proof}
We will split the proof into several steps.

\vspace{1mm}
\noindent{\sc Starting point.} By Theorem \ref{globalization} we can find a representation of $f_k$ as a sum of $q_k=2^{n-1}p_k$ squares on $\R^n$ with controlled denominators. This provides the following:

\vspace{1mm}
\begin{enumerate}
\item By Lemma \ref{commonext},  there are an invariant open neighborhood $\Omega_0$ of $\R^n$ in $\C^n$ and invariant holomorphic functions $\widehat{F}_k\in\hol(\Omega_0)$ such that the family $\{\widehat{F}_k=0\}_{k\geq1}$ is locally finite in $\C^n$, $\widehat{F}_{k|{\R^n}}=u_kf_k$ and $u_k\in\an(\R^n)$ is a positive unit.
\item  There are invariant open neighborhoods $U_k$ of $\R^n$ in $\C^n$ and invariant holomorphic functions $G_k',B_{kj}:U_k\to\C$ such that $G_k'^2\widehat{F}_k=\sum_{j=1}^{q_k}B_{kj}^2$ and $\{G_k'=0\}\cap\R^n\subset X_k$.
\end{enumerate}
\vspace{1mm}
\noindent{\sc Step 1: Extension of denominators.} 
The family $\{T_k=\{\widehat{F}_k=0\}\cap U_k\}_{k\geq1}$ is, by (1) above, locally finite in $\C^n$. By Lemma \ref{neighs}, there is a locally finite family in $\C^n$ of invariant open sets $\{\Omega_k\}_{k\geq1}$ in $\C^n$ such that $T_k\subset\Omega_k\subset U_k$. Consider the invariant sets $S_k=\{G_k'^2\widehat{F}_{k|{U_k}}=0\}\cap\Omega_k$. By Lemma \ref{piccolo}, we have $\cl_{\C^n}(S_k)\cap\R^n=X_k=S_k\cap\R^n$ for all $k\geq1$. 

By Lemma \ref{commonext2}, we find an invariant open neighborhood $\Omega\subset\Omega_0$ of $\R^n$ in $\C^n$ such that $\Omega\cap S_k$ is a closed subset of $\Omega$ and invariant sheaf-extensions $G_k=\widehat{G}_k'\in\hol(\Omega)$ of the functions $G_k'\,$. To simplify notations, denote again by $\widehat{F}_k$ its restriction to $\Omega$ and by $\Omega_k$ and $S_k$ their respective intersections with $\Omega$.
Let $v_k\in\hol(\Omega_k)$ be an invariant unit such that $G_k|_{\Omega_k}=v_kG'_{k|{\Omega_k}}$. Thus,  if we write $C_{kj}=v_kB_{{kj}|{\Omega_k}}$, we have
$$
(G_k^2\widehat{F}_k)_{|{\Omega_k}}=v_k^2(G_k'^2\widehat{F}_k)_{|{\Omega_k}}=\sum_{j=1}^{q_k}(v_kB_{{kj}|{\Omega_k}})^2=\sum_{j=1}^{q_k}C_{kj}^2.
$$
Moreover, the family $\{S_k=\{G_k^2\widehat{F}_k=0\}\}_{k\geq1}$ is locally finite in $\Omega$ and $\{G_k^2\widehat{F}_k=0\}\cap\R^n=X_k$ for all $k\geq1$.

\vspace{1mm}
\noindent{\sc Step 2: Globalization of sums of squares.} In this Step  we replace the sums of squares $\sum_{j=1}^{q_k}C_{kj}^2$, which are defined only on the sets $\Omega_k$, by global sum of squares $\sum_{j=1}^{q_k}A_{kj}^2$ satisfying the conditions in the statement.

Fix $k\geq1$ and observe that $\Omega_k\cap\{G_k^2\widehat{F}_k=0\}=S_k$ which is a closed subset of $\Omega$. We apply Lemma \ref{serie} to $\Phi=(G_k^2\widehat{F}_k)^2$, $V=\Omega_k$ and $C_j=C_{kj}$ and we find invariant holomorphic functions $A_{kj}$ on $\Omega$ such that: 
\begin{itemize}
\item $\sum_{j=1}^{q_k}\sup_K|A_{kj}|^2<+\infty$ for all compact sets $K\subset\Omega$,  
\item $(G_k^2\widehat{F}_k)^2$ divides $A_{kj}-C_{kj}$ on $\Omega_k$. 
\end{itemize}
Next, let $\{K_k\}_{k\geq1}$ be an exhaustion of $\Omega$ by compact sets such that $S_k\cap K_k=\varnothing$ for all $k\geq1$ (see Lemma \ref{lfc}). Denote $A_k'=\sum_{j=1}^{q_k}A_{kj}^2$ and let 
$$
c_k=1+\frac{\max_{K_k}(|A_k'|)}{\min_{K_k}(|G_k^2\widehat{F}_k|^2)}>0.
$$
Write $A_{k0}=d_kG_k^2\widehat{F}_k$ for $d_k\geq c_k$ and define $A_k=A_{k0}^2+A_k'=\sum_{j\geq0}A_{kj}^2$. Each function $A_k$ is holomorphic and invariant on $\Omega$ and its restriction to $\R^n$ vanishes only at $X_k$, because so does $A_{{k0}|{\R^n}}$. Moreover, let us check that for such choice of $d_k$'s the family $\{A_k=0\}_{k\geq1}$ is locally finite in $\Omega$. By Lemma \ref{lfc} it is enough to prove that $\{A_k=0\}\cap K_k=\varnothing$ for all $k\geq1$. 
Indeed, suppose  there exists $z\in\{A_k=0\}\cap K_{\ell(k)}$. Then, $(d_kG_k^2\widehat{F}_k)^2(z)=-A_k'(z)$ and since $(G_k^2\widehat{F}_k)(z)\neq0$, we have 
$$
1\leq c_k^2\leq d_k^2=\frac{|A_k'(z)|}{|(G_k^2\widehat{F}_k)(z)|^2}<1+\frac{\max_{K_{\ell(k)}}(|A_k'|)}{\min_{K_{\ell(k)}}(|G_k^2\widehat{F}_k|^2)}=c_k\leq d_k,
$$
which is a contradiction.

\vspace{1mm}
Next, we claim that the meromorphic function $Q_k=A_k/(G_k^2\widehat{F}_k)$ is holomorphic on $\Omega$ and it does not vanish on a neighborhood of $\{G_k^2\widehat{F}_k=0\}\cup\R^n$.

\vspace{1mm}
Since the zeroset of $G_k^2\widehat{F}_k$ in $\Omega$ is $S_k\subset\Omega_k$, we have
$$
\{A_k=0\}\cap\R^n=X_k\subset S_k\subset\Omega_k
$$
and our claim reduces to show that $Q_k$ is a holomorphic unit on an open neighborhood of $S_k$. Indeed, on $\Omega_k$ we have:
$$
\sum_{j=1}^{q_k}A_{kj}^2-G_k^2\widehat{F}_k=\sum_{j=1}^{q_k}A_{kj}^2-\sum_{j=1}^{q_k}C_{kj}^2=\sum_{j=1}^{q_k}(A_{kj}^2-C_{kj}^2).
$$
Of course, all involved series are convergent on the compact sets of $\Omega_k$. Also, by construction, $(G_k^2\widehat{F}_k)^2$ divides on $\Omega_k$ each term $A_{kj}^2-C_{kj}^2=(A_{kj}+C_{kj})(A_{kj}-C_{kj})$. Thus, it divides their sum $\sum_{j=1}^{q_k}A_{kj}^2-G_k^2\widehat{F}_k$ and we can write
$$
A_k=\sum_{j=0}^{q_k}A_{kj}^2=A_{k0}^2+\sum_{j=1}^{q_k}A_{kj}^2=G_k^2\widehat{F}_k+H_k{(G_k^2\widehat{F}_k)}^2=Q_kG_k^2\widehat{F}_k,
$$
where $H_k\in\hol(\Omega_k)$ and $Q_k=1+H_k{G_k^2\widehat{F}_k}$ in $\Omega_k$; hence, shrinking $\Omega_k$ if necessary, we get that $Q_k$ is an invariant holomorphic unit on $\Omega_k$, as required.
\end{proof}
 
\subsection{Bounded analytic Pfister's formula.}\label{s3}

Classical Pfister's formula  says that in a field $K$ the subset of sums of $2^r$ squares is closed under multiplication. In order to ease the writing of what follows we will use the standard notation due to Pfister. Given an element $f$ of a ring $A$, the expression $f=\nsq{p}$ means that $f$ is a sum of $p$ squares in $A$. When several $\nsq{p}$'s appear in the same formula, the squared functions need not to be the same. For instance, Pfister's formula can be re-written as $\nsq{2^r}\nsq{2^r}=\nsq{2^r}$. 

In particular note that for $r=1,2,3$ (by means of complex numbers, quaternions and octonions) we have  $\nsq{2^r}\nsq{2^r}=\nsq{2^r}$  for each ring $A$. The main purpose of this section is to prove the following results.

\begin{thm}[Analytic Pfister's formulae]\label{vmain}
Let $\{f_k\}_{k\geq1}$ be a collection of positive semidefinite analytic functions on $\R^n$ such that the family $\{X_k=\{f_k=0\}\}_{k\geq1}$ of their zerosets is locally finite. Let $f$ be a positive semidefinite sheaf-product of the functions $f_k$, that is, a positive semidefinite analytic function on $\R^n$ such that 
$$
f\an_{\R^n,x}=\left\{\begin{array}{cl}
\prod_{k\geq1, x\in X_k}f_{k,x}\an_{\R^n,x}&\text{if $x\in\bigcup_{k\geq1}X_k$,}\\[4pt]
\an_{\R^n,x}&\text{otherwise.}
\end{array}\right.
$$
Then,
\begin{itemize}
\item[(i)] If each $f_k$ is a finite sum of squares of meromorphic functions on $\R^n$, $f$ is a possibly infinite sum of squares of meromorphic functions on $\R^n$.
\item[(ii)] If there is $r$ such that  each $f_k$ is a sum of $2^r$ squares of meromorphic functions, $f$ is a sum of $2^{n+r}$ squares of meromorphic functions  on $\R^n$.
\item[(iii)] If some $f_k$ is an infinite sum of squares, there is a positive semidefinite analytic function $q$ on $\R^n$ such that $\{q=0\}\subset\{f=0\}$, $\dim\{q=0\}<\dim\{f=0\}$ and $qf$ is an infinite sum of squares of meromorphic functions.
\end{itemize}
\end{thm}
\begin{thm}\label{nash}
Let $\{f_k\}_{k\geq1}$ be a collection of positive semidefinite analytic functions on $\R^n$ such that the family $\{X_k=\{f_k=0\}\}_{k\geq1}$ of their zerosets is locally finite and each $f_k$ is of Nash type. Let $f$ be a positive semidefinite sheaf-product of the functions $f_k$. Then, $f$ is a sum of $2^{n+1}$ squares of meromorphic functions on $\R^n$.
\end{thm}

By means of Proposition \ref{bad}, Theorem \ref{vmain} (ii) reduces to prove the following statement.

\begin{thm}[Bounded analytic Pfister's formula]\label{main1}
Let $\{f_k\}_{k\geq1}$ be a collection of analytic functions on $\R^n$ such that the family $\{X_k=\{f_k=0\}\}_{k\geq1}$ of their zerosets is locally finite. Assume that each $f_k$ is a sum of $2^r$ squares of meromorphic function germs  at $X_k$ with controlled denominators (for a fixed $r$). Then, $\shprod f_k$ is a sum of $2^{r+1}$ squares of meromorphic functions with controlled denominators. 
\end{thm}

\subsubsection{Pfister's multiplicative formulae.}

One way to prove Pfister's multiplicative formulae consists of providing, given an element $a=\nsq{2^r}$ in a field $K$ (resp. a ring  $A$), a matrix $M_a\in{\mathfrak M}_{2^r}(K)$ of order $2^r$ and coefficients in $K$ (resp. in $A$) such that $M_a^tM_a=aI_{2^r}$. Once this is done, if $a,b$ are sums of $2^r$ squares, we have $(M_aM_b)^t(M_aM_b)=abI_{2^r}$ and this provides a representation of $ab$ as a sum of $2^r$ squares. Of course, $I_{2^r}$ denotes the identity matrix of the ring ${\mathfrak M}_{2^r}(K)$ (resp. ${\mathfrak M}_{2^r}(A)$).

\begin{lem}\label{matrix0}
Let $A$ be a unitary commutative ring and let $g_i,f_i\in A$ for $i=1,2$. Suppose that there exist matrices $M_1,M_2\in{\mathfrak M}_{2^r}(A)$ such that $M_i^tM_i=g_i^2f_iI_{2^r}$. Then, there is a matrix $M\in{\mathfrak M}_{2^{r+1}}(A)$ such that $M^tM=g^2fI_{2^r}$ where $g=g_1g_2^2f_2$ and $f=f_1+f_2$.
\end{lem}
\begin{proof}
A straighforward computation shows that
$$
M=\begin{pmatrix}
g_2^2f_2M_1&-g_1g_2f_2M_2\\[4pt]
g_1g_2f_2M_2&M_2M_1^tM_2
\end{pmatrix}\in{\mathfrak M}_{2^r}(A)
$$
is the matrix we sought.
\end{proof}

In our case $A=\an(\R^n)$, we are concerned about ``controlling denominators''. This is done by means of the following result.

\begin{lem}\label{matrix}
Let $f,a\in\an(\R^n)$ be analytic functions such that $f=a^2+\nsq{2^{r-1}}$ in $\an(\R^n)$ and $\{a=0\}=\{f=0\}$. Then, there exist $g\in\an(\R^n)$ with $\{g=0\}\subset\{f=0\}$ and a matrix $M\in{\mathfrak M}_{2^r}(\an(\R^n))$ such that $M^tM=g^2fI_{2^r}$.\enlargethispage{0.4mm}
\end{lem}
\begin{proof}
Indeed, if $r=1$, we have $f=a^2+b^2$ and it is enough to take $g=1$ and 
$$
M=\begin{pmatrix}
a&-b\\
b&a
\end{pmatrix}\in{\mathfrak M}_2(\an(\R^n)).
$$
Assume now $r>1$, and let $\displaystyle f_1=\left(\frac{3}{5}a\right)^2+\nsq{2^{r-2}}$ and
$\displaystyle f_2=\left(\frac{4}{5}a\right)^2+\nsq{2^{r-2}}$ such that $f=f_1+f_2$. Notice that $\{f_1=0\}=\{f_2=0\}=\{a=0\}=\{f=0\}$. By induction, there are analytic functions $g_1,g_2\in\an(\R^n)$ such that $\{g_i=0\}\subset\{f_i=0\}$ and matrices $M_i\in{\mathfrak M}_{2^{r-1}}(\an(\R^n))$ satisfaying $M_i^tM_i=g_i^2f_iI_{2^{r-1}}$ for $i=1,2$. By Lemma \ref{matrix0}, there is a matrix $M\in{\mathfrak M}_{2^{r+1}}(\an(\R^n))$ such that $M^tM=g^2fI_{2^r}$ where $g=g_1g_2^2f_2$. Observe that $\{g=0\}=\{g_1=0\}\cup\{g_2=0\}\cup\{f_2=0\}=\{f=0\}$, as required.
\end{proof}

\begin{remarks}\label{matrixr}

\begin{itemize}\parindent =0pt 
\item[(i)] {\em If $f=\nsq{p}$ in $\an(\R^n)$, there exists $a\in\an(\R^n)$ such that $f=a^2+\nsq{p}$ in $\an(\R^n)$ and $\{a=0\}=\{f=0\}$.}

Indeed, $1+f$ is a positive unit of $\an(\R^n)$, so $\displaystyle u=\frac{1}{\sqrt{1+f}}\in\an(\R^n)$. Thus, if we denote $a=fu$, we have $f=(f+f^2)u^2=a^2+u^2\nsq{p}=a^2+\nsq{p}$ and $\{a=0\}=\{f=0\}$.

\item[(ii)] {\em  If $f=\nsq{2^{r}}$ in $\an(\R^n)$ with $r\leq3$, we may choose a matrix $M_f\in{\mathfrak M}_{2^r}(\an(\R^n))$ such that $M_f^tM_f=fI_{2^r}$.} 
Namely,

$
\begin{array}{lc}
f=a_1^2+a_2^2, &   M_f={\Large \left(\begin{smallmatrix}
a_1&-a_2 \\[4 pt]
a_2&a_1
\end{smallmatrix}\right)}    \\                            \\
 f=a_1^2+a_2^2+a_3^2+a_4^2 &  M_f={\Large\left(\begin{smallmatrix}
a_1&-a_2&-a_3&-a_4\\[4pt]
a_2&a_1&a_4&-a_3\\[4pt]
a_3&-a_4&a_1&a_2\\[4pt]
a_4&a_3&-a_2&a_1
\end{smallmatrix}\right)}        \\               \\
 f=a_1^2 +a_2^2+ \cdots   +a_8^2 &         
M_f={\Large\left(\begin{smallmatrix}
a_1&-a_2&-a_3&-a_4&-a_5&-a_6&-a_7&-a_8\\[4pt]
a_2&a_1&-a_5&-a_8&a_3&-a_7&a_6&a_4\\[4pt]
a_3&a_5&a_1&-a_6&-a_2&a_4&-a_8&a_7\\[4pt]
a_4&a_8&a_6&a_1&-a_7&-a_3&a_5&-a_2\\[4pt]
a_5&-a_3&a_2&a_7&a_1&-a_8&-a_4&a_6\\[4pt]
a_6&a_7&-a_4&a_3&a_8&a_1&-a_2&-a_5\\[4pt]
a_7&-a_6&a_8&-a_5&a_4&a_2&a_1&-a_3\\[4pt]
a_8&-a_4&-a_7&a_2&-a_6&a_5&a_3&a_1
\end{smallmatrix}\right)}.\\
\end{array}
$
\end{itemize}
\end{remarks} 

\subsubsection{Pfister's bundles.}

As we did in Section \ref{lowdim} to approach the representation of a positive semidefinite analytic function on an analytic surface as sum of squares of analytic functions, we  associate a suitable analytic vector bundle endowed with a Riemannian metric to each finite sum of squares in $\an(\R^n)$.

\begin{defn} 
Let $f\in\an(\R^n)$ be an analytic function. We define a \em Pfister's bundle for $f$ \em as a tuple $\psd(f)=((E(f),\pi,\R^n), \qq{\cdot,\cdot},s)$ consisting of: 
\begin{itemize} 
\item an analytic vector bundle $(E(f),\pi,\R^n)$ of finite rank,
\item a Riemannian metric $\qq{\cdot,\cdot}$ on $E$, 
\item a global analytic section $s:\R^n\to E(f)$ such that $\qq{s,s}=f$.
\end{itemize}
\end{defn}

\begin{remark}\label{gsr} 
An analytic vector bundle over $\R^n$ is topologically, and so analitically,  trivial because $\R^n$ is contractible. Thus, a Pfister's bundle $\psd(f)$ for $f$ of rank $p$ admits $p$ unitary analytic sections $s_1,\ldots,s_p:\R^n\to E$ which are pairwise orthogonal, that is, $\qq{s_i,s_j}=\delta_{ij}$ for $1\leq i,j\leq n$.
\end{remark}
Now we characterize the finite sums of squares in terms of Pfister's bundles. Namely,
\begin{lem}\label{pfister}
Let $f\in\an(\R^n)$ be an analytic function. Then, $f$ is a finite sum of squares if and only if there exist $g\in\an(\R^n)\setminus\{0\}$ and a Pfister's bundle $\psd(g^2f)$ of finite rank for $g^2f$. Moreover, if $\psd(g^2f)$ has rank $p$, then $f$ can be represented as a sum of $p$ squares of meromorphic functions and the zeroset of the denominators of such representation is contained in $\{g=0\}$.
\end{lem}
\begin{proof}
Suppose first that $f$ is a sum of $2^{r-1}$ squares, that is, there is $g_0\in\an(\R^n)\setminus\{0\}$ such that $g_0^2f$ is a sum of $2^{r-1}$ squares in $\an(\R^n)$. By Remark \ref{matrixr}(i) and Lemma \ref{matrix}, there exist $g_1\in\an(\R^n)\setminus\{0\}$ and a matrix $M_h\in{\mathfrak M}_{2^r}(\an(\R^n))$ such that $M_h^tM_h=hI_{2^r}$ where $h=(g_0g_1)^2f$. Denote $g=g_0g_1$. 

Write $p=2^r$ and consider the analytic vector bundle $(E,\pi,\R^n)$ where $E=\R^n\times\R^p$ and $\pi:\R^n\times\R^p\to\R^n$ is the projection onto the first space. Consider the analytic Riemannian metric on $(E\oplus E,\pi',\R^n)=(\R^n\times\R^{2p},\pi',\R^n)$  given by
$$
\qq{(x,(v_1,\ldots,v_p)),(x,(w_1,\ldots,w_p))}=\sum_{\ell=1}^pv_\ell w_\ell.
$$

Let $u\in\R^p$ be any unitary vector and consider the global section 
$$
s:\R^n\to E,\ x\mapsto(x,M_h(x)(u)).
$$ 
Since $\qq{s,s}=h$, the tuple $\psd(h)=((E,\pi,\R^n),s)$ is a Pfister's bundle for $h=g^2f$.

Conversely, suppose that there exists a Pfister's bundle $\psd(g^2f)=((E,\pi,\R^n),\qq{\cdot,\cdot},s)$ of rank $p$ for $g^2f$, where $g\in\an(\R^n)\setminus\{0\}$. By Remark \ref{gsr} there exists a family of analytic sections $\{s_1,\ldots,s_p\}$ of $(E,\pi,\R^n)$ such that $\qq{s_i,s_j}=\delta_{ij}$ for $1\leq i,j\leq p$. Write now $s=\sum_{i=1}^pa_is_i$, where $a_i=\qq{s,s_i}\in\an(\R^n)$, and notice that 
$$
g^2f=\qq{s,s}=a_1^2+\cdots+a_p^2
$$
is a sum of $p$ squares in $\an(\R^n)$. Thus, $f$ is a sum of $p$ squares and the zeroset of the denominators of such representation is contained in $\{g=0\}$, as required.
\end{proof}

\begin{proof}[Proof of Theorem \ref{main1}]
By Lemma \ref{pfister} it is enough to provide $g\in\an(\R^n)$ with $\{g=0\}\subset\{f=0\}$ and a Pfister's bundle of rank $p=2^{r+1}$ for $g^2f$. We proceed in several steps.
\vspace{1mm}

\noindent{\sc Step 1.} 
For each $k\geq 1$ there exist by Proposition \ref{bad} analytic functions $g_{0k},a_k\in\an(\R^n)$ such that $\{g_{0k}=0\}\subset\{a_k=0\}=\{f_k=0\}$ and $g_{0k}^2f_k=a_k^2+\nsq{2^r}$ in $\an(\R^n)$. For each $k\geq1$ there are by Lemma \ref{matrix} an analytic function $g_{1k}\in\an(\R^n)$ with $\{g_{1k}=0\}\subset\{f_k=0\}$ and a matrix $M_k\in{\mathfrak M}_p(\an(\R^n))$ such that $M_k^tM_k=h_kI_p$ where $h_k=(g_{0k}g_{1k})^2f_k$.

Denote $f=\shprod f_k$ and $g=\shprod (g_{0k}g_{1k})$ and observe that $\{g=0\}\subset\{f=0\}$ and $h=g^2f=\shprod h_k$ where $h_k=(g_{0k}g_{1k})^2f_k$. Thus, it is enough to construct a Pfister's bundle $\psd(h)=((E,\pi,\R^n),s)$ of rank $p$ for $h$.

For each $k\geq0$ consider the set $W_k=\R^n\setminus\bigcup_{i>k}X_i$. Observe that since the family $\{X_k\}_{k\geq1}$ is locally finite, each set $W_k$ is open in $\R^n$ and $\R^n=\bigcup_{k\geq1}W_k$. 

\vspace{1mm}
\noindent{\sc Step 2: Construction of the vector bundle.} 
Note that since $M_k^tM_k=h_kI_p$, we have  $M_k(x)\in \GL (\R^p)$ if and only if $x\in\R^n\setminus\{h_k=0\}$. Observe that if $0\leq i<j$, then $W_i\cap W_j=W_i$ and we define the following \em transition maps\em:
\begin{align*}\displaystyle
 g_{ij}:W_i\cap W_j=W_i\to\GL(\R^p),\  x\mapsto\tfrac{1}{\sqrt{(h_{j|W_i})(x)\cdots(h_{{i+1}|W_i})(x)}}M_j(x)\cdots M_{i+1}(x),\\
g_{ji}:W_i\cap W_j=W_i\to\GL(\R^p),\ x\mapsto\tfrac{1}{\sqrt{(h_{j|W_i})(x)\cdots(h_{{i+1}|W_i})(x)}}M_{i+1}^t(x)\cdots M_j^t(x).
\end{align*}
We define $g_{ii}:W_i\cap W_i=W_i\to\GL(\R^p),\ x\mapsto I_p$. Using the fact that $M_k^tM_k=h_kI_p$ for all $k\geq 1$, we have $g_{jk}(x)\cdot g_{ij}(x)=g_{ik}(x)$ for all $0\leq i,j,k$ and all $x\in W_i\cap W_j\cap W_k=W_{\min\{i,j,k\}}$. The functions $g_{ij}:W_i\cap W_j\to\GL(\R^p)$ are the transition maps of an analytic vector bundle $(E,\pi,\R^n)$ with fiber $\R^p$. For each $i\geq 0$ let $\phi_i:\pi^{-1}(W_i)\to W_i\times\R^p$ be homeomorphisms such that $(\phi_j\circ\phi_i^{-1})(z,v)=(z,g_{ij}(z)(v))$ for all $(z,v)\in(W_i\cap W_j)\times\R^p$ and all $j\geq0$. Moreover, for each $x\in W_i$ and each $j\geq0$, the matrices $g_{ij}(x)\in\GL(\R^p)$ are in fact ortogonal matrices, that is, $g_{ij}(x)g_{ij}(x)^t=I_p$.

\vspace{1mm}
\noindent{\sc Step 3: Construction of the Riemaniann metric.}
Consider the analytic Riemannian metric on $(E\oplus E,\pi',\R^n)$ given by
$$
\qq{\varphi_i^{-1}(x,(v_1,\ldots,v_p)),\varphi_i^{-1}(x,(w_1,\ldots,w_p))}=\sum_{\ell=1}^pv_\ell w_\ell,
$$
 Since each matrix $g_{ij}(x)\in\GL(\R^p)$ is orthogonal, we deduce that $\qq{\cdot,\cdot}$ is well-defined.

\vspace{1mm}

\noindent{\sc Step 4: Construction of the analytic section.}

Let $u\in\R^p$ be any unitary vector and let $s:\R^n\to E$ be the global analytic section satisfying
$$
\varphi_0\circ s_{|W_0}:W_0\to W_0\times\R^p\subset W_i\times\R^p,\ x\mapsto(x,\sqrt{(h_{|W_0})(x)}u).
$$
Let us check that such section exists. Indeed, for each $i\geq1$, we have
$$\displaystyle
(\varphi_i\circ s_{|W_0})=(\varphi_i\circ\varphi_0^{-1})\circ(\varphi_0\circ s_{|W_0}):W_0\to W_0\times\R^p,\ x\mapsto(x,\sqrt{(h_{|W_0})(x)}g_{0i}(x)u)
$$
where 
$$\displaystyle
\sqrt{(h_{|W_0})(x)}g_{0i}(x)(u)=\sqrt{\tfrac{(h_{|W_0})(x)}{(h_{1|W_0})(x)\cdots(h_{i|W_0})(x)}}M_i(x)\cdots M_1(x)(u).\qquad(\star)\label{*}
$$
Moreover, in $W_i=\R^n\setminus\bigcup_{j>i}X_j$ the function $\frac{h|_{W_i}}{(h_1|_{W_i})\cdots(h_i|_{W_i})}$ is analytic and vanishes nowhere. Thus, it admits an analytic square root in $W_i$. Since $\R^n=\bigcup_{i\geq0}W_i$, we have, by means of equality \eqref{*} above, an analytic section $s:\R^n\to E$ of $(E,\pi,\R^n)$ such that
$$
\phi_i\circ s:W_i\mapsto W_i\times\R^p,\ x\mapsto\Big(x,\sqrt{\tfrac{(h|_{W_i})(x)}{(h_1|_{W_i})(x)\cdots(h_i|_{W_i})(x)}}M_i(x)\cdots M_1(x)(u)\Big).
$$
Moreover, the analytic function $\qq{s,s}:\R^n\to\R$ is determined by its value on the points $x\in W_0$, so
$$
\qq{s,s}_{|W_0}(x)=\qq{\sqrt{(h_{|W_0})(x)}u,\sqrt{(h_{|W_0})(x)}u}=(h_{|W_0})(x)\qq{u,u}=(h_{|W_0})(x).
$$
Consequently  $\qq{s,s}=h$, as equired. 
\end{proof}

\begin{remark}\label{main1r}
By means of Remark \ref{matrixr}(ii), one can adapt the previous proof to show that: \em If each $f_k$ in the statement of Theorem \em \ref{main1} \em is a sum of $2^r$ squares in $\an(\R^n)$ with $r=0,1,2,3$, then also $f$ is a sum of $2^r$ squares in $\an(\R^n)$\em. 
\end{remark}

 Since positive semidefinite Nash functions are sums of squares with controlled denominators, we get an improvement of  Theorem \ref{nash}, that is, 
the representation as sum of squares of the involved sheaf-product $f$ has controlled denominators. 

\begin{proof}[Proof of Theorem \ref{nash}]
First, let $W_k$ be an open neighborhood of $\{f_k=0\}$ in $\R^n$ and let $\varphi_k:M_k\to W_k$ be an analytic diffeomorphism such that $M_k$ is a Nash manifold and $f_k\circ\varphi_k=u_kg_k$, where $g_k:M_k\to\R$ is a positive semidefinite Nash function and $u_k\in\an(M_k)$ is a positive unit. 

 Each Nash function $g_k$ is a sum of $2^n$ squares with controlled denominators in the quotient field of the ring ${\mathcal N}(M_k)$ of Nash functions on $M_k$. Thus, each $f_k$ is a sum of $2^n$ squares at $\{f_k=0\}$ with controlled denominators. Consequently, $f$ is by Theorem \ref{main1} a sum of $2^{n+1}$ squares of meromorphic functions with controlled denominators.
\end{proof}

\subsection{Unbounded analytic Pfister's formula.}\label{s4}

In this section we prove the first part of Theorem \ref{vmain}. More precisely we  prove the following.

\begin{thm}[Unbounded analytic Pfister's formula]\label{main2}
Let $\{f_k\}_{k\geq1}$ be a collection of positive semidefinite analytic functions on $\R^n$ such that the family $\{X_k\}_{k\geq1}$ of their zerosets is locally finite and each $f_k$ is a finite sum of squares at $X_k$. Then, $\shprod f_k$ is an infinite sum of squares (with controlled denominators).
\end{thm}

We have shown in Lemma \ref{pfister} that if $f\in\an(\R^n)$ is a finite sum of squares, there exist $g\in\an(\R^n)\setminus\{0\}$ and a Pfister's bundle $\psd(g^2f)$ of finite rank $p$ for $g^2f$. In fact, $p$ is an upper bound for the number of squares needed to represent $f$ as a sum of squares. Thus, to approach the unbounded situation presented in Theorem \ref{main2}, it seems convenient to work with a bundle of infinite rank. An initial candidate for the fiber could be the Hilbert space $\ell^2(\R)$ of real square-summable sequences. However, since we also need to take care about the uniform convergence of infinite sums of squares, we must work in a complex open Stein neighborhood $\Omega$ of $\R^n$ in $\C^n$ and we must build up a vector bundle with fiber the complex Hilbert space $\ell^2(\C)$ of complex square-summable sequences. We begin by fixing and recalling some notation and terminology from Hilbert spaces and vector bundles with fiber a Hilbert space.

\subsubsection{Hilbert spaces, holomorphic functions and vector bundles.}\label{hshf}

The Hilbert space $\ell^2(\C)$ is the vector space
$$
\ell^2(\C)=\Big\{\mbox{\bf w} =(w_k)_{k\geq1}:\ \sum_{k\geq1}|w_k|^2<+\infty\Big\}
$$
endowed with the usual (hermitian) inner product given by the formula
$$
\qq{\mbox{\bf w},\mbox{\bf v}}=\sum_{k\geq1}w_k\ol{v_k}.
$$

 Denote by $\Env =\{\mbox{\bf e}_m=(\delta_{mk})_{k\geq1}\}_{m\geq1}$ the \em standard 
\rm orthonormal system of $\ell^2(\C)$ and by $\bounded(\ell^2(\C))$ the $\C$-linear space of bounded (linear) operators $\ell^2(\C)\to\ell^2(\C)$. The set of those bounded operators which are invertible and whose inverse is also bounded will be denoted by $\bounded_*(\ell^2(\C))$. Given a bounded operator $T\in\bounded(\ell^2(\C))$,  its \em adjoint \rm is defined as the unique bounded operator $T^*\in\bounded(\ell^2(\C))$ such that $\qq{\mbox{\bf v},T(\mbox{\bf w})}=\qq{T^*(\mbox{\bf v}),\mbox{\bf w}}$ for all $\mbox{\bf v},\mbox{\bf w}\in\ell^2(\C)$.

Let $\Omega$ be an open subset of $\C^n$. From the general theory of holomorphic functions with values in a Banach space, one deduces easily  that a function $F:\Omega\to\ell^2(\C)$ is \em holomorphic \rm if and only if $F$ belongs to the $\hol(\Omega)$-module
$$
\Big\{F=(F_k)_{k\geq1}:\ F_k\in\hol(\Omega),\ \sum_{k\geq1}\sup_K|F_k|^2<+\infty\ \forall K\subset\Omega\text{ compact}\Big\}.
$$
Observe that if $G,H:\Omega\to\ell^2(\C)$ are holomorphic, then $\qq{G,H\circ\sigma}:\Omega\to\C$ defines a  holomorphic function on $\Omega$. Moreover  if $\Omega$ is invariant and $G=H$ this product is invariant and  an infinite sum of squares.  

\begin{prop}\label{tm1}
Let $M=(m_{ij})\in{\mathfrak M}_p(\C)$ be a matrix and consider the linear operator $T_{M}:\ell^2(\C)\to\ell^2(\C)$ whose matrix with respect to the standard basis $\stb$ of $\ell^2(\C)$ is the infinite matrix
$$
\left(\begin{array}{c|c|c|c}
M&0&0&\cdots\\
\hline
0&M&0&\cdots\\
\hline
0&0&M&\cdots\\
\hline
\vdots&\vdots&\vdots&\ddots
\end{array}\right).
$$
Then, 
\begin{itemize}
\item[(i)] $T_M$ is a bounded operator of $\ell^2(\C)$.
\item[(ii)] $T_{M_1\cdot M_2}=T_{M_1}\circ T_{M_2}$ for all $M_1,M_2\in{\mathfrak M}_p(\C)$.
\item[(iii)] $T_{I_p}=\id_{\ell^2(\C)}$, where $I_p$ denotes the identity matrix of ${\mathfrak M}_p(\C)$.
\item[(iv)] If $M$ is invertible, then $T_M$ is an invertible bounded operator whose inverse is the bounded operator $T_{M^{-1}}$.
\item[(v)] $T_M^*=T_{M^*}$, where $M^*$ denotes the transpose conjugated of $M$.
\end{itemize}
\end{prop}

\begin{proof}
The linear operator $T_{M}$ is given by the formula 
$$\mbox{\bf v}=(v_k)_{k\geq1}\mapsto\sum_{k\geq1}v_kT_M(\mbox{\bf e}_k),$$ 
where
$$
T_M(\mbox{\bf e}_k)=\sum_{i=1}^p m_{ij}\mbox{\bf e}_{pq+i}\quad\text{if $k=pq+j$ with $q\geq 0$ and $1\leq j\leq p$}.
$$
Notice that $T_M(\mbox{\bf v})=(u_k)_{k\geq1}$ is a well defined sequence of complex numbers because $u_{pq+i}=\sum_{j=1}^p m_{ij}v_{pq+j}$ for all $q\geq 0, 1\leq i\leq p$.

Only the fact that $T_M(\ell^2(\C))\subset\ell^2(\C)$ and the boundness of $T_M$ require some comment. The other statements (ii)-(v) are straightforward. From the inequality

$$
\qq{T_M(\mbox{\bf v}),T_M(\mbox{\bf v})}\leq p^3\max_{1\leq i,j\leq p}\{|m_{ij}|^2\}\qq{\mbox{\bf v},\mbox{\bf v}}\qquad\forall\,\mbox{\bf v}\in\ell^2(\C),\qquad(\ast) 
$$\label{2*}
follows that $T_M(\mbox{\bf v})\in\ell^2(\C)$ for all $\mbox{\bf v}\in\ell^2(\C)$ and that $T_M$ is a bounded operator. Thus, we proceed with the proof of \eqref{2*}.

Indeed, denote $m=\max_{1\leq i,j\leq p}\{|m_{ij}|^2\}$ and observe first that if $1\leq r,s\leq p$ and $q\geq 0$,
$$
|\qq{T_M(\mbox{\bf e}_{pq+r}),T_M(\mbox{\bf e}_{pq+s})}|=|\qq{T_M(\mbox{\bf e}_r),T_M(\mbox{\bf e}_s)}|=\Big|\sum_{i=1}^pm_{ir}m_{is}\Big|\leq\sum_{i=1}^p|m_{ir}m_{is}|\leq pm.
$$

Consider the sequence of positive real numbers $v^*=(|v_k|)_{k\geq 1}$ and $w_s=(w_{sk})_{k\geq 1}$ for $1\leq s\leq p$, where $w_{sk}=|v_{pj+s}|$ if $k=pj+r$ with $1\leq r\leq p$ and $j\geq 0$. We get
$$\qq{v^*,v^*}= \sum_{k\geq 1} |v_k|^2= \qq{v,v} \quad \mbox{\rm  and }\quad \qq{w_s,w_s}= p \sum_{j\geq 1} |v_{pj+s}|^2\leq p \qq{v,v}$$.

Hence, in particular $\mbox{\bf v}^*,\mbox{\bf w}_s\in\ell^2(\C)$. Moreover
$$
\qq{\mbox{\bf v}^*,\mbox{\bf w}_s}=\sum_{k\geq1}v_k\ol{w}_{sk}=\sum_{j\geq0}\sum_{r=1}^p|v_{pj+r}||w_{s,pj+r}|=\sum_{j\geq0}\sum_{r=1}^p|v_{pj+r}||v_{pj+s}|.
$$
Next, observe that if $\ell_1=pq_1+r_1$ and $\ell_2=pq_2+r_2$ where $1\leq r_1,r_2\leq p$ and $q_1\neq q_2$, then $\qq{T_M(\mbox{\bf  e}_{\ell_1}),T_M(\mbox{\bf e}_{\ell_2})}=0$. Thus, for each $q\geq1$, we have (using Schwarz's inequality)

\begin{multline*}
0\leq\Big|\qq{\sum_{k=1}^{pq}v_kT_M(\mbox{\bf  e}_k),\sum_{\ell=1}^{pq}v_\ell T_M(\mbox{\bf  e}_\ell)}\Big|=\Big|\sum_{k=1}^{pq}\sum_{\ell=1}^{pq}v_k\ol{v}_\ell \qq{T_M(\mbox{\bf  e}_k),T_M(\mbox{\bf  e}_\ell)}\Big|\\
=\Big|\sum_{j=0}^q\sum_{r=1}^p\sum_{s=1}^pv_{pj+r}\ol{v}_{pj+s}\qq{T_M(\mbox{\bf  e}_r),T_M(\mbox{\bf e}_s)}\Big|\leq\sum_{j=0}^q\sum_{r=1}^p\sum_{s=1}^p|v_{pj+r}||\ol{v}_{pj+s}|pm\\
=pm\sum_{s=1}^p\Big(\sum_{j=0}^q\sum_{r=1}^p|v_{pj+r}||v_{pj+s}|\Big)\leq pm\sum_{s=1}^p\qq{\mbox{\bf v}^*,\mbox{\bf w}_s}\\
\leq pm\sum_{s=1}^p\sqrt{\qq{\mbox{\bf v}^*,\mbox{\bf v}^*}}\sqrt{\qq{\mbox{\bf w}_s,\mbox{\bf w}_s}}\leq p^3m\qq{\mbox{\bf v},\mbox{\bf v}}.
\end{multline*}
Since the previous inequality holds for all $q\geq 1$, we have $\qq{T_M(\mbox{\bf  v}),T_M(\mbox{\bf  v})}\leq p^3m\qq{\mbox{\bf v},\mbox{\bf v}}$,as required.
\end{proof}

\begin{defn}\label{tm2}
We call $T_M$ the {\em bounded operator of $\ell^2(\C)$ associated to $M$}.
\end{defn}

\begin{remark}\label{tm}
Let $\Omega\subset\C^n$ be an invariant open set and  $M=(m_{ij})\in{\mathfrak M}_p(\hol(\Omega))$ be an invariant matrix such that $M(z)^tM(z)=I_p$ for all $z\in\Omega$. Then, $M^t\circ\sigma=\sigma\circ M^t$, so $T_{M(z)}^*=T_{M^t\circ\sigma(z)}$ and $T_{M\circ\sigma(z)}^*\circ T_{M(z)}=T_{M(z)}\circ T_{M\circ\sigma(z)}^*=\id$ for each $z\in\Omega$.
\end{remark}

\subsubsection{Pfister's multiplicative formulae.}
After  the   introduction  of  the  bounded   operator  of  $\ell^2(\C)$
associated  to  a   complex  square  matrix,  we   associate  to  each
invariant    sum    of    squares   of    $\hol(\Omega)$    an
invariant Pfister matrix in ${\mathfrak M}_{2^{r+1}}(\hol(\Omega))$ that is almost unitary. This result can
be understood as the counterpart  of Lemma \ref{matrix} in the setting
of this section. More precisely,

\begin{lem}[Pfister's complex matrix]\label{matrix2}
Let $\Omega\subset\C^n$ be an invariant open set and let $B\in\hol(\Omega)$ be invariant such that its zeroset does not meet the compact set $K\subset\Omega$. Let $A=\nsq{2^r}$ in $\hol(\Omega)$ be invariant. Then, there exists a  real constant $c>0$ such that for each $d\geq c$, there exist a $\sigma$-invariant $G\in\hol(\Omega)$ satisfying $\{G^2=0\}\cap K=\varnothing$ and an invariant matrix $M\in{\mathfrak M}_{2^{r+1}}(\hol(\Omega))$ such that $M^tM=G^2((dB)^2+A)I_{2^{r+1}}$.
\end{lem}

\begin{proof}
If $r=0$, then $A=A_0^2$ and taking $\displaystyle c=1+\frac{\max_K(|A|)}{\min_K(|B|^2)}$, $G=1$ and
$$
M=\left(
\begin{array}{cc}
dB&-A_0\\
A_0&dB
\end{array}
\right)\in{\mathfrak M}_2(\hol(\Omega))
$$
we are done.

Assume $r\geq 1$, and let $A_1=\nsq{2^{r-1}}$ and $A_2=\nsq{2^{r-1}}$ such that $A=A_1+A_2$. By induction there are positive real constants $c_i>0$ such that for each $d_i\geq c_i$, there are invariant $G_i\in\hol(\Omega)$ satisfying $\{G_i^2((d_iB)^2+A_i)=0\}\cap K=\varnothing$ and an invariant matrix $M_i\in{\mathfrak M}_{2^{r}}(\hol(\Omega))$ such that $M_i^tM_i=G_i^2((d_iB)^2+A_i)I_{2^{r}}$. Take $c=\max\{c_0,\sqrt{2}c_1,\sqrt{2}c_2\}$ where $\displaystyle c_0=1+\frac{\max_K(|A|)}{\min_K(|B|^2)}$, and for each $d\geq c$ consider the invariant functions $G_i\in\hol(\Omega)$ and the invariant matrices $M_i\in{\mathfrak M}_{2^{r}}(\hol(\Omega))$ corresponding to the constants $\displaystyle d_i=\frac{d}{\sqrt{2}}\geq c_i$. Next, take $G=G_1G_2^2F_2$, where $F_2=(d_2B)^2+A_2$, and 
$$
M=
\left(\begin{array}{cc}
G_2^2F_2M_1&-G_1G_2F_2M_2\\
G_1G_2F_2M_2&M_2M_1^tM_2
\end{array}\right)
$$
and notice that $\{G^2((dB)^2+A)=0\}\cap K=\varnothing$. A straighforward computation shows that these are the $M$ and $G$ we sought.
\end{proof}

The following result assures the existence of invariant open Stein neighborhoods of $\R^n$ in $\C^n$ that admit open exhaustions, whose members (not necessarily contractible) have a good behaviour with respect to the square root operator. Namely,

\begin{lem}\label{reliable}
Let $\Omega$ be an invariant open neighborhood of $\R^n$ in $\C^n$ and let $\{\Delta_i\}_{i\geq1}$ be an open covering of $\Omega$ such that $\Delta_i\subset\Delta_{i+1}$ for all $i\geq1$. Then, there exist an invariant contractible open Stein neighborhood $\Omega'\subset\Omega$ of $\R^n$ in $\C^n$ and an open covering $\{\Omega_i\}_{i\geq1}$ of $\Omega'$ such that 
\begin{itemize}
\item $\Omega_i\subset\Omega_{i+1}$ and $W_i=\R^n\cap\Delta_i\subset\Omega_i\subset\Delta_i$ for each $i\geq1$.
\item Each invariant holomorphic unit $u\in\hol(\Omega_i)$ admits an invariant holomorphic square root on $\Omega_i$.
\end{itemize}
\end{lem}

\begin{proof}
Denote $W_0=\varnothing$ and observe that the sets $\{W_i\setminus W_{i-1}\}_{i\geq1}$ constitute a partition of $\R^n$. For each $x=(x_1,\ldots,x_n)\in W_i\setminus W_{i-1}$ choose an invariant open polydisc 
$$
B_x=\{z=(z_1,\ldots,z_n)\in\C^n:\ |z_i-x_i|<\veps_x,\ 1\leq i\leq n\}\subset \Delta_i
$$ 
centered at $x$ and of radius $\veps_x>0$. For each $i\geq1$ define $\Omega_i'=\bigcup_{x\in W_i}B_x$, which is an invariant open neighborhood of $W_i$ in $\Delta_i$. Of course, since $W_i\subset W_{i+1}$ for $i\geq1$, also $\Omega_i'\subset\Omega_{i+1}'$. Moreover, $\R^n=\bigcup_{i\geq1}W_i\subset\bigcup_{i\geq1}\Omega_i'\subset\bigcup_{i\geq1}\Delta_i=\Omega$.

Let $u\in\hol(\Omega_i')$ be an invariant holomorphic unit. For each $x\in W_i$ the restriction $u_{|B_x}$ admits an invariant holomorphic square root $v_x\in\hol(B_x)$ whose restriction to $B_x\cap\R^n$ is strictly positive. By the Identity Principle the functions $v_x$ glue properly to provide an invariant (no-where vanishing) holomorphic function on $\Omega_i'$
$$
v:\Omega_i'\to\C,\ z\mapsto v_x(z)\quad\text{if $z\in B_x$,}
$$
satisfying $v^2=u$. Thus, $u$ admits an invariant holomorphic square root in $\Omega_i'$. 

Finally, we choose an invariant contractible open Stein neighborhood $\Omega'\subset\bigcup_{i\geq1}\Omega_i'$ of $\R^n$ in $\C^n$. A straightforward computation shows that the invariant open sets $\Omega_i=\Omega_i'\cap\Omega'$ are the ones we sought.
\end{proof}

The proof of Theorem \ref{main2}
runs rather similar to the one of Theorem \ref{main1}, but it is quite more involving and  more technical, so  we provide  all  its  details.

\begin{proof}[Proof of Theorem \ref{main2}]
The proof is conducted in several steps.

\vspace{1mm}
\noindent{\sc Step 1.} 
We may assume that $p_k=2^{r_k}$ and we write $q_k=2^{n+1}p_k$. By Lemma \ref{extension} there exist an invariant open neighborhood $\Omega$ of $\R^n$ in $\C^n$ and invariant homolomorphic functions $G_k,\widehat{F}_k,A_{kj}:\Omega\to\C$ for $k\geq1$ and $1\leq j\leq q_k$ such that: 
\begin{itemize}
\item[(i)] $\widehat{F}_{k|\R^n}=u_kf_k$ for a positive unit $u_k\in\an(\R^n)$.
\item[(ii)] The family $\{G_k^2\widehat{F}_k=0\}_{k\geq1}$ is locally finite in $\Omega$ and $\{G_k^2\widehat{F}_k=0\}\cap\R^n=X_k$.
\item[(iii)] For each family of values $\{d_k\}_{k\geq1}$ satisfying $d_k\geq c_k$, the finite  sums $A_k=(d_kG_k^2\widehat{F}_k)^2+\sum_{j=1}^{q_k}A_{kj}^2$ satisfy
\begin{itemize} 
\item $\{A_{k|\R^n}=0\}=X_k$ and the family $\{A_k=0\}_{k\geq1}$ is locally finite in $\Omega$; 
\item $A_k=G_k^2\widehat{F}_kQ_k$ for some $\sigma$-invariant $Q_k\in\hol(\Omega)$, which does not vanish on a neighborhood $V_k$ of $\{G_k^2\widehat{F}_k=0\}\cup\R^n$.
\end{itemize}
\end{itemize}
Since the family $\{G_k^2\widehat{F}_k=0\}_{k\geq1}$ is locally finite, there exists by Lemma \ref{lfc} an exhaustion $\{K_k\}_{k\geq1}$ of $\Omega$ by compact sets such that $\{G_k^2\widehat{F}_k=0\}\cap K_k=\varnothing$ for all $k\geq 1$. By Lemma \ref{matrix2} applied to $B_k=G_k^2\widehat{F}_k$, the sum of squares $\sum_{j=1}^{q_k}A_{kj}^2$ and the compact set $K_k$, there exist a real number $d_k\geq c_k$, an invariant $G_k'\in\hol(\Omega)$ satisfying $\{G_k'^2=0\}\cap K_k=\varnothing$ and an invariant matrix $M_k\in{\mathfrak M}_{2q_k}(\hol(\Omega))$ such that $M_k^tM_k=G_k'^2A_kI_{2q_k}$, where $A_k=(d_kB_k)^2+\sum_{j=1}^{q_k}A_{kj}^2$. 

To simplify notations, we denote again by $G_k$ the function $G_kG_k'$ and by $A_k$ the function $G_k'^2A_k$. Observe that, by Lemma \ref{lfc}, the family $\{G_k=0\}_{k\geq1}$ is locally finite. Let $\widehat{F}=\shprod\widehat{F}_k$, $G=\shprod G_k$ and $A=\shprod A_k$ and denote $f=\widehat{F}_{|\R^n}$, $g=G_{|\R^n}$ and $a=A_{|\R^n}$. Observe that $f=\shprod f_k$ and in fact, by (iii) above, $g^2f=ua$ for a positive unit $u\in\an(\R^n)$. Thus, we are reduced to prove that $A$ is an infinite sum of squares of invariant holomorphic functions on $\Omega$ and this is done in the subsequent steps.

\vspace{1mm}
\noindent{\sc Step 2: Construction of a suitable vector bundle with fiber $\ell^2(\C)$.} 
For each $i\geq0$, denote by $\Delta_i$ the open subset $\Omega\setminus\bigcup_{k>i}\{A_k=0\}$ of $\C^n$. By Lemma \ref{reliable} there exist an invariant contractible open Stein neighborhood $\Omega'\subset\Omega$ of $\R^n$ in $\C^n$ and an open covering $\{\Omega_i\}_{i\geq1}$ of $\Omega'$ such that
\begin{itemize}
\item $\Omega_i\subset\Omega_{i+1}$ and $W_i=\R^n\cap\Delta_i\subset\Omega_i\subset\Delta_i$ for each $i\geq0$.
\item Each invariant holomorphic unit $u\in\hol(\Omega_i)$ admits an invariant holomorphic square root on $\Omega_i$.
\end{itemize}
In particular, the restrictions to $\Omega_i$ of the quotients (which vanish nowhere in $\Omega_i$)

$$
H_i=\frac{A}{A_1\cdots A_i}\quad\text{ and }\quad B_{ij}=\frac{1}{A_{i+1}\cdots A_j}
$$ 
admit invariant holomorphic square roots in $\Omega_i$ for $0\leq i<j$. To simplify notations, we denote $\Omega'$ again by $\Omega$.
 For each $k\geq0$ consider the product $\Omega_k\times\ell^2(\C)$. Since the invariant matrix $M_k\in{\mathfrak M}_{2q_k}(\hol(\Omega))$ constructed in the previous step satisfies $M_k^tM_k=A_kI_{2q_k}$, we deduce that $T_{M_k(z)}:\ell^2(\C)\to\ell^2(\C)$ is an isomorphism for each $z\in\Omega\setminus \{A_k=0\}$.

Observe that if $0\leq i<j$, then $\Omega_i\cap\Omega_j=\Omega_i$ and we define the following transition functions:
\begin{align*}
\gamma_{ij}:\Omega_i\cap\Omega_j=\Omega_i\to\bounded_*(\ell^2(\C)),\ &z\mapsto\sqrt{B_{ij|\Omega_i}(z)}\,T_{M_j(z)}\circ\cdots\circ T_{M_{i+1}(z)},\\
\gamma_{ji}:\Omega_i\cap\Omega_j=\Omega_i\to\bounded_*(\ell^2(\C)),\ &z\mapsto\sqrt{B_{ij|\Omega_i}(z)}\,T^*_{M_{i+1}\circ\sigma(z)}\circ\cdots\circ T^*_{M_j\circ\sigma(z)}.
\end{align*}
Obviously, we define $\gamma_{ii}:\Omega_i\cap\Omega_i=\Omega_i\to\bounded_*(\ell^2(\C)),\ z\mapsto\id$. Using the fact that
$$
T_{M_k\circ\sigma(z)}^*\circ T_{M_k(z)}=T_{M_k(z)}\circ T_{M_k\circ\sigma(z)}^*=A_k(z)\id
$$
for all $k\geq 1$ and all $z\in\Omega$, one checks that $\gamma_{jk}(z)\cdot\gamma_{ij}(z)=\gamma_{ik}(z)$ for all $0\leq i,j,k$ and all $z\in\Omega_i\cap\Omega_j\cap\Omega_k=\Omega_{\min\{i,j,k\}}$. Thus, the holomorphic functions $\gamma_{ij}:\Omega_i\cap\Omega_j\to\bounded_*(\ell^2(\C))$ are the transition functions of a holomorphic Hilbert vector bundle $\mathscr{E}=(E,\pi,\Omega)$ with fiber $\ell^2(\C)$. For each $i\geq 0$ let $\psi_i:\pi^{-1}(\Omega_i)\to\Omega_i\times\ell^2(\C)$ be a homeomorphism such that $\psi_j\circ\psi_i^{-1}(z,v)=(z,\gamma_{ij}(z)(v))$ for all $(z,v)\in\Omega_{\min\{i,j\}}\times\ell^2(\C)$ and all $i,j\geq0$. Observe that for each $z\in\Omega_i$ and each $j\geq0$ the bounded operators $\gamma_{ij}(z)\in\bounded_*(\ell^2(\C))$ satisfy the equality 
$$
\gamma_{ij}(z)\cdot\gamma_{ij}^*(\sigma(z))=\gamma_{ij}^*(\sigma(z))\cdot\gamma_{ij}(z)=\id.\qquad(\ast)
$$\label{3*}

\vspace{1mm}
\noindent{\sc Step 3: Construction of a suitable Hermitian structure.} 
Next, consider the Whitney sum $\mathscr{E}\oplus\mathscr{E}^\sigma=(E\oplus E^\sigma,\pi',\Omega)$ of $\mathscr{E}=(E,\pi,\Omega)$ with $\mathscr{E}^\sigma=(E^\sigma,\pi,\Omega)$, where this last vector bundle corresponds to the antiholomorphic structure of $\mathscr{E}$. Recall that given $0\leq i,j$ the transition function of $\mathscr{E}\oplus\mathscr{E}^\sigma$ corresponding to the intersection $\Omega_i\cap\Omega_j$ is
$$
\Gamma_{ij}:\Omega_i\cap\Omega_j\to\bounded_*(\ell^2(\C)\oplus\ell^2(\C)),\ z\mapsto
\gamma_{ij}(z)\oplus\gamma_{ij}(\sigma(z)).
$$
For each $i\geq 0$ let $\Psi_i:\pi'^{-1}(\Omega_i)\to\Omega_i\times(\ell^2(\C)\oplus\ell^2(\C))$ be a homeomorphisms such that $\Psi_j\circ\Psi_i^{-1}(z,v,w)=(z,\Gamma_{ij}(z)(v,w))$ for all $(z,v,w)\in\Omega_{\min\{i,j\}}\times(\ell^2(\C)\oplus\ell^2(\C))$ and all $i,j\geq1$. We define the Hermitian metric as follows. For each $i\geq0$ and each $z\in\Omega_i$, consider 
$$
\qq{\cdot,\cdot}\circ\Psi_{i}^{-1}:\Omega_i\times(\ell^2(\C)\times\ell^2(\C))\to\C,\ (z,v,w)\mapsto\qq{v,w}.
$$

Let us see  that $\qq{\cdot,\cdot}$ is well defined. To  that end, fix
$i,j\geq1$   and
$(z,v,w)\in\Omega_{\min\{i,j\}}\times(\ell^2(\C)\oplus\ell^2(\C))$ and
let us check the following equality
$$
\qq{\cdot,\cdot}\circ\Psi_{j}^{-1}(z,v,w)=\qq{\cdot,\cdot}\circ\Psi_{i}^{-1}(z,v,w).
$$
Indeed, using \eqref{3*} above we deduce
\begin{equation*}
\begin{split}
\qq{\cdot,\cdot}\circ\Psi_{j}^{-1}(z,v,w)&=\qq{\cdot,\cdot}\circ(\Psi_i^{-1}\circ\Psi_i)\circ\Psi_{j}^{-1}(z,v,w)=\qq{\cdot,\cdot}\circ\Psi_i^{-1}\circ\Gamma_{ij}(z,v,w)=\\&=\qq{\cdot,\cdot}\circ\Psi_i^{-1}(z,\gamma_{ij}(z)(v),(\gamma_{ij}\circ\sigma)(z)(w))=\\
&=\qq{\gamma_{ij}(z)(v),(\gamma_{ij}\circ\sigma)(z)(w)}=\qq{\gamma_{ij}^*(\sigma(z))\cdot\gamma_{ij}(z)(v),w}=\\
&=\qq{v,w}=\qq{\cdot,\cdot}\circ\Psi_{i}^{-1}(z,v,w).
\end{split}
\end{equation*}
Given holomorphic sections $\alpha,\beta:\Omega\to\mathscr{E}$, we define in the obvious way the section $\alpha\oplus (\beta\circ\sigma):\Omega\to\mathscr{E}\oplus\mathscr{E}^{\sigma}$, to which we can apply $\qq{\cdot,\cdot}:\mathscr{E}\oplus\mathscr{E}^{\sigma}\to\C$, and we get the holomorphic function $\qq{\alpha,\beta\circ\sigma}:\Omega\to\C$.

\vspace{1mm}
\noindent{\sc Step 4: Construction of suitable holomorphic sections.} 
Denote by $\stb=\{\mbox{\bf e}_\ell\}_{\ell\geq1}$ the standard basis of $\ell^2(\C)$. Let us construct holomorphic sections $s_\ell:\Omega\to\mathscr{E}$ such that
$$
\psi_0\circ s_{\ell_|\Omega_0}:\Omega_0\to\Omega_0\times\ell^2(\C)\subset\Omega_i\times\ell^2(\C),\ z\mapsto(z,\sqrt{A_{|\Omega_0}(z)}\mbox{\bf e}_\ell).
$$
Since $A$ vanishes nowhere in $\Omega_0$, the square root $\sqrt{A_{|\Omega_0}}$ is well-defined and holomorphic in $\Omega_0$. For each $i\geq1$ we have
$$
(\psi_i\circ s_{\ell|\Omega_0})=(\psi_i\circ\psi_0^{-1})\circ(\psi_0\circ s_{\ell|\Omega_0}):\Omega_0\to\Omega_0\times\ell^2(\C),\ z\mapsto(z,\sqrt{A_{|\Omega_0}(z)}\gamma_{0i}(z)(\mbox{\bf e}_\ell))
$$
where 
$$
\sqrt{A_{|\Omega_0}(z)}\gamma_{0i}(z)(\mbox{\bf e}_\ell)=\sqrt{H_{i|\Omega_0}(z)}(T_{M_i(z)}\circ\cdots\circ T_{M_1(z)})(\mbox{\bf e}_\ell).
$$
Moreover, the square roots $\sqrt{H_{i|\Omega_i}}:\Omega_i\to\C$ are well-defined and holomorphic on $\Omega_i$. Since $\Omega=\bigcup_{i\geq0}\Omega_i$, we define $s_\ell:\Omega\to\mathscr{E}$ such that
$$
\psi_i\circ s_\ell:\Omega_i\to\Omega_i\times\ell^2(\C),\ z\mapsto\Big(z,\sqrt{H_{i|\Omega_i}(z)}(T_{M_i(z)}\circ\cdots\circ T_{M_1(z)})(\mbox{\bf e}_\ell)\Big),
$$
which is a holomorphic section of $\mathscr{E}$. 

Consider the holomorphic functions $\qq{s_k,s_\ell\circ\sigma}:\Omega\to\C$. Notice that, by the Identity Principle, the holomorphic functions $\qq{s_k,s_\ell\circ\sigma}:\Omega\to\C$ are determined by their value on the points $z\in \Omega_0$. Thus, since
\begin{multline*}
\qq{s_k,s_\ell\circ\sigma}_{|\Omega_0}(z)=\qq{(\psi_0\circ s_k)(z),(\psi_0\circ s_\ell\circ\sigma)(z)}=\\
=\qq{\sqrt{A_{|\Omega_0}(z)}\mbox{\bf e}_k,\sqrt{A_{|\Omega_0}\circ\sigma(z)}\mbox{\bf e}_\ell}=(\sqrt{A_{|\Omega_0}(z))}^2\qq{\mbox{\bf e}_k,\mbox{\bf e}_\ell}=A_{|\Omega_0}(z)\delta_{k\ell},
\end{multline*}
we deduce that $\qq{s_k,s_\ell\circ\sigma}=A\delta_{k\ell}$ for $k,\ell\geq1$.

\vspace{1mm}
\noindent{\sc Step 5: Representation of $A$ as a sum of squares.} 
Since the fiber bundle is trivial on $\Omega$, there exists 
a nowhere-vanishing holomorphic section $\tau:\Omega\to E$. Thus, the restriction to $\R^n$ of the invariant holomorphic function $u=\qq{\tau,\tau\circ\sigma}:\Omega\to\C$ is strictly positive. Thus, shrinking $\Omega$ if necessary (but keeping all its properties) there exists a no-where vanishing invariant holomorphic function $v:\Omega\to\C$ such that $u=v^2$.

Next, let us check that on $\Omega^*=\Omega\setminus\{A=0\}$ the following equality holds
\begin{equation}\label{bullet}
A_{|\Omega^*}\tau_{|\Omega^*}=\sum_{\ell\geq1}\qq{\tau, {s_\ell\circ\sigma}_{|\Omega^*}{s_\ell}_{|\Omega^*}}\qquad(\bullet)
\end{equation}
Indeed, since $\Omega^*=\bigcup_{i\geq0}\Omega_i^*$, where $\Omega_i^*={\Omega^*\cap\Omega_i}$, it is enough to check that
$$
A_{|\Omega_i^*}\tau_{|\Omega_i^*}=\sum_{\ell\geq1}\qq{\tau,s_\ell\circ\sigma}_{|\Omega_i^*}{s_{\ell}}_{|\Omega_i^*}\qquad(\star)
$$\label{star}
for all $i\geq0$. Let $\eta_i:\Omega_i^*\to\ell^2(\C)$ be a nowhere vanishing holomorphic function such that $\psi_i\circ\tau_{|\Omega_i^*}(z)=(z,\eta_i(z))$ for each $z\in\Omega_i^*$. Consider the holomorphic function
$$
\beta_i:\Omega_i^*\to\ell^2(\C),\ z\mapsto\frac{B_{0i}(z)}{\sqrt{H_{i|\Omega_i}(z)}}(T_{M_1(z)}^*\circ\cdots\circ T_{M_i(z)}^*)(\eta_i(z))
$$
and write $\beta_i=\sum_{\ell\geq1}\beta_{i\ell}\mbox{\bf e}_\ell$, where $\beta_{i\ell}\in\hol(\Omega_i^*)$ and $\sum_{\ell\geq1}\sup_K|\beta_{i\ell}|^2<+\infty$ for all $K\subset\Omega_i^*$ compact. Then,

\begin{multline*}
\psi_i\circ\tau_{|\Omega_i^*}(z)=(z,\eta_i(z))=\Big(z,\sqrt{H_{i|\Omega_i}(z)}(T_{M_i(z)}\circ\cdots\circ T_{M_1(z)})(\beta_i(z))\Big)=\\
=\Big(z,\sum_{k\geq1}\beta_{ik}(z)\Big(\sqrt{H_{i|{\Omega_i}(z)} }(T_{M_i(z)}\circ\cdots\circ T_{M_1(z)})(\e_k)\Big)\Big)=\sum_{\ell\geq1}\beta_{i\ell}(z)(\psi_i\circ {s_\ell}_{|\Omega_i^*})(z).
\end{multline*}

Thus, $\tau_{|\Omega_i^*}=\sum_{\ell\geq1}\beta_{i\ell}(z)s_{\ell|\Omega_i^*}$. Now, using the fact that $\qq{s_k,s_\ell\circ\sigma}=A\delta_{k\ell}$ for $k,\ell\geq1$, we deduce that

$$
\qq{\tau_{|\Omega_i^*},s_k\circ\sigma_{|\Omega_i^*}}=\qq{\sum_{\ell\geq1}\beta_{i\ell}(z)s_{\ell|\Omega_i^*},s_k\circ\sigma_{|\Omega_i^*}}=\beta_{ik}\qq{s_{k|\Omega_i^*},s_k\circ\sigma_{|\Omega_i^*}}=\beta_{ik}A_{|\Omega_i^*},
$$

from which follows equality \eqref{star} and \eqref{bullet}. 
Extending formula \eqref{bullet} by zero, we deduce that $A\tau=\sum_{\ell\geq1}\qq{\tau,s_\ell\circ\sigma}s_\ell$. Consequently,
\begin{multline*}
Av^2=A\qq{\tau,\tau\circ\sigma}=\qq{A\tau,\tau\circ\sigma}=\sum_{\ell\geq1}\qq{\tau,s_\ell\circ\sigma} \ \qq{s_\ell,\tau\circ\sigma}=\\
=\sum_{\ell\geq1}\qq{\tau,s_\ell\circ\sigma}(\ol{\qq{\tau,s_\ell\circ\sigma}}\circ\sigma).
\end{multline*}
Since each addend $\qq{\tau,s_\ell\circ\sigma}(\ol{\qq{\tau,s_\ell\circ\sigma}}\circ\sigma)$ is a sum of two squares of invariant holomorphic functions on $\Omega$, we conclude that $A$ is an infinite sum of squares of invariant holomorphic functions on $\Omega$, as required.
\end{proof}

\subsection{Sums of countably many squares.}\label{s5}

In this section we prove the following result, from which Theorem \ref{vmain}(iii) follows straightforwardly. 

\begin{thm}\label{main3}
Let $\{f_k\}_{k\geq1}$ be a collection of positive semidefinite analytic functions on $\R^n$ such that the family $X_k=\{f_k=0\}$ of their zerosets is locally finite in $\R^n$ and each $f_k$ is an infinite sum of squares at $X_k$. Let $\{Y_i\}_{i\geq1}$ be the collection of the global irreducible components  of the global analytic set $X=\bigcup_{k\geq1}X_k$ and for each $i\geq1$ define 
$$
{\mathfrak F}_i=\{k\geq1:\ Y_i\cap X_k\neq\varnothing\}\quad\text{ and }\quad I=\{i\geq1:\ \#{\mathfrak F}_i=+\infty\}. 
$$
Then, there exist a positive semidefinite analytic function $q$ on $\R^n$ such that
\begin{itemize}
\item[(i)] $\{q=0\}\subset\bigcup_{j,\ell\in I,j\neq\ell}(Y_j\cap Y_\ell)$.
\item[(ii)] $q \ \shprod f_k$ is an infinite sum of squares with controlled denominators.
\end{itemize}
\end{thm}
\begin{remarks}\label{main3r}
(i) If each irreducible component $Y_i$ is compact, then $I=\varnothing$ and $q$ is just a positive unit of $\an(\R^n)$. Hence, $f$ is an infinite sum of squares of analytic functions on $\R^n$. See Example \ref{examples} (1), for an example of a function whose zeroset has countably many global irreducible components and all of them are compact.

(ii) If for each $j,\ell\in I$ with $j\neq\ell$ the intersection $Y_j\cap Y_\ell$ is a discrete set, then $f$ is an infinite sum of squares on $\R^n$. This follows from the fact that $\{q=0\}$ is contained in a discrete set, so by Section  \ref{Milnor} $q$ is a finite sum of squares. 

(iii) If $n=3$, then $\dim Y_i=1$ for all $i\geq1$. Thus, each intersection $Y_j\cap Y_\ell$ with $j\neq\ell$ is a discrete set,  so by (ii) $f$ is an infinite sum of squares. 

(iv) Let $f:\R^n\to\R$ be a positive semidefinite analytic function and suppose that there exist non-identically zero positive semidefinite analytic functions $q_1,q_2\in\an(\R^n)$ such that $h_1=fq_1$ and $h_2=fq_2$ are infinite sum of squares. Then, $q_1$ is an infinite sum of squares if and only if so is $q_2$. This, follows straightforwardly from the fact that $q_1q_2f^2=h_1h_2$ is an infinite sum of squares.
\end{remarks}
\vspace{1mm}
\begin{proof}[Proof  of Theorem \ref{main3}]

\noindent For our purposes it is enough to construct analytic functions $g,h,q$, which are  restrictions to $\R^n$ of $\sigma$-invariant holomorphic functions on a suitable invariant open neighborhood $\Omega$ of $\R^n$ in $\C^n$, satisfying the following conditions:
\begin{itemize}
\item $\{g=0\}\subset X$.
\item $h$ is an infinite sum of squares  such that $\{h=0\}\subset X$.
\item $q$ is positive semidefinite and $\{q=0\}\subset\bigcup_{j,\ell\in I,j\neq\ell}(Y_j\cap Y_\ell)$.
\item $g^2\ \shprod f_k=qh$.
\end{itemize}
The proof is conducted in several steps.

\vspace{1mm}
\noindent{\sc Step 1.} 
By Lemma \ref{extension} there are an invariant open neighborhood $\Omega$ of $\R^n$ in $\C^n$, positive real constants $c_k$ for $k\geq1$ and invariant homolomorphic functions $G_k,\widehat{F}_k,A_{kj}\in\hol(\Omega)$ for $k,j\geq1$ such that: 
\begin{itemize}
\item $\widehat{F}_{k|\R^n}=u_kf_k$ for a positive unit $u_k\in\an(\R^n)$.
\item The family $\{G_k^2\widehat{F}_k=0\}_{k\geq1}$ is locally finite in $\Omega$ and $\{G_k^2\widehat{F}_k=0\}\cap\R^n=X_k$.
\item The sum $A_k=(c_kG_k^2\widehat{F}_k)^2+\sum_{j=1}^{p_k}A_{kj}^2$ converges uniformely on the compact subsets of $\Omega$ . 
\item $\{A_{k|\R^n}=0\}=X_k$ and the family $\{A_k=0\}_{k\geq1}$ is locally finite in $\Omega$. 
\item $A_k=G_k^2\widehat{F}_kQ_k$ for some $\sigma$-invariant $Q_k\in\hol(\Omega)$, which does not vanish on a neighborhood $V_k$ of $\{G_k^2\widehat{F}_k=0\}\cup\R^n$.
\end{itemize}

Denote $\widehat{F}=\shprod\widehat{F}_k$, $G_0=\shprod G_k$, $\widehat{f}=\widehat{F}_{|\R^n}$ and $g_0=G_{0 |\R^n}$. Since each $f_k$ is positive semidefinite so is each $\widehat{F}_{k|\R^n}$ and we may assume that $\widehat{f}$ is positive semidefinite. In fact, since $\widehat{F}_{k|\R^n}=u_kf_k$ for each $k\geq1$, there exists a positive unit $u\in\an(\R^n)$ such that $\widehat{f}=uf$. Observe also that $\{g_0=0\}\subset\bigcup_{k\geq1}X_k=X$.

For each $i\geq1$, we define ${\mathfrak G}_i=\{k\geq1:\ Y_i\subset X_k\}$ and notice that ${\mathfrak G}_i$ is a finite set for each $i\geq1$. For each $i\geq1$, consider
$$
{\mathfrak H}_i=\left\{\begin{array}{ll}
{\mathfrak F}_i&\text{ if ${\mathfrak F}_i$ is finite},\\[4pt]
{\mathfrak G}_i&\text{ otherwise},
\end{array}\right.
$$
and define $P_i=\prod_{k\in{\mathfrak H}_i}G_k^2\widehat{F}_k$ and $B_i=\prod_{k\in{\mathfrak H}_i}A_k$, which is a convergent sum of squares $\sum_{\ell\geq1}B_{i\ell}^2$ on $\Omega$. Each function $P_i$ divides $G_0^2\widehat{F}$ in $\hol(\Omega)$ for all $i\geq1$. Moreover, $B_i=P_i\prod_{k\in{\mathfrak H}_i}Q_k=P_iD_i$ and the restriction of $D_i=\prod_{k\in{\mathfrak H}_i}Q_k$ to $\R^n$ is strictly positive on $\R^n$. Thus, $B_{i|\R^n}$ is $P_{i|\R^n}$ times a positive unit of $\an(\R^n)$.

\vspace{1mm}
\noindent{\sc Step 2.}
In this step we construct the sum of squares $h$ announced above. 
 Fix an exhaustion of $\Omega$ by compact sets $\{K_i\}_{i\geq1}$, that is, $K_1\neq\varnothing$, $K_i\subset\Int_{\C^n}(K_{i+1})$ and $\bigcup_{i\geq1}K_i=\Omega$. For each $i\geq1$, we set
$$
\mu_i=\sup_{K_i}\Big|\frac{G_0^2\widehat{F}}{P_i}\Big|^2\sum_{\ell\geq1}\sup_{K_i}|{B_{i\ell}}^2|\quad\text{and}\quad\gamma_i=\frac{1}{\sqrt{2^i\mu_i}}.
$$ 
For each $i\geq1$ we have the following inequalities
$$
\sum_{\ell\geq1}\sup_{K_i}\Big|\gamma_i\,\frac{G_0^2\widehat{F}}{P_i}B_{i\ell}\Big|^2
\le\gamma_i^2\sup_{K_i}\Big|\frac{G_0^2\widehat{F}}{P_i}\Big|^2\sum_{\ell\geq1}
\sup_{K_i}|{B_{i\ell}}^2|\le\frac{1}{2^i}\,.\quad{(\blacktriangleright)}
$$\label{tri}
Let $K$ be a compact subset of $\Omega$. As $\Omega\subset\bigcup_{i\ge1}\Int_{\C^n}(K_i)$, $K$ is contained in some $K_{i_0}$, hence in all $K_i$ for $i\ge i_0$, so:
\begin{equation*}
\begin{split}
\sum_{i,\ell\geq1}\sup_K\Big|\gamma_i\,\frac{G_0^2\widehat{F}}{P_i}B_{i\ell}\Big|^2&=\sum_{i=1}^{i_0-1}\sum_{\ell\geq1}\sup_K\Big|\gamma_i\,\frac{G_0^2\widehat{F}}{P_i}B_{i\ell}\Big|^2+
\sum_{i\geq i_0}\sum_{\ell\geq1}\sup_K\Big|\gamma_i\,\frac{G_0^2\widehat{F}}{P_i}B_{i\ell}\Big|^2\\
&\leq\sum_{i=1}^{i_0-1}\sup_K\Big|\gamma_i\,\frac{G_0^2\widehat{F}}{P_i}\Big|^2
\sum_{\ell\geq1}\sup_K |{B_{i\ell}}^2|+\sum_{i\geq
i_0}\sum_{\ell\geq1}\sup_{K_i}\Big|\gamma_i\,\frac{G_0^2\widehat{F}}{P_i}B_{i\ell}\Big|^2\\
&\leq\sum_{i=1}^{i_0-1}\sup_K\Big|\gamma_i\,\frac{G_0^2\widehat{F}}{P_i}\Big|^2
\sum_{\ell\geq1}\sup_K |{B_{i\ell}}^2 |+\sum_{i\geq i_0}\frac{1}{2^i}<+\infty.
\end{split}
\end{equation*}
Consequently, we define the invariant holomorphic function we were looking for as
$$
H=G_0^4\widehat{F}^2+\sum_{i,\ell\geq1}\Big(\gamma_i\,\frac{G_0^2\widehat{F}}{P_i}B_{i\ell}\Big)^2.
$$ 
Hence, $h=H_{|\R^n}$ is an infinite sum of squares of analytic functions on $\R^n$ and 
$$
\{h=0\}\subset\{g_0=0\}\cup\{\widehat{f}=0\}\subset X.
$$ 
For our later purposes it is more convenient to have a different representation of H. Namely, let us check that
$$
H=G_0^4\widehat{F}^2+\sum_{i\geq1}\Big(\gamma_i\,\frac{G_0^2\widehat{F}}{P_i}\Big)^2B_i.\quad(\star) 
$$\label{stella}
Indeed, since $B_i=\sum_{\ell\geq1}B_{i\ell}^2$ for each $i\geq1$, it is enough to check that the series $\displaystyle \sum_{i\geq1}\left(\gamma_i\,\frac{G_0^2\widehat{F}}{P_i}\right)^2B_i$ converges. Indeed, observe that for each $i,j\geq1$, we have $\sup_{K_j}|B_i|\leq\sum_{\ell\geq1}\sup_{K_j}|B_{i\ell}|^2$, so using \eqref{tri}, we deduce 
$$
\sup_{K_j}\Big|\gamma_i\,\frac{G_0^2\widehat{F}}{P_i}\Big|^2|B_i|=\sup_{K_j}\Big|\gamma_i\,\frac{G_0^2\widehat{F}}{P_i}\Big|^2\sup_{K_j}|B_i|
\le\gamma_i^2\sup_{K_j}\Big|\frac{G_0^2\widehat{F}}{P_i}\Big|^2\sum_{\ell\geq1}
\sup_{K_j}|{B_{i\ell}}^2|\leq\frac{1}{2^i}.
$$
Let $K$ be a compact subset of the open set $\Omega$ and let $i_0\geq1$ be such that $K\subset K_i$ for $i\ge i_0$. We have
\begin{equation*}
\begin{split}
\sum_{i\geq1}\sup_K\Big|\gamma_i\,\frac{G_0^2\widehat{F}}{P_i}\Big|^2|B_i|&=\sum_{i=1}^{i_0-1}\sup_K\Big|\gamma_i\,\frac{G_0^2\widehat{F}}{P_i}\Big|^2|B_i|+
\sum_{i\geq i_0}\sup_K\Big|\gamma_i\,\frac{G_0^2\widehat{F}}{P_i}\Big|^2|B_i|\\
&\leq\sum_{i=1}^{i_0-1}\sup_K\Big|\gamma_i\,\frac{G_0^2\widehat{F}}{P_i}\Big|^2
\sup_K |B_i|+\sum_{i\geq
i_0}\frac{1}{2^i}<+\infty.
\end{split}
\end{equation*}

\vspace{1mm}
\noindent{\sc Step 3.}
In this step we end the proof constructing the positive semidefinite function $q$ and the ``denominator'' $g$. 

Let us check first that $G_0^2\widehat{F}$ divides H in $\an(\Omega)$.  It is enough to check that $G_0^2\widehat{F}$ divides each addend of H on a neighborhood of $Z=\{G_0^2\widehat{F}=0\}=\bigcup_{k\geq1}Z_k$ where $Z_k=\{G_k^2\widehat{F}_k=0\}$. Fix $i\geq 1$ and observe that since $B_i=P_iD_i$, we have
$$
\Big(\gamma_i\,\frac{G_0^2\widehat{F}}{P_i}\Big)^2B_i=\gamma_i^2\frac{G_0^2\widehat{F}}{P_i}\frac{G_0^2\widehat{F}}{P_i}B_i=G_0^2\widehat{F}\Big(\gamma_i^2D_i\frac{G_0^2\widehat{F}}{P_i}\Big).
$$

Thus, there exists a holomorphic function $Q:\Omega\to\C$ such that $QG_0^2\widehat{F}=H$. Since $\Omega$ and $G_0,\widehat{F},H$ are invariant, so is $Q$. Since $g_0^2={G_0^2}_{|\R^n},\widehat{f}=\widehat{F}_{|\R^n}$ and $h=H_{|\R^n}$ are positive semidefinite analytic functions, so is $q=Q_{|\R^n}$. In fact, since $qg_0^2\widehat{f}=h$, we have $\{q=0\}\subset\{h=0\}\subset X$. Denote $g=qg_0\sqrt{u}$, where $u$ is the positive analytic unit satisfying $\widehat{f}=uf$ already constructed in {\sc Step 1}. Observe that
$$
g^2f=q^2g_0^2uf=q^2g_0^2\widehat{f}=qh\quad\text{ and }\quad\{g=0\}=\{g_0=0\}\cup\{q=0\}\subset X
$$
Thus, we have got the required representation $g^2f=hq$ once we prove that 
$$
\{q=0\}\subset\bigcup_{j,\ell\in I,j\neq\ell}(Y_j\cap Y_\ell).
$$
Indeed, fix $i\geq1$ and recall that $B_i=P_iD_i$, where $D_i\in\an(\Omega)$ is a holomorphic function on $\Omega$ whose restriction to $\R^n$ is a positive analytic unit. Moreover,
$$
H_i=\Big(Q-\gamma_i^2D_i\frac{G_0^2\widehat{F}}{P_i}\Big)G_0^2\widehat{F}=H-\Big(\gamma_i\,\frac{G_0^2\widehat{F}}{P_i}\Big)^2B_i=G_0^4\widehat{F}^2+\sum_{j\geq1,j\neq i}\Big(\gamma_j\,\frac{G_0^2\widehat{F}}{P_j}\Big)^2B_j.
$$
is an invariant holomorphic function on $\Omega$ whose restriction to $\R^n$ is an infinite sum of squares of analytic functions on $\R^n$. Write $\displaystyle Q_i'=Q-\gamma_i^2D_i\frac{G_0^2\widehat{F}}{P_i}$, which is an invariant holomorphic function on $\Omega$ such that
$$
H_i=Q_i'G_0^2\widehat{F}\quad\text{ and }\quad Q=Q_i'+\gamma_i^2D_i\frac{G_0^2\widehat{F}}{P_i}.
$$
Since the restrictions to $\R^n$ of $H_i,G_0^2,\widehat{F}$ are positive semidefinite also the restriction of $Q_i'$ to $\R^n$ is positive semidefinite and so 
$$
\{q=0\}\subset\left\{\frac{G_0^2\widehat{F}}{P_i}=0 \right\}\cap\R^n=\bigcup_{k\not\in{\mathfrak H}_i}X_k.
$$
Now, if $i\not\in I$, that is, if ${\mathfrak F}_i$ is a finite set, we have that ${\mathfrak H}_i={\mathfrak F}_i$ and 
$$
Y_i\cap\{q=0\}\subset Y_i\cap\bigcup_{k\not\in{\mathfrak H}_i}X_k=\varnothing;
$$
Putting all together, we deduce
$$
\{q=0\}=X\cap\{q=0\}=\bigcup_{i\geq1}Y_i\cap\{q=0\}=\bigcup_{i\in I}Y_i\cap\{q=0\}\subset\bigcup_{i,j\in I,i\neq j}Y_i\cap Y_j,
$$
as required.
\end{proof}

\subsection{Applications to Hilbert's 17th Problem.}\label{s6}
Let $\Omega \subset \C^n$ be a Stein open neighbourhood of $\R^n$ in $\C^n$ and assume that $\Omega$ retracts on $\R^n$.

Although the ring $\Oo(\Omega)$ is far from being a unique factorization domain, we begin this section giving a  reasonable definition  of what we mean as {\em irreducible holomorphic function} and what we mean as {\em irreducible factors} of a holomorphic function in $\Oo(\Omega)$. We give similar definitions for real analytic functions.

\subsubsection{Irreducible factors.}\label{irredfactorsc} 

As we know from Chapter 2, every complex analytic subset $Z$ of an open set $\Omega\subset \C^n$ admits a unique countable locally finite family of {\em irreducible components}, that is $Z=\bigcup_{k\geq 1} Z_k$ where for each $k$ $Z_k$ is an irreducible complex analytic subset of $\Omega$ for each $k\geq 1$.

 We can give a similar notion for holomorphic functions in $\Oo(\Omega)$. We say that $F\in \Oo(\Omega)$ is {\em irreducible} if it cannot be written as $F=G_1 G_2$ where the zerosets of $G_1$ and $G_2$ are both not empty, that is both are not units in $\Oo(\Omega)$.

Also we get the notion of {\em divisibility}. We say that $G$ {\em divides $F$ with multiplicity $k$} if $F=G^k H$ whereas  $G^{k+1}$ does not divide $F$.

Multiplicity can be computed taking germs at points where both $F$ and $G$ vanish.

Take now $F\in \Oo(\Omega)$. Its zeroset $Z$ is of pure dimension $n-1$, that is, if $\{Z_k\}_{k\geq 1}$ are the irreducible components of $Z$ all of them have dimension $n-1$. So for each $z\in Z_k$ there is a holomorphic function germ $\xi_{k,z}$ that generates the ideal of the germ $Z_{k,z}$ in $\Oo_{\Omega,z}$. Consider the  locally principal coherent analytic sheaf of ideals $\Jj_k$ defined as    

 $$
\Jj_{k,z}=\left\{
\begin{array}{lcl}
\xi_{k,z}\an_{\Omega,z}&\text{ if}& z\in Z_k,\\[4pt]
\an_{\Omega,z}&\text{ if}& z\in \Omega\setminus Z_k.\\ 
\end{array}
\right.
$$     

Since $\Jj_k$ is coherent and H$^1(\Omega, \Oo_\Omega^*)= 0$, $\Jj$ is generated by a global holomorphic function $H_k$.

This provides a way to decompose $F$ into irreducible factors. Indeed $H_k$ is irreducible for each $k$ because its zeroset is the irreducible analytic set $Z_k$. Each $H_k$ divides $F$ with some multiplicity $m_k$. A straighforward computation shows that $F=\shprod H_k^{m_k}$.     

We call $\{H_k\}_{k\geq1}$ the collection of the {\em irreducible factors of $F$}.

The factorization of a real analytic functions is similar but quite more involving.

We say that an analytic function $f\in\Oo(W)$ on an open set $W\subset\R^n$ is {\em irreducible}  if it cannot be written as the product of two analytic functions with nonempty zeroset. Similarly, given analytic functions $g,f\in\an(\R^n)$, we say that \em $g$ divides $f$ with multiplicity $k\geq 0$ \em if $g^k$ divides $f$ but $g^{k+1}$ does not. Again such a $k$ exists and it is finite.

Given a closed set $Z\subset\C^n$, germs (of sets or of holomorphic functions) at $Z$ are defined exactly as germs at a point, through neighborhoods of $Z$ in $\C^n$. We denote by $F_Z$ the germ at $Z$ of a holomorphic function $F$ defined in some neighborhood of $Z$. If $F\in\Oo(\Omega)$ is  invariant, $\R^n\subset\Omega$ and $Z=\R^n$, then the germ $F_{\R^n}$ can be identified with the function $f=F_{|\R^n}$, because each of them determines uniquely the other one.

We define the {\em irreducible factors} of $f\in\Oo(\R^n)$ as follows. 

Let $F:\Omega\to\C$ be a holomorphic extension of $f$ to an open neighborhood $\Omega$ of $\R^n$ in $\C^n$ and consider the irreducible components of its zero set.  Their germs at $\R^n$ define
  a unique locally finite family of irreducible germs $\{Y_k\}_{k\geq 1}$ at $\R^n$ such that $Y_j\not\subset Y_k$ if $j\neq k$ and $\{F=0\}_{\R^n}=\bigcup_{k\geq1}Y_k$\footnote{ It may happen that the same irreducible component  of $Z$ induces several irreducible germs $Y_k$}. For each $k\geq1$ there exists an invariant contractible open Stein neighborhood $\Omega_k$ of $\R^n$ in $\C^n$ and an irreducible analytic set $Z_k$ in $\Omega_k$ such that $Z_{k,\R^n}=Y_k$. Since $\dim Z_k=n-1$, there exists $H_k\in\Oo(\Omega_k)$ that generates the ideal $\Jj_{\Omega_k}(Z_k)$ of holomorphic functions on $\Omega_k$ vanishing identically on $Z_k$. Notice that if $Z_k$ is invariant, we may further assume by Lemma \ref{suit} that $H_k$ is also invariant. On the other hand, if $Z_k$ is not invariant, then $\ol{H_k\circ\sigma}$ divides $F$ in $\Omega_k$ with the same multiplicity as $H_k$. For each $k\geq1$, we define
\begin{center}
$$ 
F_k=\left\{\begin{array}{cl}

H_k&\text{if $Z_k$ is invariant}\\[4pt]
H_k\ol{H_k\circ\sigma}&\text{if $Z_k$ is not invariant}
\end{array}\right.
$$\end{center} 

\vspace{1mm}
Observe that $f_k=F_{k |\R^n}\in\Oo(\R^n)$ is irreducible and divides $f$ for all $k\geq1$. We eliminate  repetitions and denote again by $\{f_k\}_{k\geq1}$ the collection of the previous irreducible factors. Denote by $m_k$ the multiplicity of$f_k$ as a  divisor of $f$. A straighforward computation shows that $f=\shprod f_k^{m_k}$. We call $\{f_k\}_{k\geq1}$ the collection of the {\em real irreducible factors of} $f$.

We focus a special attention in a particular type of irreducible factor.
\begin{defn}\label{special}
  We say that $f_k$ is a \em special factor \em of $f$ if: 
\begin{itemize}
\item[(i)] The germ $Z_{k,\R^n}$ is $\sigma$-invariant.
\item[(ii)] $f_k$ divides $f$ with odd multiplicity. 
\item[(iii)] The dimension $d$ of $\{f_k=0\}$ satisfies $1\leq d\leq n-2$.
\end{itemize}
\end{defn}

Of course, we may always assume that the special factors $f_k$ of $f$ are positive semidefinite analytic functions. The reason for asking that the dimension of the zeroset of a special factor is strictly positive is the following: as we know, an irreducible factor $f_k$ of $f$ satisfying conditions (i) and (ii) above and whose zero set has dimension zero is a finite sum of squares as we saw in Section \ref{Milnor}, hence, it produces no obstruction to represent $f$ as an either finite or infinite sum of squares. Also the factors having zeroset of dimension $n-1$ have necessarily even multiplicity because if not $f$ cannot be positive semidefinite, so they are not special.

Contrary to what happens in the complex case the zeroset of a special factor can be reducible as a real global analytic set. See next section for an example.

\subsubsection{Applications.}

Now, we apply all we have done concerning  factorization of analytic functions to localize the obstructions to represent a positive semidefinite analytic function as an either finite or infinite sum of squares.

\begin{lem}\label{2decom}
Let $f:\R^n\rightarrow\R$ be a positive semidefinite analytic function. Then, $f=gh$
where $g$ is a finite sum of squares and $h$ is a square free positive semidefinite analytic function whose irreducible factors are the special factors of $f$. In fact, $g$ is a sum of $2^{n+1}$ squares.
\end{lem}
\begin{proof}
First, let $\{h_k\}_{k\geq1}$ be the special factors of $f$ and let $\{g_j\}_{j\geq1}$ be all the others irreducible factors. Denote $2n_k+1$ the multiplicity of division of $f$ by $h_k$ and $m_j$ be the multiplicity of division of $f$ by $g_j$. Notice that if $m_j$ is odd and $\{g_j=0\}$ is not a singleton, then there exists an invariant open neighborhood $\Omega_j$ of $\R^n$ in $\C^n$, a holomorphic extension $G_j$ of $g_j$ to $\Omega_j$ and an (irreducible) holomorphic function $H_j\in\hol(\Omega_j)$ such that $G_j=H_j\ol{H_j\circ\sigma}$. Thus, in particular $g_j$ and so $g_j^{m_j}$ are sums of two squares in $\an(\R^n)$. On the other hand, if $\{g_j=0\}$ has dimension zero (and so it is a point) we know that $g_j$ and so $g_j^{m_j}$ are sums of $2^{n+1}$ squares (with controlled denominators). Now, the claim follows from Theorem \ref{globalization}.

Next, consider the sheaf-product $g=(\shprod h_k^{n_k})^2(\shprod g_j^{m_j})$ and observe that, by Theorem \ref{main1}, $g$ is a sum of $2^{n+1}$ squares. Next, let $h=\shprod h_k$ and observe that we may assume that $f=gh$. Clearly, the irreducible factors of $h$ are the special factors of $f$, as required.
\end{proof}

\begin{remark}\label{2decomr}
If $n=2$ it is clear that the zerosets of the irreducible factors have either dimension $0$ or $1$. Of course, if $f$ is positive semidefinite, the zerosets of the irreducible factors $g_j$ that divide $f$ with odd multiplicity have dimension $0$, and so they are singletons $\{a_j\}$. In fact as we saw in Section \ref{lowdim}   $g_{j}$ and $f$  is a sum of two squares of analytic functions. 

Moreover, the previous discussion proves that the special factors only appear for $n\geq3$. Hence, obstructions for a positive answer to H17 and H17$_\infty$ can only arise for $n\geq 3$ and, of course, they are focused on the special factors.\qed 
\end{remark}

As a consequence we give a simple proof of the following well known result.

\begin{thm}\label{bksj} Let $f\in\Oo(\R^n)$ be a semidefinite positive analytic function whose zero set is the union of a compact set  and a discrete set. Then $f$ is a finite sum of squares. 
\end{thm}
\begin{proof}
First by Lemma \ref{2decom} we write $f=gh$ where $g$ is a finite sum of squares and $h$ is a square free positive semidefinite analytic function whose irreducible factors are the special factors of $f$. Since the zero set of $f$ is the union of a discrete set and a compact set, it follows that the zero set of $h$ is compact. Hence, $h$ is by Theorem \ref{compatto} a finite sum of squares,  so also $f$ is a finite sum of squares, as required.
\end{proof}

Theorems \ref{compatto},\ref{vmain} and \ref{nash} together with Lemma \ref{2decom} provide the following results.

\begin{cor}\label{con1}
Let $f:\R^n\to\R$ be a positive semidefinite analytic function and let $\{h_k\}_{k\geq1}$ be its special factors. We have,
\begin{itemize}
\item[(i)] If each $h_k$ is a sum of $2^r$ squares at $\{h_k=0\}$, then $f$ is a sum of $2^{n+r}$ squares (with controlled denominators).
\item[(ii)] If each $h_k$ is a finite sum of squares at $\{h_k=0\}$, then $f$ is an infinite sum of squares.
\item[(iii)] If each $h_k$ is an infinite sum of squares at $\{h_k=0\}$, there is a positive semidefinite $q\in\an(\R^n)$ such that $\{q=0\}\subset\{f=0\}$, $\dim\{q=0\}<\dim\{f=0\}$ and $qf$ is an infinite sum of squares.
\end{itemize}
\end{cor}

\begin{cor}\label{con2}
Let $f:\R^n\to\R$ be a positive semidefinite analytic function and let $\{h_k\}_{k\geq1}$ be its special factors. We have,
\begin{itemize}
\item[(i)] If each $h_k$ is of Nash type, then $f$ is a sum of $2^{n+1}$ squares (with controlled denominators).
\item[(ii)] If each $\{h_k=0\}$ is either compact or of Nash type, then $f$ is an infinite sum of squares.
\end{itemize}
\end{cor}

We have another corollary that can be considered as a converse to what we saw in Section \ref{pitagora}.

\begin{cor}\label{pit-1}
Assume the Pythagoras number $p(\mathcal M (\R^n) =p<\infty$. Let $f$ be a positive semidefinite analytic function and let $\{h_k\}_{k\geq1}$ be its special factors. Assume each $h_k$ is a sum of squares. Then $f$ is a finite sum of squares of meromorphic functions.
\end{cor}
\begin{proof} Since $p(\mathcal M (\R^n)) =p$, each $h_k$ is a sum of $p$ squares of meromorphic functions. Assuming $p\leq 2^r$, by Cor \ref{con1} (i) we get $f$ is a sum of $2^{n+r}$ squares with controlled denominator.
\end{proof}

\section{Examples.}\label{esempi}

It is well known that the first example of a positive semidefinite polynomial which is not a sum of squares of polynomials is the celebrated Motzkin's polynomial, namely 

$$m(x,y) = x^4y^2 + x^2y^4 -3x^2y^2 +1.$$

Of course, by Artin's solution of H17 it is a sum of squares of rational functions. Such a representation is obtained in several ways. For both $(x^2+y^2+1)m(x,y)$ and $(x^2+y^2)^2 m(x,y) $ are sums of squares of polynomials.

From an analytic view point the first representation gives also $m(x,y)$ as a sum of squares of analytic functions (accordingly to what we saw in Section \ref{lowdim}).  

Using  the homogeneous Motzkin's polynomial $\tilde m (x,y,z) = x^4y^2 + x^2y^4 -3x^2y^2 z^2 + z^6$ one realizes that denominators are needed also in the analytic case. Indeed  let $f$ be a positive semidefinite analytic function in $\Oo(\R^n) \, n\geq 3$ and assume that the Taylor expansion of $f$ at some point of its zeroset is a sum of homogeneous polynomials  that begins with   $\tilde m$. If $f$ is a sum of squares of meromorphic functions, this representation is {\em with denominators. } 
Otherwise, we would get a representation of $\tilde m$ as a sum of squares of homogeneous polynomials. In particular, $\tilde m$ is not a sum of squares of analytic functions on $\R^3$ and  is a special analytic function whose zeroset is reducible.

Before giving a list of more sophisticated special factors we need a lemma.

\begin{lem}\label{irreducible}
The homogeneous polynomial $F(x,y,z)=(z+ax)^2(z+by)^2+z^4+c^2x^2y^2\in\R[x,y,z]$ is irreducible in $\C[x,y,z]$ for all
$a,b,c\in\R$ such that $a,b,c\neq 0$, $c^2\neq a^2b^2$.
\end{lem}
\begin{proof}
Write $F$ as polynomial in $\C[y,z][x].$

$$
F=Ax^2+2Bx+C
$$

where
$$
\begin{array}{l}
A=a^2z^2+2a^2byz+(a^2b^2+c^2)y^2= a^2(z+by)^2+c^2y^2\\
B=az(z+by)^2\\
C=z^2((z+by)^2+z^2).
\end{array}
$$

Observe that $A,B,C$ as polynomials  in $y,z$ are coprime. Indeed, $\gcd(B,C)=z$ and $z$ does not divide $A$.

Assume $a,b,c,a^2b^2-c^2\neq 0$. Suppose that $F$ is reducible and
write  $F=G_{1}G_{2}$.

If $G_{1}$ (or $G_2)$ is a polynomial of degree $0$ with respect to $x$ then,  either it has real coefficients or not, in the second case consider  $G_1\overline{G_1}$. Observe that in both cases we find   a common divisor of $A,B,C$, which is a contradiction.

  Next, we see that $F$ cannot be written as the product of two linear factors with respect to $x$. Assume 
$$
F=(\alpha_{1}x+\beta_{1})(\alpha_{2}x+\beta_{2}).
$$
If the factors are real, the set $\{\alpha_{1}x+\beta_{1}=0\}\cap \R^3$, which has dimension $2$, would be a subset of $\{F=0\}\cap \R^3$, while the real zeroset of $F$      is the union of the two lines $x=0,z=0$ and $y=0,z=0$, so it has dimension $1$, which is a contradiction. 

Thus,  $F$ is irreducible as a real polynomial.

Next suppose the factors are not real.   Hence $F$, as polynomial in $x$,  has  
 two conjugated roots  which are rational functions in $\C(y,z)$.
This implies that 
 $\sqrt{B^2-AC}$ is a rational function $H'$ whose square is the real polynomial  $B^2 -AC$. But $H'$ satisfies an integral dependence equation, namely $T^2 -(B^2-AC)=0$, hence it is in fact a polynomial because the ring of polynomials is integrally closed in its quotient field. 
A computation shows that $H'$ must be either a real polynomial or a polynomial with pure imaginary coefficients.  The first case was already considered, so we can assume $H'=\sqrt{B^2-AC}\in i \R[y,z]$.  So $H = \frac{1}{i}H' = \sqrt{AC-B^2} \in \R[y,z]$.
 Therefore, $AC=B^2+H^2$, where $H\in\R[y,z]$ is a quadratic form. Look at the factors of $B^2+H^2$, $A$ and $C$ and observe that

\begin{align*}
B^2+H^2 &=  (B+iH)(B-iH) \\ 
A  &= (az+aby+icy)(az+aby-icy) \\ 
C &= z^2(z+by+iz)(z+by-iz).
 \end{align*}

Hence, we have essentially the following possibilities.
\begin{itemize}\parindent= 0pt
\item[(i)] $C$ divides $B+iH$. Since $C\in\R[y,z]$, we have that $C$ divides $B$, which is a contradiction.
\item[(ii)] $E_{1}=(az+aby+icy)(z+by+iz)= az^2+(2ab -c)yz+a^2b^2y^2+i(bcy^2+az^2 +(c+ab)yz)$ (which is a factor of $A$ times a factor of $C$) divides $B+iH$, that is, $B+iH=\eta E_{1}$ for some $\eta= \lambda +i \mu \in\C$. Then, 
$$\lambda(az^2+(2ab-c)yz+b^2ay^2)+\mu(az^2+(ab+c)yz+cby^2)=B=$$
 $$=a(z+by)^2=az^2+2abyz+ab^2y.$$
Thus, $\lambda+\mu=1$, $\lambda(2ab-c)+\mu(ab+c)=2ab$ and $\lambda ab^2+\mu cb=ab^2$. Therefore, 
$$
\lambda=1-\mu,\qquad \mu(2c-ab)=c,\qquad b\mu(c-ab)=0.
$$
As $b,c$ are not $0$, we deduce that $c=ab$, which is a contradiction.
\item[(iii)] $E_{2}=(az+aby+icy)(z+by-iz)$ divides $B+iH$. Proceeding as in the previous case, we deduce that $c=-ab$, which is a contradiction.
\end{itemize} 
We conclude that $F$ is irreducible in $\C[x,y,z]$.
\end{proof}

\begin{examples}\label{examples} 

\begin{enumerate}{\parindent=0pt
\item For each $k\geq 0$ consider the polynomial functions on $\R^2$ and $\R^3$ respectively given by the formulae $f_k(x,y)=(x-k)^2+y^2-1$ and $g_k(x,y,z)=z+f_k$. Let $f:\R^2\to\R$ and $g:\R^3\to\R$ be the sheaf-products of the $f_k$'s and the $g_k's$ respectively. For each $k,\ell\geq0$ with $\ell\neq k$, we have $\{f_k=0\}\cap\{f_\ell=0\}\neq\varnothing$ if and only if $|\ell-k|\leq1$; and, if such is the case, the two intersection points of $\{f_k=0\}$ and $\{f_\ell=0\}$ are 
$$
p_{\ell,k,\pm}=\Big(\frac{\ell+k}{2},\pm\frac{\sqrt{3-(k-\ell)^2}}{2}\Big).
$$ 
For each real number $M>0$, consider the analytic function 
$$
h_M:\R^3\to\R,\ (x,y,z)\mapsto g^2+z^4+M^2f^2(z).
$$
Observe that $\{h_M=0\}=\bigcup_{k\geq0}\{(x,y,0):\ (x-k)^2+y^2-1\}$, which is a countable locally finite union of circunferences, which are compact analytic subsets of $\R^3$. Observe that for each $k,\ell\geq0$ with $k\neq\ell$ and $|\ell-k|\leq1$, the initial form of the germ of $h_M((x,y,z)+q_{\ell,k,\pm})$ at $q_{\ell,k,\pm}=(p_{\ell,k,\pm},0)$ is a homogeneous polynomial of the type
$$
P_{\ell,k,\pm}=\alpha^2(z+u)^2(z+v)^2+z^4+M\beta^2u^2v^2
$$
where $u=(\ell-k)x\pm\sqrt{3-(k-\ell)^2}y$ and $v=-(\ell-k)x\pm\sqrt{3-(k-\ell)^2}y$ and $\alpha,\beta$ are nonzero real numbers depending of $q_{\ell,k,\pm}$. By Lemma \ref{irreducible} we may choose $M>0$ in such a way that all the homogeneous polynomials $P_{\ell,k,\pm}$ are irreducible. Hence for this choice of $M$  we get that $h_M$ is irreducible, so it is a  special factor whose zeroset is a locally finite union of compact analytic sets. \qed

\item Let $H:\C\rightarrow\C$ be an invariant holomorphic function such that
$H^{-1}(0)=\{k\in\Z:\ k\geq 0\}$ and $H$ has a zero of order one at each point of its zero set. Such a function
exists by the Weierstrass Factorization Theorem. Let $a_k=H'(k)\neq 0$ for all integer $k\geq 0$ and let $M>0$
be a real number such that $M^2\neq\dfrac{1}{a_k^2a_\ell^2}$  for all couple of integers $k,\ell\geq 0$. Let
$$
F(x,y,z)=(z+\sin(\pi x))^2(z+\sin(\pi y))^2+z^4+M^2H(x)^2H(y)^2,
$$
which is an invariant holomorphic function on $\C^3$.

Let us show that $f=F|_{\R^3}$
is a special analytic function whose real zero set is the net
$$
S=\bigcup_{k\geq 0}\{x=k,z=0\}\cup\bigcup_{\ell\geq 0}\{y=\ell,z=0\},
$$
which has infinitely many irreducible components. 

 Indeed, a straighforward computation shows that $\{F=0\}\cap \R^n=S$. To show that $f$ is a special analytic function we have to check that the restrictions of $F$ to small enough invariant neighbourhoods $U$ of $\R^n$ in $\C^n$ cannot be written as the product of two holomorphic functions $G_{1},G_{2}:U\to\C$ such that $G_{i}^{-1}(0)\cap\R^n\neq \varnothing$.
First, we show some crucial properties of $F$ to prove the irreducibility of $f$.
\begin{itemize}{\parindent=0pt
\item[(a)] For each pair of non-negative integers $k,\ell\geq 0$ consider the point $p_{k,\ell}=(k,\ell,0)$. Then the function germs $F_{p_{k,\ell}}$ are irreducible  in the ring $\an_{\C^3,p_{k,\ell}}$ for all $k,\ell$. This is because their initial forms at the points $p_{k,\ell}$ are irreducible by Lemma \ref{irreducible}.
\item[(b)] For each integer $k\geq 0$ and each $\lambda\in\R\setminus\Z$ consider the point $p_{k,\lambda}=(k,\lambda,0)$. Then, for all $k,\lambda$ as before, the function germs $F_{p_{k,\lambda}}$ are the product  in $\an_{\C^3,p_{k,\lambda}}$ of two irreducible factors of order $1$ that vanish at the line germ $x=k,z=0$.

  Indeed, after  translating the point $p_{k,\lambda}$  to the origin,
  the initial form of  $F_{p_{k,\lambda}}$ is $(z+ax)^2b^2+c^2x^2$ for
  some   real   number   $a,b,c>0$.   Thus,   by   classification   of
  singularities, $F_{p_{k,\lambda}}$ is analytically equivalent either
  to  $x^2+z^2$ or  to a  polynomial of  the type  $x^2+z^2+\veps y^k$
  where  $k\geq  3$   and  $\veps=\pm  1$.  Since  the   zero  set  of
  $F_{p_{k,\lambda}}$  is the  line germ  $x=k,z=0$, we  conclude that
  $F_{p_{k,\lambda}}$ is analytically equivalent to $x^2+z^2$. Hence,
$$
F_{p_{k,\lambda}}=F_{1}^2+F_{2}^2=(F_{1}+i F_{2})(F_{1}- iF_{2}),
$$ 
where the factors $F_{1}+iF_{2},F_{1}-i F_{2}$ have order $1$, are irreducible and  vanish at the line germ $x=k,z=0$. 
}\end{itemize}

Suppose now that there exist an open invariant neighbourhood $U$ of
$\R^n$ in $\C^n$ and two holomorphic functions $G_{1},G_{2}:U\to\C$
such that $G_{i}^{-1}(0)\cap\R^n\neq \varnothing$ and
$F=G_{1}G_{2}$. Then, for each $p\in S=F^{-1}(0)\cap\R^n$ we have that
$F_{p}=G_{1,p}G_{2,p}$. Since the line $x=0,z=0$ is irreducible and it
is contained in $\{F=0\}$, we may assume that it is also contained in
$\{G_{1}=0\}$. As the germ $F_{p}$ is irreducible for the points
$p=p_{0,\ell}=(0,\ell,0)$, we have that all the lines $y=\ell,z=0$ are
contained in $\{G_{1}=0\}$. Furthermore, since the germ $F_{p}$ is
irreducible for the points $p=p_{k,0}=(k,0,0)$, we have that all the
lines $x=k,z=0$ are contained in $\{G_{1}=0\}$. Hence, $S$ is a subset of
$\{G_{1}=0\}$. Again, since the germ $F_{p}$ is irreducible for the points
$p=p_{k,\ell}=(k,\ell,0)$, we have that no line of $S$ can be
contained in $\{G_{2}=0\}$. As the lines are irreducible, we deduce
that $G_{2}^{-1}(0)\cap \R^n$ is a discrete set contained in
$S$ but which does not intersect the set $\{(k,\ell,0):\ k,\ell\geq
0\}$.

Next,   we   take   $p\in  G_{2}^{-1}(0)\cap\R^n$.   We   may   assume
$p=(k,\lambda,0)$   for  certain   integer  $k\geq   0$  and   certain
$\lambda\in\R$   which  is   not   a  non   negative  integer.   Since
$F_{p}=G_{1,p}G_{2,p}$  is a  product  of two  irreducible factors  of
order $1$  that vanish at  the line  germ $x=k,z=0$, we  conclude that
$G_{2,p}$ must vanish at the line germ $x=k,z=0$, which is a contradiction.

Thus, $f$ is a special analytic function. 
}\end{enumerate}
\end{examples}

We provide some examples of positive semidefinite analytic functions whose special factors are costructed using the examples above.

\vspace{2mm}
\begin{examples}\label{eh17}
\begin{enumerate}{\parindent=0pt
\item Let $f_{0}:\R^3\rightarrow\R$ be the analytic function given by
$$
f_{0}(x,y,z)=(z+x)^2(z+y)^2+z^4+4x^2y^2,
$$
which by Lemma \ref{irreducible} is a special analytic function.  Note that $f_0$ is a sum of three squares of analytic functions. 

For each integer $\ell\geq 1$, let $f_{\ell}(x,y,z)=f_{0}(x-q_\ell,y-(\ell-q_\ell),z)$ where $q_{\ell}=\left [{\frac{\ell}{2}}\right ]$. The zero set of $f_\ell$ is 
$$
X_\ell= \{x=q_\ell,z=0\}\cup\{y=\ell-q_\ell,z=0\}.
$$
Since the family $\{X_\ell\}_\ell$ is locally finite, the set $X=\bigcup_{\ell}X_\ell$ is closed in $\R^3$. Hence,
the sheaf
$$
\J_x=
\left\{
\begin{array}{lcl}
\prod_{\ell,x\in X_\ell}f_\ell\cdot \an(\R^3)_x&\text{if}& x\in X\\[5pt]
\an(\R^3)_x&\text{if}& x\not\in X\\ 
\end{array}
\right.
$$
is a subsheaf of the structure sheaf $\an_{\R^3}$. Observe that $\J$ is a locally principal coherent sheaf of ideals whose zeroset is $X$, hence it is  principal. Let  $f$ be a global generator.  Note that since $X$ has codimension $\geq 2$ we may assume that $f$ is positive semidefinite on $\R^3$. Clearly, the special factors of $f$ are the functions $f_\ell$, for $\ell\geq 1$. 

Next, note that 
$$
X_{\ell}\cap X_{\ell+1}=
\left\{
\begin{array}{ll}
\{x=q_\ell,z=0\}&\text{if $\ell$ is even,}\\[5pt]
\{y=\ell-q_\ell,z=0\}&\text{if $\ell$ is odd,}
\end{array}
\right.
$$
which is not a discrete set for all $\ell\geq 1$. 

\item
Let $f_{0}:\R^3\rightarrow\R$ be the special analytic function described in \ref{examples} (2), which is a sum of three squares in $\an(\R^3)$.  For each integer $\ell\geq 1$ consider the analytic function $f_{\ell}(x,y,z)=f_{0}(x-\ell,y-\ell,z)$, whose zero set is 
$$
X_\ell=\bigcup_{k\geq \ell}\{x=k,z=0\}\cup\bigcup_{k\geq \ell}\{y=k,z=0\}
$$
Since the family $\{X_\ell\}_\ell$ is locally finite,   there exists a positive semidefinite analytic function $f:\R^n\to\R$ whose special factors are the functions $f_\ell$, for $\ell\geq 1$ (and each one divides $f$ with multiplicity one). 
\item To describe next example, consider the following distribution of the natural numbers in infinite arrays.
$$
\begin{array}{ccccccc}
\quad\swarrow&\quad\swarrow&\quad\swarrow&\quad\swarrow&\quad\swarrow&\quad\swarrow\\
1&2&4&7&11&16&\cdots\\
3&5&8&12&17&\cdots\\
6&9&13&18&\cdots\\
10&14&19&\cdots\\
15&20&\cdots\\
21&\cdots\\
\vdots&\ddots
\end{array}
$$
and the finite sets $S_k=\{a_k+1,\ldots,a_k+k\}$ where $a_k=\frac{1}{2}k(k-1)$, for $k\geq 1$ corresponding to the $k^{th}$ oblique line. The set $S_k$  (in matrix notation) is given by $\{\alpha_{1,k},\alpha_{2,k-1},\ldots,\alpha_{k,1}\}$ of the previous array. 

For each $k\geq 1$ we also construct (inductively) a set 
$$T_k=\left\{b_{kj}:\ k-\left [{\frac{k}{2}}\right ]\leq j\leq k-1\right\}$$ 
such that
$b_{kj}\in S_j\setminus\bigcup_{\ell=1}^{k-1}T_\ell$. We take $T_1=\varnothing$ and $T_2=\{1\}$. By the definition of the sets $T_k$, for any given $j$, we have $S_j\cap T_k\neq\varnothing$ if and only if $k-\left [{\frac{k}{2}}\right] \leq j\leq k-1$. Thus, $j+1\leq k\leq 2j$, and this means that $S_j$ intersects exactly $j$ of the $T_k$'s, which are $T_{j+1},\ldots,T_{2j}$. Since the set $S_j$ has $j$ different elements then the $T_k$ can be constructed with the desired conditions. We denote $C_k=S_k\cup T_k$.

Next, for each $k\geq 1$ we consider the holomorphic function
$$
F_k=(\sin(\pi x)+z)^2(\sin(\pi y)+z)^2+z^4+M^2(x-k)^2\prod_{\ell\in C_k}(y-\ell)^2
$$
where $M>0$ is a positive real number such that $M^2\cdot\prod_{\ell\in C_k,\ell\neq j}(j-\ell)^2\neq 1$ for all $j\in C_k$. We have that the real analytic function $f_k={F_k}_{|\R^n}$ is a special analytic function whose real zero set is 
$$
X_k=\{x=k,z=0\}\cup\bigcup_{\ell\in C_k}\{y=\ell,z=0\}.
$$
One can check, proceeding similarly to  Exemple \ref{examples}(2) , that $f$ is a special factor.

Once again, the family $\{X_k\}_k$ is locally finite. Hence, the set $X=\bigcup_{k}X_k$ is closed in $\R^3$ and there exists a positive semidefinite analytic function $f:\R^n\to\R$ whose special factors are the functions $f_k$, for $k\geq 1$.

}\end{enumerate}
\end{examples} 

As a consequence of Theorem \ref{main1} all examples  in \ref{eh17} are sums of at most  $2^3$  squares.

\subsection*{ Bibliographic and Historical Notes.} \rm

The theory of real fields was developped by Artin and Schreier (\cite{arsc}) in order to solve  Hilbert's $17^\text{th}$ problem, that was indeed solved by Artin in \cite{ar}. One can found the whole theory  in many textbooks of algebra, for example \cite{lang}. The whole theory in the real algebraic setting is excellentely exposed in \cite{bcr}. Concerning Artin-Lang homomorphism E. Artin proved the homomorphism theorem for $\R$ , S. Lang extended his result to any real closed field.

The proof of Theorem \ref{anbn} follows the one  of Lassalle \cite{la}. Note that  the same type of ideas allows to prove the Nullstellensatz for the ring $\Oo_n$.

The idea of using  fiber bundles to solve H17 is inspired by Jaworski \cite{j1}

Classification of 1-dimensional real analytic manifolds  can be found in \cite{mi}.
For topological triviality of vector bundles of rank 2 over $S^1$ see \cite{st}, which is also a good reference for generalities on vector bundles .

 Analytic triviality for topological trivial bundles on analytic manifolds follows by  Cartan's invertible holomorphic matrices theorem \cite{c3} whose real version was  proved by Tognoli \cite {t3}.

The number of squares needed to describe a nonnegative analytic function germ of 2 variables (that is $p(\R\{x_1,x_2\})$ appears in \cite{br}.

Section  2.B is strongly inspired by  \cite{j1}.

Pytagoras numbers of singular curves and surfaces can be found in \cite{abfr2} and \cite{abfr1}.

 The proof of Lemma \ref{estensione} follows from   Corollary 3.3 in Chapter 7 of \cite{abr}. It was firstly proved by Ruiz in \cite{rz2}. 

Theorem \ref{compatto} was proved indipendentely by Jaworski \cite{j2} and Ruiz \cite {rz0}.

Section 4  contains the results of \cite{abfr3}. The case of semidefinite analytic functions with discrete zeroset was proved  in \cite{bks}.
The criterion for an analytic function to be $k$-determined is Theorem 9.1.3
of the book \cite{jp}.

The multiplicative Pfister's formula says that in a field $F$   if $n=2^k$ then the identity $(x_1^2+\cdots +x_n^2 )(y_1^2+\cdots +y_n^2) =(z_1^2+\cdots +z_n^2)$ holds true.  It can be found in \cite{pf}. 

In Section 6 are exposed mainly the results in \cite{abf3}.  Functions admitting different factorizations can be found in \cite{sh}.

The positive answer to H17 for Nash functions was proved by Mostowski \cite{m} and Efroymson \cite{e}, see also \cite[8.5.6, 8.8.4]{bcr} and \cite[7.4]{mh1}.

The proof of Lemma \ref{matrix0}   is  a slight modification of the classical procedure to construct such matrices (see \cite[7.1]{mh1}).

Artin-Mazur Theorem can be found in \cite[8.8.4]{bcr} and \cite[7.4]{mh1}.

For unbounded Pfister formula we use bundles having as fiber a not finite dimensional Hilbert space: general references for this subject are \cite{hp} and \cite{zkkp}.

Several examples of Section \ref{esempi} come from  \cite{fe2}.

\newpage

\chapter{Analytic inequalities.}



In Real  Geometry it is natural to consider also inequalities. In this way arose    the concept of {\em semianalytic set}, due to \L ojasiewicz, which generalises to the real analytic setting the concept of semialgebraic set, imitating the local definition of a real analytic set. 

{\sc Definition.}
Let $M$ be a real analytic manifold. A set $S\subset M$ is a {\em semianalytic subset of} $M$ if for each $x\in M$ there is an open neighborhood $U^x\subset M$ such that the intersection $S\cap U^x$ is a finite union of sets of the form $\{f=0, g_1>0,\ldots , g_s>0\}$ where $f,g_1,\ldots ,g_s\in\Oo(U^x)$.

The class of semianalytic sets behaves well with respect to boolean and topological
 operations, but it is not stable under proper analytic maps. This fact led 
 \L ojasiewicz  and Hironaka and many others to introduce
  and develop the theory of {\em subanalytic sets}.

The definition of C-analytic set that we saw in Chapter 1 is of global nature while  semianalytic sets are  locally defined as complex analytic sets.

 Some remarkable subsets of a C-analytic set $X$  are known to be semianalytic. For instance, the set of points  where the local dimension of $X$ is equal to a certain integer, or the set of points  where $X$ is not coherent. The semianalytic nature of these sets makes no reference to the global nature of the C-analytic set $X$. Thus, it seems reasonable to ask whether there exists a notion of global semianalytic set that restricts the class of semianalytic sets and mimics the definition of C-analytic sets proposed by Cartan. More precisely, we wonder whether there exists a class of semianalytic sets defined using only global analytic functions defined on $M$, but having a similar behavior with respect to boolean and topological operations as that of \L ojasiewicz semianalytic sets. In addition, we would like that such class satisfies also some reasonable properties with respect to images under proper analytic maps. 

\section{Global semianalytic sets.}\label{gas}

A first tentative in the  direction above was explored by Andradas, Broecker and Ruiz under compactness assumptions and by Andradas and Castilla in the general approach but only for low dimension. 
We recall their definition.

\begin{defn} \label{glas} A {\em global semianalytic subset} of a real connected analytic manifold $M$  or a C-analytic space $Z$ is a definable subset with respect to the ring $\Oo(M)$, resp. $\Oo(Z)$,  that is  a finite boolean combination of equalities and inequalities involving global analytic functions on $M$ or $Z$. 
\end{defn}
\smallskip

This class is clearly closed with respect to finite unions or intersections and under complement. However it is not clear whether this class  is closed under topological operations as taking connected components, interior part or closure, results that are proved only for subsets  of low dimensional C-analytic sets or under some compactness assumpion. 

However this class has the so called {\em finiteness property}.

\begin{prop}[Finiteness property] 
Let $S\subset Z$ be a global semianalytic set in a real C-analytic space $Z$. 
\begin{itemize}
\item[(i)] Suppose $S$ to be open in $Z$. Then it is a finite union of open basic global semianalytic sets, that is, $S$ is a finite union of global semianalytic sets of  type $\{f_1>0,\ldots,f_r>0\}$ where each $f_i\in\Oo(Z)$.
\item[(ii)] Suppose $S$ to be closed in $Z$. Then it is a finite union of closed basic global semianalytic sets, that is, $S$ is a finite union of global semianalytic sets of  type $\{f_1\geq0,\ldots,f_r\geq0\}$ where each $f_i\in\Oo(Z)$.
\end{itemize}
\end{prop}
\begin{proof}It is enough to prove finiteness for an open global semianalytic set since the complement of a closed one is open and global. So assume $S$ is open and described as

$$S = \bigcup_{j=1}^t\{f_j=0, g_{1,j}>0,\cdots, g_{k(j),j} >0\}$$
where $f_j, g_{i,j}$ are global analytic functions on $Z$. We need to replace, for each $j$ such that $f_j$ is not identically zero, the set $Z_j = \{f_j=0\}\cap \{g_{1,j}>0,\cdots, g_{k(j),j} >0\} $ by an open global basic neighbourhood $B$ of $Z_j$  such that $B\subset S$. So we can assume we are in the following situation.

$$U= \{g_1>0,\cdots,g_k>0\}, Y=\{f=0\}$$

where $U\cap Y\subset S$. 

Consider the closed set $V= \{g_1\geq 0, \cdots,g_k\geq 0\}\setminus S$ and note that $\overline U \setminus S\subset V$. Set $g=g_1\cdot g_2 \cdots g_k$ and consider the function $f^2+g^2$. Then 

$$\{f^2+g^2 =0\}\cap V = \{f=0\}\cap V$$

In fact if $x\in V$ and $f(x)=0$ then at least one $g_i$ vanishes at $x$ because $x\notin S$. Thus we can apply  Theorem \ref{weakLoj} of Chapter 3 and we find
a global function $h$ that is positive semidefinite on $Z$, whose zeroset is $\{f^2 +g^2=0\}$ and satisfies $h\leq |f|$ on $V$. 

Next consider the basic open set $B= \{f^2-h^2 <0, g_1 > 0, \ldots , g_k > 0\}$. We get $B\subset S$. Indeed $B\subset U$ and $f^2-h^2 \geq 0$ outside $S$. Also $Y\cap U\subset B$ because inside $Y\cap U$ one has $f=0$ and $h^2 >0$. 
The proof is complete.
\end{proof}

\subsection{Global semianalytic subsets of an analytic curve.}

The first result is that semianalytic subsets of a C-analytic set of dimension $1$ are all global.

\begin{thm}\label{dim1}
Let $Z\subset \R^n$ be a C-analytic curve and $S\subset Z$ be a semianalytic set. Then $S$ is global.  
\end{thm}
\begin{proof} It is easy to prove that   discrete subsets of $\R^n$ are C-sets using a local polynomial equation and Theorem B. So we may assume that $Z$ has pure dimension $1$ and that $S$ is closed. Then $D =  S \setminus\stackrel{\circ}{S}$ is a discrete set and for each $p\in D$ the closed semianalytic set germs $S_p$ and $\overline {Z\setminus S}_p$ meet only in $\{p\}$. Hence there is a polynomial $h_p \in \R[x_1,\ldots,x_n]$ negative on $\overline {Z\setminus S}_p \setminus\{p\}$ and positive on $S_p\setminus \{p\}$. This is true in a small open ball $B_p$ and we may assume that for different points in $D$ the corresponding balls are disjoint.
Next we apply density of global analytic functions in the ring of germs. If $G\in \Oo(\R^n)$ verifies, for all $p\in D$,  $G_p -h_p\in \gtm_p^{\nu_p}$ for $\nu_p$ sufficiently bigger than the degree of $h_p$ then, possibly shrinking the ball $B_p$, one has $B_p\cap S = \{x\in B_p: G(x) \geq 0\}$.
Take smaller balls $B''_p\subset B'_p \subset B_p$ with $\overline {B''_p} \subset B'_p$. Then the sets $\bigcup_p \overline {B'_p}$ and $\R^n \setminus \bigcup B_p$ are disjoint closed sets, hence there is an analytic function $g$ positive on the first set and negative on the second set.

The curve $Z$ has global equations in $\R^n$. Taking $f$ as the sum of squares of these equations we can write $Z= \{x\in \R^n : f(x) = 0\}$. Hence $S \cap \{g\leq 0\} = \{G\leq 0, g\leq 0, f=0\}$. To describe $S$ outside the union of the balls $B'_p$ we do a similar trick two more times. We consider $\bigcup_p \overline {B''_p}$ and $\R^n \setminus \bigcup_p B'_p$. Since they are disjoint closed sets there is $b\in \Oo(\R^n)$ positive on the first set and negative on the second. Finally we consider $S\setminus (\overline {\R^n \setminus \bigcup_p \overline {B''_p}})$ and $(Z\setminus S) \bigcap \overline {\R^n \setminus \bigcup_p \overline {B''_p}}$. Again they are closed disjoint sets and we can find $a\in \Oo(\R^n)$ which is positive on the first set and negative on the second one. Finally we get

$$S= \{f=0, g\leq 0, G\leq 0\} \cup \{f=0, b\geq 0, a \geq 0\}$$

as wanted.         
\end{proof}

\begin{remark}
In the previous proposition the condition on $S$ to be a subset of a real analytic curve cannot be avoided. Indeed there are $1$-dimensional semianalytic sets in $\R^2$ that are not global as the following example shows. 
\end{remark}

\begin{example}\label{segmenti}
Consider the set
$$ S = \bigcup_{m\in \N} S_m\ \mbox{\rm where} \ S_m  =  \{(x,y) \in \R^2: y=mx, m\leq x \leq m+1\}.$$

$S_m$  is a  segment, contained in the line $\{y=mx\}$. $S$ is of course a semianalytic set, but if an analytic function vanishes on $S$ it has to vanish on all lines $\{y=mx\}$ and since these lines accumulate on the $y$-axis, it has to be the zero function. Hence $S$ is not global.   
\end{example}
\smallskip

\subsection{Global semianalytic subsets of $2$-dimensional manifolds.}

We consider now a connected real analytic manifold $M$ of dimension $2$ and a global semianalytic set $S\subset M$. The questions we want to answer are the following.
\begin{itemize}  
\item  The connected components of $S$ are still global semianalytic sets.
\item  The closure and the interior part of $S$ are still global.
\end{itemize}

By Grauert's  embedding theorem, we can assume $M$ to be a closed submanifold of $\R^n$. Hence $M$ is defined as the zeroset of finitely many analytic functions $f_1,\ldots, f_l$ on $\R^n$ and we can write $M= \{f_1^2+ \cdots+ f_l^2 =0\}$.
Also it holds: $\Oo_M = \Oo_{\R^n}/\Ii_M$ where $\Ii_M$ is the sheaf of ideals of germs vanishing on $M$. It is a coherent sheaf since $M$ is a manifold.
\smallskip

Let $T$ be a connected component of a global semianalytic set $S\subset M$. If $T$ has dimension $0$ it is a point, hence it is global. If $T$ has dimension $1$, if $S=\bigcup_{j=1}^t\{f_j=0, g_{1,j}>0,\cdots, g_{k(j),j} >0\}$ and $f$  is the product of all  $f_j$ that are not identically $0$, we have $T\subset \{f=0\}$, so $T$ is global  by Proposition \ref{dim1}. So, we can assume that $T$ has dimension $2$. Also we can assume $S$ not to be closed. Indeed unions of  connected components of a closed $S$ are closed sets disjoint from their complement in $S$. Hence it is easy to separate them  from their complement by an analytic function so that they are global semianalytic sets.
\smallskip

We get the following result.

\begin{thm}\label{cc}
Any union $T$ of connected components of a global semianalytic set $S$ in a $2$-dimensional manifold $M$ is a global semianalytic set.
\end{thm}
\begin{proof}
The strategy to prove this result  is as follows.

First of all we search a global analytic functions $g$ such that both

$$\overline{T\cap \{g>0\}}\cap \overline{(S\setminus T)\cap  \{g>0\}}   \mbox{\rm   \ and \  } 
\overline{T\cap \{g<0\}}\cap \overline{(S\setminus T)\cap  \{g<0\}}$$
 are discrete sets.

Since $T= (T\cap \{g>0\}) \cup (T\cap \{g=0\})\cup (T\cap \{g<0\})$ and $T\cap \{g=0\}$ is global because it is a semianalytic subset of an analytic curve, we are reduced to prove that the other two pieces are global. 

So, it is enough to prove the following statement.

\begin{quotation}
{\em  Let $S\subset M$ be a global semianalytic set and $T\subset S$ be a union of connected components of $S$. If the set $D =\overline T \cap \overline{S\setminus T}$ is discrete then, $T$ is a global semianalytic set.}
\end{quotation}

 In general $\overline T \cap \overline{S\setminus T}$ is a semianalytic set $F$ of dimension $1$ contained in the curve $Y$ whose equation is the product af all the functions appearing in a given description of $S$. So $F$ is global. Consider  the discrete set  $A= F\setminus \stackrel{\circ}{F}$. 
   
The subsheaf of ideals $\Ii_Y$ of $\Oo_M$ of germs vanishing on $Y$ is a coherent locally principal ideal sheaf. So $H^0(M, \Ii_Y)$ is finitely generated, say, by global analytic functions $h_1, \ldots, h_k$. For each irreducible component $Y^l$ of $Y$ take a regular point $a_l \in Y^l \setminus \bigcup_{j\neq l}Y^j$. The fiber of $\Ii_Y$ at $a_l$ is principal, generated by one $h_i$ among $h_1, \ldots, h_k$.    

Next  consider the function $\lambda_1h_1 +\cdots +\lambda_kh_k$, where $(\lambda_1,\ldots,\lambda_k) \in \R^k$. 

The subset of $\R^k$ of those parameters such that the germ of $\lambda_1h_1 +\cdots +\lambda_kh_k$ at $a_l$ is not in $\Ii_{Y,a_l}$ is an hyperplane of $\R^k$.
Indeed assume $\Ii_{Y,a_l}= (h_1)_{a_l}\Oo_{M,a_l}$. So there are units $u_i, i=1,\ldots, k$ such that $(h_i)_{a_l}= u_i(h_1)_{a_l}$, hence $(\lambda_1h_1 +\cdots +\lambda_kh_k)_{a_l} =(\lambda_1u_1 +\cdots +\lambda_ku_k) (h_1)_{a_l}$. 
So, the germ at $a_l$ of the sum $\sum_i\lambda_i h_i$ generates $\Ii_{Y,a_l}$ if and only if $\lambda_1u_1 +\cdots +\lambda_k u_k$ is a unit, that is, $(\lambda_1u_1 +\cdots +\lambda_ku_k)(a_l) \neq 0$. 

We can use Baire's  theorem and find $(\lambda_1,\ldots,\lambda_k) \in \R^k$ such that $\sum \lambda_ih_i$ generates $\Ii_{Y,a_l}$ for all $l$. But this implies that this sum generates all the fibers of $\Ii_{Y^l}$ except for a discrete countable set of points $D^l$. So we get.

\begin{lem}\label{discreto}
Take $D= \bigcup_l D^l \cup \Sing(Y)$
and $g= \sum_i\lambda_ih_i$.
Then both sets
$$\overline{T\cap \{g>0\}}\cap \overline {(S\setminus T)\cap\{g>0\}} \mbox{\ and\ } 
\overline{T\cap \{g<0\}}\cap \overline {(S\setminus T)\cap\{g<0\}}$$
are   subsets of $A\cup (D\cap F) =D'$ which is a discrete set.
\end{lem}
\begin{proof}
take  $a\in F\setminus D'$. Then $a$ belongs to the interior of $F$ and is regular for the curve $Y$. The function $g$ generates the fiber $\Ii_{y,a}$ hence it changes sign at $a$, that is  a small neighbourhood $U_a$ is the union $U_a= \{g>0\} \cup Y\cup \{g<0\}\cap U_a$. So, if $a\in \overline {T\cap\{g>0\}} \cap U_a$ then $(S\setminus T)\cap \{g>0\} \cap U_a =\varnothing$. The same works for  $a\in \overline {T\cap\{g<0\}}$
\end{proof}

This fact is not enough to separate $T$ from $S\setminus T$, not even locally. Nevertheless it is possible to separate germs meeting only in one point by a suitable global semianalytic set and this is done in the following lemmas. To prove the first one we use the bounds on complexity of semianalytic set germs in dimension $2$, proved by Br\"ocker.  We will use also a result of Br\"ocker that was given for a couple of basic closed semialgebraic sets, where at least one has dimension $\leq 2$. But his argument works whenever there is a one-to-one correspondence between the boolean algebra of costructible sets in the real spectrum of the ring and the boolean algebra of sets definible by the ring and this is the case for the ring $\Oo_{M,p}$ and the boolean algebra of semianalytic set germs.

\begin{lem}\label{locale} Let $p$ be a point in $M$ and $F_1, F_2$ be closed semianalytic set germs at $p$ such that $F_1\cap F_2= \{p\}$. Then, there is an open semianalytic set germ $S_p$ such that:
\begin{itemize}
\item $F_1 \subset S_p\cup \{p\}$,
\item $F_2 \cap S_p =\varnothing$.
\end{itemize} 
Moreover $S_p$ can be written as 
$$ S_p = \bigcup_{i=1}^4 \{a_{i,1} >0, a_{i,2}>0\}$$ 
\end{lem}

\begin{proof}By the finiteness property the closed germs $F_1,F_2$ are finite union of closed basic sets, $F_1 = \bigcup_{i=1}^q A_i, F_2 = \bigcup_{j=1}^s B_j$.

By the quoted result of Br\"ocker for each couple $A_i, B_j$ there is an analytic function germ $f_{i,j}$ such that it is $\geq 0$ on $A_i$, it is $\leq 0$ on $B_j$ and $\{f_{i,j} =0\} \cap (A_i \cup B_j) = \{p\}$.  

Then we can take $S_p = \bigcup _{i=1}^q \{f_{i,1}>0, \ldots, f_{i,s}>0\}$.
Using the bounds on complexity of semianalytic set germs in dimension $2$  we get the desired expression.
\end{proof}

Now  we want to prove that we can move the functions defining $S_p$ up to a certain order without loosing the  properties of Lemma \ref{locale}.

\begin{lem}\label{m-adic} Let $F$ be a closed semianalytic set germ and $S$ be an open germ defined as 
$$ S = \bigcup_{i=1}^q \{f_{i,1} >0, \ldots, f_{i,s_i} >0 \}$$
We assume $F \subset S \cup \{p\}$. Then, there exist $t\in \N$ such that if $f'_{i,j} -f_{i,j} \in \mathfrak M_p^t$ (where $\mathfrak M_p$ is the maximal ideal of germs vanishing at $p$) and $S' =\bigcup_{i=1}^q \{f'_{i,1} >0, \ldots, f'_{i,s_i}>0\}$ 
we get again  $F \subset S' \cup \{p\}$.   
\end{lem}
\begin{proof} Assume first $S$ to be basic, that is $S =\{f_1>0, \ldots, f_s>0\}$. Since $F \subset \{f_i>0\} \cup \{p\}$, we can work with only one function $f$, find $t$ and then take the biggest over $i=1,\ldots,s$.

The situation being local,  we can put $M=\R^2, p=0 $ and $\Oo_{M,p}= \R\{x,y\}$. Up to a linear change of coordinates, using Weierstrass Preparation Theorem, we can assume $f$ and all germs defining $F$ are polynomials in $y$ with coefficients in $\R\{x\}$. 

Consider now the curve germ $F\setminus \{0\}$. Each one of its components is the zeroset of a root of a polynomial with coefficients in $\R\{x\}$, that is a Puiseux series. So we get Puiseux series $\eta_1(x), \ldots ,\eta_m(x)\in \R\{x^{1/q}\}$, for a certain positive integer $q$. Also we get the roots $\xi_1(x), \ldots, \xi_r(x) \in \R\{x^{1/q}\}$ of $f$. Since $F\setminus \{0\} \subset S$ we get $\xi_j \neq \eta_i$ for each $i,j$. Now roots are continuous with respect to the coefficients of the polynomial, so we can find a natural number $t_0$ such that if $f'-f \in \mathfrak M^{t_0}$ then the $\mathfrak M$-adic distance between the roots of $f'$ and the corresponding ones of $f$ is less than the distance between $\eta_i$ and $\xi_j$ for all $i,j$.  
Consider now the connected components $F_1,\ldots, F_n$ of $F\setminus \{0\}$ and take a curve germ $\beta_i(t) = (x_i(t), y_i(t)) \in (\R\{t\})^2$ in $F_i \cup 
\{0\}$. Then $f(x_i(t), y_i(t))>0$ for $t>0$ sufficientely small. But this fact depends only in a truncation of $f$. So for each $i$ there is a $t_i$ such that $f'-f \in \mathfrak {M}^{t_i}$ implies  $f'(x_i(t), y_i(t))>0$. Let $t=\max(t_0, \ldots ,t_n)$. Then if $ f'-f \in \mathfrak {M}^{t}$ we get $F\setminus \{0\}\subset \{f' >0\}$ because the condition on the roots implies $f'$ has constant sign on the components $F_1,\ldots, F_n$.

As for the general case, write $S= S_1\cup \ldots \cup S_l$ where all $S_i$ are basic open sets. Then we can decompose $F$ as a union of closed sets $F_i$ such that $F_i\setminus \{0\}\subset S_i$, $F_i =\overline{F\setminus \{0\}\cap S_i}$.
Then for each $i$ we find $t_i$ as before. The biggest among these numbers works for all $i$'s and the result follows.    
\end{proof}

\smallskip

Of course we will apply the previous lemma to all points in the discrete set $D$ of the statement above. But before we have to define globally a neighbourhood basis for $D$.

\begin{lem}\label{multiloc} For each $p\in D$ fix an open ball $B(p,\varepsilon)$ in such a way that $B(p,\varepsilon) \cap B(q,\varepsilon) =\varnothing$ and take $U_{p,\varepsilon} = B(p,\varepsilon)\cap M$. Then for any $\varepsilon$ there is an analytic function $h_\varepsilon$ on $M$ such that
$$\bigcup_{p\in D}U_{p,\varepsilon/2} \subset \{h_\varepsilon >0\}, \quad M\setminus \bigcup_{p\in D}U_{p,\varepsilon} \subset  \{h_\varepsilon <0\}$$    
\end{lem}
\begin{proof}By construction the two sets have disjoint closures. Hence there is a smooth function separating them and $h_\varepsilon$ is just a suitable analytic approximation.
\end{proof} 
\smallskip

Next we have to globalize the germs $S_p, p\in D$.

\begin{lem}\label{globale} There are global analytic functions $A_{i,1},A_{i,2}, i=1,2,3,4$ such that the germ at $p\in D$ of the  global semianalytic set $G=  \bigcup_{i=1}^4  \{A_{i,1}>0, A_{i,2}>0\}$ verifies the conditions of Lemma \ref{locale} for all $p\in D$ with respect to the semianalytic set germs $\overline {T}_p$ and $\overline{S\setminus T}_p$ 
\end{lem}
\begin{proof} 
 Remember that for each $p\in D$ we found an integer $t_p$ for which the conclusion of Lemma \ref{m-adic} hold true.
We define a sheaf $\Gg$ on $M$ by putting $\Gg_p = \mathfrak M_p^{t_p}\Oo_{M,p}$ if $p\in D$, $\Gg_x =\Oo_{M,x}$ otherwise.  Then we define $2$ sections of $\Oo_M  /\Gg$ as $s_i(p) =0$ if $p\notin D, s_i(p)= a^p_{k,i} +\mathfrak M_p^{t_p}$ if $p\in D, i=1,2$, where $a^p_{k,1},a^p_{k,2}$ are the germs in the description of $S_p$ in Lemma \ref{locale}. By Theorem B there are global functions $A_{k,i}, i=1,2$ inducing the given sections. Consider the set $G= \bigcup_k \{A_{k,1}>0,  A_{k,2}>0 \}$. Then its germ at $p\in D$ verifies Lemma \ref{locale} with respect to the sets $\overline {T}_p$ and $\overline{S\setminus T}_p$.     
\end{proof}

We have got the following statement:
\begin{quotation}
  $\overline {T\cap \{h_\varepsilon >0\}} \subset G \cup D$ and   $\overline{S\setminus T}\cap G=\varnothing$
\end{quotation}

Hence we got 
\begin{quotation}
 $T\cap \{h_\varepsilon >0\} = \{x\in S: A_{k,1}(x)>0,  A_{k,2}(x) >0, k=1,2,3,4\}$ so, it is global. 
\end{quotation}

We are left with $T\cap \{h_\varepsilon \leq 0\}$. But the closure of this set is disjoint from the closure of $S\setminus T$, so there is a global analytic function $h$ such that $T\cap \{h_\varepsilon \leq 0 \}= \{x\in S : h(x)> 0\}$. 

The proof of Theorem \ref{cc} is complete.  
\end{proof}

Next we come to closure of global semianalytic sets. We prove that the closure of a global semianalytic set in a real analytic manifold $M$ is global as soon as $\mbox{\rm dim }M \leq 3$. The main tool will be a description of the set with {\em almost square free} analytic functions.
We begin with some definitions.

\begin{defn} 

1) A germ $\xi \in \Oo_{M,x}$ is {\em elliptic} if its zeroset has codimension at least $2$.

2) A germ $f \in \Oo_{M,x}$ is {\em almost square free} if its factorization in   $\Oo_{M,x}$ is $f=f_1\cdots f_s\cdot \xi_1\cdots \xi_t$ where $\xi_1,\ldots,\xi_t$ are elliptic germs and the zerosets of all $f_i$ have codimension $1$. It is {\em square free} if moreover all $\xi_j$ are units.

3) A function $f\in \Oo(M)$ is almost square  free (resp. square free) if $f_x$ is almost square free (resp. square free) for all $x\in M$. 
\end{defn}
\smallskip

Let $f$ be an analytic function on $M$ and denote $\ceros (f)$ its zeroset. There are points $x\in \ceros (f)$ where the sign of $f$ changes, that is both germs $\{f>0\}_x$ and $\{f<0\}_x$ are not empty. Denote by $\ceros'(f)$ the set of these points and by  $\ceros''(f)$ its complement. 

As a first step we prove that $\ceros'(f)$ is a global semianalytic set when dim $M =2$.

\begin{prop} Assume dim $M =2$. Then, $\ceros'(f)$ and $\overline{\ceros''(f)}$ are global semianalytic sets.
\end{prop}  
\begin{proof} We will prove that both sets are union of irreducible components of the C-analytic set $\ceros(f)$. Let $Y$ be one such component. If dim $Y=0$ then of course $Y\subset \ceros''(f)$. If dim $Y=1$ it is of pure dimension $1$ and connected. Denote $D_1 = Y\cap Y'$ where $Y'$ is the union of the irreducible components of $\ceros(f)$ different from $Y$. $D_1$ is a discrete analytic subset of $Y$.

Assume now $(Y\setminus D_1)\cap \ceros' (f) \neq \varnothing$ and take a point $a\in (Y\setminus D_1)\cap \ceros'(f)$. The ideal $\Ii(Y_a)$ is a principal ideal. Applying  Theorem A  we find $g\in \Oo(M)$ such that $g\in \Ii(Y)$ and  $\Ii(Y_a) =g_a\Oo_{M,a}$. The set of points in $Y$ where the germ of $g$ does not generate the fiber $\Ii(Y_b)$ is a discrete set $D_2\subset Y$.

Now $f_a\in g_a\Oo_{M,a}$, but $f$ changes sign at $a$ so, $f_a =u_a g_a^{\alpha_a}$ where $u_a$ is a unit and $\alpha_a$ is odd.

Let us consider the coherent analytic sheaf $\Ff$ defined by

$$ \Ff= (g^{\alpha_a}\Oo_M + f\Oo_M)/f\Oo_M$$
 
The support of this sheaf is a closed analytic set and it is  the set of points
in $x\in \ceros(f)$ where $g^{\alpha_a}\Oo_{M,x}$ is not included in $f \Oo_{M,x}$. Since $a\notin$ Supp$\Ff$, Supp$\Ff \cap Y$ is a discrete set $D_3$. 

Put $D=D_1\cup D_2\cup D_3$. Then $Y\setminus D \subset \ceros'(f)$ because $f$ changes sign at all points of $Y\setminus D$. But $\ceros'(f)$ is closed so, $Y\subset \ceros'(f)$.

Assume now $(Y\setminus D_1)\cap \ceros''(f) \neq \varnothing$. The same argument as before shows that $Y\setminus \ceros''(f)= Y\cap \ceros'(F)$ is a discrete set and so, $Y\subset \overline{\ceros''(f)}$.
The proof is complete.   
\end{proof}

\begin{prop}\label{quasilibera} Let $M$ be a real analytic manifold of dimension $p$. Let $f$ be an analytic function on $M$  and $\ceros'(f)$ be the set of points where $f$ changes sign. Then, there is a function $\tilde f\in \Oo(M)$ with the following properties.
\begin{enumerate}
\item $\tilde f$ is almost square free.
\item $f=\tilde f \sum_{i=1}^q A_i^2, A_i\in  \Oo(M)$.
\item $\ceros'(f)\subset \ceros(\tilde f)\subset \ceros(f)$ and $\mbox{\rm dim }\ceros(\tilde f)_x\leq p-2$ for all $x\in \ceros(\tilde f)\setminus \ceros'(f)$.
\item There is a C-analytic subset $Y\subset \ceros(\tilde f)$ of codimension $\geq 2$ in $M$ such that $\Ii(\ceros(\tilde f)_x)= \tilde f_x\Oo_{M,x}$ for all $x\notin Y$ and $\ceros(\tilde f)\setminus\ceros'(f)\subset Y$
\item $\tilde f$ is unique up to units.
\end{enumerate} 
\end{prop}
\begin{proof}Let $y$ be a point in $M$. Since $\Oo_{M,y}$ is a UFD, a local factorization of $f$ is given by $f_y= g_y^2h_y\rho_y$ where $h_y$ is the product of all factors of $f$ with codimension $1$ zerosets and $\rho_y$ is the product of the ones with higer codimension zerosets, say $h_y =h_1\cdots h_t, \rho = \xi_1\cdots \xi_s$. 

Consider the sheaf $\Ff$ whose fiber at $y$ is $\Ff_y = g_y \Oo_{M,y}$. It is a sheaf because $h_y\rho_y$ does not have quadratic factors, so, up to a unit, it is  the not quadratic part  of the germ $f_z$ for $z$ close to $y$. This implies that $g_y$ defines an analytic function $g$ in a neighbourhood of $y$ and $f_z =g_z^2h'_z\rho'_z u_z$. Being $\Ff$ locally principal $H^0(M, \Ff)$ is finitely generated by, say, $l_1,\ldots,l_s$. From $(l_1,\ldots,l_s)\Oo_{M,y}= g_y\Oo_{M,y}$ we deduce there is $i\in \{1,\ldots,s\}$ such that $l_i\Oo_{M,y}= g_y\Oo_{M,y}$. This implies that taking $l=\sum_{i=1}^sl_i^2$ we get $l_y = g_y^2v_y$ where $v_y$ is a unit. Hence the quotient $\tilde f= f/l$ is a well defined analytic function which verifies $\tilde f_y = h_y\rho_yw_y$ for a suitable unit $w_y$ for all $y\in M$. 

This proves 1, 2 and 3. To prove 4 we consider the complexification $\tilde M$ of $M$ and a Stein neighbourhood $\Omega$ of $M$ in $\tilde M$ to which $\tilde f$ extends to a holomorphic function $F$. Then the zeroset $\ceros(F)$ is an invariant complex space whose real part is $\ceros(\tilde f) \subset M$. Denote $Y= \Sing \ceros(F)\cap M$. Then $Y$ is a C-analytic set of codimension $\geq 2$ in $M$. Take $x\in \ceros (\tilde f)\setminus Y$. It is a regular point both for $\ceros(\tilde f)$ and $\ceros(F)$ and the ideal of $ \ceros(\tilde f)_x$ is generated by $\tilde f_x$ because $F_x$ generates the complex ideal of $\ceros(F)_x$. So $\tilde f$, and hence $f$, changes sign at all points in $\ceros (\tilde f)\setminus Y$ as wanted.

Statement 5 is clear from the definition of $\tilde f$.
\end{proof}

\begin{cor}
Let $f\in \Oo_M(M)$ be an analytic function on a real analytic manifold.
Assume dim $M \leq 3$. Then, $\ceros'(f)$ is a global semianalytic set.
\end{cor}
\begin{proof} From Proposition \ref{quasilibera} we know that $\ceros'(f)= \ceros'(\tilde f)$ and the set of points where $\tilde f$ does not change sign is the set of points where the dimension of its zeroset drops. This is a semianalytic subset $C$  of $Y$ and dim $Y\leq 1$. So $C$ is global semianalytic and $\ceros'(f) = \ceros(\tilde f) \setminus C$ is also global semianalytic.    
\end{proof}

Next proposition is the key for the closure of a global semianalytic set.

\begin{prop}\label{chiusurabasico}
Let $f_1,\ldots, f_r$ be not identically zero analytic functions on $M$ and let $S= \{x\in M: f_1(x) >0, \dots, f_r(x)>0\}$ be a basic semianalytic set. Put $g_{i,j}= f_if_j, 1\leq i<j\leq r$ and consider $T= \{x\in M: \tilde f_1(x) \geq 0, \ldots , \tilde f_r(x)\geq 0, \tilde g_{1,2}(x) \geq 0, \ldots, \tilde g_{r-1,r}(x)\geq 0\}$. Then, $T\setminus \overline S$ is a semianalytic subset of a C-analytic subset of $M$ of codimension $\geq 2$. 
\end{prop}
\begin{proof}Define $F= \big( \prod_i \tilde f_i\big)\big(\prod_{i<j}\tilde g_{i,j}\big)$. We want to prove that $T\setminus \overline S$ is a subset of $\ceros(F)$ and of the singular part of the complex zero set of $\tilde F$, where $\tilde F$ is the holomorphic extension of $F$ to an open neighbourhood of $M$ in its complexification. This way the statement would be proved. 

So assume for a contradiction that there is a point $x\in T\setminus \overline S$ which is regular as a point of $\ceros (F)$. Hence there is a unique irreducible component $Y$ of $\ceros (F)$ such that $x\in Y$. Assume $Y\subset \ceros (\tilde f_1)$.  If $Y$ is not a subset of $\ceros (\tilde f_i), i>1$ then, since $\tilde f_1$ changes sign at $x$ and $\tilde f_i(x)>0$, we get $x\in \overline S$. 

So there is another function, say $\tilde f_2$ such that $Y\subset \ceros (\tilde f_2)$. But $x$ is regular, hence also $f_2$ changes sign at $x$. This implies $\tilde g_{1,2}$ does not contain $Y$ because it has constant sign at $x$ and in turn this implies that $\tilde g_{1,2}(x)\neq 0$ , hence $ \tilde g_{1,2}(x)>0$. Again, if $\tilde f_i(x)>0$ for all $i>2$, we get $x\in \overline S$. 

So there is another function, say $\tilde f_3$, such that $Y\subset \ceros (\tilde f_3)$. Arguing as before we get that for all neighborhoouds $U$ of $x$, $U\cap \{\tilde f_1>0, \tilde f_2>0, \tilde f_3>0\} \neq \varnothing$ so that $x\in \overline S$. After a finite number of steps we conclude that $x$ must be a singular point.   
\end{proof}

We are ready to prove that in dimension $\leq 3$ the closure of a global semianalytic set is a closed global semianalytic set.

\begin{thm}\label{chiusura}
Let $M$ be a real analytic manifold, dim $M\leq 3$, and $S\subset M$ be a global semianalytic set. Then $\overline S$ is a global semianalytic set. 
\end{thm}
\begin{proof} Let $S$ be described as $S= \bigcup_{i=1}^t\{f_i =0, g_{i,1}>0, \ldots, g_{i,s_i} >0 \}$. It is enough to prove that the closure of each piece is global. So consider the set $S=\{f =0, g_{1}>0, \ldots, g_{s} >0 \}$. 

If $f$ is the zero function, then $S$ is basic open and Proposition \ref{chiusurabasico} applies. So we get that $T\setminus \overline S =C$ is a semianalytic set in a C-analytic set of dimension at most $1$ and so it is a global semianalytic set. Hence $\overline S = T\setminus C$ is a global semianalytic set.

If $f$ is not identically zero it is clear that $\overline S\subset \overline{\{g_1>0, \ldots, g_s >0\}} \cap \ceros(f)$. Denote $S' = \{x\in M:g_1(x)>0, \ldots, g_s(x) >0\}$   

Consider the decomposition $\ceros(f) =\bigcup Y_\alpha$ in irreducible components and denote $L$ the set $\{\alpha: Y_\alpha \cap S' \neq \varnothing\}$. Then, $X= \bigcup_{\alpha\in L}Y_\alpha$ is a C-analytic set and $\overline S = \overline{\bigcup_{\alpha\in L}S'\cap Y_\alpha} = \bigcup_{\alpha\in L}\overline{S'\cap Y_\alpha}\subset \overline S' \cap X$. 
It is enough to prove that $T= \overline S'\cap X\setminus \overline S$ is a global semianalytic set.

Note that $T\subset \bigcup_{\alpha\in L}(\overline S'\cap Y_\alpha)\setminus \overline{S'\cap Y_\alpha}$. Now $(\overline S'\cap Y_\alpha)\setminus \overline{S'\cap Y_\alpha}\subset Y_\alpha$ does not contain any open set $W$ of maximal dimension. Indeed if so, there would be an $i\in\{1,\ldots,s\}$ such that $g_i$ would vanish on $W$, hence on $Y_\alpha$ because $Y_\alpha$ is irreducible. But then we would get $Y_\alpha \cap S' =\varnothing$ in contradiction with the definition of $L$.

So $T$ is a subset of $\bigcup_{\alpha\in L}(\overline S'\cap Y_\alpha)\setminus \overline{S'\cap Y_\alpha}\subset \bigcap_{\alpha \in L}Y_\alpha\cap \ceros(g_1\cdots g_s)$. But being $Y_\alpha$ irreducible its intersection with $\ceros(g_1\cdots g_s)$ has dimension $\leq 1$. So $T$ is semianalytic in an analytic curve, hence it is global as wanted.  

Of course if dim $M=2$ the proof simplifies because in this case dim $T= 0$.     \end{proof}

\begin{remark}\label{basicochiuso}
In the case dim $M=2$ since $T\setminus \overline S$ is a discrete set $D$ it is easy to find a global analytic function on $M$ such that $g$ is strictly negative on $D$ and positive on $\overline S$. So, if $S$ is basic open, $\overline S$ is basic closed. It is enough to add to the inequalities defining $T$ the inequality $g\geq 0$. A similar result holds true for dim $M=3$ but it requires more work, as we see in the next proposition.
\end{remark} 

\begin{prop}\label{basico3} 
Let $M$ be a real analytic manifold of dimension $3$ and $S\subset M$ be a not empty basic open set. Then $\overline S$ is a  union of two  basic closed semianalytic sets .
\end{prop}
\begin{proof}
We know with the usual notations that $\overline S = \{\tilde f_1\geq 0, \ldots,\tilde f_r\geq 0, \tilde g_{1,2}\geq 0, \ldots \tilde g_{r-1,r}\geq 0\} \setminus Y$ where $Y$ is global semianalytic and dim $Y=1$. In particular $\overline S \cap \overline Y= D$ is a discrete set.  We can use a similar argument as the one of Lemmas \ref{locale} and \ref{multiloc}. Assume $M$ to be embedded in $\R^n$ as a closed submanifold. For all $p\in D$ there is a polynomial $h_p\in \R[x_1,\ldots,x_n]$ such that it is $\geq 0$ on $\overline S_p$ and $<0$ on $Y_p\setminus \{p\}$. So there is a small ball $B_p$ such that $B_p\cap \overline S = \{x\in B_p: h_p(x)\geq 0\}$. Also we can assume $B_p\cap B_q =\varnothing$ if $p$ and $q$ are different points of $D$. Applying Theorems A and B we find $H\in \Oo(\R^n)$ such that $H_p= h_p +t_p$ where $t_p$ is a series of order sufficientely big with respect to the degree of $h_p$ for all $p\in D$ in such a way that,   shrinking $B_p$ if needed, we get $B_p\cap \overline S = \{x\in B_p: H_p(x)\geq 0\}$ for all $p\in D$. Take now smaller balls $B'_p\subset \overline{B'_p}\subset B_p$.Then, there is an analytic function $g\in \Oo(\R^n)$ such that $g>0$ on $\bigcup_p\overline{B'_p}$ and $g<0$ on $\R^n \setminus \bigcup_pB_p$. So, $\overline S \cap \{g\geq 0\} = \{H\geq 0, g\geq 0\}$

Finally remark that $\overline S \cap \{g\leq 0\}$ and $\overline Y \cap \{g\leq 0\}$ are  closed disjoint sets, hence there is $b\in \Oo(\R^n)$ such that it is positive on the first and negative on the second. As a consequence we get.

$$ \overline S = \overline S \cap \{g\leq 0\} \cup \overline S \cap \{g\geq 0\}=$$
$$= \{-g\geq 0, b\geq 0, \tilde f_1\geq 0, \ldots \tilde f_r \geq 0, g_{1,2} \geq 0, \ldots , g_{r-1,r} \geq 0\} \cup$$
 $$ \cup \{g\geq 0, H\geq 0, \tilde f_1\geq 0, \ldots \tilde f_r \geq 0, g_{1,2} \geq 0, \ldots , g_{r-1,r} \geq 0\}$$

as wanted.     
\end{proof}

\subsection{General dimension.} 

As said there are not results on closure and connected components of a global semianalytic set of dimension greather than $2$. In this section we prove that both problems have a positive answer  locally, i.e. if $S$ is global semianalytic set in $\R^n$, then, for all points $x\in \overline S \setminus \stackrel{\circ}{S}$ there is an open neighbourhood $U^x$ such that 
\begin{itemize} 
\item $\overline S \cap U^x$ is a global semianalytic set (Theorem \ref{fujitagabrie}).
\item Any union of connected components of $S\cap U^x$ is a global semianalytic set (Theorem \ref{ccruiz}).
\end{itemize}

When  the semianalytic set $S$ has  relatively compact boundary, then both, the connected components and the closure of $S$ are global semianalytic sets.

\bigskip

In what follows we consider an analytic manifold $M$ of dimension $p$ embedded as a closed submanifold in $\R^n$.

We begin with closures.

\begin{thm}\label{fujitagabrie}
Let $S\subset M$ be a global semianalytic set. Then, for all $x\in \overline S \setminus \stackrel{\circ}{S}$ there is a neighbourhood $U_x$ such that $\overline S \cap U_x$ is a global semianalytic set.
\end{thm}

Before proving this theorem we recall  Tarski-Seidenberg's  result on  projections of semialgebraic sets in the form proved by \L ojasiewicz. 

Consider an algebra $\Aa$ of real functions on a real manifold $M$. A subset $B\subset M$ is called {\em costructible} by $\Aa$ if it is described by functions in $\Aa$, namely

$$ B=\bigcup_{i=1}^t \{f_i =0, g_{i,1} >0, \ldots , g_{i,s_i} >0\}$$
where  $f_i, g_{i,j} \in \Aa$.

\begin{thm}\label{TS}  
Let $A$ be costructible in $M\times \R^k$ by the algebra $\Aa[t_1,\ldots, t_k]$. Then its projection to $M$ is costructible by $\Aa$.
\end{thm}

The strategy to prove Theorem \ref{fujitagabrie} is the following. 
First of all it is enough to prove the statement for each piece of $S$, that is we can assume $S$ to be described by one equality $f=0$ and several strict inequalities $g_1>0, \ldots, g_s>0$.  

Then we construct a set in $M\times S$ defined as $Y_{q,c} = \{(x,y)\in M\times S: f(y)=0, g_i(y)> c||x-y||^q\}$ which, for suitable $q$ and any $c$,  has the following property: 
\begin{center}
 $$x \in  \overline S  \Leftrightarrow  x \in \overline{Y}_{x;q,c} = 
\overline{p_2(p_1^{-1}(x)\cap Y_{q,c})}$$ 
\end{center}
where $p_1$ and $p_2$ are the projections on $M$ and on $S$.

The suitable $q$ is found using \L ojasiewicz inequality in a neighbourhood of $x\in \overline S\setminus S$.

Next step is to find $q,q'$ in a neighbourhood of $x$ such that substituting in $Y_{q,c}$ the Taylor expansion $\hat f$ of $f$ up to the order $q'$ and the Taylor expansions $\hat g_i$ of the  $g_i$'s up to the order  $q$ again we get

$$x\in \overline S \Leftrightarrow x\in \overline {T}_{x; q,q', c,c'}$$

where $T_{q,q', c,c'}= \{(x,y) : |\hat {f}| \leq c'||x-y||^{q'}, \hat g_i \geq c||x-y||^q\}$ and $T_{x;q,q',c,c'}=p_2 (p_1^{-1} (x) )\cap T_ {q,q',c,c'}.$

Now $T_{q,q', c,c'} \subset U_x\times U_x$ where $U_x$ is a neighbourhood of $x$ (where the \L ojasiewicz inequality holds true) and it is costructible in the algebra $\Oo(M)[y_1,\ldots ,y_n]$. Hence we can apply Theorem \ref {TS} and we get $\overline S \cap U_x$ is a global semianalytic set.   

{\sc Proof of Theorem \ref{fujitagabrie}.}
 \parindent = 0pt
The proof is done in several steps.

Consider the set $Y_{q,c} = \{(x,y)\in M\times M: y\in S, g_i(y)>c ||y-x||^q\}$ and let $p_1,p_2$ be  the two projections of $Y_{q,c}$ on $M$. Denote $Y_{x;q,c}$ the set 
$p_2(p_1^{-1}(x) \cap Y_{q,c})$.  
\smallskip
 \parindent = 0pt

{\sc Step 1} {\em There is $q_0$ such that for all $q\geq q_0$ and for all positive $c\in \R$ we get $x\in \overline S \Leftrightarrow x\in \overline {Y}_{x;q,c}$.} 

\smallskip

If $x\in S$ then  clearly $x\in \overline {Y}_{x;q,c}$. Also if $x\notin \overline S$ then $x\notin \overline {Y}_{x;q,c}$ because  $Y_{x;q,c}\subset S$. 
We are left with the case $x\in \overline S \setminus S$. 

Consider the set  $\{y\in S: ||y-x||=t\}$. It is not empty. 
Define
$$\delta(x,t) = \max_{y\in S, ||y-x||=t}(\min_i g_i(y)).$$

$\delta$ is positive on $\{y\in S: ||y-x||=t\}$ because $g_i(y) >0, i=1,\ldots,s$. Since $\{y\in \overline S: ||y-x||=t\}$ is compact there is a point $y_t$ such that $\delta$ has its maximum value at $y_t$. $\delta(x,0) =0$ so $\delta$ verifies a \L ojasiewicz inequality for $t$ sufficientely small, namely $\delta(x,t)> ct^{q_0}$. Here is where $q_0$ appears.

Now the points $y_t$ are in $Y_{x;q,c} $ because of the definition of $\delta$ and converge to $x$ for $t\to 0$ hence $x\in \overline {Y}_{x;q,c}$ for all $q\geq q_0$ and for all $c>0$.

\smallskip

Substitute in the definition of $Y_{x;q,c}$ the functions $g_i$ with their Taylor expansions $\hat {g}_i$ of order $q\geq q_0$ at the point $x$.
More precisely put
 
$Z_{q,c} =\{(x,y) \in M\times M:  \hat g_i(x,y)>c ||y-x||^q, i=1,\ldots ,s \}$, 

$$Z_{x; q,c} = p_2(p_1^{-1}(x) \cap Z_{q,c}), \quad\quad Z'_{x; q,c} = Z_{x; q,c} \cap \{f=0\}.$$

{\sc Step 2} 
 {\em $x\in \overline S \Leftrightarrow x\in \overline {Z'}_{x; q,c}$.} 
\parindent = 0pt 

 \smallskip

 $(\Leftarrow)$ If $x\in \overline S$, by Step 1, $x\in \overline{Y}_{x,q,2c}$ and there is a sequence $\{y_n\}$ in $Y_{x,q,2c}$ converging to $x$. 

We have $|g_i(y) -\hat g_i(x,y)| \geq C||x-y||^{q+1}$. But for $t$ sufficientely small we have $c ||x-y||^q > C ||x-y||^{q+1}$, so $| \hat g_i(x,y)|> \half (2c) ||x-y||^q$. This implies $y_n \in Z'_{x; q,c}$ because $y_n \in S$ and so $x\in \overline{Z'}_{x; q,c}$.
\smallskip

$(\Rightarrow)$  $x\in \overline{Z'}_{x; q,c}$. It is enough to prove $x\in \overline{Y}_{x;q, \frac{c}{2}}\subset \overline S$ by Step 1.

Since $x\in \overline{Z'}_{x; q,c}$, there is a sequence $\{y_n\}$ in $Z'_{x; q,c}$ converging to $x$. But  $|g_i(y) -\hat g_i(x,y)| \geq C||x-y||^{q+1}$ and we fixed $c$ in such a way that $c||x-y||^q >C||x-y||^{q+1}$. So $g_i(y)$ verifies $g_i(y) >\frac{c}{2} ||x-y||^q$, and so $y_n \in Y_{x;q, \frac{c}{2}}$.

Now we substitute $f$ with its Taylor $\hat f $  up to an order $q' \geq q'_0$ where $q'_0$ is to be defined.
 Define $Y_{q,q',c,c'} = \{(x,y) \in M\times M: \hat {g}_i(x,y)>c||x-y||^q, i=1,\ldots, s, |\hat f(y)|\leq c'||x-y||^{q'}\}$ and as usual  $Y_{x; q,q',c,c'}= p_2(p_1^{-1}(x)\cap Y_{q,q',c,c'})$.
\smallskip

{\sc Step 3} $x\in \overline S \Leftrightarrow x\in \overline {Y}_{x; q,q',c.c'}$.

\smallskip

$(\Rightarrow)$ If $x\in \overline S$, by Step 2 $x\in \overline{Z'}_{x; q,c}$. and $Z'_{x; q,c} \subset Y_{x; q,q',c,c'} \forall c'$.   
\smallskip

$(\Leftarrow)$ We have to prove $x\notin \overline S \Rightarrow x\notin \overline{Y}_{x; q,q',c,c'}$.

This is clear if $f(x) \neq 0$. If $f(x) =0$ but $x\notin \overline S$ there is at least one $i$ such that $g_i(x)<0$ so $x\notin \overline { Y}_{x; q,q',c,c'} $.

We are left with the case $f(x)=0$, $\min_i g_i(x)=0$.
\smallskip

Put $\displaystyle V= \{x\in M: f(x) =0, \min_i g_i(x) =0, x\in \overline {Z}_{x;q,c}\}\setminus \overline S$. We have to prove $x\in V\Rightarrow x\notin \overline{ Y}_{x; q,q',c,c'}$.

Define,  for $x\in V$, $\displaystyle \delta'(x,t) = \min_{y\in Z_{x;q,c}, ||x-y||=t}|f(y)|$.

On the compact set $\{y\in M:||x-y||=t, \hat g_i(x,y)\geq c||x-y||^q, i=1,\ldots ,s \}$ there is a point $y_t$ where $\delta '$ has its minimum value. 

We claim this minimum value is positive, that is $\delta '(x,t)>0$.
If not, there would be $x\in V$ such that $\delta'(x,t) =0$. We would get a sequence $\{y_n\}$ in $Z_{x;q,c}$ converging to $x$ as $t\to 0$ with $f(y_n) =0$ so $y_n\in Z'_{x;q,c}$ hence $x \in \overline {Z'}_{x;q,c}\Rightarrow x\in \overline S$,
and this is impossible because $V\cap \overline S =\varnothing$.
\smallskip

Since $\delta'(x,t)>0$ for $x\in V$ there is $q'_0$ such that for all $q'\geq q'_0$ and for all $c'>0$ we get $\delta'(x,t) >c' \ t^{q'}$ for $t$ sufficientely small and this implies $x\notin \overline { Y}_{x; q,q',c,c'} $ because in $ Y_{x; q,q',c,c'}$ we have $|f(x)|<c'\ t^{q'}$.

\smallskip
 
We have got the integer $q_0'$.
Next we substitute $f$ with its Taylor series $\hat f$ of order $q'\geq q'_0$ at the point $x$.

We put $T_{q,q',c,c'} = \{(x,y) \in M\times M: \hat g_i(x,y)>c||x-y||^q, i=1,\ldots,s, |\hat f(x,y)|<c'||x-y||^{q'}\}$ and $T_{x;q,q',c,c'} = p_2(p_1^{-1}(x)\cap T_{q,q',c,c'})$.

\smallskip

{\sc Step 4 } $x\in \overline S \Leftrightarrow x\in \overline {T}_{x;q,q', c,c'}$.

\smallskip

The proof is similar to the one of {Step 2}. 

$(\Rightarrow)$ $x\in \overline S$ implies $x\in \overline {Y}_{x; q,q',c,c'} $ for all $c,c'$ by Step 3. 
Assume $|f(y)- \hat f(x,y)| <C'||x-y||^{q'+1}$ and take $c'$ such that $c'||x-y||^{q'}> C'||x-y||^{q'+1}$. Then if $x\in \overline {Y}_{x; q,q',c,2c'}$ there is a sequence $\{y_n\}$ in $Y_{x; q,q',c,2c'}$ converging to $x$ but $|f(y_n)|< 2c'||x-y||^{q'}$ implies $|\hat f(y_n)| < c'||x-y||^{q'}$, hence $y_n \in T_{x; q,q',c,c'}$ and $x\in \overline {T}_{x; q,q',c,c'}$ as wanted.

\smallskip

$(\Leftarrow)$  Assume $x\in \overline {T}_{x; q,q',c,c'}$. It is enough to prove $x\in \overline {Y}_{x;q,\varepsilon}$ for $q\geq q_0$ and some $\varepsilon >0$. We have a sequence $\{y_n\}$ in $T_{x; q,q',c,c'}$ converging to $x$. 

Remember that 
$|g_i(y) -\hat g_i(x,y)| <c||x-y||^{q+1}$. 

If $2\varepsilon ||x-y||^q > c||x-y||^{q+1}$ then $g_i(y)>\varepsilon ||x-y||^q$. So, $y_n \in Y_{x;q,\varepsilon}$ and $x\in \overline {Y}_{x;q,\varepsilon}$. 

This implies $x\in \overline S$ by Step 1.

{\sc Conclusion} The previous steps imply that in a small neighbourhood $U_{x_0} = \{y\in M: ||x_0-y||<t_0\}$  of $x_0\in  \overline S\setminus S$ we have $ x\in \overline S \cap U_{x_0}$ if and only if  $(x,y)\in  \overline {T}_{q,q',c,c'} \cap U_{x_0}\times U_{x_0}$ for all $y\in p_1^{-1}(x)\cap T_{q,q',c,c'}\cap U_{x_0}\times U_{x_0}$.

Since  $T_{ q,q',c,c'}$ and its closure are semialgebraic sets and $T_{q,q',c,c'}$ is constructible by $\Oo (M)[t_1,\ldots,t_n]$ by Theorem \ref{TS}, we get $\overline S\cap U_{x_0}$ is costructible by  $\Oo (M)$, because the condition {\em being in $\overline S$} is equivalent to a first order sentence in $f,g_1,\ldots, g_s$ and their derivatives up to the order $\max \{q,q'\}$. 

This implies that $\overline S\cap U_{x_0}$ is a global semianalytic set. 
 \qed

\smallskip

Next we prove a similar result concerning connected components.
\smallskip

\begin{thm}\label{ccruiz}
Let $S\subset M$ be a global semianalytic set. For each $x\in M$ there is a global semianalytic neighbourhood $U_x$ of $x$ such that the connected components of $S\cap U_x$ are global semianalytic sets.   
\end{thm}
\begin{proof} 
We can assume, as in the proof of Theorem \ref{fujitagabrie}, that $S=\{f=0, g_1>0, \ldots ,g_s>0\}$.
Fix a point $x_0\in M$ and consider the set 

$$Y_{q,c} = \{y\in S: g_i(y) > c ||y-x_0||^q, \, i=1,\ldots,s\}$$

Note that this set is not the same as in Theorem \ref{fujitagabrie}. Indeed since $x_0$ is fixed we do not need to consider the product $M\times S$. We determine an integer $q_0$ as in Step 1 of Theorem \ref{fujitagabrie} and a $t_0$ such that we get.

\begin{quotation}
$y_1,y_2$ are in the same connected component of $ \{z\in S: ||z-x_0||<t\leq t_0\}$ if and only if they are in the same component of $Y_{q,c} \cap \{||z-x_0||<t\leq t_0\}$.
\end{quotation}

\smallskip

Indeed if so there is an analytic arc $\gamma: [0,1] \to S\cap \{z\in S: ||z-x_0||<t\}$ with $\gamma(0)=y_1, \gamma(1)=y_2$. The same arc lives in $Y_{q,c}$ by the choice of $q_0\leq q$ and conversely.
\smallskip

Next determine $q_0'$ as in Step 3 of Theorem \ref{fujitagabrie} and define for $q\geq q_0 , q'\geq q_0'$
 
$$Y_{q,q',c,c'} = \{y\in M: |f(y)|< c'||y-x_0||^{q'}, g_i(y) >c ||y-x_0||^q, i=1,\ldots,s\}.$$

Again both $y_1, y_2$ are in  $ Y_{q,q',c,c'}$ and the arc $\gamma$ takes values in $Y_{q,q',c,c'}\cap \{f=0\}$, because $Y_{q,c} \subset Y_{q,q',c,c'}$. 

Now approximate $g_i, i=1,\ldots,s$ and $f$ by their Taylor expansion $\hat g_i, \hat f$ at the point $x_o$ up to the order $q$ resp. $q'$ and consider the set 

$$T_{q,q',c,c'} = \{|\hat f| < c'||y-x_0||^{q'}, \hat g_i >c ||y-x_0||^q, i=1,\ldots,s\}.$$ 

For $t< t_0$ sufficientely small one has $|f(y)-\hat f(y)| < C'||y-x_0||^{q'+1}, |g_i(y) - \hat g_i(y)| <C ||y-x_0||^{q+1}$. So, if $c,c'$ are sufficientely small $y_1,y_2 \in T_{q,q',c,c'}$ together with the arc $\gamma$ as soon as $||y_j-x_0||
<t, j=1,2$. This implies $y_1, y_2$ are in the same connected component of $T_{q,q',c,c'}$ or, better, they are in the same connected component of $T_{q,q',c,c'}\cap \{f=0\}$.

Conversely if $y_1, y_2$ are in the same connected component of $T_{q,q',c,c'}\cap \{f=0\}$ they are in the same connected component of  $ Y_{q,q',c,c'}$, because 
$\hat g_i(y) >c||y-x_0||^q$ implies $g_i(y)>c||y-x_0||^q$ and $f(y_1) = f(y_2) =f(\gamma(t))=0$. This in turn implies $y_1,y_2$ are in the same connected component of $Y_{q,c}$ and hence of $S\cap \{||y-x_0||<t\}$.

Now a connected component of $T_{q,q',c,c'}$ is given by a first order formula on $g_1, \ldots , g_s, f$ and their derivatives up to a finite order, so it is a global semianalytic set. Hence $T_{q,q',c,c'}\cap \{f=0\}$ is also global semianalytic. Consequently the condition {\em to belong to a connected component of $S\cap \{||y-x_0||<t\}$} is equivalent to a first order formula on $g_1,\ldots,g_s, f$ and their derivatives plus the condition $f=0$. This proves that connected components of $S\cap \{||y-x_0||<t\}$ are global semianalytic sets. 
\end{proof}
\begin{thm}\label{bb} Assume $S\subset M$ is a global semianalytic set such that its boundary $\overline S\setminus \stackrel{\circ} S$ is compact. Then $\overline S$ is a global semianalytic set and any union of connected components of $S$ is a global semianalytic set.
\end{thm}
\begin{proof} By Theorem \ref{fujitagabrie}  for each $x\in \overline S\setminus \stackrel{\circ} S$ there is a neighbourhood $U_x$ such that $\overline S\cap U_x$ is a global semianalytic set. $\bigcup_x U_x$ is an open covering of $\overline S\setminus \stackrel{\circ} S$, it has a finite subcovering $U_{x_1}, \ldots, U_{x_l}$ and hence there is an open neighbourhood $U$ of $\overline S\setminus \stackrel{\circ} S$ which is a global semianalytic set. Take a smaller neighbourhood $V$ such that $\overline V \subset U$. Then $\overline V$ and $M\setminus U$ are closed disjoint sets. Hence there is a global analytic function $h$ which is negative on $\overline V$ and positive on $M\setminus U$ . So $\overline S = (S\cap \{h>0\}) \cup (\overline S\cap U)$ is global semianalytic.

A similar argument works for connected components.  

Let $T$ be a union of connected components of $S$. Then the argument above shows that there is an open global semianalytic neighbourhood $U$ of $B = \overline S\setminus \stackrel{\circ} S$ such that $T\cap U$ is global semianalytic.  Consider a smaller global neighbourhood $V$ of $B$ with $\overline V \subset U$. Since $T= (T\cap U) \cup (T\setminus  \overline V)$ it is enough to prove that the latter is global. 
Now $S\setminus \overline V$ is a global semianalytic set and $(S\setminus T) \setminus \overline V$ and $T\setminus \overline V$ have disjoint closures. So there is an analytic function $h\in \Oo(M)$ which is positive on $T\setminus \overline V$ and negative on $(S\setminus T) \setminus \overline V$. So $T\setminus \overline V = (S\setminus \overline V) \cap \{h>0\}$ and hence it is a global semianalytic set.
\end{proof}

\section{Strict Positivstellensatz.}

A problem that can be considered in a global setting is the so called {\em Positivstellensatz}.

Let $S \subset \R^n$ be a basic closed semianalytic set, that is $S=\{x\in \R^n: f_1(x)\geq 0,\ldots ,f_k(x)\geq 0\}$ where $f_1,\ldots ,f_k \in \Oo(\R^n)$.
Let $f\in \Oo(\R^n)$ be positive on $S$.  Is there a relation between $f$ and the functions $f_1,\ldots ,f_k$ ?

When $f,f_1, \ldots ,f_k$ are polynomials one has the following result, due to K.Schm\" udgen.

\begin{thm}\label{SCH}  
Assume $S$ to be a compact basic closed semialgebraic set, $S= \{x\in \R^n: f_1(x)\geq 0,\ldots ,f_k(x)\geq 0\}$ and let $f$ be a polynomial strictly positive on $S$. Then, 

$$f= \sum a_\varepsilon f_1^{\varepsilon_1}\cdots f_k^{\varepsilon_k}$$

where each $\varepsilon_i$ is $0$ or $1$ and $a_\varepsilon$ is a sum of squares of polynomials. In other words $f$ belongs to the cone generated by $f_1,\ldots ,f_k$ over the set $\Sigma$ of sums of squares.
\end{thm}

We prove a similar result in the analytic setting with some differences. Namely $S$ is not supposed to be compact and  $f$ is proved to be a linear combination of $1,f_1,\ldots ,f_k$ with coefficients in $\Sigma$.

We denote $\mathcal T$ the cone generated by $f_1,\ldots f_k$ over the set $\Sigma$ of sums of squares of global analytic functions and by $\mathcal M$ the module generated over the squares by $1,f_1,\ldots ,f_k$ .

Assume $S\neq \varnothing$.
Consider on the space $C(\R^n,\R)$ of continuous functions on $\R^n$ the  
Whitney topology. An open  neighbourhood of a function $g$ is given by all functions $f$ verifying $|f(x) - g(x)| < \varepsilon(x)$, where $\varepsilon$ is a strictly positive continuous function on $\R^n$. Remember that $\Oo(\R^n)$ is dense in $C(\R^n,\R)$ with respect to this topology. 

We begin by an easy lemma that will be used several times.

\begin{lem}\label{moltiplicatore}
Let $\varphi$ be a continuous function on $\R^n$ and assume $\varphi$ is strictly positive on a closed set $\Omega \subset \R^n$. Then,  for all  $ \psi$ continuous on $\R^n$, there is a strictly positive continuous function $\varepsilon$ on $\R^n$ such that on $\Omega$ one has $\varepsilon \psi <\varphi$.
\end{lem}  

\begin{proof}
Define $\displaystyle \varepsilon(x)  =\frac{\varphi(x)}{1+\psi(x)}$ on the closed set $\{\psi\geq 0\} \cap \Omega $. 

Since $\varepsilon$ is strictly positive on this set it can be extended to the whole $\R^n$ as a strictly positive function. Then $\varepsilon \psi <\varphi $ on $\Omega$ since it is so on $\{\psi \geq 0 \} \cap \Omega$ and
trivially outside. 
\end{proof}

We prove first the Positivstellensatz when $S$ is a principal closed set.

\begin{prop}\label{principal} 
Let $f,g \in \Oo(\R^n)$ and assume $g$ is strictly positive on $\{f\geq 0\}$. Then there are strictly positive analytic functions $s,t \in \Oo(\R^n)$ such that $g=s+tf$.
\end{prop}
\begin{proof}
Set $F= \{f\geq 0\}$ and $G= \{g\leq 0\}$. Then $F\cap G =\varnothing$. Define 

\begin{center}
\obeylines
$ u= (f+g)/f$ over $G$,
$ u= g/(f+g)$ over $F$,
$ u=1$ otherwise.
\end{center}

It is an easy verification that $u$ is continuous and strictly positive on $\R^n$. Also $g-uf>0 $ on $\R^n$. By Lemma \ref{moltiplicatore} there is $\varepsilon >0$ on $\R^n$ such that $\varepsilon f < g-uf$ on $\R^n$. Take an analytic approximation $t$ of $u$ verifying $|t-u| < {\rm min}\{u/2,\varepsilon\}$. Then $t>0$ and $s= g-tf >0$ as wanted. 
\end{proof}

Next lemma reduces the general case to the one showed above.  

\begin{lem}\label{riduzione} 
Let $S=\{f_1 \geq 0, \ldots , f_k\geq 0\}$. For any open set $U\supset S$ there is $h\in \mathcal M$ such that $S\subset \{h\geq 0\}\subset U$ and we can take $h= \sum s_if_i$  with $s_i >0$ on $\R^n$. 
\end{lem}
\begin{proof}
Put $U_0 = U$ and $U_j = \{f_j < 0\}, j=1,\ldots , k$. Then $\{U_0, U_1, \ldots ,U_k\}$ is an open covering of $\R^n$. Let $\{\varphi_0, \ldots , \varphi_k\}$ be a smooth partition of the unity subordinate to this covering. Then $\varphi = \sum_{j=1}^k \varphi_j f_j$ is a smooth function which is strictly negative on the closed set $\Omega = \R^n\setminus U_0$, that is $\{\varphi \geq 0 \}\subset U_0$.

We can further assume $\varphi_j > 0$ for all $j=1,\ldots , k$. In fact by Lemma \ref{moltiplicatore} there are for  $j=1,\ldots , k$ strictly positive functions $\varepsilon_j$ such that $\varepsilon_j f_j < -\varphi/k$ on $\Omega$. Then on $\Omega$ we get      

$$\varphi +\sum_{j=1}^k \varepsilon_j f_j <\varphi +k(-\varphi/k) =0.$$

Replace $\varphi$ by 
$$\varphi + \sum_{j=1}^k \varepsilon_j f_j =\sum_{j=1}^k (\varphi_j +\varepsilon_j)f_j$$

whose coefficients are strictly positive on $\R^n$.

Thus we may assume $\varphi_j> 0, j=1,\ldots , k$ and we can approximate them by strictly positive analytic functions $r_j$ such that the function $h=\sum r_jf_j$ verifies $S\subset \{h\geq 0\} \subset U$.

More precisely by Lemma \ref{moltiplicatore} we find for $j=1,\ldots , k$ continuous strictly positive functions $\delta_j$ such that $\delta_j|f_j|<-\varphi/2k$ on $\Omega$. Then we take analytic functions $r_j$ such that $|r_j-\varphi_j| < { \min}\{\varphi_j/2,\delta_j\}$. So, $r_j> 0$ and

$$|\varphi - h| =\left| \sum_{j=1}^k \varepsilon_j f_j - \sum_{j=1}^k r_jf_j\right| =
 \left|\sum_{j=1}^k (\varphi_j -r_j)f_j\right| \leq \sum_{j=1}^k |\varphi_j -r_j||f_j| < \sum_{j=1}^k \varphi/2k = \varphi/2$$ 

on $\Omega$. So we get $S\subset \{h\geq 0\} \subset U$ as wanted.
\end{proof}

Finally we get.

\begin{thm}[Strict analytic Positivstellensatz]\label{spss}
Assume $g\in \Oo(\R^n)$ to be positive  on $S= \{f_1\geq 0, \ldots , f_k\geq 0\} \neq \varnothing$. Then $g\in \mathcal M$. Moreover we get $g= s_0^2 + s_1^2f_1 +\cdots + s_k^2 f_k$  with $s_j \in \Oo(\R^n)$  strictly positive on $\R^n$ for $j=1,\ldots , k$.
\end{thm}
\begin{proof}
By Lemma \ref{riduzione} there exist an analytic function $h= \sum r_jf_j$ such that $S \subset \{h\geq 0\} \subset \{g> 0\}$ and $ r_j >0$ on $\R^n$ for $j=1,\ldots , k$.  By Proposition \ref{principale} we get $g=s+th$ and $s,t >0$ on $\R^n$. Substituting $h$ by its expression we get the thesis. 
\end{proof}

\section{C-semianalytic sets.}

\rm
We consider now  a class of globally defined semianalytic subsets of a real analytic manifold $M$.

\begin{defn}\label{Csemi} 
 A subset $S\subset M$ is {\em C-semianalitic} if $S$ is a locally finite union of global basic semianalytic sets, that is, sets of the form $\{f=0,g_1>0,\ldots,g_s>0\}$ where $f,g_j\in\Oo(M)$. 
\end{defn}

The previous definition is equivalent to the following one, which is more similar to the one provided by \L ojasiewicz for classical semianalytic sets.

\begin{defn}\label{Csemibis} 
 A subset $S\subset M$ is {\em C-semianalytic} if for each $x\in M$ there is an open neighborhood $U^x\subset M$ such that $S\cap U^x$ is a global semianalytic set in $M$ (in the sense of Section \ref{gas}).
\end{defn}

To prove the equivalence we need  an easy lemma on countable locally finite refinements of open coverings, in the spirit of Lemma \ref{riduzione}. 

\begin{lem}\label{refinement}
Let $M$ be a real analytic manifold and ${\mathfrak U}= \{U_i\}_{i\in I}$ be an open covering of $M$. Then there exists a countable locally finite refinement ${\mathfrak V} = \{V_j\}_{j\geq 1}$ where for all $j$ $V_j = \{g_j >0\}$ for some analytic function $g_j\in \Oo(M)$.
\end{lem}
\begin{proof}
Since $M$ is paracompact there exist countable locally finite open refinements ${\mathfrak W}=\{W_j\}_{j\geq 1}$ and $ {\mathfrak W'}=\{W'_j\}_{j\geq 1}$ such that $\overline {W'_j} \subset W_j$ for all $j\geq 1$. Now the closed sets $ \overline {W'_j}$ and $M\setminus W_j$ are disjoint, hence there is a continuous function $\eta_j$ which takes the value $1$ on $ \overline {W'_j}$ and the value $-1$ on $M\setminus W_j$. So it is enough to take an analytic approximation $g_j$ of $\eta_j$ such that $|g_j - \eta_j| <1/2$. Then  $ \overline {W'_j}\subset V_j=\{g_j >0\} \subset W_j$. Thus ${\mathfrak V}= \{V_j\}_{j\geq 1}$ is the refinement we asked for.
\end{proof}

{\sc Proof of the equivalence of the two definitions.}
Assume first $S$ satisfies Definition \ref{Csemi}. Let $x\in M$ and assume $x\in U = \{g >0\}$ where $U$ intersect only finitely many $S_i$, say $S_1,\ldots,S_r$. Thus, $S\cap U = \bigcup^r_{i=1} S_i \cap \{g>0\}$  is a global semianalytic set.
\smallskip

Conversely assume $S$  satisfies Definition \ref{Csemibis}.
For each $x\in M$ let $W_x$ be the open neighbourhood of $x$ such that $S\cap W_x$ is a global semianalytic set. Then, by Lemma \ref{refinement} there is countable locally finite refinement ${\gtV} = \{V_j\}_{j\geq 1}$ of $\{W_x, x\in M\}$ such that $V_j= \{g_j >0\}$ for some $g_j \in \Oo(M)$. Assume $V_j\subset W_{x_j}$ and let $S'_{j_1}, \ldots, S'_{j_{r_j}}$ be  C-basic semianalytic sets such that $S\cap W_{x_j} =  S'_{j_1}\cup \ldots \cup S'_{j_{r_j}}$. Then, $S\cap V_{j} =  S_{j_1}\cup \ldots \cup S_{j_{r_j}}$ where $S_j = S'_j\cap V_j = S'_j \cap \{g_j >0\}$. So, $S_j$ is a basic C-semianalytic set and

$$S= S\cap \bigcup_{j\geq 1}V_j = \bigcup_{j\geq 1} S\cap V_j$$ 

where $S\cap V_j$ is a basic C-semianalytic set and the union is locally finite because so is the open covering $\gtV$.
\qed

\begin{examples}\hfill

\begin{enumerate}
\item[(i)] Consider the real analytic set $X$ of the Example \ref{a(z)} of Chapter 1, that is $X$ is the zeroset of $f= z(x^2+y^2)-x^3a(z)$.

\noindent Remember the definition of $a(z)$.
$$
a(z)=\begin{cases} \exp{\frac{1}{z^2-1}} &\text{ if } -1<z<1 \\ 0 &\text{ if } z \leq -1 \text{ or } z \geq 1 \end{cases}
$$

\noindent Since $f$ 
is analytic in $\R^3\setminus \{z=1\}\cup\{z=-1\}$ and $X\cap \{z^2\geq 1\}$ is contained in the line $\{x=0,y=0\}$, we get that $X$ is semianalytic in $\R^3$ . However $X$ is not a C-semianalytic subset of any of its open neighbourhoods in $\R^3$. 
Fix an open neighbourhood $U$ of $X$ in $\R^3$. We know from Example \ref{a(z)} of Chapter 1 that every analytic function in $\Oo(U)$ vanishing identically on $X$ is the zero function. Assume $X$ to be C-semianalytic. Then there is an open neighbourhood $V$ of $0\in \R^3$ such that $X\cap V$ is a finite union of basic global semianalytic sets. One of them 
$$S=\{h=0, g_1>0, \ldots, g_s>0 \}, \ h,g_1,\ldots ,g_s \in \Oo(U)$$
contains a not empty open set of the connected real analytic manifold $N= X\setminus \{x=0, y=0\}$. As $h$ vanishes on an open set of $N$ it vanishes identically on $N$. As the germ of $X$ at $0$ and the line $\{x=0, y=0\}$ are irreducible, $h$ is identically zero on $X\cap V$, hence it is the zero function of $\Oo(V)$, and consequentely of $\Oo(U)$, a contradiction.   
 
\item[(ii)] 
Consider the closed subset $X\subset \R^3$ defined by the equation,   already considered in Chapter 1, Examples \ref{a(z)}, 
$$\quad \quad (1-4(x^2+y^2+z^2))((x^2+z^2 - 1)^2 +y^2)=((x^2 +z^2 -1)^4 +y^4)a(z) $$
where  $a(z)$ is the same function as before. Note that again $X$ is semianalytic and it is compact.

\noindent  As before  $X$ cannot be C-semianalytic because no global analytic function can vanish on $X$, without beeing the zero function.   
\end{enumerate}
\end{examples}
\smallskip

Next  we  prove that the class of C-semianalytic sets is closed under boolean and topological operations. Infact we get

\begin{prop}\label{booletopological}
The class of C-semianalytic sets in $M$ is closed under the following boolean and topological operations 
\begin{enumerate}
\item locally finite unions, intersections and complement,
\item inverse image under analytic maps between real analytic manifolds, 
\item taking closure, interior and  connected components.
\end{enumerate}
\end{prop}
\begin{proof}
The first property is clear.

As for the second one assumes to have an analytic map $f:N\to M$ between real analytic manifolds and let $S= \{h=0, g_1>0,\ldots, g_r>0\}$ be a basic C-semianalytic set in $M$. Then, $f^{-1}(S) = \{x\in N: f(x)\in S\} = \{h\circ f =0, g_1\circ f >0, \ldots, g_r\circ f >0\}$ is again a basic C-semianalytic set in $N$.

For the third first we look at {\em closure}. Note that if $M= \bigcup_{j\in J}U_j$, then 
$$\overline S= \bigcup_{j\in J} \overline S \cap U_j =  \bigcup_{j\in J} \overline{(S\cap U_j)}\cap U_j.$$

 Hence it is enough to check that for each point $x\in M$ there is a neighbourhood $U_x$ such that $ \overline{(S\cap U_x)}\cap U_x$ is a global semianalytic set. But this is a consequence of Definition \ref{Csemibis} and  Theorem \ref{fujitagabrie}.  

As for the {\em interior}, the complement of $S$ is C-semianalytic and the interior part of $S$ is the complement of the closure of its complement.

We are left with {\em connected components}.

Let $S$ be a C-semianalytic set and $x\in M$ be a point. Only finitely many connected components of $S$ can be adherent to $x$. The number of such connected components is the number of connected components of the germ $S_x$. So, take a global neighborhood $U_x =\{g>0\}$ of $x$ (if $M$ is embedded in some $\R^n$ as a closed submanifold $U_x$ could be a small ball intersected with $M$). Then $S\cap U_x$ is a finite union of connected global semianalytic sets by Theorem \ref{ccruiz}. This proves that any union of connected components of $S$ is a C-semianalytic set.  
\end{proof}

Concerning closure and connected components there are deeper results that we will use later. Consider a reduced Stein space $(X,\Oo_X)$ endowed with an anti-involution $\sigma$ and assume its fixed set $X^\sigma$ is not empty. Denote by $\mathcal A$ the ring of $\sigma$-invariant holomorphic functions on $X^\sigma$ and consider, for a tuple $u= (u_1,\ldots, u_m)$, the polynomial rings $\Oo(X^\sigma)[u]$ and $\mathcal A[u]$ that we will call both as $\Aa$. A C-semianalytic set $S$ in $X^\sigma \times \R^m$  will be called $\Aa$-definible if all the functions in a description of $S$ are taken from $\Aa$.  Then we get the following.

\begin{prop}\label{sigmainvariante}
Let $S\subset X^\sigma \times \R^m$ be either a bounded global semianalytic set or a C-semianalytic set. Assume $S$ to be $\Aa$-definable. Then
\begin{enumerate}
\item For each point $x\in X^\sigma$ there is an open $\Aa$-definable basic C-semianalytic neighbourhood $U_x$ such that $S\cap U_x$ is the union of $t$ connected disjoint $\Aa$-definable C-semianalytic sets and $t$ is precisely the number of connected components of the germ $S_x$.
\item Connected components and closure of $S$ are $\Aa$-definable global semianalytic sets in the first case, $\Aa$-definable C-semianalytic sets in the second one. 
\end{enumerate}
\end{prop}
\begin{proof} Since the situation is local around $x$ we can assume $X^\sigma\subset \R^n$ and $X\subset \C^n$. Then a neighbourhood basis of $x$ is given by small balls $B(x,\varepsilon) \cap X^\sigma$. These are given by a polynomial inequality and polynomials belong to $\Aa$. If $\varepsilon$ is sufficientely small, the connected components of $S\cap B(x,\varepsilon)$ correspond to the connected components of $S_x$. This proves  (1).

The proof of (2) is done exactly as for Proposition \ref{booletopological}. In the first case being $\overline S$ compact we can use also Corollary \ref{bb}. 
\end{proof}

\bigskip
Open C-semianalytic sets in a real analytic manifold $M$ verify a {\em locally finiteness property} similar to finiteness property of global open semianalytic sets.

\begin{lem}\label{odes}
Let $U\subset M$ be an open C-semianalytic set. Then there exists a locally finite countable family $\{U_j\}_{j\geq1}$ of global open basic semianalytic sets such that $U=\bigcup_{j\geq1}U_j$.
\end{lem}
\begin{proof}
By finiteness property it is enough to prove: \em there exists a locally finite countable family $\{U_j\}_{j\geq1}$ of open global semianalytic sets such that $U=\bigcup_{j\geq1}U_j$\em.
 
For each $x\in M$ pick an open neighborhood $U^x$ such that $U\cap U^x$ is an open global semianalytic set. As $M$ is paracompact there exists a countable locally finite open refinements ${\mathscr W}=\{W_j\}_{j\geq1}$ and ${\mathscr W}'=\{W_j'\}_{j\geq1}$ of ${\mathscr U}=\{U^x\}_{x\in M}$ such that $\ol{W_j'}\subset W_j$ for each $j\geq1$. As  the closed sets $\ol{W_j'}$ and $M\setminus W_j$ are disjoint, there exists an analytic function $g_j:M\to\R$ such that $g_{j|\ol{W_j'}}> 0$ and $f_{j|M\setminus W_j}<0$. So
$$
\ol{W_j'}\subset V_j=\{g_j>0\}\subset W_j
$$ 
for each $j\geq1$. Thus, the family $\{U_j=U\cap V_j\}_{j\geq1}$ is countable, locally finite, its members are open global C-semianalytic sets and satisfies $U=\bigcup_{j\geq1}U_j$, as required.
\end{proof}

The {\em dimension} of a C-semianalytic set $S$ is defined as dim $\displaystyle S= \sup_{x\in M}$ dim $S_x$. This dimension is the same as dim$\overline{S_x}^{\zar}$, where $\overline{S_x}^{\zar}$ means the Zariski closure of $S_x$, that is the smallest analytic set germ containing $S_x$. It is easy to prove that a global semianalytic set and its Zariski closure get the same dimension.

We prove now that the set of points of dimension $k$ of a C-semianalytic set is a C-semianalytic set.

\begin{prop} Let $S\subset M$ be a C-semianalytic set. Denote by $S_{(k)}$ the subset of points of $S$ of local dimension $k$. Then $S_{(k)}$ is a C-semianalytic set.
\end{prop}
\begin{proof} Let $d$ be the dimension of $S$. We first prove that $S_{(d)}$ is C-semianalytic. Take $x\in \overline{S_{(d)}}$ and let $U_x\subset M$ be a neighbourhood of $x$ such that $S\cap U_x = \bigcup_{i=1}^r S_i$ where $S_i$ is global and basic. Assume dim $S_i=d$ exactly for $i=1,\ldots,l$. Let $T_i$ be the union of the connected components of $S_i \setminus {\rm Sing}(\overline {S_x}^{\zar})$ of dimension $d$. So $T_i$ is a C-semianalytic set by Proposition \ref{booletopological}.

The set of points of dimension $d$ of $S_i$ is the closure of $T_i$ intersected with $S_i$. Thus,

$$S_{(d)} \cap U_x= \bigcup_{i=1}^l \overline{T_i}\cap S_i$$

is a C-semianalytic set.

Next consider $S' =S\setminus S_{(d)}$ which is again is C-semianalytic set and let $d'<d$ be the dimension of $S'$, Note that $S_{(d')}= S'_{(d')}$ so, by the previous step it is a C-semianalytic set. Proceeding recursively we get $S_{(k)}$ is a C-semianalytic set for $k\geq 0$ as required.
\end{proof}

We end this section with a characterization of C-semianalytic sets of dimension less or equal to $k$.

\begin{prop}\label{BM}
Let $S\subset M$. Then $S$ is a C-semianalytic set of dimension $\leq k$ if and only if for each $x\in M$ there is a C-semianalytic neighbourhood $U_x$ of $x$ in $M$ and a C-analytic set $Z\subset M$ of dimension $\leq k$ such that $S\cap U_x\subset Z$ and the closure of $S\cap U_x$ minus $S\cap U_x$ and $S\cap U_x$ minus the interior part of $S\cap U_x$ are both C-semianalytic sets of dimension less or equal than $k-1$.
\end{prop}

The proof of Proposition \ref{BM} requires a lemma.

\begin{lem}\label{BM1} Let $S\subset T \subset M$ where $T$ is C-semianalytic in $M$. Let $S_1$ (resp. $S_2$) be the closure (resp. the interior part) of $S$ in $T$. Then $S$ is C-semianalytic if and only if $S_1\setminus S$ and $S\setminus S_2$ are C-semianalytic.
\end{lem}
\begin{proof} If $S$ is C-semianalytic  so are $S_1,S_2, S_1\setminus S, S\setminus S_2$. Conversely assume the last two to be C-semianalytic. Then $S_1 \setminus S_2 = (S_1\setminus S) \cup (S\setminus S_2)$ is C-semianalytic. Hence $T\setminus (S_1\setminus S_2)$ is C-semianalytic and it is the union of $T\setminus S_1$ and $S_2$ which are open and closed disjoint sets in a C-semianalytic set. Hence both are C-semianalytic because they are unions of connected components of a C-semianalytic set. Then $S=S_2 \cup (S\setminus S_2)$ is C-semianalytic.
\end{proof}

Recall that if $S\subset M$ is a semianalytic set then $(\overline S)_x = \overline{S_x}$ and dim $\overline{S_x} \setminus S_x <$ dim $S_x$ for each $x\in \overline {S_x}$. We will use this statement for germs in the   proof of Proposition \ref{BM}.

\begin{proof}[Proof of Proposition \ref{BM}] We prove first the "if" part. By Lemma \ref{BM1} $S\cap U_x$ is C-semianalytic. On the other hand $S\cap U_x$ is the union of its interior part in $Z$ and its complement in $S\cap U_x$. Hence its dimension is less or equal to dim $Z \leq k$ while the complement of its interior part has dimension less or equal to $k-1$ because this dimension is equal to the one of its germ.
\smallskip

Next assume $S$ to be C-analytic of dimension $\leq k$. We show first that each point $x\in M$ has an open C-semianalytic neighbourhood $U_x$ with the following properties.

\begin{enumerate}
\item dim $S_x =$ dim $S\cap U_x =$ dim $\overline{S\cap U_x}^{\zar}$.
\item $\overline S\cap U_x$ is a C-semianalytic set.
\end{enumerate}

Since $S$ is C-semianalytic, for each $x\in M$ there is an open neighbourhood such that $S\cap U_x$ is a finite union of global basic semianalytic sets $S_1, \ldots, S_r$, that is $S\cap U_x$ is a global semianalytic set. Up to shrinking $U_x$ we can assume $x$ to be adherent to at least one $S_i$ Also $U_x$ can be taken open basic, hence ${\rm dim} S\cap U_x =$ dim $\overline{S_i\cap U_x}^{\zar}$. So we get:

 $$\dim S_x =\max_{i=1,\ldots,r}\{{\dim}S_i\cap U_x\}=
\max_{i=1,\ldots,r}\{\dim \overline{S_i\cap U_x}^{\zar}\}.$$

Hence dim $S_x ={\dim}(S\cap U_x)= {\dim}(\overline{S\cap U_x}^{\zar})$.

Take $Z= \overline{S\cap U_x}^{\zar}$ and put $A= S\cap U_x$. Then $Z$ and $A$ get the same dimension and dim $\overline A_y\setminus A_y \leq k-1$ for all $y\in \overline A$

We are left with $A$ minus its interior part  $\stackrel{\circ}A=Z \setminus \overline { (Z\setminus A)}$ and again its dimension is equal to the one of the corresponding germ at a point $y\in \overline A$. So we get
$$A\setminus\stackrel{\circ}A =A \cap (Z\setminus \stackrel{\circ}A) =A\cap \overline{Z\setminus A}= (\overline{Z\setminus A})\setminus({Z\setminus A}).$$ 

For $y\in \overline A$ we get dim $A_y \setminus (\stackrel{\circ}A)_y < {\rm dim}A_y ={\rm dim} Z\leq k$, hence both sets $\overline A\setminus A$ and $A\setminus  \stackrel{\circ}A$ are C-semianalytic sets of dimension $\leq k-1$.  
\end{proof}

\subsection{The direct image theorem.}

C-semianalytic sets verify a deeper property, analogous to the {\em Direct Image Theorem } of Grauert, stating that the family of complex analytic sets is stable 
under proper holomorphic maps between complex analytic spaces.

Let $(X,\Oo_X)$ and $(Y,\Oo_Y)$ be reduced Stein spaces. Let $\sigma:X\to X$ and $\tau:Y\to Y$ be anti-holomorphic involutions. Assume $X^\sigma$ and $Y^\tau$ are not empty. As defined in Section 2,  a C-semianalytic set $S\subset X^\sigma$ is \em ${\mathcal A}(X^\sigma)$-definable \em if for each $x\in X^\sigma$ there exists an open neighborhood $U^x$ such that $S\cap U^x$ is a finite union of sets of the type $\{F_{|X^\sigma}=0,{G_1}_{|X^\sigma}>0,\ldots,{G_r}_{|X^\sigma}>0\}$ where $F,G_i\in\Oo(X)$ are invariant holomorphic sections. We denote the set of $\sigma$-invariant holomorphic functions of $X$ restricted to $X^\sigma$ with ${\mathcal A}(X^\sigma)$.  We prove:

\begin{thm}[Direct image under proper holomorphic maps]\label{directimage}
Let 
$$
F:(X,\Oo_X)\to(Y,\Oo_Y)
$$ 
be an invariant proper holomorphic map, that is, $\tau\circ F=F\circ\sigma$. Let $S\subset X^\sigma$ be an ${\mathcal A}(X^\sigma)$-definable C-semianalytic set. We have 
\begin{itemize}
\item[(i)] $F(S)$ is a C-semianalytic subset of $Y^\tau$ of the same dimension as $S$. 
\item[(ii)] If $E=\overline{F^{-1}(Y^\tau)\setminus X^\sigma}$, then $F(E\cap S)$ is a C-semianalytic subset of $Y^\tau$.
\item[(iii)] If $S$ is a C-analytic set and $F^{-1}(Y^\tau)=X^\sigma$, then $F(S)$ is also a C-analytic subset of $Y^\tau$.
\end{itemize}
\end{thm}

We present here some examples that justify the hypothesis of Theorem \ref{directimage}.
\newpage
\begin{examples}\hfill

\begin{enumerate}
\item Consider the map $f: \R^2 \to \R^2, f(x,y)=((x^2+y^2)x,(x^2+y^2)y)$. Note that $f$ is an analytic map with continuous inverse map $g$ given by

$$\displaystyle g(u,v) \begin{cases} =(\frac{u}{\sqrt{u^2+v^2}},\frac{v}{\sqrt{u^2+v^2}})&
 {\mbox if}\quad (u,v)\neq (0,0)\\ =(0,0) &{\mbox if} \quad (u,v)=(0,0).
\end{cases}$$

 \parindent=0pt  Thus $f$ is an homeomorphism, hence a proper map with finite fibers. However $f$ cannot be extended to a proper holomorphic map on a neighbourhood of $\R^2$ in $\C^2$.
\item Let $Y = \{x^2 -zy^2 =0\}\subset \C^3$ be the complex Whitney's umbrella endowed with the complex conjugation $\tau$, so $Y^\tau = \{x^2 -zy^2 =0\}\subset \R^3$. 
Let $X=\C^2$ endowed with the complex conjugation $\sigma$. Consider the finite holomorphic map $F:\C^2\to Y, F(s,t)=(st,t,s^2)$. Then $X^\sigma =\R^2$ is strictly contained in the inverse image of $Y^\tau$ and  $F(X^\sigma)= \{x^2 -zy^2 =0\}\setminus \{z<0\}\subset \R^3$ which is not a C-analytic subset of $Y^\tau$.
\end{enumerate}
\end{examples}
\smallskip

The   following  result,   which   is  the   key   to  prove   Theorem
\ref{directimage}, analyzes  the local structure of  proper surjective
holomorphic  morphisms between  Stein  spaces. In its proof we use deep algebraic results on completions that can be found in \cite[23.K,L, 24.C]{mat}.  

For  each  $x\in X$  we
denote the maximal  ideal of $\Oo(X)$ associated to  $x$ with $\gtm_x$
and  for  each $y\in  Y$  we  denote  the  maximal ideal  of  $\Oo(Y)$
associated to $y$  with $\gtn_y$. Compact analytic subsets  of a Stein
space  are finite  sets, so  the fibers  of a  proper holomorphic  map
between Stein spaces are finite sets. Let $F:(X,\Oo_X)\to(Y,\Oo_Y)$ be
a surjective proper  holomorphic map between reduced  Stein spaces and
write $F^*(\Oo(Y))=\{G\circ F:\ G\in\Oo(Y)\}\subset\Oo(X)$ and
$$
F^*(\Oo(Y)_{\gtn_y})=\Big\{\frac{G\circ F}{H\circ F}:\ G,H\in\Oo(Y)\ \text{and}\ H\notin\gtn_y\Big\}.
$$

Then we get

\begin{thm}\label{finite76}
  Let  $y_0\in Y$  with fiber  $F^{-1}(y_0)=\{x_1,\ldots,x_\ell\}$ and
  denote  $\mathcal S=\Oo(X)\setminus(\gtm_{x_1}\cup\cdots\cup\gtm_{x_\ell})$.      Then
  $\mathcal  S^{-1}\Oo(X)$     is    a     finitely     generated
  $\Oo(Y)_{\gtn_{y_0}}$-module     and    there     exist    invariant
  $H_1,\ldots,H_m\in\Oo(X)$ such that
\begin{center}
$\mathcal S^{-1}\Oo(X)=F^*(\Oo(Y)_{\gtn_{y_0}})[H_1,\ldots,H_m].$
\end{center}
\end{thm}

\begin{proof}
 Consider the sheaf $F_*(\Oo_X)$ of $\Oo_Y$-modules. Recall that if $V\subset Y$ is open then $H^0(V,F_*(\Oo_X) ) =H^0(F^{-1}(V),\Oo_X)$. 

 Fix  $y_0  \in Y$  and  let  $F^{-1}(y_0)= \{x_1,\ldots,x_l\}$.  Then
 $F_*(\Oo_X)_{y_0}=\prod_{i=1}^l  \Oo_{X,x_i}$ and  $F_*(\Oo_X)$ is  a
 coherent  sheaf of  $\Oo_Y$-mo\-du\-les.   By Theorem  A there  exist
 $H_1,\ldots,   H_m   \in   H^0(Y,F_*(\Oo_X)   )=\Oo(X)$   such   that
 $F_*(\Oo_X)_{y_0}$  is   generated  by   $H_1,\ldots,  H_m  $   as  a
 $\Oo_{Y,y}$-module.  Then  we  get   a  diagram  of  faithfully  flat
 homomorphism between local excellent rings.

$$ 
\xymatrix{
\Oo(X)_{\gtm_{x_i}}\ar[d] \ar@{^{(}->}[r] &\Oo_{X,x_i}\ar[d]\\
 \widehat{\Oo(X)_{\gtm_{x_i}}} \ar[r]^{\widetilde{=}} &\widehat{\Oo_{X,x_i}}
}
$$

Consider the inclusion of $F^*(\Oo(Y)_{\gtn_{y_0}})$-modules

$$M_1 =\sum_{i=1}^m H_i\cdot F^*(\Oo(Y)_{\gtn_{y_0}})\subset M_2 =F^*(\Oo(Y)_{\gtn_{y_0}})[H_1,\ldots ,H_m]\subset \mathcal S^{-1}\Oo(X),$$ 
where $\mathcal S= \Oo(X)\setminus (\gtm_{x_1} \cup \ldots \cup \gtm_{x_l})$

Denote $B=\mathcal S^{-1}\Oo(X)$. We have to prove $M_1 = M_2=  B$

In what  follows completions of  rings are considered with  respect to
the Jacobson  radical. In case of  a local ring, its  Jacobson radical
coincide  with its  unique maximal  ideal. The  maximal ideals  of the
semilocal ring $B$ are $\gtm_{x_1}B,\dots ,\gtm_{x_l}B$.

 The  completion   $\hat  B$  of  $B$  is
isomorphic  to  $\prod_{i=1}^l \widehat{B_{\gtm_{x_i}B}}$.  Note  that
$B_{\gtm_{x_i}B}$ is isomorphic to $\Oo(X)_{\gtm_{x_i}}$. So
$B\subset     \prod_{i=1}^l     B_{\gtm_{x_i}B}$   is  isomorphic     to
$\prod_{i=1}^l\Oo(X)_{\gtm_{x_i}}\subset
\prod_{i=1}^l\Oo_{X,x_i}=(F_*(\Oo_X))_{y_0}.$

Since  $\Oo(X)_{{\gtm_{x_i}}}$   and  $\Oo_{X,x_i}$  share   the  same
completion which is isomorphic  to the completion of $B_{\gtm_{x_i}}$,
we get that the completion  of $F_*(\Oo_X)_{y_0}$ is isomorphic to the
completion of $B$.

Now $(F_*(\Oo_X))_{y_0}$ is  a finitely generated $\Oo_{Y,y_0}$-module
(it      is     generated      by      $H_1,\ldots     ,H_m$), its completion is

$$\widehat {F_*(\Oo_X)_{y_0}}= F_*(\Oo_X)_{y_0}\otimes_{\Oo_{Y,y_0}}\widehat{\Oo_{Y,y_0}}.$$

So $\widehat{F_*(\Oo_X)_{y_0}}$ is generated by $H_1,\ldots ,H_m$ as a $\widehat{\Oo_{Y,y_0}}$-module. 
By the same argument the completion of $M_1$ with respect to the maximal ideal of $\Oo(Y)_{\gtn_{y_0}}$ is
isomorphic to $M_1 \otimes_{\Oo(Y)_{\gtn_{y_0}}}\widehat{\Oo(Y)_{\gtn_{y_0}}}$.

As $M_1$ is generated by $H_1,\ldots ,H_m$ as a $\Oo(Y)_{\gtn_{y_0}}$-module its completion $\widehat{M_1}$ is generated by the same elements as a $\widehat{\Oo(Y)_{\gtn_{y_0}}}$-module. As a consequence, using that $\Oo(Y)_{\gtn_{y_0}}$
and $\Oo_{Y,y_0}$ have the same completion, we get that $\widehat{M_1}$ is isomorphic 
to $\widehat{F_*(\Oo_X)_{y_0}}$.  As the inclusion of $\Oo(Y)_{\gtn_{y_0}}$ in its completion is faithfully flat because $\Oo(Y)_{\gtn_{y_0}}$ is a local excellent ring, we get the following commutative diagram of $\widehat{\Oo(Y)_{\gtn_{y_0}}}$-modules.

$$
\xymatrix{
M_1\otimes_{\Oo(Y)_{\gtn_{y_0}}}\widehat{\Oo(Y)_{\gtn_{y_0}}}\ar[d]\ar@{^{(}->}[r]&
M_2\otimes_{\Oo(Y)_{\gtn_{y_0}}}\widehat{\Oo(Y)_{\gtn_{y_0}}}\ar@{^{(}->}[r]&
B\otimes_{\Oo(Y)_{\gtn_{y_0}}}\widehat{\Oo(Y)_{\gtn_{y_0}}}\ar@{^{(}->}[d]            \\
\widehat{M_1}\ar[u]_{\widetilde =}\ar[r]^{\widetilde =}&\widehat{  F_*(\Oo_X)_{y_0}}\ar[l]\ar[r]^{\widetilde =}&F_*(\Oo_X)_{y_0}\otimes{\Oo_{Y,y_0}} \widehat{\Oo_{Y,y_0}}\ar[l]
}
$$

So, all inclusions are isomorphisms, hence

$$M_1\otimes_{\Oo(Y)_{\gtn_{y_0}}} \widehat{\Oo(Y)_{\gtn_{y_0}}}=M_2\otimes_{\Oo(Y)_{\gtn_{y_0}}} \widehat{\Oo(Y)_{\gtn_{y_0}}}=
B\otimes_{\Oo(Y)_{\gtn_{y_0}}} \widehat{\Oo(Y)_{\gtn_{y_0}}}.$$

As the inclusion $\Oo(Y)_{\gtn_{{y_0}}}\subset \widehat{\Oo(Y)_{\gtn_{y_0}}}$ is faithfully flat we get $M_1 =M_2 =B$ as wanted.
\end{proof}

\begin{remark}
If $F:(X,\Oo_X)\to (Y,\Oo_Y)$ in addition is invariant, we can assume each $H_i$ to be invariant by taking $\displaystyle \Re(H_i) =\frac{H_i(z)+\overline{H_i(z)}}{2},
 \Im (H_i)=\frac{H_i(z)-\overline{H_i(z)}}{2i}$, which are invariant, instead of $H_i$
\end{remark}

In order to prove Theorem \ref{directimage} we need a more suitable version of Tarsky-Seidenberg Theorem and an application of M.Artin' approximation. 

 Let $(X,\Oo_X)$ be a reduced Stein space endowed with an anti-involution $\sigma$ such that its fixed part $X^\sigma$ is not empty. Let $u=(u_1,\ldots, u_m), m\geq 0$ be a tuple. Denote by $\mathcal A$ either $\Oo(X^\sigma)$ or $\Aa(X^\sigma)$.

\begin{thm}\label{TS'}
Let $S\subset X^\sigma \times \R^n$ be an $\mathcal A[u]-$definable global semianalytic set. Consider the projection $\pi: X^\sigma \times \R^n \to X^\sigma$ onto the first factor. Then, $\pi(S)$ is an $\mathcal A$-definable global semianalytic set. 
\end{thm}

\begin{lem}\label{artin}
Let $X_0\subset \C^n_0$ and $Y_0\subset \C^m_0$ be complex analytic set germs at the origin. Let $f:\Oo_{X,0}\to \Oo_{Y,0}$ be a local analytic homomorphism such that its extension to completions $\hat f :\widehat{\Oo_{X,0}} \to \widehat{\Oo_{Y,0}}$ is an isomorphism. Then $f$ is also an isomorphism.
\end{lem}

\begin{proof}
Denote the variables of $\C^n$ as $(x_1,\ldots,x_n )$ and the variables of $\C^m$ as $(y_1,\ldots ,y_m)$. Put $G_i =f(x_i)$ for $i=1,\ldots ,n$. By the universal property of local homomorphisms between power series rings, $\hat f$ is given by 

$$\hat f :\widehat{\Oo_{X,0}} \to \widehat {\Oo_{Y,0}}$$
 $$\hat f (\zeta) = \zeta(G_1,\ldots ,G_n),$$

in particular $f(F)=F(G_1,\ldots ,G_n)$. Denote $\eta_j = \hat f^{-1}(y_j)$ for $j=1,\ldots ,m$. As before $\hat f^{-1}(\xi)=\xi(\eta_1,\ldots ,\eta_m)$.
So, the tuple $(\eta_1,\ldots ,\eta_m)\in (\widehat{\Oo_{X,0}})^m$ is a formal solution of the system

$$\begin{cases}
G_1(y_1,\ldots ,y_m)= x_1\\
\phantom{G_1(y_1,\ldots ,y_m)\,}{\cdots} \\
G_n(y_1,\ldots ,y_m)= x_n
\end{cases}
$$

By Artin's approximation Theorem there exists a solution $(F_1,\dots, F_m)\in
 (\Oo_{X,0})^m$.  Consider the local homomorphism $g: \Oo_{Y,0} \to \Oo_{X,0}$ defined by
 $g(\xi)= \xi(F_1,\cdots , F_m)$.  Since $(F_1,\ldots , F_m)$ is a solution of the
 previous system one gets $g\circ f$ is the identity of $\Oo_{X,0}$, and $f$  is an isomorphism.   
\end{proof}

\begin{proof}[Proof of Theorem \ref{directimage}]
 There are, by hypothesis, an invariant proper holomorphic map $F:(X,\Oo_X)\to (Y,\Oo_Y)$ and $S\subset X^\sigma$ which is a C-semianalytic set defined by invariant holomorphic functions in $\Aa(X^\sigma)$. 
As $F$ is proper, we may assume in the proof that $X$ is irreducible and $F$ is surjective, because if $\{X_\alpha\}$ is the locally finite family of the irreducible components of $X$, $Y_\alpha =F(X_\alpha)$ is, by Grauert's Theorem on the direct image, a locally finite family of irreducible complex analytic subsets of $Y$.   

Fix a point $y_0 \in Y^\tau $. We have to find  an open neighbourhood $A\subset Y^\tau $ of $y_0$ such that $F(S)\cap A$ and $F(E\cap S)\cap A$ are global semianalytic sets in $Y^\tau$. Call $D = F^{-1}(Y^\tau) \setminus X^\sigma$ so that  $E= \overline{D} $.

{\sc Step 1.}  $F^{-1}(y_0)$ is a finite set, say $ \{x_1,\ldots,x_l\}\subset X$. Consider  $\mathcal S = \Oo(X)\setminus (\gtm_{x_1} \cup \dots \cup \gtm_{x_l})$. By Theorem \ref{finite76} and its Remark, there exist invariant holomorphic functions  $H_1,\ldots,H_m \in \Oo(X)$ such that 

$$\mathcal S^{-1}\Oo(X) = F^*(\Oo(Y)_{\gtn_{y_0}})[H_1,\ldots,H_m]$$

Consider the evaluation epimorphism $\theta: \Oo(Y)_{\gtn_{y_0}}[z_1,\ldots,z_m] \to \mathcal S^{-1}\Oo(X) $ defined by 
$\theta(Q(y,z_1,\ldots,z_m ))= Q(F(x), H_1(x), \ldots, H_m(x)).$

As $X$ is irreducible, $\mathcal S^{-1}\Oo(X)$ is an integral domain, hence the kernel of $\theta$ is a prime ideal $\gtp$ in $\Oo(Y)_{\gtn_{y_0}}[z_1,\ldots,z_m]$. 
So $\theta$ induces a ring isomorphism $$\overline \theta: \Oo(Y)_{\gtn_{{y_0}}}[z_1,\ldots ,z_m]/\gtp \to \mathcal S^{-1}\Oo(X).$$

Note that   $\Oo(Y)_{\gtn_{{y_0}}}$ is excellent, hence noetherian and  so $\gtp$ is generated by a finite number of polinomials, say $\{P_1,\ldots,P_t\}$. 

Consider the set $\Omega =\{(y,z)\in Y\times \C^m: P_j(y,z)=0, j=1,\ldots, t\}\subset Y\times \C^m$
and the map $H:X\to Y\times \C^m$ mapping $x$ to $(F(x), H_1(x),\ldots, H_m(x))$. We get $H(X)\subset \Omega$ and $H^*: \Oo(\Omega)= \Oo(Y)[z_1,\ldots ,z_m]/\gtp\ \to \Oo(X)$  is the restriction of the isomorphism $\overline \theta$.

Denote $\omega_i = H(x_i)$,  $A =\Oo(Y)_{\gtn_{y_0}}[z_1,\ldots ,z_m]/\gtp $ 
and $B= \mathcal S^{-1}\Oo(X)$.
$A$ is a semilocal ring, whose maximal ideals are the ideals $\gtm_{\omega_i}$ associated   to the points $\omega_i, i=1,\ldots,l$.

The isomorphism $\overline \theta$ induces an isomorphism between the local rings $A_{\gtn_{\omega_i}}$ and $B_{\gtm_{x_i}}$ hence an isomorphism between their completions. 

As $(X,\Oo(X))$ and $(Y,\Oo(Y))$ are reduced Stein spaces, one has $\widehat{\Oo(X)_{\gtm_{x_i}}} =\widehat{\Oo_{X,x}}$ and  $\widehat{A_{\gtn_{\omega_i}}} = \widehat{\Oo_{\Omega,\omega_i}}$.

Thus we have   isomorphisms among the four rings

$$ \widehat{\Oo_{\Omega ,\omega_i}}, \ \widehat{A_{\gtn_{\omega_i}}},\  \widehat{B_{\gtm_{x_i}}},\ \widehat{\Oo_{X,x_i}}$$

for $i=1,\ldots,l$. 

By Lemma \ref{artin} the composition homomorphism $\Oo_{\Omega,\omega_i} \to \Oo_{X,x_i}$, which sends a germ $G$ at $\omega_i$ to the germ at $x_i$ of $G\circ H$, is an isomorphism for $i=1,\ldots,l$. Consequentely there exist invariant open neighbourhoods $U$ of $\{x_1,\ldots,x_l\}$ and $V$ of $\{\omega_1,\ldots,\omega_l\}$ such that $H$ is an invariant complex analytic diffeomorphism between $U$ and $V$.

{\sc Step 2.} As $F$ is invariant and each $H_i$ is invariant also $H$ is invariant, so $H(X^\sigma  )\subset \Omega^{\tau'} = \Omega\cap Y^\tau$, where $\tau '$ is the anti-involution 
 $\tau'(y,z)= (\tau(y),\overline z)$. Let $f, h$ be the restrictions of $F, H$ to
 $X^\sigma$, so $f:X^\sigma \to Y^\tau$ and $h:X^\sigma \to \Omega^{\tau'}$. As
 $S$ is $\mathcal A (X^\sigma )$-definable, there exists an open neighbourhood
 $W\subset U\cap X^\sigma$ of the finite set $f^{-1}(y_0)= F^{-1}(y_0)\cap
 X^\sigma$ such that $S\cap W$ is a global $\mathcal A(X^\sigma)$-definable
 semianalytic set, that is,
 $S\cap W =\bigcup_{i=1}^r\bigcap_{j=1}^s S_{i,j}$ where $S_{i,j}$ is either $\{x\in X^\sigma:f_{i,j}>0\}$ or $\{x\in X^\sigma: f_{i,j}=0\}$. In addition we can assume $W\subset U\cap X^\sigma$, so  $h_{|W}:W\to h(W)$ is a real
 analytic diffeomorphism, $h^{-1}(h(W))=W$ and $h(W)$ is an open neighbourhood of
 $h(f^{-1}(y_0))$ in $\Omega^{\tau '}$. 
As $ H^*$ and the composition homomorphism  are isomorphisms, we get
\begin{itemize}
\item $f_{i,j}=q_{i,j}\circ h$ for some $q_{i,j}\in \Oo(Y^\tau)[u_1,\ldots,u_m]$.
\item $h(S\cap W) =T$ where $T= \bigcup_{i=1}^r \bigcap_{j=1}^s T_{i,j}$ and $T_{i,j}$ is either $\{(y,u)\in \Omega^{\tau '}:q_{i,j}>0\}$ or $\{(y,u)\in \Omega^{\tau'}: q_{i,j}=0\}$ according to the corresponding signs of the $f_{i,j}$.
\end{itemize}

Let $\pi:Y^\tau \times \R^m \to Y^\tau$ be the projection on the first factor. As $f =\pi \circ h$ we conclude by Theorem \ref{TS'} that $f(S\cap W)=\pi(h(S\cap W))=\pi (T)$ is a global semianalytic set. 

{\sc Step 3.} Note that
$H(D)= H(F^{-1}(Y^\tau)\setminus X^\sigma)= \{(y,u + i v)\in Y^\tau\times \C^m: P_1(y,u+iv) = 0, \ldots , P_r(y,u+iv)=0, v_1^2 +\cdots +v_m^2 \neq 0 \}.$

Thus we can understand $H(F^{-1}(Y^\tau)\setminus X^\sigma)$ as an $\Oo(Y^\tau)[u,v]$-definable global semianalytic subset of $Y^\tau \times \R^m \times \R^m$ where $u=(u_1,\ldots , u_m)$ and $v=(v_1,\ldots ,v_m)$. Hence its closure is locally global by Theorem \ref{fujitagabrie}. So, there is an open set  
  $V' \subset V\cap (Y^\tau \times \R^m \times \R^m  )$, that we can assume to be   a bounded  $\Oo(Y^\tau)[u,v]$-definable global semianalytic neighbourhood of $(\omega_1, \ldots , \omega_l)$, such that $\overline{H(D)} \cap V'$  
is an $\Oo(Y^\sigma)[u,v]$-definable global semianalytic subset of $Y^\tau \times \R^m \times \R^m$. 

We can assume $W \subset H^{-1}(V')$. So $H$ is an analytic diffeomorphism on $W$ because $W\subset  H^{-1}(V')\subset H^{-1}(V)$. Hence $\overline{H(D)} \cap V'= H(\overline D)\cap V' = H(E)\cap V'$. This proves that the latter is definable and global semianalytic.

Look now at $f(E\cap S\cap W) = \pi \circ h(E\cap S\cap W)$. 

Since $h$ is an analytic diffeomorphism on $W$ we have $h(E\cap S\cap W)=
h(E \cap W) \cap h(S\cap W)$.

Since $W\subset H^{-1}(V')$ we can write $ h(E\cap S\cap W)= H(E \cap H^{-1}(V') \cap W) \cap h(S\cap W)= H(E) \cap V' \cap H(W) \cap h(S\cap W)$. So

$$h(E\cap S\cap W) = H(E)\cap V' \cap h(W) \cap h(S\cap W) =   H(E)\cap V' \cap  h(S\cap W). 
$$
Both pieces are global semianalytic sets. Hence by Theorem \ref{TS} $f(E\cap S\cap W)$ is a definable  global semianalytic set. 

{\sc Step 4.} To complete the proof of (i) and (ii) we need an open C-semianalytic neighbourhood  $A\subset Y^\sigma$ of $y_0$ such that $f(S)\cap A = f(S\cap W)\cap A$ and $f(E\cap S)\cap A = f(E\cap S\cap W)\cap A$, so that $f(S)$ and $f(E\cap S)$ are C-semianalytic sets.

As $W$ is an open neighbourhood of $f^{-1}(y_0)$ in $X^\sigma$ and $f$ is proper, we get $C= f(X^\sigma \setminus W)$ is a closed subset of $Y$ that does not contain $y_0$. Let $A\subset Y^\tau$ be a global open basic semianalytic neighbourhood of $y_0$ which does not intersect C. As

\begin{center}$ f(S) =f(S\cap W)\cup f(S\cap (X^\sigma \setminus W))$ 
\end{center}
and 
\begin{center}
 $f(E\cap S) =f(E\cap S\cap W)\cup f(E\cap S \cap (X^\sigma \setminus W))$,
\end{center}

we conclude $f(S)\cap A = f(S\cap W)\cap A$ and $f(E\cap S)\cap A = f(E\cap S\cap W)\cap A$ as required.

{\sc Step 5.} We prove (iii).
By hypothesis there exists a complexification $\widetilde{Z} \subset X$ of $Z$ that is closed in $X$, hence $\widetilde{Z}$ is a complex analytic subset of $X$ and $\widetilde{Z} \cap X^\sigma = Z$. By Grauert's Theorem $F(\widetilde{Z})$ is a complex analytic subset of $Y$, so $F(\tilde{Z})\cap Y^\tau$ is a C-analytic subset of $Y^\tau$ that contains $F(Z)$. Thus,

$$ F(Z) \subset F(\widetilde{Z})\cap Y^\tau = F(\widetilde{Z}\cap F^{-1}(Y^\tau))= F(\widetilde{Z}\cap X^\sigma) =F(Z)$$

So $F(Z) =F(\widetilde{Z})\cap Y^\tau$ is a C-analytic subset of $Y^\tau$ as required.
\end{proof}

\smallskip

We end this section with a sufficient condition under which $F^{-1}(Y^\tau) = X^\sigma$ (with the notations of  Theorem \ref{directimage}).

\begin{prop} Let $F:(X,\Oo_X) \to (Y,\Oo_y)$ be an invariant surjective proper holomorphic map and assume $Y^\tau$ to be coherent. Let $Z\subset Y$ be a complex analytic set such that $X\setminus F^{-1}(Z)$ is dense in $X$. Assume $F_{|X\setminus F^{-1}(Z)}:X\setminus F^{-1}(Z)\to Y\setminus Z$ is a biholomorphism. Then, $F^{-1}(Y^\tau) = X^\sigma$.  
\end{prop}

\begin{proof} We may assume $X$ to be irreducible, hence also $Y$ is irreducible. In particular both are pure dimensional and have the same dimension, say $d$.
We have to prove $F^{-1}(Y^\tau) \subset X^\sigma$, since the converse inclusion is always true. 

Suppose by contradiction that there exists $z\in F^{-1}(Y^\tau) \setminus X^\sigma$. As $F$ is invariant $F(\sigma(z)) = F(z) \in Y^\sigma$, so $\sigma(z) \in F^{-1}(Y^\tau) \setminus X^\sigma$. By Grauert's Theorem $F(X_z)$ and $F(X_{\sigma(z)})$ are unions of irreducible components of $Y_{F(z)}$. As $X\setminus F^{-1}(Z)$ is dense in $X$ and the restriction of $F$ to it is a biholomorphism we conclude dim$_\C(F(X_z)\cap F(X_{\sigma(z)})) <d$. But $F$ is invariant, so $F(X_z)\cap Y^\tau = F(X_{\sigma(z)})\cap Y^\tau$ and dim$_\R(F(X_z)\cap Y^\tau) <d$. As $Y$ is irreducible and coherent, it is pure dimensional, so all irreducible components of the germ $Y^\tau_{F(z)}$ have dimension $d$. In addition the irreducible components of $Y^\tau_{F(z)}$ are intersections with $Y^\tau$ of the irreducible components of $Y_{F(z)}$. Here is the contradiction because $F(X_z)$ is a union of irreducible components of $Y_{F(z)}$ hence dim $Y_{F(z)}\cap Y^\tau = d$ and cannot be smaller.     \end{proof}

\subsection{Subanalytic sets as proper images of C-semianalytic sets.}

The family of semianalytic sets is not closed under  proper analytic maps. The class of subanalytic sets was introduced to get rid of this problem. Let us recall the notion of  subanalytic set. 

\begin{defn}
A subset $S\subset M$ is \em subanalytic in $M$  \em if each point $x\in M$ admits a neighborhood $U^x$ such that $S\cap U^x$ is a projection of a relatively compact semianalytic set, that is, there exists a real analytic manifold $N$ and a relatively compact semianalytic subset $A$ of $M\times N$ such that $S\cap U^x=\pi(A)$ where $\pi:M\times N\to M$ is the projection. 
\end{defn}

It could sound reasonable to consider the family of C-subanalytic sets. However, this is useless because, as we will prove, each subanalytic set is the image of a C-semianalytic set under a proper analytic map. Thus, one can replace semianalytic sets by C-semianalytic sets when one defines subanalytic sets.

Let us begin with some examples.

\begin{example} Consider the compact analytic set $Z\subset \R^3$ defined by the equation 
$$f= z^2(1-4(x^2+y^2+z^2)) -a(z)(x^2+ (y-1)^2  -1)^2 =0$$

where, as usual 
$$a(z) = \begin{cases} \exp \frac{1}{z^2-1} & {\rm if}\  |z|<1 \\
                             
           0& {\rm if}\  |z|\geq 1 
\end{cases}
$$

Observe that $ \R^3 \setminus (\{z=1\} \cup \{z=-1\}) \bigcap Z$ is a C-analytic subset of $\R^3 \setminus (\{z=1\} \cup \{z=-1\})$ and the $2$-dimensional part $Z_{(2)}$ of $Z$ is a subset of  $\{g =1-4(x^2+y^2+z^2) \geq 0\}$. Define 
$$ X_1 = \{(x,y,z,t)\in \R^4 : f=0, 4t^2 -g(x,y,z) =0\}$$

wich is a compact $2$-dimensional C-analitic subset of the sphere $S_1 = \{x^2+y^2+z^2 +t^2 = \frac{1}{4}\}$. 

Define $\pi:\R^4 \to \R^3$ by $\pi(x,y,z,t) = (x,y,z)$. The map $\pi_{|S_1}$ is proper, has finite fibers and satisfies $\pi(X_1) = Z\cap\{g\geq 0\}$. 
Consider the sphere $S_2 = \{x^2 +(y-1)^2 +z^2 +(t-2)^2 = 1\}$ and $X_2 = S_2 \cap \{x=0, t=2\}$. The map $\pi_{|S_2}$ is  analytic, has singleton fibers and satisfies $\pi(X_2)= Z\cap\{x=0\} = \{x=0, (y-1)^2 +z^2 =1\}$. Let $M=S_1 \cup S_2$. Of course $X = X_1 \cup X_2$ is a C-analytic subset of $M$, $\pi_{|M}$ is proper, analytic, with finite fibers and satisfies $\pi(X_1\cup X_2) = Z$ which is  analytic hence semianalytic but not C-semianalytic.
\end{example}

Next example is the well known Osgood's example of a subanalytic set which is not semianalytic.

\begin{example} Let $f:\R^2 \to \R^3$ be given by $f(x,y) = (x,xy, xe^y)$. Then, $S= f(\{x^2+y^2 \leq \varepsilon^2\})$ is subanalytic but not semianalytic.

$S$ is subanalytic because it is the image of a compact semianalytic set under an analytic map and has dimension $2$. To see that it is not semianalytic it is enough to prove the following statement.

\begin{quotation}
{\em If $G \in \R[[u,v,w]]$ is a formal power series such that $G(x,xy,xe^y) = 0$ then  $G=0$.}
\end{quotation}

Indeed, if so, the smallest real analytic set containing the  germ of  $S$ at $0$  is the whole $\R^3$ so that $S$ cannot be semianalytic.

Write $G(u,v,w) = \sum_{j\geq 0} G_j(u,v,w)$ where for all $j$ $G_j$ is an  homogeneous polynomial. Then,
$$0 = G(x,xy,xe^y) = \sum_{j\geq 0}G_j(x,xy,xe^y) = \sum_{j\geq0} x^j G_j(1,y, e^y)$$

Therefore all $G_j$ should vanish, hence they are identically zero, so $G=0$ as wanted.
\end{example}
\begin{remark}
The image of a C-semianalytic set under an analytic map is not in general a subanalytic set. For instance take $S= \{(\frac{1}{k},k), k>0 \}$ and let $\pi:\R^2 \to \R$ be the projection on the first factor. Then $\pi(S) = \{\frac{1}{k}\}$ cannot be semianalytic because in any neighbourhood of $0\in \R$ it has not finitely many connected components. It is not even subanalytic. $\pi(S) \subset \R\setminus \{0\}$, if it were subanalytic, in a suitable neighbourhood of $0$ it should be the projection of a relatively compact semianalytic set $A\subset \R\times N$ where $N$ is a real analytic manifold. But again the projection of $A$ should get finitely many connected components.     
\end{remark}

We recall two relevant results about subanalytic sets due to Bierstone and Milman.
\begin{prop}\label{bm1} Let $M$ be a real analytic manifold and let $S\subset M$ be a closed subanalytic set. Then each point of $S$ admits a neighbourhood $U$ such that $S\cap U= \pi(A)$ is the projection on the first factor of a closed analytic subset $A\subset U\times \R^q$ for some $q$. Also {\rm dim}$(A)=$ {\rm dim}$(S\cap U)$ and $\pi_{|A}$ is proper. 
\end{prop}

\begin{thm}[Uniformization Theorem]\label{bm2}
Let $X$ be a closed analytic subset of a real analytic manifold $M$. Then there exist a real analytic manifold $N$ of the same dimension as $X$ and a proper real analytic map $p:N\to M$ such that $p(N) = X$.
\end{thm}

We are ready to prove the announced result. 

\begin{thm}\label{sub}
Let $S$ be a subset of a real analytic manifold $N$. The following assertions are equivalent: 
\begin{itemize}
\item[(i)] $S$ is subanalytic.
\item[(ii)] There exists a global basic semianalytic subset $T$ of a real analytic manifold $M$ and an analytic map $f:M\to N$ such that $f_{|{\overline T}}:\overline T \to N$ is proper and $S=f(T)$.
\item[(iii)] There exists a C-semianalytic subset $T$ of a real analytic manifold $M$ and an analytic map $f:M\to N$ such that $f_{|{\overline T}}:\overline T\to N$ is proper and $S=f(T)$.
\end{itemize}
\end{thm} 

\begin{proof} 

(i)$\Rightarrow$ (ii)

This is done in several steps. Let $S\subset N$ be a subanalytic set.

{\sc  Local construction.}
Fix $x_0 \in N$ and denote $n$ the dimension of $N$. In this step we prove:

\begin{quotation}
{\em There exist an open neighbourhood $U\subset N$ of $x_0$, a compact real analytic manifold $M\subset \R^{2n+1}$, an analytic function $g\in \Oo(M)$ and an analytic map $\pi:M\to N$ such that 
$\pi(\{g>0\} \cap \pi^{-1}(U)) =S\cap U$.}
\end{quotation}

As $S$ is subanalytic, there exist an open neighbourhood $U\subset N$ of $x_0$, a real analytic manifold $N'$ and a relatively compact semianalytic subset $A\subset N\times N'$ such that $S\cap U = \pi_1(A)$, where $\pi_1$ is the projection on the first factor. We can assume dim$(S\cap U)=$ dim$(S\cap V)$ for each open neighbourhood $V\subset U$ of $x_0$.

Next we prove that we can assume dim$(A)=$ dim$(S\cap U)$ and that $\pi_1^{-1}(x_0)\cap A$ is a finite set.

By  Theorem \ref{bm2} there exist finitely many smooth semianalytic subsets $B_k$ of $A$ such that

\begin{itemize}
\item $S\cap U = \pi_1(A) = \pi_1(\bigcup_k B_k)$
\item For each $k$ the restriction $\pi_{|B_k}: B_k\to N$ is an immersion (that is has maximum rank at each point of $B_k$).
\end{itemize} 

This means that there exists a relatively compact semianalytic subset $B=\bigcup_k B_k$  of $N\times N'$ of the same dimension as $S\cap U$ such that $\pi_1(B) = S\cap U$ and $\pi_1^{-1}(x_0) \cap B$ is finite. So replacing  $A$ by $B$ we are done.

Since $\pi_1^{-1}(x_0) \cap A$ is finite, we can find an open relatively compact  neighbourhood $W\subset N\times N'$ of this finite set such that $A\cap W = \bigcup_{i=1}^r \{f_i=0, g_{i,1}>0,\ldots , g_{i,s}>0\}$ with $f_i, g_{i,j} \in \Oo(W)$ for all $i,j$. 
Note that we can assume the number of inequalities to be constant by  taking $s$ as the maximum and adding when necessary some trivial inequalities as $1>0$.

Next take an analytic function $h\in \Oo(W)$ such that $\pi_1^{-1}(x_0) \cap A \subset \{h >0\}$ and $\{h\geq 0\}$ is compact. Up to shrinking $U$ we can assume $S\cap U = \pi_1(A\cap  \{h\geq 0\} \cap \pi_1^{-1}(U))$.

Now we transform inequalities into equalities. For each $i=1,\ldots ,r$ define

$$X_i = \{(x,y,z)\in N\times N'\times \R^{s+2}: (x,y)\in W,  f_i(x,y)  =0,$$
$$ z_1^2- g_{i,1}(x,y)=0,\ldots ,z_{s}^2 -g_{i,s}=0, z_{s+1}^2-h(x,y) =0, z_{s + 2} =i\}.$$

Then $X= \bigcup_{i=1}^r X_i$ is a compact analytic subset of $W\times \R^{s+2}$ which has the same dimension as $A$, so dim$(X) \leq $ dim$(N)$. Let $\pi_2:N\times N'\times \R^{s+2} \to N\times N'$ be the projection on the first two factors. 

Put $g_i =\prod_{j=1}^s g_{i,j}$ and $g'= h \prod_{i=1}^r (g_i^2 +(z_{s+2}-i)^2)$. Note that 

$$\pi_2(X\setminus (\{g'=0\}\cap \pi_2^{-1}(\pi_1^{-1}(U))) =A\cap\{h>0\}\cap \pi_1^{-1}(U)$$
$$\pi_1(\pi_2(X\setminus (\{g'=0\}\cap \pi_2^{-1}(\pi_1^{-1}(U))))) =S\cap U$$

By Theorem \ref{bm2} there are a compact real analytic manifold $M$ (of the same dimension as $X$)  and a proper real analytic map $p:M\to W\times \R^{s+2}$ such that $p(M) =X$.
Using Grauert's embedding theorem we can embed $M$ in $\R^{2n+1}$ as a closed submanifold, where $n=$ dim$(N) \geq$ dim$(M)$.

Write $g=(g')^2\circ p$ and $\pi = \pi_1\circ\pi_2 \circ p:M\to N$. Then $\pi(\{g>0\}\cap \pi^{-1}(U))=S\cap U$ as required.

{\sc Global construction.} By Lemma \ref{refinement} there exists a countable locally finite open refinement $\{V_j\}_{j\geq 0}$ of $\{U_x\}_{x\in N}$ such that each $V_j= \{h_j>0\}$ is an open global semianalytic subset of $N$ and $h_j \in \Oo(N)$ for all $j$.
We have seen in Step 1 that for each $j$ there exist a compact real analytic submanifold $M_j \subset \R^{2n+1} \times \{j\} \subset \R^{2n+2}$, an analytic function $g_j\in \Oo(M_j)$ and an analytic map $\pi_j:M_j \to N$ such that

$$S\cap V_j = \pi_j(\{g_j >0\}\cap \pi^{-1}(V_j)) = \pi_j(\{g_j >0, h_j\circ \pi_j >0\}).$$

Consider the real analytic manifold $M= \bigcup_{j\geq 0}M_j \subset \R^{2n+2}$, whose connected components are compact. Let $g,h\in \Oo(M)$ be given by $g_{|M_j}= g_j, h_{|M_j} = h_j\circ \pi_j$. Consider the analytic map $\pi:M\to N$ defined by $\pi_{|M_j} = \pi_j$ and define $T= \{g>0, h>0\}$ which is a global basic semianalytic subset of $M$. We get

$$\pi(T) = \pi(\{g>0, h>0\}) = \bigcup_{j\geq 0}\pi_j(\{g_j>0, (h_j\circ \pi_j)>0\}) = \bigcup_j S\cap V_j =S.$$

It only remains to prove that $\pi_{|\overline T}: \overline T \to N$ is proper.
 
 Let $K_0$ be a compact subset of $N$ and denote $K= K_0\cap \overline S$. As $S= \bigcup_j S\cap V_j$ and the family $\{V_j\}$ is locally finite, we get $\overline S = \bigcup_j \overline {S\cap V_j}$ and the family $\{\overline {S\cap V_j}\}$ is locally finite. As $K$ is compact it does not intersect $\overline{S\cap V_j}$ for $j\geq l$. As the family $\{M_j\}$ is locally finite and $M_j\cap M_k =\varnothing$ for $j\neq k$ we get $\overline T \cap M_j = \overline{T\cap M_j}$. In addition, $\pi(T\cap M_j) = S\cap V_j$. 

We claim $\pi^{-1}(K)\cap \overline T\cap M_j =\varnothing$ for $j\geq l$.
Suppose by contradiction that there exists $x\in \pi^{-1}(K)\cap \overline T \cap M_j$ for some $j\geq l$. Thus, as $\pi$ is continuous, $\pi(x)\in K\cap \pi(\overline T\cap M_j) = K\cap \pi(\overline {T\cap M_j}) \subset K\cap \overline{\pi(T\cap M_j)} = K\cap (S\cap V_j) = \varnothing$ because $j\geq l$, which is a contradiction.

As $\displaystyle \pi^{-1}(K)\subset \bigcup_{j=1}^{l-1}\overline T\cap M_j$ and each $M_j$ is compact, we conclude that $\pi^{-1}(K) $ is compact as required.

(ii)$\Rightarrow$ (iii)

This implication is clear.  

(iii) $\Rightarrow$ (i)

Let $\Gamma_f$ be the graph of $f$ and let $\pi_2: M\times N \to N$ be the second projection. Put $T' = \Gamma_f \cap (T\times N),\  C=\Gamma_f \cap (\overline T \times N)$. Both are semianalytic subsets of $M\times N$. 

First of all we want to prove $C= \overline T'$. $C$  is closed in $\Gamma_f$ hence it contains the closure of $T'$. Let us prove the converse. Pick $(x,f(x)) \in C$ and let $U\times V$ be a neighbourhood of $(x,f(x))$ in $M\times N$. As $f$ is continuous we can assume $f(U) \subset V$. As $x\in \overline T$ there exists $x'\in T\cap U$, hence $(x',f(x')) \in U\times V\cap (\Gamma_f\cap (T\times N))$. So, $(x,f(x)) \in \overline T'$.

Next we prove that ${\pi_2}_{|C}: C\to N$ is proper.

Indeed $\Gamma_f$ is a real analytic submanifold of $M\times N$ and it is analytically isomorphic to $M$ via the first projection $\pi_1: \Gamma_f\to M$. Let $K$ be a compact subset of $N$. We get:
$$f^{-1}(K)\cap \overline T = \pi_1(\Gamma_f\cap (\overline T \times N))\cap(M\times K)=$$     $$= \pi_1(C\cap \pi_2^{-1}(K)) = \pi_1(\pi_{2|C}(K)).$$

As $f_{|\overline T}:\overline T \to N$ is proper $f^{-1}(K)\cap \overline T$ is compact. As $\pi_{1|{\Gamma_f}}$ is an analytic diffeomorphism ${\pi_{2|C}}^{-1}(K)$ is also compact, hence $\pi_{2|C}$ is proper.

Take $y\in N$ and let $U$ be an open semianalytic neighbourhood of $y$ in $N$ such that its closure is compact and put $K=\overline U$. Then ${\pi_{2|C}}^{-1}(K)$ is compact because the map is proper. As $f(T) = S$, it holds $\pi_2(T') = S$, so $\pi_2(T'\cap \pi_2^{-1}(U)) =S\cap U$. It only remains to prove that $A = T'\cap {\pi_2}^{-1}(U)$ is a relatively compact semianalytic set of $M\times N$. Since $T'$ and $U$ are semianalytic, $A$ is semianalytic. To prove that it is relatively compact it is enough to prove that it is a subset of a compact set. Indeed

$$A = T'\cap {\pi_2}^{-1}(U)\subset C\cap {\pi_2}^{-1}(U) = {\pi_{2|C}}^{-1}(U) \subset {\pi_{2|C}}^{-1}(K)$$

which is compact because $\pi_{2|C}$ is proper.
\end{proof}

As a consequence we recover the well known result  that subanalytic subsets of $\R$ are in fact semianalytic.

\begin{cor}\label{subR}
Subanalytic sets in $\R$ are semianalytic i.e. locally finite unions of points and segments.
\end{cor}
\begin{proof} Let $S \subset \R$ be a subanalytic set. Then there exist a relatively compact C-semianalytic set $T\subset \R\times M$ such that $S= \pi_1(T)$. Here $M$ is a real analytic manifold and $\pi_1$ is the projection on the first factor and it is proper on $\overline T$. The family of its connected components $\{T_i\}$ is locally finite (even finite, since $\overline T$ is compact) and $\pi_1(T_i)$ is a connected subset of $\R$ for each $i$. Hence $S=\bigcup_i \pi_1(T_i)$ is a (finite)  union of points and segments, i.e. it is semialgebraic, hence semianalytic. 
\end{proof}
\smallskip

\begin{defn}\label{c-property}
Let $P$ be a property concerning either C-semianalytic or C-ana\-lytic sets. We say that $P$ is a {\em C-property} if the set of points of an either C-semianalytic or C-analytic set $S$ satisfying $P$ is a C-semianalytic set. 
\end{defn}

For instance  the set of points where the dimension of the C-semianalytic set $S$ is $k$ is a C-semianalytic set, that is, {\em ``to be a point of local dimension $k$''} is a C-property.

In the next sections we prove that in a real analytic manifold 
 the set of local extrema of a real analytic function and  the set of points where a  C-analytic set is not coherent are C-semianalytic, as a consequence 
{\em ``to be a local maximum or  a local minimum''} and  {\em ``to be a point of non-coherence or to be a point of coherence''} are C-properties.

\smallskip

\subsection{Local extrema of a real analytic function.}
In this section we want to prove that the set of local maxima (resp. the set of local minima) of a non costant analytic function is a C-semianalytic set. 
\smallskip

 Let $\Oo(M)$ be the ring of analytic functions on a real analytic manifold $M$.

A point $x\in M$ is {\em critical} for $f$ if the gradient of $f$ vanishes at $x$. The set of critical points $C(f)$ is locally given by the vanishing of the partial derivatives of $f$ with respect to a local system of coordinates on $M$. Thus $C(f)$ is a closed analytic set in $M$ which contains the local extrema of $f$.

\begin{remarks}

\begin{itemize}
\item We can assume $M$ to be imbedded as a closed submanifold in some $\R^n$. Then $M$ has an analytic tubular neighbourood $(\Omega,p)$ so that for each $f\in \Oo(M), f\circ p$ is analytic on $\Omega$. Hence we are reduced to consider analytic functions on open sets of $\R^n$.
\item If $f$ is analytic on $\Omega$, $C(f)$ is a closed C-analytic set in $\Omega$.
\item If $M$ is connected and $f\in \Oo(M)$ is not constant then the set Max$(f)$ of local maxima and the set Min$(f)$ of local minima are disjoint sets. Since Min$(f) =$ Max$(-f)$ it is enough to study the set of local maxima.
\end{itemize}
\end{remarks}
For $\lambda \in \R$ consider the set of local maxima of $f$ having $\lambda$ as value, which we denote Max$_\lambda(f)$. It is easy to describe this set.

$${\rm Max}_\lambda(f) = \{x\in M: f(x)=\lambda\}\setminus \overline {\{x\in M: f(x)>\lambda\}}$$

Hence Max$_\lambda(f)$ is a C-semianalytic set and Max$(f)=\bigcup_{\lambda \in \R}$Max$_\lambda(f)$. So it is enough to prove that the family $\{$Max$_\lambda(f)\}_{\lambda\in \R}$ is locally finite.

\begin{thm}\label{sub1}
The family $\{{\rm Max}_\lambda(f)\}_{\lambda\in \R}$ is locally finite.
\end{thm}
\begin{proof}
Let $\{K_n\}$ be an exhaustion of $M$ by compact sets such that $K_{n-1}$ is contained in the interior part of $K_{n}$. It is enough to prove that the values of local maxima of $f$ in $K_n$ are finitely many.

Consider the graph $\Gamma$ of the restriction of $f$ to $K_n$. If $I_n$ is a closed interval such that $f(K_n)\subset I_n$, then $\Gamma \subset K_n\times I_n$ is a compact set. Since $\Gamma$ is homeomorphic to $K_n$ by the first projection, it contains a copy  of the set Max$(f) \cap K_n$. Consider a connected component $C$ of Max$(f) \cap K_n$. Its image $f(C)$ is a connected subset of $I_n$ hence it is a point or an interval. Now Max$(f) \cap K_n$ is relatively compact and it is semianalytic because it can be described as the set of points $x\in K_n$ such that there exists $\varepsilon >0$ verifying $f(x)\geq f(y)$ for all $y\in K_n \cap \{||x-y|| <\varepsilon\}$. 
Hence its image is subanalytic in $I_n\subset \R$.  But subanalytic sets in $\R$ are semianalytic (Corollary \ref{subR}), so $f(\mbox{\rm Max}(f)\cap K_n) \subset I_n$ is a finite union of points and segments.

We have to prove that there are only points and this is done in the next lemma.
\end{proof}

\begin{lem}\label{segmento}
$f(\mbox {\rm Max}(f)\cap K_n)$ is a finite set of points.
\end{lem}

\begin{proof} Let $J\subset I_n$ be an interval in $f(\mbox{\rm Max}(f) \cap K_n)$. Take off all points  $t \in J$ such that $f^{-1}(t)\cap K_n$ is discrete. We get some smaller intervals. If $J_l$ is one of them,  $f^{-1}(J_l)\cap K_n$ is semianalytic in Max$(f)\cap K_n$ and admits a triangulation. Let $\sigma$ be a simplex in $f^{-1}(J_l)\cap K_n$ whose image has dimension $1$. Then there is a segment $S$ in $\sigma$  which meets all fibers of the restriction  of $f$ to $\sigma$. So,  we can define a continuous map $\alpha:f(\sigma)\to S$ such that $f(\alpha(t)) =t$. Now take $t_0 \in f(\sigma)$, hence $x_0=\alpha(t_0)$ is a local maximum of the restriction of $f$ to $K_n$ and so, it has an open neighbourhood $U$ such that all $y\in K_n\cap U$ verify $f(y)\leq t_0$. But there is in $\sigma$ a point $y\in U\cap \alpha(J_l)$ such that $f(y)=t_1 >t_0$ and this is a contradiction. So, $f(\mbox {\rm Max}(f)\cap K_n)$ is a finite set of points.
\end{proof}

\subsection{ Points where a C-analytic set is not coherent.}

What we want to prove is the following

\begin{thm}\label{ncc}
The set $N(X)$ of points of non-coherence of a $C$-analytic set $X\subset M$ is a $C$-semianalytic set of dimension $\leq\dim(X)-2$.
\end{thm}

Thus analytic  curves are  coherent and surfaces  fail to  be coherent
only at isolated points. Since coherence  is an open property, the set
$N(X)$ is a closed subset of the singular locus of $X$.

Let us  give some general  ideas about how  the set $N(X)$  arises. By
Proposition \ref{Noncoerenza} of Chapter 1, a real analytic set $X$ is
not  coherent at  the  point  $a\in X$  when  there  exist points  $b$
arbitrarily close  to $a$  such that the  complexification of  the set
germ  $X_b$  is  not  induced  by the  complexification  of  the  germ
$X_a$.  Roughly speaking,  this means  that  a branch  of real  points
becomes a branch of complex  points when crossing a non-coherence point
(as real roots  can disappear when passing through a  double root of a
polynomial). Many times this translates in a drop of dimension and the
points of  non-coherence are those  points of  $X$ where the  drops of
dimension arise.  Classical examples  of this situation  are Whitney's
umbrella   $W_1=\{x^2-zy^2=0\}\subset\R^3$   and   Cartan's   umbrella
$W_2=\{x^3-z(x^2+y^2)=0\}\subset\R^3$.    Both   examples    are   two
dimensional $C$-analytic  sets that have $1$-dimensional  tails and in
both cases  the point of non-coherence  is the origin. However,  it is
also possible that the tail  is hidden inside the 2-dimensional part
of     $X$.      An     example      of     this      situation     is
$W_3=\{z(x+y)(x^2+y^2)-x^4=0\}$. The points of the tail are those in
the line $\ell=\{x=0,y=0\}$.  In this case the  point of non-coherence
of $W_3$ is the origin, but if $b\in\ell$ is close to the origin, then
$\dim(W_{3,b})=\dim(W_{3,0})$.

Tails in the  examples above are obtained locally  as intersections of
complex conjugated  analytic germs.  This first  type of  tails cannot
occur in a normal $C$-analytic  set because such $C$-analytic sets are
locally  irreducible  and  their complexifications  are  also  locally
irreducible. There is a second way  to produce tails. Let $X\subset Y$
be a $C$-analytic set inside its complexification $Y$. It could happen
that there exist points $x\in X$ such  that the germ $X_x$ is a subset
of the  singular locus  of $Y_x$ and  $\dim(X_x)\leq\dim(Y_x)-2$. This
situation is reproduced  in the following example. In  what follows we
will call of {\em   type 1} the tails  of the first type and  of {\em type 2}
all the other ones.

\begin{example}\label{tail}
  Consider the  pencil of conics given  by $x^2+y^2=t$ where $t$  is a
  real parameter.  Then $X=\{(x,y,z,t)\in\R^4:\  x^2+y^2-tz^2=0\}$ can
  be understood as the family of  cones of vertex the origin and basis
  the  pencil  of conics  above.  Consider  the complex  analytic  set
  $Y\subset\C^4$ given by the same equation  as $X$. It holds that $X$
  is  a  $C$-analytic set  in  $\R^4$  and it  is  the  fixed part  of
  $Y$. Write $p=(0,0,0,d)$. If $d\geq  0$ the germ $X_p$ has dimension
  $3$,  while for  $d<0$  the germ  $X_p$  is the  germ  at the  point
  $(0,0,0,d)$  of the  line $\ell=\{x=0,y=0,z=0\}$.  Observe that  the
  germ  $X_p$  is  contained  in  the  singular  locus  of  $Y_p$  and
  $\dim(X_p)=1=\dim(Y_p)-2$.  Thus, we  have found  a one  dimensional
  tail  which does  not  come  from the  intersection  of two  complex
  conjugate branches.

  The  set of  singular  points of  $Y$ is  the  complex analytic  set
  $\Sing(Y)=\{x=0,y=0,z=0\}\cup\{x=0,y=0,t=0\}\subset\C^4$,  which has
  codimension $2$  in $Y$.  As $Y$ is  a complex  irreducible analytic
  hypersurface, we  deduce by Oka's criterion  that $Y$ is a  normal complex
  analytic set. Thus, $X$ is a  normal $C$-analytic set. As $X$ is irreducible but not
  pure dimensional, it is not coherent.
\end{example} 

Let  us see  in  an intuitive  way  how we  can  characterize the  set
$N(X)$.  Assume that  $X$  is an  irreducible  $C$-analytic subset  of
$\R^n$. Let  $\widetilde{X}$ be a  complexification of $X$ that  is an
invariant  complex  analytic  subset  of an  open  Stein  neighborhood
$\Omega\subset\C^n$   of    $\R^n$.   Denote   the    restriction   to
$\widetilde{X}$   of   the   complex  conjugation   on   $\C^n$   with
$\sigma:\widetilde{X}\to\widetilde{X}$.            It            holds
$d=\dim_\R(X)=\dim_\C(\widetilde{X})$   and   $X=\{x\in\widetilde{X}:\
\sigma(x)=x\}$.  Let $\pi:Y\to\widetilde{X}$  be the  normalization of
$\widetilde{X}$. As $\widetilde{X}$  is Stein, also $Y$  is Stein. The
complex conjugation of $\widetilde{X}$  extends to an anti-holomorphic
involution $\widehat{\sigma}$ on $Y$  that makes the following diagram
commutative (compare Chapter 2, Section \ref{norm})
$$
\xymatrix{
&Y^{\widehat{\sigma}}\ar[d]_{\pi|_{Y^{\widehat{\sigma}}}}\ar@{^{(}->}[r]&Y\ar[r]^{\widehat{\sigma}}\ar[d]_{\pi}&Y\ar[d]^{\pi}\\
X\ar@{=}[r]&\widetilde{X}^{\sigma}\ar@{^{(}->}[r]&\widetilde{X}\ar[r]^{\sigma}&\widetilde{X}}
$$

Roughly speaking, the inverse images of tails of $X$ correspond to:
\begin{itemize} 
\item The set $\pi^{-1}(X) \setminus Y^{\widehat{\sigma}}$ (this set can be understood intuitively as the inverse image of those tails of type 1, which disappear when irreducible local components of $X$ are separated after we apply normalization).
\item The own tails of $Y^{\widehat{\sigma}}$ (which provide tails of type 2 in $X$, see Example \ref{tail} for further details).
\end{itemize} 

The set $N_d(X)$ of points of $X$ such that the germ $X_x$ has a non-coherent irreducible component of dimension $d$ is obtained as follows. Define 
\begin{align*}
&C_1=\pi^{-1}(X)\setminus Y^{\widehat{\sigma}}&\text{(preimages of the tails of type 1)}\\ 
&C_2=Y^{\widehat{\sigma}}\setminus\overline{Y^{\widehat{\sigma}}\setminus\Sing(Y^{\widehat{\sigma}})}&\text{(preimages of the tails of type 2)}
\end{align*} 
and denote $A_i=\overline{C_i}\cap\overline{Y^{\widehat{\sigma}}\setminus\Sing(Y^{\widehat{\sigma}})}$ (points where the preimages of tails of type $i$ attach to the $d$-dimensional part of $Y^{\widehat{\sigma}}$). Consequently, $N_d(X)=\pi(A_1)\cup\pi(A_2)$ and we deduce that $N_d(X)$ is a $C$-semianalytic set as a consequence of the Direct Image Theorem \ref{directimage}.
 
The construction of the full set $N(X)$ is much more involved, but it follows from the same kind of ideas. When $X$ is not irreducible the costruction is even more complicated and require a more careful discussion.

\smallskip

Let $M$ be a real analytic manifold and $X\subset M$ be a C-analytic set of dimension $d$. For each $0\leq k\leq d$ let $\mathcal{F}_k$ be the collection of all irreducible C-analytic sets $Z\subset X$ of dimension $k$ that are irreducible components of $\Sing_l(X)$ fore some $l\geq 0$. Here $\Sing_0(X) =X$, $\Sing_1(X)=\Sing(X)$, $\Sing_l(X) = \Sing(\Sing_{l-1}(X)$). 

Define
 
\begin{center}
$Z_k = \bigcup_{Z\in \mathcal{F}_k}Z$ \  and  \ $R_k = \bigcup_{j=k+1}^d Z_{j,(j)}$  
\end{center}

where
$Z_{j,(j)} = \{z\in Z_j:{\rm dim}_\R(Z_j)_z =j\} =\overline{Z_j \setminus {\rm Sing}(Z_j)}.$

Let $\widetilde Z_k$ be a complexification of $Z_k$ and let $(Y_k,\pi_k)$ be the normalization of $\widetilde Z_k$ endowed with the anti-involution $\sigma_k$ induced by the natural conjugation of $\widetilde Z_k$.

Let $Y_k^{\sigma_k} = \{y\in Y_k: \sigma_k(y) =y\}$ be the fixed part of $Y_k$, which is a C-analytic space. Define

$$Y^{\sigma_k}_{k,(k)}= \{y\in Y^{\sigma_k}_k:{\rm dim}_\R(Y^{\sigma_k}_k)_z =k\} =\overline{Y^{\sigma_k}_k \setminus {\rm Sing}(Y^{\sigma_k}_k)},$$ 

$$ C_{k,1} = \pi_k^{-1}(Z_k)\setminus Y^{\sigma_k}_k,\quad \quad   C_{k,2} = Y^{\sigma_k}_k\setminus Y^{\sigma_k}_{k,(k)},$$

$$ A_{k,i} = \overline{C_{k,i}} \cap \overline{Y^{\sigma_k}_{k,(k)}\setminus \pi_k^{-1}(R_k)}.$$

Then we get.

\begin{thm}\label{ncdes} Let $N(X)$ be the set of points of non-coherence of a C-analytic set $X$ of dimension $d$ and put $N_k(Z_k,R_k) = \pi_k(A_{k,1}) \bigcup \pi_k(A_{k,2})$.
Then
\begin{itemize}
\item[(i)] $N_k(Z_k,R_k)$ is a C-semianalytic set of dimension $\leq k-2$.
\item[(ii)] $\bigcup_{k=j}^d N_k(Z_k,R_k)$ is the set of points $x \in X$ such that the germ $X_x$ has a non-coherent irreducible component of dimension $\geq j$.
\item[(iii)] $N(X) = \bigcup_{k=2}^d N_k(Z_k,R_k)$.
\end{itemize}
\end{thm}

As a consequence we get the statement of Theorem \ref{ncc}.

Before proving Theorem \ref{ncdes} we give some lemmas on germs. Then  we look at the set $N_d(Z_d,R_d)$, that is points of non-coherence in components of maximal dimension. 

\begin{lem}\label{germs}
Let $x\in X$. Then
\begin{itemize}
\item[(i)]If $Z_x$ is an irreducible component of $X_x$  of dimension $d =\mbox{\rm  dim} (X)$,  its complexification $\widetilde{Z_x}$ is an irreducible component of $\widetilde{X}_x$. If $x=\pi(y)$ and $\pi(Y_y) = Z_x$, then $y\in Y^{\hat{\sigma}}$ and {\rm dim}$_\R( Y^{\hat{\sigma}}_y)=d$ 
\item[(ii)]Let $y\in Y^{\hat{\sigma}}$ be such that $\pi(y)=x$ and $\mbox{\rm dim}_\R(\pi(Y_y) \cap X_x)<d$. Then $\mbox{\rm dim}_\R(Y^{\hat{\sigma}}_y)<d$. In particular, if $\mbox{\rm dim}_\R(X_x) < d$, we get $\mbox{\rm dim}_\R(Y^{\hat{\sigma}}_y) < d$ for all $y\in \pi^{-1}(x) \cap Y^{\hat{\sigma}}$. 
\end{itemize}
\end{lem} 
\begin{proof} (i) It is clear that $\widetilde{Z_x}$ is an irreducible component of $\widetilde{X}_x$ of dimension $d$. It determinates a point $y= (\widetilde{Z_x},x) \in Y$ where $x=\pi(y)$. If $y$ were not in $Y^{\hat{\sigma}}$ then $\hat{\sigma}(y) \neq y$ hence $ (\widetilde{Z_x},x) \neq \hat{\sigma}(\widetilde{Z_x},x)$. But $\sigma(x) =x$ so,
 $Z_x = \widetilde{Z_x}\cap \sigma( \widetilde{Z_x})$ would get dimension less than $d$. Contradiction. 

(ii) We know that  {\rm dim}$_\R(\pi(Y_y) \cap X) <d$.  Hence $\mbox{\rm dim}_{\R}(\pi(Y^{\hat{\sigma}}_y) \cap X) <d$. But $\pi$ is finite, so {\rm dim}$_{\R}(Y^{\hat{\sigma}}_y) <d$.
If {\rm dim}$_\R(X_x)<d$, since $\pi(Y^{\hat{\sigma}}_y) \subset X_x$ for all $y\in \pi^{-1}(x)$, $\mbox{\rm dim}_\R(Y^{\hat{\sigma}}_y) <d$ for all $y\in \pi^{-1}(x)$.
\end{proof}

\begin{lem}\label{coherent}
Let $x\in X$ be a point such that all irreducible components of $X_x$ have dimension $d$. The germ $X_x$ is coherent if and only if there exists an invariant open neighbourhood $V$ of $x$ in $\widetilde X$ such that
\begin{itemize}
\item[(i)] $\pi^{-1}(X\cap V)= Y^{\hat{\sigma}}\cap \pi^{-1}(V)$.
\item[(ii)] $\mbox{\rm dim}_\R(Y^{\hat{\sigma}}_y) =d$ for all $y\in \pi^{-1}(x)$.
\end{itemize}
\end{lem}  
\begin{proof} Assume $X_x$ to be coherent. Then there exist an invariant open neighbourhood $V$ of $x$ in $\widetilde X$ such that the complexification $\widetilde{X_x}$ induces the complexification $\widetilde{X_{x'}}$ for all $x'\in V\cap X$. In particular all irreducible components of $\widetilde{X_x}$ are $\sigma$ invariant and of complex dimension $d$. So for all  $y\in \pi^{-1}(x)$ the germ $Y_y$ is $\hat{\sigma}-$invariant of complex dimension $d$. This implies $y\in Y^{\hat{\sigma}}$ for all $y\in \pi^{-1}(x)$ and so $\pi^{-1}(X\cap V)= Y^{\hat{\sigma}}\cap \pi^{-1}(V)$ and $\mbox{\rm dim}_\R(Y^{\hat{\sigma}}_y) =d$ for all $y\in \pi^{-1}(x)$. 

Conversely assume to have $V$ such that (i) and (ii) hold true. By Lemma \ref{germs}  $\pi^{-1}(x) \subset Y^{\hat{\sigma}}$ and $Y^{\hat{\sigma}}\cap \pi^{-1}(V)$ has pure real dimension $d$. Also we get the number of irreducible components of the germ $X_x$ is the same as the one of $\widetilde{X_x}$. Indeed if this number is smaller, it  means that some irreducible  component $T$ of $\widetilde{X_x}$ is not invariant. So  $T\cap\sigma (T)\subset X$  produces an irreducible component of $X_x$ of dimension less than $d$. This is true not only in $x$ but in all points of $V\cap X$. Contradiction.

By shrinking $V$ we can assume that $\widetilde{X_x}$ is the germ of a  $\sigma$ invariant analytic subset $X'$  of $V$.
We have to show that $X'_z$ is the complexification of $X_z$ for $z$ close to $x$. What we know by the remark above is that the number of irreducible components of $X'_z$ is the same as the number of irreducible components of $X_z$. 
So for $z\in V\cap X$ denote $n_z$ the number of irreducible components of both $X_z$ and  $X'_z$. The complexification of the germ $X_z$ is a complex analytic subset of $X'_z$ and since  (i) and (ii) hold true in $z$ it has only components of complex dimension $d$ and the number of such irreducible components is again $n_z$ by the same argument as for $x$. So it coincides with $X'_z$ and this implies that $X_x$ is coherent.   
\end{proof}

\smallskip

Now we describe the set of points of non-coherence of maximal dimension.

Let $X$ be a C-analytic set in $\R^n$ and let $d$ its dimension.
 Let $R\subset X$ be an ${\mathcal A}(Y^{\hat{\sigma}})$-definable closed C-semianalytic set. Let $Y'$ be the union of those irreducible components of $Y$ that have dimension strictly smaller than $d$. Define

$$Y^{\hat{\sigma}}_{(d)} = \{y\in Y^{\hat{\sigma}}: {\rm dim}_\R(Y^{\hat{\sigma}}_y) = d\} = \overline{Y^{\hat{\sigma}}\setminus ({\rm Sing}(Y^{\hat{\sigma}})\cup Y')}$$

$$C_{1} = \pi^{-1}(X)\setminus Y^{\hat{\sigma}}\quad  C_{2} = Y^{\hat{\sigma}}\setminus Y^{\hat{\sigma}}_{(d)}$$

$$A_i = \overline{C_i}\cap \overline{Y^{\hat{\sigma}}_{(d)}\setminus \pi^{-1}(R)} \subset Y^{\hat{\sigma}} \quad {\rm for} \quad  i=1,2.$$

We get $\pi(A_i)$ is a C-semianalytic set.
Indeed
the C-semianalytic set $Y^{\hat{\sigma}}\setminus ({\rm Sing}(Y^{\hat{\sigma}})\cup Y')$ is ${\mathcal A}(Y^{\hat{\sigma}})$-definable, so by Proposition
  \ref{sigmainvariante}  $Y^{\hat{\sigma}}_{(d)}$ is ${\mathcal A}(Y^{\hat{\sigma}})$-definable. The same happens with $\overline{C_2}$ again by Proposition 
  \ref{sigmainvariante}. As $R$ is ${\mathcal A}(Y^{\hat{\sigma}})$-definable, then $\pi^{-1}(R)\cap Y^{\hat{\sigma}}$ is an  ${\mathcal A}     
  (Y^{\hat{\sigma}})
 $-definable C-semianalytic set and the same happens for $\overline{Y^{\hat{\sigma}}_{(d)}\setminus \pi^{-1}(R)}$. By Theorem \ref{directimage} $\pi(A_i)\subset X$
is a C-semianalytic set for $i=1,2$.

Define 
$N_d(X,R) = \pi(A_1) \cup \pi(A_2) \subset X,$ 
which is a C-semianalytic set.

\begin{lem}\label{dimensione}
\begin{itemize}
\item[(i)] {\rm dim}$_\R(C_1) \leq d-1$ and {\rm dim}$_\R(C_2)\leq d-2$.
\item[(ii)] {\rm dim}$_\R(\pi(A_1))\leq d-2$ and  {\rm dim}$_\R(\pi(A_2))\leq d-3$.
\end{itemize}
\end{lem}
\begin{proof}
We prove that for all $y\in C_1$ dim$_R(C_1)_y \leq d-1$ so that dim$(C_1) \leq d-1$.
If $y\in C_1$ then $\hat{\sigma}(y) \neq y$ and $\hat{\sigma}(Y_y) = Y_{\hat{\sigma}(y)}$. Note that $\pi(y) = \pi( \hat{\sigma}(y))$, so $\pi(Y_y)$ and $\pi(Y_{\hat{\sigma}(y)})$ are two different conjugate components of $\tilde{X}_x$. Hence 

$$X_x \cap \pi(Y_y) = X\cap \pi(Y_y) \cap \sigma(\pi(Y_y)) = X \cap \pi(Y_y)\cap  \pi (Y_{\sigma(y)}),$$ 

and so dim$_\R(X_x\cap\pi(Y_y))\leq d-1$. This implies dim$(C_1)_y \leq d-1$ because $(C_1)_y \subset \pi^{-1}(X)_y\subset Y_y$ and  $\pi$ is proper with  finite fibers.

As $Y$ is a normal space, its singular set has complex dimension less than $d-2$. Since Sing$(Y^{\hat{\sigma}}) \subset$ Sing$(Y) \cap Y^{\hat{\sigma}}$ we get dim$_\R($Sing$(Y^{\hat{\sigma}}))\leq d-2$. As $C_2 \subset$ Sing$(Y^{\hat{\sigma}})$, dim$_\R(C_2) \leq d-2$.

Note that $C_i\cap Y^{\hat{\sigma}}_{(d)} =\varnothing$, so $A_i\subset \overline{C_i} \setminus C_i$ and hence dim$(A_i) \leq$ dim$(C_i) -1$ for $i=1,2$, that means dim$(A_1) \leq d-2$ and dim$(A_2) \leq d-3$. Now, since $\pi$ is proper with finite fibers, $\pi(A_i)$ has the same dimension as $A_i, i=1,2 $. This ends the proof.        
\end{proof}

\begin{prop}\label{max}
A point $x\in X$ belongs to $N_d(X,R)$ if and only if the germ $X_x$ has a non-coherent irreducible component $T_x$ of dimension $d$ such that {\rm dim}$(T_x \setminus R_x) =d$. 
\end{prop}
\begin{proof}
We prove first the 'only if' part. To that end, it is enough to show the following
\begin{quotation}
{\em let $x\notin N_d(X,R)$ be such that the germ $X_x$ has an irreducible component $B_x$ of dimension $d$ and {\rm dim}$(B_x \setminus R_x) =d$. Then $B_x$ is coherent.}
\end{quotation}
By Lemma \ref{germs} (i) the complexification $\tilde{B_x}$ is an irreducible component of $\tilde{X}_x$. In addition there is a point $y\in Y^{\hat{\sigma}}$ such that $\pi(Y_y) = \tilde{B_x}$ and dim$(Y^{\hat{\sigma}}_y)=d$.

Let us check that $y\in Y^{\hat{\sigma}} \setminus \overline{C_1}\cup \overline{C_2}$. 

First of all $y\in  Y^{\hat{\sigma}}_{(d)}$. We prove next dim$(Y^{\hat{\sigma}}_{(d)})_y \setminus (\pi^{-1}(R))_y) =d$.

Indeed, $B_x = X_x\cap \pi(Y_y)$, so
$$ \pi^{-1}(B_x)\cap Y_y= \pi^{-1}(X_x)\cap \pi^{-1}(\pi(Y_y))\cap Y_y = \pi^{-1}(X)_y \cap Y_y = \pi^{-1}(X)_y.$$

But $\pi^{-1}(X)_y = (\pi^{-1}(X)\setminus Y^{\hat{\sigma}})_y\cup (Y^{\hat{\sigma}} \setminus Y^{\hat{\sigma}}_{(d)})_y $.

So, $\pi^{-1}(X)_y = (C_1)_y \cup (C_2)_y \cup (Y^{\hat{\sigma}}_{(d)})_y$.

As a consequence 
$$\pi^{-1}(B_x\setminus R_x)= (\pi^{-1}(B_x) \cap Y_y)\setminus \pi^{-1}(R)_y \subset (C_1)_y \cup (C_2)_y \cup((Y^{\hat{\sigma}}_{(d)})_y \setminus \pi^{-1}(R)_y).$$
 
Now $\pi(\pi^{-1}(B_x\setminus R_x) \cap Y_y) =B_x\setminus R_x$ and dim$(B_x\setminus R_x) =d$. As $\pi$ is proper with finite fibers we have dim$((C_1)_y \cup (C_2)_y \cup (Y^{\hat{\sigma}}_{(d)})_y) =d$. Since by Lemma \ref{dimensione} dim$((C_1)_y \cup (C_2)_y)< d$ we get dim$((Y^{\hat{\sigma}}_{(d)})_y \setminus \pi^{-1}(R)_y) =d$.

In particular $y\in \overline{Y^{\hat{\sigma}}_{(d)}\setminus \pi^{-1}(R)}\subset Y^{\hat{\sigma}}_{(d)}$. If $y\in \overline{C_i}$ then $y\in A_i$, hence $\pi(y) =x \in N_d(X,R)$ which is a contradiction. Thus $y\in Y^{\hat{\sigma}}_{(d)}\setminus \overline{C_1} \cup \overline{C_2}$.

Write $\pi^{-1}(x) = \{y_1,\ldots , y_r\}$ and assume $y=y_1$. Take invariant disjoint open neighbourhoods $W_1,\ldots ,W_r$ rispectively of $y_1,\ldots , y_r$  and let $V$ be an open neighbourhood of $x$ such that $\pi^{-1}(V) = \bigcup_i W_i$. We may assume $W_1$ to be connected and disjoint from $\overline{C_1} \cup \overline{C_2}$. As $W_1$ is closed in $\pi^{-1}(V)$ and $\pi$ is proper, also its restriction  $\pi: W_1\to V$ is proper. By Grauert's theorem $\pi(W_1)$ is a complex analytic subset of $V$. Also $\pi:W_1 \to \pi(W_1)$ is the normalization of $\pi(W_1)$. Consider $B= \pi(W_1) \cap X$: it  is a representative of $B_x$. As dim$\pi(W_1) =$ dim$B$ and $\pi(W_1)$ is irreducible, since $W_1$ is normal and connected,  we get $\pi(W_1)$ is a complexification of $B$.

Next we prove that $(\pi_{|W_1})^{-1}(B) = \{y\in W_1: \hat{\sigma}(y) =y\}$. Call this last set $B^{\hat{\sigma}}$ and note that $B^{\hat{\sigma}}= Y^{\hat{\sigma}} \cap W_1$. Also

$$(\pi_{|W_1})^{-1}(B)= \pi^{-1}(\pi(W_1))\cap \pi^{-1}(X) \cap W_1 = \pi^{-1}(X) \cap W_1      
$$  

Thus we have to prove $Y^{\hat{\sigma}} \cap W_1= \pi^{-1}(X) \cap W_1 $. As $W_1\cap \overline{C_1} =\varnothing$, we get $ (\pi^{-1}(X) \setminus Y^{\hat{\sigma}})\cap W_1 =\varnothing$ and this implies  $Y^{\hat{\sigma}} \cap W_1= \pi^{-1}(X) \cap W_1 $.

Finally, since $W_1\cap \overline{C_2} =\varnothing$, we get $ (Y^{\hat{\sigma}}\setminus Y^{\hat{\sigma}}_{(d)} )\cap W_1 =\varnothing$, so dim$(Y^{\hat{\sigma}}_z) =d$ for all $z\in (\pi_{|W_1})^{-1}(a)$ where $a\in B$. By Lemma \ref{coherent} the irreducible component $B_x$ is coherent.

Now we prove the 'if' part of the statement. It is equivalent to prove the following
\begin{quotation}
{\em if all the non-coherent irreducible components $B_x\subset X_x$ of dimension $d$ satisfy {\rm dim}$(B_x\setminus R_x) < d$ then $x\notin N_d(X,R)$}
\end{quotation}
 that is $\pi^{-1}(x)\cap (A_1 \cup A_2) = \varnothing$.

Let $y\in \pi^{-1}(x)$. If $y\notin Y^{\hat{\sigma}}_{(d)}$ then $y\notin A_1\cup A_2$. So assume  $y\in Y^{\hat{\sigma}}_{(d)}$. By Lemma \ref{germs} (ii) we deduce dim$(\pi (Y_y) \cap X_x) =d$. Consequently $B_x = \pi (Y_y)\cap X_x$ is an irreducible analytic germ of dimension $d$ in $X_x$. Thus $B_x$ is an irreducible component of $X_x$ and $\pi (Y_y)$ is its complexification. In particular $\pi (Y^{\hat{\sigma}}_y) \subset B_x$. There are two cases.

{\sc Case 1.} {\em $B_x$ is not coherent.} Then dim$(B_x\setminus R_x) <d$, so dim$(\pi( Y^{\hat{\sigma}}_y)\setminus R_x) < d$. Denote $E_y = (Y^{\hat{\sigma}}_{(d)})_y \setminus \pi^{-1}(R)_y$. As $\pi$ has finite fibers and $\pi(E_y)\subset \pi( Y^{\hat{\sigma}}_y)\setminus R_x$ we get dim$(E_y) <d$. As $R$ is closed and dim$(E_y) <d$ we get 

$$ (Y^{\hat{\sigma}}_{(d)})_y = \overline{(Y^{\hat{\sigma}}_{(d)})_y\setminus E_y} \subset \overline{\pi^{-1}(R)}_y = \pi^{-1}(R)_y$$

As a consequence $(Y^{\hat{\sigma}}_{(d)})_y\setminus \pi^{-1}(R)_y = \varnothing$, hence $y\notin A_1\cup A_2$.

{\sc Case 2.} {\em $B_x$ is coherent.} Proceeding as before we find an invariant neighbourhood $W_1$ of $y$ in $Y$ such that $\pi(W_1)$ is an irreducible complex analytic subset of an invariant neighbourhood $V$ of $x$ in $\tilde X$ and $B= \pi (W_1) \cap X$ is a representative of the germ $B_x$. As $B_x$ is coherent, by Lemma \ref{coherent} we may shrink $V$ and $W_1$ in order to get
\begin{enumerate}
\item $\pi^{-1}(X)\cap W_1 = Y^{\hat{\sigma}}\cap W_1.$
\item dim$(Y^{\hat{\sigma}}_z) =d$ for all $z\in \pi^{-1}(a)$, for $a\in B\cap V.$  
\end{enumerate}
Condition (1) is equivalent to $(\pi^{-1}(X) \setminus Y^{\hat{\sigma}})\cap W_1 =\varnothing$, so $y\notin \overline{C_1}$. Condition (2) means $y\notin \overline{C_2}$ . So $y\notin \overline{C_1}\cup \overline{C_2}$ and $y\notin A_1\cup A_2 \subset \overline{C_1}\cup \overline{C_2}$. 

Thus $\pi^{-1}(x) \cap (A_1 \cup A_2) = \varnothing$ as required.   
\end{proof}

\smallskip

\begin{remark} By definition the set $R_d$ is empty. We introduced the  set $R$ of Proposition \ref{max} because in this way its proof remains valid for the set $N_k(Z_k,R_k)$,  as we will see in the proof of Theorem \ref{ncdes}.   
\end{remark}

\smallskip

Next remark deals with irreducible components of analytic set germs.

\begin{remark}\label{componenti} Let $A_x\subset B_x\subset \R^n_x$ be analytic set germs and let $T_x$ be an irreducible component of $B_x$. Assume $T_x\subset A_x$. Then, $T_x$ is an irreducible component of $A_x$. Indeed since $T_x$ is irreducible there is an irreducible component $S_x$ of $A_x$ such that $T_x\subset S_x$. In the same way $S_x$ being irreducible in $B_x$ is contained in an irreducible component $S'_x $ of $B_x$, hence $T_x\subset S'_x$ and so $T_x = S'_x= S_x$.    
\end{remark}

{\sc Proof of Theorem  \ref{ncdes}.}
Each C-semianalytic set $R_k$ is $\mathcal A (\tilde X^\sigma)$-definable by definition.  Hence $N_k(Z_k,R_k)$ is C-semianalytic of codimension at least $2$, that is of dimension $\leq k-2$ by Lemma \ref{dimensione}. So (i) is proved.

Next we prove (iii). Take $x\in N(X)$ and let $T_x$ be an irreducible component of $X_x$ of dimension $\ell$ which is not coherent. Certainly $\ell\geq 2$ because curves are always coherent. 

We claim { $T_x$ is an irreducible component of $(Z_\ell)_x$.}

In fact let $Z$ be an irreducible component of $X$ such that $T_x\subset Z$. By Remark \ref{componenti}, $T_x$ is an irreducible component of $Z_x$. As $Z=\bigcup_{l\geq 0}{\rm Reg}({\rm Sing}_l(Z))$, $Z_x = \bigcup_{l\geq 0}\overline{{\rm Reg}({\rm Sing}_l(Z))_x}^{\zar}$, there exists $l$ such that $T_x\subset \overline{{\rm Reg}({\rm Sing}_l(Z))_x}^{\zar} \subset Z_x$. Since all irreducible components of $\overline{{\rm Reg}({\rm Sing}_l(Z))_x}^{\zar}$ have the same dimension and $(T_x)$ has dimension $\ell$ we deduce dim ${\rm Sing}_l(Z)=\ell$. By Remark \ref{componenti}, there exists an irreducible component $Z'$ of ${\rm Sing}_l(Z)$ of dimension $\ell$ such that $T_x \subset  Z'_x$. Of course $Z' \in \mathcal{F}_\ell$, so $T_x$ is an irreducible component $(Z_\ell)_x$.

We claim: {\em $x\in N_\ell(Z_\ell,R_\ell)$}. By Proposition \ref{max} this is the same as to claim: 
\begin{quotation}{ The germ $(Z_\ell)_x$ has an irreducible component $B_x$ of dimension $\ell$ which is not coherent and verifies {\rm dim}$(B_x \setminus (R_\ell)_x) =\ell$.}
\end{quotation}

It is enough to check  dim$(T_x \setminus (R_\ell)_x) =\ell$. If not then dim$(T_x \setminus (R_\ell)_x) <\ell$ so 
$$T_x \subset \overline{(R_\ell)_x}^{\zar} = \bigcup_{j=\ell+1}^d\overline{(Z_{j,(j)})_x}^{\zar}$$

As a consequence there exists $\ell+1\leq j\leq d$ such that $T_x \subset \overline{(Z_{j,(j)})_x}^{\zar}\subset X_x$. By Lemma \ref{germs} $T_x$ is an irreducible component of $\overline{(Z_{j,(j)})_x}^{\zar}$ which is a contradiction because all the irreducible components of $\overline{(Z_{j,(j)})_x}^{\zar}$ have dimension  $j>\ell$.

So $N(X)\subset \bigcup N_k(Z_k,R_k)$. To prove the converse inclusion take a point $x\in N_l(Z_l,R_l)$ for some $2\leq l\leq d$. We will prove that the germ $X_x$ has a non-coherent component of dimension $\geq l$. This implies in particular $x\in N(X)$.

As $x\in N_l(Z_l,R_l)$ the germ $(Z_l)_x$ has an irreducible component $T_x$ of dimension $l$ that is not coherent and such that dim$(T_x\setminus (R_l)_x)$ has again dimension $l$. As $(Z_l)_x\subset X_x$, there is an irreducible component $A_x$ of $X_x$ that contains $T_x$. If $T_x=A_x, A_x$ is not coherent and $x\in N(X)$. Otherwise $l={\rm dim}T_x < {\rm dim}A_x =j$. As  
$$X_x =\bigcup_{k=0}^d\overline{(Z_{k,(k)})_x}^{\zar},$$  
we deduce that $A_x$ is an irreducible component of some $\overline{(Z_{k,(k)})_x}^{\zar}$. All components of this set have dimension $k$ hence $j=k$.

Assume for a contradiction that $A_x$ is coherent. Then $A_x$ is pure $j$-dimensional. So $T_x\subset A_x\subset (Z_{j,(j)})_x\subset (R_l)$ and this is a contradiction. (iii) is proved.

Finally using what was proved before, we deduce that if $X_x$ has a non-coherent component of dimension $j$ then $x\in N_j(Z_j,R_j)$, while by (iii) $x\in N_j(Z_j,R_j)$ implies $X_x$ has a non-coherent component of dimension $\geq j$. 
  
So (ii) holds true and the proof is complete.
\qed

\smallskip

We give now some examples illustrating   the ideas under the results above.

\begin{examples}

\begin{itemize}{\parindent =0pt
\item[(i)] Consider the C-analytic set $X$ in $\R^4$ with coordinates $(x,y,w,z)$ defined by the equation
 $x^3-x^2wz -wy^2 =0$. 
 $X$ is constructed as follows. Consider the family of plane cubics $\{x^3-cx^2 -y^2=0\}$ depending on the parameter $c$. For $c<0$ the corresponding curve has a duble nodal point  at $(0,0)$, for $c=0$ it has a cusp while for $c>0$ the curve has an isolated singular point. Consider the cubic with parameter $c$ as a curve in the plane $\{w=1, z=c\}$ in $\R^4$. The set $X$ is the corresponding family of cones from the point $(0,0,0,c)$ over those cubics. Note that for $c=1$ the corresponding cone is the classical Cartan's cone.

We get dim$X =3$, Sing$(X)$ is the plane $\{x=0,y=0\}$ which is the locus of the  lines joining the vertex with the singular point of the cubic. The locus of vertices is the line $\{(0,0,0, c)\}\subset {\rm Sing}(X)$.

The set $X$ is irreducible. $X_x$ is also irreducible except for  points belonging to the double lines minus the corresponding vertex. In particular $X_x$ is irreducible if $x$ is a vertex. 

It is easy to see that $N(X)$ is the union of the line $\{(0,0, w,0)\}$ and the half line $\{(0,0,0,z), z\geq 0\}$. The line joins the cusp of the cubic $\{x^3 -y^2=0\}$ with the corresponding vertex and  is the boundary of the half plane of lines joining vertices and isolated points. The half line contains the vertices of the cones over cubics with isolated points, that is the points where the sections $X\cap \{z=c\geq 0 \}$ are not coherent. Together they give the set $A_1$ where the $2$-dimensional and  the $3$-dimensional part of $X$ meet.         
\item[(ii)]
Let $X_1 =\{(x^2-(z+1)y^2)^2z-u^2 =0\}\subset \R^4$ and $X_2 = \{u=0\}$. Let us prove that 
$$N(X_1)= \{x^2-y^2 =0,z=0,u=0\} \cup \{(0,0,-1,0\}$$ 

while $N(X_1 \cup X_2) = \{x^2-y^2 =0,z=0,u=0\}.$

Compute first $N(X_1)$. Consider $Z_1 = \{(x^2-(z+1)y^2)^2z-u^2 =0\}\subset \C^4$ which is a complexification of $X_1$. The map $\pi: \C^3 \to \C^4$ given by $\pi(x,y,v) = (x,y, v^2, v(x^2-(v^2+1)y^2))$ is the normalization of $Z_1$. The singular locus of $Z_1$ is 

$${\rm Sing}(Z_1) = \{x^2 -(z+1)y^2 =0, u=0\}.$$

It holds 
$$\quad \pi^{-1}(X_1) = \R^3 \cup \{(\pm s\sqrt{1-t^2},s,it):  0<t\leq 1\} \cup \{(0,0,it): |t |>1\}.$$

Consider the images of these three sets                                     
\begin{align*}
\hspace{1,5 truecm} T_1 &= \pi(\R^3) = X_1 \cap \{z\geq 0\}\\
\hspace{1,5 truecm} T_2 &= \pi(\{(\pm s \sqrt{1-t^2},s,it):0<t\leq 1\}) =\\
&= \{(\pm  s \sqrt{1-t^2},s,-t^2,0): 0<t\leq 1\}\\
\hspace{1,5 truecm} T_3 & = \pi(\{(0,0,it): |t |>1\}) = \{(0,0,-t^2,0)\}
\end{align*}

Thus $T_1$ is the set of points of $X_1$ of (maximal) dimension $3$. By Theorem \ref{ncdes} (or better by Proposition \ref{max}) we have
$Y^{\hat{\sigma}}= Y^{\hat{\sigma}}_{(3)} = \R^3$ and
$$\quad \quad \quad C_1 =  \{(\pm s\sqrt{1-t^2},s,it):  0<t\leq 1\} \cup \{(0,0,it): |t |>1\}, \quad \ C_2 = \varnothing.$$

So the set of points where $X_1$ has a non-coherent component of dimension $3$ is 
$\pi (\R^3 \cap \overline{C_1})$, that is

$$ \pi (\R^3 \cap (\overline{\{(\pm s\sqrt{1-t^2},s,it):  0<t\leq 1\}\cup \{(0,0,it): |t |>1\}}))=$$
 $$=\{x-y =0, z=0, u=0\} \cup \{x+y=0, z=0, u=0\}.$$

To find the set of points of $X_1$ that have a non-coherent component of smaller dimension we have to look to $T_2\cap T_3 = \{ x^2 -(z+1) y^2 =0, z<0\}$ which is an open subset of Witney's umbrella. The only point of non-coherence of this set is the point $(0,0,-1,0)$. So

$$N(X_1) = \{x^2-y^2 =0,z=0,u=0\} \cup \{(0,0,-1,0)\}.$$

On the other hand $N(X_1 \cup X_2) = \{x^2-y^2 =0,z=0,u=0\}$ because at these points of $X_1\cup X_2$ the corresponding germ has a non-coherent component of dimension $3$ while at the point $(0,0,-1,0)$ the unique irreducible component is $\{u=0\}$ which is coherent.
}\end{itemize}
\end{examples}

\section{Amenable C-semianalytic sets and irreducible components.}

In this section we would like to develop a theory of irreducibility and irreducible components in the class of C-semianalytic sets. This theory should generalize theories of irreducible components  for C-analytic sets, complex analytic subsets of Stein manifolds and semialgebraic sets because all of them are C-semianalytic sets. In the algebraic, complex analytic or Nash setting a geometric object is irreducible if it is not the union of two proper geometric objects of the same nature. In C-semianalytic setting this definition cannot work because as soon as a C-semianalytic set $S$ has two different points $p,q\in S$ it is enough to take an open C-semianalytic neighbourhood $W$ of $p$ such that $q\notin W$ and $S= (S\cap W) \cup (S\setminus \{p\})$ would be reducible.   

In the previous settings the fact that a geometric object $X$ is irreducible is equivalent to the fact that the corresponding ring of polynomial, analytic or Nash functions on $X$ is an integral domain. This equivalence suggests to attach to each C-semianalytic set $S\subset \R^n$ the ring $\Oo(S)$ of those functions that admit an analytic extension to an open neighbourhood of $S$ in $\R^n$. This means that $\Oo(S)$ is the quotient of $\Oo(R^n)$ by the ideal $I(S)$ of the analytic functions vanishing at $S$.

Here a first difficulty appears, in a certain sense similar to what happens in 
 Cartan's Example \ref{a(z)} in  Chapter 1,  
namely there are C-semianalytic sets $S$ such that any analytic function vanishing on $S$ is the zero function. For instance  in  Example \ref{segmenti} 
the  C-semianalytic set $S$ cannot be considered irreducible even if $I(S)=\{0\}$  is a prime ideal of $\Oo(\R^2)$.  

The zeroset of $I(S)$ is the smallest C-analytic set containing $S$, i.e. the so called {\em Zariski closure} $\ol{S}^{\zar}$ of $S$. Note that in  Example \ref{segmenti} dim $S=1$ while dim $\ol{S}^{\zar} =2$.

So a possible request could be to ask dim $S$ to be the same as  dim $\ol{S}^{\zar}$. But again this cannot work: it is enough to consider the $S$ of the Example \ref{segmenti} in the plain $\{z=0\} \subset \R^3$ and to add to $S$ a different plane. The new $S'$ has dimension $2$ as its Zariski closure. Its Zariski closure has two irreducible components, while $S'$ should get infinitely many components.    

So besides the request on dimension, we give two  attractive properties that a good theory should get. 



\begin{itemize} 
\item[(i)] Irreducible objects should be connected.  
\item[(ii)] If $U$ is a neighbourhood  of an irreducible object $S$, since the restriction map $\an(\ol{S}^{\zar}_U)\hookrightarrow\an(S), f\mapsto f_{|S}$ is injective, the Zariski closure of $S$ in $U$ should be irreducible, too.
\end{itemize}
Consider a C-semianalytic set $S=X\cap U$, where $X$ is a C-analytic set and $U$ is an open C-semianalytic set. If  $\Oo(S)$ is  a domain,
then $S$ verifies all the requests above.
This remark justifies  the following definition.

\begin{defn}
A C-semianalytic subset $S\subset \R^n$ is  \em amenable  \em if it is a finite union of C-semianalytic sets of the type $X\cap U$, where $X\subset \R^n$ is a C-analytic set, and $U\subset \R^n$ is an open C-semianalytic set.  
\end{defn}
By definition, the Zariski closure of an amenable C-semianalytic set $S$ has the same dimension as $S$.

The class of amenable C-semianalytic sets is the smallest family of subsets of $M$ that contains C-analytic sets, open C-semianalytic sets and is closed for finite unions and intersections. A first example of amenable C-semianalytic sets is given by global basic semianalytic sets, that is, sets of the type $S=\{f=0,g_1>0,\ldots,g_s>0\}$ where $f,g_1,\ldots,g_s\in\an(M)$. Consequently, global semianalytic sets and semialgebraic subsets of $\R^n$ are amenable C-semianalytic sets. This class has another remarkable property: as we will see, it has a  theory of irreducible components similar to the one for semialgebraic sets developed in \cite{fg}. This theory  generalizes the classical ones for other families of classical sets as complex algebraic and analytic sets, $C$-analytic sets, Nash sets, semialgebraic sets, etc.

We present first some examples. Namely
\begin{enumerate} 
\item The first example is a C-semianalytic set which is not amenable in any of its neighbourhoods. 
\item The second is a C-semianalytic set that is locally C-analytic but it is not amenable. 
\item The third is an amenable  C-semianalytic set that is not a global C-se\-mi\-a\-na\-ly\-tic set, so that one sees that the class of amenable C-semianalytic sets is strictly larger than the one of global semianalytic sets.
\end{enumerate}

\begin{examples}\label{tame}\hfill 
\begin{enumerate}
\item[(i)] {\parindent = 0pt 
For each integer $k\geq 0$ consider the basic semianalytic set 
$$
S_k=\{z^2-(y-k)x^2=0, y\geq k\}\subset\R^3. 
$$
Observe that the family $\{S_k\}_{k\geq 0}$ is locally finite, so $S=\bigcup_{k\geq 0}S_k$ is a connected  C-semianalytic set. 
Observe that $S$ is not an amenable C-semianalytic subset of any open neighborhood $U$ of $S$ in $\R^3$ because at the points $\{(0,y,0):\ y\geq 0\}$ the family of the Zariski closures $\{\ol{S}^{\zar}_{k,U}\}_{k\geq1}$ is not locally finite, so $\ol{S}^{\zar}=\R^3$. Nevertheless all finite unions  $S=\bigcup_{0\leq k\leq  \ell}S_k$ are amenable.

\item[(ii)]
Let $S_{kj}$ be the subset of $\R^3$ 
$$
\hphantom{S_{kj}=}S_{kj}=\Big\{x^2+2\Big(1+\frac{1}{j}\sin y\Big)xz+z^2=0, 2k\pi\leq y\leq(2k+1)\pi\Big\}\cup\{x=0,z=0\}
$$
for $j\geq 1$ and $k\in\Z$ . 

The Zariski closure of $S_{\ell j}$ is $X_j=\bigcup_{k\in\Z}S_{kj}$. Define $S=\bigcup_{j\geq1}S_{jj}$. It holds that for each $x\in \R^3$ there exists an open neigborhood $U^x$ such that $S\cap U^x=\{f=0\}\cap U^x$ for some $f\in\an(\R^3)$. As $\bigcup_{j\geq1}X_j\subset\ol{S}^{\zar}$, we conclude that $\ol{S}^{\zar}=\R^3$. 

Consequently, $S$ is not an amenable C-semianalytic set.}

\item[(iii)]{\parindent = 0pt 
Consider the C-semi\-analytic set $S=\bigcup_{k\geq1} S_k$ where
$$
S_k=\{0<x<k,\ 0<y<1/k\}\subset\R^2. 
$$
As $S$ is an open C-semianalytic set, it is amenable. 
We claim that  $S$ is not a global C-semianalytic set.
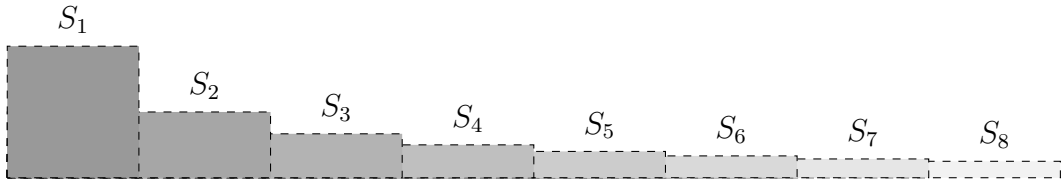
\begin{figure}[ht]
\centering
\begin{tikzpicture}[x=.35cm,y=.35cm]

\draw[dashed,fill=black!5!white] (0,0) -- (0,0.625) -- (40,0.625) -- (40,0) -- (0,0);
\draw[dashed,fill=black!10!white] (0,0) -- (0,0.714) -- (35,0.714) -- (35,0) -- (0,0);
\draw[dashed,fill=black!15!white] (0,0) -- (0,0.83) -- (30,0.83) -- (30,0) -- (0,0);
\draw[dashed,fill=black!20!white] (0,0) -- (0,1) -- (25,1) -- (25,0) -- (0,0);
\draw[dashed,fill=black!25!white] (0,0) -- (0,1.25) -- (20,1.25) -- (20,0) -- (0,0);
\draw[dashed,fill=black!30!white] (0,0) -- (0,1.67) -- (15,1.67) -- (15,0) -- (0,0);
\draw[dashed,fill=black!35!white] (0,0) -- (0,2.5) -- (10,2.5) -- (10,0) -- (0,0);
\draw[dashed,fill=black!40!white] (0,0) -- (0,5) -- (5,5) -- (5,0) -- (0,0);

\draw (2.5,6) node{$S_1$};
\draw (7.5,3.5) node{$S_2$};
\draw (12.5,2.67) node{$S_3$};
\draw (17.5,2.25) node{$S_4$};
\draw (22.5,2) node{$S_5$};
\draw (27.5,1.83) node{$S_6$};
\draw (32.5,1.714) node{$S_7$};
\draw (37.5,1.625) node{$S_8$};

\end{tikzpicture}
\caption{Amenable C-semianalytic set $S=\bigcup_{k\geq1} S_k$  of Example \ref{tame} (iii)}
\end{figure}

Otherwise there exist finitely many analytic functions $g_{ij}\in\an(M)$ not all identically zero such that $S=\bigcup_{i=1}^r\{g_{i1}>0,\ldots,g_{is}>0\}$. Notice that the boundary of $S$ is contained in the the C-analytic set $X=\bigcup_{i=1}^r\bigcup_{j=1}^s\{g_{ij}=0\}$. Thus, there exist indices $i,j$ such that $g_{ij}$ is not identically zero but it vanishes on infinitely many segments of the type $(a_k-\veps_k,a_k+\veps_k)\times\{1/k\}$ for different $k$ where $a_k\in(k,k+1)$ and $\veps_k>0$. But this implies that $g_{ij}$ vanishes identically in infinitely many lines of the type $y=1/k$, so $g_{ij}=0$, which is a contradiction. Thus, $S$ is not a global C-semianalytic set.
}
\end{enumerate}

\end{examples}

Next lemma deals with an open C-semianalytic set in a C-analytic set. 

\begin{lem}\label{openX}
Let $X\subset M$ be a C-analytic set and let $S\subset X$ be a C-semianalytic set that is open in $X$. Then there exists an open C-semianalytic set $W$ such that $S=X\cap W$. In particular, $S$ is an amenable C-semianalytic set.
\end{lem}
\begin{proof}
Fix a point $x\in X$ and let $U^x$ be an open neighborhood of $x$ in $M$ such that $S\cap U^x$ is a global C-semianalytic set. By the finiteness property  there exists a finite union $W^x$ (maybe empty) of open basic C-semianalytic subsets of $M$ such that $S\cap U^x=X\cap W^x$. Denote $W=(M\setminus X)\cup\bigcup_{x\in X}W^x$, which is an open C-semianalytic set. Observe that $S=X\cap W$, as required.
\end{proof}

We give now some basic properties of amenable C-semianalitic sets in a real analytic manifold $M$.
\begin{prop}The family of amenable C-semianalytic sets in a real analytic manifold $M$ is closed under
\begin{itemize}
\item finite unions and  finite intersections, 
\item inverse image under analytic maps between real analytic manifolds, 
\item interior part, 
\item connected components or unions of some of them,
\item set of points of pure dimension $k$.
\end{itemize}
\end{prop}
\begin{proof}
The first  property is clear. For the remaining ones we proceed as follows. 

\em Inverse image\em. Let $f:M\to N$ be an analytic map between real analytic manifolds and let $S\subset N$ be an amenable C-semianalytic set. 
We may assume that $S=X\cap W$ where $X$ is a C-analytic subset of $N$ and $W$ is an open C-semianalytic subset of $N$. As $f^{-1}(S)=f^{-1}(X)\cap f^{-1}(W)$ and $f^{-1}(X)$ is a C-analytic subset of $M$ and $f^{-1}(W)$ is an open C-semianalytic subset of $M$, we get   $f^{-1}(S)\subset M$ is an amenable C-semianalytic set, as required.

\em Interior part \em. Let $S\subset M$ be an amenable C-semianalytic set and let $X$ be its Zariski closure. It holds that $\stackrel{\circ}{S}=S\setminus \overline{X\setminus S}=S\cap U$ where $U=M\setminus\overline{X\setminus S}$ is an open C-semianalytic set (closure and interior part are taken in $X$). Then $\stackrel{\circ}{S}$ is amenable.

\em Connected components\em. Let $S\subset M$ be an amenable C-semianalytic set and let $\{S_i\}_{i\geq1}$ be the family of the connected components of $S$. We know  that $\{S_i\}_{i\geq1}$ is a locally finite family, so the connected components of $S$ are open and closed subsets of $S$. Let ${\mathfrak F}\subset\{i\in \N : i\geq 1\}$ be any non-empty subset. We claim that $T=\bigcup_{i\in{\mathfrak F}}S_i$ is an amenable C-semianalytic set. Notice that
$$
T=S\cap\left (M\setminus\bigcup_{i\not\in{\mathfrak F}}\overline{S_i}\right )=S\cap U
$$
where $U=M\setminus\bigcup_{i\not\in{\mathfrak F}}\overline{S_i}$ is an open C-semianalytic set. This last claim follows because the family $\{S_i\}_{i\not\in{\mathfrak F}}$ is locally finite and each $\overline{S_i}$ is a C-semianalytic set. Consequently $T$ is an amenable C-semianalytic set.

\em Set of points of pure dimension $k$\em. Let $S^*_{(k)}$ be the subset of points $x\in S$ such that the germ $S_x$ is pure dimensional and $\dim S_x=k$. We have $S^*_{(k)}=S\cap(M\setminus\bigcup_{j\neq k}\overline{S_{(j)}})$ for each $k\geq 0$ where $S_{(j)}$ is the set of points of $S$ of local dimension $j$ for $j\geq 0$. Consequently, $S^*_{(k)}$ is amenable C-semianalytic.
\end{proof}

\begin{remarks}
(i) A locally finite union of amenable C-semianalytic sets is not in general an amenable C-semianalytic set. Each C-semianalytic set is a locally finite union of basic semianalytic sets, but there are many C-semianalytic sets that are not amenable. 

(ii) Let $U=\R^3\setminus S$ where $S$ is the closed C-semianalytic set defined in Example \ref{tame} (ii). As $U$ is an open C-semianalytic set, it is amenable. However its complement $S=\R^3\setminus U$ is not amenable.

(iii) As we will see in  Example \ref{dimktamenot}, the (C-semianalytic) set of points of local dimension $k$ of an amenable C-semianalytic set may  be not  amenable.
\end{remarks}

Next example shows the bad behaviour of the closure of amenable C-semianalytic sets.

\begin{examples}\label{ex} \hfill 
\begin{itemize}{\parindent=0pt
\item [(i)]  Consider the open C-semianalytic set
$
S=\bigcup_{k\geq1}S_k$ where 
$$
S_k=\{y<kx+z,y>kx-z,k<x<k+1,0<z<1\}\subset\R^3
$$
Clearly $S$ is an amenable C-semianalytic set. We have
$\ol{S}=\bigcup_{k\geq1}\ol{S_k}$ where
$$\ol{S_k}=\{y\leq kx+z,y\geq kx-z,k\leq x\leq k+1,0\leq z\leq1\}.
$$
Consider the C-analytic set $X=\{z=0\}$ and observe that 
$
\ol{S}\cap X=\bigcup_{k\geq1}\ol{S_k}\cap X$ where $\ol{S_k}\cap X=\{y=kx,k\leq x\leq k+1,z=0\}$
which is not an amenable C-semianalytic set. 

Indeed,  $\ol{\ol{S}\cap X}^{\zar}=\{z=0\}$ has dimension $2$ while $\ol{S}\cap X$ has dimension $1$. As $X$ is amenable, $\ol{S}$ is not amenable.

\item [(ii)] 
The closure $T=\ol{S}$ in $\R^2$ of the amenable C-semianalytic set $S$ in Example \ref{tame}(ii)  is not an amenable C-semianalytic set. To prove this, one can apply Theorem \ref{algor1} below.}
\end{itemize}
\end{examples}

\begin{defn}
An amenable C-semianalytic set $S\subset \R^n$ is {\em irreducible} if $\Oo(S)$ is a domain. 
\end{defn}
One deduces straightforwardly the following facts concerning irreducibility.
\begin{itemize} 

\item[(i)] Irreducible amenable C-semianalytic sets are connected because the ring of analytic functions of a disconnected amenable C-semianalytic set is the direct sum of the rings of analytic functions of its connected components, so it contains zero divisors. In particular, a real analytic manifold is irreducible if and only if it is connected. 
\item[(ii)] Let $U$ be a neighbourhood  of an irreducible amenable C-semianalytic set $S$ and $ \ol{S}^{\zar}_U)$ its Zariski closure in $U$. Since the restriction map $\an(\ol{S}^{\zar}_U)\hookrightarrow\an(S), f\mapsto f_{|S}$ is injective, also $ \ol{S}^{\zar}_U)$ is irreducible. As $S$ is amenable, $\dim \ol{S}^{\zar}_U=\dim S$.
\item[(iii)] An amenable C-semianalytic set that is the image of an irreducible amenable C-semi\-analytic set under an analytic map is irreducible. In particular, the irreducibility of amenable C-semianalytic sets is preserved under analytic diffeomorphisms.
\item[(iv)] Let $T\subset S\subset\R^n$ be amenable C-semianalytic sets such that $T$ is irreducible. Then the ideal
$$
\ideal(T,S)=\{f\in\an(S):\ f_{|T}=0\}
$$ 
is a prime ideal of $\an(S)$ because $\an(T)$ is an integral domain and $\ideal(T,S)$ is the kernel of the restriction homomorphism $\an(S)\to\an(T),\ f\mapsto f|_T$.
\end{itemize}

\begin{lem}\label{charirred}
Let $S$ be an amenable C-semianalytic set. The following assertions are equivalent.
\begin{itemize}
\item[(i)] $S$ is irreducible.
\item[(ii)] For each open neighborhood $U$ of $S$ in $M$, the C-analytic set $\ol{S}^{\zar}_U$ is irreducible. 
\item[(iii)] If $f\in\an(S)$ and $\dim\ceros(f)=\dim S$, then $f$ is identically zero.
\end{itemize}
\end{lem}
\begin{proof}
(i) $\Longrightarrow$ (ii). The ideal $\ideal(S,U)$ is a prime ideal of $\an(U)$, so $\ol{S}^{\zar}_U=\ceros(\ideal(S,U))$ is an irreducible C-analytic subset of $U$.

(ii) $\Longrightarrow$ (iii). Let $f'\in\an(U)$ be an analytic extension of $f$ to an open neighborhood $U\subset M$ of $S$. As $f'$ vanishes on a subset of maximal dimension of the irreducible C-analytic set $\ol{S}^{\zar}_U$, we have $f'_{|\ol{S}^{\zar}_U}\equiv 0$, so $f=f'_{|S}\equiv0$.

(iii) $\Longrightarrow$ (i).  Let $f_1, f_2\in\an(S)$ be such that $f_1f_2\equiv 0$. Let $x\in S$ be such that the germ $S_x$ is regular of maximal dimension. As $S_x$ is irreducible, we may assume $f_1$ is identically zero on an open neighborhood of $x$ in $S$. Thus, $\dim \ceros(f_1)=\dim S$, so $f_1\equiv0$. Consequently, $S$ is irreducible.
\end{proof}

\begin{remark} Example \ref{segmenti} and Example (i) from Examples \ref{tame} show clearly that the notion of irreducibility cannot be applied to C-semianalytic sets that are not amenable.
\end{remark}

Next natural step is to give the notion of {\em irreducible component} of an amenable C-semianalytic set. This notion should generalize the theory for C-analytic sets, for Stein spaces endowed with their real structure or semialgebrais sets.

Let $S$ be a subset of an  analytic manifold $M$.  We say that $T\subset S$ is an \em $S$-tame C-semianalytic set \em if there exists an open neighborhood $U\subset M$ of $S$ such that $T$ is an amenable C-semianalytic subset of $U$. There is no ambiguity to say that an $S$-tame C-semianalytic set $T\subset S$ is \em irreducible \em if $\an(T)$ is an integral domain. 

We get the following definition.

\begin{defn}\label{irredcomptame0} Let $S\subset M$ be an amenable C-semianalytic. A countable locally finite family $\{S_i\}$ of amenable C-analytic subsets of $S$ is a family of {\em irreducible components} if it satisfies the following properties. 
\begin{itemize}
\item[(1)] Each $S_i$ is irreducible.
\item[(2)] If $S_i\subset T\subset S$ and $T$ is an irreducible $S$-tame C-semianalytic set, then $S_i=T$.
\item[(3)] $S_i\neq S_j$ if $i\neq j$.
\item[(4)] $S=\bigcup_{i\geq1} S_i$.
\end{itemize} 
We call the family $\{S_i\}$  a family of {\em weak} irreducible components if we remove the hypothesis $S_i$ to be amenable and if the local finiteness is supposed true in $S$ and not necessarily in $M$.
\end{defn} 

\begin{remark}\hfill

\begin{itemize}
\item[(i)]  Of course for a set $X$ the notion {\em to be irreducible} depends on the structure considered on it. Recall for instance Example \ref{doppioombrello} in Chapter 2. Nevertheless when  $X$ is a C-analytic set, its irreducible components $\{X_i\}_{i\geq1}$ as a C-analytic set coincide with the ones obtained if we consider $X$ as an amenable C-semianalytic set.
\begin{proof}
Let us check that $\{X_i\}_{i\geq1}$ satisfies the conditions in Definition \ref{irredcomptame0}. Only condition (2) requires a comment. Let $T$ be an irreducible amenable C-semianalytic set such that $X_1\subset T\subset X$. Hence there exists $j\geq1$ such that $X_1\subset T\subset X_j$, so $j=1$ and $T=X_1$, as required.
\end{proof}
\item[(ii)] Recall that if $X$ is an irreducible complex analytic subset of a Stein manifold  then $\Reg(X)=X\setminus\Sing(X)$ is connected as we saw in Chapter 2, Section 1A. If we consider the real structure $(X^\R,\an_X^\R)$ induced on $X$, we deduce that $X^\R$ is irreducible as a C-analytic set. In addition if $X$ is a general complex analytic subset of a Stein manifold, the irreducible components of $X$ are  the closures in $X$ of the connected components of $\Reg(X)$ as proved in Theorem \ref{decomposition} of Chapter 2. Consequently the irreducible components of $X$ and the irreducible components of $X^\R$ as a C-analytic set coincide. Thus, the irreducible components of $X$ as a complex analytic set coincide with the ones obtained if we consider $X$ as an amenable C-semianalytic set. 
\item[(iii)] If $S$ is a semialgebraic set, its irreducible components as  semialgebraic set, described by Fernando and Gamboa, coincide with the ones obtained  considering  $X$ as an amenable
 C-semianalytic set. 
\end{itemize}
\end{remark}

We shall prove the following.
\begin{thm}\label{irredcomp1}
Let $S \subset M$ be an amenable C-semianalytic set. Then there exists a family of irreducible components $\{S_i\}_{i\geq1}$ of $S$ and it is unique. In addition it satisfies 
\begin{itemize}
\item[(i)] $S_i=\ceros(\ideal(S_i,S))$ for $i\geq1$. In particular, $S_i$ is a closed subset of $S$.
\item[(ii)] The ideals $\ideal(S_i,S)$ are the minimal prime (saturated) ideals of $\an(S)$.
\end{itemize}
\end{thm}

The strategy is to prove first existence and uniqueness of weak irreducible components and then to prove that they are in fact irreducible components. But before we need to develop a more complete theory of amenable C-semianalytic sets.

\subsection{Characterization of amenable C-semianalytic sets.}

In  Definition \ref{Csemi} we defined  C-semi\-analytic set as locally finite countable unions of global basic semianalytic sets. Here we obtain the corresponding result for amenable C-semianalytic sets.

\begin{prop}\label{normal2}
Let $S$ be a C-semianalytic set in a real analytic manifold $M$. Then $S$ is amenable if and only if there exists a countable locally finite family of global basic semianalytic sets $\{S_i\}_{i\geq1}$ such that $S=\bigcup_{i\geq1}S_i$ and the family $\{\ol{S_i}^{\zar}\}_{i\geq1}$ is locally finite (after eliminating repetitions).
\end{prop}

Next result allows  to reduce the proof of Proposition \ref{normal2} to the case of open C-semianalytic subsets of  real analytic manifolds. 

\begin{lem}\label{decomp}
Let $S\subset M$ be an amenable C-semianalytic set. Then there exist finitely many C-analytic sets $Z_1,\ldots,Z_r$ and finitely many open C-semianalytic sets $U_1,\ldots,U_r$ such that $S=\bigcup_{i=1}^r(Z_i\cap U_i)$, each $Z_i\cap U_i$ is a real analytic manifold, $Z_i$ is the Zariski closure of $Z_i\cap U_i$ and $\dim(Z_{i+1}\cap U_{i+1})<\dim(Z_i\cap U_i)$ for $i=1,\ldots,r-1$.
\end{lem}

\begin{proof} 
We proceed by induction on the dimension $d$ of $S$. If $S$ has dimension $0$, the result is clearly true. Assume the result true if the dimension of $S$ is less than $d$ and let us check that this implies that it is also true if $S$ has dimension $d$.
 
As $S$ is an amenable C-semianalytic set, there exist C-analytic sets $Y_1,\ldots,Y_s$ and open C-semianalytic sets $V_1,\ldots,V_s$ such that $S=\bigcup_{i=1}^s(Y_i\cap V_i)$. Denote $Y=\bigcup_{i=1}^sY_i$, which has dimension $d$. Let $Y'$ be the union of the irreducible components of $Y$ of dimension $d$ and let $Y''$ the union of those of smaller dimension.
We prove first that  that the C-semianalytic set $Y'\setminus(\Sing(Y')\cup Y'')$ is a union of connected components of $ Y\setminus(\Sing(Y')\cup Y'')$.

 Now  $Y\setminus\Sing(Y)$ is a real analytic manifold of dimension $d$ and $Y'$ is a C-analytic subset of $Y$. Thus, $Y'\setminus\Sing(Y)$ is  open and closed  in  $Y\setminus \Sing(Y)$, so $Y'\setminus\Sing(Y)$ is a union of connected components of $Y \setminus\Sing(Y)$, as required. 
Let $C$ be the complement in $Y$ of $Y'\setminus(\Sing(Y')\cup Y'')$

It holds:
\begin{equation}\label{squnion}
S=\Int_{Y'}(S\setminus C)\cup(S\cap(\Sing(Y')\cup Y'')).
\end{equation}
We only need to prove the inclusion from left to right. Let $x\in S\setminus(\Sing(Y')\cup Y'')$. We have to check that $x$ belongs to the interior part in $Y'$ of $S\setminus C$.

Assume $x\in Y_1\cap V_1$. As $\dim(Y_1)=d$ and $\Sing(Y_1)\subset C$, the difference $Y_1\setminus C$ is a real analytic manifold of dimension $d$ contained in $M\setminus C$, so it is an open subset of $M$. As
$$
S\setminus(\Sing(Y')\cup Y'')\subset Y'\setminus(\Sing(Y')\cup Y'')=Y'\setminus C,
$$
we deduce $x\in (Y_1\setminus C)\cap V_1\subset\Int_{Y'}(S\setminus C)$.

Define  $U_1=M\setminus(\ol{M\setminus(S\setminus C)})$. It holds that $Y_1$ is a C-analytic  subset of $ M$ and $U_1$ is an open  C-semianalytic set by Proposition \ref{booletopological}. We have 
$$
Y_1\cap U_1=Y_1\setminus(\ol{M\setminus(S\setminus C)})=\Int_{Y_1}(S\setminus C),
$$
which is a real analytic manifold.

As $S\cap \ (\Sing(Y')\cup Y'')$ is an  amenable C-semianalytic set of dimension $<d$, there exist by induction hypothesis finitely many  C- analytic subsets $Y_2,\ldots,Y_r$ of $M$ and finitely many open C-semianalytic subsets $U_2,\ldots,U_r$ of $M$ such that $Y_i\cap U_i$ is a real analytic manifold, $\dim(Y_{i+1}\cap U_{i+1})<\dim(Y_i\cap U_i)$ for $i=2,\ldots,r-1$ and
$$
S\cap(\Sing(Y')\cup Y'')=\bigcup_{i=2}^r(Y_i\cap U_i).
$$ 
Observe that $\dim(Y_1\cap U_1)=d>\dim(S\cap(\Sing(Y')\cup Y''))=\dim(Y_2\cap U_2)$. By equation \eqref{squnion} $S=\bigcup_{i=1}^r(Y_i\cap U_i)$, as required.
\end{proof}

\begin{proof}[Proof of Proposition \rm{\ref{normal2}}]
We prove first the `only if' implication. 

By Lemma \ref{decomp} we may assume $S=Z\cap U$ where $Z$ is a C-analytic set, $U$ is an open C-semianalytic set,  $Z$ is the Zariski closure of $S$. Let $\{C_{\ell}\}_{\ell\geq1}$ be the collection of the connected components of $Z\cap U$. Notice that the Zariski closure $\ol{C_{\ell}}^{\zar}$ is an irreducible component of $Z$. Each C-semianalytic set $C_{\ell}$ is an open subset of $Z$, so by Lemma \ref{openX} there exists an open C-semianalytic set $V_{\ell}$ such that $C_\ell=Z\cap V_\ell$. By Lemma \ref{odes} we write $V_\ell$ as a  locally finite countable union $V_\ell=\bigcup_{j\geq1}V_{\ell j}$ of open basic C-semianalytic sets. Consequently $S_{\ell j}=Z\cap V_{\ell j}$ is either empty or a basic C-semianalytic set whose Zariski closure is an irreducible component of $Z$. The collection $\{S_{\ell,j}\}_{\ell,j\geq1}$ satisfies the required conditions. 

We prove now the `if' implication by induction on the dimension of $S$. If $\dim S =0$, then $S$ is a C-analytic set and in particular it is amenable. Assume that the result is true if $\dim S <d$ and let us check that this implies that it is also true for dimension $d$. 

Consider the C-analytic set $X=\bigcup_{i\geq1}\ol{S_i}^{\zar}$ and let $X'$ be the union of $\Sing(\ol{S_i}^{\zar})$ and the irreducible components of $\ol{S_i}^{\zar}$ of dimension $<d$ for $i\geq1$. It holds that $X'$ is a C-analytic set of dimension $<d$. Consequently, each intersection $S_i\cap X'$ is a basic C-semianalytic set of dimension $<d$. In addition, the countable family $\{S_i\cap X'\}_{i\geq1}$ is locally finite and $\{\ol{S_i\cap X'}^{\zar}\}_{i\geq1}$ is locally finite (after eliminating repetitions). By induction hypothesis $S\cap X'=\bigcup_{i\geq1}S_i\cap X'$ is an amenable C-semianalytic set. 

It only remains to check that $S\setminus X'$ is an amenable C-semianalytic set. Notice that $S_i\setminus X'$ is either empty or a basic C-semianalytic set of dimension $d$ and  $S_i\setminus X'$ is an open subset of the real analytic manifold $X\setminus X'$ for each $i\geq1$, so the C-semianalytic set $S\setminus X'=\bigcup_{i\geq1}S_i\setminus X'$ is an open subset of $X$. 

Consequently, $S\setminus X'$ is by Lemma \ref{openX} an amenable C-semianalytic set. Indeed, write $S_i=\ol{S_i}^{\zar}\cap V_i$ where $V_i$ is a basic open C-semianalytic set. As $\Sing(\ol{S_i}^{\zar})$ and the irreducible components of $\ol{S_i}^{\zar}$ of dimension strictly smaller than $d$ are contained in $X'$, the basic C-semianalytic set $S_i\setminus X'=(\ol{S_i}^{\zar}\setminus X')\cap V_i$ is an open subset of the real open analytic manifold $(\ol{S_i}^{\zar}\setminus X')$, which is in turn an open subset of the real analytic manifold $X\setminus X'$, as required. 
\end{proof}

Proposition \ref{normal2} implies  a sufficient condition for a locally finite countable union of amenable C-semianalytic to be amenable.

\begin{cor}\label{lfuok2}
Let $\{S_i\}_{i\geq1}$ be a locally finite collection of amenable C-semia\-na\-ly\-tic sets such that the family $\{\ol{S_i}^{\zar}\}_{i\geq1}$ is locally finite (after eliminating repetitions). Then $S=\bigcup_{i\geq1}S_i$ is an amenable C-semianalytic set.
\end{cor} 

\begin{proof}
When needed we always consider locally finite unions after eliminating repetitions.
For each $i\geq1$ there exists by Proposition \ref{normal2} a countable family $\{S_{ij}\}_{j\geq1}$ of basic C-semianalytic sets such that $S_i=\bigcup_{j\geq1}S_{ij}$ and the family $\{\ol{S_{ij}}^{\zar}\}_j$ is locally finite. Notice that $\{S_{ij}\}_{i,j\geq1}$ is a countable family of basic C-semianalytic sets. As $\ol{S_{ij}}^{\zar}\subset\ol{S_i}^{\zar}$ for each $i,j\geq1$ and the families $\{\ol{S_i}^{\zar}\}_{i\geq1}$ and $\{\ol{S_{ij}}^{\zar}\}_{j\geq1}$ are locally finite , we conclude that the family $\{\ol{S_{ij}}^{\zar}\}_{i,j\geq1}$ is also locally finite. By Proposition \ref{normal2} we conclude that $S=\bigcup_{i\geq1}S_i=\bigcup_{i,j\geq1}S_{ij}$ is amenable, as required.
\end{proof}

\subsection{Images of amenable C-semianalytic sets under proper holomorphic maps}\setcounter{paragraph}{0} 

Let $(X,\an_X), (Y, \Oo_Y)$ be reduced Stein spaces endowed with  anti-invo\-lu\-tion $\sigma, \tau$ such that their  fixed part $X^\sigma, Y^\tau$ are both not empty. Let ${\mathcal A}(X)\subset\an(X)$ be the subring of all invariant holomorphic sections of $\an(X)$ and 
$$
{\mathcal A}(X^\sigma)=\{F_{|X^\sigma}:\ F\in{\mathcal A}(X)\}\subset\an(X^\sigma). 
$$
In this section we prove the following result.

\begin{thm}\label{properint-tame}
Let $F:(X,\an_X)\to(Y,\an_Y)$ be an invariant proper holomorphic map, that is, $\tau\circ F=F\circ\sigma$. Let $S\subset X^\sigma$ be a ${\mathcal A}(X^\sigma)$-definable and amenable C-semianalytic set and let $S'\subset Y^\tau$ be an amenable C-semianalytic set. We have 
\begin{itemize}
\item[(i)] $F(S)$ is an amenable C-semianalytic subset of $Y^\tau$ of the same dimension as $S$.
\item[(ii)] If $T$ is a union of connected components of $F^{-1}(S')\cap X^\sigma$, then $F(T)$ is an amenable C-semianalytic set.
\end{itemize}
\end{thm}

\begin{lem}\label{zarx}
Let $Y$ be a C-analytic subset of $X^\sigma$ of the same dimension and let $Z$ be the Zariski closure of $Y$ in $X$. Then $\dim_\C Z=\dim_\R Y$, $\Sing(Y)\subset\Sing(Z)$ and $Y\setminus\Sing(Z)$ is a union of connected components of $(X^\sigma\cap Z) \setminus \Sing(Z)$.
\end{lem}
\begin{proof}
If $\dim_\R Y<\dim_\C Z$, it holds $Y\subset\Sing(Z)\subsetneq Z$, which is a contradiction. Thus, $d=\dim_\C Z=\dim_\R Y$. The inclusion $\Sing(Y)\subset\Sing(Z)$ is clear, so $Y\setminus\Sing(Z)$ is a real analytic manifold of dimension $d$. In addition, $X^\sigma\setminus Z$ is a real analytic set of dimension $d$ and $Y$ is a C-analytic subset of $X^\sigma\cap Z$. Thus, $Y\setminus\Sing(Z)$ is an open and closed subset of $(X^\sigma\cap  Z)\setminus \Sing(Z)$, so $Y\setminus\Sing(Z)$ is a union of connected components of $(X^\sigma\cap Z) \setminus\Sing(Z)$, as required. 
\end{proof}

Next result and its proof generalize Lemma \ref{decomp} to  ${\mathcal A}(X^\sigma)$-definable and amenable C-semianalytic sets.  

\begin{lem}\label{pd}
Let $S\subset X^\sigma$ be an ${\mathcal A}(X^\sigma)$-definable and amenable C-semianalytic set. Then there exist finitely many invariant complex analytic subsets $Z_1,\ldots,Z_r$ of $X$ and finitely many open ${\mathcal A}(X^\sigma)$-definable C-semianalytic subsets $U_1,\ldots,U_r$ of $X^\sigma$ such that $S=\bigcup_{i=1}^r(Z_i\cap U_i)$, each $Z_i\cap U_i$ is a real analytic manifold and $\dim(Z_{i+1}\cap U_{i+1})<\dim(Z_i\cap U_i)$ for $i=1,\ldots,r-1$.
\end{lem}
\begin{proof}
We proceed by induction on the dimension $d$ of $S$. If $S$ has dimension $0$, the result is clearly true. Assume the result true if the dimension of $S$ is less than $d$ and let us check that it is also true if $S$ has dimension $d$.
 
As $S$ is an amenable C-semianalytic set, there exist C-analytic sets $Y_1,\ldots,Y_s$ and open C-semianalytic sets $V_1,\ldots,V_s$ such that $S=\bigcup_{i=1}^s(Y_i\cap V_i)$. Denote $Y=\bigcup_{i=1}^sY_i$, which has dimension $d$. Let $Y'$ be the union of the irreducible components of $Y$ of dimension $d$ and let $Y''$ the union of those of smaller dimension. Let $Z'$ be the Zariski closure of $Y'$ in $X$ and $Z''$ the Zariski closure of $Y''$ in $X$. Then $Z=Z'\cup Z''$ is the Zariski closure of $Y$ in $X$.

By Lemma \ref{zarx} the C-semianalytic set $Y'\setminus(\Sing(Z')\cup Z'')$ is a union of connected components of $(X^\sigma\cap Z)\setminus((\Sing(Z')\cup Z''))$. By Proposition \ref{sigmainvariante} the set $Y'\setminus(\Sing(Z')\cup Z'')$ is a ${\mathcal A}(X^\sigma)$-definable C-semianalytic subset of $X^\sigma$. As $Y'\setminus(\Sing(Z')\cup Z'')$ is an open C-semianalytic subset of $X^\sigma\cap Z$,
$$
C=(X^\sigma\cap Z)\setminus(Y'\setminus(\Sing(Z')\cup Z''))
$$
is a closed ${\mathcal A}(X^\sigma)$-definable C-semianalytic subset of $X^\sigma$.

We claim:
\begin{equation}\label{squnion2}
S=\Int_{X^\sigma\cap Z'}(S\setminus C)\cup(S\cap(\Sing(Z')\cup Z'')).
\end{equation}
We only need to prove the inclusion from left to right. Let $x\in S\setminus(\Sing(Z')\cup Z'')$. We have to check that $x$ belongs to the interior part in $X^\sigma\cap Z'$ of $S\setminus C$.

Assume $x\in Y_1\cap V_1$. As $\dim(Y_1)=d$ and $\Sing(Y_1)\subset C$, the difference $Y_1\setminus C$ is a real analytic manifold of dimension $d$ contained in $(X^\sigma\cap Z')\setminus C$, so it is an open subset of $X^\sigma\cap Z'$. As
$$
S\setminus(\Sing(Z')\cup Z'')\subset Y'\setminus(\Sing(Z')\cup Z'')=(X^\sigma\cap Z')\setminus C,
$$
we deduce $x\in (Y_1\setminus C)\cap V_1\subset\Int_{X^\sigma\cap Z'}(S\setminus C)$.

Define $Z_1=Z'$ and $U_1=X^\sigma\setminus\ol{X^\sigma\setminus(S\setminus C)}$. It holds that $Z_1$ is an invariant complex analytic subset of $X$ and $U_1$ is an open ${\mathcal A}(X^\sigma)$-definable C-semianalytic set by Proposition \ref{sigmainvariante}. We have 
$$
Z_1\cap U_1=(X^\sigma\cap Z')\setminus\ol{X^\sigma\setminus(S\setminus C)}=\Int_{X^\sigma\cap Z'}(S\setminus C),
$$
which is a real analytic manifold.

As $S\cap \ (\Sing(Z')\cup Z'')$ is an ${\mathcal A}(X^\sigma)$-definable and amenable C-semianalytic set of dimension $<d$, there exist by induction hypothesis finitely many invariant complex analytic subsets $Z_2,\ldots,Z_r$ of $X$ and finitely many open ${\mathcal A}(X^\sigma)$-definable C-semianalytic subsets $U_2,\ldots,U_r$ of $X^\sigma$ such that $Z_i\cap U_i$ is a real analytic manifold, $\dim(Z_{i+1}\cap U_{i+1})<\dim(Z_i\cap U_i)$ for $i=2,\ldots,r-1$ and
$$
S\cap(\Sing(Z')\cup Z'')=\bigcup_{i=2}^r(Z_i\cap U_i).
$$ 
Observe that $\dim(Z_1\cap U_1)=d>\dim(S\cap(\Sing(Z')\cup Z''))=\dim(Z_2\cap U_2)$. By equation \eqref{squnion2} $S=\bigcup_{i=1}^r(Z_i\cap U_i)$, as required.
\end{proof}

\begin{lem}\label{strata}
Let $F:(X,\an_X)\to(Y,\an_Y)$ be a proper surjective map between reduced irreducible Stein spaces of the same dimension $d$. Then, there exist complex analytic subsets $X'\subset X$ and $Y'\subset Y$ of dimension $<d$ such that
\begin{itemize}
\item[(i)] $F^{-1}(Y')=X'$,
\item[(ii)]$M=X\setminus X'$ and $N=Y\setminus Y'$ are complex analytic manifolds respectively dense in $X$ and $Y$,
\item[(iii)] $F_{|M}:M\to N$ is an  proper surjective open holomorphic map of constant rank $d$.
\end{itemize}
\end{lem}
\begin{proof}
Recall that compact analytic subsets of a Stein space are finite sets, so the fibers of $F$ are finite. So, the set 
$$
X_0=\{z\in\Reg(X):\ {\rm rk}_z(F)\leq d-1\}\cup\Sing(X)
$$
is a complex analytic subset of $X$ of dimension $<d$. 

The singular set $\Sing(Y)$ has dimension $<d$. As $F$ is proper, $F(X_0)$ is by Grauert's Theorem a complex analytic subset of $Y$ of dimension $<d$. Let $Y'=\Sing(Y)\cup F(X_0)$ and $X'=F^{-1}(Y')$, which is a complex analytic subset of $X$ of dimension $<d$ because $F$ has finite fibers. Put $X_1= X'\cup X_0$ and let $M=X\setminus X_1$ and $N=Y\setminus Y_1$, where $Y_1 =Y' = F(X_1)$, which are complex analytic manifolds respectively dense in $X$ and $Y$. As $M=F^{-1}(N)$, the map $F_{|M}:M\to N$ is proper and surjective. In addition, $F_{|M}$ has constant rank $d$, so by the rank theorem $F_{|M}$ is open, as required.
\end{proof}
We recall a topological fact.

\begin{lem}\label{proylc}
Let $\pi:X\to Y$ be a proper map with finite fibers between two topological spaces. Let $\{A_i\}_{i\in I}$ be a locally finite family of $X$. Then $\{\pi(A_i)\}_{i\in I}$ is a locally finite family of $Y$.
\end{lem}
\begin{proof}
Let $y\in Y$ and write $\pi^{-1}(y)=\{x_1,\ldots,x_r\}$. For each $j=1,\ldots,r$ let $V_j$ be an open neighborhood of $x_j$ that only intersects finitely many $A_i$. Let $C=X\setminus\bigcup_{j=1}^rV_j$, which is a closed subset of $Y$. As $\pi$ is proper, $\pi(C)$ is closed. As $\pi^{-1}(y)\cap C=\varnothing$, we have $y\not\in\pi(C)$, so $U=Y\setminus\pi(C)$ is an open neighborhood of $y$. Let us check that \em if $\pi(A_i)\cap U\neq\varnothing$ then there exists $j=1,\ldots,r$ such that $A_i\cap V_j\neq\varnothing$\em. Thus $\{\pi(A_i)\}_{i\in I}$ is a locally finite family of $X$.

Suppose by contradiction that $A_i\cap V_j=\varnothing$ for all $j=1,\ldots,r$. Then $A_i\subset C$, so $\pi(A_i)\subset\pi(C)$ and $\pi(A_i)\cap U\neq\varnothing$, which is a contradiction, as required.
\end{proof}


\begin{proof}[Proof of Theorem \ref{properint-tame}]\hfill

(i) By Lemma \ref{pd} $S=\bigcup_{i=1}^rZ_i\cap V_i$ where 
\begin{itemize}
\item $Z_i$ is an invariant complex analytic subset of $X$ of (complex) dimension $d_i$ and $Z_i\cap X^\sigma$ has real dimension $d_i$,
\item $V_i$ is an open ${\mathcal A}(X^\sigma)$-definable C-semianalytic subset of $X^\sigma$.
\end{itemize}
As $F(S)=\bigcup_{i=1}^rF(Z_i\cap V_i)$, it is enough to prove that $F(Z_i\cap V_i)$ is an amenable C-semianalytic set, so we assume $S=Z_1\cap V_1$. To soften notation write $Z=Z_1$ and $V=V_1$. 

We proceed by induction on the dimension of $S$. If $S$ has dimension $0$, it is a discrete subset of $X^\sigma$ and as $F$ is proper also $F(S)$ is a discrete subset of $Y^\tau$, so it is amenable. Assume the result true if the dimension of $S$ is $<d$ and we check that this implies it is also true if $S$ has dimension $d$. 

Let $\{Z_\alpha\}_\alpha$ be the locally finite family of the irreducible components of $Z$. By Grauert's Theorem and Lemma \ref{proylc} $\{Y_\alpha=F(Z_\alpha)\}_\alpha$ is a locally finite family of irreducible complex analytic subsets of $Y$. As $F(S)=\bigcup_\alpha F(V\cap Z_\alpha)$ and the family $\{Y_\alpha\}_\alpha$ is locally finite, by Corollary \ref{lfuok2} it is enough to show  that $F(V\cap Z_\alpha)$ is an amenable C-semianalytic set. Thus, we assume $X,Y$ to be irreducible, they have complex dimension $d$, $F$ is surjective and $S=V$.

By Lemma \ref{strata} there exist complex analytic subsets $X'\subset X$ and $Y'\subset Y$ of dimension $<d$ such that
\begin{itemize}
\item $F^{-1}(Y')=X'$,
\item $M=X\setminus X'$ and $N=Y\setminus Y'$ are invariant complex analytic manifolds respectively dense in $X$ and $Y$,
\item $F_{|M}:M\to N$ is a proper open surjective holomorphic map of constant rank equal to $d$.
\end{itemize}
By induction hypothesis $F(V\cap X')$ is an amenable C-semianalytic set, so it is enough to prove that $F(V\cap M)$ is an amenable C-semianalytic set. By Theorem \ref{directimage} $F(V\cap M)$ is a C-semianalytic set. Denote $M^\sigma=M\cap X^\sigma$ and $N^\tau=N\cap Y^\tau$. As $F$ is invariant its restriction $f$ to $M^\sigma$ takes values in $N^\tau$.   As ${\rm rk}_z(F)=d$ for all $z\in M$, it holds ${\rm rk}_x(f)=d$ for all $x\in M^\sigma$. As $\dim_\R(M^\sigma)=\dim_\R(N^\tau)=d$, we deduce by the rank theorem that $f$ is open, so $f(V\cap M)$ is an open C-semianalytic subset of $Y^\tau$, as required.

(ii) After shrinking $Y$ if necessary, we write $S'=\bigcup_{i=1}^rZ_i\cap V_i$ where $V_i$ is an open C-semianalytic set and $Z_i$ is an invariant complex analytic subset of $Y$. As $Y$ is Stein and $F_{|F^{-1}(Y)}:F^{-1}(Y)\to Y$ is proper, also $F^{-1}(Y)$ is Stein. We substitute $X$ by $F^{-1}(Y)$.

Write $Z_i'=F^{-1}(Z_i)$ and $W_i=F^{-1}(V_i)\cap X^\sigma$. Observe that $Z_i'$ is an invariant complex analytic subset of $X$ and $W_i$ is an open C-semianalytic subset of $X^\sigma$. As $T$ is a union of connected components of $F^{-1}(S)=\bigcup_{i=1}^rZ_i'\cap W_i$, the intersection $T\cap Z_i'\cap W_i$ is a union of connected components of $Z_i'\cap W_i$. Thus, we may assume 
\begin{itemize}
\item $S'=Z\cap V$, where $Z$ is an invariant complex analytic set and $V$ is an open C-semianalytic subset of $Y^\tau$.
\item $T$ is a union of connected components of $F^{-1}(S)\cap X^\sigma=Z'\cap W$, where $Z'=F^{-1}(Z)$ is an invariant complex analytic set and $W=F^{-1}(V)\cap X^\sigma$ is an open C-semianalytic subset of $X^\sigma$. 
\end{itemize}
Let $\{Z'_\alpha\}_\alpha$ be the locally finite family of the irreducible components of $Z'$. Again, as before,   $\{Y_\alpha=F(Z_\alpha')\}_\alpha$ is a locally finite family of irreducible complex analytic subsets of $Y$. As $T$ is a union of connected components of $Z'\cap W$, it holds that $T\cap Z'_\alpha$ is a union of connected components of $Z'_\alpha\cap W$. As
$$
F(T)=\bigcup_\alpha F(T\cap Z'_\alpha)
$$
and the family $\{Y_\alpha\}_\alpha$ is locally finite, it is enough to show by Corollary \ref{lfuok2} that $F(T\cap Z'_\alpha)$ is an amenable C-semianalytic set. Thus, we may assume that $X,Y$ are irreducible, they have complex dimension $d$ and $F$ is surjective. We have to prove that if $V$ is an open C-semianalytic subset of $Y^\tau$ and $T$ is a union of connected components of $F^{-1}(V)\cap X^\sigma$, then $F(T)$ is an amenable C-semianalytic subset of $Y^\tau$. Denote $W=F^{-1}(V)$.

We proceed by induction on the dimension of $X$. If $X$ has dimension $0$, then $T$ is a discrete set and as $F$ is proper, $F(T)$ is also a discrete set, so it is amenable. Assume the result for dimension $<d$ and let us check that it is also true for dimension $d$. 

By Lemma \ref{strata} there exist complex analytic subsets $X'\subset X$ and $Y'\subset Y$ of dimension $<d$ such that
\begin{itemize}
\item $F^{-1}(Y')=X'$,
\item $M=X\setminus X'$ and $N=Y\setminus Y'$ are invariant complex analytic manifolds respectively dense in $X$ and $Y$,
\item $F_{|M}:M\to N$ is a proper open surjective holomorphic map of constant rank equal to $d$.
\end{itemize}

As $T\cap X'$ is a union of connected components of $W\cap X'$ and $\dim(X')<d$, by induction hypothesis $F(T\cap X')$ is an amenable C-semianalytic set, so it is enough to prove that $F(T\cap M)$ is an amenable C-semianalytic set. Denote $M^\sigma=M\cap X^\sigma$ and $N^\tau=N\cap Y^\tau$. As $F$ is invariant, $f=F_{|M^\sigma}:M^\sigma\to N^\tau$ and, as we have commented above, $f$ is open. As $T\cap M$ is an open subset of $M^\sigma$, we conclude that $F(T\cap M)=f(T\cap M)$ is an open subset of $Y^\tau$. It only remains to show that $F(T\cap M)$ is a C-semianalytic subset of $Y^\tau$.

By Lemma \ref{odes} $V=\bigcup_{j\geq1}V_j$ where $\{V_j\}_{j\geq1}$ is a locally finite family of open basic C-semianalytic set. Fix $j\geq1$ and observe that $T\cap F^{-1}(V_j\setminus Y')$ is a union of connected components of $F^{-1}(V_j\setminus Y')\cap X^\sigma$. Let $Y_j\subset Y$ be an invariant Stein open neighborhood of $Y^\tau$ such that $V_j$ is ${\mathcal A}(Y^\tau_j)$-definable. Then, $F^{-1}(V_j\setminus Y')\cap X^\sigma$ is ${\mathcal A}(X^\sigma_j)$-definable, where $X_j=F^{-1}(Y_j)$  In addition the map $F_j=F_{|X_j}:X_j\to Y_j$ is proper and surjective. It is a classical result that $X_j$ is a Stein space. As $T\cap F_j^{-1}(V_j\setminus Y')$ is a union of connected components of $F^{-1}(V_j\setminus Y')\cap X^\sigma$, we deduce by Proposition \ref{sigmainvariante} that $T\cap F_j^{-1}(V_j\setminus Y')$ is ${\mathcal A}(X^\sigma_j)$-definable. By Theorem \ref{directimage} 
$$
F_j(T\cap F_j^{-1}(V_j\setminus Y'))=F_j(T\cap M\cap F_j^{-1}(V_j))=F_j(T\cap M)\cap V_j
$$ 
is a C-semianalytic subset of $Y^\tau$. Thus, the locally finite union of C-semianalytic subsets of $Y^\tau$
$$
F(T\cap M)=\bigcup_{j\geq1}F_j(T\cap M)\cap V_j
$$
is a C-semianalytic subset of $Y^\tau$, as required.
\end{proof}

The following example shows that amenability is not preserved by proper analytic maps with finite fibers.

\begin{example}
Let $M=\bigcup_{k\geq1}\sph_k$ where $\sph_k=\{(x,y,z)\in\R^3:\ \left(x-k+\frac{1}{2}\right)^2+y^2+(z-2k)^2=4\}$, which is a locally finite union of pairwise disjoint spheres. Let 
$$
S_k=\{(x,y,z)\in\sph_k:\ k-1\leq x\leq k,\ 0\leq y\leq\tfrac{1}{k}\}.
$$
It holds that $S=\bigcup_{k\geq1}S_k$ is an amenable C-semianalytic set. Let $\pi:\R^3\to\R^2,\ (x,y,z)\mapsto(x,y)$ be the projection onto the first two variables. Observe that $\rho=\pi_{|M}:M\to\R^2$ is a proper map with finite fibers. However 
$$
\rho(S)=\bigcup_{k\geq1}\{(x,y)\in\R^2:\ k-1\leq x\leq k,\ 0\leq y\leq\tfrac{1}{k}\}
$$
is not amenable as we have seen in Example \ref{ex}.
\end{example}

\subsection{Tameness-algorithm for C-semianalytic sets.}\label{algorithmtame}

In this  section we give  an algorithm to determine whether a C-semianalytic set is amenable. 

The announced algorithm is based on the fact that while the behavior of the Zariski closure of general C-semianalytic sets can be wild, as we saw in several examples as Example \ref{segmenti} or Example (i) from Examples \ref{tame}, the Zariski closure of an amenable C-semianalytic set behaves neatly.

\begin{lem}\label{gbzc}
Let $S\subset M$ be an amenable C-semianalytic set and for each $x\in M$ let $U^x\subset M$ be any open C-semianalytic neighborhood. Then 
$$
\ol{S}^{\zar}=\bigcup_{x\in\ol{S}}\ol{S\cap U^x}^{\zar}.
$$
\end{lem}
\begin{proof}
By Lemma \ref{decomp} we may assume $S=X\cap W$ is a real analytic manifold where $X$ is the Zariski closure of $S$ and $W$ is an open C-semianalytic set. Let $\{S_i\}_{i\geq1}$ be the family of the connected components of $S$. Then
\begin{itemize}
\item $S_i$ is a pure dimensional connected amenable C-semianalytic set.
\item Each $\ol{S}_i^{\zar}$ is an irreducible component of $X$.
\item $S=\bigcup_{i\geq1}S_i$ and $X=\bigcup_{i\geq1}\ol{S_i}^{\zar}$.
\end{itemize}
Let $x\in\ol{S}$ and assume that $x\in\ol{S_i}$ only for $i=1,\ldots,r$. Let $C=M\setminus\bigcup_{i\leq r+1}\ol{S_i}$ and let $V=U^x\setminus C$, which is an open C-semianalytic neighborhood of $x$. Notice that $\dim S_x=\dim S\cap V=\dim \ol{S\cap V}^{\zar}$ and
$$
\bigcup_{i=1}^r\ol{S}_i^{\zar}=\ol{S\cap V}^{\zar}\subset\ol{S\cap U^x}^{\zar}\subset X
$$
As $\ol{S}=\bigcup_{i\geq1}\ol{S_i}$, it holds $X=\bigcup_{i\geq1}\ol{S}_i^{\zar}=\bigcup_{x\in\ol S}\ol{S\cap U^x}^{\zar}$, as required.
\end{proof}

Let $S\subset M$ be a C-semianalytic set. Let $X_1$ be the union of the irreducible components of $\ol{S}^{\zar}$ of maximal dimension. Define 
$$
T_1(S)=\Int_{\Reg(X_1)}(S\cap\Reg(X_1)). 
$$
Assume we have already defined $T_1(S),\ldots,T_k(S)$ and let us define $T_{k+1}(S)$. Let $R_{k+1}=S\setminus\bigcup_{j=1}^kT_j(S)$ and let $X_{k+1}$ be the union of the irreducible components of $\ol{R_{k+1}}^{\zar}$ of maximal dimension. Consider 
$$
T_{k+1}(S)=\Int_{\Reg(X_{k+1})}(S\cap\Reg(X_{k+1})).
$$
We describe first some properties of the operators $T_k$.

\begin{prop}\label{operators}
Let $S\subset M$ be a C-semianalytic set.
\begin{itemize}
\item[(i)]  $T_k(S)$ is an amenable C-semianalytic set.
\item[(ii)]  If $T_k(S)=\varnothing$, then $R_{k+\ell}=R_k$ for all $\ell\geq1$ and $T_{k+\ell}(S)=T_k(S)=\varnothing$ for all $\ell\geq1$.
\item[(iii)]  If $T_{k+1}(S)\subset T_k(S)$, then $R_{k+\ell}=R_{k+1}$ for all $\ell\geq1$ and $T_{k+\ell}(S)=T_{k+1}(S)$ for all $\ell\geq1$.
\item[(iv)]  If $T_{k+1}(S)\neq\varnothing$, then $\dim T_{k+1}(S)=\dim X_{k+1}\leq\dim X_k=\dim T_k(S)$  because $R_{k+1}\subset R_k$.
\item[(v)]  If $\dim X_{k+1}=\dim X_k$, then $\varnothing\neq T_{k+1}(S)\setminus\Sing(X_k)\subset T_k(S)$ and $T_{k+2}(S)=T_{k+1}(S)$.
\item[(vi)]  There exists $k_0\geq1$ such that $T_{k_0}(S)=T_{k_0+\ell}(S)$ for all $\ell\geq1$.
\end{itemize}
\end{prop}
\begin{proof}
Assertion (i) is a consequence of Lemma \ref{openX}. Assertions (ii), (iii), 
(iv) are straightforwardly checked.

(v) As $\dim X_{k+1}=\dim X_k$, then $\dim \ol{R_{k+1}}^{\zar}=\dim \ol{R_k}^{\zar}$. As $R_{k+1}\subset R_k$, it holds that $\ol{R_{k+1}}^{\zar}\subset\ol{R_k}^{\zar}$, so $X_{k+1}\subset X_k$ and $\Reg(X_{k+1})\setminus\Sing(X_k)$ is a non-empty open subset of $\Reg(X_k)$ because $\dim \Sing(X_k)<\dim X_{k+1}$. Thus $T_{k+1}(S)\setminus\Sing(X_k)\subset T_k(S)$.

We claim that $R_{k+1}\setminus\Sing(X_k)\subset R_{k+2}$. 

As $T_{k+1}(S)\setminus\Sing(X_k)\subset T_k(S)$, we have $R_{k+1}\setminus((T_{k+1}(S)\setminus\Sing(X_k))=R_{k+1}$ and
\begin{equation*}
\begin{split}
R_{k+2}=R_{k+1}\setminus T_{k+1}(S)&=R_{k+1}\setminus((T_{k+1}(S)\setminus\Sing(X_k))\cup(T_{k+1}(S)\cap\Sing(X_k)))=\\
&=(R_{k+1}\setminus(T_{k+1}(S)\setminus\Sing(X_k)))\setminus(T_{k+1}(S)\cap\Sing(X_k))=\\
&=R_{k+1}\setminus(T_{k+1}(S)\cap\Sing(X_k))\supset R_{k+1}\setminus\Sing(X_k).
\end{split}
\end{equation*}

As $\dim \Sing(X_k)<\dim X_{k+1}$ and $R_{k+1}\setminus\Sing(X_k)\subset R_{k+2}\subset R_{k+1}$, we have 
$$
X_{k+1}\subset\ol{R_{k+1}\setminus\Sing(X_k)}^{\zar}\subset\ol{R_{k+2}}^{\zar}\subset\ol{R_{k+1}}^{\zar}.
$$
Thus, $X_{k+1}=X_{k+2}$, so $T_{k+2}(S)=T_{k+1}(S)$.

(vi) Let $k\geq1$. If $T_k(S)=\varnothing$, then by (ii) $T_k(S)=T_{k+\ell}(S)$ for all $\ell\geq1$. Thus, we assume $T_k(S)\neq\varnothing$. We know by (iv) that $\dim X_{k+1}\leq\dim X_k$. If $\dim X_{k+1}=\dim X_k$, then by (v) $T_{k+2}(S)=T_{k+1}(S)$, so by (iii) $T_{k+\ell}(S)=T_{k+1}(S)$ for all $\ell\geq1$. The case missing is $\dim X_{k+1}<\dim X_k$. Repeat the previous argument with the index $k+1$ replacing  $k$ and observe that in finitely many steps we achieve the statement.
\end{proof}

The announced algorithm is a consequence of the definited operators $\{T_j\}$ and of  the following theorem.

\begin{thm}\label{algor1}
Let $S\subset M$ be a C-semianalytic set. The following assertions are equivalent
\begin{itemize}
\item[(i)] $S$ is amenable.
\item[(ii)] $S=\bigcup_{j\geq1}T_j(S)$.
\end{itemize}
\end{thm}

\begin{proof}
(ii) $\Longrightarrow$ (i) follows directely by the properties of $T_k$.
   
 (i) $\Longrightarrow$ (ii). We claim
\begin{quotation}{\em For each $k\geq1$ it holds $S=\bigcup_{j=1}^kT_j(S)\cup S'_{k+1}$ where $S'_{k+1}$ is either empty or an amenable C-semianalytic set such that $\ol{S'_{k+1}}^{\zar}=\ol{R_{k+1}}^{\zar}$ and $\dim S'_{k+1}<\dim T_k(S)$.}
\end{quotation}

Let $k\geq 1$ and assume $S=L_k\cup S'_k$ where 
$$
L_k=\begin{cases}
\varnothing&\text{if $k=1$}\\
\bigcup_{j=1}^{k-1}T_j(S)&\text{if $k\geq2$}
\end{cases}
$$
$R_1=S'_1=S$ and $S'_k$ is either empty or an amenable C-semianalytic set such that $\ol{S'_k}^{\zar}=\ol{R_k}^{\zar}$ and $\dim S_k'<\dim T_{k-1}(S)$ if $k\geq2$. Let us prove that $S=L_{k+1}\cup S'_{k+1}$ where $L_{k+1}$ and $S'_{k+1}$ satisfy the corresponding properties.

If $S_k'=\varnothing$, then $T_k(S)=\varnothing$ and it is enough to take $S_{k+1}'=\varnothing$. Assume $S_k'\neq\varnothing$, so it is an amenable C-semianalytic set such that $\ol{S'_k}^{\zar}=\ol{R_k}^{\zar}$ and $\dim S_k'<\dim T_{k-1}(S)$.

Notice that $S_k'\cap\Reg(X_k)\subset S\cap\Reg(X_k)$ has not empty interior in $\Reg(X_k)$, so $T_k(S)\neq\varnothing$. Write $S_k'=\bigcup_{i=1}^r(Z_i\cap U_i)$ where $Z_i$ is a C-analytic set, $U_i$ is an open C-semianalytic set, each $Z_i\cap U_i$ is a real analytic manifold, $Z_i$ is the Zariski closure of $Z_i\cap U_i$ and $\dim Z_{i+1}\cap U_{i+1}<\dim Z_i\cap U_i$ for $i=1,\ldots,r-1$ (see Lemma \ref{decomp}).
Let $S_{k+1}'=S_k'\cap\bigcup_{i=1}^r\ol{Z_i\cap U_i\cap R_{k+1}}^{\zar}$, which is amenable because it is the intersection of two amenable C-semianalytic sets. In addition $S_k'\cap R_{k+1}\subset S_{k+1}'$, so
$$
S=\bigcup_{j=1}^{k-1}T_j(S)\cup S_k'=\bigcup_{j=1}^kT_j(S)\cup (S_k'\cap R_{k+1})=\bigcup_{j=1}^{k}T_j(S)\cup S_{k+1}'=L_{k+1}\cup S_{k+1}'. 
$$
We have to prove that $\ol{S_{k+1}'}^{\zar}=\ol{R_{k+1}}^{\zar}$. 

For each $i=1,\ldots,r$ we have $\ol{Z_i\cap U_i\cap R_{k+1}}^{\zar}\subset\ol{R_{k+1}}^{\zar}$. Thus $S_{k+1}'\subset\ol{R_{k+1}}^{\zar}$, so $\ol{S_{k+1}'}^{\zar}\subset\ol{R_{k+1}}^{\zar}$. The converse inclusion follows because $R_{k+1}=S\setminus L_{k+1}\subset S_{k+1}'$.

We claim
\begin{quotation}  $\dim\ol{Z_i\cap U_i\cap R_{k+1}}^{\zar}<\dim T_k(S)$ for $i=1,\ldots,r$.\end{quotation} 

For $i=2,\ldots,r$ the result is clear, so let us prove $\dim\ol{Z_1\cap U_1\cap R_{k+1}}^{\zar}<\dim T_k(S)$.

As $\dim Z_1\cap U_1=\dim S_k'=\dim T_k(S)$, we have $(Z_1\cap U_1)\setminus\Sing(X_k)$ is a real analytic manifold of dimension $\dim X_k $, so it is an open subset of $\Reg(X_k)$ and $(Z_1\cap U_1)\setminus\Sing(X_k)\subset T_k(S)$. Thus, $(Z_1\cap U_1)\setminus T_k(S)\subset\Sing(X_k)$, so 
$$
\dim \ol{(Z_1\cap U_1)\setminus T_k(S)}^{\zar}\leq\dim \Sing(X_k)<\dim X_k=\dim S=\dim T_k(S).
$$
We conclude $\dim(S_{k+1}')<\dim(T_k(S))$ and the last claim is proved.

To finish we prove that there exists an index $k\geq1$ such that $S_k'=\varnothing$. For each $k\geq1$, it holds that either $S_k'=\varnothing$ or
$$
\dim X_k=\dim \ol{R_k^{\zar}}=\dim \ol{S_k'}^{\zar}=\dim S_k'<\dim T_{k-1}(S)=\dim X_{k-1}.
$$
As $\dim X_1=\dim S<+\infty$, there exist only finitely many possible $k$'s such that $S_k'\neq\varnothing$. Consequently, there exists $k\geq1$ such that $S=\bigcup_{j=1}^kT_j(S)=\bigcup_{j\geq1}T_j(S)$, as required.

\end{proof}

\begin{example}\label{dimktamenot}
Consider the open C-semianalytic set
$$
S_0=\bigcup_{k\geq1}\{0<x<k, 0<y<1/k\}\subset\R^3
$$
and let $S_1=\{z=0\}$. Define $S=S_0\cup S_1$ which is amenable because it is the union of two amenable C-semianalytic sets. Let $S_{(3)}$ be the set of points of $S$ of local dimension $3$ and let us check that it is not amenable. Indeed, observe first that $S_{(3)}=S_0\cup\ol{S_0\cap\{z=0\}}$. Assume by way of contradiction that $S_{(3)}$ is amenable. Then also $S'=S_{(3)}\cap\{z=0\}=\ol{S_0\cap\{z=0\}}$ should be amenable. To prove that $S'$ is not amenable we apply Lemma \ref{algor1}. Observe that $X_1=\{z=0\}$, $T_1(S')=S_0\cap\{z=0\}$, $R_2=\overline{(S_0\cap\{z=0\})}\setminus(S_0\cap\{z=0\})$, $X_2=\{z=0\}$, $T_2(S')=S_0\cap\{z=0\}$ and $S'\neq T_1(S')$. By Theorem \ref{algor1} $S'$ is not amenable, so $S_{(3)}$ is not amenable.
\end{example}

\
\subsection{Normalization and irreducibility.}
The irreducibility of an amenable C-semianalytic set $S$ has a close relation with the connectedness of a certain subset of the normalization of the Zariski closure of $S$. Let $S\subset M$ be an amenable C-semianalytic set and let $X$ be its Zariski closure. Let $(\widetilde{X},\sigma)$ be a Stein complexification of $X$ together with the anti-involution $\sigma:\widetilde{X}\to\widetilde{X}$ whose set of fixed points is $X$. Let $(Y,\pi)$ be the normalization of $\widetilde{X}$ and let $\widehat{\sigma}:Y\to Y$ be the anti-holomorphic involution induced by $\sigma$ in $Y$, which satisifies $\pi\circ\widehat{\sigma}=\sigma\circ\pi$.

More precisely we have.
\begin{thm}\label{dpm}
The amenable C-semianalytic set $S$ is irreducible if and only if there exists a connected component $T$ of $\pi^{-1}(S)$ such that $\pi(T)=S$.
\end{thm}
\begin{proof}\setcounter{paragraph}{0} 
It is enough to prove: \em For each open neighborhood $U\subset M$ of $S$ the Zariski closure $X=\ol{S}^{\zar}_U$ is irreducible if and only there exists a connected component $T$ of $\pi^{-1}(S)$ such that $\pi(T)=S$\em.

Suppose first $\pi(T)=S$ for some connected component $T$ of $\pi^{-1}(S)$. Assume $M$ is imbedded in $\R^n$ as a closed real analytic submanifold and $\widetilde{X}$ is a complex analytic subset of $\C^n$. Note that $\dim_\C \widetilde{X}=\dim_\R X=d$. Fix an open neighborhood $U\subset M$ of $S$ and let $\Omega\subset\C^n$ be an invariant open neighborhood of $S$ such that $\Omega\cap M=U$. As we will see in  Remark \ref{norm1} $(\pi^{-1}(\widetilde{X}\cap\Omega),\pi_|)$ is the normalization of $\widetilde{X}\cap\Omega$. As $T$ is connected and $\dim_\R T=d$, it is contained in a connected component $Z$ of $\pi^{-1}(\widetilde{X}\cap\Omega)$ of dimension $d$. Moreover $\pi(Z)$ is an irreducible component of $\widetilde{X}\cap\Omega$ of dimension $d$. As $S=\pi(T)\subset\pi(Z)\subset \widetilde{X}$, we have $\ol{S}^{\zar}_U\subset\pi(Z)\cap U$. As $\pi(Z)$ is irreducible and has dimension $d$, it is the complex Zariski closure of $\ol{S}^{\zar}_U$ in $\Omega$. As this holds for all invariant open neighborhood $\Omega\subset\C^n$ of $\ol{S}^{\zar}_U$, we deduce $\ol{S}^{\zar}_U$ is irreducible.
 
Suppose next that $\ol{S}^{\zar}_U$ is irreducible for each open neighborhood $U\subset M$ of $S$. We will construct a suitable open neighborhood $U\subset M$ of $S$ in $M$ to prove the existence of a connected component $T$ of $\pi^{-1}(S)$ such that $\pi(T)=S$. 

Let $\{T_i\}_{i\geq1}$ be the connected components of $\pi^{-1}(S)$. As $S$ is invariant, $\pi^{-1}(S)$ is invariant. Let $\{\Theta_i'\}_{i\geq1}$ be pairwise disjoint open subsets of $Y$ such that $T_i\subset\Theta_i'$ for $i\geq1$. Denote $\Theta'=\bigcup_{i\geq1}\Theta_i'$ and $\Theta=\Theta'\cap\widehat{\sigma}(\Theta')$. Notice that $\Theta$ is an invariant neighborhood of $\pi^{-1}(S)$ in $Y$ and $\Theta_i=\Theta_i'\cap\Theta\subset\widetilde{X}$ is an open neighborhood of $T_i$. Clearly, $\Theta_i\cap\Theta_j=\varnothing$ if $i\neq j$.

Define $C=Y\setminus\Theta$ which is a closed invariant subset of $Y$ that does not intersect $\pi^{-1}(S)$. As $\pi$ is proper and invariant, $\pi(C)$ is an invariant closed subset of $\widetilde{X}$. It holds $S\cap\pi(C)=\varnothing$, so $\pi^{-1}(S)\cap\pi^{-1}(\pi(C))=\varnothing$. Substituting C by the invariant closed set $\pi^{-1}(\pi(C))$, we may assume that $C=\pi^{-1}(\pi(C))$, so the restriction map $\pi_{|Y\setminus C}:Y\setminus C\to \widetilde{X}\setminus\pi(C)$ is proper and surjective. Define $\Omega=\C^n\setminus\pi(C)$ and $U=\Omega\cap M$, which is an open neighborhood of $S$ in $M$.

As $\ol{S}^{\zar}_U$ is irreducible, the complex Zariski closure $Z$ of $\ol{S}^{\zar}_U$ in $\Omega$ is irreducible, it is contained in $\widetilde{X}$ and its dimension equals $\dim_\R S=\dim_\R X=\dim_\C\widetilde{X}$. Thus, it is an irreducible component of $\widetilde{X}'=\widetilde{X}\cap\Omega$. In addition $(Y'=\pi^{-1}(\widetilde{X}'),\pi_{|Y'})$ is by Remark \ref{norm1} the normalization of $\widetilde{X}'$. Thus, there exists a connected component $K$ of $Y'$ such that $Z=\pi(K)$.

As $K\subset\bigcup_{i\geq1}\Theta_i$ is connected, $K\subset\Theta_{i_0}$ for some $i_0\geq1$. As $T_i\cap K\subset T_i\cap\Theta_{i_0}=\varnothing$ if $i\neq i_0$,
$$
\pi(T_{i_0})=\pi\Big(K\cap\bigcup_{i\geq1}T_i\Big)=\pi(K\cap\pi^{-1}(S))=Z\cap S=S,
$$
as required. 
\end{proof}

\noindent\begin{remarks}
 Let $Y^{\widehat{\sigma}}$ be the set of fixed points of $\widehat{\sigma}$.
\begin{itemize}
\item[(i)]  Recall that if $X$ is coherent by   Proposition \ref{cohe} of Chapter 2 we get $\pi^{-1}(X)=Y^{\widehat{\sigma}}$. If such is the case, the connected component $T$ in the statement of Theorem \ref{dpm} is contained in $Y^{\widehat{\sigma}}$.
\item[(ii)] If $X$ is not coherent but $\pi(\pi^{-1}(S)\cap Y^{\widehat{\sigma}})=S$, we cannot assure that the connected component $T$ of $\pi^{-1}(S)$ such that $\pi(T)=S$ satisfies in addition $\pi(T\cap Y^{\widehat{\sigma}})=S$.

Let $X$ be the irreducible C-analytic set of equation $x^4-z^2(4-z^2)y^4=0$. Consider the C-semianalytic set 
$$
S=(X\cap\{0<z<1,y\neq0\})\cup\{x=0,y=0,-1<z<1\}.
$$


The complex analytic set $\widetilde{X}=\{x^4-z^2(4-z^2)y^4=0\}\subset\C^3$ is a complexification of $X$. Its normalization is the non-singular complex analytic set $Y=\{u^2-vz=0,v^2+z^2-4=0\}\subset\C^4$ together with the holomorphic map $\pi:Y\to\widetilde{X},\ (y,z,u,v)\mapsto(uy,y,z)$. We have $\pi^{-1}(S)=T_1\cup T_2\cup T_3\cup T_4$ where
\begin{align*}
T_1&=\Big\{(y,z,\pm\sqrt{z\sqrt{4-z^2}},\sqrt{4-z^2}), 0\leq z<1\Big\}\\
T_2&=\Big\{(0,z,\pm\sqrt{-z\sqrt{4-z^2}}i,\sqrt{4-z^2}),-1<z\leq0\Big\}\\
T_3&=\Big\{(0,z,\pm\sqrt{z\sqrt{4-z^2}}i,-\sqrt{4-z^2}),0\leq z<1\Big\}\\
T_4&=\Big\{(0,z,\pm\sqrt{-z\sqrt{4-z^2}},-\sqrt{4-z^2}),-1<z\leq0\Big\}
\end{align*}
The connected components of $\pi^{-1}(S)$ are $T=T_1\cup T_2$ and $T'=T_3\cup T_4$. It holds $\pi(T)=S$. In addition $\pi^{-1}(S)\cap Y^{\widehat{\sigma}}=T_1\cup T_4$ and it satisfies $\pi(T_1\cup T_4)=S$. However, $T_1\cup T_4$ is not connected and $\pi(T\cap Y^{\widehat{\sigma}})=\pi(T_1)\subsetneq S$.
\end{itemize}
\end{remarks}

\subsection{Irreducible components of an amenable C-semianalytic set.}\label{s5b}

We begin by proving in Theorem \ref{irredcomp2}
existence and uniqueness of the family of weak irreducible components of an amenable C-semianalytic set. 
We recall the definition. 

\begin{defn}[Weak irreducible components]\label{irredcomptame}
Let $S\subset M$ be an amenable C-semi\-analytic set. A countable locally finite family $\{S_i\}_{i\geq1}$ in $S$ of $S$-tame C-semi\-analy\-tic sets is \em a family of weak irreducible components of \em $S$ if the following conditions are fulfilled 
\begin{itemize}
\item[(1)] Each $S_i$ is irreducible.
\item[(2)] If $S_i\subset T\subset S$ and $T$ is an irreducible $S$-tame C-semianalytic set, then $S_i=T$.
\item[(3)] $S_i\neq S_j$ if $i\neq j$.
\item[(4)] $S=\bigcup_{i\geq1} S_i$.
\end{itemize} 
\end{defn}

\begin{thm}\label{irredcomp2}
Let $S \subset\R^n$ be an amenable C-semianalytic set. Then there exists the family of weak irreducible components $\{S_i\}_{i\geq1}$ of $S$ and it is unique. In addition it satisfies 
\begin{itemize}
\item[(i)] $S_i=\ceros(\ideal(S_i,S))$ for $i\geq1$. In particular, $S_i$ is a closed subset of $S$.
\item[(ii)] The ideals $\ideal(S_i,S)$ are the minimal prime (saturated) ideals of $\an(S)$.
\end{itemize}
\end{thm}

Before proving Theorem \ref{irredcomp2} we introduce some auxiliary results.

\begin{lem}\label{sheaf}
Let $S\subset M$ be a semianalytic set and let $\gta$ be an ideal of $\an(S)$. Then $\ceros(\gta)$ is the intersection of $S$ with a C-analytic subset $X$ of an open neighborhood $U\subset M$ of $S$. 
\end{lem}
\begin{proof}
We have defined $\an(S)$  as the quotient $H^0(S,(\an_M)_{|S})/\ideal(S)$. Let $\gtb\supset\ideal(S)$ be the ideal of $H^0(S,(\an_M)_{|S})$ such that $\gta=\gtb/\ideal(S)$. The sheaf of ideals $\gtb(\an_M)_{|S}$ is $(\an_M)_{|S}$-coherent.  \em There exists an open neighborhood $U\subset M$ of $S$ and an analytic $\an_U$-coherent sheaf ${\mathcal F}$ such that $\gtb(\an_M)_{|S}={\mathcal F}_{|S}$.\em

Let $X$ be the zero set of the $\an_U$-coherent sheaf ${\mathcal F}$, which is a C-analytic subset of $U$. As $\gtb(\an_M)_{|S}={\mathcal F}_{|S}$, we deduce $\ceros(\gta)=\ceros(\gtb)=S\cap X$, as required.
\end{proof}

\begin{lem}\label{clue2}
Let $S\subset M$ be a semianalytic set and let $\gta$ be an ideal of $\an(S)$. Denote $T=\ceros(\gta)$. Then for each $f\in\an(T)$ there exists $g\in\an(S)$ such that $\ceros(f)=\ceros(g)$.
\end{lem}
\begin{proof}
By Lemma \ref{sheaf} there exists an open neighborhood $U_0\subset M$ of $S$ and a C-analytic set $X\subset U_0$ such that $T=S\cap X$. Let $V\subset U_0$ be an open neighborhood of $T$ and ${f'}\in\an(V)$ an analytic function such that ${f'}_{|T}=f$. 

We claim: $S\subset U_1$ where $U_1=(U\setminus X)\cup V$. 
 Indeed, $S\cap U_1=(S\setminus X)\cup(V\cap S)\supset(S\setminus T)\cup T=S$, so $S\subset U_1$. 

It holds: \em  $X'=X\cap U_1=X\cap V$ is a C-analytic subset of $U_1$. \em

Indeed let $g$ be an analytic equation of $X'$ in $V$ and consider the $\an_{U_1}$-coherent sheaf of ideals
$$
{\mathcal F}_x=\begin{cases}
g\an_{U_1,x}&\text{ if $x\in V$},\\
\an_{U_1,x}&\text{ otherwise}.
\end{cases}
$$
As $X'$ is the zero set of ${\mathcal F}$, we conclude that $X'$ is a C-analytic subset of $U_1$.

In addition $S\cap X'=S\cap X\cap V=T$. Let $h\in\an(U_1)$ be an analytic equation of $X'$ in $U_1$ and consider the $\an_{U_1}$-coherent sheaf of ideals
$$
{\mathcal F}_x=\begin{cases}
(h^2+{f'}^2)\an_{U_1,x}&\text{ if $x\in X'$},\\
\an_{U_1,x}&\text{ otherwise}.
\end{cases}
$$
Its zero set is a C-analytic subset of $U_1$, so there exists $g'\in\an(U_1)$ such that its zero set coincides with the zero set $\ceros(h,{f'})=X'\cap\ceros({f'})$ of ${\mathcal F}$. Thus, if $g=g'_{|S}$, we have
$$
\ceros(g)=S\cap\ceros(g')=S\cap X'\cap\ceros({f'})=T\cap\ceros({f'})=\ceros(f),
$$
as required.
\end{proof}
\begin{lem}\label{clue3}
Let $S\subset M$ be an amenable C-semianalytic set and let $\gta$ be an ideal of $\an(S)$. Then 
\begin{itemize}
\item[(i)] $\ceros(\gta)$ is an $S$-tame C-semianalytic set.
\item[(ii)] $\ceros(\gta)$ is irreducible if and only if $\ideal(\ceros(\gta),S)$ is a prime ideal of $\an(S)$. 
\end{itemize}
\end{lem}
\begin{proof}
(i) This statement follows from Lemma \ref{sheaf}.

(ii) Recall that if $T=\ceros(\gta)$ is irreducible, then $\ideal(T,S)$ is prime by definition. Conversely, assume that $\ideal(T,S)$ is prime and let $f_1,f_2\in\an(T)$ be such that $f_1f_2=0$. By \ref{clue2} there exist analytic functions $g_1,g_2\in\an(S)$ such that $\ceros(g_i)=\ceros(f_i)$ for $i=1,2$. Thus, $\ceros(g_1g_2)=\ceros(f_1f_2)=T$, so $g_1g_2\in \ideal(T,S)$. As $\ideal(T,S)$ is a prime ideal, we assume $g_1\in\ideal(T,S)$. Thus, $\ceros(f_1)=\ceros(g_1)=T$, so $f_1=0$. Consequently, $\an(T)$ is an integral domain and $T$ is irreducible.
\end{proof}
\begin{lem}\label{ni}
Let $S\subset T\subset E\subset M$ be $E$-tame C-semianalytic sets such that $S$ is irreducible and let $\{T_i\}_{i\geq1}$ be a family of $E$-tame C-semianalytic sets. Assume 
\begin{itemize}
\item[(i)] $T_i=\ceros(\ideal(T_i,E))$ for $i\geq1$,
\item[(ii)] $T=\bigcup_{i\geq1}T_i$,
\item[(iii)] The family $\{T_i\}_{i\geq1}$ is locally finite in $E$,
\end{itemize}
Then there exists $i\geq1$ such that $S\subset T_i$.
\end{lem}
\begin{proof}
As $S$ is irreducible, $\ideal(S,E)$ is a saturated prime ideal of $\an(E)$. The family of saturated ideals $\{\ideal(T_i,E)\}_{i\geq1}$ is locally finite and $\bigcap_{i\geq1}\ideal(T_i,E)=\ideal(T,E)\subset \ideal(S,E)$. By Theorem \ref{deb} and its Corollary in Chapter 3  $\ideal(T_i,E)\subset\ideal(S,E)$ for some $i\geq1$, so $S\subset\ceros(\ideal(S,E))\subset\ceros(\ideal(T_i,E))=T_i$, as required.
\end{proof}
As a straightforward consequence of Lemma \ref{ni} and Theorem \ref{irredcomp2} proved below, we have:
\begin{cor}\label{nic}
Let $S\subset M$ be an amenable C-semianalytic set and let $\{S_i\}_{i\geq1}$ be the family of the irreducible components of $S$. Then $S_k\not\subset\bigcup_{i\neq k}S_i$ for each $k\geq1$.
\end{cor}

{\noindent
\begin{proof}[Proof of Theorem \ref{irredcomp2}] \hfill

{\em Existence of the weak irreducible components.}

 Let $\ideal(S)=\bigcap_{i\geq1}\gtp_i$ be a locally finite (irredundant) primary decomposition of $\ideal(S)$. As $\ideal(S)$ is a radical ideal, the ideals $\gtp_i$ are  prime ideals of the ring $H^0(S,(\an_M)_{|_S})$. We have:
\begin{itemize}
\item $S_i=\ceros(\gtp_i)\subset S$ is by Lemma \ref{sheaf} an $S$-tame C-semianalytic set.
\item $S=\bigcup_{i\geq 1}S_i$.
\item The family $\{S_i\}_{i\geq1}$ is locally finite in $S$.
\end{itemize}

We claim: \em Each $S$-tame C-semianalytic set $S_i$ is irreducible and $\ideal(\ceros(\gtp_i),S)=\gtp_i$ for $i\geq1$\em. 

By Lemma \ref{clue3} it is enough to show that $\ideal(\ceros(\gtp_i),S)\subset\gtp_i$ for each $i\geq1$. As the primary decomposition is irredundant, $(\bigcap_{j\neq i}\gtp_j)\setminus\gtp_i\neq\varnothing$. Pick $g_i\in(\bigcap_{j\neq i}\gtp_j)\setminus\gtp_i$. Observe that $h_ig_i\in \ideal(S)\subset\gtp_i$ for each $h_i\in \ideal(\ceros(\gtp_i),S)$ because $h_ig_i$ vanishes identically on $S=S_i\cup\bigcup_{j\neq i}S_j$. As $g_i\not\in\gtp_i$, we conclude $h_i\in\gtp_i$, that is, $\ideal(\ceros(\gtp_i),S)\subset\gtp_i$.

The $S$-tame C-semianalytic sets $S_i$ for $i\geq1$ satisfy conditions (1), (3), (4) in Definition \ref{irredcomptame}. Let us check that they also satisfy condition (2).

Indeed, let $T$ be an irreducible amenable C-semianalytic set such that $S_i\subset T\subset S$. By Lemma \ref{ni} there exists $j\geq1$ such that $S_i\subset T\subset S_j$, so $\gtp_j\subset\gtp_i$. As $\gtp_i$ is a minimal prime ideal between those containing $\ideal(S)$, we deduce $\gtp_j=\gtp_i$, so $S_i=T=S_j$, as required.

\medskip
 {\em Uniqueness of weak irreducible components.} 

Let $\{S_i\}_{i\geq1}$ be the family of weak irreducible components constructed above and let $\{T_j\}_{j\geq1}$ be another family of weak irreducible components of $S$ satisfying the conditions in Definition \ref{irredcomptame}. By Lemma \ref{ni} each $T_i\subset S_j$ for some $j\geq1$. By condition (2) in Definition \ref{irredcomptame} we have $T_i=S_j$. It follows straightforwardly by Lemma \ref{ni} that $\{S_i\}_{i\geq1}=\{T_j\}_{j\geq1}$, as required.
\end{proof}

Next we show that the behavior of the Zariski closure of an amenable C-semianalytic set $S$ in a small enough open neighborhood $U\subset M$ of $S$ with respect to the weak irreducible components is neat.

\begin{prop}\label{neatwirred}
Let $S\subset M$ be an amenable C-semi\-analytic set. There exist an open neighborhood $U\subset M$ of $S$ such that if $X=\ol{S}^{\zar}$ and $\{X_i\}_{i\geq1}$ are the irreducible components of $X$, then $\{S_i=X_i\cap S\}_{i\geq1}$ is the family of the weak irreducible components of $S$ and $X_i=\ol{S_i}^{\zar}$ for $i\geq1$. In particular, if $S$ is a global C-semianalytic subset of $M$, each $S_i$ is a global C-semianalytic subset of $U$.
\end{prop}
\begin{proof}
For each $i\geq1$ there exist by Lemma \ref{sheaf} an open neighborhood $U_i\subset M$ of $S$ and a C-analytic set $Y_i$ such that $S_i=Y_i\cap S$. 

As the family $\{S_i\}_{i\geq1}$ is locally finite in $S$, we may assume after substituting $M$ by an open neighborhood of $S$ that \em the family $\{S_i\}_{i\geq1}$ is locally finite in $M$\em. Indeed, for each $x\in S$ let $U^x\subset M$ be an open neighborhood of $x$ such that only finitely many $S_i$ meet $U^x$. It is enough to take $M'=\bigcup_{x\in S}U^x$ instead of $M$.

By Lemma \ref{neighs} there exists a locally finite family $\{U_i'\}_{i\geq1}$ of open neighborhoods $U_i'\subset U_i$ of $S_i$. As $S_i$ is an amenable C-analytic subset of $U_i'$, the Zariski closure $Y_i'\subset Y_i\cap U_i'$ of $S_i$ in $U_i'$ has its same dimension. Each $Y_i'$ is a closed subset of $U_i'$ and it holds 
$$
S_i=Y_i'\cap S\subset\cl_M(Y_i')\cap S\subset\cl_M(Y_i)\cap S\cap U_i=\cl_{U_i}(Y_i)\cap S=Y_i\cap S=S_i,
$$
that is, $\cl_M(Y_i')\cap S=Y_i\cap S$. By Lemma \ref{bigneigh} there exists an open neighborhood $U\subset M$ of $S$ such that $Y_i''=Y_i'\cap U$ is a closed subset of $U$ for $i\geq1$. Let $h_i$ be an analytic equation of $Y_i'$ in $\an(U_i')$. Notice that $Y_i''$ is a C-analytic subset of $U$ because it is the zero set of the coherent sheaf on $U$:
$$
{\mathcal F}_x=\begin{cases}
h_i\an_{U,x}&\text{ if $x\in Y_i''$},\\
\an_{U,x}&\text{ otherwise}.
\end{cases}
$$
As $S_i$ is irreducible, the Zariski closure $X_i\subset Y_i''$ of $S_i$ in $U$ is an irreducible C-analytic set. Notice that 
\begin{itemize}
\item[(1)] $S_i\subset S\cap X_i\subset S\cap Y_i=S_i$, so $X_i\cap S=S_i$,
\item[(2)] The family $\{X_i\}_{i\geq1}$ is locally finite because $X_i\subset U_i'$ and the family $\{U_i'\}_{i\geq1}$ is locally finite. 
\end{itemize}

Putting all together, we are done.
\end{proof}


In order  to prove Theorem \ref{irredcomp1} we have to prove that the weak irreducible components are in fact the irreducible components of $S$. To do this in particular we have to prove that
the weak irreducible components of $S$ are amenable C-semianalytic sets and constitute a locally finite family of $M$.

We denote the family of the weak irreducible components of $S$ with $\{S_i\}_{i\geq1}$. We begin by some lemmas

\begin{lem}\label{neigh}
For each $i\geq1$ there exists an open C-semianalytic set $U$ such that $S\cap U=S_i\cap U$ is a real analytic manifold of dimension $\dim S_i$.
\end{lem}
\begin{proof}
For simplicity we prove the result for $S_1$. By Corollary \ref{nic} $S_1\not\subset\bigcup_{j>1}S_j$. Let $V\subset M$ be an open neighborhood of $S$ such that $S_1$ is an amenable C-semianalytic subset of $V$ and let $X_1$ be the Zariski closure of $S_1$ in $V$. By Lemma \ref{sheaf} we may assume after shrinking $V$ that there exists $h\in\bigcap_{j>1}\ideal(S_j,S)\cap\an(V)$ such that $\ceros(h)\cap S=\bigcup_{j>1}\ceros(\ideal(S_j,S))$. As $S_1$ is irreducible and $S_1\not\subset\bigcup_{j>1}S_j$, we deduce by Lemma \ref{charirred} that $\dim \ceros(h)\cap S_1<\dim S_1$. Pick a point $x\in\Int_{\Reg(X_1)}(S_1\cap\Reg(X_1)\setminus\bigcup_{j>1}S_j)$ and let $U$ be an open C-semianalytic neighborhood of $x$ in $M$ such that 
$$
S\cap U=\Int_{\Reg(X_1)}\Big(S_1\cap\Reg(X_1)\setminus\bigcup_{j>1}S_j\Big)\cap U.
$$
Observe that $U$ satisfies the conditions in the statement.
\end{proof}

\begin{lem}\label{zclc}
The equality $\dim \ol{S_i}^{\zar}=\dim S_i$ holds for $i\geq1$ and,  after eliminating repetitions, the family $\{\ol{S_i}^{\zar}\}_{i\geq1}$ is locally finite.
\end{lem}

\begin{proof}
As $S$ is amenable, there exists by Theorem \ref{algor1} and \ref{operators} (vi) an index $r\geq1$ such that $S=\bigcup_{k=1}^rT_k(S)$ and $\dim T_{k+1}(S)<\dim T_k(S)$. Recall that each $T_k(S)$ is a real analytic manifold. For each $i\geq1$ there exists by Lemma \ref{neigh} an open C-semianalytic set $U_i$ such that $S\cap U_i=S_i\cap U_i$ is a real analytic manifold of dimension $\dim S_i$. As $S\cap U_i=\bigcup_{k=1}^rT_k(S)\cap U_i$, there exists $1\leq k\leq r$ such that $\dim T_k(S)=\dim T_k(S)\cap U_i=\dim S\cap U_i=\dim S_i$. By Lemma \ref{charirred}(iii) $\ol{S_i}^{\zar}=\ol{S\cap U_i}^{\zar}$ is an irreducible component of $\ol{T_k(S)}^{\zar}$. As $S\cap U_i$ is an amenable C-semianalytic set, $\dim(\ol{S_i}^{\zar})=\dim \ol{S\cap U_i}^{\zar}=\dim S\cap U_i=\dim S_i$.

As the family $\{\ol{T_k(S)}^{\zar}\}_{k=1}^r$ is finite and the irreducible components of each $\ol{T_k(S)}^{\zar}$ constitute a locally finite family, we conclude that the family $\{\ol{S_i}^{\zar}\}_{i\geq1}$ is locally finite after eliminating repetitions.
\end{proof}

\begin{cor}
We have:
\begin{itemize}
\item[(i)] Let $U\subset M$ be an open neighborhood of $S$ and let $\{X_j\}_{j\geq1}$ be the irreducible components of the C-analytic set $X=\ol{S}^{\zar}_U$. Then for each $j\geq1$ there exists $i\geq1$ such that $X_j=\ol{S_i}^{\zar}_U$ and $\I(X_j,X)=\I(S_i,S)\cap\an(X)$.
\item[(ii)] There exists an open neighborhood $V\subset M$ of $S$ such that if $X_i=\ol{S_i}^{\zar}_V$, then $\{X_i\}_{i\geq1}$ is the family of the irreducible components of $X=\ol{S}^{\zar}_V$ and $\I(X_i,X)=\I(S_i,S)\cap\an(X)$ for $i\geq1$.
\end{itemize}
\end{cor}
\begin{proof}
(i) For simplicity consider $U=M$. By Lemma \ref{zclc} $\dim \ol{S_i}^{\zar}=\dim S_i$  and the family $\{\ol{S_i}^{\zar}\}_{i\geq1}$ is locally finite after eliminating repetitions. Thus, $Z=\bigcup_{i\geq1}\ol{S_i}^{\zar}\subset X$ is a C-analytic set that contains $S$, so $Z=X$. By Lemma \ref{ni} there exists $i\geq1$ such that $X_j\subset\ol{S_i}^{\zar}\subset X$, so $X_j=\ol{S_i}^{\zar}$. Consequently
$$
\I(S_i,S)\cap\an(X)=\{f\in\an(X):\ f_{|S_i}=0\}=\{f\in\an(X):\ f_{|\ol{S_i}^{\zar}}=0\}=\I(X_j,X).
$$ 

(ii) Apply (i) to the open neighborhood $V=U$ of $S$ constructed in Proposition \ref{neatwirred}.
\end{proof}

\begin{remark}\label{norm1}Let $\left((Y,\Oo_y),\pi\right)$ be the normalization of a complex analytic set $X$.
If $\Omega$ is an open subset of $X$, then $(\pi^{-1}(\Omega),\pi_{|\Omega})$ is the normalization of $\Omega$. If $Z$ is a connected component of $\pi^{-1}(\Omega)$, the map $\pi_{|Z}:Z\to \Omega$ is proper and $\pi(Z)$ is   a complex analytic subspace of $\Omega$. In addition $Z$ is irreducible because it is connected and normal. Thus, $Z\setminus(\pi_{|Z})^{-1}(\Sing(\pi(Z)))$ is  connected. Consequently,
$$
\pi(Z\setminus(\pi_{|Z})^{-1}(\Sing(\pi(Z))))=\pi(Z)\setminus\Sing(\pi(Z))=\Reg(\pi(Z))
$$
is connected and $\pi(Z)$ is irreducible. Thus, $\pi(Z)$ is an irreducible component of $\Omega$ and $(Z,\pi_{|Z})$ is the normalization of $\pi(Z)$.
\end{remark}

\begin{prop}\label{tamend1}
The weak irreducible components of $S$ are amenable.
\end{prop}
\begin{proof}
Let $S_1$ be a weak irreducible component of $S$ and let us show that $S_1$ is an amenable C-semianalytic set. Let $X$ be the Zariski closure of $S_1$. Note that $X$ is an irreducible C-analytic set.

The first step is to prove that we can replace $S$ by $S'= S\cap X$. 
 This can be done once we prove that $S_1$ is a  weak irreducible component of the amenable C-semi\-analytic set $S'$.

Let $\{S'_k\}_{k\geq1}$ be the family of weak irreducible components of $S'$. By Lemma \ref{ni}, there exists $k\geq1$ such that $S_1\subset S'_k$. We want to prove that $S'_k$ is an $S$-tame C-semianalytic set 

It is enough to check that $S'_k=\ceros(h)$ for some analytic function $h$ on an open neighborhood of $S$. Let $h_0$ be an analytic function on an open neighborhood $W$ of $S'$ such that $S'_k=S'\cap\ceros(h_0)$. As $X$ is a C-analytic subset of $M$, it holds that $X\cap W$ is a C-analytic subset of $M\setminus(X\setminus W)$. By Cartan's Theorem B ${h_0}_{|X\cap W}$ is the restriction to $X\cap W$ of an analytic function $h$ on $M\setminus(X\setminus W)$ such that $\ceros(h)=\ceros(h_0)\cap X\cap W$. Observe that $S\subset M\setminus(X\setminus W)$. Consequently, $S'_k=S\cap\ceros(h)$ is a $S$-tame C-semianalytic set. 

By Lemma \ref{ni} there exists a weak irreducible component $S_j$ of $S$ such that $S_1\subset S'_k\subset S_j$ and we conclude $S'_k=S_1$. Hence we can assume $S=S\cap X$.  

Consider now an irreducible complexification $\widetilde{X}$  of the irreducible C-analytic set $X$. Let $(\widetilde{X}^\R,\an_{\widetilde{X}}^\R)$ be the underlying real analytic structure of $(\widetilde{X},\an_{\widetilde{X}})$.

Recall that if $\Omega$ is an open subset of $\widetilde{X}$, then the irreducible components of $\Omega$ are all pure dimensional and coincide with the closures of the connected components of the complex analytic manifold $\Omega\setminus\Sing(\widetilde{X})$, as proved in Chapter 2, Theorem \ref{decomposition}.  Notice also that $\Reg(\widetilde{X}^\R)=\Reg(\widetilde{X})$ and the irreducible components of $\Omega^\R$ arise as the underlying real structures of the irreducible components of $\Omega$.

Let $(Z,\an_Z)$ be a Stein complexification of $(\widetilde{X}^\R,\an_{\widetilde{X}}^\R)$ and let $\sigma:(Z,\an_Z)\to(Z,\an_Z)$ be an anti-involution whose fixed locus is $\widetilde{X}^\R$. Recall that $\Sing(\widetilde{X}^\R)=\Sing(Z)\cap \widetilde{X}^\R$. Denote the reduction of $(Z,\an_Z)$ with $(Z_1,\an_{Z_1})=(Z,\an_Z^r)$. Observe that $\sigma$ induces an anti-involution on $(Z_1,\an_{Z_1})$ whose fixed part space $(Z^\sigma_1,\an_{Z^\sigma_1})$ satisfies $Z^\sigma_1=\widetilde{X}^\R$ and $\an_{Z^\sigma_1}$ is a quotient (coherent) sheaf of $\an_{\widetilde{X}}^\R$. 

For each $z\in\widetilde{X}^\R$ it holds $\an_{Z^\sigma_1,z}\cong\an_{\widetilde{X},z}^\R/\gtn(\an_{\widetilde{X},z}^\R)$ where $\gtn(\an_{\widetilde{X},z}^\R)$ is the ideal of nilpotents elements of $\an_{\widetilde{X},z}^\R$. As we know from Chapter 1 $(Z_1,\an_{Z_1})$ is a Stein space.

Let $\pi:Y\to Z_1$ be the normalization of $(Z_1,\an_{Z_1})$. As $(Z_1,\an_{Z_1})$ is Stein, $(Y,\an_Y)$ is  Stein by what we saw in Chapter 2.  The anti-involution on $Z_1$ extends to an anti-involution $\widehat{\sigma}$ on $Y$ such that $\pi\circ\widehat{\sigma}=\sigma\circ\pi$. Denote the set of fixed points of $\widehat{\sigma}$ with $Y^{\widehat{\sigma}}=\{y\in Y:\ \widehat{\sigma}(y)=y\}$. As the restriction $\pi|:Y\setminus\pi^{-1}(\Sing(Z_1))\to Z_1\setminus\Sing(Z_1)$ is an invariant holomorphic diffeomorphism, 
$$
\pi(Y^{\widehat{\sigma}}\setminus\pi^{-1}(\Sing(Z_1))=Z_1^{\sigma}\setminus\Sing(Z_1)=\Reg(\widetilde{X}^\R).
$$
As $\pi$ is proper, $\pi(Y^{\widehat{\sigma}})=\widetilde{X}^\R$. Note that 
$\pi^{-1}(S)\cap Y^{\widehat{\sigma}}$ is an amenable C-semianalytic subset of $Y^{\widehat{\sigma}}$.

By Proposition \ref{neatwirred} there exists an open neighborhood $V\subset M$ of $S$ such that if $T=\ol{S}^{\zar}_V$ and $\{T_i\}_{i\geq1}$ are the irreducible components of $T$, we may assume $S_i=T_i\cap S$ for $i\geq1$. Observe that $\dim T_1=\dim S_1=\dim X$, so $T_1$ is an irreducible component of $X\cap V$.

Let $\Omega$ be an open neighborhood of $X\cap V$ in $\widetilde{X}$ such that $\Omega\cap X=V\cap X$ and for each irreducible component $X'$ of $X\cap V$ there exists an irreducible component $\Omega'$ of $\Omega$ such that $\Omega'\cap X=X'$. As $S\subset\Omega$, it holds $\pi^{-1}(S)\subset\pi^{-1}(\Omega)$. Let $\Omega_1$ be the irreducible component of $\Omega$ such that $\Omega_1\cap X=T_1$. Recall that $\Omega_1^\R$ is an irreducible component of $\Omega^\R$. Let $\Theta$ be an open neighborhood of $\Omega$ in $Z_1$ such that $\Theta\cap\widetilde{X}^\R=\Omega^\R$ and for each irreducible component ${\Omega'}^\R$ of $\Omega^\R$ there exists an irreducible component $\Theta'$ of $\Theta$ such that $\Theta'\cap\widetilde{X}^\R={\Omega'}^\R$. Let $\Theta_1$ be the irreducible component of $\Theta$ such that $\Theta_1\cap\widetilde{X}^\R=\Omega_1^\R$. Let $Y_1'$ be the connected component of $Y'=\pi^{-1}(\Theta)$ such that $\pi(Y_1')=\Theta_1$ (see Remark \ref{norm1}).

As $S\subset X\cap V\subset\Omega\subset\Theta$, we have $\pi^{-1}(S)\subset\pi^{-1}(\Theta)=Y'$. As $Y_1'$ is a connected component of $Y'$, the intersection $R_1=Y_1'\cap\pi^{-1}(S)\cap Y^{\widehat{\sigma}}$ is an open and closed subset of $R=\pi^{-1}(S)\cap Y^{\widehat{\sigma}}$, so $R_1$ is a union of connected components of $R$. Let us see $\pi(R_1)=S_1$.
Indeed, $\Omega_1^\R$ is the closure of the connected component $\Omega_1^\R\setminus\Sing(Z_1)$ of $\Omega^\R\setminus\Sing(Z_1)$. As the restriction $\pi_|:Y\setminus\pi^{-1}(\Sing(Z_1))\to Z_1\setminus\Sing(Z_1)$ is an invariant holomorphic diffeomorphism and $\Omega_1^\R=\Theta_1\cap\widetilde{X}^\R$, we conclude 
$$
\pi((Y_1'\setminus\pi^{-1}(\Sing(Z_1)))\cap Y^{\widehat{\sigma}})=(\Theta_1\setminus\Sing(Z_1))\cap Z_1^\sigma=\Omega_1^\R\setminus\Sing(Z_1).
$$
As $\pi$ is proper, $\pi(Y_1'\cap Y^{\widehat{\sigma}})=\Omega_1^\R$. Thus,
$$
\pi(R_1)=\pi(Y_1'\cap Y^{\widehat{\sigma}}\cap\pi^{-1}(S))=\pi(Y_1'\cap Y^{\widehat{\sigma}})\cap S
=\Omega_1^\R\cap S=\Omega_1^\R\cap X\cap S=T_1\cap S=S_1. 
$$

As $R_1$ is a union of connected components of $R$, we deduce by Theorem \ref{properint-tame}(ii) that $S_1=\pi(R_1)$ is an amenable C-semianalytic subset of $Z_1^\sigma$. $(X,\an_X)$ is a closed subspace of $(\widetilde{X}^\R,\an_{\widetilde{X}}^\R)$. As $\an_{X,x}$ contains no nilpotent element for each $x\in X$ (recall that we have considered on $X$ the well-reduced structure, see Chapter 1), it holds that $(X,\an_X)$ is a closed subspace of $(Z_1^\sigma,\an_{Z_1^\sigma})$. Consequently, $S_1$ is an amenable C-analytic subset of $X$. As $X\subset M$ is a C-analytic set, $S_1$ is by Cartan's Theorem B an amenable C-analytic subset of $M$, as required.
\end{proof}

\begin{prop}\label{tamend2}
  The family of the weak irreducible components of $S$ is locally finite in $M$.
\end{prop}
\begin{proof}
 By Lemma \ref{zclc} it is enough to prove the following statement.

\begin{quotation} 
\em Let $X=\ol{S_i}^{\zar}$ for some $i\geq1$ and let ${\mathfrak F}=\{j\geq1:\ \ol{S_j}^{\zar}=X\}$. Then $\{S_j\}_{j\in{\mathfrak F}}$ is locally finite in $M$\em.
\end{quotation}

As seen in the proof of Proposition \ref{tamend1}, we may assume that the Zariski closure of $S$ is $X$. Let $(\widetilde{X},\an_{\widetilde{X}})$ be a complexification of $(X,\an_X)$ an let $(Y,\pi)$ be its normalization. Let $V\subset M$ be an open neighborhood of $S$ such that the irreducible components $\{T_i\}_{i\geq1}$ of $T=\ol{S}^{\zar}_V$ satisfy $S_i=S\cap T_i$ (see Proposition \ref{neatwirred}). The family $\{T_j\}_{j\in{\mathfrak F}}$ is a collection of irreducible components of $X\cap V$ of its same dimension. Let $\Omega$ be an open neighborhood of $X\cap V$ in $\widetilde{X}$ such that $\Omega\cap X=X\cap V$ and for each irreducible component $X'$ of $X\cap V$ there exists an irreducible component $\Omega'$ of $\Omega$ such that $\Omega'\cap X=X'$. Let $\Omega_j$ be the irreducible component of $\Omega$ such that $\Omega_j\cap V=T_j$ for $j\in{\mathfrak F}$. 
By Remark \ref{norm1} $(\Theta=\pi^{-1}(\Omega),\pi_{|\Theta})$ is the normalization of $\Omega$ and for each $j\in{\mathfrak F}$ there exists a connected component $\Theta_j$ of $\Theta$ such that $\pi(\Theta_j)=\Omega_j$ and $(\Theta_j,\pi_{|\Theta_j})$ is the normalization of $\Omega_j$. As $\pi^{-1}(S)\subset\Theta$, the intersection $\pi^{-1}(S)\cap\Theta_j$ is a union of connected components of $\pi^{-1}(S)$. 

We claim: \em $\pi^{-1}(S)\cap\Theta_j=\pi^{-1}(S_j)\cap\Theta_j$. In particular, each connected component of $\pi^{-1}(S_j)\cap\Theta_j$ is a connected component of $\pi^{-1}(S)$.\em

As $\pi(\Theta_j)=\Omega_j$ and $\Omega_j\cap S=\Omega_j\cap X\cap S=T_j\cap S=S_j$, we deduce 
$$
\pi^{-1}(S_j)\cap\Theta_j=\pi^{-1}(\Omega_j)\cap\pi^{-1}(S)\cap\Theta_j=\pi^{-1}(S)\cap\Theta_j.
$$
Fix $j\in{\mathfrak F}$. As $X$ is the Zariski closure of $S_j$ in $M$, there exists by Theorem \ref{dpm}(i) a connected component $R_j$ of $\pi^{-1}(S_j)$ such that $\pi(R_j)=S_j$. As $S_j\subset\Omega_j$, we deduce that $R_j\subset\Theta_j$, so $R_j$ is a connected component of $\pi^{-1}(S_j)\cap\Theta_j$. Consequently, $R_j$ is a connected component of $\pi^{-1}(S)$. As $\pi^{-1}(S)$ is a semianalytic subset of the underlying real analytic structure $(Y^\R,\an_Y^\R)$ of $(Y,\an_Y)$, the family of its connected components is  locally finite. Consequently, the family $\{R_j\}_{j\in{\mathfrak F}}$ is locally finite. By Lemma \ref{proylc} the family $\{S_j=\pi(R_j)\}_{j\in{\mathfrak F}}$ is locally finite in $X$, so it is locally finite in $M$, as required.
\end{proof}

The first statement of Theorem \ref{irredcomp1} is a direct  consequence of Propositions \ref{tamend1} and \ref{tamend2}  if we prove that there is only one family of irreducible components. Indeed let $\{S_i\}_{i\geq1}$ be the family of the weak irreducible components of $S$ and let $\{T_j\}_{j\geq1}$ be a family of irreducible components of $S$ satisfying the conditions of Definition \ref{irredcomptame}. By Lemma \ref{ni} there exists an index $i\geq1$ such that $T_j\subset S_i$ for each $j\geq1$. By condition (2) in Definition \ref{irredcomptame} we have $T_j=S_i$. By Corollary \ref{nic} we conclude $\{T_j\}_{j\geq1}=\{S_i\}_{i\geq1}$, so the family of irreducible components of $S$ is unique.

Summing up what done in this section, we get the second statement of  Theorem \ref{irredcomp1} in a final form.

\begin{thm}\label{finale} There exists a bijection between the irreducible components of an amenable C-semianalytic set $S\subset M$ and the minimal prime ideals of the ring $\Oo(S)$.
If $S$ is an amenable C-semianalytic subset of $M$, there exists an open neighbourhood $U\subset M$ of $S$ such that if $X$ is the Zariski closure of $S$ in $U$ and $\{ X_i\}_{i\geq1}$ are the irreducible components of $X$, then the family $\{S_i= X_i\cap S\}_{i\geq 1}$ is the family of the irreducible components of $S$ and $X_i$  is the Zariski closure of $S_i$ in $U$ for $i\geq 1$. In particular, if $S$ is a global C-semianalytic subset of $M$, each $S_i$ is a global C-semianalytic subset of $U$.
\end{thm}
\begin{cor} Any union of irreducible components of an amenable C-se\-mi\-analytic set $S\subset M$ is an amenable C-semianalytic set.
 \end{cor}

\begin{example}\label{nonpure}
The irreducible components of a pure dimensional amenable C-semianalytic set need not to be pure dimensional. Let $S=T_1\cup T_2\cup T_3\subset\R^3$ where
\begin{multline*}
T_1=[-1,1]\times[-2,2]\times\{0\},\quad T_2=[-2,-1]\times\{-1,1\}\times[-1,1],\\
 \quad T_3=[1,2]\times\{-1,1\}\times[-1,1]
\end{multline*}
By Proposition \ref{neatwirred} and Theorem \ref{finale} the irreducible components of $X$ are the intersections $S_1=S\cap\{x_3=0\}$, $S_2=S\cap\{x_2=1\}$ and $S_3=S\cap\{x_2=-1\}$ and none of them is pure dimensional.
\end{example}

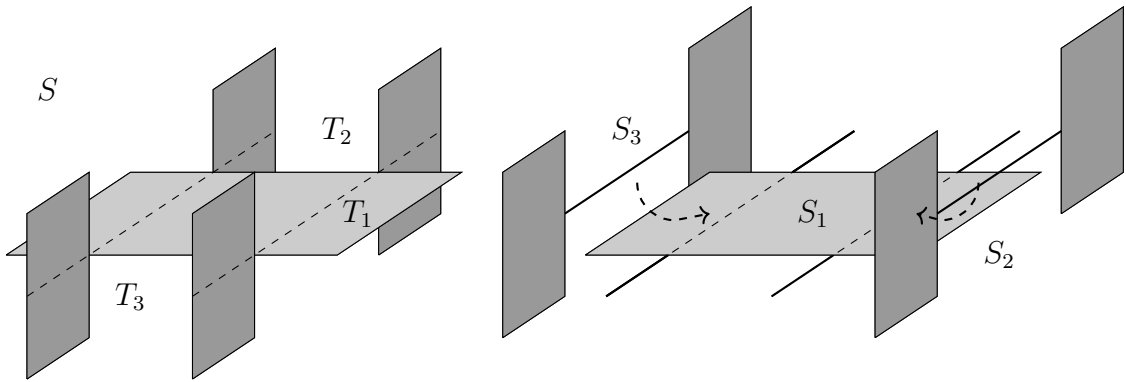
\begin{figure}[ht]
\centering
\begin{tikzpicture}[x=.275cm,y=.275cm]
\draw (10,6) -- (10,14) -- (13,16) -- (13,8) -- (10,6);
\draw[fill=black!40!white] (10,6) -- (10,14) -- (13,16) -- (13,8) -- (10,6);

\draw (18,6) -- (18,14) -- (21,16) -- (21,8) -- (18,6);
\draw[fill=black!40!white] (18,6) -- (18,14) -- (21,16) -- (21,8) -- (18,6);

\draw (0,6) -- (6,10) -- (22,10) -- (16,6) -- (0,6);
\draw[fill=black!20!white] (0,6) -- (6,10) -- (22,10) -- (16,6) -- (0,6);

\draw (1,0) -- (1,8) -- (4,10) -- (4,2) -- (1,0);
\draw[fill=black!40!white] (1,0) -- (1,8) -- (4,10) -- (4,2) -- (1,0);

\draw (9,0) -- (9,8) -- (12,10) -- (12,2) -- (9,0);
\draw[fill=black!40!white] (9,0) -- (9,8) -- (12,10) -- (12,2) -- (9,0);

\draw[dashed] (1,4) -- (13,12);
\draw[dashed] (9,4) -- (21,12);

\draw[dashed] (29,4) -- (41,12);
\draw[dashed] (37,4) -- (49,12);
\draw[thick] (29,4) -- (32,6);
\draw[thick] (38,10) -- (41,12);
\draw[thick] (37,4) -- (40,6);
\draw[thick] (46,10) -- (49,12);
\draw (33,8) -- (33,16) -- (36,18) -- (36,10) -- (33,8);
\draw[fill=black!40!white] (33,8) -- (33,16) -- (36,18) -- (36,10) -- (33,8);
\draw[thick] (27,8) -- (33,12);

\draw (51,8) -- (51,16) -- (54,18) -- (54,10) -- (51,8);
\draw[fill=black!40!white] (51,8) -- (51,16) -- (54,18) -- (54,10) -- (51,8);

\draw (28,6) -- (34,10) -- (50,10) -- (44,6) -- (28,6);
\draw[fill=black!20!white] (28,6) -- (34,10) -- (50,10) -- (44,6) -- (28,6);

\draw[thick] (29,4) -- (32,6);
\draw[thick] (38,10) -- (41,12);
\draw[dashed] (29,4) -- (41,12);
\draw[dashed] (37,4) -- (49,12);

\draw (24,2) -- (24,10) -- (27,12) -- (27,4) -- (24,2);
\draw[fill=black!40!white] (24,2) -- (24,10) -- (27,12) -- (27,4) -- (24,2);

\draw (42,2) -- (42,10) -- (45,12) -- (45,4) -- (42,2);
\draw[fill=black!40!white] (42,2) -- (42,10) -- (45,12) -- (45,4) -- (42,2);
\draw[thick] (45,8) -- (51,12);

\draw (2,14) node{$S$};
\draw (17,8) node{$T_1$};
\draw (16,12) node{$T_2$};
\draw (6,4) node{$T_3$};

\draw (39,8) node{$S_1$};
\draw (48,6) node{$S_2$};
\draw (30,12) node{$S_3$};

\draw[dashed,->,thick] (30.5,9.5) arc (180:270:0.5cm) -- (34,8);
\draw[dashed,->,thick] (47,9.5) arc (0:-90:0.5cm) -- (44,8);

\end{tikzpicture}
\caption{Irreducible components $S_1$, $S_2$ and $S_3$ of $S$ (Example \ref{nonpure})}
\end{figure}


We end this section relating the irreducible components of an amenable C-semi\-ana\-lytic set $S\subset M$ with the normalization $Y$ of a complexification $\widetilde X$ of its Zariski closure $X$. More precisely  
we want to know if the irreducible components of $S$ can be computed as the images of some  connected components of $\pi^{-1}(S)$. The following example shows that this is not true in general. However,  Proposition \ref{dpm2} above shows that under suitable conditions (achieved in Proposition \ref{neatwirred} ), the result is true.

\begin{example}
Consider the amenable C-semianalytic subset $S=S_0\cup S_1\cup S_2\cup S_3$ of $\R^4$ where 
\begin{align*}
S_0&=\{x^2-zy^2=0,z>0\},\\
S_1&=\{x=0,y=0,w=0,z<0\},\\
S_2&=\{x=0,y=0,w+z=-1,z<0\},\\
S_3&=\{x=0,y=0,w-z=1,z<0\}.
\end{align*}
The irreducible components of $S$ are $S_0,S_1,S_2$ and $S_3$. The Zariski closure of $S$ is $X=\{x^2-zy^2=0\}\subset\R^4$. Consider the complexification $\widetilde{X}=\{x^2-zy^2=0\}\subset\C^4$ and its normalization $\pi:Y=\C^3\to\widetilde{X},\ (s,t,w)\mapsto(st,s,t^2,w)$. Observe that $\pi^{-1}(S)=\bigcup_{i=0}^6T_i$ where 

$$\begin{array}{lll}
T_0=\{(s,t,w)\in\R^3:\ t\neq0\}\\[4pt]
T_1=\{(0,i t,0)\in\C^3:\ t>0\},&T_4=\{(0,-i t,0)\in\C^3:\ t>0\},\\[4pt]
T_2=\{(0,i t,-1+t)\in\C^3:\ t>0\},&T_5=\{(0,-i t,-1+t)\in\C^3:\ t>0\},\\[4pt]
T_3=\{(0,i t,1-t)\in\C^3:\ t>0\}, &T_6=\{(0,-i  t,1-t)\in\C^3:\ t>0\}.
\end{array}
$$

In addition $\pi(T_0)=S_0$, $\pi(T_1)=\pi(T_4)=S_1$, $\pi(T_2)=\pi(T_5)=S_2$ and $\pi(T_3)=\pi(T_6)=S_3$. Observe that $T_1\cap T_2\cap T_3=\{(0,i,0)\}$ and $T_4\cap T_5\cap T_6=\{(0,-i,0)\}\}$. Consequently, $\pi^{-1}(S)$ has three connected components $C_0=T_0$, $C_1=T_1\cup T_2\cup T_3$ and $C_2=T_4\cup T_5\cup T_6$, while $S$ has four irreducible components.
\end{example}

The announced result is similar to the statement of 
Proposition \ref{neatwirred}.
\begin{prop}\label{dpm2}
Let $S\subset M$ be an amenable C-semianalytic set and $X\subset M$ be a closed C-analytic set. Denote $\{S_i\}_{i\geq1}$ the irreducible components of $S$ and let $\{X_j\}_{j\geq1}$ be the irreducible components of $X$.
Assume   $S_i=X_i\cap S$ and $\ol{S_i}^{\zar}=X_i$ for $i\geq1$. Then for each $i\geq1$ there exists a connected components $T_i$ of $\pi^{-1}(S)$ (in the normalisation of a complexification of $X$) such that $\pi(T_i)=S_i$.
\end{prop}
\begin{proof}
Let $\widetilde{X}$ be a complexification of $X$ and take $U \subset \widetilde X$ be an open neighborhood of $X$ such that each irreducible component $X_i$ of $X$ is the intersection with $X$ of an irreducible component $U_i$ of $U$. 

By Remark \ref{norm1} $(Y'=\pi^{-1}(U),\pi_{|Y'})$ is the normalization of $U$ and for each $i\geq1$ there exists a connected component $Y'_i$ of $Y'$ such that $\pi(Y'_i)=U_i$ and $(Y'_i,\pi_{|{Y'_i}})$ is the normalization of $U_i$. As $S\subset X$, we have $\pi^{-1}(S)\subset Y'$. Consequently, the intersection $\pi^{-1}(S)\cap Y'_i$ is a union of connected components of $\pi^{-1}(S)$.

As $S_i$ is an amenable C-semianalytic set, there exists by Theorem \ref{dpm} a connected component $T_i$ of $\pi^{-1}(S)\cap Y'_i$ such that $\pi(T_i)=S_i$. Observe that $T_i$ is a connected component of $\pi^{-1}(S)$, as required. 
\end{proof}

\subsection*{ Bibliographic and Historical Notes.} \rm

Global semianalytic sets were considered by Ruiz in \cite{rz0,rz1,rz2,rz3} and more generally in \cite{abr}. In \cite[Th. 2.12 Cap VIII]{abr} one finds also bounds on complexity. 
Ruiz proved that the closure of a global semianalytic set is locally global, that is Theorem \ref{fujitagabrie} using real spectrum: the proof presented here is more geometric and  is inspired by \cite {fu} and \cite{gab}. We use a similar argument to prove Theorem \ref{ccruiz}.
 
The finiteness property is in \cite{abs}. Results for dimension 1 are classical. As for global analytic sets in a real analytic surface the theory was developped in the  thesis of Ana Castilla and can be found in \cite{ca}. 

Theorem \ref{SCH} is in \cite{sch} while Theorem 1.28 and its proof come from \cite{aab}. The general theory of C-semianalytic sets is developped in \cite{abf2} and proceeds in a similar way to the corresponding theory of the semianalytic sets in \cite{bm}. In particular the definition of subanalytic sets, Proposition \ref{bm1} and Theorem \ref{bm2} are in \cite{bm}. 

Theorem \ref{directimage} generalizes the result of  Galbiati \cite{gal2}, where she proved that if $f:X\to Y$ is a proper analytic map between real analytic spaces that admits a proper complexification $\widetilde{f}:\widetilde{X}\to\widetilde{Y}$ and $Z$ is a C-analytic subset of $X$, then $f(X\setminus Z)$ is a semianalytic set.  Hironaka in \cite{hi4} quoted this result and remarked that $f(X\setminus Z)$ is {\em globally semianalytic} in $Y$ with  respect to the given complexification $\widetilde{Y}$ of $Y'$  in the same line as  Theorem \ref{directimage}.

The fact that the set of local extrema of a real analytic function is C-semianalytic was proved by Fernando in (\cite{fernando}) as a particular case of the general framework of {\em weak cathegories.} 

The set of points $N(X)$ where an analytic set $X\subset M$ is not coherent was studied first by Fensch in \cite[I.\S2]{fe} where he proved that  it is contained in a semianalytic set of dimension $\leq\dim(X)-2$. This result was revisited by Galbiati in \cite{gal} and she proved that  it is in fact a semianalytic set. Later Tancredi-Tognoli provided in \cite{tt} a simpler proof of Galbiati's result. Their procedure has helped us to understand the global structure of the set of points of non-coherence of a $C$-analytic set.

The section on amenable C-semianalytic sets comes from \cite{fe1}. 
Irreducibility and irreducible components are usual concepts in Geometry and Algebra. Both concepts are strongly related with prime ideals and primary decomposition of ideals. There is an important background concerning this matter in Algebraic and Analytic Geometry. These concepts have been  developed for complex algebraic sets (Lasker-N\"other \cite{la}), complex analytic sets and Stein spaces (Cartan \cite{c1}, Forster \cite{of}, Remmert-Stein \cite{rs}), $C$-analytic sets (Whitney-Bruhat \cite{wb}), Nash sets (Efroymson \cite{e}, Mostowski \cite{m}, Risler \cite{r}) and semialgebraic sets (Fernando-Gamboa \cite{fg}).
The notion of irreducible component of amenable C-semianalytic sets generalizes the one available for semialgebraic  and for C-analytic sets.  The example (ii) in \ref{tame} is in \cite{wb}.

In some proofs we used the fact that the inverse image under a proper holomorphic map of a Stein manifold is a Stein manifold. This can be found in b \cite[M. Thm. 3]{gu3}.

\newpage

\cleardoublepage


\begin{thebibliography}{SaSaSiTo}
	
\bibitem[Ab]{ab} S.  Abhyankar : Local Analytic Geometry. {Academic Press} New York (1964).
\bibitem[AcAnBr]{aab} F. Acquistapace, C.Andradas, F. Broglia: The strict Positivstellensatz for global analytic functions and the moment problem for semianalytic. {\em Math. Ann. }  316(4)  (2000) 609-616.

\bibitem[AcBr]{moreabout} F. Acquistapace, F., Broglia : More about  signatures  and approximation.   {\em Geom. Dedicata} {\bf 50} (1994), 107-116.




\bibitem[AcBrFe1]{abfr1} F. Acquistapace, F. Broglia, J.F. Fernando, J.M. Ruiz: On the Pythagoras numbers of real-analytic surfaces, {\em Ann. Sci. \'Ecole Norm. Sup.} {\bf38} (2005), no. 5, 751--772.
\bibitem[AcBrFe2]{abfr2} F. Acquistapace, F. Broglia, J.F. Fernando, J.M. Ruiz: On the Pythagoras numbers of real-analytic curves. \em Math. Z. \em {\bf 257} (2007) no. 1, 13--21.

\bibitem[AcBrFe3]{abfr3} F. Acquistapace, F. Broglia, J.F. Fernando, J.M. Ruiz: On the finiteness of Pythagoras numbers of real meromorphic functions, {\em Bull. Soc. Math. France} {\bf 138} (2010), no. 2, 291--307.

\bibitem[AcBrFe4]{abf2} F. Acquistapace, F. Broglia, J.F. Fernando: On globally defined semianalytic sets. {\em Math. Ann.} {\bf 366 } (2016), no. 1-2, 613-654.\\ 
{\tt DOI 10.1007/s00208-015-1342-5, arXiv:1503.00987}
\bibitem[AcBrFe5]{abf1} F. Acquistapace, F. Broglia, J.F. Fernando: On the Nullstellensatz for Stein spaces and real C-analytic sets. {\em Trans. Amer. Math. Soc.} {\bf368} (2016), no. 6, 3899--3929.

\bibitem [AcBrFe5]{abf3} F. Acquistapace, F. Broglia, J.F. Fernando: On Hilbert's 17th problem and Pfister's multiplicative formulae for the ring of real analytic functions, {\em Ann. Sc. Norm. Super. Pisa Cl. Sci.} (5) 13 (2014), no. 2, 333-369.



\bibitem[AcBrNi]{abn} F. Acquistapace, F. Broglia, A. Nicoara: A Nullstellensatz for \L ojasiewicz ideals, {\em Rev. Mat. Iberoam.} {\bf30} (2014), no. 4, 1479--1487. 

\bibitem[AcBrSh]{abs} F. Acquistapace, F. Broglia, M. Shiota: The finiteness property and \L ojasiewicz inequality for global semianalytic sets. {\em Adv. in Geom.} {\bf5} (2005), 453--466.

\bibitem[AcBrTo1]{abt} F. Acquistapace, F. Broglia, A. Tognoli: Sull'insieme di non-coerenza di un insieme analitico reale. \em Atti Accad. Naz. Lincei Rend. Cl. Sci. Fis. Mat. Natur. \em (8) {\bf55} (1973), 42--45 (1974).
\bibitem[AcBrTo2]{AcBrTo} Acquistapace, F., Broglia, F., Tognoli, A.
\newblock Sulla normalizzazione degli spazi analitici reali.
\newblock {\em Boll. Un. Mat. Ital. (4)} {\bf 12} (1975), 26--36.


\bibitem[AcDc]{acdc} F.Acquistapace, A. D\'{\i}az-Cano: Divisors in global analytic set.{ \em  J. Eur. Math. Soc.} {\bf 13} (2011), 297-307.
 
\bibitem[AnBrRu1]{abr2} C. Andradas, L. Br\"ocker, J.M. Ruiz: Minimal generation of basic open semianalytic sets. \em Invent. Math. \em {\bf92} (1988), no. 2, 409--430.

\bibitem[AnBrRu2]{abr} C. Andradas, L. Br\"ocker, J.M. Ruiz: Constructible sets in real geometry. \em Ergeb. Math. \em {\bf 33}. Berlin Heidelberg New York: Springer Verlag (1996). 

\bibitem[AnCa]{ac} C. Andradas, A. Castilla: Connected components of global semianalytic subsets of 2-dimensional analytic manifolds. \em J. Reine Angew. Math. \em {\bf475} (1996), 137--148.

\bibitem[AnDCRu]{AnDCRz}
C. Andradas, A. D\'{\i}az-Cano, J.M. Ruiz :
\newblock The  Artin-Lang  property  for   normal  real  analytic surfaces.
\newblock{\em J. Reine angew. Math.} {\bf 556} (2003), 99--111.




\bibitem[Ar]{ar} E. Artin: \"Uber die Zerlegung definiter Funktionen in Quadrate. {\em Hamb. Abh. Math. Sem. Univ. Hamburg} {\bf 5} (1927), no. 1, 100--115.

\bibitem[ArSc]{arsc} E. Artin, O. Schreier: Algebraische Konstruktion reelle K\"orper. {\em Hamb. Abh. Math. Sem. Univ. Hamburg} {\bf 5} (1926),  85-99.



\bibitem[AtMc]{amd} M.F. Atiyah, I.G. Macdonald: Introduction to commutative algebra{ \em Reading : Addison-Wesley (1969)}

\bibitem[BiMi]{bm}E. Bierstone, P.D. Milman.; Semianalytic and subanalytic sets. {\em Inst. Hautes \'{E}tudes Sci. Publ. Math.} {\bf 67} (1988), 5--42




       



\bibitem[BoCoRo]{bcr} J. Bochnak, M. Coste, M.-F. Roy: Real algebraic geometry. {\em Ergeb. Math.} {\bf 36}, Springer-Verlag, Berlin, (1998).

\bibitem[BoKuSh]{bks} J. Bochnak, W. Kucharz, M.Shiota: On equivalence of ideals of real-analytic functions and the 17th Hilbert problem. {\em Invent. Math.} { \bf 63} (1981), no. 3, 403--421. 






\bibitem[BoRi]{br} J. Bochnak,  J.-J. Risler: Le th\'eor\`eme des z\'{e}ros pour les vari\'{e}t\'{e}s analytiques r\'{e}elles de dimension {$2$}. {\em Ann. Sci. \'{E}cole Norm. Sup.}{\bf (4)} Vol.8 (1975)353--363.





\bibitem[BoHae]{BoHa}
A. Borel, A., A. Haefliger:  La classe d'homologie fundamental d'un espace analytique. {\em Bull. Soc. Math. Fr.} {\bf 89} (1961), 461--513.

\bibitem[Br1]{br1} L.Br\"{o}cker: On the stability index of N\"{e}therian rings.
{\em Real analytic and algebraic geometry ({T}rento, 1988)},
{\em Lecture Notes in Math.}, {\bf 1420} (1990) 72--80. 

\bibitem[Br2]{br2} L.Br\"{o}cker:
      On the separation of basic semialgebraic sets by polynomials.
   {\em Manuscripta Math.} {\bf 60} no. 4, 497--508 (1988)
  
 \bibitem[BrPi]{BrPie}
F. Broglia, F. Pieroni: Separation of global  semianalytic subsets of 2-dimensional analytic
manifolds.
 \newblock{\em Pac. J. of Math.} {\bf 214}(1) (2004), 1--16.





\bibitem[BrPi]{bp} F. Broglia, F. Pieroni: The Nullstellensatz for real coherent analytic surfaces. {\em Rev. Mat. Iberoam.} {\bf 25} (2009), no. 2, 781--798. 


\bibitem[BruCa1]{bc1} F. Bruhat, H. Cartan: Sur les composantes irr\'eductibles d'un sous-ensemble analytique-r\'eel. {\em C. R. Acad. Sci. Paris} {\bf244} (1957), 1123--1126. 

\bibitem[BruC2]{bc2} F. Bruhat, H. Cartan: Sur la structure des sous-ensembles analytiques r\'eels. {\em C. R. Acad. Sci. Paris} {\bf244} (1957), 988--990.

\bibitem[Ca1]{c0} H. Cartan: Id\'eaux et modules de fonctions analytiques de variables complexes. {\em Bull. Soc. Math. France} {\bf78}, (1950), 29--64.

\bibitem[Ca2]{c1} H. Cartan: S\'eminaires, \em E.N.S. \em 1951/52 et 1953/54.



 
\bibitem[Ca3]{c} H. Cartan: Vari\'et\'es analytiques r\'eelles et vari\'et\'es analytiques complexes. {\em Bull. Soc. Math. France} {\bf85} (1957), 77--99.

\bibitem[Ca4]{c2} H. Cartan: Sur les fonctions de plusieurs variables complexes: les espaces analytiques. (1960) \em Proc. Internat. Congress Math\em. (1958) pp. 33--52 Cambridge Univ. Press, New York.

\bibitem[Ca5]{c5} H. Cartan: S\'eminaires, \em E.N.S. \em 1960/61 

\bibitem[Ca6]{c3}H. Cartan, Espaces fibr\'es analytiques. {\em Symposium Internacional de topologia algebraica} (1958) .Mexico

\bibitem[Cs]{ca} A.Castilla: Artin-Lang property for analytic manifolds of dimension two. Math.Z. {\bf 217}, 5-14 (1994)

\bibitem[Co]{Co}
S. Coen :
\newblock Sul rango dei fasci coerenti.
\newblock{\em Boll. Un. Mat. Ital.} {\bf 22} (1967), 373-383.




\bibitem[Da]{d} J.P. D'Angelo: Real and complex geometry meet the Cauchy-Riemann equations. \em Analytic and algebraic geometry\em, 77--182, IAS/Park City Math. Ser., {\bf17}, \em Amer. Math. Soc.\em, Providence, RI, 2010.

\bibitem[dB1]{db} P. de Bartolomeis: Algebre di Stein nel caso reale. {\em Rend. Accad. Naz.} {\bf XL} no. 5 1/2 (1975/76), 105--144 (1977).

\bibitem[dB2]{db1} P. de Bartolomeis: Una nota sulla topologia delle algebre reali coerenti. {\em Boll. Un. Mat. Ital.} {\bf 13A} (1976) no. 5, 123--125. 

\bibitem[DeMa]{dm} C.N. Delzell, J.J. Madden: Lattice-ordered rings and semialgebraic geometry. I. \em Real analytic and algebraic geometry \em (Trento, 1992), 103Ð129, de Gruyter, Berlin, 1995.

\bibitem[Do]{dl} P. Dolbeault: Formes differentielles et cohomologie sur une vari\'et\'e analytique complexe. I. {\em Ann. of Math.} (2) {\bf64} (1956), 83--130.

\bibitem[Ef]{e} G.A. Efroymson: A Nullstellensatz for Nash rings. \em Pacific J. Math. \em {\bf54} (1974), 101--112.

\bibitem[Fe]{fe} W. Fensch: Reell-analytische Strukturen. \em Schr. Math. Inst. Univ. M\"unster \em {\bf34} (1966), 56 pages.

\bibitem[Fe1]{fe2} J.F. Fernando: On Hilbert's 17th problem for global analytic functions in dimension 3. {\em Comment. Math. Helv.} {\bf 83} (2008), 353-363.

\bibitem[Fe2]{fe1} J.F. Fernando: On the irreducible components of globally defined semianalytic sets. {\em Math. Z.} {\bf 283} (2016), no. 3-4, 1071--1109.

\bibitem[Fe3]{fernando}J.F. Fernando: On the set of local extrema of a subanalytic function. {\em Collect. Math.}{\bf 71}(2020) no 1, 1-24

\bibitem[FeGa]{fg} J.F. Fernando, J.M. Gamboa: On the irreducible components of a semialgebraic set. {\em Internat. J. Math.}  {\bf 23} (2012), no. 4, 1250031, 40 pp 




\bibitem[Fo]{of} O. Forster: Prim\"arzerlegung in Steinschen Algebren. {\em Math. Ann.} {\bf 154} (1964), 307--329.

\bibitem[Fr]{f} J. Frisch: Points de platitude d'un morphisme d'espaces analytiques complexes. {\em Invent. Math.} {\bf 4} (1967) 118--138.

\bibitem [Fu]{fu} M. Fujita: Closure and connected component of planar global semianalytic sets defined by analytic functions definable in o-minimal structure. {\em Arc.Math.} {\bf 109 } (2017), 529-538.

\bibitem [Gab] {gab} A. Gabrielov: Complements of subanalytic sets and existential formulas for analytic functions. {\em Invent. Math.} {\bf 125} (1996) 1-12 

\bibitem[Ga1]{gal} M. Galbiati: Stratifications et ensemble de non-coh\'erence d'un espace analytique r\'eel. \em Invent. Math\em. {\bf34} (1976), no. 2, 113--128.

\bibitem[Ga2]{gal2} M. Galbiati: Sur l'image d'un morphisme analytique r\'eel propre. \em 
Ann. Scuola Norm. Sup. Pisa Cl. Sci\em. (4) {\bf3} (1976), no. 2, 311--319. 

\bibitem[Gr1]{g} H. Grauert,  Analytische Faserungen \"uber holomorph-vollst\"andigen R\"aumen. {\em. Math. Ann.} {\bf 135} (1958) 263--273. 


\bibitem[Gr2]{g2} H.Grauert: Ein Theorem der analytische Garbentheorie und die Modulr\"aume complexer Structuren. {\em Publ. Inst. Hautes \'Etudes Sci.} {\bf 5}(1960) 233--292.

\bibitem[Gr,Re]{gare} H. Grauert, R. Remmert: Theory of Stein spaces. Translated from the German by Alan Huckleberry. Reprint of the 1979 translation. \em Classics in Mathematics\em. Springer--Verlag, Berlin: 2004.


\bibitem[Gr,R]{gare1} H. Grauert, R. Remmert: Coherent analytic sheaves.
{\em Grundlehren der Mathematischen Wissenschaften.} Vol. 265.  Springer--Verlag, Berlin: 1984.

\bibitem[Gro]{gro} A. Grothendieck: Techniques de construction en g\'eom\'etrie analytique. II. G\'en\'eralit\'es sur les espaces annel\'es et les espaces analytiques. {\em S\'eminaire Henri Cartan}, {\bf13} no. 1, (1960-1961) Exp. No. 9, 14 p. 

\bibitem[GuMaTa]{gmt} F. Guaraldo, P. Macr\`{\i}, A. Tancredi: Topics on real analytic spaces. \em Advanced Lectures in Mathematics\em. Friedr. Vieweg \& Sohn, Braunschweig: 1986.

\bibitem[Gu1]{gu} R.C. Gunning: Introduction to holomorphic functions of several variables. Vol. II. Local theory. The Wadsworth \& Brooks/Cole Mathematics Series. Wadsworth \& Brooks/Cole Advanced Books \& Software, Monterey, CA: 1990.

\bibitem[Gu2]{gu3} R.C. Gunning: Introduction to holomorphic functions of several variables. Vol. III. Local theory. The Wadsworth \& Brooks/Cole Mathematics Series. Wadsworth \& Brooks/Cole Advanced Books \& Software, Monterey, CA: 1990.


\bibitem[Gu3]{gu4} R.C. Gunning: Introduction to holomorphic functions of several variables. Vol. I-III. Local theory. The Wadsworth \& Brooks/Cole Mathematics Series. Wadsworth \& Brooks/Cole Advanced Books \& Software, Monterey, CA (1990)


\bibitem[GuRo]{gr} R. Gunning, H. Rossi: Analytic functions of several complex variables. Englewood Cliff: Prentice Hall (1965).

\bibitem[HiPh]{hp} E. Hille and R. S. Phillips,  Functional analysis and semigroups. \em Amer. Math. Soc. Colloq. Publ. {\bf31} (rev. ed.), Amer. Math. Soc, Providence, R.I., (1957).\em  Translation:  Funktsional'nyi analiz i polugruppy. \em Izdat. Inost. Lit.,\em  Moscow (1962).

\bibitem[H]{hil} D. Hilbert: \"Uber die vollen Invariantensysteme. {\em Math. Ann.} {\bf42} (1893), no. 3, 313--373.

\bibitem[Hi1]{hi1} H. Hironaka: Introduction aux ensembles sous-analytiques.  R\'edig\'e par Andr\'e Hir\-schowitz et Patrick Le Barz. Singularit\'es \`a Carg\`ese \em (Rencontre Singularit\'es en G\'eom. Anal., Inst. \'Etudes Sci., Carg\`ese, (1972)),\em pp. 13--20.  Asterisque, {\bf7} et {\bf8}, Soc. Math. France, Paris, (1973).

\bibitem[Hi2]{hi2} H. Hironaka: Subanalytic sets. \em Number theory, algebraic geometry and commutative algebra\em, in honor of Yasuo Akizuki, pp. 453--493. Kinokuniya, Tokyo, (1973).

\bibitem[Hi3]{hi4} H. Hironaka: Stratification and flatness. Real and complex singularities (\em Proc. Ninth Nordic Summer School/NAVF Sympos. Math.) \em Oslo  (1976), pp. 199--265. Sijthoff and Noordhoff, Alphen aan den Rijn, (1977).

\bibitem[JoPf]{jp} T. de Jong, G.Pfister: Local analytic geometry. \em Advanced Lectures in Mathematics, \em Friedr. Vieweg \& Sohn, Braunschweig. (2000)  

\bibitem[Jw1]{j1} P. Jaworski: Positive Definite Analytic Functions and Vector Bundles. {\em Bull. Ac. Pol. Sc.  S\'erie Mathematique} {\bf 30}, no. 11-12, (1982).
 
\bibitem[Jw2]{j2} P. Jaworski: Extensions of orderings on fields of quotients of rings of real-analytic functions. {\em Math. Nachr.} {\bf 125} (1986), 329--339.

\bibitem[Ko]{k} J.J Kohn: Subellipticity of the $\ol{\partial}$-Neumann problem on pseudo-convex domains: sufficient conditions. {\em Acta Math}. {\bf142} (1979), no. 1--2, 79--122.


\bibitem[L]{lang}S. Lang: Algebra. Revised third edition. Graduate Texts in Mathematics, {\bf 211}. Springer-Verlag, New York, 2002.
 
\bibitem[La]{la} E. Lasker: Zur Theorie der moduln und Ideale. \em Math. Ann. \em {\bf60} (1905), no. 1, 20--116.

\bibitem[\L o1]{l} S. \L ojasiewicz: Ensembles semi-analytiques, Cours Facult\'e des Sciences d'Orsay, \em Mimeographi\'e I.H.E.S.\em, Bures-sur-Yvette, July: 1965.\\ 
{\tt http://perso.univ-rennes1.fr/michel.coste/Lojasiewicz.pdf}

\bibitem[\L o2]{l1} S. \L ojasiewicz: Triangulation of semi-analytic sets. \em Ann. Scuola Norm. Sup. Pisa \em (3) {\bf18} (1964) 449--474.

\bibitem[Mh1]{mh1} L.Mah\'e: Level and Phythagoras number of some geometric rings, \em Math. Z. {\bf 204} \rm (1990), no.4, 615-629.  

\bibitem[Mh2]{mh2} L.Mah\'e: Erratum: `` Level and Phythagoras number of some geometric rings'' \em Math. Z. {\bf 209} \rm (1992), no.3, 481-483.

\bibitem[Ma]{m1} B. Malgrange: Ideals of differentiable functions. \em Tata Institute of Fundamental Research Studies in Mathematics\em, {\bf3} Tata Institute of Fundamental Research, Bombay; Oxford University Press, London (1967).

\bibitem[Mas]{massey}
\newblock   W.S. Massey, 
\newblock{ Homology and Cohomology Theory}.
\newblock Marcel Dekker, New York, 1978.

\bibitem[Ma]{mat} 
\newblock H.Matsumura: 
\newblock { Commutative rimg theory.} 
\newblock Cambridge Studies in Advanced Mathematics.
\newblock Cambridge University Press Cambridge (1986).

\bibitem[Ma]{mat1}
\newblock H.Matsumura:
\newblock { Commutative algebra.}
\newblock Mathematics Lecture Note Series, vol.56, Second Edition.  
\newblock Benjamin/Cummings Publishing Co., Inc., Reading, Mass. (1980).


\bibitem[Mi]{mi} J.W. Milnor: Topology from the differentiable
viewpoint. Based on notes by David W. Weaver. Revised reprint of the 1965
original. Princeton Landmarks in Mathematics. Princeton University Press,
Princeton, NJ, 1997.



\bibitem[Mo]{m} T. Mostowski: Some properties of the ring of Nash functions. \em Ann. Scuola Norm. Sup. Pisa Cl. Sci. \em (4) {\bf3} (1976), no. 2, 245--266.


\bibitem[Mum]{mumford}
\newblock Mumford, D.
\newblock { The red book of varieties and schemes.}
\newblock Lecture Notes in Mathematics, 1358.
\newblock Springer-Verlag, Berlin, 1988.

\bibitem[Mz]{mz} T.S. Motzkin: The arithmetic-geometric inequality. {\em Inequalities} (Proc. Sympos. Wright-Patterson Air Force Base, Ohio, 1965) pp. 205--224 Academic Press, New York, (1967).

\bibitem[Na1]{n1} R. Narasimhan: Sur les espaces complexes holomorphiquement complets. \em C. R. Acad. Sci. Paris \em {\bf250} (1960) 3560--3561.

\bibitem[Na2]{n2} R. Narasimhan: Imbedding of holomorphically complete complex spaces. \em Amer. J. Math. \em {\bf82} (1960) 917--934.

\bibitem[Na3]{n4} R. Narasimhan: The Levi problem for complex spaces. \em Math. Ann. \em {\bf142} (1960/1961), 355--365.

\bibitem[Na4]{n3} R. Narasimhan: A note on Stein spaces and their normalisations. \em Ann. Scuola Norm. Sup. Pisa \em (3) {\bf16} (1962), 327--333.

\bibitem[Na5]{n5} R. Narasimhan: The Levi problem for complex spaces II. {\em Math. Ann.} {\bf146} (1962) 195--216. 

\bibitem[Na6]{n} R. Narasimhan: Introduction to the theory of analytic spaces, {\em Lecture Notes in Math.} {\bf 25}, Springer-Verlag, Berlin-New York, (1966).

\bibitem[Na7]{n6} R. Narasimhan: Analysis on Real and Complex Manifolds, {\em Advances Stuidies in
Pure Mathematics}, Masson $\&$ Cie, Editeur-Paris; Noth-Holland Publishing Company Amsterdam
(1968).




\bibitem[Ok]{ok} K. Oka: Sur les fonctions analytiques de plusieurs variables. VIII. Lemme fondamental. \em J. Math. Soc. Japan \em {\bf3} (1951) 204--214.


\bibitem[Per]{tesipern}
\newblock L. Pernazza:
\newblock Basicness  and separability with  analytic  and  $C^{\infty}$
function  germs.
\newblock  {\em  Tesi  di Perfezionamento Scuola Normale Superiore}, 2001.

\bibitem[Pf]{pf} A. Pfister: Multiplikative quadratische Formen. {\em Arch. Math.}{\bf 16} (1965), 363--370. 

\bibitem[RS]{rs} R. Remmert, K. Stein: \"Uber dei wesentlichen Singularit\"aten analytischer Mengen. \em Math. Ann. \em {\bf126}, (1953) 263--306.

\bibitem[Ri1]{r3} J.J. Risler: Une caract\'erisation des id\'eaux des vari\'et\'es alg\'ebriques r\'eelles. {\em C. R. Acad. Sci. Paris} S\'er. A-B {\bf271} (1970) A1171--A1173.

\bibitem[Ri2]{r} J.J. Risler: Sur l'anneau des fonctions de Nash globales. \em C. R. Acad. Sci. Paris S\'er\em. A-B {\bf276} (1973), A1513--A1516.

\bibitem[Ri3]{r2} J.J. Risler: Le th\'eor\`eme des z\'eros en g\'eom\'etries alg\'ebrique et analytique r\'eelles. {\em Bull. Soc. Math. France} {\bf104} (1976), no. 2, 113--127.

\bibitem[R\" u]{ru} W. R\"uckert: Zum Eliminationsproblem der Potenzreihenideale. \em 
Math. Ann. \em {\bf107} (1933), no. 1, 259--28
\bibitem[Ru1]{rz0} J.M. Ruiz: On Hilbert's 17th problem and real Nullstellensatz for analytic functions. {\em Math. Z.} {\bf190} (1985), no. 3, 447--454.

\bibitem[Ru2]{rz1} J.M. Ruiz: On the real spectrum of a ring of global analytic functions. \em Alg\`ebre\em, 84--95, \em Publ. Inst. Rech. Math. Rennes\em, 1986-4, Univ. Rennes I, Rennes, (1986).

\bibitem[Ru3]{rz2} J.M. Ruiz: On the connected components of a global semianalytic set. \em J. Reine Angew. Math. \em {\bf392} (1988), 137--144.

\bibitem[Ru4]{rz3} J.M. Ruiz: On the topology of global semianalytic sets. Real analytic and algebraic geometry (Trento, 1988), 237--246, \em Lecture Notes in Math.\em, {\bf1420}, Springer, Berlin, (1990).

\bibitem [Sch]{sch} K. Schm\" udgen: The $K$-moment problem for compact semialgebraic sets. \em Math. Ann. {\bf 289}  \em  (1991)  2.  203-206  

\bibitem[Sh]{sh} M.Shiota: Sur la factorialit\'e de l'anneau des fonctions analytiques. {\em C.R.Acad.Sci. Paris.S\'er. A 285} \rm (1977), 253-255.

\bibitem[ShYo]{shio} M.Shiota, M.Yokoi: Triangulations of subanalytic sets and locally subanalytic manifolds.  Trans. Amer. Math. Soc. 286 (1984), no. 2, 727–750.

\bibitem [Si]{siu} Y.T.Siu: Hilbert Nullstellensatz in global complex analytic case. \em Proc. Amer. Math.Soc. \bf 19 \rm (1968)  296-298
\bibitem[SaSaSiTo]{ssst} G. della Sala, A. Saracco, A. Simioniuc, G. Tomassini: Lectures on complex analysis and analytic geometry. Appunti. {\em Scuola Normale Superiore di Pisa} (Nuova Serie) {\bf3}. Edizioni della Normale, Pisa (2006).

\bibitem[St]{st}N. Steenrod: The topology of fibre bundles. Princeton Landmarks in Mathematics. Princeton University Press, Princeton, NJ (1999).

\bibitem[TaTo]{tt} A. Tancredi, A. Tognoli: On a decomposition of the points of non-coherence of a real analytic space. \em Riv. Mat. Univ. Parma \em (4) {\bf6} (1980), 401--405 (1981).

\bibitem[To1]{t} A. Tognoli: Propriet\`a globali degli spazi analitici reali. \em Ann. Mat. Pura Appl. \em (4) {\bf75} (1967) 143--218.

\bibitem[To2]{tog}
Tognoli, A.
\newblock Une remarque sur les fibr\'es vectoriels analytiques et de Nash.    
\newblock {\em C. R. Acad. Sci. Paris S\'er. A-B, } {\bf 290}(7) (1980), A321-A323.

\bibitem[To1]{t3} A. Tognoli: L'analogo del teorema delle matrici olomorfe invertibili nel caso analitico reale.  {\em Ann. Sc. Norm. Sup.} (3){\bf 22} 4 527-558 (1968)

\bibitem [Wh]{Wh} Whitney, H.: Analytic extensions of differentiable functions defined in closed sets. {\em Trans. Amer. Math. Soc.} 36 (1934), no. 1, 63–89. 


\bibitem[WhBr]{wb} H. Whitney, F. Bruhat: Quelques propi\'et\'es fondamentales des ensembles analytiques r\'eels. \em Comment. Math. Helv. \em {\bf33} (1959), 132--160. 
\bibitem[ZKKP]{zkkp} M.G. Za\u{\i}denberg, S.G. Kre\u{\i}n, P.A. Ku\v{c}ment and A.A. Pankov, \em Banach bundles and linear operators\em (Russian). {Uspehi Mat. Nauk} {\bf 30} (1975), no. 5 (185), 101--157. English translation: Russian Math. Surveys {\bf 30} (1975), no. 5, 115--175.
\end{thebibliography}
\end{document}